\documentclass[sigplan,screen,nonacm]{acmart}

\usepackage[utf8x]{inputenc}
\usepackage{textgreek}
\usepackage{xspace}
\usepackage{microtype}
\usepackage{balance}

\newcommand{\ie}{\emph{i.e.}}
\newcommand{\eg}{\emph{e.g.}}

\newcommand{\epsn}{\varepsilon_0}
\newcommand{\eqdef}{=}%{:\equiv}

\newcommand{\AD}[1]{\AgdaDatatype{#1}}
\newcommand{\AF}[1]{\AgdaFunction{#1}}
\newcommand{\AS}[1]{\AgdaSymbol{#1}}
\newcommand{\AB}[1]{\AgdaBound{#1}}
\newcommand{\AC}[1]{\AgdaInductiveConstructor{#1}}

\newcommand{\ASto}{\AS{→}\xspace}
\newcommand{\ASequiv}{\AgdaOperator{\AF{≡}}\xspace}

\newcommand{\oabr}{\AgdaOperator{\AgdaInductiveConstructor{ω\textasciicircum{}}}~\AgdaBound{a}~\AgdaOperator{\AgdaInductiveConstructor{+}}~\AgdaBound{b}~\AgdaOperator{\AgdaInductiveConstructor{[}}~\AgdaBound{r}~\AgdaOperator{\AgdaInductiveConstructor{]}}}
\newcommand{\ocds}{\AgdaOperator{\AgdaInductiveConstructor{ω\textasciicircum{}}}~\AgdaBound{c}~\AgdaOperator{\AgdaInductiveConstructor{+}}~\AgdaBound{d}~\AgdaOperator{\AgdaInductiveConstructor{[}}~\AgdaBound{s}~\AgdaOperator{\AgdaInductiveConstructor{]}}}

\AtBeginDocument{
\theoremstyle{acmdefinition}
\newtheorem{remark}[theorem]{Remark}
}

\usepackage{agda}
\DeclareUnicodeCharacter{8718}{\qedsymbol}
\DeclareUnicodeCharacter{7476}{$^{\textsf{H}}$}
\DeclareUnicodeCharacter{7481}{$^{\textsf{M}}$}
\DeclareUnicodeCharacter{7482}{$^{\mathbb{N}}$}

\begin{code}[hide]%
\>[0]\<%
\\
\>[0]\AgdaSymbol{\{-\#}\AgdaSpace{}%
\AgdaKeyword{OPTIONS}\AgdaSpace{}%
\AgdaPragma{--cubical}\AgdaSpace{}%
\AgdaPragma{--safe}\AgdaSpace{}%
\AgdaSymbol{\#-\}}\<%
\\
\\[\AgdaEmptyExtraSkip]%
\>[0]\AgdaKeyword{open}\AgdaSpace{}%
\AgdaKeyword{import}\AgdaSpace{}%
\AgdaModule{Agda.Primitive}\AgdaSpace{}%
\AgdaKeyword{public}\<%
\\
\\[\AgdaEmptyExtraSkip]%
\>[0]\AgdaKeyword{open}\AgdaSpace{}%
\AgdaKeyword{import}\AgdaSpace{}%
\AgdaModule{Function}\AgdaSpace{}%
\AgdaKeyword{using}\AgdaSpace{}%
\AgdaSymbol{(}\AgdaOperator{\AgdaFunction{\AgdaUnderscore{}∘\AgdaUnderscore{}}}\AgdaSymbol{)}\AgdaSpace{}%
\AgdaKeyword{public}\<%
\\
\\[\AgdaEmptyExtraSkip]%
\>[0]\AgdaKeyword{open}\AgdaSpace{}%
\AgdaKeyword{import}\AgdaSpace{}%
\AgdaModule{Cubical.Foundations.Everything}\<%
\\
\>[0][@{}l@{\AgdaIndent{0}}]%
\>[1]\AgdaKeyword{using}\AgdaSpace{}%
\AgdaSymbol{(}\AgdaFunction{Type}\AgdaSpace{}%
\AgdaSymbol{;}\AgdaSpace{}%
\AgdaFunction{Type₀}\AgdaSpace{}%
\AgdaSymbol{;}\AgdaSpace{}%
\AgdaPostulate{PathP}\AgdaSpace{}%
\AgdaSymbol{;}\AgdaSpace{}%
\AgdaOperator{\AgdaFunction{\AgdaUnderscore{}≡\AgdaUnderscore{}}}\AgdaSpace{}%
\AgdaSymbol{;}\AgdaSpace{}%
\AgdaFunction{refl}\AgdaSpace{}%
\AgdaSymbol{;}\AgdaSpace{}%
\AgdaOperator{\AgdaFunction{\AgdaUnderscore{}⁻¹}}\AgdaSpace{}%
\AgdaSymbol{;}\AgdaSpace{}%
\AgdaOperator{\AgdaFunction{\AgdaUnderscore{}∙\AgdaUnderscore{}}}\AgdaSymbol{;}\AgdaSpace{}%
\AgdaFunction{J}\AgdaSpace{}%
\AgdaSymbol{;}\AgdaSpace{}%
\AgdaFunction{subst}\AgdaSymbol{)}\<%
\\
\\[\AgdaEmptyExtraSkip]%
\>[0]\AgdaKeyword{import}\AgdaSpace{}%
\AgdaModule{Relation.Binary.PropositionalEquality}\AgdaSpace{}%
\AgdaSymbol{as}\AgdaSpace{}%
\AgdaModule{P}\<%
\\
\>[0]\<%
\end{code}

\newcommand{\LevelVar}{
\begin{code}%
\>[0]\AgdaKeyword{variable}\AgdaSpace{}%
\AgdaGeneralizable{ℓ}\AgdaSpace{}%
\AgdaGeneralizable{ℓ'}\AgdaSpace{}%
\AgdaGeneralizable{ℓ''}\AgdaSpace{}%
\AgdaSymbol{:}\AgdaSpace{}%
\AgdaPostulate{Level}\<%
\end{code}}

\newcommand{\IdFunNoLevel}{
\begin{code}%
\>[0]\AgdaFunction{id}\AgdaSpace{}%
\AgdaSymbol{:}\AgdaSpace{}%
\AgdaSymbol{\{}\AgdaBound{A}\AgdaSpace{}%
\AgdaSymbol{:}\AgdaSpace{}%
\AgdaFunction{Type}\AgdaSpace{}%
\AgdaGeneralizable{ℓ}\AgdaSymbol{\}}\AgdaSpace{}%
\AgdaSymbol{→}\AgdaSpace{}%
\AgdaBound{A}\AgdaSpace{}%
\AgdaSymbol{→}\AgdaSpace{}%
\AgdaBound{A}\<%
\end{code}}

\begin{code}[hide]%
\>[0]\<%
\\
\>[0]\AgdaFunction{id}\AgdaSpace{}%
\AgdaBound{a}\AgdaSpace{}%
\AgdaSymbol{=}\AgdaSpace{}%
\AgdaBound{a}\<%
\\
\\[\AgdaEmptyExtraSkip]%
\>[0]\AgdaKeyword{infixr}\AgdaSpace{}%
\AgdaNumber{2}\AgdaSpace{}%
\AgdaOperator{\AgdaDatatype{\AgdaUnderscore{}⊎\AgdaUnderscore{}}}\<%
\\
\>[0]\<%
\end{code}

\newcommand{\SumType}{
\begin{code}%
\>[0]\AgdaKeyword{data}\AgdaSpace{}%
\AgdaOperator{\AgdaDatatype{\AgdaUnderscore{}⊎\AgdaUnderscore{}}}\AgdaSpace{}%
\AgdaSymbol{(}\AgdaBound{A}\AgdaSpace{}%
\AgdaSymbol{:}\AgdaSpace{}%
\AgdaFunction{Type}\AgdaSpace{}%
\AgdaGeneralizable{ℓ}\AgdaSymbol{)}\AgdaSpace{}%
\AgdaSymbol{(}\AgdaBound{B}\AgdaSpace{}%
\AgdaSymbol{:}\AgdaSpace{}%
\AgdaFunction{Type}\AgdaSpace{}%
\AgdaGeneralizable{ℓ'}\AgdaSymbol{)}\AgdaSpace{}%
\AgdaSymbol{:}\AgdaSpace{}%
\AgdaFunction{Type}\AgdaSpace{}%
\AgdaSymbol{(}\AgdaBound{ℓ}\AgdaSpace{}%
\AgdaOperator{\AgdaPrimitive{⊔}}\AgdaSpace{}%
\AgdaBound{ℓ'}\AgdaSymbol{)}\<%
\\
\>[0][@{}l@{\AgdaIndent{0}}]%
\>[1]\AgdaKeyword{where}\<%
\\
\>[1][@{}l@{\AgdaIndent{0}}]%
\>[2]\AgdaInductiveConstructor{inj₁}\AgdaSpace{}%
\AgdaSymbol{:}\AgdaSpace{}%
\AgdaBound{A}\AgdaSpace{}%
\AgdaSymbol{→}\AgdaSpace{}%
\AgdaBound{A}\AgdaSpace{}%
\AgdaOperator{\AgdaDatatype{⊎}}\AgdaSpace{}%
\AgdaBound{B}\<%
\\
\>[2]\AgdaInductiveConstructor{inj₂}\AgdaSpace{}%
\AgdaSymbol{:}\AgdaSpace{}%
\AgdaBound{B}\AgdaSpace{}%
\AgdaSymbol{→}\AgdaSpace{}%
\AgdaBound{A}\AgdaSpace{}%
\AgdaOperator{\AgdaDatatype{⊎}}\AgdaSpace{}%
\AgdaBound{B}\<%
\end{code}}

\newcommand{\SigmaType}{
\begin{code}%
\>[0]\AgdaKeyword{record}\AgdaSpace{}%
\AgdaRecord{Σ}\AgdaSpace{}%
\AgdaSymbol{\{}\AgdaBound{A}\AgdaSpace{}%
\AgdaSymbol{:}\AgdaSpace{}%
\AgdaFunction{Type}\AgdaSpace{}%
\AgdaGeneralizable{ℓ}\AgdaSymbol{\}}\AgdaSpace{}%
\AgdaSymbol{(}\AgdaBound{B}\AgdaSpace{}%
\AgdaSymbol{:}\AgdaSpace{}%
\AgdaBound{A}\AgdaSpace{}%
\AgdaSymbol{→}\AgdaSpace{}%
\AgdaFunction{Type}\AgdaSpace{}%
\AgdaGeneralizable{ℓ'}\AgdaSymbol{)}\AgdaSpace{}%
\AgdaSymbol{:}\AgdaSpace{}%
\AgdaPrimitiveType{Set}\AgdaSpace{}%
\AgdaSymbol{(}\AgdaBound{ℓ}\AgdaSpace{}%
\AgdaOperator{\AgdaPrimitive{⊔}}\AgdaSpace{}%
\AgdaBound{ℓ'}\AgdaSymbol{)}\<%
\\
\>[0][@{}l@{\AgdaIndent{0}}]%
\>[1]\AgdaKeyword{where}\<%
\\
\>[1][@{}l@{\AgdaIndent{0}}]%
\>[2]\AgdaKeyword{constructor}\AgdaSpace{}%
\AgdaOperator{\AgdaInductiveConstructor{\AgdaUnderscore{},\AgdaUnderscore{}}}\<%
\\
\>[2]\AgdaKeyword{field}\<%
\\
\>[2][@{}l@{\AgdaIndent{0}}]%
\>[3]\AgdaField{pr₁}\AgdaSpace{}%
\AgdaSymbol{:}\AgdaSpace{}%
\AgdaBound{A}\<%
\\
\>[3]\AgdaField{pr₂}\AgdaSpace{}%
\AgdaSymbol{:}\AgdaSpace{}%
\AgdaBound{B}\AgdaSpace{}%
\AgdaField{pr₁}\<%
\end{code}}

\begin{code}[hide]%
\>[0]\<%
\\
\>[0]\AgdaKeyword{open}\AgdaSpace{}%
\AgdaModule{Σ}\AgdaSpace{}%
\AgdaKeyword{public}\<%
\\
\\[\AgdaEmptyExtraSkip]%
\>[0]\AgdaKeyword{infixr}\AgdaSpace{}%
\AgdaNumber{3}\AgdaSpace{}%
\AgdaOperator{\AgdaFunction{\AgdaUnderscore{}×\AgdaUnderscore{}}}\<%
\\
\>[0]\<%
\end{code}

\newcommand{\BinProd}{
\begin{code}%
\>[0]\AgdaOperator{\AgdaFunction{\AgdaUnderscore{}×\AgdaUnderscore{}}}\AgdaSpace{}%
\AgdaSymbol{:}\AgdaSpace{}%
\AgdaFunction{Type}\AgdaSpace{}%
\AgdaGeneralizable{ℓ}\AgdaSpace{}%
\AgdaSymbol{→}\AgdaSpace{}%
\AgdaFunction{Type}\AgdaSpace{}%
\AgdaGeneralizable{ℓ'}\AgdaSpace{}%
\AgdaSymbol{→}\AgdaSpace{}%
\AgdaFunction{Type}\AgdaSpace{}%
\AgdaSymbol{(}\AgdaGeneralizable{ℓ}\AgdaSpace{}%
\AgdaOperator{\AgdaPrimitive{⊔}}\AgdaSpace{}%
\AgdaGeneralizable{ℓ'}\AgdaSymbol{)}\<%
\\
\>[0]\AgdaBound{A}\AgdaSpace{}%
\AgdaOperator{\AgdaFunction{×}}\AgdaSpace{}%
\AgdaBound{B}\AgdaSpace{}%
\AgdaSymbol{=}\AgdaSpace{}%
\AgdaRecord{Σ}\AgdaSpace{}%
\AgdaSymbol{\textbackslash{}(}\AgdaBound{\AgdaUnderscore{}}\AgdaSpace{}%
\AgdaSymbol{:}\AgdaSpace{}%
\AgdaBound{A}\AgdaSymbol{)}\AgdaSpace{}%
\AgdaSymbol{→}\AgdaSpace{}%
\AgdaBound{B}\<%
\end{code}}

\begin{code}[hide]%
\>[0]\<%
\\
\>[0]\AgdaKeyword{infix}%
\>[7]\AgdaNumber{1}\AgdaSpace{}%
\AgdaOperator{\AgdaFunction{begin\AgdaUnderscore{}}}\<%
\\
\>[0]\AgdaKeyword{infixr}\AgdaSpace{}%
\AgdaNumber{2}\AgdaSpace{}%
\AgdaOperator{\AgdaFunction{\AgdaUnderscore{}≡⟨\AgdaUnderscore{}⟩\AgdaUnderscore{}}}\<%
\\
\>[0]\AgdaKeyword{infix}%
\>[7]\AgdaNumber{3}\AgdaSpace{}%
\AgdaOperator{\AgdaFunction{\AgdaUnderscore{}∎}}\<%
\\
\\[\AgdaEmptyExtraSkip]%
\>[0]\AgdaKeyword{private}\<%
\\
\>[0][@{}l@{\AgdaIndent{0}}]%
\>[1]\AgdaKeyword{variable}\<%
\\
\>[1][@{}l@{\AgdaIndent{0}}]%
\>[2]\AgdaGeneralizable{A}\AgdaSpace{}%
\AgdaSymbol{:}\AgdaSpace{}%
\AgdaFunction{Type}\AgdaSpace{}%
\AgdaGeneralizable{ℓ}\<%
\\
\>[2]\AgdaGeneralizable{x}\AgdaSpace{}%
\AgdaGeneralizable{y}\AgdaSpace{}%
\AgdaGeneralizable{z}\AgdaSpace{}%
\AgdaSymbol{:}\AgdaSpace{}%
\AgdaGeneralizable{A}\<%
\\
\\[\AgdaEmptyExtraSkip]%
\>[0]\AgdaOperator{\AgdaFunction{begin\AgdaUnderscore{}}}\AgdaSpace{}%
\AgdaSymbol{:}\AgdaSpace{}%
\AgdaGeneralizable{x}\AgdaSpace{}%
\AgdaOperator{\AgdaFunction{≡}}\AgdaSpace{}%
\AgdaGeneralizable{y}\AgdaSpace{}%
\AgdaSymbol{→}\AgdaSpace{}%
\AgdaGeneralizable{x}\AgdaSpace{}%
\AgdaOperator{\AgdaFunction{≡}}\AgdaSpace{}%
\AgdaGeneralizable{y}\<%
\\
\>[0]\AgdaOperator{\AgdaFunction{begin}}\AgdaSpace{}%
\AgdaBound{x≡y}\AgdaSpace{}%
\AgdaSymbol{=}\AgdaSpace{}%
\AgdaBound{x≡y}\<%
\\
\\[\AgdaEmptyExtraSkip]%
\>[0]\AgdaOperator{\AgdaFunction{\AgdaUnderscore{}≡⟨\AgdaUnderscore{}⟩\AgdaUnderscore{}}}\AgdaSpace{}%
\AgdaSymbol{:}\AgdaSpace{}%
\AgdaSymbol{(}\AgdaBound{x}\AgdaSpace{}%
\AgdaSymbol{:}\AgdaSpace{}%
\AgdaGeneralizable{A}\AgdaSymbol{)}\AgdaSpace{}%
\AgdaSymbol{→}\AgdaSpace{}%
\AgdaBound{x}\AgdaSpace{}%
\AgdaOperator{\AgdaFunction{≡}}\AgdaSpace{}%
\AgdaGeneralizable{y}\AgdaSpace{}%
\AgdaSymbol{→}\AgdaSpace{}%
\AgdaGeneralizable{y}\AgdaSpace{}%
\AgdaOperator{\AgdaFunction{≡}}\AgdaSpace{}%
\AgdaGeneralizable{z}\AgdaSpace{}%
\AgdaSymbol{→}\AgdaSpace{}%
\AgdaBound{x}\AgdaSpace{}%
\AgdaOperator{\AgdaFunction{≡}}\AgdaSpace{}%
\AgdaGeneralizable{z}\<%
\\
\>[0]\AgdaBound{x}\AgdaSpace{}%
\AgdaOperator{\AgdaFunction{≡⟨}}\AgdaSpace{}%
\AgdaBound{x≡y}\AgdaSpace{}%
\AgdaOperator{\AgdaFunction{⟩}}\AgdaSpace{}%
\AgdaBound{y≡z}\AgdaSpace{}%
\AgdaSymbol{=}\AgdaSpace{}%
\AgdaBound{x≡y}\AgdaSpace{}%
\AgdaOperator{\AgdaFunction{∙}}\AgdaSpace{}%
\AgdaBound{y≡z}\<%
\\
\\[\AgdaEmptyExtraSkip]%
\>[0]\AgdaOperator{\AgdaFunction{\AgdaUnderscore{}∎}}\AgdaSpace{}%
\AgdaSymbol{:}\AgdaSpace{}%
\AgdaSymbol{(}\AgdaBound{x}\AgdaSpace{}%
\AgdaSymbol{:}\AgdaSpace{}%
\AgdaGeneralizable{A}\AgdaSymbol{)}\AgdaSpace{}%
\AgdaSymbol{→}\AgdaSpace{}%
\AgdaBound{x}\AgdaSpace{}%
\AgdaOperator{\AgdaFunction{≡}}\AgdaSpace{}%
\AgdaBound{x}\<%
\\
\>[0]\AgdaSymbol{\AgdaUnderscore{}}\AgdaSpace{}%
\AgdaOperator{\AgdaFunction{∎}}\AgdaSpace{}%
\AgdaSymbol{=}\AgdaSpace{}%
\AgdaFunction{refl}\<%
\\
\>[0]\<%
\end{code}

\newcommand{\isPS}{
\begin{code}%
\>[0]\AgdaFunction{isProp}\AgdaSpace{}%
\AgdaFunction{isSet}\AgdaSpace{}%
\AgdaSymbol{:}\AgdaSpace{}%
\AgdaFunction{Type}\AgdaSpace{}%
\AgdaGeneralizable{ℓ}\AgdaSpace{}%
\AgdaSymbol{→}\AgdaSpace{}%
\AgdaFunction{Type}\AgdaSpace{}%
\AgdaGeneralizable{ℓ}\<%
\\
\>[0]\AgdaFunction{isProp}\AgdaSpace{}%
\AgdaBound{A}\AgdaSpace{}%
\AgdaSymbol{=}\AgdaSpace{}%
\AgdaSymbol{(}\AgdaBound{x}\AgdaSpace{}%
\AgdaBound{y}\AgdaSpace{}%
\AgdaSymbol{:}\AgdaSpace{}%
\AgdaBound{A}\AgdaSymbol{)}\AgdaSpace{}%
\AgdaSymbol{→}\AgdaSpace{}%
\AgdaBound{x}\AgdaSpace{}%
\AgdaOperator{\AgdaFunction{≡}}\AgdaSpace{}%
\AgdaBound{y}\<%
\\
\>[0]\AgdaFunction{isSet}\AgdaSpace{}%
\AgdaBound{A}\AgdaSpace{}%
\AgdaSymbol{=}\AgdaSpace{}%
\AgdaSymbol{\{}\AgdaBound{x}\AgdaSpace{}%
\AgdaBound{y}\AgdaSpace{}%
\AgdaSymbol{:}\AgdaSpace{}%
\AgdaBound{A}\AgdaSymbol{\}}\AgdaSpace{}%
\AgdaSymbol{→}\AgdaSpace{}%
\AgdaFunction{isProp}\AgdaSpace{}%
\AgdaSymbol{(}\AgdaBound{x}\AgdaSpace{}%
\AgdaOperator{\AgdaFunction{≡}}\AgdaSpace{}%
\AgdaBound{y}\AgdaSymbol{)}\<%
\end{code}}

\begin{code}[hide]%
\>[0]\<%
\\
\>[0]\AgdaKeyword{open}\AgdaSpace{}%
\AgdaKeyword{import}\AgdaSpace{}%
\AgdaModule{Agda.Builtin.Nat}\<%
\\
\>[0][@{}l@{\AgdaIndent{0}}]%
\>[1]\AgdaKeyword{using}\AgdaSpace{}%
\AgdaSymbol{(}\AgdaInductiveConstructor{zero}\AgdaSpace{}%
\AgdaSymbol{;}\AgdaSpace{}%
\AgdaInductiveConstructor{suc}\AgdaSymbol{)}\<%
\\
\>[1]\AgdaKeyword{renaming}\AgdaSpace{}%
\AgdaSymbol{(}\AgdaDatatype{Nat}\AgdaSpace{}%
\AgdaSymbol{to}\AgdaSpace{}%
\AgdaDatatype{ℕ}\AgdaSymbol{)}\<%
\\
\\[\AgdaEmptyExtraSkip]%
\>[0]\AgdaKeyword{infix}\AgdaSpace{}%
\AgdaNumber{10}\AgdaSpace{}%
\AgdaOperator{\AgdaDatatype{\AgdaUnderscore{}≤ᴺ\AgdaUnderscore{}}}\AgdaSpace{}%
\AgdaOperator{\AgdaFunction{\AgdaUnderscore{}≥ᴺ\AgdaUnderscore{}}}\<%
\\
\>[0]\<%
\end{code}

\newcommand{\NatLeq}{
\begin{code}%
\>[0]\AgdaKeyword{data}\AgdaSpace{}%
\AgdaOperator{\AgdaDatatype{\AgdaUnderscore{}≤ᴺ\AgdaUnderscore{}}}\AgdaSpace{}%
\AgdaSymbol{:}\AgdaSpace{}%
\AgdaDatatype{ℕ}\AgdaSpace{}%
\AgdaSymbol{→}\AgdaSpace{}%
\AgdaDatatype{ℕ}\AgdaSpace{}%
\AgdaSymbol{→}\AgdaSpace{}%
\AgdaFunction{Type₀}\AgdaSpace{}%
\AgdaKeyword{where}\<%
\\
\>[0][@{}l@{\AgdaIndent{0}}]%
\>[1]\AgdaInductiveConstructor{z≤n}\AgdaSpace{}%
\AgdaSymbol{:}\AgdaSpace{}%
\AgdaSymbol{\{}\AgdaBound{n}\AgdaSpace{}%
\AgdaSymbol{:}\AgdaSpace{}%
\AgdaDatatype{ℕ}\AgdaSymbol{\}}\AgdaSpace{}%
\AgdaSymbol{→}\AgdaSpace{}%
\AgdaInductiveConstructor{zero}\AgdaSpace{}%
\AgdaOperator{\AgdaDatatype{≤ᴺ}}\AgdaSpace{}%
\AgdaBound{n}\<%
\\
\>[1]\AgdaInductiveConstructor{s≤s}\AgdaSpace{}%
\AgdaSymbol{:}\AgdaSpace{}%
\AgdaSymbol{\{}\AgdaBound{n}\AgdaSpace{}%
\AgdaBound{m}\AgdaSpace{}%
\AgdaSymbol{:}\AgdaSpace{}%
\AgdaDatatype{ℕ}\AgdaSymbol{\}}\AgdaSpace{}%
\AgdaSymbol{→}\AgdaSpace{}%
\AgdaBound{n}\AgdaSpace{}%
\AgdaOperator{\AgdaDatatype{≤ᴺ}}\AgdaSpace{}%
\AgdaBound{m}\AgdaSpace{}%
\AgdaSymbol{→}\AgdaSpace{}%
\AgdaInductiveConstructor{suc}\AgdaSpace{}%
\AgdaBound{n}\AgdaSpace{}%
\AgdaOperator{\AgdaDatatype{≤ᴺ}}\AgdaSpace{}%
\AgdaInductiveConstructor{suc}\AgdaSpace{}%
\AgdaBound{m}\<%
\end{code}}

\begin{code}[hide]%
\>[0]\<%
\\
\>[0]\AgdaFunction{≤ᴺ-refl}\AgdaSpace{}%
\AgdaSymbol{\{}\AgdaNumber{0}\AgdaSymbol{\}}%
\>[16]\AgdaSymbol{=}\AgdaSpace{}%
\AgdaInductiveConstructor{z≤n}\<%
\\
\>[0]\AgdaFunction{≤ᴺ-refl}\AgdaSpace{}%
\AgdaSymbol{\{}\AgdaInductiveConstructor{suc}\AgdaSpace{}%
\AgdaBound{n}\AgdaSymbol{\}}\AgdaSpace{}%
\AgdaSymbol{=}\AgdaSpace{}%
\AgdaInductiveConstructor{s≤s}\AgdaSpace{}%
\AgdaFunction{≤ᴺ-refl}\<%
\\
\\[\AgdaEmptyExtraSkip]%
\>[0]\AgdaOperator{\AgdaFunction{\AgdaUnderscore{}≥ᴺ\AgdaUnderscore{}}}\AgdaSpace{}%
\AgdaSymbol{:}\AgdaSpace{}%
\AgdaDatatype{ℕ}\AgdaSpace{}%
\AgdaSymbol{→}\AgdaSpace{}%
\AgdaDatatype{ℕ}\AgdaSpace{}%
\AgdaSymbol{→}\AgdaSpace{}%
\AgdaFunction{Type₀}\<%
\\
\>[0]\<%
\end{code}

\begin{code}[hide]%
\>[0]\<%
\\
\>[0]\AgdaFunction{n≤1+n}\AgdaSpace{}%
\AgdaSymbol{\{}\AgdaNumber{0}\AgdaSymbol{\}}%
\>[14]\AgdaSymbol{=}\AgdaSpace{}%
\AgdaInductiveConstructor{z≤n}\<%
\\
\>[0]\AgdaFunction{n≤1+n}\AgdaSpace{}%
\AgdaSymbol{\{}\AgdaInductiveConstructor{suc}\AgdaSpace{}%
\AgdaBound{n}\AgdaSymbol{\}}\AgdaSpace{}%
\AgdaSymbol{=}\AgdaSpace{}%
\AgdaInductiveConstructor{s≤s}\AgdaSpace{}%
\AgdaFunction{n≤1+n}\<%
\\
\>[0]\<%
\end{code}

\newcommand{\PropEqfromPath}{
\begin{code}%
\>[0]\AgdaFunction{PropEqfromPath}\AgdaSpace{}%
\AgdaSymbol{:}\AgdaSpace{}%
\AgdaSymbol{\{}\AgdaBound{A}\AgdaSpace{}%
\AgdaSymbol{:}\AgdaSpace{}%
\AgdaPrimitiveType{Set}\AgdaSpace{}%
\AgdaGeneralizable{ℓ}\AgdaSymbol{\}}\AgdaSpace{}%
\AgdaSymbol{\{}\AgdaBound{x}\AgdaSpace{}%
\AgdaBound{y}\AgdaSpace{}%
\AgdaSymbol{:}\AgdaSpace{}%
\AgdaBound{A}\AgdaSymbol{\}}\AgdaSpace{}%
\AgdaSymbol{→}\AgdaSpace{}%
\AgdaBound{x}\AgdaSpace{}%
\AgdaOperator{\AgdaFunction{≡}}\AgdaSpace{}%
\AgdaBound{y}\AgdaSpace{}%
\AgdaSymbol{→}\AgdaSpace{}%
\AgdaBound{x}\AgdaSpace{}%
\AgdaOperator{\AgdaDatatype{P.≡}}\AgdaSpace{}%
\AgdaBound{y}\<%
\\
\>[0]\AgdaFunction{PropEqfromPath}\AgdaSpace{}%
\AgdaSymbol{\{}\AgdaArgument{x}\AgdaSpace{}%
\AgdaSymbol{=}\AgdaSpace{}%
\AgdaBound{x}\AgdaSymbol{\}}\AgdaSpace{}%
\AgdaBound{p}\AgdaSpace{}%
\AgdaSymbol{=}\AgdaSpace{}%
\AgdaFunction{subst}\AgdaSpace{}%
\AgdaSymbol{(}\AgdaBound{x}\AgdaSpace{}%
\AgdaOperator{\AgdaDatatype{P.≡\AgdaUnderscore{}}}\AgdaSymbol{)}\AgdaSpace{}%
\AgdaBound{p}\AgdaSpace{}%
\AgdaInductiveConstructor{P.refl}\<%
\end{code}}

\newcommand{\OpCubical}{
\begin{code}%
\>[0]\AgdaSymbol{\{-\#}\AgdaSpace{}%
\AgdaKeyword{OPTIONS}\AgdaSpace{}%
\AgdaPragma{--cubical}\AgdaSpace{}%
\AgdaSymbol{\#-\}}\<%
\end{code}}

\begin{code}[hide]%
\>[0]\<%
\\
\>[0]\AgdaKeyword{open}\AgdaSpace{}%
\AgdaKeyword{import}\AgdaSpace{}%
\AgdaModule{Agda.Primitive}\<%
\\
\\[\AgdaEmptyExtraSkip]%
\>[0]\AgdaKeyword{import}\AgdaSpace{}%
\AgdaModule{Agda.Builtin.Cubical.Glue}\<%
\\
\\[\AgdaEmptyExtraSkip]%
\>[0]\AgdaKeyword{open}\AgdaSpace{}%
\AgdaKeyword{import}\AgdaSpace{}%
\AgdaModule{Cubical.Foundations.Everything}\AgdaSpace{}%
\AgdaSymbol{as}\AgdaSpace{}%
\AgdaModule{Cub}\<%
\\
\>[0][@{}l@{\AgdaIndent{0}}]%
\>[1]\AgdaKeyword{using}\AgdaSpace{}%
\AgdaSymbol{(}\AgdaFunction{Type}\AgdaSpace{}%
\AgdaSymbol{;}\AgdaSpace{}%
\AgdaFunction{Type₀}\AgdaSpace{}%
\AgdaSymbol{;}\AgdaSpace{}%
\AgdaDatatype{I}\AgdaSpace{}%
\AgdaSymbol{;}\AgdaSpace{}%
\AgdaInductiveConstructor{i0}\AgdaSpace{}%
\AgdaSymbol{;}\AgdaSpace{}%
\AgdaInductiveConstructor{i1}\AgdaSymbol{)}\<%
\\
\\[\AgdaEmptyExtraSkip]%
\>[0]\AgdaKeyword{variable}\AgdaSpace{}%
\AgdaGeneralizable{ℓ}\AgdaSpace{}%
\AgdaGeneralizable{ℓ'}\AgdaSpace{}%
\AgdaGeneralizable{ℓ''}\AgdaSpace{}%
\AgdaSymbol{:}\AgdaSpace{}%
\AgdaPostulate{Level}\<%
\\
\>[0]\<%
\end{code}

\newcommand{\EmptyType}{
\begin{code}%
\>[0]\AgdaKeyword{data}\AgdaSpace{}%
\AgdaDatatype{⊥}\AgdaSpace{}%
\AgdaSymbol{:}\AgdaSpace{}%
\AgdaFunction{Type₀}\AgdaSpace{}%
\AgdaKeyword{where}\<%
\end{code}}

\newcommand{\IdFun}{
\begin{code}%
\>[0]\AgdaFunction{id}\AgdaSpace{}%
\AgdaSymbol{:}\AgdaSpace{}%
\AgdaSymbol{\{}\AgdaBound{ℓ}\AgdaSpace{}%
\AgdaSymbol{:}\AgdaSpace{}%
\AgdaPostulate{Level}\AgdaSymbol{\}}\AgdaSpace{}%
\AgdaSymbol{\{}\AgdaBound{A}\AgdaSpace{}%
\AgdaSymbol{:}\AgdaSpace{}%
\AgdaFunction{Type}\AgdaSpace{}%
\AgdaBound{ℓ}\AgdaSymbol{\}}\AgdaSpace{}%
\AgdaSymbol{→}\AgdaSpace{}%
\AgdaBound{A}\AgdaSpace{}%
\AgdaSymbol{→}\AgdaSpace{}%
\AgdaBound{A}\<%
\\
\>[0]\AgdaFunction{id}\AgdaSpace{}%
\AgdaBound{a}\AgdaSpace{}%
\AgdaSymbol{=}\AgdaSpace{}%
\AgdaBound{a}\<%
\end{code}}

\newcommand{\PathP}{
\begin{code}%
\>[0]\AgdaFunction{PathP}\AgdaSpace{}%
\AgdaSymbol{:}\AgdaSpace{}%
\AgdaSymbol{(}\AgdaBound{A}\AgdaSpace{}%
\AgdaSymbol{:}\AgdaSpace{}%
\AgdaDatatype{I}\AgdaSpace{}%
\AgdaSymbol{→}\AgdaSpace{}%
\AgdaPrimitiveType{Set}\AgdaSpace{}%
\AgdaGeneralizable{ℓ}\AgdaSymbol{)}\AgdaSpace{}%
\AgdaSymbol{→}\AgdaSpace{}%
\AgdaBound{A}\AgdaSpace{}%
\AgdaInductiveConstructor{i0}\AgdaSpace{}%
\AgdaSymbol{→}\AgdaSpace{}%
\AgdaBound{A}\AgdaSpace{}%
\AgdaInductiveConstructor{i1}\AgdaSpace{}%
\AgdaSymbol{→}\AgdaSpace{}%
\AgdaPrimitiveType{Set}\AgdaSpace{}%
\AgdaGeneralizable{ℓ}\<%
\end{code}}

\newcommand{\Eq}{
\begin{code}%
\>[0]\AgdaOperator{\AgdaFunction{\AgdaUnderscore{}≡\AgdaUnderscore{}}}\AgdaSpace{}%
\AgdaSymbol{:}\AgdaSpace{}%
\AgdaSymbol{\{}\AgdaBound{A}\AgdaSpace{}%
\AgdaSymbol{:}\AgdaSpace{}%
\AgdaPrimitiveType{Set}\AgdaSpace{}%
\AgdaGeneralizable{ℓ}\AgdaSymbol{\}}\AgdaSpace{}%
\AgdaSymbol{→}\AgdaSpace{}%
\AgdaBound{A}\AgdaSpace{}%
\AgdaSymbol{→}\AgdaSpace{}%
\AgdaBound{A}\AgdaSpace{}%
\AgdaSymbol{→}\AgdaSpace{}%
\AgdaPrimitiveType{Set}\AgdaSpace{}%
\AgdaGeneralizable{ℓ}\<%
\\
\>[0]\AgdaOperator{\AgdaFunction{\AgdaUnderscore{}≡\AgdaUnderscore{}}}\AgdaSpace{}%
\AgdaSymbol{\{}\AgdaArgument{A}\AgdaSpace{}%
\AgdaSymbol{=}\AgdaSpace{}%
\AgdaBound{A}\AgdaSymbol{\}}\AgdaSpace{}%
\AgdaSymbol{=}\AgdaSpace{}%
\AgdaFunction{PathP}\AgdaSpace{}%
\AgdaSymbol{(λ}\AgdaSpace{}%
\AgdaBound{\AgdaUnderscore{}}\AgdaSpace{}%
\AgdaSymbol{→}\AgdaSpace{}%
\AgdaBound{A}\AgdaSymbol{)}\<%
\end{code}}

\begin{code}[hide]%
\>[0]\<%
\\
\>[0]\AgdaFunction{PathP}\AgdaSpace{}%
\AgdaSymbol{=}\AgdaSpace{}%
\AgdaPostulate{Cub.PathP}\<%
\\
\>[0]\AgdaFunction{Iso}\AgdaSpace{}%
\AgdaSymbol{=}\AgdaSpace{}%
\AgdaRecord{Cub.Iso}\<%
\\
\>[0]\AgdaOperator{\AgdaFunction{\AgdaUnderscore{}≃\AgdaUnderscore{}}}\AgdaSpace{}%
\AgdaSymbol{=}\AgdaSpace{}%
\AgdaOperator{\AgdaFunction{Agda.Builtin.Cubical.Glue.\AgdaUnderscore{}≃\AgdaUnderscore{}}}\<%
\\
\\[\AgdaEmptyExtraSkip]%
\>[0]\AgdaKeyword{private}\<%
\\
\>[0][@{}l@{\AgdaIndent{0}}]%
\>[1]\AgdaKeyword{variable}\<%
\\
\>[1][@{}l@{\AgdaIndent{0}}]%
\>[2]\AgdaGeneralizable{A}\AgdaSpace{}%
\AgdaGeneralizable{X}\AgdaSpace{}%
\AgdaSymbol{:}\AgdaSpace{}%
\AgdaFunction{Type}\AgdaSpace{}%
\AgdaGeneralizable{ℓ}\<%
\\
\>[2]\AgdaGeneralizable{B}\AgdaSpace{}%
\AgdaSymbol{:}\AgdaSpace{}%
\AgdaGeneralizable{A}\AgdaSpace{}%
\AgdaSymbol{→}\AgdaSpace{}%
\AgdaFunction{Type}\AgdaSpace{}%
\AgdaGeneralizable{ℓ'}\<%
\\
\>[2]\AgdaGeneralizable{C}\AgdaSpace{}%
\AgdaSymbol{:}\AgdaSpace{}%
\AgdaSymbol{(}\AgdaBound{a}\AgdaSpace{}%
\AgdaSymbol{:}\AgdaSpace{}%
\AgdaGeneralizable{A}\AgdaSymbol{)}\AgdaSpace{}%
\AgdaSymbol{→}\AgdaSpace{}%
\AgdaSymbol{(}\AgdaBound{b}\AgdaSpace{}%
\AgdaSymbol{:}\AgdaSpace{}%
\AgdaGeneralizable{B}\AgdaSpace{}%
\AgdaBound{a}\AgdaSymbol{)}\AgdaSpace{}%
\AgdaSymbol{→}\AgdaSpace{}%
\AgdaFunction{Type}\AgdaSpace{}%
\AgdaGeneralizable{ℓ''}\<%
\\
\>[2]\AgdaGeneralizable{w}\AgdaSpace{}%
\AgdaGeneralizable{x}\AgdaSpace{}%
\AgdaGeneralizable{y}\AgdaSpace{}%
\AgdaGeneralizable{z}\AgdaSpace{}%
\AgdaSymbol{:}\AgdaSpace{}%
\AgdaGeneralizable{A}\<%
\\
\>[0]\<%
\end{code}

\newcommand{\ListOfPathFunctions}{
\begin{code}%
\>[0]\AgdaFunction{refl}\AgdaSpace{}%
\AgdaSymbol{:}\AgdaSpace{}%
\AgdaGeneralizable{x}\AgdaSpace{}%
\AgdaOperator{\AgdaFunction{≡}}\AgdaSpace{}%
\AgdaGeneralizable{x}\<%
\\
\>[0]\AgdaOperator{\AgdaFunction{\AgdaUnderscore{}⁻¹}}\AgdaSpace{}%
\AgdaSymbol{:}\AgdaSpace{}%
\AgdaGeneralizable{x}\AgdaSpace{}%
\AgdaOperator{\AgdaFunction{≡}}\AgdaSpace{}%
\AgdaGeneralizable{y}\AgdaSpace{}%
\AgdaSymbol{→}\AgdaSpace{}%
\AgdaGeneralizable{y}\AgdaSpace{}%
\AgdaOperator{\AgdaFunction{≡}}\AgdaSpace{}%
\AgdaGeneralizable{x}\<%
\\
\>[0]\AgdaOperator{\AgdaFunction{\AgdaUnderscore{}∙\AgdaUnderscore{}}}\AgdaSpace{}%
\AgdaSymbol{:}\AgdaSpace{}%
\AgdaGeneralizable{x}\AgdaSpace{}%
\AgdaOperator{\AgdaFunction{≡}}\AgdaSpace{}%
\AgdaGeneralizable{y}\AgdaSpace{}%
\AgdaSymbol{→}\AgdaSpace{}%
\AgdaGeneralizable{y}\AgdaSpace{}%
\AgdaOperator{\AgdaFunction{≡}}\AgdaSpace{}%
\AgdaGeneralizable{z}\AgdaSpace{}%
\AgdaSymbol{→}\AgdaSpace{}%
\AgdaGeneralizable{x}\AgdaSpace{}%
\AgdaOperator{\AgdaFunction{≡}}\AgdaSpace{}%
\AgdaGeneralizable{z}\<%
\\
\>[0]\AgdaFunction{cong}%
\>[136I]\AgdaSymbol{:}\AgdaSpace{}%
\AgdaSymbol{(}\AgdaBound{f}\AgdaSpace{}%
\AgdaSymbol{:}\AgdaSpace{}%
\AgdaSymbol{(}\AgdaBound{a}\AgdaSpace{}%
\AgdaSymbol{:}\AgdaSpace{}%
\AgdaGeneralizable{A}\AgdaSymbol{)}\AgdaSpace{}%
\AgdaSymbol{→}\AgdaSpace{}%
\AgdaGeneralizable{B}\AgdaSpace{}%
\AgdaBound{a}\AgdaSymbol{)}\AgdaSpace{}%
\AgdaSymbol{(}\AgdaBound{p}\AgdaSpace{}%
\AgdaSymbol{:}\AgdaSpace{}%
\AgdaGeneralizable{x}\AgdaSpace{}%
\AgdaOperator{\AgdaFunction{≡}}\AgdaSpace{}%
\AgdaGeneralizable{y}\AgdaSymbol{)}\<%
\\
\>[.][@{}l@{}]\<[136I]%
\>[5]\AgdaSymbol{→}\AgdaSpace{}%
\AgdaFunction{PathP}\AgdaSpace{}%
\AgdaSymbol{(λ}\AgdaSpace{}%
\AgdaBound{i}\AgdaSpace{}%
\AgdaSymbol{→}\AgdaSpace{}%
\AgdaGeneralizable{B}\AgdaSpace{}%
\AgdaSymbol{(}\AgdaBound{p}\AgdaSpace{}%
\AgdaBound{i}\AgdaSymbol{))}\AgdaSpace{}%
\AgdaSymbol{(}\AgdaBound{f}\AgdaSpace{}%
\AgdaGeneralizable{x}\AgdaSymbol{)}\AgdaSpace{}%
\AgdaSymbol{(}\AgdaBound{f}\AgdaSpace{}%
\AgdaGeneralizable{y}\AgdaSymbol{)}\<%
\\
\>[0]\AgdaFunction{cong₂}%
\>[161I]\AgdaSymbol{:}\AgdaSpace{}%
\AgdaSymbol{(}\AgdaBound{f}\AgdaSpace{}%
\AgdaSymbol{:}\AgdaSpace{}%
\AgdaSymbol{(}\AgdaBound{a}\AgdaSpace{}%
\AgdaSymbol{:}\AgdaSpace{}%
\AgdaGeneralizable{A}\AgdaSymbol{)}\AgdaSpace{}%
\AgdaSymbol{→}\AgdaSpace{}%
\AgdaSymbol{(}\AgdaBound{b}\AgdaSpace{}%
\AgdaSymbol{:}\AgdaSpace{}%
\AgdaGeneralizable{B}\AgdaSpace{}%
\AgdaBound{a}\AgdaSymbol{)}\AgdaSpace{}%
\AgdaSymbol{→}\AgdaSpace{}%
\AgdaGeneralizable{C}\AgdaSpace{}%
\AgdaBound{a}\AgdaSpace{}%
\AgdaBound{b}\AgdaSymbol{)}\<%
\\
\>[.][@{}l@{}]\<[161I]%
\>[6]\AgdaSymbol{→}\AgdaSpace{}%
\AgdaSymbol{(}\AgdaBound{p}\AgdaSpace{}%
\AgdaSymbol{:}\AgdaSpace{}%
\AgdaGeneralizable{x}\AgdaSpace{}%
\AgdaOperator{\AgdaFunction{≡}}\AgdaSpace{}%
\AgdaGeneralizable{y}\AgdaSymbol{)}\<%
\\
\>[6]\AgdaSymbol{→}\AgdaSpace{}%
\AgdaSymbol{\{}\AgdaBound{u}\AgdaSpace{}%
\AgdaSymbol{:}\AgdaSpace{}%
\AgdaGeneralizable{B}\AgdaSpace{}%
\AgdaGeneralizable{x}\AgdaSymbol{\}}\AgdaSpace{}%
\AgdaSymbol{\{}\AgdaBound{v}\AgdaSpace{}%
\AgdaSymbol{:}\AgdaSpace{}%
\AgdaGeneralizable{B}\AgdaSpace{}%
\AgdaGeneralizable{y}\AgdaSymbol{\}}\AgdaSpace{}%
\AgdaSymbol{(}\AgdaBound{q}\AgdaSpace{}%
\AgdaSymbol{:}\AgdaSpace{}%
\AgdaFunction{PathP}\AgdaSpace{}%
\AgdaSymbol{(λ}\AgdaSpace{}%
\AgdaBound{i}\AgdaSpace{}%
\AgdaSymbol{→}\AgdaSpace{}%
\AgdaGeneralizable{B}\AgdaSpace{}%
\AgdaSymbol{(}\AgdaBound{p}\AgdaSpace{}%
\AgdaBound{i}\AgdaSymbol{))}\AgdaSpace{}%
\AgdaBound{u}\AgdaSpace{}%
\AgdaBound{v}\AgdaSymbol{)}\<%
\\
\>[6]\AgdaSymbol{→}\AgdaSpace{}%
\AgdaFunction{PathP}\AgdaSpace{}%
\AgdaSymbol{(λ}\AgdaSpace{}%
\AgdaBound{i}\AgdaSpace{}%
\AgdaSymbol{→}\AgdaSpace{}%
\AgdaGeneralizable{C}\AgdaSpace{}%
\AgdaSymbol{(}\AgdaBound{p}\AgdaSpace{}%
\AgdaBound{i}\AgdaSymbol{)}\AgdaSpace{}%
\AgdaSymbol{(}\AgdaBound{q}\AgdaSpace{}%
\AgdaBound{i}\AgdaSymbol{))}\AgdaSpace{}%
\AgdaSymbol{(}\AgdaBound{f}\AgdaSpace{}%
\AgdaGeneralizable{x}\AgdaSpace{}%
\AgdaBound{u}\AgdaSymbol{)}\AgdaSpace{}%
\AgdaSymbol{(}\AgdaBound{f}\AgdaSpace{}%
\AgdaGeneralizable{y}\AgdaSpace{}%
\AgdaBound{v}\AgdaSymbol{)}\<%
\\
\>[0]\AgdaFunction{transport}\AgdaSpace{}%
\AgdaSymbol{:}\AgdaSpace{}%
\AgdaGeneralizable{A}\AgdaSpace{}%
\AgdaOperator{\AgdaFunction{≡}}\AgdaSpace{}%
\AgdaGeneralizable{X}\AgdaSpace{}%
\AgdaSymbol{→}\AgdaSpace{}%
\AgdaGeneralizable{A}\AgdaSpace{}%
\AgdaSymbol{→}\AgdaSpace{}%
\AgdaGeneralizable{X}\<%
\\
\>[0]\AgdaFunction{subst}\AgdaSpace{}%
\AgdaSymbol{:}\AgdaSpace{}%
\AgdaGeneralizable{x}\AgdaSpace{}%
\AgdaOperator{\AgdaFunction{≡}}\AgdaSpace{}%
\AgdaGeneralizable{y}\AgdaSpace{}%
\AgdaSymbol{→}\AgdaSpace{}%
\AgdaGeneralizable{B}\AgdaSpace{}%
\AgdaGeneralizable{x}\AgdaSpace{}%
\AgdaSymbol{→}\AgdaSpace{}%
\AgdaGeneralizable{B}\AgdaSpace{}%
\AgdaGeneralizable{y}\<%
\end{code}}

\newcommand{\IsoEquiv}{
\begin{code}%
\>[0]\AgdaFunction{isoToEquiv}\AgdaSpace{}%
\AgdaSymbol{:}\AgdaSpace{}%
\AgdaFunction{Iso}\AgdaSpace{}%
\AgdaGeneralizable{A}\AgdaSpace{}%
\AgdaGeneralizable{B}\AgdaSpace{}%
\AgdaSymbol{→}\AgdaSpace{}%
\AgdaGeneralizable{A}\AgdaSpace{}%
\AgdaOperator{\AgdaFunction{≃}}\AgdaSpace{}%
\AgdaGeneralizable{B}\<%
\end{code}}

\begin{code}[hide]%
\>[0]\<%
\\
\>[0]\AgdaKeyword{postulate}\<%
\\
\>[0][@{}l@{\AgdaIndent{0}}]%
\>[1]\AgdaPostulate{blah}\AgdaSpace{}%
\AgdaSymbol{:}\AgdaSpace{}%
\AgdaGeneralizable{A}\<%
\\
\\[\AgdaEmptyExtraSkip]%
\>[0]\AgdaFunction{refl}\AgdaSpace{}%
\AgdaSymbol{=}\AgdaSpace{}%
\AgdaPostulate{blah}\<%
\\
\>[0]\AgdaOperator{\AgdaFunction{\AgdaUnderscore{}⁻¹}}\AgdaSpace{}%
\AgdaSymbol{=}\AgdaSpace{}%
\AgdaPostulate{blah}\<%
\\
\>[0]\AgdaOperator{\AgdaFunction{\AgdaUnderscore{}∙\AgdaUnderscore{}}}\AgdaSpace{}%
\AgdaSymbol{=}\AgdaSpace{}%
\AgdaPostulate{blah}\<%
\\
\>[0]\AgdaFunction{cong}\AgdaSpace{}%
\AgdaSymbol{=}\AgdaSpace{}%
\AgdaPostulate{blah}\<%
\\
\>[0]\AgdaFunction{cong₂}\AgdaSpace{}%
\AgdaSymbol{=}\AgdaSpace{}%
\AgdaPostulate{blah}\<%
\\
\>[0]\AgdaFunction{transport}\AgdaSpace{}%
\AgdaSymbol{=}\AgdaSpace{}%
\AgdaPostulate{blah}\<%
\\
\>[0]\AgdaFunction{subst}\AgdaSpace{}%
\AgdaSymbol{=}\AgdaSpace{}%
\AgdaPostulate{blah}\<%
\\
\>[0]\AgdaFunction{isoToEquiv}\AgdaSpace{}%
\AgdaSymbol{=}\AgdaSpace{}%
\AgdaPostulate{blah}\<%
\end{code}

\begin{code}[hide]%
\>[0]\<%
\\
\>[0]\AgdaKeyword{open}\AgdaSpace{}%
\AgdaKeyword{import}\AgdaSpace{}%
\AgdaModule{Preliminaries}\<%
\\
\\[\AgdaEmptyExtraSkip]%
\>[0]\AgdaFunction{Ex}\AgdaSpace{}%
\AgdaSymbol{:}\AgdaSpace{}%
\AgdaGeneralizable{x}\AgdaSpace{}%
\AgdaOperator{\AgdaFunction{≡}}\AgdaSpace{}%
\AgdaGeneralizable{y}\AgdaSpace{}%
\AgdaSymbol{→}\AgdaSpace{}%
\AgdaGeneralizable{y}\AgdaSpace{}%
\AgdaOperator{\AgdaFunction{≡}}\AgdaSpace{}%
\AgdaGeneralizable{z}\AgdaSpace{}%
\AgdaSymbol{→}\AgdaSpace{}%
\AgdaGeneralizable{x}\AgdaSpace{}%
\AgdaOperator{\AgdaFunction{≡}}\AgdaSpace{}%
\AgdaGeneralizable{z}\<%
\\
\>[0]\AgdaFunction{Ex}\AgdaSpace{}%
\AgdaSymbol{\{}\AgdaArgument{x}\AgdaSpace{}%
\AgdaSymbol{=}\AgdaSpace{}%
\AgdaBound{x}\AgdaSymbol{\}}\AgdaSpace{}%
\AgdaSymbol{\{}\AgdaBound{y}\AgdaSymbol{\}}\AgdaSpace{}%
\AgdaSymbol{\{}\AgdaBound{z}\AgdaSymbol{\}}\AgdaSpace{}%
\AgdaBound{p}\AgdaSpace{}%
\AgdaBound{q}\AgdaSpace{}%
\AgdaSymbol{=}\<%
\end{code}

\newcommand{\EqSyntax}{
\begin{code}%
\>[0][@{}l@{\AgdaIndent{1}}]%
\>[1]\AgdaOperator{\AgdaFunction{begin}}\<%
\\
\>[1]\AgdaBound{x}\AgdaSpace{}%
\AgdaOperator{\AgdaFunction{≡⟨}}\AgdaSpace{}%
\AgdaBound{p}\AgdaSpace{}%
\AgdaOperator{\AgdaFunction{⟩}}\<%
\\
\>[1]\AgdaBound{y}\AgdaSpace{}%
\AgdaOperator{\AgdaFunction{≡⟨}}\AgdaSpace{}%
\AgdaBound{q}\AgdaSpace{}%
\AgdaOperator{\AgdaFunction{⟩}}\<%
\\
\>[1]\AgdaBound{z}\AgdaSpace{}%
\AgdaOperator{\AgdaFunction{∎}}\<%
\end{code}}

\newcommand{\NatDef}{
\begin{code}%
\>[0]\AgdaKeyword{data}\AgdaSpace{}%
\AgdaDatatype{ℕ}\AgdaSpace{}%
\AgdaSymbol{:}\AgdaSpace{}%
\AgdaFunction{Type₀}\AgdaSpace{}%
\AgdaKeyword{where}\<%
\\
\>[0][@{}l@{\AgdaIndent{0}}]%
\>[1]\AgdaInductiveConstructor{zero}\AgdaSpace{}%
\AgdaSymbol{:}\AgdaSpace{}%
\AgdaDatatype{ℕ}\<%
\\
\>[1]\AgdaInductiveConstructor{suc}%
\>[6]\AgdaSymbol{:}\AgdaSpace{}%
\AgdaDatatype{ℕ}\AgdaSpace{}%
\AgdaSymbol{→}\AgdaSpace{}%
\AgdaDatatype{ℕ}\<%
\end{code}}

\newcommand{\SO}{\AgdaFunction{SigmaOrd}\xspace}

\begin{code}[hide]%
\>[0]\<%
\\
\>[0]\AgdaSymbol{\{-\#}\AgdaSpace{}%
\AgdaKeyword{OPTIONS}\AgdaSpace{}%
\AgdaPragma{--cubical}\AgdaSpace{}%
\AgdaPragma{--safe}\AgdaSpace{}%
\AgdaSymbol{\#-\}}\<%
\\
\\[\AgdaEmptyExtraSkip]%
\>[0]\AgdaKeyword{module}\AgdaSpace{}%
\AgdaModule{SigmaOrd}\AgdaSpace{}%
\AgdaKeyword{where}\<%
\\
\\[\AgdaEmptyExtraSkip]%
\>[0]\AgdaKeyword{open}\AgdaSpace{}%
\AgdaKeyword{import}\AgdaSpace{}%
\AgdaModule{Cubical.Foundations.Everything}\<%
\\
\>[0][@{}l@{\AgdaIndent{0}}]%
\>[1]\AgdaKeyword{using}%
\>[8I]\AgdaSymbol{(}\AgdaFunction{Type}\AgdaSpace{}%
\AgdaSymbol{;}\AgdaSpace{}%
\AgdaFunction{Type₀}\AgdaSpace{}%
\AgdaSymbol{;}\<%
\\
\>[8I][@{}l@{\AgdaIndent{0}}]%
\>[8]\AgdaOperator{\AgdaFunction{\AgdaUnderscore{}≡\AgdaUnderscore{}}}\AgdaSpace{}%
\AgdaSymbol{;}\AgdaSpace{}%
\AgdaFunction{refl}\AgdaSpace{}%
\AgdaSymbol{;}\AgdaSpace{}%
\AgdaFunction{J}\AgdaSpace{}%
\AgdaSymbol{;}\AgdaSpace{}%
\AgdaFunction{cong}\AgdaSpace{}%
\AgdaSymbol{;}\AgdaSpace{}%
\AgdaFunction{cong₂}\AgdaSpace{}%
\AgdaSymbol{;}\AgdaSpace{}%
\AgdaOperator{\AgdaFunction{\AgdaUnderscore{}∙\AgdaUnderscore{}}}\AgdaSpace{}%
\AgdaSymbol{;}\AgdaSpace{}%
\AgdaOperator{\AgdaFunction{\AgdaUnderscore{}⁻¹}}\AgdaSpace{}%
\AgdaSymbol{;}\AgdaSpace{}%
\AgdaFunction{subst}\AgdaSymbol{)}\<%
\\
\\[\AgdaEmptyExtraSkip]%
\>[0]\AgdaKeyword{open}\AgdaSpace{}%
\AgdaKeyword{import}\AgdaSpace{}%
\AgdaModule{Cubical.Data.Empty}\<%
\\
\\[\AgdaEmptyExtraSkip]%
\>[0]\AgdaKeyword{open}\AgdaSpace{}%
\AgdaKeyword{import}\AgdaSpace{}%
\AgdaModule{Cubical.Relation.Nullary}\<%
\\
\>[0]\AgdaKeyword{open}\AgdaSpace{}%
\AgdaKeyword{import}\AgdaSpace{}%
\AgdaModule{Cubical.Relation.Nullary.DecidableEq}\<%
\\
\\[\AgdaEmptyExtraSkip]%
\>[0]\AgdaKeyword{open}\AgdaSpace{}%
\AgdaKeyword{import}\AgdaSpace{}%
\AgdaModule{Preliminaries}\<%
\\
\\[\AgdaEmptyExtraSkip]%
\>[0]\AgdaKeyword{infix}\AgdaSpace{}%
\AgdaNumber{40}\AgdaSpace{}%
\AgdaOperator{\AgdaInductiveConstructor{ω\textasciicircum{}\AgdaUnderscore{}+\AgdaUnderscore{}}}\<%
\\
\>[0]\<%
\end{code}

\newcommand{\Tree}{
\begin{code}%
\>[0]\AgdaKeyword{data}\AgdaSpace{}%
\AgdaDatatype{Tree}\AgdaSpace{}%
\AgdaSymbol{:}\AgdaSpace{}%
\AgdaFunction{Type₀}\AgdaSpace{}%
\AgdaKeyword{where}\<%
\\
\>[0][@{}l@{\AgdaIndent{0}}]%
\>[1]\AgdaInductiveConstructor{𝟎}\AgdaSpace{}%
\AgdaSymbol{:}\AgdaSpace{}%
\AgdaDatatype{Tree}\<%
\\
\>[1]\AgdaOperator{\AgdaInductiveConstructor{ω\textasciicircum{}\AgdaUnderscore{}+\AgdaUnderscore{}}}\AgdaSpace{}%
\AgdaSymbol{:}\AgdaSpace{}%
\AgdaDatatype{Tree}\AgdaSpace{}%
\AgdaSymbol{→}\AgdaSpace{}%
\AgdaDatatype{Tree}\AgdaSpace{}%
\AgdaSymbol{→}\AgdaSpace{}%
\AgdaDatatype{Tree}\<%
\end{code}}

\begin{code}[hide]%
\>[0]\<%
\\
\>[0]\AgdaFunction{caseTree}\AgdaSpace{}%
\AgdaSymbol{:}\AgdaSpace{}%
\AgdaSymbol{∀}\AgdaSpace{}%
\AgdaSymbol{\{}\AgdaBound{ℓ}\AgdaSymbol{\}}\AgdaSpace{}%
\AgdaSymbol{→}\AgdaSpace{}%
\AgdaSymbol{\{}\AgdaBound{A}\AgdaSpace{}%
\AgdaSymbol{:}\AgdaSpace{}%
\AgdaFunction{Type}\AgdaSpace{}%
\AgdaBound{ℓ}\AgdaSymbol{\}}\AgdaSpace{}%
\AgdaSymbol{(}\AgdaBound{x}\AgdaSpace{}%
\AgdaBound{y}\AgdaSpace{}%
\AgdaSymbol{:}\AgdaSpace{}%
\AgdaBound{A}\AgdaSymbol{)}\AgdaSpace{}%
\AgdaSymbol{→}\AgdaSpace{}%
\AgdaDatatype{Tree}\AgdaSpace{}%
\AgdaSymbol{→}\AgdaSpace{}%
\AgdaBound{A}\<%
\\
\>[0]\AgdaFunction{caseTree}\AgdaSpace{}%
\AgdaBound{x}\AgdaSpace{}%
\AgdaBound{y}%
\>[14]\AgdaInductiveConstructor{𝟎}%
\>[24]\AgdaSymbol{=}\AgdaSpace{}%
\AgdaBound{x}\<%
\\
\>[0]\AgdaFunction{caseTree}\AgdaSpace{}%
\AgdaBound{x}\AgdaSpace{}%
\AgdaBound{y}\AgdaSpace{}%
\AgdaSymbol{(}\AgdaOperator{\AgdaInductiveConstructor{ω\textasciicircum{}}}\AgdaSpace{}%
\AgdaSymbol{\AgdaUnderscore{}}\AgdaSpace{}%
\AgdaOperator{\AgdaInductiveConstructor{+}}\AgdaSpace{}%
\AgdaSymbol{\AgdaUnderscore{})}\AgdaSpace{}%
\AgdaSymbol{=}\AgdaSpace{}%
\AgdaBound{y}\<%
\\
\\[\AgdaEmptyExtraSkip]%
\>[0]\AgdaFunction{𝟎≢ω\textasciicircum{}+}\AgdaSpace{}%
\AgdaSymbol{:}\AgdaSpace{}%
\AgdaOperator{\AgdaFunction{¬}}\AgdaSpace{}%
\AgdaSymbol{(}\AgdaInductiveConstructor{𝟎}\AgdaSpace{}%
\AgdaOperator{\AgdaFunction{≡}}\AgdaSpace{}%
\AgdaOperator{\AgdaInductiveConstructor{ω\textasciicircum{}}}\AgdaSpace{}%
\AgdaGeneralizable{a}\AgdaSpace{}%
\AgdaOperator{\AgdaInductiveConstructor{+}}\AgdaSpace{}%
\AgdaGeneralizable{b}\AgdaSymbol{)}\<%
\\
\>[0]\AgdaFunction{𝟎≢ω\textasciicircum{}+}\AgdaSpace{}%
\AgdaBound{e}\AgdaSpace{}%
\AgdaSymbol{=}\AgdaSpace{}%
\AgdaFunction{subst}\AgdaSpace{}%
\AgdaSymbol{(}\AgdaFunction{caseTree}\AgdaSpace{}%
\AgdaDatatype{Tree}\AgdaSpace{}%
\AgdaDatatype{⊥}\AgdaSymbol{)}\AgdaSpace{}%
\AgdaBound{e}\AgdaSpace{}%
\AgdaInductiveConstructor{𝟎}\<%
\\
\\[\AgdaEmptyExtraSkip]%
\>[0]\AgdaFunction{ω\textasciicircum{}+≢𝟎}\AgdaSpace{}%
\AgdaSymbol{:}\AgdaSpace{}%
\AgdaOperator{\AgdaFunction{¬}}\AgdaSpace{}%
\AgdaSymbol{(}\AgdaOperator{\AgdaInductiveConstructor{ω\textasciicircum{}}}\AgdaSpace{}%
\AgdaGeneralizable{a}\AgdaSpace{}%
\AgdaOperator{\AgdaInductiveConstructor{+}}\AgdaSpace{}%
\AgdaGeneralizable{b}\AgdaSpace{}%
\AgdaOperator{\AgdaFunction{≡}}\AgdaSpace{}%
\AgdaInductiveConstructor{𝟎}\AgdaSymbol{)}\<%
\\
\>[0]\AgdaFunction{ω\textasciicircum{}+≢𝟎}\AgdaSpace{}%
\AgdaBound{e}\AgdaSpace{}%
\AgdaSymbol{=}\AgdaSpace{}%
\AgdaFunction{subst}\AgdaSpace{}%
\AgdaSymbol{(}\AgdaFunction{caseTree}\AgdaSpace{}%
\AgdaDatatype{⊥}\AgdaSpace{}%
\AgdaDatatype{Tree}\AgdaSymbol{)}\AgdaSpace{}%
\AgdaBound{e}\AgdaSpace{}%
\AgdaInductiveConstructor{𝟎}\<%
\\
\>[0]\<%
\end{code}

\newcommand{\TreeFst}{
\begin{code}%
\>[0]\AgdaFunction{fst}\AgdaSpace{}%
\AgdaSymbol{:}\AgdaSpace{}%
\AgdaDatatype{Tree}\AgdaSpace{}%
\AgdaSymbol{→}\AgdaSpace{}%
\AgdaDatatype{Tree}\<%
\\
\>[0]\AgdaFunction{fst}\AgdaSpace{}%
\AgdaInductiveConstructor{𝟎}\AgdaSpace{}%
\AgdaSymbol{=}\AgdaSpace{}%
\AgdaInductiveConstructor{𝟎}\<%
\\
\>[0]\AgdaFunction{fst}\AgdaSpace{}%
\AgdaSymbol{(}\AgdaOperator{\AgdaInductiveConstructor{ω\textasciicircum{}}}\AgdaSpace{}%
\AgdaBound{a}\AgdaSpace{}%
\AgdaOperator{\AgdaInductiveConstructor{+}}\AgdaSpace{}%
\AgdaSymbol{\AgdaUnderscore{})}\AgdaSpace{}%
\AgdaSymbol{=}\AgdaSpace{}%
\AgdaBound{a}\<%
\end{code}}

\begin{code}[hide]%
\>[0]\<%
\\
\>[0]\AgdaFunction{rest}\AgdaSpace{}%
\AgdaSymbol{:}\AgdaSpace{}%
\AgdaDatatype{Tree}\AgdaSpace{}%
\AgdaSymbol{→}\AgdaSpace{}%
\AgdaDatatype{Tree}\<%
\\
\>[0]\AgdaFunction{rest}%
\>[6]\AgdaInductiveConstructor{𝟎}%
\>[16]\AgdaSymbol{=}\AgdaSpace{}%
\AgdaInductiveConstructor{𝟎}\<%
\\
\>[0]\AgdaFunction{rest}\AgdaSpace{}%
\AgdaSymbol{(}\AgdaOperator{\AgdaInductiveConstructor{ω\textasciicircum{}}}\AgdaSpace{}%
\AgdaSymbol{\AgdaUnderscore{}}\AgdaSpace{}%
\AgdaOperator{\AgdaInductiveConstructor{+}}\AgdaSpace{}%
\AgdaBound{b}\AgdaSymbol{)}\AgdaSpace{}%
\AgdaSymbol{=}\AgdaSpace{}%
\AgdaBound{b}\<%
\\
\\[\AgdaEmptyExtraSkip]%
\>[0]\AgdaFunction{TreeIsDiscrete}\AgdaSpace{}%
\AgdaSymbol{:}\AgdaSpace{}%
\AgdaFunction{Discrete}\AgdaSpace{}%
\AgdaDatatype{Tree}\AgdaSpace{}%
\AgdaComment{-- (a b : Tree) → Dec (a ≡ b)}\<%
\\
\>[0]\AgdaFunction{TreeIsDiscrete}%
\>[16]\AgdaInductiveConstructor{𝟎}%
\>[27]\AgdaInductiveConstructor{𝟎}%
\>[37]\AgdaSymbol{=}\AgdaSpace{}%
\AgdaInductiveConstructor{yes}\AgdaSpace{}%
\AgdaFunction{refl}\<%
\\
\>[0]\AgdaFunction{TreeIsDiscrete}%
\>[16]\AgdaInductiveConstructor{𝟎}%
\>[26]\AgdaSymbol{(}\AgdaOperator{\AgdaInductiveConstructor{ω\textasciicircum{}}}\AgdaSpace{}%
\AgdaSymbol{\AgdaUnderscore{}}\AgdaSpace{}%
\AgdaOperator{\AgdaInductiveConstructor{+}}\AgdaSpace{}%
\AgdaSymbol{\AgdaUnderscore{})}\AgdaSpace{}%
\AgdaSymbol{=}\AgdaSpace{}%
\AgdaInductiveConstructor{no}\AgdaSpace{}%
\AgdaFunction{𝟎≢ω\textasciicircum{}+}\<%
\\
\>[0]\AgdaFunction{TreeIsDiscrete}\AgdaSpace{}%
\AgdaSymbol{(}\AgdaOperator{\AgdaInductiveConstructor{ω\textasciicircum{}}}\AgdaSpace{}%
\AgdaSymbol{\AgdaUnderscore{}}\AgdaSpace{}%
\AgdaOperator{\AgdaInductiveConstructor{+}}\AgdaSpace{}%
\AgdaSymbol{\AgdaUnderscore{})}%
\>[27]\AgdaInductiveConstructor{𝟎}%
\>[37]\AgdaSymbol{=}\AgdaSpace{}%
\AgdaInductiveConstructor{no}\AgdaSpace{}%
\AgdaFunction{ω\textasciicircum{}+≢𝟎}\<%
\\
\>[0]\AgdaFunction{TreeIsDiscrete}\AgdaSpace{}%
\AgdaSymbol{(}\AgdaOperator{\AgdaInductiveConstructor{ω\textasciicircum{}}}\AgdaSpace{}%
\AgdaBound{a}\AgdaSpace{}%
\AgdaOperator{\AgdaInductiveConstructor{+}}\AgdaSpace{}%
\AgdaBound{c}\AgdaSymbol{)}\AgdaSpace{}%
\AgdaSymbol{(}\AgdaOperator{\AgdaInductiveConstructor{ω\textasciicircum{}}}\AgdaSpace{}%
\AgdaBound{b}\AgdaSpace{}%
\AgdaOperator{\AgdaInductiveConstructor{+}}\AgdaSpace{}%
\AgdaBound{d}\AgdaSymbol{)}\AgdaSpace{}%
\AgdaKeyword{with}\AgdaSpace{}%
\AgdaFunction{TreeIsDiscrete}\AgdaSpace{}%
\AgdaBound{a}\AgdaSpace{}%
\AgdaBound{b}\<%
\\
\>[0]\AgdaFunction{TreeIsDiscrete}\AgdaSpace{}%
\AgdaSymbol{(}\AgdaOperator{\AgdaInductiveConstructor{ω\textasciicircum{}}}\AgdaSpace{}%
\AgdaBound{a}\AgdaSpace{}%
\AgdaOperator{\AgdaInductiveConstructor{+}}\AgdaSpace{}%
\AgdaBound{c}\AgdaSymbol{)}\AgdaSpace{}%
\AgdaSymbol{(}\AgdaOperator{\AgdaInductiveConstructor{ω\textasciicircum{}}}\AgdaSpace{}%
\AgdaBound{b}\AgdaSpace{}%
\AgdaOperator{\AgdaInductiveConstructor{+}}\AgdaSpace{}%
\AgdaBound{d}\AgdaSymbol{)}\AgdaSpace{}%
\AgdaSymbol{|}\AgdaSpace{}%
\AgdaInductiveConstructor{yes}\AgdaSpace{}%
\AgdaBound{a=b}\AgdaSpace{}%
\AgdaKeyword{with}\AgdaSpace{}%
\AgdaFunction{TreeIsDiscrete}\AgdaSpace{}%
\AgdaBound{c}\AgdaSpace{}%
\AgdaBound{d}\<%
\\
\>[0]\AgdaFunction{TreeIsDiscrete}\AgdaSpace{}%
\AgdaSymbol{(}\AgdaOperator{\AgdaInductiveConstructor{ω\textasciicircum{}}}\AgdaSpace{}%
\AgdaBound{a}\AgdaSpace{}%
\AgdaOperator{\AgdaInductiveConstructor{+}}\AgdaSpace{}%
\AgdaBound{c}\AgdaSymbol{)}\AgdaSpace{}%
\AgdaSymbol{(}\AgdaOperator{\AgdaInductiveConstructor{ω\textasciicircum{}}}\AgdaSpace{}%
\AgdaBound{b}\AgdaSpace{}%
\AgdaOperator{\AgdaInductiveConstructor{+}}\AgdaSpace{}%
\AgdaBound{d}\AgdaSymbol{)}\AgdaSpace{}%
\AgdaSymbol{|}\AgdaSpace{}%
\AgdaInductiveConstructor{yes}\AgdaSpace{}%
\AgdaBound{a=b}\AgdaSpace{}%
\AgdaSymbol{|}\AgdaSpace{}%
\AgdaInductiveConstructor{yes}\AgdaSpace{}%
\AgdaBound{c=d}\AgdaSpace{}%
\AgdaSymbol{=}\AgdaSpace{}%
\AgdaInductiveConstructor{yes}\AgdaSpace{}%
\AgdaSymbol{(}\AgdaFunction{cong₂}\AgdaSpace{}%
\AgdaOperator{\AgdaInductiveConstructor{ω\textasciicircum{}\AgdaUnderscore{}+\AgdaUnderscore{}}}\AgdaSpace{}%
\AgdaBound{a=b}\AgdaSpace{}%
\AgdaBound{c=d}\AgdaSymbol{)}\<%
\\
\>[0]\AgdaFunction{TreeIsDiscrete}\AgdaSpace{}%
\AgdaSymbol{(}\AgdaOperator{\AgdaInductiveConstructor{ω\textasciicircum{}}}\AgdaSpace{}%
\AgdaBound{a}\AgdaSpace{}%
\AgdaOperator{\AgdaInductiveConstructor{+}}\AgdaSpace{}%
\AgdaBound{c}\AgdaSymbol{)}\AgdaSpace{}%
\AgdaSymbol{(}\AgdaOperator{\AgdaInductiveConstructor{ω\textasciicircum{}}}\AgdaSpace{}%
\AgdaBound{b}\AgdaSpace{}%
\AgdaOperator{\AgdaInductiveConstructor{+}}\AgdaSpace{}%
\AgdaBound{d}\AgdaSymbol{)}\AgdaSpace{}%
\AgdaSymbol{|}\AgdaSpace{}%
\AgdaInductiveConstructor{yes}\AgdaSpace{}%
\AgdaBound{a=b}\AgdaSpace{}%
\AgdaSymbol{|}\AgdaSpace{}%
\AgdaInductiveConstructor{no}%
\>[53]\AgdaBound{c≠d}\AgdaSpace{}%
\AgdaSymbol{=}\AgdaSpace{}%
\AgdaInductiveConstructor{no}\AgdaSpace{}%
\AgdaSymbol{(λ}\AgdaSpace{}%
\AgdaBound{e}\AgdaSpace{}%
\AgdaSymbol{→}\AgdaSpace{}%
\AgdaBound{c≠d}\AgdaSpace{}%
\AgdaSymbol{(}\AgdaFunction{cong}\AgdaSpace{}%
\AgdaFunction{rest}\AgdaSpace{}%
\AgdaBound{e}\AgdaSymbol{))}\<%
\\
\>[0]\AgdaFunction{TreeIsDiscrete}\AgdaSpace{}%
\AgdaSymbol{(}\AgdaOperator{\AgdaInductiveConstructor{ω\textasciicircum{}}}\AgdaSpace{}%
\AgdaBound{a}\AgdaSpace{}%
\AgdaOperator{\AgdaInductiveConstructor{+}}\AgdaSpace{}%
\AgdaBound{c}\AgdaSymbol{)}\AgdaSpace{}%
\AgdaSymbol{(}\AgdaOperator{\AgdaInductiveConstructor{ω\textasciicircum{}}}\AgdaSpace{}%
\AgdaBound{b}\AgdaSpace{}%
\AgdaOperator{\AgdaInductiveConstructor{+}}\AgdaSpace{}%
\AgdaBound{d}\AgdaSymbol{)}\AgdaSpace{}%
\AgdaSymbol{|}\AgdaSpace{}%
\AgdaInductiveConstructor{no}%
\>[43]\AgdaBound{a≠b}\AgdaSpace{}%
\AgdaSymbol{=}\AgdaSpace{}%
\AgdaInductiveConstructor{no}\AgdaSpace{}%
\AgdaSymbol{(λ}\AgdaSpace{}%
\AgdaBound{e}\AgdaSpace{}%
\AgdaSymbol{→}\AgdaSpace{}%
\AgdaBound{a≠b}\AgdaSpace{}%
\AgdaSymbol{(}\AgdaFunction{cong}\AgdaSpace{}%
\AgdaFunction{fst}\AgdaSpace{}%
\AgdaBound{e}\AgdaSymbol{))}\<%
\\
\\[\AgdaEmptyExtraSkip]%
\>[0]\AgdaKeyword{infix}%
\>[7]\AgdaNumber{30}\AgdaSpace{}%
\AgdaOperator{\AgdaDatatype{\AgdaUnderscore{}<\AgdaUnderscore{}}}\AgdaSpace{}%
\AgdaOperator{\AgdaFunction{\AgdaUnderscore{}≤\AgdaUnderscore{}}}\AgdaSpace{}%
\AgdaOperator{\AgdaFunction{\AgdaUnderscore{}>\AgdaUnderscore{}}}\AgdaSpace{}%
\AgdaOperator{\AgdaFunction{\AgdaUnderscore{}≥\AgdaUnderscore{}}}\<%
\\
\>[0]\<%
\end{code}

\newcommand{\TreeLt}{
\begin{code}%
\>[0]\AgdaKeyword{data}\AgdaSpace{}%
\AgdaOperator{\AgdaDatatype{\AgdaUnderscore{}<\AgdaUnderscore{}}}\AgdaSpace{}%
\AgdaSymbol{:}\AgdaSpace{}%
\AgdaDatatype{Tree}\AgdaSpace{}%
\AgdaSymbol{→}\AgdaSpace{}%
\AgdaDatatype{Tree}\AgdaSpace{}%
\AgdaSymbol{→}\AgdaSpace{}%
\AgdaFunction{Type₀}\AgdaSpace{}%
\AgdaKeyword{where}\<%
\\
\>[0][@{}l@{\AgdaIndent{0}}]%
\>[1]\AgdaInductiveConstructor{<₁}\AgdaSpace{}%
\AgdaSymbol{:}\AgdaSpace{}%
\AgdaInductiveConstructor{𝟎}\AgdaSpace{}%
\AgdaOperator{\AgdaDatatype{<}}\AgdaSpace{}%
\AgdaOperator{\AgdaInductiveConstructor{ω\textasciicircum{}}}\AgdaSpace{}%
\AgdaGeneralizable{a}\AgdaSpace{}%
\AgdaOperator{\AgdaInductiveConstructor{+}}\AgdaSpace{}%
\AgdaGeneralizable{b}\<%
\\
\>[1]\AgdaInductiveConstructor{<₂}\AgdaSpace{}%
\AgdaSymbol{:}\AgdaSpace{}%
\AgdaGeneralizable{a}\AgdaSpace{}%
\AgdaOperator{\AgdaDatatype{<}}\AgdaSpace{}%
\AgdaGeneralizable{c}\AgdaSpace{}%
\AgdaSymbol{→}\AgdaSpace{}%
\AgdaOperator{\AgdaInductiveConstructor{ω\textasciicircum{}}}\AgdaSpace{}%
\AgdaGeneralizable{a}\AgdaSpace{}%
\AgdaOperator{\AgdaInductiveConstructor{+}}\AgdaSpace{}%
\AgdaGeneralizable{b}\AgdaSpace{}%
\AgdaOperator{\AgdaDatatype{<}}\AgdaSpace{}%
\AgdaOperator{\AgdaInductiveConstructor{ω\textasciicircum{}}}\AgdaSpace{}%
\AgdaGeneralizable{c}\AgdaSpace{}%
\AgdaOperator{\AgdaInductiveConstructor{+}}\AgdaSpace{}%
\AgdaGeneralizable{d}\<%
\\
\>[1]\AgdaInductiveConstructor{<₃}\AgdaSpace{}%
\AgdaSymbol{:}\AgdaSpace{}%
\AgdaGeneralizable{a}\AgdaSpace{}%
\AgdaOperator{\AgdaFunction{≡}}\AgdaSpace{}%
\AgdaGeneralizable{c}\AgdaSpace{}%
\AgdaSymbol{→}\AgdaSpace{}%
\AgdaGeneralizable{b}\AgdaSpace{}%
\AgdaOperator{\AgdaDatatype{<}}\AgdaSpace{}%
\AgdaGeneralizable{d}\AgdaSpace{}%
\AgdaSymbol{→}\AgdaSpace{}%
\AgdaOperator{\AgdaInductiveConstructor{ω\textasciicircum{}}}\AgdaSpace{}%
\AgdaGeneralizable{a}\AgdaSpace{}%
\AgdaOperator{\AgdaInductiveConstructor{+}}\AgdaSpace{}%
\AgdaGeneralizable{b}\AgdaSpace{}%
\AgdaOperator{\AgdaDatatype{<}}\AgdaSpace{}%
\AgdaOperator{\AgdaInductiveConstructor{ω\textasciicircum{}}}\AgdaSpace{}%
\AgdaGeneralizable{c}\AgdaSpace{}%
\AgdaOperator{\AgdaInductiveConstructor{+}}\AgdaSpace{}%
\AgdaGeneralizable{d}\<%
\end{code}}

\newcommand{\TreeGeq}{
\begin{code}%
\>[0]\AgdaOperator{\AgdaFunction{\AgdaUnderscore{}>\AgdaUnderscore{}}}\AgdaSpace{}%
\AgdaOperator{\AgdaFunction{\AgdaUnderscore{}≥\AgdaUnderscore{}}}\AgdaSpace{}%
\AgdaSymbol{:}\AgdaSpace{}%
\AgdaDatatype{Tree}\AgdaSpace{}%
\AgdaSymbol{→}\AgdaSpace{}%
\AgdaDatatype{Tree}\AgdaSpace{}%
\AgdaSymbol{→}\AgdaSpace{}%
\AgdaFunction{Type₀}\<%
\\
\>[0]\AgdaBound{a}\AgdaSpace{}%
\AgdaOperator{\AgdaFunction{>}}\AgdaSpace{}%
\AgdaBound{b}\AgdaSpace{}%
\AgdaSymbol{=}\AgdaSpace{}%
\AgdaBound{b}\AgdaSpace{}%
\AgdaOperator{\AgdaDatatype{<}}\AgdaSpace{}%
\AgdaBound{a}\<%
\\
\>[0]\AgdaBound{a}\AgdaSpace{}%
\AgdaOperator{\AgdaFunction{≥}}\AgdaSpace{}%
\AgdaBound{b}\AgdaSpace{}%
\AgdaSymbol{=}\AgdaSpace{}%
\AgdaBound{a}\AgdaSpace{}%
\AgdaOperator{\AgdaFunction{>}}\AgdaSpace{}%
\AgdaBound{b}\AgdaSpace{}%
\AgdaOperator{\AgdaDatatype{⊎}}\AgdaSpace{}%
\AgdaBound{a}\AgdaSpace{}%
\AgdaOperator{\AgdaFunction{≡}}\AgdaSpace{}%
\AgdaBound{b}\<%
\end{code}}

\begin{code}[hide]%
\>[0]\<%
\\
\>[0]\AgdaOperator{\AgdaFunction{\AgdaUnderscore{}≤\AgdaUnderscore{}}}\AgdaSpace{}%
\AgdaSymbol{:}\AgdaSpace{}%
\AgdaDatatype{Tree}\AgdaSpace{}%
\AgdaSymbol{→}\AgdaSpace{}%
\AgdaDatatype{Tree}\AgdaSpace{}%
\AgdaSymbol{→}\AgdaSpace{}%
\AgdaFunction{Type₀}\<%
\\
\>[0]\AgdaBound{a}\AgdaSpace{}%
\AgdaOperator{\AgdaFunction{≤}}\AgdaSpace{}%
\AgdaBound{b}\AgdaSpace{}%
\AgdaSymbol{=}\AgdaSpace{}%
\AgdaBound{b}\AgdaSpace{}%
\AgdaOperator{\AgdaFunction{≥}}\AgdaSpace{}%
\AgdaBound{a}\<%
\\
\\[\AgdaEmptyExtraSkip]%
\>[0]\AgdaFunction{≥𝟎}\AgdaSpace{}%
\AgdaSymbol{:}\AgdaSpace{}%
\AgdaGeneralizable{a}\AgdaSpace{}%
\AgdaOperator{\AgdaFunction{≥}}\AgdaSpace{}%
\AgdaInductiveConstructor{𝟎}\<%
\\
\>[0]\AgdaFunction{≥𝟎}\AgdaSpace{}%
\AgdaSymbol{\{}\AgdaInductiveConstructor{𝟎}\AgdaSymbol{\}}%
\>[14]\AgdaSymbol{=}\AgdaSpace{}%
\AgdaInductiveConstructor{inj₂}\AgdaSpace{}%
\AgdaFunction{refl}\<%
\\
\>[0]\AgdaFunction{≥𝟎}\AgdaSpace{}%
\AgdaSymbol{\{}\AgdaOperator{\AgdaInductiveConstructor{ω\textasciicircum{}}}\AgdaSpace{}%
\AgdaBound{a}\AgdaSpace{}%
\AgdaOperator{\AgdaInductiveConstructor{+}}\AgdaSpace{}%
\AgdaBound{b}\AgdaSymbol{\}}\AgdaSpace{}%
\AgdaSymbol{=}\AgdaSpace{}%
\AgdaInductiveConstructor{inj₁}\AgdaSpace{}%
\AgdaInductiveConstructor{<₁}\<%
\\
\\[\AgdaEmptyExtraSkip]%
\>[0]\AgdaFunction{<-irrefl}\AgdaSpace{}%
\AgdaSymbol{:}\AgdaSpace{}%
\AgdaOperator{\AgdaFunction{¬}}\AgdaSpace{}%
\AgdaSymbol{(}\AgdaGeneralizable{a}\AgdaSpace{}%
\AgdaOperator{\AgdaDatatype{<}}\AgdaSpace{}%
\AgdaGeneralizable{a}\AgdaSymbol{)}\<%
\\
\>[0]\AgdaFunction{<-irrefl}\AgdaSpace{}%
\AgdaSymbol{(}\AgdaInductiveConstructor{<₂}\AgdaSpace{}%
\AgdaBound{r}\AgdaSymbol{)}%
\>[18]\AgdaSymbol{=}\AgdaSpace{}%
\AgdaFunction{<-irrefl}\AgdaSpace{}%
\AgdaBound{r}\<%
\\
\>[0]\AgdaFunction{<-irrefl}\AgdaSpace{}%
\AgdaSymbol{(}\AgdaInductiveConstructor{<₃}\AgdaSpace{}%
\AgdaSymbol{\AgdaUnderscore{}}\AgdaSpace{}%
\AgdaBound{r}\AgdaSymbol{)}\AgdaSpace{}%
\AgdaSymbol{=}\AgdaSpace{}%
\AgdaFunction{<-irrefl}\AgdaSpace{}%
\AgdaBound{r}\<%
\\
\\[\AgdaEmptyExtraSkip]%
\>[0]\AgdaFunction{<-irreflexive}\AgdaSpace{}%
\AgdaSymbol{:}\AgdaSpace{}%
\AgdaGeneralizable{a}\AgdaSpace{}%
\AgdaOperator{\AgdaFunction{≡}}\AgdaSpace{}%
\AgdaGeneralizable{b}\AgdaSpace{}%
\AgdaSymbol{→}\AgdaSpace{}%
\AgdaOperator{\AgdaFunction{¬}}\AgdaSpace{}%
\AgdaSymbol{(}\AgdaGeneralizable{a}\AgdaSpace{}%
\AgdaOperator{\AgdaDatatype{<}}\AgdaSpace{}%
\AgdaGeneralizable{b}\AgdaSymbol{)}\<%
\\
\>[0]\AgdaFunction{<-irreflexive}\AgdaSpace{}%
\AgdaSymbol{\{}\AgdaBound{a}\AgdaSymbol{\}}\AgdaSpace{}%
\AgdaSymbol{=}\AgdaSpace{}%
\AgdaFunction{J}\AgdaSpace{}%
\AgdaSymbol{(λ}\AgdaSpace{}%
\AgdaBound{b}\AgdaSpace{}%
\AgdaBound{\AgdaUnderscore{}}\AgdaSpace{}%
\AgdaSymbol{→}\AgdaSpace{}%
\AgdaOperator{\AgdaFunction{¬}}\AgdaSpace{}%
\AgdaSymbol{(}\AgdaBound{a}\AgdaSpace{}%
\AgdaOperator{\AgdaDatatype{<}}\AgdaSpace{}%
\AgdaBound{b}\AgdaSymbol{))}\AgdaSpace{}%
\AgdaFunction{<-irrefl}\<%
\\
\>[0]\<%
\end{code}

\newcommand{\Sfacts}{
\begin{code}%
\>[0]\AgdaFunction{TreeIsSet}\AgdaSpace{}%
\AgdaSymbol{:}\AgdaSpace{}%
\AgdaFunction{isSet}\AgdaSpace{}%
\AgdaDatatype{Tree}\<%
\\
\>[0]\AgdaFunction{<IsPropValued}\AgdaSpace{}%
\AgdaSymbol{:}\AgdaSpace{}%
\AgdaFunction{isProp}\AgdaSpace{}%
\AgdaSymbol{(}\AgdaGeneralizable{a}\AgdaSpace{}%
\AgdaOperator{\AgdaDatatype{<}}\AgdaSpace{}%
\AgdaGeneralizable{b}\AgdaSymbol{)}\<%
\end{code}}

\begin{code}[hide]%
\>[0]\<%
\\
\>[0]\AgdaFunction{TreeIsSet}\AgdaSpace{}%
\AgdaSymbol{=}\AgdaSpace{}%
\AgdaFunction{Discrete→isSet}\AgdaSpace{}%
\AgdaFunction{TreeIsDiscrete}\AgdaSpace{}%
\AgdaSymbol{\AgdaUnderscore{}}\AgdaSpace{}%
\AgdaSymbol{\AgdaUnderscore{}}\<%
\\
\\[\AgdaEmptyExtraSkip]%
\>[0]\AgdaFunction{<IsPropValued}%
\>[15]\AgdaInductiveConstructor{<₁}%
\>[24]\AgdaInductiveConstructor{<₁}%
\>[32]\AgdaSymbol{=}\AgdaSpace{}%
\AgdaFunction{refl}\<%
\\
\>[0]\AgdaFunction{<IsPropValued}\AgdaSpace{}%
\AgdaSymbol{(}\AgdaInductiveConstructor{<₂}\AgdaSpace{}%
\AgdaBound{r}\AgdaSymbol{)}%
\>[23]\AgdaSymbol{(}\AgdaInductiveConstructor{<₂}\AgdaSpace{}%
\AgdaBound{s}\AgdaSymbol{)}%
\>[32]\AgdaSymbol{=}\AgdaSpace{}%
\AgdaFunction{cong}\AgdaSpace{}%
\AgdaInductiveConstructor{<₂}\AgdaSpace{}%
\AgdaSymbol{(}\AgdaFunction{<IsPropValued}\AgdaSpace{}%
\AgdaBound{r}\AgdaSpace{}%
\AgdaBound{s}\AgdaSymbol{)}\<%
\\
\>[0]\AgdaFunction{<IsPropValued}\AgdaSpace{}%
\AgdaSymbol{(}\AgdaInductiveConstructor{<₂}\AgdaSpace{}%
\AgdaBound{r}\AgdaSymbol{)}%
\>[23]\AgdaSymbol{(}\AgdaInductiveConstructor{<₃}\AgdaSpace{}%
\AgdaBound{q}\AgdaSpace{}%
\AgdaBound{s}\AgdaSymbol{)}\AgdaSpace{}%
\AgdaSymbol{=}\AgdaSpace{}%
\AgdaFunction{⊥-elim}\AgdaSpace{}%
\AgdaSymbol{(}\AgdaFunction{<-irreflexive}\AgdaSpace{}%
\AgdaBound{q}\AgdaSpace{}%
\AgdaBound{r}\AgdaSymbol{)}\<%
\\
\>[0]\AgdaFunction{<IsPropValued}\AgdaSpace{}%
\AgdaSymbol{(}\AgdaInductiveConstructor{<₃}\AgdaSpace{}%
\AgdaBound{p}\AgdaSpace{}%
\AgdaBound{r}\AgdaSymbol{)}\AgdaSpace{}%
\AgdaSymbol{(}\AgdaInductiveConstructor{<₂}\AgdaSpace{}%
\AgdaBound{s}\AgdaSymbol{)}%
\>[32]\AgdaSymbol{=}\AgdaSpace{}%
\AgdaFunction{⊥-elim}\AgdaSpace{}%
\AgdaSymbol{(}\AgdaFunction{<-irreflexive}\AgdaSpace{}%
\AgdaBound{p}\AgdaSpace{}%
\AgdaBound{s}\AgdaSymbol{)}\<%
\\
\>[0]\AgdaFunction{<IsPropValued}\AgdaSpace{}%
\AgdaSymbol{(}\AgdaInductiveConstructor{<₃}\AgdaSpace{}%
\AgdaBound{p}\AgdaSpace{}%
\AgdaBound{r}\AgdaSymbol{)}\AgdaSpace{}%
\AgdaSymbol{(}\AgdaInductiveConstructor{<₃}\AgdaSpace{}%
\AgdaBound{q}\AgdaSpace{}%
\AgdaBound{s}\AgdaSymbol{)}\AgdaSpace{}%
\AgdaSymbol{=}\AgdaSpace{}%
\AgdaFunction{cong₂}\AgdaSpace{}%
\AgdaInductiveConstructor{<₃}\AgdaSpace{}%
\AgdaSymbol{(}\AgdaFunction{TreeIsSet}\AgdaSpace{}%
\AgdaBound{p}\AgdaSpace{}%
\AgdaBound{q}\AgdaSymbol{)}\AgdaSpace{}%
\AgdaSymbol{(}\AgdaFunction{<IsPropValued}\AgdaSpace{}%
\AgdaBound{r}\AgdaSpace{}%
\AgdaBound{s}\AgdaSymbol{)}\<%
\\
\\[\AgdaEmptyExtraSkip]%
\>[0]\AgdaFunction{≤IsPropValued}\AgdaSpace{}%
\AgdaSymbol{:}\AgdaSpace{}%
\AgdaFunction{isProp}\AgdaSpace{}%
\AgdaSymbol{(}\AgdaGeneralizable{a}\AgdaSpace{}%
\AgdaOperator{\AgdaFunction{≤}}\AgdaSpace{}%
\AgdaGeneralizable{b}\AgdaSymbol{)}\<%
\\
\>[0]\AgdaFunction{≤IsPropValued}\AgdaSpace{}%
\AgdaSymbol{(}\AgdaInductiveConstructor{inj₁}\AgdaSpace{}%
\AgdaBound{r}\AgdaSymbol{)}\AgdaSpace{}%
\AgdaSymbol{(}\AgdaInductiveConstructor{inj₁}\AgdaSpace{}%
\AgdaBound{s}\AgdaSymbol{)}\AgdaSpace{}%
\AgdaSymbol{=}\AgdaSpace{}%
\AgdaFunction{cong}\AgdaSpace{}%
\AgdaInductiveConstructor{inj₁}\AgdaSpace{}%
\AgdaSymbol{(}\AgdaFunction{<IsPropValued}\AgdaSpace{}%
\AgdaBound{r}\AgdaSpace{}%
\AgdaBound{s}\AgdaSymbol{)}\<%
\\
\>[0]\AgdaFunction{≤IsPropValued}\AgdaSpace{}%
\AgdaSymbol{(}\AgdaInductiveConstructor{inj₁}\AgdaSpace{}%
\AgdaBound{r}\AgdaSymbol{)}\AgdaSpace{}%
\AgdaSymbol{(}\AgdaInductiveConstructor{inj₂}\AgdaSpace{}%
\AgdaBound{q}\AgdaSymbol{)}\AgdaSpace{}%
\AgdaSymbol{=}\AgdaSpace{}%
\AgdaFunction{⊥-elim}\AgdaSpace{}%
\AgdaSymbol{(}\AgdaFunction{<-irreflexive}\AgdaSpace{}%
\AgdaSymbol{(}\AgdaBound{q}\AgdaSpace{}%
\AgdaOperator{\AgdaFunction{⁻¹}}\AgdaSymbol{)}\AgdaSpace{}%
\AgdaBound{r}\AgdaSymbol{)}\<%
\\
\>[0]\AgdaFunction{≤IsPropValued}\AgdaSpace{}%
\AgdaSymbol{(}\AgdaInductiveConstructor{inj₂}\AgdaSpace{}%
\AgdaBound{p}\AgdaSymbol{)}\AgdaSpace{}%
\AgdaSymbol{(}\AgdaInductiveConstructor{inj₁}\AgdaSpace{}%
\AgdaBound{s}\AgdaSymbol{)}\AgdaSpace{}%
\AgdaSymbol{=}\AgdaSpace{}%
\AgdaFunction{⊥-elim}\AgdaSpace{}%
\AgdaSymbol{(}\AgdaFunction{<-irreflexive}\AgdaSpace{}%
\AgdaSymbol{(}\AgdaBound{p}\AgdaSpace{}%
\AgdaOperator{\AgdaFunction{⁻¹}}\AgdaSymbol{)}\AgdaSpace{}%
\AgdaBound{s}\AgdaSymbol{)}\<%
\\
\>[0]\AgdaFunction{≤IsPropValued}\AgdaSpace{}%
\AgdaSymbol{(}\AgdaInductiveConstructor{inj₂}\AgdaSpace{}%
\AgdaBound{p}\AgdaSymbol{)}\AgdaSpace{}%
\AgdaSymbol{(}\AgdaInductiveConstructor{inj₂}\AgdaSpace{}%
\AgdaBound{q}\AgdaSymbol{)}\AgdaSpace{}%
\AgdaSymbol{=}\AgdaSpace{}%
\AgdaFunction{cong}\AgdaSpace{}%
\AgdaInductiveConstructor{inj₂}\AgdaSpace{}%
\AgdaSymbol{(}\AgdaFunction{TreeIsSet}\AgdaSpace{}%
\AgdaBound{p}\AgdaSpace{}%
\AgdaBound{q}\AgdaSymbol{)}\<%
\\
\>[0]\<%
\end{code}

\newcommand{\isCNF}{
\begin{code}%
\>[0]\AgdaKeyword{data}\AgdaSpace{}%
\AgdaDatatype{isCNF}\AgdaSpace{}%
\AgdaSymbol{:}\AgdaSpace{}%
\AgdaDatatype{Tree}\AgdaSpace{}%
\AgdaSymbol{→}\AgdaSpace{}%
\AgdaFunction{Type₀}\AgdaSpace{}%
\AgdaKeyword{where}\<%
\\
\>[0][@{}l@{\AgdaIndent{0}}]%
\>[1]\AgdaInductiveConstructor{𝟎IsCNF}\AgdaSpace{}%
\AgdaSymbol{:}\AgdaSpace{}%
\AgdaDatatype{isCNF}\AgdaSpace{}%
\AgdaInductiveConstructor{𝟎}\<%
\\
\>[1]\AgdaInductiveConstructor{ω\textasciicircum{}+IsCNF}%
\>[488I]\AgdaSymbol{:}\AgdaSpace{}%
\AgdaDatatype{isCNF}\AgdaSpace{}%
\AgdaGeneralizable{a}\AgdaSpace{}%
\AgdaSymbol{→}\AgdaSpace{}%
\AgdaDatatype{isCNF}\AgdaSpace{}%
\AgdaGeneralizable{b}\AgdaSpace{}%
\AgdaSymbol{→}\AgdaSpace{}%
\AgdaGeneralizable{a}\AgdaSpace{}%
\AgdaOperator{\AgdaFunction{≥}}\AgdaSpace{}%
\AgdaFunction{fst}\AgdaSpace{}%
\AgdaGeneralizable{b}\<%
\\
\>[.][@{}l@{}]\<[488I]%
\>[10]\AgdaSymbol{→}\AgdaSpace{}%
\AgdaDatatype{isCNF}\AgdaSpace{}%
\AgdaSymbol{(}\AgdaOperator{\AgdaInductiveConstructor{ω\textasciicircum{}}}\AgdaSpace{}%
\AgdaGeneralizable{a}\AgdaSpace{}%
\AgdaOperator{\AgdaInductiveConstructor{+}}\AgdaSpace{}%
\AgdaGeneralizable{b}\AgdaSymbol{)}\<%
\end{code}}

\newcommand{\isCNFIsPropValued}{
\begin{code}%
\>[0]\AgdaFunction{isCNFIsPropValued}\AgdaSpace{}%
\AgdaSymbol{:}\AgdaSpace{}%
\AgdaFunction{isProp}\AgdaSpace{}%
\AgdaSymbol{(}\AgdaDatatype{isCNF}\AgdaSpace{}%
\AgdaGeneralizable{a}\AgdaSymbol{)}\<%
\end{code}}

\begin{code}[hide]%
\>[0]\<%
\\
\>[0]\AgdaFunction{isCNFIsPropValued}%
\>[19]\AgdaInductiveConstructor{𝟎IsCNF}%
\>[39]\AgdaInductiveConstructor{𝟎IsCNF}%
\>[58]\AgdaSymbol{=}\AgdaSpace{}%
\AgdaFunction{refl}\<%
\\
\>[0]\AgdaFunction{isCNFIsPropValued}\AgdaSpace{}%
\AgdaSymbol{(}\AgdaInductiveConstructor{ω\textasciicircum{}+IsCNF}\AgdaSpace{}%
\AgdaBound{pa}\AgdaSpace{}%
\AgdaBound{pb}\AgdaSpace{}%
\AgdaBound{ra}\AgdaSymbol{)}\AgdaSpace{}%
\AgdaSymbol{(}\AgdaInductiveConstructor{ω\textasciicircum{}+IsCNF}\AgdaSpace{}%
\AgdaBound{qa}\AgdaSpace{}%
\AgdaBound{qb}\AgdaSpace{}%
\AgdaBound{rb}\AgdaSymbol{)}\AgdaSpace{}%
\AgdaSymbol{=}\<%
\\
\>[0][@{}l@{\AgdaIndent{0}}]%
\>[2]\AgdaSymbol{λ}\AgdaSpace{}%
\AgdaBound{i}\AgdaSpace{}%
\AgdaSymbol{→}\AgdaSpace{}%
\AgdaInductiveConstructor{ω\textasciicircum{}+IsCNF}\AgdaSpace{}%
\AgdaSymbol{(}\AgdaFunction{isCNFIsPropValued}\AgdaSpace{}%
\AgdaBound{pa}\AgdaSpace{}%
\AgdaBound{qa}\AgdaSpace{}%
\AgdaBound{i}\AgdaSymbol{)}\AgdaSpace{}%
\AgdaSymbol{(}\AgdaFunction{isCNFIsPropValued}\AgdaSpace{}%
\AgdaBound{pb}\AgdaSpace{}%
\AgdaBound{qb}\AgdaSpace{}%
\AgdaBound{i}\AgdaSymbol{)}\AgdaSpace{}%
\AgdaSymbol{(}\AgdaFunction{≤IsPropValued}\AgdaSpace{}%
\AgdaBound{ra}\AgdaSpace{}%
\AgdaBound{rb}\AgdaSpace{}%
\AgdaBound{i}\AgdaSymbol{)}\<%
\\
\>[0]\<%
\end{code}

\newcommand{\SigmaOrd}{
\begin{code}%
\>[0]\AgdaFunction{SigmaOrd}\AgdaSpace{}%
\AgdaSymbol{:}\AgdaSpace{}%
\AgdaFunction{Type₀}\<%
\\
\>[0]\AgdaFunction{SigmaOrd}\AgdaSpace{}%
\AgdaSymbol{=}\AgdaSpace{}%
\AgdaRecord{Σ}\AgdaSpace{}%
\AgdaSymbol{\textbackslash{}(}\AgdaBound{a}\AgdaSpace{}%
\AgdaSymbol{:}\AgdaSpace{}%
\AgdaDatatype{Tree}\AgdaSymbol{)}\AgdaSpace{}%
\AgdaSymbol{→}\AgdaSpace{}%
\AgdaDatatype{isCNF}\AgdaSpace{}%
\AgdaBound{a}\<%
\end{code}}

\newcommand{\SigmaOrdEq}{
\begin{code}%
\>[0]\AgdaFunction{SigmaOrd⁼}\AgdaSpace{}%
\AgdaSymbol{:}\AgdaSpace{}%
\AgdaSymbol{\{}\AgdaBound{x}\AgdaSpace{}%
\AgdaBound{y}\AgdaSpace{}%
\AgdaSymbol{:}\AgdaSpace{}%
\AgdaFunction{SigmaOrd}\AgdaSymbol{\}}\AgdaSpace{}%
\AgdaSymbol{→}\AgdaSpace{}%
\AgdaField{pr₁}\AgdaSpace{}%
\AgdaBound{x}\AgdaSpace{}%
\AgdaOperator{\AgdaFunction{≡}}\AgdaSpace{}%
\AgdaField{pr₁}\AgdaSpace{}%
\AgdaBound{y}\AgdaSpace{}%
\AgdaSymbol{→}\AgdaSpace{}%
\AgdaBound{x}\AgdaSpace{}%
\AgdaOperator{\AgdaFunction{≡}}\AgdaSpace{}%
\AgdaBound{y}\<%
\end{code}}

\begin{code}[hide]%
\>[0]\<%
\\
\>[0]\AgdaFunction{SigmaOrd⁼}\AgdaSpace{}%
\AgdaSymbol{\{}\AgdaBound{a}\AgdaSpace{}%
\AgdaOperator{\AgdaInductiveConstructor{,}}\AgdaSpace{}%
\AgdaBound{p}\AgdaSymbol{\}}\AgdaSpace{}%
\AgdaBound{e}\AgdaSpace{}%
\AgdaSymbol{=}\AgdaSpace{}%
\AgdaFunction{J}\AgdaSpace{}%
\AgdaFunction{P}\AgdaSpace{}%
\AgdaFunction{pr}\AgdaSpace{}%
\AgdaBound{e}\<%
\\
\>[0][@{}l@{\AgdaIndent{0}}]%
\>[1]\AgdaKeyword{where}\<%
\\
\>[1][@{}l@{\AgdaIndent{0}}]%
\>[2]\AgdaFunction{P}\AgdaSpace{}%
\AgdaSymbol{:}\AgdaSpace{}%
\AgdaSymbol{(}\AgdaBound{b}\AgdaSpace{}%
\AgdaSymbol{:}\AgdaSpace{}%
\AgdaDatatype{Tree}\AgdaSymbol{)}\AgdaSpace{}%
\AgdaSymbol{→}\AgdaSpace{}%
\AgdaBound{a}\AgdaSpace{}%
\AgdaOperator{\AgdaFunction{≡}}\AgdaSpace{}%
\AgdaBound{b}\AgdaSpace{}%
\AgdaSymbol{→}\AgdaSpace{}%
\AgdaFunction{Type₀}\<%
\\
\>[2]\AgdaFunction{P}\AgdaSpace{}%
\AgdaBound{b}\AgdaSpace{}%
\AgdaSymbol{\AgdaUnderscore{}}\AgdaSpace{}%
\AgdaSymbol{=}\AgdaSpace{}%
\AgdaSymbol{\{}\AgdaBound{p}\AgdaSpace{}%
\AgdaSymbol{:}\AgdaSpace{}%
\AgdaDatatype{isCNF}\AgdaSpace{}%
\AgdaBound{a}\AgdaSymbol{\}}\AgdaSpace{}%
\AgdaSymbol{\{}\AgdaBound{q}\AgdaSpace{}%
\AgdaSymbol{:}\AgdaSpace{}%
\AgdaDatatype{isCNF}\AgdaSpace{}%
\AgdaBound{b}\AgdaSymbol{\}}\AgdaSpace{}%
\AgdaSymbol{→}\AgdaSpace{}%
\AgdaSymbol{(}\AgdaBound{a}\AgdaSpace{}%
\AgdaOperator{\AgdaInductiveConstructor{,}}\AgdaSpace{}%
\AgdaBound{p}\AgdaSymbol{)}\AgdaSpace{}%
\AgdaOperator{\AgdaFunction{≡}}\AgdaSpace{}%
\AgdaSymbol{(}\AgdaBound{b}\AgdaSpace{}%
\AgdaOperator{\AgdaInductiveConstructor{,}}\AgdaSpace{}%
\AgdaBound{q}\AgdaSymbol{)}\<%
\\
\>[2]\AgdaFunction{pr}\AgdaSpace{}%
\AgdaSymbol{:}\AgdaSpace{}%
\AgdaFunction{P}\AgdaSpace{}%
\AgdaBound{a}\AgdaSpace{}%
\AgdaFunction{refl}\<%
\\
\>[2]\AgdaFunction{pr}\AgdaSpace{}%
\AgdaSymbol{\{}\AgdaBound{p}\AgdaSymbol{\}}\AgdaSpace{}%
\AgdaSymbol{\{}\AgdaBound{q}\AgdaSymbol{\}}\AgdaSpace{}%
\AgdaBound{i}\AgdaSpace{}%
\AgdaSymbol{=}\AgdaSpace{}%
\AgdaBound{a}\AgdaSpace{}%
\AgdaOperator{\AgdaInductiveConstructor{,}}\AgdaSpace{}%
\AgdaFunction{isCNFIsPropValued}\AgdaSpace{}%
\AgdaBound{p}\AgdaSpace{}%
\AgdaBound{q}\AgdaSpace{}%
\AgdaBound{i}\<%
\\
\>[0]\<%
\end{code}

\newcommand{\HO}{\AgdaDatatype{HIT\-Ord}\xspace}

\begin{code}[hide]%
\>[0]\<%
\\
\>[0]\AgdaSymbol{\{-\#}\AgdaSpace{}%
\AgdaKeyword{OPTIONS}\AgdaSpace{}%
\AgdaPragma{--cubical}\AgdaSpace{}%
\AgdaPragma{--safe}\AgdaSpace{}%
\AgdaSymbol{\#-\}}\<%
\\
\\[\AgdaEmptyExtraSkip]%
\>[0]\AgdaKeyword{module}\AgdaSpace{}%
\AgdaModule{HITOrd}\AgdaSpace{}%
\AgdaKeyword{where}\<%
\\
\\[\AgdaEmptyExtraSkip]%
\>[0]\AgdaKeyword{open}\AgdaSpace{}%
\AgdaKeyword{import}\AgdaSpace{}%
\AgdaModule{Cubical.Foundations.Everything}\<%
\\
\>[0][@{}l@{\AgdaIndent{0}}]%
\>[1]\AgdaKeyword{using}%
\>[8I]\AgdaSymbol{(}\AgdaFunction{Type}\AgdaSpace{}%
\AgdaSymbol{;}\AgdaSpace{}%
\AgdaFunction{Type₀}\AgdaSpace{}%
\AgdaSymbol{;}\<%
\\
\>[8I][@{}l@{\AgdaIndent{0}}]%
\>[8]\AgdaOperator{\AgdaFunction{\AgdaUnderscore{}≡\AgdaUnderscore{}}}\AgdaSpace{}%
\AgdaSymbol{;}\AgdaSpace{}%
\AgdaFunction{refl}\AgdaSpace{}%
\AgdaSymbol{;}\AgdaSpace{}%
\AgdaFunction{cong}\AgdaSpace{}%
\AgdaSymbol{;}\AgdaSpace{}%
\AgdaFunction{cong₂}\AgdaSpace{}%
\AgdaSymbol{;}\AgdaSpace{}%
\AgdaOperator{\AgdaFunction{\AgdaUnderscore{}∙\AgdaUnderscore{}}}\AgdaSpace{}%
\AgdaSymbol{;}\AgdaSpace{}%
\AgdaOperator{\AgdaFunction{\AgdaUnderscore{}⁻¹}}\AgdaSpace{}%
\AgdaSymbol{;}\<%
\\
\>[8]\AgdaPostulate{PathP}\AgdaSpace{}%
\AgdaSymbol{;}\AgdaSpace{}%
\AgdaFunction{isOfHLevel→isOfHLevelDep}\AgdaSpace{}%
\AgdaSymbol{;}\AgdaSpace{}%
\AgdaFunction{isProp→isSet}\AgdaSpace{}%
\AgdaSymbol{;}\AgdaSpace{}%
\AgdaFunction{toPathP}\AgdaSymbol{)}\<%
\\
\\[\AgdaEmptyExtraSkip]%
\>[0]\AgdaKeyword{open}\AgdaSpace{}%
\AgdaKeyword{import}\AgdaSpace{}%
\AgdaModule{Cubical.Data.Empty}\<%
\\
\\[\AgdaEmptyExtraSkip]%
\>[0]\AgdaKeyword{open}\AgdaSpace{}%
\AgdaKeyword{import}\AgdaSpace{}%
\AgdaModule{Preliminaries}\<%
\\
\\[\AgdaEmptyExtraSkip]%
\>[0]\AgdaKeyword{infixr}\AgdaSpace{}%
\AgdaNumber{40}\AgdaSpace{}%
\AgdaOperator{\AgdaInductiveConstructor{ω\textasciicircum{}\AgdaUnderscore{}⊕\AgdaUnderscore{}}}\<%
\\
\>[0]\<%
\end{code}

\newcommand{\HITOrd}{
\begin{code}%
\>[0]\AgdaKeyword{data}\AgdaSpace{}%
\AgdaDatatype{HITOrd}\AgdaSpace{}%
\AgdaSymbol{:}\AgdaSpace{}%
\AgdaFunction{Type₀}\AgdaSpace{}%
\AgdaKeyword{where}\<%
\\
\>[0][@{}l@{\AgdaIndent{0}}]%
\>[1]\AgdaInductiveConstructor{𝟎}\AgdaSpace{}%
\AgdaSymbol{:}\AgdaSpace{}%
\AgdaDatatype{HITOrd}\<%
\\
\>[1]\AgdaOperator{\AgdaInductiveConstructor{ω\textasciicircum{}\AgdaUnderscore{}⊕\AgdaUnderscore{}}}\AgdaSpace{}%
\AgdaSymbol{:}\AgdaSpace{}%
\AgdaDatatype{HITOrd}\AgdaSpace{}%
\AgdaSymbol{→}\AgdaSpace{}%
\AgdaDatatype{HITOrd}\AgdaSpace{}%
\AgdaSymbol{→}\AgdaSpace{}%
\AgdaDatatype{HITOrd}\<%
\\
\>[1]\AgdaInductiveConstructor{swap}\AgdaSpace{}%
\AgdaSymbol{:}\AgdaSpace{}%
\AgdaSymbol{∀}\AgdaSpace{}%
\AgdaBound{a}\AgdaSpace{}%
\AgdaBound{b}\AgdaSpace{}%
\AgdaBound{c}\AgdaSpace{}%
\AgdaSymbol{→}\AgdaSpace{}%
\AgdaOperator{\AgdaInductiveConstructor{ω\textasciicircum{}}}\AgdaSpace{}%
\AgdaBound{a}\AgdaSpace{}%
\AgdaOperator{\AgdaInductiveConstructor{⊕}}\AgdaSpace{}%
\AgdaOperator{\AgdaInductiveConstructor{ω\textasciicircum{}}}\AgdaSpace{}%
\AgdaBound{b}\AgdaSpace{}%
\AgdaOperator{\AgdaInductiveConstructor{⊕}}\AgdaSpace{}%
\AgdaBound{c}\AgdaSpace{}%
\AgdaOperator{\AgdaFunction{≡}}\AgdaSpace{}%
\AgdaOperator{\AgdaInductiveConstructor{ω\textasciicircum{}}}\AgdaSpace{}%
\AgdaBound{b}\AgdaSpace{}%
\AgdaOperator{\AgdaInductiveConstructor{⊕}}\AgdaSpace{}%
\AgdaOperator{\AgdaInductiveConstructor{ω\textasciicircum{}}}\AgdaSpace{}%
\AgdaBound{a}\AgdaSpace{}%
\AgdaOperator{\AgdaInductiveConstructor{⊕}}\AgdaSpace{}%
\AgdaBound{c}\<%
\\
\>[1]\AgdaInductiveConstructor{trunc}\AgdaSpace{}%
\AgdaSymbol{:}\AgdaSpace{}%
\AgdaFunction{isSet}\AgdaSpace{}%
\AgdaDatatype{HITOrd}\<%
\end{code}}

\begin{code}[hide]%
\>[0]\<%
\\
\>[0]\AgdaKeyword{private}\<%
\\
\>[0][@{}l@{\AgdaIndent{0}}]%
\>[1]\AgdaKeyword{variable}\AgdaSpace{}%
\AgdaGeneralizable{a}\AgdaSpace{}%
\AgdaGeneralizable{b}\AgdaSpace{}%
\AgdaGeneralizable{c}\AgdaSpace{}%
\AgdaSymbol{:}\AgdaSpace{}%
\AgdaDatatype{HITOrd}\<%
\\
\>[0]\<%
\end{code}

\newcommand{\Hexample}{
\begin{code}%
\>[0]\AgdaFunction{example}\AgdaSpace{}%
\AgdaSymbol{:}\AgdaSpace{}%
\AgdaSymbol{(}\AgdaBound{a}\AgdaSpace{}%
\AgdaBound{b}\AgdaSpace{}%
\AgdaBound{c}\AgdaSpace{}%
\AgdaSymbol{:}\AgdaSpace{}%
\AgdaDatatype{HITOrd}\AgdaSymbol{)}\<%
\\
\>[0][@{}l@{\AgdaIndent{0}}]%
\>[2]\AgdaSymbol{→}\AgdaSpace{}%
\AgdaOperator{\AgdaInductiveConstructor{ω\textasciicircum{}}}\AgdaSpace{}%
\AgdaBound{a}\AgdaSpace{}%
\AgdaOperator{\AgdaInductiveConstructor{⊕}}\AgdaSpace{}%
\AgdaOperator{\AgdaInductiveConstructor{ω\textasciicircum{}}}\AgdaSpace{}%
\AgdaBound{b}\AgdaSpace{}%
\AgdaOperator{\AgdaInductiveConstructor{⊕}}\AgdaSpace{}%
\AgdaOperator{\AgdaInductiveConstructor{ω\textasciicircum{}}}\AgdaSpace{}%
\AgdaBound{c}\AgdaSpace{}%
\AgdaOperator{\AgdaInductiveConstructor{⊕}}\AgdaSpace{}%
\AgdaInductiveConstructor{𝟎}\AgdaSpace{}%
\AgdaOperator{\AgdaFunction{≡}}\AgdaSpace{}%
\AgdaOperator{\AgdaInductiveConstructor{ω\textasciicircum{}}}\AgdaSpace{}%
\AgdaBound{c}\AgdaSpace{}%
\AgdaOperator{\AgdaInductiveConstructor{⊕}}\AgdaSpace{}%
\AgdaOperator{\AgdaInductiveConstructor{ω\textasciicircum{}}}\AgdaSpace{}%
\AgdaBound{b}\AgdaSpace{}%
\AgdaOperator{\AgdaInductiveConstructor{⊕}}\AgdaSpace{}%
\AgdaOperator{\AgdaInductiveConstructor{ω\textasciicircum{}}}\AgdaSpace{}%
\AgdaBound{a}\AgdaSpace{}%
\AgdaOperator{\AgdaInductiveConstructor{⊕}}\AgdaSpace{}%
\AgdaInductiveConstructor{𝟎}\<%
\\
\>[0]\AgdaFunction{example}\AgdaSpace{}%
\AgdaBound{a}\AgdaSpace{}%
\AgdaBound{b}\AgdaSpace{}%
\AgdaBound{c}\AgdaSpace{}%
\AgdaSymbol{=}\AgdaSpace{}%
\AgdaOperator{\AgdaFunction{begin}}\<%
\\
\>[0][@{}l@{\AgdaIndent{0}}]%
\>[1]\AgdaOperator{\AgdaInductiveConstructor{ω\textasciicircum{}}}\AgdaSpace{}%
\AgdaBound{a}\AgdaSpace{}%
\AgdaOperator{\AgdaInductiveConstructor{⊕}}\AgdaSpace{}%
\AgdaOperator{\AgdaInductiveConstructor{ω\textasciicircum{}}}\AgdaSpace{}%
\AgdaBound{b}\AgdaSpace{}%
\AgdaOperator{\AgdaInductiveConstructor{⊕}}\AgdaSpace{}%
\AgdaOperator{\AgdaInductiveConstructor{ω\textasciicircum{}}}\AgdaSpace{}%
\AgdaBound{c}\AgdaSpace{}%
\AgdaOperator{\AgdaInductiveConstructor{⊕}}\AgdaSpace{}%
\AgdaInductiveConstructor{𝟎}\AgdaSpace{}%
\AgdaOperator{\AgdaFunction{≡⟨}}\AgdaSpace{}%
\AgdaInductiveConstructor{swap}\AgdaSpace{}%
\AgdaBound{a}\AgdaSpace{}%
\AgdaBound{b}\AgdaSpace{}%
\AgdaSymbol{\AgdaUnderscore{}}\AgdaSpace{}%
\AgdaOperator{\AgdaFunction{⟩}}\<%
\\
\>[1]\AgdaOperator{\AgdaInductiveConstructor{ω\textasciicircum{}}}\AgdaSpace{}%
\AgdaBound{b}\AgdaSpace{}%
\AgdaOperator{\AgdaInductiveConstructor{⊕}}\AgdaSpace{}%
\AgdaOperator{\AgdaInductiveConstructor{ω\textasciicircum{}}}\AgdaSpace{}%
\AgdaBound{a}\AgdaSpace{}%
\AgdaOperator{\AgdaInductiveConstructor{⊕}}\AgdaSpace{}%
\AgdaOperator{\AgdaInductiveConstructor{ω\textasciicircum{}}}\AgdaSpace{}%
\AgdaBound{c}\AgdaSpace{}%
\AgdaOperator{\AgdaInductiveConstructor{⊕}}\AgdaSpace{}%
\AgdaInductiveConstructor{𝟎}\AgdaSpace{}%
\AgdaOperator{\AgdaFunction{≡⟨}}\AgdaSpace{}%
\AgdaFunction{cong}\AgdaSpace{}%
\AgdaSymbol{(}\AgdaOperator{\AgdaInductiveConstructor{ω\textasciicircum{}}}\AgdaSpace{}%
\AgdaBound{b}\AgdaSpace{}%
\AgdaOperator{\AgdaInductiveConstructor{⊕\AgdaUnderscore{}}}\AgdaSymbol{)}\AgdaSpace{}%
\AgdaSymbol{(}\AgdaInductiveConstructor{swap}\AgdaSpace{}%
\AgdaBound{a}\AgdaSpace{}%
\AgdaBound{c}\AgdaSpace{}%
\AgdaSymbol{\AgdaUnderscore{})}\AgdaSpace{}%
\AgdaOperator{\AgdaFunction{⟩}}\<%
\\
\>[1]\AgdaOperator{\AgdaInductiveConstructor{ω\textasciicircum{}}}\AgdaSpace{}%
\AgdaBound{b}\AgdaSpace{}%
\AgdaOperator{\AgdaInductiveConstructor{⊕}}\AgdaSpace{}%
\AgdaOperator{\AgdaInductiveConstructor{ω\textasciicircum{}}}\AgdaSpace{}%
\AgdaBound{c}\AgdaSpace{}%
\AgdaOperator{\AgdaInductiveConstructor{⊕}}\AgdaSpace{}%
\AgdaOperator{\AgdaInductiveConstructor{ω\textasciicircum{}}}\AgdaSpace{}%
\AgdaBound{a}\AgdaSpace{}%
\AgdaOperator{\AgdaInductiveConstructor{⊕}}\AgdaSpace{}%
\AgdaInductiveConstructor{𝟎}\AgdaSpace{}%
\AgdaOperator{\AgdaFunction{≡⟨}}\AgdaSpace{}%
\AgdaInductiveConstructor{swap}\AgdaSpace{}%
\AgdaBound{b}\AgdaSpace{}%
\AgdaBound{c}\AgdaSpace{}%
\AgdaSymbol{\AgdaUnderscore{}}\AgdaSpace{}%
\AgdaOperator{\AgdaFunction{⟩}}\<%
\\
\>[1]\AgdaOperator{\AgdaInductiveConstructor{ω\textasciicircum{}}}\AgdaSpace{}%
\AgdaBound{c}\AgdaSpace{}%
\AgdaOperator{\AgdaInductiveConstructor{⊕}}\AgdaSpace{}%
\AgdaOperator{\AgdaInductiveConstructor{ω\textasciicircum{}}}\AgdaSpace{}%
\AgdaBound{b}\AgdaSpace{}%
\AgdaOperator{\AgdaInductiveConstructor{⊕}}\AgdaSpace{}%
\AgdaOperator{\AgdaInductiveConstructor{ω\textasciicircum{}}}\AgdaSpace{}%
\AgdaBound{a}\AgdaSpace{}%
\AgdaOperator{\AgdaInductiveConstructor{⊕}}\AgdaSpace{}%
\AgdaInductiveConstructor{𝟎}\AgdaSpace{}%
\AgdaOperator{\AgdaFunction{∎}}\<%
\end{code}}

\begin{code}[hide]%
\>[0]\AgdaFunction{ind}\AgdaSpace{}%
\AgdaBound{A}\AgdaSpace{}%
\AgdaBound{a₀}\AgdaSpace{}%
\AgdaOperator{\AgdaBound{\AgdaUnderscore{}⋆\AgdaUnderscore{}}}\AgdaSpace{}%
\AgdaBound{⋆swap}\AgdaSpace{}%
\AgdaBound{sv}\AgdaSpace{}%
\AgdaInductiveConstructor{𝟎}\AgdaSpace{}%
\AgdaSymbol{=}\AgdaSpace{}%
\AgdaBound{a₀}\<%
\\
\>[0]\AgdaFunction{ind}\AgdaSpace{}%
\AgdaBound{A}\AgdaSpace{}%
\AgdaBound{a₀}\AgdaSpace{}%
\AgdaOperator{\AgdaBound{\AgdaUnderscore{}⋆\AgdaUnderscore{}}}\AgdaSpace{}%
\AgdaBound{⋆swap}\AgdaSpace{}%
\AgdaBound{sv}\AgdaSpace{}%
\AgdaSymbol{(}\AgdaOperator{\AgdaInductiveConstructor{ω\textasciicircum{}}}\AgdaSpace{}%
\AgdaBound{x}\AgdaSpace{}%
\AgdaOperator{\AgdaInductiveConstructor{⊕}}\AgdaSpace{}%
\AgdaBound{y}\AgdaSymbol{)}\AgdaSpace{}%
\AgdaSymbol{=}\AgdaSpace{}%
\AgdaFunction{ind}\AgdaSpace{}%
\AgdaBound{A}\AgdaSpace{}%
\AgdaBound{a₀}\AgdaSpace{}%
\AgdaOperator{\AgdaBound{\AgdaUnderscore{}⋆\AgdaUnderscore{}}}\AgdaSpace{}%
\AgdaBound{⋆swap}\AgdaSpace{}%
\AgdaBound{sv}\AgdaSpace{}%
\AgdaBound{x}\AgdaSpace{}%
\AgdaOperator{\AgdaBound{⋆}}\AgdaSpace{}%
\AgdaFunction{ind}\AgdaSpace{}%
\AgdaBound{A}\AgdaSpace{}%
\AgdaBound{a₀}\AgdaSpace{}%
\AgdaOperator{\AgdaBound{\AgdaUnderscore{}⋆\AgdaUnderscore{}}}\AgdaSpace{}%
\AgdaBound{⋆swap}\AgdaSpace{}%
\AgdaBound{sv}\AgdaSpace{}%
\AgdaBound{y}\<%
\\
\>[0]\AgdaFunction{ind}\AgdaSpace{}%
\AgdaBound{A}\AgdaSpace{}%
\AgdaBound{a₀}\AgdaSpace{}%
\AgdaOperator{\AgdaBound{\AgdaUnderscore{}⋆\AgdaUnderscore{}}}\AgdaSpace{}%
\AgdaBound{⋆swap}\AgdaSpace{}%
\AgdaBound{sv}\AgdaSpace{}%
\AgdaSymbol{(}\AgdaInductiveConstructor{swap}\AgdaSpace{}%
\AgdaBound{x}\AgdaSpace{}%
\AgdaBound{y}\AgdaSpace{}%
\AgdaBound{z}\AgdaSpace{}%
\AgdaBound{i}\AgdaSymbol{)}\AgdaSpace{}%
\AgdaSymbol{=}\AgdaSpace{}%
\AgdaBound{⋆swap}%
\>[280I]\AgdaSymbol{(}\AgdaFunction{ind}\AgdaSpace{}%
\AgdaBound{A}\AgdaSpace{}%
\AgdaBound{a₀}\AgdaSpace{}%
\AgdaOperator{\AgdaBound{\AgdaUnderscore{}⋆\AgdaUnderscore{}}}\AgdaSpace{}%
\AgdaBound{⋆swap}\AgdaSpace{}%
\AgdaBound{sv}\AgdaSpace{}%
\AgdaBound{x}\AgdaSymbol{)}\<%
\\
\>[.][@{}l@{}]\<[280I]%
\>[45]\AgdaSymbol{(}\AgdaFunction{ind}\AgdaSpace{}%
\AgdaBound{A}\AgdaSpace{}%
\AgdaBound{a₀}\AgdaSpace{}%
\AgdaOperator{\AgdaBound{\AgdaUnderscore{}⋆\AgdaUnderscore{}}}\AgdaSpace{}%
\AgdaBound{⋆swap}\AgdaSpace{}%
\AgdaBound{sv}\AgdaSpace{}%
\AgdaBound{y}\AgdaSymbol{)}\<%
\\
\>[45]\AgdaSymbol{(}\AgdaFunction{ind}\AgdaSpace{}%
\AgdaBound{A}\AgdaSpace{}%
\AgdaBound{a₀}\AgdaSpace{}%
\AgdaOperator{\AgdaBound{\AgdaUnderscore{}⋆\AgdaUnderscore{}}}\AgdaSpace{}%
\AgdaBound{⋆swap}\AgdaSpace{}%
\AgdaBound{sv}\AgdaSpace{}%
\AgdaBound{z}\AgdaSymbol{)}\AgdaSpace{}%
\AgdaBound{i}\<%
\\
\>[0]\AgdaFunction{ind}\AgdaSpace{}%
\AgdaBound{A}\AgdaSpace{}%
\AgdaBound{a₀}\AgdaSpace{}%
\AgdaOperator{\AgdaBound{\AgdaUnderscore{}⋆\AgdaUnderscore{}}}\AgdaSpace{}%
\AgdaBound{⋆swap}\AgdaSpace{}%
\AgdaBound{sv}\AgdaSpace{}%
\AgdaSymbol{(}\AgdaInductiveConstructor{trunc}\AgdaSpace{}%
\AgdaBound{p}\AgdaSpace{}%
\AgdaBound{q}\AgdaSpace{}%
\AgdaBound{i}\AgdaSpace{}%
\AgdaBound{j}\AgdaSymbol{)}\AgdaSpace{}%
\AgdaSymbol{=}%
\>[311I]\AgdaFunction{isOfHLevel→isOfHLevelDep}\AgdaSpace{}%
\AgdaSymbol{\{}\AgdaArgument{n}\AgdaSpace{}%
\AgdaSymbol{=}\AgdaSpace{}%
\AgdaNumber{2}\AgdaSymbol{\}}\<%
\\
\>[311I][@{}l@{\AgdaIndent{0}}]%
\>[64]\AgdaSymbol{(λ}\AgdaSpace{}%
\AgdaBound{x}\AgdaSpace{}%
\AgdaBound{a}\AgdaSpace{}%
\AgdaBound{b}\AgdaSpace{}%
\AgdaSymbol{→}\AgdaSpace{}%
\AgdaBound{sv}\AgdaSpace{}%
\AgdaSymbol{\{}\AgdaBound{x}\AgdaSymbol{\}}\AgdaSpace{}%
\AgdaSymbol{\{}\AgdaBound{a}\AgdaSymbol{\}}\AgdaSpace{}%
\AgdaSymbol{\{}\AgdaBound{b}\AgdaSymbol{\})}\AgdaSpace{}%
\AgdaSymbol{\AgdaUnderscore{}}\AgdaSpace{}%
\AgdaSymbol{\AgdaUnderscore{}}\<%
\\
\>[64]\AgdaSymbol{(}\AgdaFunction{cong}\AgdaSpace{}%
\AgdaSymbol{(}\AgdaFunction{ind}\AgdaSpace{}%
\AgdaBound{A}\AgdaSpace{}%
\AgdaBound{a₀}\AgdaSpace{}%
\AgdaOperator{\AgdaBound{\AgdaUnderscore{}⋆\AgdaUnderscore{}}}\AgdaSpace{}%
\AgdaBound{⋆swap}\AgdaSpace{}%
\AgdaBound{sv}\AgdaSymbol{)}\AgdaSpace{}%
\AgdaBound{p}\AgdaSymbol{)}\<%
\\
\>[64]\AgdaSymbol{(}\AgdaFunction{cong}\AgdaSpace{}%
\AgdaSymbol{(}\AgdaFunction{ind}\AgdaSpace{}%
\AgdaBound{A}\AgdaSpace{}%
\AgdaBound{a₀}\AgdaSpace{}%
\AgdaOperator{\AgdaBound{\AgdaUnderscore{}⋆\AgdaUnderscore{}}}\AgdaSpace{}%
\AgdaBound{⋆swap}\AgdaSpace{}%
\AgdaBound{sv}\AgdaSymbol{)}\AgdaSpace{}%
\AgdaBound{q}\AgdaSymbol{)}\<%
\\
\>[64]\AgdaSymbol{(}\AgdaInductiveConstructor{trunc}\AgdaSpace{}%
\AgdaBound{p}\AgdaSpace{}%
\AgdaBound{q}\AgdaSymbol{)}\AgdaSpace{}%
\AgdaBound{i}\AgdaSpace{}%
\AgdaBound{j}\<%
\\
\>[0]\<%
\end{code}

\newcommand{\Hrec}{
\begin{code}%
\>[0]\AgdaFunction{rec}%
\>[343I]\AgdaSymbol{:}\AgdaSpace{}%
\AgdaSymbol{\{}\AgdaBound{A}\AgdaSpace{}%
\AgdaSymbol{:}\AgdaSpace{}%
\AgdaFunction{Type}\AgdaSpace{}%
\AgdaGeneralizable{ℓ}\AgdaSymbol{\}}\<%
\\
\>[.][@{}l@{}]\<[343I]%
\>[4]\AgdaSymbol{→}\AgdaSpace{}%
\AgdaFunction{isSet}\AgdaSpace{}%
\AgdaBound{A}\<%
\\
\>[4]\AgdaSymbol{→}\AgdaSpace{}%
\AgdaBound{A}\<%
\\
\>[4]\AgdaSymbol{→}\AgdaSpace{}%
\AgdaSymbol{(}\AgdaOperator{\AgdaBound{\AgdaUnderscore{}⋆\AgdaUnderscore{}}}\AgdaSpace{}%
\AgdaSymbol{:}\AgdaSpace{}%
\AgdaBound{A}\AgdaSpace{}%
\AgdaSymbol{→}\AgdaSpace{}%
\AgdaBound{A}\AgdaSpace{}%
\AgdaSymbol{→}\AgdaSpace{}%
\AgdaBound{A}\AgdaSymbol{)}\<%
\\
\>[4]\AgdaSymbol{→}\AgdaSpace{}%
\AgdaSymbol{(∀}\AgdaSpace{}%
\AgdaBound{x}\AgdaSpace{}%
\AgdaBound{y}\AgdaSpace{}%
\AgdaBound{z}\AgdaSpace{}%
\AgdaSymbol{→}\AgdaSpace{}%
\AgdaBound{x}\AgdaSpace{}%
\AgdaOperator{\AgdaBound{⋆}}\AgdaSpace{}%
\AgdaSymbol{(}\AgdaBound{y}\AgdaSpace{}%
\AgdaOperator{\AgdaBound{⋆}}\AgdaSpace{}%
\AgdaBound{z}\AgdaSymbol{)}\AgdaSpace{}%
\AgdaOperator{\AgdaFunction{≡}}\AgdaSpace{}%
\AgdaBound{y}\AgdaSpace{}%
\AgdaOperator{\AgdaBound{⋆}}\AgdaSpace{}%
\AgdaSymbol{(}\AgdaBound{x}\AgdaSpace{}%
\AgdaOperator{\AgdaBound{⋆}}\AgdaSpace{}%
\AgdaBound{z}\AgdaSymbol{))}\<%
\\
\>[4]\AgdaSymbol{→}\AgdaSpace{}%
\AgdaDatatype{HITOrd}\AgdaSpace{}%
\AgdaSymbol{→}\AgdaSpace{}%
\AgdaBound{A}\<%
\end{code}}

\begin{code}[hide]%
\>[0]\<%
\\
\>[0]\AgdaFunction{rec}\AgdaSpace{}%
\AgdaBound{hset}\AgdaSpace{}%
\AgdaBound{a₀}\AgdaSpace{}%
\AgdaOperator{\AgdaBound{\AgdaUnderscore{}⋆\AgdaUnderscore{}}}\AgdaSpace{}%
\AgdaBound{⋆swap}\AgdaSpace{}%
\AgdaSymbol{=}\AgdaSpace{}%
\AgdaFunction{ind}\AgdaSpace{}%
\AgdaSymbol{\AgdaUnderscore{}}\AgdaSpace{}%
\AgdaBound{a₀}\AgdaSpace{}%
\AgdaOperator{\AgdaBound{\AgdaUnderscore{}⋆\AgdaUnderscore{}}}\AgdaSpace{}%
\AgdaBound{⋆swap}\AgdaSpace{}%
\AgdaBound{hset}\<%
\\
\>[0]\<%
\end{code}

\newcommand{\HindProp}{
\begin{code}%
\>[0]\AgdaFunction{indProp}%
\>[388I]\AgdaSymbol{:}\AgdaSpace{}%
\AgdaSymbol{(}\AgdaBound{P}\AgdaSpace{}%
\AgdaSymbol{:}\AgdaSpace{}%
\AgdaDatatype{HITOrd}\AgdaSpace{}%
\AgdaSymbol{→}\AgdaSpace{}%
\AgdaFunction{Type}\AgdaSpace{}%
\AgdaGeneralizable{ℓ}\AgdaSymbol{)}\<%
\\
\>[.][@{}l@{}]\<[388I]%
\>[8]\AgdaSymbol{→}\AgdaSpace{}%
\AgdaSymbol{(∀}\AgdaSpace{}%
\AgdaSymbol{\{}\AgdaBound{x}\AgdaSymbol{\}}\AgdaSpace{}%
\AgdaSymbol{→}\AgdaSpace{}%
\AgdaFunction{isProp}\AgdaSpace{}%
\AgdaSymbol{(}\AgdaBound{P}\AgdaSpace{}%
\AgdaBound{x}\AgdaSymbol{))}\<%
\\
\>[8]\AgdaSymbol{→}\AgdaSpace{}%
\AgdaBound{P}\AgdaSpace{}%
\AgdaInductiveConstructor{𝟎}\<%
\\
\>[8]\AgdaSymbol{→}\AgdaSpace{}%
\AgdaSymbol{(∀}\AgdaSpace{}%
\AgdaSymbol{\{}\AgdaBound{x}\AgdaSpace{}%
\AgdaBound{y}\AgdaSymbol{\}}\AgdaSpace{}%
\AgdaSymbol{→}\AgdaSpace{}%
\AgdaBound{P}\AgdaSpace{}%
\AgdaBound{x}\AgdaSpace{}%
\AgdaSymbol{→}\AgdaSpace{}%
\AgdaBound{P}\AgdaSpace{}%
\AgdaBound{y}\AgdaSpace{}%
\AgdaSymbol{→}\AgdaSpace{}%
\AgdaBound{P}\AgdaSpace{}%
\AgdaSymbol{(}\AgdaOperator{\AgdaInductiveConstructor{ω\textasciicircum{}}}\AgdaSpace{}%
\AgdaBound{x}\AgdaSpace{}%
\AgdaOperator{\AgdaInductiveConstructor{⊕}}\AgdaSpace{}%
\AgdaBound{y}\AgdaSymbol{))}\<%
\\
\>[8]\AgdaSymbol{→}\AgdaSpace{}%
\AgdaSymbol{∀}\AgdaSpace{}%
\AgdaBound{x}\AgdaSpace{}%
\AgdaSymbol{→}\AgdaSpace{}%
\AgdaBound{P}\AgdaSpace{}%
\AgdaBound{x}\<%
\end{code}}

\begin{code}[hide]%
\>[0]\<%
\\
\>[0]\AgdaFunction{indProp}\AgdaSpace{}%
\AgdaBound{P}\AgdaSpace{}%
\AgdaBound{pv}\AgdaSpace{}%
\AgdaBound{p₀}\AgdaSpace{}%
\AgdaOperator{\AgdaBound{\AgdaUnderscore{}⋆\AgdaUnderscore{}}}\AgdaSpace{}%
\AgdaSymbol{=}\AgdaSpace{}%
\AgdaFunction{ind}\AgdaSpace{}%
\AgdaSymbol{\AgdaUnderscore{}}\AgdaSpace{}%
\AgdaBound{p₀}\AgdaSpace{}%
\AgdaOperator{\AgdaBound{\AgdaUnderscore{}⋆\AgdaUnderscore{}}}\AgdaSpace{}%
\AgdaSymbol{(λ}\AgdaSpace{}%
\AgdaBound{a}\AgdaSpace{}%
\AgdaBound{b}\AgdaSpace{}%
\AgdaBound{c}\AgdaSpace{}%
\AgdaSymbol{→}\AgdaSpace{}%
\AgdaFunction{toPathP}\AgdaSpace{}%
\AgdaSymbol{(}\AgdaBound{pv}\AgdaSpace{}%
\AgdaSymbol{\AgdaUnderscore{}}\AgdaSpace{}%
\AgdaSymbol{\AgdaUnderscore{}))}\AgdaSpace{}%
\AgdaSymbol{(}\AgdaFunction{isProp→isSet}\AgdaSpace{}%
\AgdaBound{pv}\AgdaSpace{}%
\AgdaSymbol{\AgdaUnderscore{}}\AgdaSpace{}%
\AgdaSymbol{\AgdaUnderscore{})}\<%
\\
\>[0]\<%
\end{code}

\begin{code}[hide]%
\>[0]\AgdaFunction{𝟏}\AgdaSpace{}%
\AgdaFunction{ω}\AgdaSpace{}%
\AgdaSymbol{:}\AgdaSpace{}%
\AgdaDatatype{HITOrd}\<%
\\
\>[0]\AgdaFunction{𝟏}\AgdaSpace{}%
\AgdaSymbol{=}\AgdaSpace{}%
\AgdaOperator{\AgdaFunction{ω\textasciicircum{}⟨}}\AgdaSpace{}%
\AgdaInductiveConstructor{𝟎}\AgdaSpace{}%
\AgdaOperator{\AgdaFunction{⟩}}\<%
\\
\>[0]\AgdaFunction{ω}\AgdaSpace{}%
\AgdaSymbol{=}\AgdaSpace{}%
\AgdaOperator{\AgdaFunction{ω\textasciicircum{}⟨}}\AgdaSpace{}%
\AgdaFunction{𝟏}\AgdaSpace{}%
\AgdaOperator{\AgdaFunction{⟩}}\<%
\end{code}

\newcommand{\MO}{\AgdaDatatype{Mutual\-Ord}\xspace}

\begin{code}[hide]%
\>[0]\<%
\\
\>[0]\AgdaSymbol{\{-\#}\AgdaSpace{}%
\AgdaKeyword{OPTIONS}\AgdaSpace{}%
\AgdaPragma{--cubical}\AgdaSpace{}%
\AgdaPragma{--no-termination-check}\AgdaSpace{}%
\AgdaSymbol{\#-\}}\<%
\\
\\[\AgdaEmptyExtraSkip]%
\>[0]\AgdaKeyword{module}\AgdaSpace{}%
\AgdaModule{MutualOrd}\AgdaSpace{}%
\AgdaKeyword{where}\<%
\\
\\[\AgdaEmptyExtraSkip]%
\>[0]\AgdaKeyword{open}\AgdaSpace{}%
\AgdaKeyword{import}\AgdaSpace{}%
\AgdaModule{Cubical.Foundations.Everything}\<%
\\
\>[0][@{}l@{\AgdaIndent{0}}]%
\>[1]\AgdaKeyword{using}\AgdaSpace{}%
\AgdaSymbol{(}\AgdaFunction{Type₀}\AgdaSpace{}%
\AgdaSymbol{;}\AgdaSpace{}%
\AgdaFunction{Type}\AgdaSpace{}%
\AgdaSymbol{;}\AgdaSpace{}%
\AgdaOperator{\AgdaFunction{\AgdaUnderscore{}≡\AgdaUnderscore{}}}\AgdaSpace{}%
\AgdaSymbol{;}\AgdaSpace{}%
\AgdaFunction{refl}\AgdaSpace{}%
\AgdaSymbol{;}\AgdaSpace{}%
\AgdaFunction{J}\AgdaSpace{}%
\AgdaSymbol{;}\AgdaSpace{}%
\AgdaFunction{cong}\AgdaSpace{}%
\AgdaSymbol{;}\AgdaSpace{}%
\AgdaFunction{cong₂}\AgdaSpace{}%
\AgdaSymbol{;}\AgdaSpace{}%
\AgdaFunction{subst}\AgdaSpace{}%
\AgdaSymbol{;}\AgdaSpace{}%
\AgdaOperator{\AgdaFunction{\AgdaUnderscore{}⁻¹}}\AgdaSpace{}%
\AgdaSymbol{;}\AgdaSpace{}%
\AgdaOperator{\AgdaFunction{\AgdaUnderscore{}∙\AgdaUnderscore{}}}\AgdaSymbol{)}\<%
\\
\\[\AgdaEmptyExtraSkip]%
\>[0]\AgdaKeyword{open}\AgdaSpace{}%
\AgdaKeyword{import}\AgdaSpace{}%
\AgdaModule{Cubical.Data.Empty}\<%
\\
\\[\AgdaEmptyExtraSkip]%
\>[0]\AgdaKeyword{open}\AgdaSpace{}%
\AgdaKeyword{import}\AgdaSpace{}%
\AgdaModule{Cubical.Relation.Nullary}\<%
\\
\>[0]\AgdaKeyword{open}\AgdaSpace{}%
\AgdaKeyword{import}\AgdaSpace{}%
\AgdaModule{Cubical.Relation.Nullary.DecidableEq}\<%
\\
\\[\AgdaEmptyExtraSkip]%
\>[0]\AgdaKeyword{open}\AgdaSpace{}%
\AgdaKeyword{import}\AgdaSpace{}%
\AgdaModule{Preliminaries}\<%
\\
\\[\AgdaEmptyExtraSkip]%
\>[0]\AgdaKeyword{import}\AgdaSpace{}%
\AgdaModule{Relation.Binary.PropositionalEquality}\AgdaSpace{}%
\AgdaSymbol{as}\AgdaSpace{}%
\AgdaModule{P}\<%
\\
\\[\AgdaEmptyExtraSkip]%
\>[0]\AgdaFunction{fromPath}\AgdaSpace{}%
\AgdaSymbol{:}\AgdaSpace{}%
\AgdaSymbol{\{}\AgdaBound{A}\AgdaSpace{}%
\AgdaSymbol{:}\AgdaSpace{}%
\AgdaPrimitiveType{Set}\AgdaSymbol{\}\{}\AgdaBound{x}\AgdaSpace{}%
\AgdaBound{y}\AgdaSpace{}%
\AgdaSymbol{:}\AgdaSpace{}%
\AgdaBound{A}\AgdaSymbol{\}}\AgdaSpace{}%
\AgdaSymbol{->}\AgdaSpace{}%
\AgdaBound{x}\AgdaSpace{}%
\AgdaOperator{\AgdaFunction{≡}}\AgdaSpace{}%
\AgdaBound{y}\AgdaSpace{}%
\AgdaSymbol{->}\AgdaSpace{}%
\AgdaBound{x}\AgdaSpace{}%
\AgdaOperator{\AgdaDatatype{P.≡}}\AgdaSpace{}%
\AgdaBound{y}\<%
\\
\>[0]\AgdaFunction{fromPath}\AgdaSpace{}%
\AgdaSymbol{\{}\AgdaArgument{x}\AgdaSpace{}%
\AgdaSymbol{=}\AgdaSpace{}%
\AgdaBound{x}\AgdaSymbol{\}}\AgdaSpace{}%
\AgdaSymbol{=}\AgdaSpace{}%
\AgdaFunction{J}\AgdaSpace{}%
\AgdaSymbol{(λ}\AgdaSpace{}%
\AgdaBound{z}\AgdaSpace{}%
\AgdaBound{p}\AgdaSpace{}%
\AgdaSymbol{→}\AgdaSpace{}%
\AgdaBound{x}\AgdaSpace{}%
\AgdaOperator{\AgdaDatatype{P.≡}}\AgdaSpace{}%
\AgdaBound{z}\AgdaSymbol{)}\AgdaSpace{}%
\AgdaInductiveConstructor{P.refl}\<%
\\
\\[\AgdaEmptyExtraSkip]%
\>[0]\AgdaKeyword{infix}\AgdaSpace{}%
\AgdaNumber{30}\AgdaSpace{}%
\AgdaOperator{\AgdaDatatype{\AgdaUnderscore{}<\AgdaUnderscore{}}}\AgdaSpace{}%
\AgdaOperator{\AgdaFunction{\AgdaUnderscore{}≤\AgdaUnderscore{}}}\AgdaSpace{}%
\AgdaOperator{\AgdaFunction{\AgdaUnderscore{}>\AgdaUnderscore{}}}\AgdaSpace{}%
\AgdaOperator{\AgdaFunction{\AgdaUnderscore{}≥\AgdaUnderscore{}}}\<%
\\
\>[0]\<%
\end{code}

\newcommand{\Mdefs}{
\begin{code}%
\>[0]\AgdaKeyword{data}\AgdaSpace{}%
\AgdaDatatype{MutualOrd}\AgdaSpace{}%
\AgdaSymbol{:}\AgdaSpace{}%
\AgdaFunction{Type₀}\<%
\\
\>[0]\AgdaKeyword{data}\AgdaSpace{}%
\AgdaOperator{\AgdaDatatype{\AgdaUnderscore{}<\AgdaUnderscore{}}}\AgdaSpace{}%
\AgdaSymbol{:}\AgdaSpace{}%
\AgdaDatatype{MutualOrd}\AgdaSpace{}%
\AgdaSymbol{→}\AgdaSpace{}%
\AgdaDatatype{MutualOrd}\AgdaSpace{}%
\AgdaSymbol{→}\AgdaSpace{}%
\AgdaFunction{Type₀}\<%
\\
\>[0]\AgdaFunction{fst}\AgdaSpace{}%
\AgdaSymbol{:}\AgdaSpace{}%
\AgdaDatatype{MutualOrd}\AgdaSpace{}%
\AgdaSymbol{→}\AgdaSpace{}%
\AgdaDatatype{MutualOrd}\<%
\end{code}}

\newcommand{\MutualOrd}{
\begin{code}%
\>[0]\AgdaKeyword{data}\AgdaSpace{}%
\AgdaDatatype{MutualOrd}\AgdaSpace{}%
\AgdaKeyword{where}\<%
\\
\>[0][@{}l@{\AgdaIndent{0}}]%
\>[1]\AgdaInductiveConstructor{𝟎}\AgdaSpace{}%
\AgdaSymbol{:}\AgdaSpace{}%
\AgdaDatatype{MutualOrd}\<%
\\
\>[1]\AgdaOperator{\AgdaInductiveConstructor{ω\textasciicircum{}\AgdaUnderscore{}+\AgdaUnderscore{}[\AgdaUnderscore{}]}}\AgdaSpace{}%
\AgdaSymbol{:}\AgdaSpace{}%
\AgdaSymbol{(}\AgdaBound{a}\AgdaSpace{}%
\AgdaBound{b}\AgdaSpace{}%
\AgdaSymbol{:}\AgdaSpace{}%
\AgdaDatatype{MutualOrd}\AgdaSymbol{)}\AgdaSpace{}%
\AgdaSymbol{→}\AgdaSpace{}%
\AgdaBound{a}\AgdaSpace{}%
\AgdaOperator{\AgdaFunction{≥}}\AgdaSpace{}%
\AgdaFunction{fst}\AgdaSpace{}%
\AgdaBound{b}\AgdaSpace{}%
\AgdaSymbol{→}\AgdaSpace{}%
\AgdaDatatype{MutualOrd}\<%
\end{code}}

\begin{code}[hide]%
\>[0]\<%
\\
\>[0]\AgdaKeyword{variable}\<%
\\
\>[0][@{}l@{\AgdaIndent{0}}]%
\>[1]\AgdaGeneralizable{a}\AgdaSpace{}%
\AgdaGeneralizable{b}\AgdaSpace{}%
\AgdaGeneralizable{c}\AgdaSpace{}%
\AgdaGeneralizable{d}\AgdaSpace{}%
\AgdaSymbol{:}\AgdaSpace{}%
\AgdaDatatype{MutualOrd}\<%
\\
\>[1]\AgdaGeneralizable{r}\AgdaSpace{}%
\AgdaSymbol{:}\AgdaSpace{}%
\AgdaGeneralizable{a}\AgdaSpace{}%
\AgdaOperator{\AgdaFunction{≥}}\AgdaSpace{}%
\AgdaFunction{fst}\AgdaSpace{}%
\AgdaGeneralizable{b}\<%
\\
\>[1]\AgdaGeneralizable{s}\AgdaSpace{}%
\AgdaSymbol{:}\AgdaSpace{}%
\AgdaGeneralizable{c}\AgdaSpace{}%
\AgdaOperator{\AgdaFunction{≥}}\AgdaSpace{}%
\AgdaFunction{fst}\AgdaSpace{}%
\AgdaGeneralizable{d}\<%
\\
\>[0]\<%
\end{code}

\newcommand{\MOrdLt}{
\begin{code}%
\>[0]\AgdaKeyword{data}\AgdaSpace{}%
\AgdaOperator{\AgdaDatatype{\AgdaUnderscore{}<\AgdaUnderscore{}}}\AgdaSpace{}%
\AgdaKeyword{where}\<%
\\
\>[0][@{}l@{\AgdaIndent{0}}]%
\>[1]\AgdaInductiveConstructor{<₁}\AgdaSpace{}%
\AgdaSymbol{:}\AgdaSpace{}%
\AgdaInductiveConstructor{𝟎}\AgdaSpace{}%
\AgdaOperator{\AgdaDatatype{<}}\AgdaSpace{}%
\AgdaOperator{\AgdaInductiveConstructor{ω\textasciicircum{}}}\AgdaSpace{}%
\AgdaGeneralizable{a}\AgdaSpace{}%
\AgdaOperator{\AgdaInductiveConstructor{+}}\AgdaSpace{}%
\AgdaGeneralizable{b}\AgdaSpace{}%
\AgdaOperator{\AgdaInductiveConstructor{[}}\AgdaSpace{}%
\AgdaGeneralizable{r}\AgdaSpace{}%
\AgdaOperator{\AgdaInductiveConstructor{]}}\<%
\\
\>[1]\AgdaInductiveConstructor{<₂}\AgdaSpace{}%
\AgdaSymbol{:}\AgdaSpace{}%
\AgdaGeneralizable{a}\AgdaSpace{}%
\AgdaOperator{\AgdaDatatype{<}}\AgdaSpace{}%
\AgdaGeneralizable{c}\AgdaSpace{}%
\AgdaSymbol{→}\AgdaSpace{}%
\AgdaOperator{\AgdaInductiveConstructor{ω\textasciicircum{}}}\AgdaSpace{}%
\AgdaGeneralizable{a}\AgdaSpace{}%
\AgdaOperator{\AgdaInductiveConstructor{+}}\AgdaSpace{}%
\AgdaGeneralizable{b}\AgdaSpace{}%
\AgdaOperator{\AgdaInductiveConstructor{[}}\AgdaSpace{}%
\AgdaGeneralizable{r}\AgdaSpace{}%
\AgdaOperator{\AgdaInductiveConstructor{]}}\AgdaSpace{}%
\AgdaOperator{\AgdaDatatype{<}}\AgdaSpace{}%
\AgdaOperator{\AgdaInductiveConstructor{ω\textasciicircum{}}}\AgdaSpace{}%
\AgdaGeneralizable{c}\AgdaSpace{}%
\AgdaOperator{\AgdaInductiveConstructor{+}}\AgdaSpace{}%
\AgdaGeneralizable{d}\AgdaSpace{}%
\AgdaOperator{\AgdaInductiveConstructor{[}}\AgdaSpace{}%
\AgdaGeneralizable{s}\AgdaSpace{}%
\AgdaOperator{\AgdaInductiveConstructor{]}}\<%
\\
\>[1]\AgdaInductiveConstructor{<₃}\AgdaSpace{}%
\AgdaSymbol{:}\AgdaSpace{}%
\AgdaGeneralizable{a}\AgdaSpace{}%
\AgdaOperator{\AgdaFunction{≡}}\AgdaSpace{}%
\AgdaGeneralizable{c}\AgdaSpace{}%
\AgdaSymbol{→}\AgdaSpace{}%
\AgdaGeneralizable{b}\AgdaSpace{}%
\AgdaOperator{\AgdaDatatype{<}}\AgdaSpace{}%
\AgdaGeneralizable{d}\AgdaSpace{}%
\AgdaSymbol{→}\AgdaSpace{}%
\AgdaOperator{\AgdaInductiveConstructor{ω\textasciicircum{}}}\AgdaSpace{}%
\AgdaGeneralizable{a}\AgdaSpace{}%
\AgdaOperator{\AgdaInductiveConstructor{+}}\AgdaSpace{}%
\AgdaGeneralizable{b}\AgdaSpace{}%
\AgdaOperator{\AgdaInductiveConstructor{[}}\AgdaSpace{}%
\AgdaGeneralizable{r}\AgdaSpace{}%
\AgdaOperator{\AgdaInductiveConstructor{]}}\AgdaSpace{}%
\AgdaOperator{\AgdaDatatype{<}}\AgdaSpace{}%
\AgdaOperator{\AgdaInductiveConstructor{ω\textasciicircum{}}}\AgdaSpace{}%
\AgdaGeneralizable{c}\AgdaSpace{}%
\AgdaOperator{\AgdaInductiveConstructor{+}}\AgdaSpace{}%
\AgdaGeneralizable{d}\AgdaSpace{}%
\AgdaOperator{\AgdaInductiveConstructor{[}}\AgdaSpace{}%
\AgdaGeneralizable{s}\AgdaSpace{}%
\AgdaOperator{\AgdaInductiveConstructor{]}}\<%
\end{code}}

\newcommand{\MOrdFst}{
\begin{code}%
\>[0]\AgdaFunction{fst}%
\>[5]\AgdaInductiveConstructor{𝟎}%
\>[21]\AgdaSymbol{=}\AgdaSpace{}%
\AgdaInductiveConstructor{𝟎}\<%
\\
\>[0]\AgdaFunction{fst}\AgdaSpace{}%
\AgdaSymbol{(}\AgdaOperator{\AgdaInductiveConstructor{ω\textasciicircum{}}}\AgdaSpace{}%
\AgdaBound{a}\AgdaSpace{}%
\AgdaOperator{\AgdaInductiveConstructor{+}}\AgdaSpace{}%
\AgdaSymbol{\AgdaUnderscore{}}\AgdaSpace{}%
\AgdaOperator{\AgdaInductiveConstructor{[}}\AgdaSpace{}%
\AgdaSymbol{\AgdaUnderscore{}}\AgdaSpace{}%
\AgdaOperator{\AgdaInductiveConstructor{]}}\AgdaSymbol{)}\AgdaSpace{}%
\AgdaSymbol{=}\AgdaSpace{}%
\AgdaBound{a}\<%
\end{code}}

\begin{code}[hide]%
\>[0]\<%
\\
\>[0]\AgdaOperator{\AgdaFunction{\AgdaUnderscore{}≤\AgdaUnderscore{}}}\AgdaSpace{}%
\AgdaSymbol{:}\AgdaSpace{}%
\AgdaDatatype{MutualOrd}\AgdaSpace{}%
\AgdaSymbol{→}\AgdaSpace{}%
\AgdaDatatype{MutualOrd}\AgdaSpace{}%
\AgdaSymbol{→}\AgdaSpace{}%
\AgdaFunction{Type₀}\<%
\\
\>[0]\AgdaBound{a}\AgdaSpace{}%
\AgdaOperator{\AgdaFunction{≤}}\AgdaSpace{}%
\AgdaBound{b}\AgdaSpace{}%
\AgdaSymbol{=}%
\>[9]\AgdaBound{b}\AgdaSpace{}%
\AgdaOperator{\AgdaFunction{≥}}\AgdaSpace{}%
\AgdaBound{a}\<%
\\
\\[\AgdaEmptyExtraSkip]%
\>[0]\AgdaFunction{rest}\AgdaSpace{}%
\AgdaSymbol{:}\AgdaSpace{}%
\AgdaDatatype{MutualOrd}\AgdaSpace{}%
\AgdaSymbol{→}\AgdaSpace{}%
\AgdaDatatype{MutualOrd}\<%
\\
\>[0]\AgdaFunction{rest}%
\>[6]\AgdaInductiveConstructor{𝟎}%
\>[22]\AgdaSymbol{=}\AgdaSpace{}%
\AgdaInductiveConstructor{𝟎}\<%
\\
\>[0]\AgdaFunction{rest}\AgdaSpace{}%
\AgdaSymbol{(}\AgdaOperator{\AgdaInductiveConstructor{ω\textasciicircum{}}}\AgdaSpace{}%
\AgdaSymbol{\AgdaUnderscore{}}\AgdaSpace{}%
\AgdaOperator{\AgdaInductiveConstructor{+}}\AgdaSpace{}%
\AgdaBound{b}\AgdaSpace{}%
\AgdaOperator{\AgdaInductiveConstructor{[}}\AgdaSpace{}%
\AgdaSymbol{\AgdaUnderscore{}}\AgdaSpace{}%
\AgdaOperator{\AgdaInductiveConstructor{]}}\AgdaSymbol{)}\AgdaSpace{}%
\AgdaSymbol{=}\AgdaSpace{}%
\AgdaBound{b}\<%
\\
\\[\AgdaEmptyExtraSkip]%
\>[0]\AgdaFunction{≮𝟎}\AgdaSpace{}%
\AgdaSymbol{:}\AgdaSpace{}%
\AgdaSymbol{\{}\AgdaBound{a}\AgdaSpace{}%
\AgdaSymbol{:}\AgdaSpace{}%
\AgdaDatatype{MutualOrd}\AgdaSymbol{\}}\AgdaSpace{}%
\AgdaSymbol{→}\AgdaSpace{}%
\AgdaOperator{\AgdaFunction{¬}}\AgdaSpace{}%
\AgdaBound{a}\AgdaSpace{}%
\AgdaOperator{\AgdaDatatype{<}}\AgdaSpace{}%
\AgdaInductiveConstructor{𝟎}\<%
\\
\>[0]\AgdaFunction{≮𝟎}\AgdaSpace{}%
\AgdaSymbol{()}\<%
\\
\\[\AgdaEmptyExtraSkip]%
\>[0]\AgdaFunction{caseMutualOrd}\AgdaSpace{}%
\AgdaSymbol{:}\AgdaSpace{}%
\AgdaSymbol{∀}\AgdaSpace{}%
\AgdaSymbol{\{}\AgdaBound{ℓ}\AgdaSymbol{\}}\AgdaSpace{}%
\AgdaSymbol{→}\AgdaSpace{}%
\AgdaSymbol{\{}\AgdaBound{A}\AgdaSpace{}%
\AgdaSymbol{:}\AgdaSpace{}%
\AgdaFunction{Type}\AgdaSpace{}%
\AgdaBound{ℓ}\AgdaSymbol{\}}\AgdaSpace{}%
\AgdaSymbol{(}\AgdaBound{x}\AgdaSpace{}%
\AgdaBound{y}\AgdaSpace{}%
\AgdaSymbol{:}\AgdaSpace{}%
\AgdaBound{A}\AgdaSymbol{)}\AgdaSpace{}%
\AgdaSymbol{→}\AgdaSpace{}%
\AgdaDatatype{MutualOrd}\AgdaSpace{}%
\AgdaSymbol{→}\AgdaSpace{}%
\AgdaBound{A}\<%
\\
\>[0]\AgdaFunction{caseMutualOrd}\AgdaSpace{}%
\AgdaBound{x}\AgdaSpace{}%
\AgdaBound{y}%
\>[19]\AgdaInductiveConstructor{𝟎}%
\>[35]\AgdaSymbol{=}\AgdaSpace{}%
\AgdaBound{x}\<%
\\
\>[0]\AgdaFunction{caseMutualOrd}\AgdaSpace{}%
\AgdaBound{x}\AgdaSpace{}%
\AgdaBound{y}\AgdaSpace{}%
\AgdaSymbol{(}\AgdaOperator{\AgdaInductiveConstructor{ω\textasciicircum{}}}\AgdaSpace{}%
\AgdaSymbol{\AgdaUnderscore{}}\AgdaSpace{}%
\AgdaOperator{\AgdaInductiveConstructor{+}}\AgdaSpace{}%
\AgdaSymbol{\AgdaUnderscore{}}\AgdaSpace{}%
\AgdaOperator{\AgdaInductiveConstructor{[}}\AgdaSpace{}%
\AgdaSymbol{\AgdaUnderscore{}}\AgdaSpace{}%
\AgdaOperator{\AgdaInductiveConstructor{]}}\AgdaSymbol{)}\AgdaSpace{}%
\AgdaSymbol{=}\AgdaSpace{}%
\AgdaBound{y}\<%
\\
\\[\AgdaEmptyExtraSkip]%
\>[0]\AgdaFunction{𝟎≢ω}\AgdaSpace{}%
\AgdaSymbol{:}\AgdaSpace{}%
\AgdaSymbol{\{}\AgdaBound{a}\AgdaSpace{}%
\AgdaBound{b}\AgdaSpace{}%
\AgdaSymbol{:}\AgdaSpace{}%
\AgdaDatatype{MutualOrd}\AgdaSymbol{\}}\AgdaSpace{}%
\AgdaSymbol{\{}\AgdaBound{r}\AgdaSpace{}%
\AgdaSymbol{:}\AgdaSpace{}%
\AgdaBound{a}\AgdaSpace{}%
\AgdaOperator{\AgdaFunction{≥}}\AgdaSpace{}%
\AgdaFunction{fst}\AgdaSpace{}%
\AgdaBound{b}\AgdaSymbol{\}}\AgdaSpace{}%
\AgdaSymbol{→}\AgdaSpace{}%
\AgdaOperator{\AgdaFunction{¬}}\AgdaSpace{}%
\AgdaSymbol{(}\AgdaInductiveConstructor{𝟎}\AgdaSpace{}%
\AgdaOperator{\AgdaFunction{≡}}\AgdaSpace{}%
\AgdaOperator{\AgdaInductiveConstructor{ω\textasciicircum{}}}\AgdaSpace{}%
\AgdaBound{a}\AgdaSpace{}%
\AgdaOperator{\AgdaInductiveConstructor{+}}\AgdaSpace{}%
\AgdaBound{b}\AgdaSpace{}%
\AgdaOperator{\AgdaInductiveConstructor{[}}\AgdaSpace{}%
\AgdaBound{r}\AgdaSpace{}%
\AgdaOperator{\AgdaInductiveConstructor{]}}\AgdaSymbol{)}\<%
\\
\>[0]\AgdaFunction{𝟎≢ω}\AgdaSpace{}%
\AgdaBound{e}\AgdaSpace{}%
\AgdaSymbol{=}\AgdaSpace{}%
\AgdaFunction{subst}\AgdaSpace{}%
\AgdaSymbol{(}\AgdaFunction{caseMutualOrd}\AgdaSpace{}%
\AgdaDatatype{MutualOrd}\AgdaSpace{}%
\AgdaDatatype{⊥}\AgdaSymbol{)}\AgdaSpace{}%
\AgdaBound{e}\AgdaSpace{}%
\AgdaInductiveConstructor{𝟎}\<%
\\
\\[\AgdaEmptyExtraSkip]%
\>[0]\AgdaFunction{ω≢𝟎}\AgdaSpace{}%
\AgdaSymbol{:}\AgdaSpace{}%
\AgdaSymbol{\{}\AgdaBound{a}\AgdaSpace{}%
\AgdaBound{b}\AgdaSpace{}%
\AgdaSymbol{:}\AgdaSpace{}%
\AgdaDatatype{MutualOrd}\AgdaSymbol{\}}\AgdaSpace{}%
\AgdaSymbol{\{}\AgdaBound{r}\AgdaSpace{}%
\AgdaSymbol{:}\AgdaSpace{}%
\AgdaBound{a}\AgdaSpace{}%
\AgdaOperator{\AgdaFunction{≥}}\AgdaSpace{}%
\AgdaFunction{fst}\AgdaSpace{}%
\AgdaBound{b}\AgdaSymbol{\}}\AgdaSpace{}%
\AgdaSymbol{→}\AgdaSpace{}%
\AgdaOperator{\AgdaFunction{¬}}\AgdaSpace{}%
\AgdaSymbol{(}\AgdaOperator{\AgdaInductiveConstructor{ω\textasciicircum{}}}\AgdaSpace{}%
\AgdaBound{a}\AgdaSpace{}%
\AgdaOperator{\AgdaInductiveConstructor{+}}\AgdaSpace{}%
\AgdaBound{b}\AgdaSpace{}%
\AgdaOperator{\AgdaInductiveConstructor{[}}\AgdaSpace{}%
\AgdaBound{r}\AgdaSpace{}%
\AgdaOperator{\AgdaInductiveConstructor{]}}\AgdaSpace{}%
\AgdaOperator{\AgdaFunction{≡}}\AgdaSpace{}%
\AgdaInductiveConstructor{𝟎}\AgdaSymbol{)}\<%
\\
\>[0]\AgdaFunction{ω≢𝟎}\AgdaSpace{}%
\AgdaBound{e}\AgdaSpace{}%
\AgdaSymbol{=}\AgdaSpace{}%
\AgdaFunction{subst}\AgdaSpace{}%
\AgdaSymbol{(}\AgdaFunction{caseMutualOrd}\AgdaSpace{}%
\AgdaDatatype{⊥}\AgdaSpace{}%
\AgdaDatatype{MutualOrd}\AgdaSymbol{)}\AgdaSpace{}%
\AgdaBound{e}\AgdaSpace{}%
\AgdaInductiveConstructor{𝟎}\<%
\\
\\[\AgdaEmptyExtraSkip]%
\>[0]\AgdaFunction{<-irrefl}\AgdaSpace{}%
\AgdaSymbol{:}\AgdaSpace{}%
\AgdaSymbol{\{}\AgdaBound{a}\AgdaSpace{}%
\AgdaSymbol{:}\AgdaSpace{}%
\AgdaDatatype{MutualOrd}\AgdaSymbol{\}}\AgdaSpace{}%
\AgdaSymbol{→}\AgdaSpace{}%
\AgdaOperator{\AgdaFunction{¬}}\AgdaSpace{}%
\AgdaSymbol{(}\AgdaBound{a}\AgdaSpace{}%
\AgdaOperator{\AgdaDatatype{<}}\AgdaSpace{}%
\AgdaBound{a}\AgdaSymbol{)}\<%
\\
\>[0]\AgdaFunction{<-irrefl}\AgdaSpace{}%
\AgdaSymbol{(}\AgdaInductiveConstructor{<₂}\AgdaSpace{}%
\AgdaBound{r}\AgdaSymbol{)}%
\>[20]\AgdaSymbol{=}\AgdaSpace{}%
\AgdaFunction{<-irrefl}\AgdaSpace{}%
\AgdaBound{r}\<%
\\
\>[0]\AgdaFunction{<-irrefl}\AgdaSpace{}%
\AgdaSymbol{(}\AgdaInductiveConstructor{<₃}\AgdaSpace{}%
\AgdaBound{a=a}\AgdaSpace{}%
\AgdaBound{r}\AgdaSymbol{)}\AgdaSpace{}%
\AgdaSymbol{=}\AgdaSpace{}%
\AgdaFunction{<-irrefl}\AgdaSpace{}%
\AgdaBound{r}\<%
\\
\>[0]\<%
\end{code}

\begin{code}[hide]%
\>[0]\AgdaFunction{<-irreflexive}\AgdaSpace{}%
\AgdaSymbol{\{}\AgdaBound{a}\AgdaSymbol{\}}\AgdaSpace{}%
\AgdaSymbol{=}\AgdaSpace{}%
\AgdaFunction{J}\AgdaSpace{}%
\AgdaSymbol{(λ}\AgdaSpace{}%
\AgdaBound{b}\AgdaSpace{}%
\AgdaBound{\AgdaUnderscore{}}\AgdaSpace{}%
\AgdaSymbol{→}\AgdaSpace{}%
\AgdaOperator{\AgdaFunction{¬}}\AgdaSpace{}%
\AgdaBound{a}\AgdaSpace{}%
\AgdaOperator{\AgdaDatatype{<}}\AgdaSpace{}%
\AgdaBound{b}\AgdaSymbol{)}\AgdaSpace{}%
\AgdaFunction{<-irrefl}\<%
\\
\>[0]\<%
\end{code}

\newcommand{\Mfacts}{
\begin{code}%
\>[0]\AgdaFunction{MutualOrdIsSet}\AgdaSpace{}%
\AgdaSymbol{:}\AgdaSpace{}%
\AgdaFunction{isSet}\AgdaSpace{}%
\AgdaDatatype{MutualOrd}\<%
\\
\>[0]\AgdaFunction{<IsPropValued}\AgdaSpace{}%
\AgdaSymbol{:}\AgdaSpace{}%
\AgdaFunction{isProp}\AgdaSpace{}%
\AgdaSymbol{(}\AgdaGeneralizable{a}\AgdaSpace{}%
\AgdaOperator{\AgdaDatatype{<}}\AgdaSpace{}%
\AgdaGeneralizable{b}\AgdaSymbol{)}\<%
\\
\>[0]\AgdaFunction{MutualOrd⁼}%
\>[384I]\AgdaSymbol{:}\AgdaSpace{}%
\AgdaSymbol{\{}\AgdaBound{r}\AgdaSpace{}%
\AgdaSymbol{:}\AgdaSpace{}%
\AgdaGeneralizable{a}\AgdaSpace{}%
\AgdaOperator{\AgdaFunction{≥}}\AgdaSpace{}%
\AgdaFunction{fst}\AgdaSpace{}%
\AgdaGeneralizable{b}\AgdaSymbol{\}}\AgdaSpace{}%
\AgdaSymbol{\{}\AgdaBound{s}\AgdaSpace{}%
\AgdaSymbol{:}\AgdaSpace{}%
\AgdaGeneralizable{c}\AgdaSpace{}%
\AgdaOperator{\AgdaFunction{≥}}\AgdaSpace{}%
\AgdaFunction{fst}\AgdaSpace{}%
\AgdaGeneralizable{d}\AgdaSymbol{\}}\AgdaSpace{}%
\AgdaSymbol{→}\AgdaSpace{}%
\AgdaGeneralizable{a}\AgdaSpace{}%
\AgdaOperator{\AgdaFunction{≡}}\AgdaSpace{}%
\AgdaGeneralizable{c}\AgdaSpace{}%
\AgdaSymbol{→}\AgdaSpace{}%
\AgdaGeneralizable{b}\AgdaSpace{}%
\AgdaOperator{\AgdaFunction{≡}}\AgdaSpace{}%
\AgdaGeneralizable{d}\<%
\\
\>[.][@{}l@{}]\<[384I]%
\>[11]\AgdaSymbol{→}\AgdaSpace{}%
\AgdaOperator{\AgdaInductiveConstructor{ω\textasciicircum{}}}\AgdaSpace{}%
\AgdaGeneralizable{a}\AgdaSpace{}%
\AgdaOperator{\AgdaInductiveConstructor{+}}\AgdaSpace{}%
\AgdaGeneralizable{b}\AgdaSpace{}%
\AgdaOperator{\AgdaInductiveConstructor{[}}\AgdaSpace{}%
\AgdaBound{r}\AgdaSpace{}%
\AgdaOperator{\AgdaInductiveConstructor{]}}\AgdaSpace{}%
\AgdaOperator{\AgdaFunction{≡}}\AgdaSpace{}%
\AgdaOperator{\AgdaInductiveConstructor{ω\textasciicircum{}}}\AgdaSpace{}%
\AgdaGeneralizable{c}\AgdaSpace{}%
\AgdaOperator{\AgdaInductiveConstructor{+}}\AgdaSpace{}%
\AgdaGeneralizable{d}\AgdaSpace{}%
\AgdaOperator{\AgdaInductiveConstructor{[}}\AgdaSpace{}%
\AgdaBound{s}\AgdaSpace{}%
\AgdaOperator{\AgdaInductiveConstructor{]}}\<%
\end{code}}

\begin{code}[hide]%
\>[0]\AgdaFunction{≤IsPropValued}\AgdaSpace{}%
\AgdaSymbol{:}\AgdaSpace{}%
\AgdaSymbol{\{}\AgdaBound{a}\AgdaSpace{}%
\AgdaBound{b}\AgdaSpace{}%
\AgdaSymbol{:}\AgdaSpace{}%
\AgdaDatatype{MutualOrd}\AgdaSymbol{\}}\AgdaSpace{}%
\AgdaSymbol{→}\AgdaSpace{}%
\AgdaFunction{isProp}\AgdaSpace{}%
\AgdaSymbol{(}\AgdaBound{a}\AgdaSpace{}%
\AgdaOperator{\AgdaFunction{≤}}\AgdaSpace{}%
\AgdaBound{b}\AgdaSymbol{)}\<%
\\
\>[0]\AgdaFunction{MutualOrdIsDiscrete}\AgdaSpace{}%
\AgdaSymbol{:}\AgdaSpace{}%
\AgdaFunction{Discrete}\AgdaSpace{}%
\AgdaDatatype{MutualOrd}\<%
\\
\\[\AgdaEmptyExtraSkip]%
\>[0]\AgdaFunction{<IsPropValued}%
\>[15]\AgdaInductiveConstructor{<₁}%
\>[24]\AgdaInductiveConstructor{<₁}%
\>[33]\AgdaSymbol{=}\AgdaSpace{}%
\AgdaFunction{refl}\<%
\\
\>[0]\AgdaFunction{<IsPropValued}\AgdaSpace{}%
\AgdaSymbol{(}\AgdaInductiveConstructor{<₂}\AgdaSpace{}%
\AgdaBound{p}\AgdaSymbol{)}%
\>[23]\AgdaSymbol{(}\AgdaInductiveConstructor{<₂}\AgdaSpace{}%
\AgdaBound{q}\AgdaSymbol{)}%
\>[33]\AgdaSymbol{=}\AgdaSpace{}%
\AgdaFunction{cong}\AgdaSpace{}%
\AgdaInductiveConstructor{<₂}\AgdaSpace{}%
\AgdaSymbol{(}\AgdaFunction{<IsPropValued}\AgdaSpace{}%
\AgdaBound{p}\AgdaSpace{}%
\AgdaBound{q}\AgdaSymbol{)}\<%
\\
\>[0]\AgdaFunction{<IsPropValued}\AgdaSpace{}%
\AgdaSymbol{(}\AgdaInductiveConstructor{<₂}\AgdaSpace{}%
\AgdaBound{p}\AgdaSymbol{)}%
\>[23]\AgdaSymbol{(}\AgdaInductiveConstructor{<₃}\AgdaSpace{}%
\AgdaBound{e}\AgdaSpace{}%
\AgdaSymbol{\AgdaUnderscore{})}%
\>[33]\AgdaSymbol{=}\AgdaSpace{}%
\AgdaFunction{⊥-elim}\AgdaSpace{}%
\AgdaSymbol{(}\AgdaFunction{<-irreflexive}\AgdaSpace{}%
\AgdaBound{e}\AgdaSpace{}%
\AgdaBound{p}\AgdaSymbol{)}\<%
\\
\>[0]\AgdaFunction{<IsPropValued}\AgdaSpace{}%
\AgdaSymbol{(}\AgdaInductiveConstructor{<₃}\AgdaSpace{}%
\AgdaBound{e}\AgdaSpace{}%
\AgdaSymbol{\AgdaUnderscore{})}\AgdaSpace{}%
\AgdaSymbol{(}\AgdaInductiveConstructor{<₂}\AgdaSpace{}%
\AgdaBound{q}\AgdaSymbol{)}%
\>[33]\AgdaSymbol{=}\AgdaSpace{}%
\AgdaFunction{⊥-elim}\AgdaSpace{}%
\AgdaSymbol{(}\AgdaFunction{<-irreflexive}\AgdaSpace{}%
\AgdaBound{e}\AgdaSpace{}%
\AgdaBound{q}\AgdaSymbol{)}\<%
\\
\>[0]\AgdaFunction{<IsPropValued}\AgdaSpace{}%
\AgdaSymbol{(}\AgdaInductiveConstructor{<₃}\AgdaSpace{}%
\AgdaBound{e}\AgdaSpace{}%
\AgdaBound{p}\AgdaSymbol{)}\AgdaSpace{}%
\AgdaSymbol{(}\AgdaInductiveConstructor{<₃}\AgdaSpace{}%
\AgdaBound{e'}\AgdaSpace{}%
\AgdaBound{q}\AgdaSymbol{)}\AgdaSpace{}%
\AgdaSymbol{=}\AgdaSpace{}%
\AgdaFunction{cong₂}\AgdaSpace{}%
\AgdaInductiveConstructor{<₃}\AgdaSpace{}%
\AgdaSymbol{(}\AgdaFunction{MutualOrdIsSet}\AgdaSpace{}%
\AgdaBound{e}\AgdaSpace{}%
\AgdaBound{e'}\AgdaSymbol{)}\AgdaSpace{}%
\AgdaSymbol{(}\AgdaFunction{<IsPropValued}\AgdaSpace{}%
\AgdaBound{p}\AgdaSpace{}%
\AgdaBound{q}\AgdaSymbol{)}\<%
\\
\\[\AgdaEmptyExtraSkip]%
\>[0]\AgdaFunction{≤IsPropValued}\AgdaSpace{}%
\AgdaSymbol{(}\AgdaInductiveConstructor{inj₁}\AgdaSpace{}%
\AgdaBound{p}\AgdaSymbol{)}\AgdaSpace{}%
\AgdaSymbol{(}\AgdaInductiveConstructor{inj₁}\AgdaSpace{}%
\AgdaBound{q}\AgdaSymbol{)}\AgdaSpace{}%
\AgdaSymbol{=}\AgdaSpace{}%
\AgdaFunction{cong}\AgdaSpace{}%
\AgdaInductiveConstructor{inj₁}\AgdaSpace{}%
\AgdaSymbol{(}\AgdaFunction{<IsPropValued}\AgdaSpace{}%
\AgdaBound{p}\AgdaSpace{}%
\AgdaBound{q}\AgdaSymbol{)}\<%
\\
\>[0]\AgdaFunction{≤IsPropValued}\AgdaSpace{}%
\AgdaSymbol{(}\AgdaInductiveConstructor{inj₁}\AgdaSpace{}%
\AgdaBound{p}\AgdaSymbol{)}\AgdaSpace{}%
\AgdaSymbol{(}\AgdaInductiveConstructor{inj₂}\AgdaSpace{}%
\AgdaBound{e}\AgdaSymbol{)}\AgdaSpace{}%
\AgdaSymbol{=}\AgdaSpace{}%
\AgdaFunction{⊥-elim}\AgdaSpace{}%
\AgdaSymbol{(}\AgdaFunction{<-irreflexive}\AgdaSpace{}%
\AgdaSymbol{(}\AgdaBound{e}\AgdaSpace{}%
\AgdaOperator{\AgdaFunction{⁻¹}}\AgdaSymbol{)}\AgdaSpace{}%
\AgdaBound{p}\AgdaSymbol{)}\<%
\\
\>[0]\AgdaFunction{≤IsPropValued}\AgdaSpace{}%
\AgdaSymbol{(}\AgdaInductiveConstructor{inj₂}\AgdaSpace{}%
\AgdaBound{e}\AgdaSymbol{)}\AgdaSpace{}%
\AgdaSymbol{(}\AgdaInductiveConstructor{inj₁}\AgdaSpace{}%
\AgdaBound{q}\AgdaSymbol{)}\AgdaSpace{}%
\AgdaSymbol{=}\AgdaSpace{}%
\AgdaFunction{⊥-elim}\AgdaSpace{}%
\AgdaSymbol{(}\AgdaFunction{<-irreflexive}\AgdaSpace{}%
\AgdaSymbol{(}\AgdaBound{e}\AgdaSpace{}%
\AgdaOperator{\AgdaFunction{⁻¹}}\AgdaSymbol{)}\AgdaSpace{}%
\AgdaBound{q}\AgdaSymbol{)}\<%
\\
\>[0]\AgdaFunction{≤IsPropValued}\AgdaSpace{}%
\AgdaSymbol{(}\AgdaInductiveConstructor{inj₂}\AgdaSpace{}%
\AgdaBound{p}\AgdaSymbol{)}\AgdaSpace{}%
\AgdaSymbol{(}\AgdaInductiveConstructor{inj₂}\AgdaSpace{}%
\AgdaBound{q}\AgdaSymbol{)}\AgdaSpace{}%
\AgdaSymbol{=}\AgdaSpace{}%
\AgdaFunction{cong}\AgdaSpace{}%
\AgdaInductiveConstructor{inj₂}\AgdaSpace{}%
\AgdaSymbol{(}\AgdaFunction{MutualOrdIsSet}\AgdaSpace{}%
\AgdaBound{p}\AgdaSpace{}%
\AgdaBound{q}\AgdaSymbol{)}\<%
\\
\\[\AgdaEmptyExtraSkip]%
\>[0]\AgdaFunction{MutualOrdIsDiscrete}%
\>[21]\AgdaInductiveConstructor{𝟎}%
\>[38]\AgdaInductiveConstructor{𝟎}%
\>[54]\AgdaSymbol{=}\AgdaSpace{}%
\AgdaInductiveConstructor{yes}\AgdaSpace{}%
\AgdaFunction{refl}\<%
\\
\>[0]\AgdaFunction{MutualOrdIsDiscrete}%
\>[21]\AgdaInductiveConstructor{𝟎}%
\>[37]\AgdaSymbol{(}\AgdaOperator{\AgdaInductiveConstructor{ω\textasciicircum{}}}\AgdaSpace{}%
\AgdaBound{b}\AgdaSpace{}%
\AgdaOperator{\AgdaInductiveConstructor{+}}\AgdaSpace{}%
\AgdaBound{d}\AgdaSpace{}%
\AgdaOperator{\AgdaInductiveConstructor{[}}\AgdaSpace{}%
\AgdaBound{s}\AgdaSpace{}%
\AgdaOperator{\AgdaInductiveConstructor{]}}\AgdaSymbol{)}\AgdaSpace{}%
\AgdaSymbol{=}\AgdaSpace{}%
\AgdaInductiveConstructor{no}\AgdaSpace{}%
\AgdaFunction{𝟎≢ω}\<%
\\
\>[0]\AgdaFunction{MutualOrdIsDiscrete}\AgdaSpace{}%
\AgdaSymbol{(}\AgdaOperator{\AgdaInductiveConstructor{ω\textasciicircum{}}}\AgdaSpace{}%
\AgdaBound{a}\AgdaSpace{}%
\AgdaOperator{\AgdaInductiveConstructor{+}}\AgdaSpace{}%
\AgdaBound{b}\AgdaSpace{}%
\AgdaOperator{\AgdaInductiveConstructor{[}}\AgdaSpace{}%
\AgdaBound{r}\AgdaSpace{}%
\AgdaOperator{\AgdaInductiveConstructor{]}}\AgdaSymbol{)}%
\>[38]\AgdaInductiveConstructor{𝟎}%
\>[54]\AgdaSymbol{=}\AgdaSpace{}%
\AgdaInductiveConstructor{no}\AgdaSpace{}%
\AgdaFunction{ω≢𝟎}\<%
\\
\>[0]\AgdaFunction{MutualOrdIsDiscrete}\AgdaSpace{}%
\AgdaSymbol{(}\AgdaOperator{\AgdaInductiveConstructor{ω\textasciicircum{}}}\AgdaSpace{}%
\AgdaBound{a}\AgdaSpace{}%
\AgdaOperator{\AgdaInductiveConstructor{+}}\AgdaSpace{}%
\AgdaBound{b}\AgdaSpace{}%
\AgdaOperator{\AgdaInductiveConstructor{[}}\AgdaSpace{}%
\AgdaBound{r}\AgdaSpace{}%
\AgdaOperator{\AgdaInductiveConstructor{]}}\AgdaSymbol{)}\AgdaSpace{}%
\AgdaSymbol{(}\AgdaOperator{\AgdaInductiveConstructor{ω\textasciicircum{}}}\AgdaSpace{}%
\AgdaBound{c}\AgdaSpace{}%
\AgdaOperator{\AgdaInductiveConstructor{+}}\AgdaSpace{}%
\AgdaBound{d}\AgdaSpace{}%
\AgdaOperator{\AgdaInductiveConstructor{[}}\AgdaSpace{}%
\AgdaBound{s}\AgdaSpace{}%
\AgdaOperator{\AgdaInductiveConstructor{]}}\AgdaSymbol{)}\AgdaSpace{}%
\AgdaKeyword{with}\AgdaSpace{}%
\AgdaFunction{MutualOrdIsDiscrete}\AgdaSpace{}%
\AgdaBound{a}\AgdaSpace{}%
\AgdaBound{c}\<%
\\
\>[0]\AgdaFunction{MutualOrdIsDiscrete}\AgdaSpace{}%
\AgdaSymbol{(}\AgdaOperator{\AgdaInductiveConstructor{ω\textasciicircum{}}}\AgdaSpace{}%
\AgdaBound{a}\AgdaSpace{}%
\AgdaOperator{\AgdaInductiveConstructor{+}}\AgdaSpace{}%
\AgdaBound{b}\AgdaSpace{}%
\AgdaOperator{\AgdaInductiveConstructor{[}}\AgdaSpace{}%
\AgdaBound{r}\AgdaSpace{}%
\AgdaOperator{\AgdaInductiveConstructor{]}}\AgdaSymbol{)}\AgdaSpace{}%
\AgdaSymbol{(}\AgdaOperator{\AgdaInductiveConstructor{ω\textasciicircum{}}}\AgdaSpace{}%
\AgdaBound{c}\AgdaSpace{}%
\AgdaOperator{\AgdaInductiveConstructor{+}}\AgdaSpace{}%
\AgdaBound{d}\AgdaSpace{}%
\AgdaOperator{\AgdaInductiveConstructor{[}}\AgdaSpace{}%
\AgdaBound{s}\AgdaSpace{}%
\AgdaOperator{\AgdaInductiveConstructor{]}}\AgdaSymbol{)}\AgdaSpace{}%
\AgdaSymbol{|}\AgdaSpace{}%
\AgdaInductiveConstructor{yes}\AgdaSpace{}%
\AgdaBound{a≡c}\AgdaSpace{}%
\AgdaKeyword{with}\AgdaSpace{}%
\AgdaFunction{MutualOrdIsDiscrete}\AgdaSpace{}%
\AgdaBound{b}\AgdaSpace{}%
\AgdaBound{d}\<%
\\
\>[0]\AgdaFunction{MutualOrdIsDiscrete}\AgdaSpace{}%
\AgdaSymbol{(}\AgdaOperator{\AgdaInductiveConstructor{ω\textasciicircum{}}}\AgdaSpace{}%
\AgdaBound{a}\AgdaSpace{}%
\AgdaOperator{\AgdaInductiveConstructor{+}}\AgdaSpace{}%
\AgdaBound{b}\AgdaSpace{}%
\AgdaOperator{\AgdaInductiveConstructor{[}}\AgdaSpace{}%
\AgdaBound{r}\AgdaSpace{}%
\AgdaOperator{\AgdaInductiveConstructor{]}}\AgdaSymbol{)}\AgdaSpace{}%
\AgdaSymbol{(}\AgdaOperator{\AgdaInductiveConstructor{ω\textasciicircum{}}}\AgdaSpace{}%
\AgdaBound{c}\AgdaSpace{}%
\AgdaOperator{\AgdaInductiveConstructor{+}}\AgdaSpace{}%
\AgdaBound{d}\AgdaSpace{}%
\AgdaOperator{\AgdaInductiveConstructor{[}}\AgdaSpace{}%
\AgdaBound{s}\AgdaSpace{}%
\AgdaOperator{\AgdaInductiveConstructor{]}}\AgdaSymbol{)}\AgdaSpace{}%
\AgdaSymbol{|}\AgdaSpace{}%
\AgdaInductiveConstructor{yes}\AgdaSpace{}%
\AgdaBound{a≡c}\AgdaSpace{}%
\AgdaSymbol{|}\AgdaSpace{}%
\AgdaInductiveConstructor{yes}\AgdaSpace{}%
\AgdaBound{b≡d}\AgdaSpace{}%
\AgdaSymbol{=}\AgdaSpace{}%
\AgdaInductiveConstructor{yes}\AgdaSpace{}%
\AgdaSymbol{(}\AgdaFunction{MutualOrd⁼}\AgdaSpace{}%
\AgdaBound{a≡c}\AgdaSpace{}%
\AgdaBound{b≡d}\AgdaSymbol{)}\<%
\\
\>[0]\AgdaFunction{MutualOrdIsDiscrete}\AgdaSpace{}%
\AgdaSymbol{(}\AgdaOperator{\AgdaInductiveConstructor{ω\textasciicircum{}}}\AgdaSpace{}%
\AgdaBound{a}\AgdaSpace{}%
\AgdaOperator{\AgdaInductiveConstructor{+}}\AgdaSpace{}%
\AgdaBound{b}\AgdaSpace{}%
\AgdaOperator{\AgdaInductiveConstructor{[}}\AgdaSpace{}%
\AgdaBound{r}\AgdaSpace{}%
\AgdaOperator{\AgdaInductiveConstructor{]}}\AgdaSymbol{)}\AgdaSpace{}%
\AgdaSymbol{(}\AgdaOperator{\AgdaInductiveConstructor{ω\textasciicircum{}}}\AgdaSpace{}%
\AgdaBound{c}\AgdaSpace{}%
\AgdaOperator{\AgdaInductiveConstructor{+}}\AgdaSpace{}%
\AgdaBound{d}\AgdaSpace{}%
\AgdaOperator{\AgdaInductiveConstructor{[}}\AgdaSpace{}%
\AgdaBound{s}\AgdaSpace{}%
\AgdaOperator{\AgdaInductiveConstructor{]}}\AgdaSymbol{)}\AgdaSpace{}%
\AgdaSymbol{|}\AgdaSpace{}%
\AgdaInductiveConstructor{yes}\AgdaSpace{}%
\AgdaBound{a≡c}\AgdaSpace{}%
\AgdaSymbol{|}\AgdaSpace{}%
\AgdaInductiveConstructor{no}%
\>[70]\AgdaBound{b≢d}\AgdaSpace{}%
\AgdaSymbol{=}\AgdaSpace{}%
\AgdaInductiveConstructor{no}\AgdaSpace{}%
\AgdaSymbol{(λ}\AgdaSpace{}%
\AgdaBound{e}\AgdaSpace{}%
\AgdaSymbol{→}\AgdaSpace{}%
\AgdaBound{b≢d}\AgdaSpace{}%
\AgdaSymbol{(}\AgdaFunction{cong}\AgdaSpace{}%
\AgdaFunction{rest}\AgdaSpace{}%
\AgdaBound{e}\AgdaSymbol{))}\<%
\\
\>[0]\AgdaFunction{MutualOrdIsDiscrete}\AgdaSpace{}%
\AgdaSymbol{(}\AgdaOperator{\AgdaInductiveConstructor{ω\textasciicircum{}}}\AgdaSpace{}%
\AgdaBound{a}\AgdaSpace{}%
\AgdaOperator{\AgdaInductiveConstructor{+}}\AgdaSpace{}%
\AgdaBound{b}\AgdaSpace{}%
\AgdaOperator{\AgdaInductiveConstructor{[}}\AgdaSpace{}%
\AgdaBound{r}\AgdaSpace{}%
\AgdaOperator{\AgdaInductiveConstructor{]}}\AgdaSymbol{)}\AgdaSpace{}%
\AgdaSymbol{(}\AgdaOperator{\AgdaInductiveConstructor{ω\textasciicircum{}}}\AgdaSpace{}%
\AgdaBound{c}\AgdaSpace{}%
\AgdaOperator{\AgdaInductiveConstructor{+}}\AgdaSpace{}%
\AgdaBound{d}\AgdaSpace{}%
\AgdaOperator{\AgdaInductiveConstructor{[}}\AgdaSpace{}%
\AgdaBound{s}\AgdaSpace{}%
\AgdaOperator{\AgdaInductiveConstructor{]}}\AgdaSymbol{)}\AgdaSpace{}%
\AgdaSymbol{|}\AgdaSpace{}%
\AgdaInductiveConstructor{no}%
\>[60]\AgdaBound{a≢c}\AgdaSpace{}%
\AgdaSymbol{=}\AgdaSpace{}%
\AgdaInductiveConstructor{no}\AgdaSpace{}%
\AgdaSymbol{(λ}\AgdaSpace{}%
\AgdaBound{e}\AgdaSpace{}%
\AgdaSymbol{→}\AgdaSpace{}%
\AgdaBound{a≢c}\AgdaSpace{}%
\AgdaSymbol{(}\AgdaFunction{cong}\AgdaSpace{}%
\AgdaFunction{fst}\AgdaSpace{}%
\AgdaBound{e}\AgdaSymbol{))}\<%
\\
\\[\AgdaEmptyExtraSkip]%
\>[0]\AgdaFunction{MutualOrdIsSet}\AgdaSpace{}%
\AgdaSymbol{=}\AgdaSpace{}%
\AgdaFunction{Discrete→isSet}\AgdaSpace{}%
\AgdaFunction{MutualOrdIsDiscrete}\AgdaSpace{}%
\AgdaSymbol{\AgdaUnderscore{}}\AgdaSpace{}%
\AgdaSymbol{\AgdaUnderscore{}}\<%
\\
\\[\AgdaEmptyExtraSkip]%
\>[0]\AgdaFunction{MutualOrd⁼}\AgdaSpace{}%
\AgdaSymbol{\{}\AgdaBound{a}\AgdaSymbol{\}}\AgdaSpace{}%
\AgdaSymbol{\{}\AgdaBound{b}\AgdaSymbol{\}}\AgdaSpace{}%
\AgdaBound{a≡c}\AgdaSpace{}%
\AgdaBound{b≡d}\AgdaSpace{}%
\AgdaKeyword{with}\AgdaSpace{}%
\AgdaFunction{fromPath}\AgdaSpace{}%
\AgdaBound{a≡c}\AgdaSpace{}%
\AgdaSymbol{|}\AgdaSpace{}%
\AgdaFunction{fromPath}\AgdaSpace{}%
\AgdaBound{b≡d}\<%
\\
\>[0]\AgdaSymbol{...}\AgdaSpace{}%
\AgdaSymbol{|}\AgdaSpace{}%
\AgdaInductiveConstructor{P.refl}\AgdaSpace{}%
\AgdaSymbol{|}\AgdaSpace{}%
\AgdaInductiveConstructor{P.refl}\AgdaSpace{}%
\AgdaSymbol{=}\AgdaSpace{}%
\AgdaFunction{cong}\AgdaSpace{}%
\AgdaSymbol{(}\AgdaOperator{\AgdaInductiveConstructor{ω\textasciicircum{}}}\AgdaSpace{}%
\AgdaBound{a}\AgdaSpace{}%
\AgdaOperator{\AgdaInductiveConstructor{+}}\AgdaSpace{}%
\AgdaBound{b}\AgdaSpace{}%
\AgdaOperator{\AgdaInductiveConstructor{[\AgdaUnderscore{}]}}\AgdaSymbol{)}\AgdaSpace{}%
\AgdaSymbol{(}\AgdaFunction{≤IsPropValued}\AgdaSpace{}%
\AgdaSymbol{\AgdaUnderscore{}}\AgdaSpace{}%
\AgdaSymbol{\AgdaUnderscore{})}\<%
\\
\>[0]\<%
\end{code}

\newcommand{\Mtri}{
\begin{code}%
\>[0]\AgdaFunction{<-tri}\AgdaSpace{}%
\AgdaSymbol{:}\AgdaSpace{}%
\AgdaSymbol{(}\AgdaBound{a}\AgdaSpace{}%
\AgdaBound{b}\AgdaSpace{}%
\AgdaSymbol{:}\AgdaSpace{}%
\AgdaDatatype{MutualOrd}\AgdaSymbol{)}\AgdaSpace{}%
\AgdaSymbol{→}\AgdaSpace{}%
\AgdaBound{a}\AgdaSpace{}%
\AgdaOperator{\AgdaDatatype{<}}\AgdaSpace{}%
\AgdaBound{b}\AgdaSpace{}%
\AgdaOperator{\AgdaDatatype{⊎}}\AgdaSpace{}%
\AgdaBound{a}\AgdaSpace{}%
\AgdaOperator{\AgdaFunction{≥}}\AgdaSpace{}%
\AgdaBound{b}\<%
\end{code}}

\begin{code}[hide]%
\>[0]\<%
\\
\>[0]\AgdaFunction{<-tri}%
\>[7]\AgdaInductiveConstructor{𝟎}%
\>[24]\AgdaInductiveConstructor{𝟎}%
\>[40]\AgdaSymbol{=}\AgdaSpace{}%
\AgdaInductiveConstructor{inj₂}\AgdaSpace{}%
\AgdaSymbol{(}\AgdaInductiveConstructor{inj₂}\AgdaSpace{}%
\AgdaFunction{refl}\AgdaSymbol{)}\<%
\\
\>[0]\AgdaFunction{<-tri}%
\>[7]\AgdaInductiveConstructor{𝟎}%
\>[23]\AgdaSymbol{(}\AgdaOperator{\AgdaInductiveConstructor{ω\textasciicircum{}}}\AgdaSpace{}%
\AgdaBound{b}\AgdaSpace{}%
\AgdaOperator{\AgdaInductiveConstructor{+}}\AgdaSpace{}%
\AgdaBound{d}\AgdaSpace{}%
\AgdaOperator{\AgdaInductiveConstructor{[}}\AgdaSpace{}%
\AgdaSymbol{\AgdaUnderscore{}}\AgdaSpace{}%
\AgdaOperator{\AgdaInductiveConstructor{]}}\AgdaSymbol{)}\AgdaSpace{}%
\AgdaSymbol{=}\AgdaSpace{}%
\AgdaInductiveConstructor{inj₁}\AgdaSpace{}%
\AgdaInductiveConstructor{<₁}\<%
\\
\>[0]\AgdaFunction{<-tri}\AgdaSpace{}%
\AgdaSymbol{(}\AgdaOperator{\AgdaInductiveConstructor{ω\textasciicircum{}}}\AgdaSpace{}%
\AgdaBound{a}\AgdaSpace{}%
\AgdaOperator{\AgdaInductiveConstructor{+}}\AgdaSpace{}%
\AgdaBound{c}\AgdaSpace{}%
\AgdaOperator{\AgdaInductiveConstructor{[}}\AgdaSpace{}%
\AgdaSymbol{\AgdaUnderscore{}}\AgdaSpace{}%
\AgdaOperator{\AgdaInductiveConstructor{]}}\AgdaSymbol{)}%
\>[24]\AgdaInductiveConstructor{𝟎}%
\>[40]\AgdaSymbol{=}\AgdaSpace{}%
\AgdaInductiveConstructor{inj₂}\AgdaSpace{}%
\AgdaSymbol{(}\AgdaInductiveConstructor{inj₁}\AgdaSpace{}%
\AgdaInductiveConstructor{<₁}\AgdaSymbol{)}\<%
\\
\>[0]\AgdaFunction{<-tri}\AgdaSpace{}%
\AgdaSymbol{(}\AgdaOperator{\AgdaInductiveConstructor{ω\textasciicircum{}}}\AgdaSpace{}%
\AgdaBound{a}\AgdaSpace{}%
\AgdaOperator{\AgdaInductiveConstructor{+}}\AgdaSpace{}%
\AgdaBound{c}\AgdaSpace{}%
\AgdaOperator{\AgdaInductiveConstructor{[}}\AgdaSpace{}%
\AgdaSymbol{\AgdaUnderscore{}}\AgdaSpace{}%
\AgdaOperator{\AgdaInductiveConstructor{]}}\AgdaSymbol{)}\AgdaSpace{}%
\AgdaSymbol{(}\AgdaOperator{\AgdaInductiveConstructor{ω\textasciicircum{}}}\AgdaSpace{}%
\AgdaBound{b}\AgdaSpace{}%
\AgdaOperator{\AgdaInductiveConstructor{+}}\AgdaSpace{}%
\AgdaBound{d}\AgdaSpace{}%
\AgdaOperator{\AgdaInductiveConstructor{[}}\AgdaSpace{}%
\AgdaSymbol{\AgdaUnderscore{}}\AgdaSpace{}%
\AgdaOperator{\AgdaInductiveConstructor{]}}\AgdaSymbol{)}\AgdaSpace{}%
\AgdaKeyword{with}\AgdaSpace{}%
\AgdaFunction{<-tri}\AgdaSpace{}%
\AgdaBound{a}\AgdaSpace{}%
\AgdaBound{b}\<%
\\
\>[0]\AgdaFunction{<-tri}\AgdaSpace{}%
\AgdaSymbol{(}\AgdaOperator{\AgdaInductiveConstructor{ω\textasciicircum{}}}\AgdaSpace{}%
\AgdaBound{a}\AgdaSpace{}%
\AgdaOperator{\AgdaInductiveConstructor{+}}\AgdaSpace{}%
\AgdaBound{c}\AgdaSpace{}%
\AgdaOperator{\AgdaInductiveConstructor{[}}\AgdaSpace{}%
\AgdaSymbol{\AgdaUnderscore{}}\AgdaSpace{}%
\AgdaOperator{\AgdaInductiveConstructor{]}}\AgdaSymbol{)}\AgdaSpace{}%
\AgdaSymbol{(}\AgdaOperator{\AgdaInductiveConstructor{ω\textasciicircum{}}}\AgdaSpace{}%
\AgdaBound{b}\AgdaSpace{}%
\AgdaOperator{\AgdaInductiveConstructor{+}}\AgdaSpace{}%
\AgdaBound{d}\AgdaSpace{}%
\AgdaOperator{\AgdaInductiveConstructor{[}}\AgdaSpace{}%
\AgdaSymbol{\AgdaUnderscore{}}\AgdaSpace{}%
\AgdaOperator{\AgdaInductiveConstructor{]}}\AgdaSymbol{)}\AgdaSpace{}%
\AgdaSymbol{|}\AgdaSpace{}%
\AgdaInductiveConstructor{inj₁}%
\>[53]\AgdaBound{a<b}%
\>[58]\AgdaSymbol{=}\AgdaSpace{}%
\AgdaInductiveConstructor{inj₁}\AgdaSpace{}%
\AgdaSymbol{(}\AgdaInductiveConstructor{<₂}\AgdaSpace{}%
\AgdaBound{a<b}\AgdaSymbol{)}\<%
\\
\>[0]\AgdaFunction{<-tri}\AgdaSpace{}%
\AgdaSymbol{(}\AgdaOperator{\AgdaInductiveConstructor{ω\textasciicircum{}}}\AgdaSpace{}%
\AgdaBound{a}\AgdaSpace{}%
\AgdaOperator{\AgdaInductiveConstructor{+}}\AgdaSpace{}%
\AgdaBound{c}\AgdaSpace{}%
\AgdaOperator{\AgdaInductiveConstructor{[}}\AgdaSpace{}%
\AgdaSymbol{\AgdaUnderscore{}}\AgdaSpace{}%
\AgdaOperator{\AgdaInductiveConstructor{]}}\AgdaSymbol{)}\AgdaSpace{}%
\AgdaSymbol{(}\AgdaOperator{\AgdaInductiveConstructor{ω\textasciicircum{}}}\AgdaSpace{}%
\AgdaBound{b}\AgdaSpace{}%
\AgdaOperator{\AgdaInductiveConstructor{+}}\AgdaSpace{}%
\AgdaBound{d}\AgdaSpace{}%
\AgdaOperator{\AgdaInductiveConstructor{[}}\AgdaSpace{}%
\AgdaSymbol{\AgdaUnderscore{}}\AgdaSpace{}%
\AgdaOperator{\AgdaInductiveConstructor{]}}\AgdaSymbol{)}\AgdaSpace{}%
\AgdaSymbol{|}\AgdaSpace{}%
\AgdaInductiveConstructor{inj₂}\AgdaSpace{}%
\AgdaSymbol{(}\AgdaInductiveConstructor{inj₁}\AgdaSpace{}%
\AgdaBound{a>b}\AgdaSymbol{)}\AgdaSpace{}%
\AgdaSymbol{=}\AgdaSpace{}%
\AgdaInductiveConstructor{inj₂}\AgdaSpace{}%
\AgdaSymbol{(}\AgdaInductiveConstructor{inj₁}\AgdaSpace{}%
\AgdaSymbol{(}\AgdaInductiveConstructor{<₂}\AgdaSpace{}%
\AgdaBound{a>b}\AgdaSymbol{))}\<%
\\
\>[0]\AgdaFunction{<-tri}\AgdaSpace{}%
\AgdaSymbol{(}\AgdaOperator{\AgdaInductiveConstructor{ω\textasciicircum{}}}\AgdaSpace{}%
\AgdaBound{a}\AgdaSpace{}%
\AgdaOperator{\AgdaInductiveConstructor{+}}\AgdaSpace{}%
\AgdaBound{c}\AgdaSpace{}%
\AgdaOperator{\AgdaInductiveConstructor{[}}\AgdaSpace{}%
\AgdaSymbol{\AgdaUnderscore{}}\AgdaSpace{}%
\AgdaOperator{\AgdaInductiveConstructor{]}}\AgdaSymbol{)}\AgdaSpace{}%
\AgdaSymbol{(}\AgdaOperator{\AgdaInductiveConstructor{ω\textasciicircum{}}}\AgdaSpace{}%
\AgdaBound{b}\AgdaSpace{}%
\AgdaOperator{\AgdaInductiveConstructor{+}}\AgdaSpace{}%
\AgdaBound{d}\AgdaSpace{}%
\AgdaOperator{\AgdaInductiveConstructor{[}}\AgdaSpace{}%
\AgdaSymbol{\AgdaUnderscore{}}\AgdaSpace{}%
\AgdaOperator{\AgdaInductiveConstructor{]}}\AgdaSymbol{)}\AgdaSpace{}%
\AgdaSymbol{|}\AgdaSpace{}%
\AgdaInductiveConstructor{inj₂}\AgdaSpace{}%
\AgdaSymbol{(}\AgdaInductiveConstructor{inj₂}\AgdaSpace{}%
\AgdaBound{a=b}\AgdaSymbol{)}\AgdaSpace{}%
\AgdaKeyword{with}\AgdaSpace{}%
\AgdaFunction{<-tri}\AgdaSpace{}%
\AgdaBound{c}\AgdaSpace{}%
\AgdaBound{d}\<%
\\
\>[0]\AgdaFunction{<-tri}\AgdaSpace{}%
\AgdaSymbol{(}\AgdaOperator{\AgdaInductiveConstructor{ω\textasciicircum{}}}\AgdaSpace{}%
\AgdaBound{a}\AgdaSpace{}%
\AgdaOperator{\AgdaInductiveConstructor{+}}\AgdaSpace{}%
\AgdaBound{c}\AgdaSpace{}%
\AgdaOperator{\AgdaInductiveConstructor{[}}\AgdaSpace{}%
\AgdaSymbol{\AgdaUnderscore{}}\AgdaSpace{}%
\AgdaOperator{\AgdaInductiveConstructor{]}}\AgdaSymbol{)}\AgdaSpace{}%
\AgdaSymbol{(}\AgdaOperator{\AgdaInductiveConstructor{ω\textasciicircum{}}}\AgdaSpace{}%
\AgdaBound{b}\AgdaSpace{}%
\AgdaOperator{\AgdaInductiveConstructor{+}}\AgdaSpace{}%
\AgdaBound{d}\AgdaSpace{}%
\AgdaOperator{\AgdaInductiveConstructor{[}}\AgdaSpace{}%
\AgdaSymbol{\AgdaUnderscore{}}\AgdaSpace{}%
\AgdaOperator{\AgdaInductiveConstructor{]}}\AgdaSymbol{)}\AgdaSpace{}%
\AgdaSymbol{|}\AgdaSpace{}%
\AgdaInductiveConstructor{inj₂}\AgdaSpace{}%
\AgdaSymbol{(}\AgdaInductiveConstructor{inj₂}\AgdaSpace{}%
\AgdaBound{a=b}\AgdaSymbol{)}\AgdaSpace{}%
\AgdaSymbol{|}\AgdaSpace{}%
\AgdaInductiveConstructor{inj₁}%
\>[71]\AgdaBound{c<d}%
\>[76]\AgdaSymbol{=}\AgdaSpace{}%
\AgdaInductiveConstructor{inj₁}\AgdaSpace{}%
\AgdaSymbol{(}\AgdaInductiveConstructor{<₃}\AgdaSpace{}%
\AgdaBound{a=b}\AgdaSpace{}%
\AgdaBound{c<d}\AgdaSymbol{)}\<%
\\
\>[0]\AgdaFunction{<-tri}\AgdaSpace{}%
\AgdaSymbol{(}\AgdaOperator{\AgdaInductiveConstructor{ω\textasciicircum{}}}\AgdaSpace{}%
\AgdaBound{a}\AgdaSpace{}%
\AgdaOperator{\AgdaInductiveConstructor{+}}\AgdaSpace{}%
\AgdaBound{c}\AgdaSpace{}%
\AgdaOperator{\AgdaInductiveConstructor{[}}\AgdaSpace{}%
\AgdaSymbol{\AgdaUnderscore{}}\AgdaSpace{}%
\AgdaOperator{\AgdaInductiveConstructor{]}}\AgdaSymbol{)}\AgdaSpace{}%
\AgdaSymbol{(}\AgdaOperator{\AgdaInductiveConstructor{ω\textasciicircum{}}}\AgdaSpace{}%
\AgdaBound{b}\AgdaSpace{}%
\AgdaOperator{\AgdaInductiveConstructor{+}}\AgdaSpace{}%
\AgdaBound{d}\AgdaSpace{}%
\AgdaOperator{\AgdaInductiveConstructor{[}}\AgdaSpace{}%
\AgdaSymbol{\AgdaUnderscore{}}\AgdaSpace{}%
\AgdaOperator{\AgdaInductiveConstructor{]}}\AgdaSymbol{)}\AgdaSpace{}%
\AgdaSymbol{|}\AgdaSpace{}%
\AgdaInductiveConstructor{inj₂}\AgdaSpace{}%
\AgdaSymbol{(}\AgdaInductiveConstructor{inj₂}\AgdaSpace{}%
\AgdaBound{a=b}\AgdaSymbol{)}\AgdaSpace{}%
\AgdaSymbol{|}\AgdaSpace{}%
\AgdaInductiveConstructor{inj₂}\AgdaSpace{}%
\AgdaSymbol{(}\AgdaInductiveConstructor{inj₁}\AgdaSpace{}%
\AgdaBound{c>d}\AgdaSymbol{)}\AgdaSpace{}%
\AgdaSymbol{=}\AgdaSpace{}%
\AgdaInductiveConstructor{inj₂}\AgdaSpace{}%
\AgdaSymbol{(}\AgdaInductiveConstructor{inj₁}\AgdaSpace{}%
\AgdaSymbol{(}\AgdaInductiveConstructor{<₃}\AgdaSpace{}%
\AgdaSymbol{(}\AgdaBound{a=b}\AgdaSpace{}%
\AgdaOperator{\AgdaFunction{⁻¹}}\AgdaSymbol{)}\AgdaSpace{}%
\AgdaBound{c>d}\AgdaSymbol{))}\<%
\\
\>[0]\AgdaFunction{<-tri}\AgdaSpace{}%
\AgdaSymbol{(}\AgdaOperator{\AgdaInductiveConstructor{ω\textasciicircum{}}}\AgdaSpace{}%
\AgdaBound{a}\AgdaSpace{}%
\AgdaOperator{\AgdaInductiveConstructor{+}}\AgdaSpace{}%
\AgdaBound{c}\AgdaSpace{}%
\AgdaOperator{\AgdaInductiveConstructor{[}}\AgdaSpace{}%
\AgdaSymbol{\AgdaUnderscore{}}\AgdaSpace{}%
\AgdaOperator{\AgdaInductiveConstructor{]}}\AgdaSymbol{)}\AgdaSpace{}%
\AgdaSymbol{(}\AgdaOperator{\AgdaInductiveConstructor{ω\textasciicircum{}}}\AgdaSpace{}%
\AgdaBound{b}\AgdaSpace{}%
\AgdaOperator{\AgdaInductiveConstructor{+}}\AgdaSpace{}%
\AgdaBound{d}\AgdaSpace{}%
\AgdaOperator{\AgdaInductiveConstructor{[}}\AgdaSpace{}%
\AgdaSymbol{\AgdaUnderscore{}}\AgdaSpace{}%
\AgdaOperator{\AgdaInductiveConstructor{]}}\AgdaSymbol{)}\AgdaSpace{}%
\AgdaSymbol{|}\AgdaSpace{}%
\AgdaInductiveConstructor{inj₂}\AgdaSpace{}%
\AgdaSymbol{(}\AgdaInductiveConstructor{inj₂}\AgdaSpace{}%
\AgdaBound{a=b}\AgdaSymbol{)}\AgdaSpace{}%
\AgdaSymbol{|}\AgdaSpace{}%
\AgdaInductiveConstructor{inj₂}\AgdaSpace{}%
\AgdaSymbol{(}\AgdaInductiveConstructor{inj₂}\AgdaSpace{}%
\AgdaBound{c=d}\AgdaSymbol{)}\AgdaSpace{}%
\AgdaSymbol{=}\AgdaSpace{}%
\AgdaInductiveConstructor{inj₂}\AgdaSpace{}%
\AgdaSymbol{(}\AgdaInductiveConstructor{inj₂}\AgdaSpace{}%
\AgdaSymbol{(}\AgdaFunction{MutualOrd⁼}\AgdaSpace{}%
\AgdaBound{a=b}\AgdaSpace{}%
\AgdaBound{c=d}\AgdaSymbol{))}\<%
\\
\\[\AgdaEmptyExtraSkip]%
\\[\AgdaEmptyExtraSkip]%
\>[0]\AgdaFunction{<-trans}\AgdaSpace{}%
\AgdaSymbol{:}\AgdaSpace{}%
\AgdaSymbol{\{}\AgdaBound{a}\AgdaSpace{}%
\AgdaBound{b}\AgdaSpace{}%
\AgdaBound{c}\AgdaSpace{}%
\AgdaSymbol{:}\AgdaSpace{}%
\AgdaDatatype{MutualOrd}\AgdaSymbol{\}}\AgdaSpace{}%
\AgdaSymbol{→}\AgdaSpace{}%
\AgdaBound{a}\AgdaSpace{}%
\AgdaOperator{\AgdaDatatype{<}}\AgdaSpace{}%
\AgdaBound{b}\AgdaSpace{}%
\AgdaSymbol{→}\AgdaSpace{}%
\AgdaBound{b}\AgdaSpace{}%
\AgdaOperator{\AgdaDatatype{<}}\AgdaSpace{}%
\AgdaBound{c}\AgdaSpace{}%
\AgdaSymbol{→}\AgdaSpace{}%
\AgdaBound{a}\AgdaSpace{}%
\AgdaOperator{\AgdaDatatype{<}}\AgdaSpace{}%
\AgdaBound{c}\<%
\\
\>[0]\AgdaFunction{<-trans}%
\>[9]\AgdaInductiveConstructor{<₁}%
\>[17]\AgdaSymbol{(}\AgdaInductiveConstructor{<₂}\AgdaSpace{}%
\AgdaSymbol{\AgdaUnderscore{})}%
\>[26]\AgdaSymbol{=}\AgdaSpace{}%
\AgdaInductiveConstructor{<₁}\<%
\\
\>[0]\AgdaFunction{<-trans}%
\>[9]\AgdaInductiveConstructor{<₁}%
\>[17]\AgdaSymbol{(}\AgdaInductiveConstructor{<₃}\AgdaSpace{}%
\AgdaSymbol{\AgdaUnderscore{}}\AgdaSpace{}%
\AgdaSymbol{\AgdaUnderscore{})}\AgdaSpace{}%
\AgdaSymbol{=}\AgdaSpace{}%
\AgdaInductiveConstructor{<₁}\<%
\\
\>[0]\AgdaFunction{<-trans}\AgdaSpace{}%
\AgdaSymbol{(}\AgdaInductiveConstructor{<₂}\AgdaSpace{}%
\AgdaBound{r}\AgdaSymbol{)}%
\>[17]\AgdaSymbol{(}\AgdaInductiveConstructor{<₂}\AgdaSpace{}%
\AgdaBound{s}\AgdaSymbol{)}%
\>[26]\AgdaSymbol{=}\AgdaSpace{}%
\AgdaInductiveConstructor{<₂}\AgdaSpace{}%
\AgdaSymbol{(}\AgdaFunction{<-trans}\AgdaSpace{}%
\AgdaBound{r}\AgdaSpace{}%
\AgdaBound{s}\AgdaSymbol{)}\<%
\\
\>[0]\AgdaFunction{<-trans}\AgdaSpace{}%
\AgdaSymbol{(}\AgdaInductiveConstructor{<₂}\AgdaSpace{}%
\AgdaBound{r}\AgdaSymbol{)}%
\>[17]\AgdaSymbol{(}\AgdaInductiveConstructor{<₃}\AgdaSpace{}%
\AgdaBound{p}\AgdaSpace{}%
\AgdaSymbol{\AgdaUnderscore{})}\AgdaSpace{}%
\AgdaSymbol{=}\AgdaSpace{}%
\AgdaInductiveConstructor{<₂}\AgdaSpace{}%
\AgdaSymbol{(}\AgdaFunction{subst}\AgdaSpace{}%
\AgdaSymbol{(λ}\AgdaSpace{}%
\AgdaBound{x}\AgdaSpace{}%
\AgdaSymbol{→}\AgdaSpace{}%
\AgdaSymbol{\AgdaUnderscore{}}\AgdaSpace{}%
\AgdaOperator{\AgdaDatatype{<}}\AgdaSpace{}%
\AgdaBound{x}\AgdaSymbol{)}\AgdaSpace{}%
\AgdaBound{p}\AgdaSpace{}%
\AgdaBound{r}\AgdaSymbol{)}\<%
\\
\>[0]\AgdaFunction{<-trans}\AgdaSpace{}%
\AgdaSymbol{(}\AgdaInductiveConstructor{<₃}\AgdaSpace{}%
\AgdaBound{p}\AgdaSpace{}%
\AgdaSymbol{\AgdaUnderscore{})}\AgdaSpace{}%
\AgdaSymbol{(}\AgdaInductiveConstructor{<₂}\AgdaSpace{}%
\AgdaBound{s}\AgdaSymbol{)}%
\>[26]\AgdaSymbol{=}\AgdaSpace{}%
\AgdaInductiveConstructor{<₂}\AgdaSpace{}%
\AgdaSymbol{(}\AgdaFunction{subst}\AgdaSpace{}%
\AgdaSymbol{(λ}\AgdaSpace{}%
\AgdaBound{x}\AgdaSpace{}%
\AgdaSymbol{→}\AgdaSpace{}%
\AgdaBound{x}\AgdaSpace{}%
\AgdaOperator{\AgdaDatatype{<}}\AgdaSpace{}%
\AgdaSymbol{\AgdaUnderscore{})}\AgdaSpace{}%
\AgdaSymbol{(}\AgdaBound{p}\AgdaSpace{}%
\AgdaOperator{\AgdaFunction{⁻¹}}\AgdaSymbol{)}\AgdaSpace{}%
\AgdaBound{s}\AgdaSymbol{)}\<%
\\
\>[0]\AgdaFunction{<-trans}\AgdaSpace{}%
\AgdaSymbol{(}\AgdaInductiveConstructor{<₃}\AgdaSpace{}%
\AgdaBound{p}\AgdaSpace{}%
\AgdaBound{r}\AgdaSymbol{)}\AgdaSpace{}%
\AgdaSymbol{(}\AgdaInductiveConstructor{<₃}\AgdaSpace{}%
\AgdaBound{q}\AgdaSpace{}%
\AgdaBound{s}\AgdaSymbol{)}\AgdaSpace{}%
\AgdaSymbol{=}\AgdaSpace{}%
\AgdaInductiveConstructor{<₃}\AgdaSpace{}%
\AgdaSymbol{(}\AgdaBound{p}\AgdaSpace{}%
\AgdaOperator{\AgdaFunction{∙}}\AgdaSpace{}%
\AgdaBound{q}\AgdaSymbol{)}\AgdaSpace{}%
\AgdaSymbol{(}\AgdaFunction{<-trans}\AgdaSpace{}%
\AgdaBound{r}\AgdaSpace{}%
\AgdaBound{s}\AgdaSymbol{)}\<%
\\
\\[\AgdaEmptyExtraSkip]%
\>[0]\AgdaFunction{≤-trans}\AgdaSpace{}%
\AgdaSymbol{:}\AgdaSpace{}%
\AgdaSymbol{\{}\AgdaBound{a}\AgdaSpace{}%
\AgdaBound{b}\AgdaSpace{}%
\AgdaBound{c}\AgdaSpace{}%
\AgdaSymbol{:}\AgdaSpace{}%
\AgdaDatatype{MutualOrd}\AgdaSymbol{\}}\AgdaSpace{}%
\AgdaSymbol{→}\AgdaSpace{}%
\AgdaBound{a}\AgdaSpace{}%
\AgdaOperator{\AgdaFunction{≤}}\AgdaSpace{}%
\AgdaBound{b}\AgdaSpace{}%
\AgdaSymbol{→}\AgdaSpace{}%
\AgdaBound{b}\AgdaSpace{}%
\AgdaOperator{\AgdaFunction{≤}}\AgdaSpace{}%
\AgdaBound{c}\AgdaSpace{}%
\AgdaSymbol{→}\AgdaSpace{}%
\AgdaBound{a}\AgdaSpace{}%
\AgdaOperator{\AgdaFunction{≤}}\AgdaSpace{}%
\AgdaBound{c}\<%
\\
\>[0]\AgdaFunction{≤-trans}\AgdaSpace{}%
\AgdaSymbol{(}\AgdaInductiveConstructor{inj₁}\AgdaSpace{}%
\AgdaBound{a<b}\AgdaSymbol{)}\AgdaSpace{}%
\AgdaSymbol{(}\AgdaInductiveConstructor{inj₁}\AgdaSpace{}%
\AgdaBound{b<c}\AgdaSymbol{)}\AgdaSpace{}%
\AgdaSymbol{=}\AgdaSpace{}%
\AgdaInductiveConstructor{inj₁}\AgdaSpace{}%
\AgdaSymbol{(}\AgdaFunction{<-trans}\AgdaSpace{}%
\AgdaBound{a<b}\AgdaSpace{}%
\AgdaBound{b<c}\AgdaSymbol{)}\<%
\\
\>[0]\AgdaFunction{≤-trans}\AgdaSpace{}%
\AgdaSymbol{(}\AgdaInductiveConstructor{inj₁}\AgdaSpace{}%
\AgdaBound{a<b}\AgdaSymbol{)}\AgdaSpace{}%
\AgdaSymbol{(}\AgdaInductiveConstructor{inj₂}\AgdaSpace{}%
\AgdaBound{c=b}\AgdaSymbol{)}\AgdaSpace{}%
\AgdaSymbol{=}\AgdaSpace{}%
\AgdaInductiveConstructor{inj₁}\AgdaSpace{}%
\AgdaSymbol{(}\AgdaFunction{subst}\AgdaSpace{}%
\AgdaSymbol{(λ}\AgdaSpace{}%
\AgdaBound{x}\AgdaSpace{}%
\AgdaSymbol{→}\AgdaSpace{}%
\AgdaSymbol{\AgdaUnderscore{}}\AgdaSpace{}%
\AgdaOperator{\AgdaDatatype{<}}\AgdaSpace{}%
\AgdaBound{x}\AgdaSymbol{)}\AgdaSpace{}%
\AgdaSymbol{(}\AgdaBound{c=b}\AgdaSpace{}%
\AgdaOperator{\AgdaFunction{⁻¹}}\AgdaSymbol{)}\AgdaSpace{}%
\AgdaBound{a<b}\AgdaSymbol{)}\<%
\\
\>[0]\AgdaFunction{≤-trans}\AgdaSpace{}%
\AgdaSymbol{(}\AgdaInductiveConstructor{inj₂}\AgdaSpace{}%
\AgdaBound{b=a}\AgdaSymbol{)}\AgdaSpace{}%
\AgdaSymbol{(}\AgdaInductiveConstructor{inj₁}\AgdaSpace{}%
\AgdaBound{b<c}\AgdaSymbol{)}\AgdaSpace{}%
\AgdaSymbol{=}\AgdaSpace{}%
\AgdaInductiveConstructor{inj₁}\AgdaSpace{}%
\AgdaSymbol{(}\AgdaFunction{subst}\AgdaSpace{}%
\AgdaSymbol{(λ}\AgdaSpace{}%
\AgdaBound{x}\AgdaSpace{}%
\AgdaSymbol{→}\AgdaSpace{}%
\AgdaBound{x}\AgdaSpace{}%
\AgdaOperator{\AgdaDatatype{<}}\AgdaSpace{}%
\AgdaSymbol{\AgdaUnderscore{})}\AgdaSpace{}%
\AgdaBound{b=a}\AgdaSpace{}%
\AgdaBound{b<c}\AgdaSymbol{)}\<%
\\
\>[0]\AgdaFunction{≤-trans}\AgdaSpace{}%
\AgdaSymbol{(}\AgdaInductiveConstructor{inj₂}\AgdaSpace{}%
\AgdaBound{b=a}\AgdaSymbol{)}\AgdaSpace{}%
\AgdaSymbol{(}\AgdaInductiveConstructor{inj₂}\AgdaSpace{}%
\AgdaBound{c=b}\AgdaSymbol{)}\AgdaSpace{}%
\AgdaSymbol{=}\AgdaSpace{}%
\AgdaInductiveConstructor{inj₂}\AgdaSpace{}%
\AgdaSymbol{(}\AgdaBound{c=b}\AgdaSpace{}%
\AgdaOperator{\AgdaFunction{∙}}\AgdaSpace{}%
\AgdaBound{b=a}\AgdaSymbol{)}\<%
\\
\\[\AgdaEmptyExtraSkip]%
\>[0]\AgdaFunction{<≤-trans}\AgdaSpace{}%
\AgdaSymbol{:}\AgdaSpace{}%
\AgdaSymbol{\{}\AgdaBound{a}\AgdaSpace{}%
\AgdaBound{b}\AgdaSpace{}%
\AgdaBound{c}\AgdaSpace{}%
\AgdaSymbol{:}\AgdaSpace{}%
\AgdaDatatype{MutualOrd}\AgdaSymbol{\}}\AgdaSpace{}%
\AgdaSymbol{→}\AgdaSpace{}%
\AgdaBound{a}\AgdaSpace{}%
\AgdaOperator{\AgdaDatatype{<}}\AgdaSpace{}%
\AgdaBound{b}\AgdaSpace{}%
\AgdaSymbol{→}\AgdaSpace{}%
\AgdaBound{b}\AgdaSpace{}%
\AgdaOperator{\AgdaFunction{≤}}\AgdaSpace{}%
\AgdaBound{c}\AgdaSpace{}%
\AgdaSymbol{→}\AgdaSpace{}%
\AgdaBound{a}\AgdaSpace{}%
\AgdaOperator{\AgdaDatatype{<}}\AgdaSpace{}%
\AgdaBound{c}\<%
\\
\>[0]\AgdaFunction{<≤-trans}\AgdaSpace{}%
\AgdaBound{a<b}\AgdaSpace{}%
\AgdaSymbol{(}\AgdaInductiveConstructor{inj₁}\AgdaSpace{}%
\AgdaBound{b<c}\AgdaSymbol{)}\AgdaSpace{}%
\AgdaSymbol{=}\AgdaSpace{}%
\AgdaFunction{<-trans}\AgdaSpace{}%
\AgdaBound{a<b}\AgdaSpace{}%
\AgdaBound{b<c}\<%
\\
\>[0]\AgdaFunction{<≤-trans}\AgdaSpace{}%
\AgdaBound{a<b}\AgdaSpace{}%
\AgdaSymbol{(}\AgdaInductiveConstructor{inj₂}\AgdaSpace{}%
\AgdaBound{c≡b}\AgdaSymbol{)}\AgdaSpace{}%
\AgdaSymbol{=}\AgdaSpace{}%
\AgdaFunction{subst}\AgdaSpace{}%
\AgdaSymbol{(\AgdaUnderscore{}}\AgdaSpace{}%
\AgdaOperator{\AgdaDatatype{<\AgdaUnderscore{}}}\AgdaSymbol{)}\AgdaSpace{}%
\AgdaSymbol{(}\AgdaBound{c≡b}\AgdaSpace{}%
\AgdaOperator{\AgdaFunction{⁻¹}}\AgdaSymbol{)}\AgdaSpace{}%
\AgdaBound{a<b}\<%
\\
\\[\AgdaEmptyExtraSkip]%
\>[0]\AgdaFunction{Lm[≥→¬<]}\AgdaSpace{}%
\AgdaSymbol{:}\AgdaSpace{}%
\AgdaSymbol{\{}\AgdaBound{a}\AgdaSpace{}%
\AgdaBound{b}\AgdaSpace{}%
\AgdaSymbol{:}\AgdaSpace{}%
\AgdaDatatype{MutualOrd}\AgdaSymbol{\}}\AgdaSpace{}%
\AgdaSymbol{→}\AgdaSpace{}%
\AgdaBound{a}\AgdaSpace{}%
\AgdaOperator{\AgdaFunction{≥}}\AgdaSpace{}%
\AgdaBound{b}\AgdaSpace{}%
\AgdaSymbol{→}\AgdaSpace{}%
\AgdaOperator{\AgdaFunction{¬}}\AgdaSpace{}%
\AgdaBound{a}\AgdaSpace{}%
\AgdaOperator{\AgdaDatatype{<}}\AgdaSpace{}%
\AgdaBound{b}\<%
\\
\>[0]\AgdaFunction{Lm[≥→¬<]}\AgdaSpace{}%
\AgdaSymbol{(}\AgdaInductiveConstructor{inj₁}\AgdaSpace{}%
\AgdaBound{b<a}\AgdaSymbol{)}\AgdaSpace{}%
\AgdaBound{a<b}\AgdaSpace{}%
\AgdaSymbol{=}\AgdaSpace{}%
\AgdaFunction{<-irrefl}\AgdaSpace{}%
\AgdaSymbol{(}\AgdaFunction{<-trans}\AgdaSpace{}%
\AgdaBound{a<b}\AgdaSpace{}%
\AgdaBound{b<a}\AgdaSymbol{)}\<%
\\
\>[0]\AgdaFunction{Lm[≥→¬<]}\AgdaSpace{}%
\AgdaSymbol{(}\AgdaInductiveConstructor{inj₂}\AgdaSpace{}%
\AgdaBound{a=b}\AgdaSymbol{)}%
\>[24]\AgdaSymbol{=}\AgdaSpace{}%
\AgdaFunction{<-irreflexive}\AgdaSpace{}%
\AgdaBound{a=b}\<%
\\
\\[\AgdaEmptyExtraSkip]%
\\[\AgdaEmptyExtraSkip]%
\>[0]\AgdaFunction{Lm[<→¬≥]}\AgdaSpace{}%
\AgdaSymbol{:}\AgdaSpace{}%
\AgdaSymbol{\{}\AgdaBound{a}\AgdaSpace{}%
\AgdaBound{b}\AgdaSpace{}%
\AgdaSymbol{:}\AgdaSpace{}%
\AgdaDatatype{MutualOrd}\AgdaSymbol{\}}\AgdaSpace{}%
\AgdaSymbol{→}\AgdaSpace{}%
\AgdaBound{a}\AgdaSpace{}%
\AgdaOperator{\AgdaDatatype{<}}\AgdaSpace{}%
\AgdaBound{b}\AgdaSpace{}%
\AgdaSymbol{→}\AgdaSpace{}%
\AgdaOperator{\AgdaFunction{¬}}\AgdaSpace{}%
\AgdaBound{a}\AgdaSpace{}%
\AgdaOperator{\AgdaFunction{≥}}\AgdaSpace{}%
\AgdaBound{b}\<%
\\
\>[0]\AgdaFunction{Lm[<→¬≥]}\AgdaSpace{}%
\AgdaBound{a<b}\AgdaSpace{}%
\AgdaSymbol{(}\AgdaInductiveConstructor{inj₁}\AgdaSpace{}%
\AgdaBound{a>b}\AgdaSymbol{)}\AgdaSpace{}%
\AgdaSymbol{=}\AgdaSpace{}%
\AgdaFunction{<-irrefl}\AgdaSpace{}%
\AgdaSymbol{(}\AgdaFunction{<-trans}\AgdaSpace{}%
\AgdaBound{a<b}\AgdaSpace{}%
\AgdaBound{a>b}\AgdaSymbol{)}\<%
\\
\>[0]\AgdaFunction{Lm[<→¬≥]}\AgdaSpace{}%
\AgdaBound{a<b}\AgdaSpace{}%
\AgdaSymbol{(}\AgdaInductiveConstructor{inj₂}\AgdaSpace{}%
\AgdaBound{a=b}\AgdaSymbol{)}\AgdaSpace{}%
\AgdaSymbol{=}\AgdaSpace{}%
\AgdaFunction{<-irreflexive}\AgdaSpace{}%
\AgdaBound{a=b}\AgdaSpace{}%
\AgdaBound{a<b}\<%
\\
\\[\AgdaEmptyExtraSkip]%
\>[0]\AgdaFunction{Lm[¬<→≥]}\AgdaSpace{}%
\AgdaSymbol{:}\AgdaSpace{}%
\AgdaSymbol{\{}\AgdaBound{a}\AgdaSpace{}%
\AgdaBound{b}\AgdaSpace{}%
\AgdaSymbol{:}\AgdaSpace{}%
\AgdaDatatype{MutualOrd}\AgdaSymbol{\}}\AgdaSpace{}%
\AgdaSymbol{→}\AgdaSpace{}%
\AgdaOperator{\AgdaFunction{¬}}\AgdaSpace{}%
\AgdaBound{a}\AgdaSpace{}%
\AgdaOperator{\AgdaDatatype{<}}\AgdaSpace{}%
\AgdaBound{b}\AgdaSpace{}%
\AgdaSymbol{→}\AgdaSpace{}%
\AgdaBound{a}\AgdaSpace{}%
\AgdaOperator{\AgdaFunction{≥}}\AgdaSpace{}%
\AgdaBound{b}\<%
\\
\>[0]\AgdaFunction{Lm[¬<→≥]}\AgdaSpace{}%
\AgdaBound{f}\AgdaSpace{}%
\AgdaKeyword{with}\AgdaSpace{}%
\AgdaFunction{<-tri}\AgdaSpace{}%
\AgdaSymbol{\AgdaUnderscore{}}\AgdaSpace{}%
\AgdaSymbol{\AgdaUnderscore{}}\<%
\\
\>[0]\AgdaFunction{Lm[¬<→≥]}\AgdaSpace{}%
\AgdaBound{f}\AgdaSpace{}%
\AgdaSymbol{|}\AgdaSpace{}%
\AgdaInductiveConstructor{inj₁}%
\>[24]\AgdaBound{a<b}%
\>[29]\AgdaSymbol{=}\AgdaSpace{}%
\AgdaFunction{⊥-elim}\AgdaSpace{}%
\AgdaSymbol{(}\AgdaBound{f}\AgdaSpace{}%
\AgdaBound{a<b}\AgdaSymbol{)}\<%
\\
\>[0]\AgdaFunction{Lm[¬<→≥]}\AgdaSpace{}%
\AgdaBound{f}\AgdaSpace{}%
\AgdaSymbol{|}\AgdaSpace{}%
\AgdaInductiveConstructor{inj₂}\AgdaSpace{}%
\AgdaSymbol{(}\AgdaInductiveConstructor{inj₁}\AgdaSpace{}%
\AgdaBound{a>b}\AgdaSymbol{)}\AgdaSpace{}%
\AgdaSymbol{=}\AgdaSpace{}%
\AgdaInductiveConstructor{inj₁}\AgdaSpace{}%
\AgdaBound{a>b}\<%
\\
\>[0]\AgdaFunction{Lm[¬<→≥]}\AgdaSpace{}%
\AgdaBound{f}\AgdaSpace{}%
\AgdaSymbol{|}\AgdaSpace{}%
\AgdaInductiveConstructor{inj₂}\AgdaSpace{}%
\AgdaSymbol{(}\AgdaInductiveConstructor{inj₂}\AgdaSpace{}%
\AgdaBound{a=b}\AgdaSymbol{)}\AgdaSpace{}%
\AgdaSymbol{=}\AgdaSpace{}%
\AgdaInductiveConstructor{inj₂}\AgdaSpace{}%
\AgdaBound{a=b}\<%
\\
\\[\AgdaEmptyExtraSkip]%
\>[0]\AgdaFunction{<-dec}\AgdaSpace{}%
\AgdaSymbol{:}\AgdaSpace{}%
\AgdaSymbol{(}\AgdaBound{a}\AgdaSpace{}%
\AgdaBound{b}\AgdaSpace{}%
\AgdaSymbol{:}\AgdaSpace{}%
\AgdaDatatype{MutualOrd}\AgdaSymbol{)}\AgdaSpace{}%
\AgdaSymbol{→}\AgdaSpace{}%
\AgdaBound{a}\AgdaSpace{}%
\AgdaOperator{\AgdaDatatype{<}}\AgdaSpace{}%
\AgdaBound{b}\AgdaSpace{}%
\AgdaOperator{\AgdaDatatype{⊎}}\AgdaSpace{}%
\AgdaOperator{\AgdaFunction{¬}}\AgdaSpace{}%
\AgdaBound{a}\AgdaSpace{}%
\AgdaOperator{\AgdaDatatype{<}}\AgdaSpace{}%
\AgdaBound{b}\<%
\\
\>[0]\AgdaFunction{<-dec}\AgdaSpace{}%
\AgdaBound{a}\AgdaSpace{}%
\AgdaBound{b}\AgdaSpace{}%
\AgdaKeyword{with}\AgdaSpace{}%
\AgdaFunction{<-tri}\AgdaSpace{}%
\AgdaBound{a}\AgdaSpace{}%
\AgdaBound{b}\<%
\\
\>[0]\AgdaFunction{<-dec}\AgdaSpace{}%
\AgdaBound{a}\AgdaSpace{}%
\AgdaBound{b}\AgdaSpace{}%
\AgdaSymbol{|}\AgdaSpace{}%
\AgdaInductiveConstructor{inj₁}\AgdaSpace{}%
\AgdaBound{a<b}\AgdaSpace{}%
\AgdaSymbol{=}\AgdaSpace{}%
\AgdaInductiveConstructor{inj₁}\AgdaSpace{}%
\AgdaBound{a<b}\<%
\\
\>[0]\AgdaFunction{<-dec}\AgdaSpace{}%
\AgdaBound{a}\AgdaSpace{}%
\AgdaBound{b}\AgdaSpace{}%
\AgdaSymbol{|}\AgdaSpace{}%
\AgdaInductiveConstructor{inj₂}\AgdaSpace{}%
\AgdaBound{a≥b}\AgdaSpace{}%
\AgdaSymbol{=}\AgdaSpace{}%
\AgdaInductiveConstructor{inj₂}\AgdaSpace{}%
\AgdaSymbol{(}\AgdaFunction{Lm[≥→¬<]}\AgdaSpace{}%
\AgdaBound{a≥b}\AgdaSymbol{)}\<%
\\
\\[\AgdaEmptyExtraSkip]%
\\[\AgdaEmptyExtraSkip]%
\>[0]\AgdaFunction{≤≥→≡}\AgdaSpace{}%
\AgdaSymbol{:}\AgdaSpace{}%
\AgdaSymbol{\{}\AgdaBound{a}\AgdaSpace{}%
\AgdaBound{b}\AgdaSpace{}%
\AgdaSymbol{:}\AgdaSpace{}%
\AgdaDatatype{MutualOrd}\AgdaSymbol{\}}\AgdaSpace{}%
\AgdaSymbol{→}\AgdaSpace{}%
\AgdaBound{a}\AgdaSpace{}%
\AgdaOperator{\AgdaFunction{≤}}\AgdaSpace{}%
\AgdaBound{b}\AgdaSpace{}%
\AgdaSymbol{→}\AgdaSpace{}%
\AgdaBound{a}\AgdaSpace{}%
\AgdaOperator{\AgdaFunction{≥}}\AgdaSpace{}%
\AgdaBound{b}\AgdaSpace{}%
\AgdaSymbol{→}\AgdaSpace{}%
\AgdaBound{a}\AgdaSpace{}%
\AgdaOperator{\AgdaFunction{≡}}\AgdaSpace{}%
\AgdaBound{b}\<%
\\
\>[0]\AgdaFunction{≤≥→≡}\AgdaSpace{}%
\AgdaBound{a≤b}\AgdaSpace{}%
\AgdaSymbol{(}\AgdaInductiveConstructor{inj₁}\AgdaSpace{}%
\AgdaBound{a>b}\AgdaSymbol{)}\AgdaSpace{}%
\AgdaSymbol{=}\AgdaSpace{}%
\AgdaFunction{⊥-elim}\AgdaSpace{}%
\AgdaSymbol{(}\AgdaFunction{Lm[<→¬≥]}\AgdaSpace{}%
\AgdaBound{a>b}\AgdaSpace{}%
\AgdaBound{a≤b}\AgdaSymbol{)}\<%
\\
\>[0]\AgdaFunction{≤≥→≡}\AgdaSpace{}%
\AgdaBound{a≤b}\AgdaSpace{}%
\AgdaSymbol{(}\AgdaInductiveConstructor{inj₂}\AgdaSpace{}%
\AgdaBound{a=b}\AgdaSymbol{)}\AgdaSpace{}%
\AgdaSymbol{=}\AgdaSpace{}%
\AgdaBound{a=b}\<%
\\
\\[\AgdaEmptyExtraSkip]%
\>[0]\AgdaFunction{≥𝟎}\AgdaSpace{}%
\AgdaSymbol{:}\AgdaSpace{}%
\AgdaSymbol{\{}\AgdaBound{a}\AgdaSpace{}%
\AgdaSymbol{:}\AgdaSpace{}%
\AgdaDatatype{MutualOrd}\AgdaSymbol{\}}\AgdaSpace{}%
\AgdaSymbol{→}\AgdaSpace{}%
\AgdaBound{a}\AgdaSpace{}%
\AgdaOperator{\AgdaFunction{≥}}\AgdaSpace{}%
\AgdaInductiveConstructor{𝟎}\<%
\\
\>[0]\AgdaFunction{≥𝟎}\AgdaSpace{}%
\AgdaSymbol{\{}\AgdaInductiveConstructor{𝟎}\AgdaSymbol{\}}%
\>[20]\AgdaSymbol{=}\AgdaSpace{}%
\AgdaInductiveConstructor{inj₂}\AgdaSpace{}%
\AgdaFunction{refl}\<%
\\
\>[0]\AgdaFunction{≥𝟎}\AgdaSpace{}%
\AgdaSymbol{\{}\AgdaOperator{\AgdaInductiveConstructor{ω\textasciicircum{}}}\AgdaSpace{}%
\AgdaBound{a}\AgdaSpace{}%
\AgdaOperator{\AgdaInductiveConstructor{+}}\AgdaSpace{}%
\AgdaBound{b}\AgdaSpace{}%
\AgdaOperator{\AgdaInductiveConstructor{[}}\AgdaSpace{}%
\AgdaSymbol{\AgdaUnderscore{}}\AgdaSpace{}%
\AgdaOperator{\AgdaInductiveConstructor{]}}\AgdaSymbol{\}}\AgdaSpace{}%
\AgdaSymbol{=}\AgdaSpace{}%
\AgdaInductiveConstructor{inj₁}\AgdaSpace{}%
\AgdaInductiveConstructor{<₁}\<%
\\
\\[\AgdaEmptyExtraSkip]%
\>[0]\AgdaOperator{\AgdaFunction{ω\textasciicircum{}⟨\AgdaUnderscore{}⟩}}\AgdaSpace{}%
\AgdaSymbol{:}\AgdaSpace{}%
\AgdaDatatype{MutualOrd}\AgdaSpace{}%
\AgdaSymbol{→}\AgdaSpace{}%
\AgdaDatatype{MutualOrd}\<%
\\
\>[0]\AgdaOperator{\AgdaFunction{ω\textasciicircum{}⟨}}\AgdaSpace{}%
\AgdaBound{a}\AgdaSpace{}%
\AgdaOperator{\AgdaFunction{⟩}}\AgdaSpace{}%
\AgdaSymbol{=}\AgdaSpace{}%
\AgdaOperator{\AgdaInductiveConstructor{ω\textasciicircum{}}}\AgdaSpace{}%
\AgdaBound{a}\AgdaSpace{}%
\AgdaOperator{\AgdaInductiveConstructor{+}}\AgdaSpace{}%
\AgdaInductiveConstructor{𝟎}\AgdaSpace{}%
\AgdaOperator{\AgdaInductiveConstructor{[}}\AgdaSpace{}%
\AgdaFunction{≥𝟎}\AgdaSpace{}%
\AgdaOperator{\AgdaInductiveConstructor{]}}\<%
\\
\\[\AgdaEmptyExtraSkip]%
\>[0]\AgdaFunction{𝟏}\AgdaSpace{}%
\AgdaFunction{ω}\AgdaSpace{}%
\AgdaSymbol{:}\AgdaSpace{}%
\AgdaDatatype{MutualOrd}\<%
\\
\>[0]\AgdaFunction{𝟏}\AgdaSpace{}%
\AgdaSymbol{=}\AgdaSpace{}%
\AgdaOperator{\AgdaFunction{ω\textasciicircum{}⟨}}\AgdaSpace{}%
\AgdaInductiveConstructor{𝟎}\AgdaSpace{}%
\AgdaOperator{\AgdaFunction{⟩}}\<%
\\
\>[0]\AgdaFunction{ω}\AgdaSpace{}%
\AgdaSymbol{=}\AgdaSpace{}%
\AgdaOperator{\AgdaFunction{ω\textasciicircum{}⟨}}\AgdaSpace{}%
\AgdaFunction{𝟏}\AgdaSpace{}%
\AgdaOperator{\AgdaFunction{⟩}}\<%
\\
\\[\AgdaEmptyExtraSkip]%
\>[0]\AgdaFunction{fst<}%
\>[1265I]\AgdaSymbol{:}\AgdaSpace{}%
\AgdaSymbol{(}\AgdaBound{a}\AgdaSpace{}%
\AgdaBound{b}\AgdaSpace{}%
\AgdaSymbol{:}\AgdaSpace{}%
\AgdaDatatype{MutualOrd}\AgdaSymbol{)}\AgdaSpace{}%
\AgdaSymbol{(}\AgdaBound{r}\AgdaSpace{}%
\AgdaSymbol{:}\AgdaSpace{}%
\AgdaBound{a}\AgdaSpace{}%
\AgdaOperator{\AgdaFunction{≥}}\AgdaSpace{}%
\AgdaFunction{fst}\AgdaSpace{}%
\AgdaBound{b}\AgdaSymbol{)}\<%
\\
\>[.][@{}l@{}]\<[1265I]%
\>[5]\AgdaSymbol{→}\AgdaSpace{}%
\AgdaBound{a}\AgdaSpace{}%
\AgdaOperator{\AgdaDatatype{<}}\AgdaSpace{}%
\AgdaOperator{\AgdaInductiveConstructor{ω\textasciicircum{}}}\AgdaSpace{}%
\AgdaBound{a}\AgdaSpace{}%
\AgdaOperator{\AgdaInductiveConstructor{+}}\AgdaSpace{}%
\AgdaBound{b}\AgdaSpace{}%
\AgdaOperator{\AgdaInductiveConstructor{[}}\AgdaSpace{}%
\AgdaBound{r}\AgdaSpace{}%
\AgdaOperator{\AgdaInductiveConstructor{]}}\<%
\\
\>[0]\AgdaFunction{fst<}%
\>[6]\AgdaInductiveConstructor{𝟎}%
\>[22]\AgdaBound{b}\AgdaSpace{}%
\AgdaBound{r}\AgdaSpace{}%
\AgdaSymbol{=}\AgdaSpace{}%
\AgdaInductiveConstructor{<₁}\<%
\\
\>[0]\AgdaFunction{fst<}\AgdaSpace{}%
\AgdaSymbol{(}\AgdaOperator{\AgdaInductiveConstructor{ω\textasciicircum{}}}\AgdaSpace{}%
\AgdaBound{a}\AgdaSpace{}%
\AgdaOperator{\AgdaInductiveConstructor{+}}\AgdaSpace{}%
\AgdaBound{c}\AgdaSpace{}%
\AgdaOperator{\AgdaInductiveConstructor{[}}\AgdaSpace{}%
\AgdaBound{s}\AgdaSpace{}%
\AgdaOperator{\AgdaInductiveConstructor{]}}\AgdaSymbol{)}\AgdaSpace{}%
\AgdaBound{b}\AgdaSpace{}%
\AgdaBound{r}\AgdaSpace{}%
\AgdaSymbol{=}\AgdaSpace{}%
\AgdaInductiveConstructor{<₂}\AgdaSpace{}%
\AgdaSymbol{(}\AgdaFunction{fst<}\AgdaSpace{}%
\AgdaBound{a}\AgdaSpace{}%
\AgdaBound{c}\AgdaSpace{}%
\AgdaBound{s}\AgdaSymbol{)}\<%
\\
\\[\AgdaEmptyExtraSkip]%
\>[0]\AgdaFunction{rest<}%
\>[1303I]\AgdaSymbol{:}\AgdaSpace{}%
\AgdaSymbol{(}\AgdaBound{a}\AgdaSpace{}%
\AgdaBound{b}\AgdaSpace{}%
\AgdaSymbol{:}\AgdaSpace{}%
\AgdaDatatype{MutualOrd}\AgdaSymbol{)}\AgdaSpace{}%
\AgdaSymbol{(}\AgdaBound{r}\AgdaSpace{}%
\AgdaSymbol{:}\AgdaSpace{}%
\AgdaBound{a}\AgdaSpace{}%
\AgdaOperator{\AgdaFunction{≥}}\AgdaSpace{}%
\AgdaFunction{fst}\AgdaSpace{}%
\AgdaBound{b}\AgdaSymbol{)}\<%
\\
\>[.][@{}l@{}]\<[1303I]%
\>[6]\AgdaSymbol{→}\AgdaSpace{}%
\AgdaBound{b}\AgdaSpace{}%
\AgdaOperator{\AgdaDatatype{<}}\AgdaSpace{}%
\AgdaOperator{\AgdaInductiveConstructor{ω\textasciicircum{}}}\AgdaSpace{}%
\AgdaBound{a}\AgdaSpace{}%
\AgdaOperator{\AgdaInductiveConstructor{+}}\AgdaSpace{}%
\AgdaBound{b}\AgdaSpace{}%
\AgdaOperator{\AgdaInductiveConstructor{[}}\AgdaSpace{}%
\AgdaBound{r}\AgdaSpace{}%
\AgdaOperator{\AgdaInductiveConstructor{]}}\<%
\\
\>[0]\AgdaFunction{rest<}\AgdaSpace{}%
\AgdaBound{a}%
\>[9]\AgdaInductiveConstructor{𝟎}%
\>[26]\AgdaSymbol{\AgdaUnderscore{}}%
\>[34]\AgdaSymbol{=}\AgdaSpace{}%
\AgdaInductiveConstructor{<₁}\<%
\\
\>[0]\AgdaFunction{rest<}\AgdaSpace{}%
\AgdaBound{a}\AgdaSpace{}%
\AgdaSymbol{(}\AgdaOperator{\AgdaInductiveConstructor{ω\textasciicircum{}}}\AgdaSpace{}%
\AgdaBound{b}\AgdaSpace{}%
\AgdaOperator{\AgdaInductiveConstructor{+}}\AgdaSpace{}%
\AgdaBound{c}\AgdaSpace{}%
\AgdaOperator{\AgdaInductiveConstructor{[}}\AgdaSpace{}%
\AgdaBound{s}\AgdaSpace{}%
\AgdaOperator{\AgdaInductiveConstructor{]}}\AgdaSymbol{)}\AgdaSpace{}%
\AgdaSymbol{(}\AgdaInductiveConstructor{inj₁}\AgdaSpace{}%
\AgdaBound{r}\AgdaSymbol{)}\AgdaSpace{}%
\AgdaSymbol{=}\AgdaSpace{}%
\AgdaInductiveConstructor{<₂}\AgdaSpace{}%
\AgdaBound{r}\<%
\\
\>[0]\AgdaFunction{rest<}\AgdaSpace{}%
\AgdaBound{a}\AgdaSpace{}%
\AgdaSymbol{(}\AgdaOperator{\AgdaInductiveConstructor{ω\textasciicircum{}}}\AgdaSpace{}%
\AgdaBound{b}\AgdaSpace{}%
\AgdaOperator{\AgdaInductiveConstructor{+}}\AgdaSpace{}%
\AgdaBound{c}\AgdaSpace{}%
\AgdaOperator{\AgdaInductiveConstructor{[}}\AgdaSpace{}%
\AgdaBound{s}\AgdaSpace{}%
\AgdaOperator{\AgdaInductiveConstructor{]}}\AgdaSymbol{)}\AgdaSpace{}%
\AgdaSymbol{(}\AgdaInductiveConstructor{inj₂}\AgdaSpace{}%
\AgdaBound{e}\AgdaSymbol{)}\AgdaSpace{}%
\AgdaSymbol{=}\AgdaSpace{}%
\AgdaInductiveConstructor{<₃}\AgdaSpace{}%
\AgdaSymbol{(}\AgdaBound{e}\AgdaSpace{}%
\AgdaOperator{\AgdaFunction{⁻¹}}\AgdaSymbol{)}\AgdaSpace{}%
\AgdaSymbol{(}\AgdaFunction{rest<}\AgdaSpace{}%
\AgdaBound{b}\AgdaSpace{}%
\AgdaBound{c}\AgdaSpace{}%
\AgdaBound{s}\AgdaSymbol{)}\<%
\\
\>[0]\<%
\end{code}

\begin{code}[hide]%
\>[0]\<%
\\
\>[0]\AgdaSymbol{\{-\#}\AgdaSpace{}%
\AgdaKeyword{OPTIONS}\AgdaSpace{}%
\AgdaPragma{--cubical}\AgdaSpace{}%
\AgdaSymbol{\#-\}}\<%
\\
\\[\AgdaEmptyExtraSkip]%
\>[0]\AgdaKeyword{module}\AgdaSpace{}%
\AgdaModule{Equivalences}\AgdaSpace{}%
\AgdaKeyword{where}\<%
\\
\\[\AgdaEmptyExtraSkip]%
\>[0]\AgdaKeyword{open}\AgdaSpace{}%
\AgdaKeyword{import}\AgdaSpace{}%
\AgdaModule{Agda.Builtin.Cubical.Glue}\AgdaSpace{}%
\AgdaKeyword{using}\AgdaSpace{}%
\AgdaSymbol{(}\AgdaOperator{\AgdaFunction{\AgdaUnderscore{}≃\AgdaUnderscore{}}}\AgdaSymbol{)}\<%
\\
\\[\AgdaEmptyExtraSkip]%
\>[0]\AgdaKeyword{open}\AgdaSpace{}%
\AgdaKeyword{import}\AgdaSpace{}%
\AgdaModule{Cubical.Foundations.Everything}\<%
\\
\>[0][@{}l@{\AgdaIndent{0}}]%
\>[1]\AgdaKeyword{using}\AgdaSpace{}%
\AgdaSymbol{(}\AgdaFunction{Type}\AgdaSpace{}%
\AgdaSymbol{;}\AgdaSpace{}%
\AgdaFunction{Type₀}\AgdaSpace{}%
\AgdaSymbol{;}\AgdaSpace{}%
\AgdaOperator{\AgdaFunction{\AgdaUnderscore{}≡\AgdaUnderscore{}}}\AgdaSpace{}%
\AgdaSymbol{;}\AgdaSpace{}%
\AgdaFunction{refl}\AgdaSpace{}%
\AgdaSymbol{;}\AgdaSpace{}%
\AgdaFunction{cong}\AgdaSpace{}%
\AgdaSymbol{;}\AgdaSpace{}%
\AgdaFunction{cong₂}\AgdaSpace{}%
\AgdaSymbol{;}\AgdaSpace{}%
\AgdaOperator{\AgdaFunction{\AgdaUnderscore{}⁻¹}}\AgdaSpace{}%
\AgdaSymbol{;}\AgdaSpace{}%
\AgdaOperator{\AgdaFunction{\AgdaUnderscore{}∙\AgdaUnderscore{}}}\AgdaSpace{}%
\AgdaSymbol{;}\AgdaSpace{}%
\AgdaFunction{J}\AgdaSpace{}%
\AgdaSymbol{;}\AgdaSpace{}%
\AgdaInductiveConstructor{iso}\AgdaSpace{}%
\AgdaSymbol{;}\AgdaSpace{}%
\AgdaFunction{isoToEquiv}\AgdaSpace{}%
\AgdaSymbol{;}\AgdaSpace{}%
\AgdaFunction{ua}\AgdaSymbol{)}\<%
\\
\\[\AgdaEmptyExtraSkip]%
\>[0]\AgdaKeyword{open}\AgdaSpace{}%
\AgdaKeyword{import}\AgdaSpace{}%
\AgdaModule{Cubical.Data.Empty}\<%
\\
\\[\AgdaEmptyExtraSkip]%
\>[0]\AgdaKeyword{open}\AgdaSpace{}%
\AgdaKeyword{import}\AgdaSpace{}%
\AgdaModule{Preliminaries}\<%
\\
\\[\AgdaEmptyExtraSkip]%
\>[0]\AgdaKeyword{import}\AgdaSpace{}%
\AgdaModule{Relation.Binary.PropositionalEquality}\AgdaSpace{}%
\AgdaSymbol{as}\AgdaSpace{}%
\AgdaModule{P}\<%
\\
\>[0]\<%
\end{code}

\newcommand{\importOrds}{
\begin{code}%
\>[0]\AgdaKeyword{open}\AgdaSpace{}%
\AgdaKeyword{import}\AgdaSpace{}%
\AgdaModule{SigmaOrd}%
\>[23]\AgdaSymbol{as}\AgdaSpace{}%
\AgdaModule{S}\<%
\\
\>[0]\AgdaKeyword{open}\AgdaSpace{}%
\AgdaKeyword{import}\AgdaSpace{}%
\AgdaModule{MutualOrd}%
\>[23]\AgdaSymbol{as}\AgdaSpace{}%
\AgdaModule{M}\<%
\\
\>[0]\AgdaKeyword{open}\AgdaSpace{}%
\AgdaKeyword{import}\AgdaSpace{}%
\AgdaModule{HITOrd}%
\>[23]\AgdaSymbol{as}\AgdaSpace{}%
\AgdaModule{H}\<%
\end{code}}

\newcommand{\TtoMsig}{
\begin{code}%
\>[0]\AgdaFunction{T2M}\AgdaSpace{}%
\AgdaSymbol{:}\AgdaSpace{}%
\AgdaSymbol{\{}\AgdaBound{a}\AgdaSpace{}%
\AgdaSymbol{:}\AgdaSpace{}%
\AgdaDatatype{Tree}\AgdaSymbol{\}}\AgdaSpace{}%
\AgdaSymbol{→}\AgdaSpace{}%
\AgdaDatatype{isCNF}\AgdaSpace{}%
\AgdaBound{a}\AgdaSpace{}%
\AgdaSymbol{→}\AgdaSpace{}%
\AgdaDatatype{MutualOrd}\<%
\\
\>[0]\AgdaFunction{T2M[<]}%
\>[59I]\AgdaSymbol{:}\AgdaSpace{}%
\AgdaSymbol{\{}\AgdaBound{a}\AgdaSpace{}%
\AgdaBound{b}\AgdaSpace{}%
\AgdaSymbol{:}\AgdaSpace{}%
\AgdaDatatype{Tree}\AgdaSymbol{\}}\AgdaSpace{}%
\AgdaSymbol{\{}\AgdaBound{p}\AgdaSpace{}%
\AgdaSymbol{:}\AgdaSpace{}%
\AgdaDatatype{isCNF}\AgdaSpace{}%
\AgdaBound{a}\AgdaSymbol{\}}\AgdaSpace{}%
\AgdaSymbol{\{}\AgdaBound{q}\AgdaSpace{}%
\AgdaSymbol{:}\AgdaSpace{}%
\AgdaDatatype{isCNF}\AgdaSpace{}%
\AgdaBound{b}\AgdaSymbol{\}}\<%
\\
\>[59I][@{}l@{\AgdaIndent{0}}]%
\>[8]\AgdaSymbol{→}\AgdaSpace{}%
\AgdaBound{a}\AgdaSpace{}%
\AgdaOperator{\AgdaDatatype{S.<}}\AgdaSpace{}%
\AgdaBound{b}\AgdaSpace{}%
\AgdaSymbol{→}\AgdaSpace{}%
\AgdaFunction{T2M}\AgdaSpace{}%
\AgdaBound{p}\AgdaSpace{}%
\AgdaOperator{\AgdaDatatype{M.<}}\AgdaSpace{}%
\AgdaFunction{T2M}\AgdaSpace{}%
\AgdaBound{q}\<%
\\
\>[0]\AgdaFunction{T2M[≥fst]}%
\>[81I]\AgdaSymbol{:}\AgdaSpace{}%
\AgdaSymbol{\{}\AgdaBound{a}\AgdaSpace{}%
\AgdaBound{b}\AgdaSpace{}%
\AgdaSymbol{:}\AgdaSpace{}%
\AgdaDatatype{Tree}\AgdaSymbol{\}}\AgdaSpace{}%
\AgdaSymbol{\{}\AgdaBound{p}\AgdaSpace{}%
\AgdaSymbol{:}\AgdaSpace{}%
\AgdaDatatype{isCNF}\AgdaSpace{}%
\AgdaBound{a}\AgdaSymbol{\}}\AgdaSpace{}%
\AgdaSymbol{(}\AgdaBound{q}\AgdaSpace{}%
\AgdaSymbol{:}\AgdaSpace{}%
\AgdaDatatype{isCNF}\AgdaSpace{}%
\AgdaBound{b}\AgdaSymbol{)}\<%
\\
\>[.][@{}l@{}]\<[81I]%
\>[10]\AgdaSymbol{→}\AgdaSpace{}%
\AgdaBound{a}\AgdaSpace{}%
\AgdaOperator{\AgdaFunction{S.≥}}\AgdaSpace{}%
\AgdaFunction{S.fst}\AgdaSpace{}%
\AgdaBound{b}\AgdaSpace{}%
\AgdaSymbol{→}\AgdaSpace{}%
\AgdaFunction{T2M}\AgdaSpace{}%
\AgdaBound{p}\AgdaSpace{}%
\AgdaOperator{\AgdaFunction{M.≥}}\AgdaSpace{}%
\AgdaFunction{M.fst}\AgdaSpace{}%
\AgdaSymbol{(}\AgdaFunction{T2M}\AgdaSpace{}%
\AgdaBound{q}\AgdaSymbol{)}\<%
\end{code}}
\newcommand{\TtoMEqsig}{
\begin{code}%
\>[0]\AgdaFunction{T2M[≡]}%
\>[105I]\AgdaSymbol{:}\AgdaSpace{}%
\AgdaSymbol{\{}\AgdaBound{a}\AgdaSpace{}%
\AgdaBound{b}\AgdaSpace{}%
\AgdaSymbol{:}\AgdaSpace{}%
\AgdaDatatype{Tree}\AgdaSymbol{\}}\AgdaSpace{}%
\AgdaSymbol{\{}\AgdaBound{p}\AgdaSpace{}%
\AgdaSymbol{:}\AgdaSpace{}%
\AgdaDatatype{isCNF}\AgdaSpace{}%
\AgdaBound{a}\AgdaSymbol{\}}\AgdaSpace{}%
\AgdaSymbol{\{}\AgdaBound{q}\AgdaSpace{}%
\AgdaSymbol{:}\AgdaSpace{}%
\AgdaDatatype{isCNF}\AgdaSpace{}%
\AgdaBound{b}\AgdaSymbol{\}}\<%
\\
\>[.][@{}l@{}]\<[105I]%
\>[7]\AgdaSymbol{→}\AgdaSpace{}%
\AgdaBound{a}\AgdaSpace{}%
\AgdaOperator{\AgdaFunction{≡}}\AgdaSpace{}%
\AgdaBound{b}\AgdaSpace{}%
\AgdaSymbol{→}\AgdaSpace{}%
\AgdaFunction{T2M}\AgdaSpace{}%
\AgdaBound{p}\AgdaSpace{}%
\AgdaOperator{\AgdaFunction{≡}}\AgdaSpace{}%
\AgdaFunction{T2M}\AgdaSpace{}%
\AgdaBound{q}\<%
\\
\>[0]\AgdaFunction{T2M[≡]}\AgdaSpace{}%
\AgdaBound{a=b}\AgdaSpace{}%
\AgdaKeyword{with}\AgdaSpace{}%
\AgdaFunction{PropEqfromPath}\AgdaSpace{}%
\AgdaBound{a=b}\<%
\\
\>[0]\AgdaFunction{T2M[≡]}\AgdaSpace{}%
\AgdaBound{a=b}\AgdaSpace{}%
\AgdaSymbol{|}\AgdaSpace{}%
\AgdaInductiveConstructor{P.refl}\AgdaSpace{}%
\AgdaSymbol{=}\AgdaSpace{}%
\AgdaFunction{cong}\AgdaSpace{}%
\AgdaFunction{T2M}\AgdaSpace{}%
\AgdaSymbol{(}\AgdaFunction{isCNFIsPropValued}\AgdaSpace{}%
\AgdaSymbol{\AgdaUnderscore{}}\AgdaSpace{}%
\AgdaSymbol{\AgdaUnderscore{})}\<%
\end{code}}
\newcommand{\TtoMimp}{
\begin{code}%
\>[0]\AgdaFunction{T2M}\AgdaSpace{}%
\AgdaInductiveConstructor{𝟎IsCNF}\AgdaSpace{}%
\AgdaSymbol{=}\AgdaSpace{}%
\AgdaInductiveConstructor{𝟎}\<%
\\
\>[0]\AgdaFunction{T2M}\AgdaSpace{}%
\AgdaSymbol{(}\AgdaInductiveConstructor{ω\textasciicircum{}+IsCNF}\AgdaSpace{}%
\AgdaBound{p}\AgdaSpace{}%
\AgdaBound{q}\AgdaSpace{}%
\AgdaBound{r}\AgdaSymbol{)}\AgdaSpace{}%
\AgdaSymbol{=}\<%
\\
\>[0][@{}l@{\AgdaIndent{0}}]%
\>[2]\AgdaOperator{\AgdaInductiveConstructor{ω\textasciicircum{}}}\AgdaSpace{}%
\AgdaSymbol{(}\AgdaFunction{T2M}\AgdaSpace{}%
\AgdaBound{p}\AgdaSymbol{)}\AgdaSpace{}%
\AgdaOperator{\AgdaInductiveConstructor{+}}\AgdaSpace{}%
\AgdaSymbol{(}\AgdaFunction{T2M}\AgdaSpace{}%
\AgdaBound{q}\AgdaSymbol{)}\AgdaSpace{}%
\AgdaOperator{\AgdaInductiveConstructor{[}}\AgdaSpace{}%
\AgdaFunction{T2M[≥fst]}\AgdaSpace{}%
\AgdaBound{q}\AgdaSpace{}%
\AgdaBound{r}\AgdaSpace{}%
\AgdaOperator{\AgdaInductiveConstructor{]}}\<%
\end{code}}

\begin{code}[hide]%
\>[0]\AgdaFunction{T2M[<]}\AgdaSpace{}%
\AgdaSymbol{\{\AgdaUnderscore{}\}}\AgdaSpace{}%
\AgdaSymbol{\{\AgdaUnderscore{}\}}\AgdaSpace{}%
\AgdaSymbol{\{}\AgdaInductiveConstructor{𝟎IsCNF}\AgdaSymbol{\}}\AgdaSpace{}%
\AgdaSymbol{\{}\AgdaInductiveConstructor{ω\textasciicircum{}+IsCNF}\AgdaSpace{}%
\AgdaSymbol{\AgdaUnderscore{}}\AgdaSpace{}%
\AgdaSymbol{\AgdaUnderscore{}}\AgdaSpace{}%
\AgdaSymbol{\AgdaUnderscore{}\}}\AgdaSpace{}%
\AgdaInductiveConstructor{<₁}\AgdaSpace{}%
\AgdaSymbol{=}\AgdaSpace{}%
\AgdaInductiveConstructor{<₁}\<%
\\
\>[0]\AgdaFunction{T2M[<]}\AgdaSpace{}%
\AgdaSymbol{\{\AgdaUnderscore{}\}}\AgdaSpace{}%
\AgdaSymbol{\{\AgdaUnderscore{}\}}\AgdaSpace{}%
\AgdaSymbol{\{}\AgdaInductiveConstructor{ω\textasciicircum{}+IsCNF}\AgdaSpace{}%
\AgdaSymbol{\AgdaUnderscore{}}\AgdaSpace{}%
\AgdaSymbol{\AgdaUnderscore{}}\AgdaSpace{}%
\AgdaSymbol{\AgdaUnderscore{}\}}\AgdaSpace{}%
\AgdaSymbol{\{}\AgdaInductiveConstructor{ω\textasciicircum{}+IsCNF}\AgdaSpace{}%
\AgdaSymbol{\AgdaUnderscore{}}\AgdaSpace{}%
\AgdaSymbol{\AgdaUnderscore{}}\AgdaSpace{}%
\AgdaSymbol{\AgdaUnderscore{}\}}\AgdaSpace{}%
\AgdaSymbol{(}\AgdaInductiveConstructor{<₂}\AgdaSpace{}%
\AgdaBound{r}\AgdaSymbol{)}\AgdaSpace{}%
\AgdaSymbol{=}\AgdaSpace{}%
\AgdaInductiveConstructor{<₂}\AgdaSpace{}%
\AgdaSymbol{(}\AgdaFunction{T2M[<]}\AgdaSpace{}%
\AgdaBound{r}\AgdaSymbol{)}\<%
\\
\>[0]\AgdaFunction{T2M[<]}\AgdaSpace{}%
\AgdaSymbol{\{\AgdaUnderscore{}\}}\AgdaSpace{}%
\AgdaSymbol{\{\AgdaUnderscore{}\}}\AgdaSpace{}%
\AgdaSymbol{\{}\AgdaInductiveConstructor{ω\textasciicircum{}+IsCNF}\AgdaSpace{}%
\AgdaSymbol{\AgdaUnderscore{}}\AgdaSpace{}%
\AgdaSymbol{\AgdaUnderscore{}}\AgdaSpace{}%
\AgdaSymbol{\AgdaUnderscore{}\}}\AgdaSpace{}%
\AgdaSymbol{\{}\AgdaInductiveConstructor{ω\textasciicircum{}+IsCNF}\AgdaSpace{}%
\AgdaSymbol{\AgdaUnderscore{}}\AgdaSpace{}%
\AgdaSymbol{\AgdaUnderscore{}}\AgdaSpace{}%
\AgdaSymbol{\AgdaUnderscore{}\}}\AgdaSpace{}%
\AgdaSymbol{(}\AgdaInductiveConstructor{<₃}\AgdaSpace{}%
\AgdaBound{e}\AgdaSpace{}%
\AgdaBound{r}\AgdaSymbol{)}\AgdaSpace{}%
\AgdaSymbol{=}\AgdaSpace{}%
\AgdaInductiveConstructor{<₃}\AgdaSpace{}%
\AgdaSymbol{(}\AgdaFunction{T2M[≡]}\AgdaSpace{}%
\AgdaBound{e}\AgdaSymbol{)}\AgdaSpace{}%
\AgdaSymbol{(}\AgdaFunction{T2M[<]}\AgdaSpace{}%
\AgdaBound{r}\AgdaSymbol{)}\<%
\\
\\[\AgdaEmptyExtraSkip]%
\>[0]\AgdaFunction{T2M[≥fst]}\AgdaSpace{}%
\AgdaInductiveConstructor{𝟎IsCNF}\AgdaSpace{}%
\AgdaSymbol{\AgdaUnderscore{}}\AgdaSpace{}%
\AgdaSymbol{=}\AgdaSpace{}%
\AgdaFunction{M.≥𝟎}\<%
\\
\>[0]\AgdaFunction{T2M[≥fst]}\AgdaSpace{}%
\AgdaSymbol{(}\AgdaInductiveConstructor{ω\textasciicircum{}+IsCNF}\AgdaSpace{}%
\AgdaSymbol{\AgdaUnderscore{}}\AgdaSpace{}%
\AgdaSymbol{\AgdaUnderscore{}}\AgdaSpace{}%
\AgdaSymbol{\AgdaUnderscore{})}\AgdaSpace{}%
\AgdaSymbol{(}\AgdaInductiveConstructor{inj₁}\AgdaSpace{}%
\AgdaBound{r}\AgdaSymbol{)}\AgdaSpace{}%
\AgdaSymbol{=}\AgdaSpace{}%
\AgdaInductiveConstructor{inj₁}\AgdaSpace{}%
\AgdaSymbol{(}\AgdaFunction{T2M[<]}\AgdaSpace{}%
\AgdaBound{r}\AgdaSymbol{)}\<%
\\
\>[0]\AgdaFunction{T2M[≥fst]}\AgdaSpace{}%
\AgdaSymbol{(}\AgdaInductiveConstructor{ω\textasciicircum{}+IsCNF}\AgdaSpace{}%
\AgdaSymbol{\AgdaUnderscore{}}\AgdaSpace{}%
\AgdaSymbol{\AgdaUnderscore{}}\AgdaSpace{}%
\AgdaSymbol{\AgdaUnderscore{})}\AgdaSpace{}%
\AgdaSymbol{(}\AgdaInductiveConstructor{inj₂}\AgdaSpace{}%
\AgdaBound{e}\AgdaSymbol{)}\AgdaSpace{}%
\AgdaSymbol{=}\AgdaSpace{}%
\AgdaInductiveConstructor{inj₂}\AgdaSpace{}%
\AgdaSymbol{(}\AgdaFunction{T2M[≡]}\AgdaSpace{}%
\AgdaBound{e}\AgdaSymbol{)}\<%
\\
\>[0]\<%
\end{code}

\newcommand{\StoM}{
\begin{code}%
\>[0]\AgdaFunction{S2M}\AgdaSpace{}%
\AgdaSymbol{:}\AgdaSpace{}%
\AgdaFunction{SigmaOrd}\AgdaSpace{}%
\AgdaSymbol{→}\AgdaSpace{}%
\AgdaDatatype{MutualOrd}\<%
\\
\>[0]\AgdaFunction{S2M}\AgdaSpace{}%
\AgdaSymbol{(}\AgdaBound{a}\AgdaSpace{}%
\AgdaOperator{\AgdaInductiveConstructor{,}}\AgdaSpace{}%
\AgdaBound{p}\AgdaSymbol{)}\AgdaSpace{}%
\AgdaSymbol{=}\AgdaSpace{}%
\AgdaFunction{T2M}\AgdaSpace{}%
\AgdaBound{p}\<%
\end{code}}

\newcommand{\MtoT}{
\begin{code}%
\>[0]\AgdaFunction{M2T}\AgdaSpace{}%
\AgdaSymbol{:}\AgdaSpace{}%
\AgdaDatatype{MutualOrd}\AgdaSpace{}%
\AgdaSymbol{→}\AgdaSpace{}%
\AgdaDatatype{Tree}\<%
\\
\>[0]\AgdaFunction{M2T}%
\>[5]\AgdaInductiveConstructor{𝟎}%
\>[21]\AgdaSymbol{=}\AgdaSpace{}%
\AgdaInductiveConstructor{𝟎}\<%
\\
\>[0]\AgdaFunction{M2T}\AgdaSpace{}%
\AgdaSymbol{(}\AgdaOperator{\AgdaInductiveConstructor{ω\textasciicircum{}}}\AgdaSpace{}%
\AgdaBound{a}\AgdaSpace{}%
\AgdaOperator{\AgdaInductiveConstructor{+}}\AgdaSpace{}%
\AgdaBound{b}\AgdaSpace{}%
\AgdaOperator{\AgdaInductiveConstructor{[}}\AgdaSpace{}%
\AgdaSymbol{\AgdaUnderscore{}}\AgdaSpace{}%
\AgdaOperator{\AgdaInductiveConstructor{]}}\AgdaSymbol{)}\AgdaSpace{}%
\AgdaSymbol{=}\AgdaSpace{}%
\AgdaOperator{\AgdaInductiveConstructor{ω\textasciicircum{}}}\AgdaSpace{}%
\AgdaSymbol{(}\AgdaFunction{M2T}\AgdaSpace{}%
\AgdaBound{a}\AgdaSymbol{)}\AgdaSpace{}%
\AgdaOperator{\AgdaInductiveConstructor{+}}\AgdaSpace{}%
\AgdaSymbol{(}\AgdaFunction{M2T}\AgdaSpace{}%
\AgdaBound{b}\AgdaSymbol{)}\<%
\end{code}}

\newcommand{\MtoSLemmas}{
\begin{code}%
\>[0]\AgdaFunction{M2T[<]}%
\>[256I]\AgdaSymbol{:}\AgdaSpace{}%
\AgdaSymbol{\{}\AgdaBound{a}\AgdaSpace{}%
\AgdaBound{b}\AgdaSpace{}%
\AgdaSymbol{:}\AgdaSpace{}%
\AgdaDatatype{MutualOrd}\AgdaSymbol{\}}\<%
\\
\>[.][@{}l@{}]\<[256I]%
\>[7]\AgdaSymbol{→}\AgdaSpace{}%
\AgdaBound{a}\AgdaSpace{}%
\AgdaOperator{\AgdaDatatype{M.<}}\AgdaSpace{}%
\AgdaBound{b}\AgdaSpace{}%
\AgdaSymbol{→}\AgdaSpace{}%
\AgdaFunction{M2T}\AgdaSpace{}%
\AgdaBound{a}\AgdaSpace{}%
\AgdaOperator{\AgdaDatatype{S.<}}\AgdaSpace{}%
\AgdaFunction{M2T}\AgdaSpace{}%
\AgdaBound{b}\<%
\\
\>[0]\AgdaFunction{M2T[≥fst]}%
\>[270I]\AgdaSymbol{:}\AgdaSpace{}%
\AgdaSymbol{\{}\AgdaBound{a}\AgdaSpace{}%
\AgdaSymbol{:}\AgdaSpace{}%
\AgdaDatatype{MutualOrd}\AgdaSymbol{\}}\AgdaSpace{}%
\AgdaSymbol{(}\AgdaBound{b}\AgdaSpace{}%
\AgdaSymbol{:}\AgdaSpace{}%
\AgdaDatatype{MutualOrd}\AgdaSymbol{)}\<%
\\
\>[.][@{}l@{}]\<[270I]%
\>[10]\AgdaSymbol{→}\AgdaSpace{}%
\AgdaBound{a}\AgdaSpace{}%
\AgdaOperator{\AgdaFunction{M.≥}}\AgdaSpace{}%
\AgdaFunction{M.fst}\AgdaSpace{}%
\AgdaBound{b}\AgdaSpace{}%
\AgdaSymbol{→}\AgdaSpace{}%
\AgdaFunction{M2T}\AgdaSpace{}%
\AgdaBound{a}\AgdaSpace{}%
\AgdaOperator{\AgdaFunction{S.≥}}\AgdaSpace{}%
\AgdaFunction{S.fst}\AgdaSpace{}%
\AgdaSymbol{(}\AgdaFunction{M2T}\AgdaSpace{}%
\AgdaBound{b}\AgdaSymbol{)}\<%
\end{code}}

\begin{code}[hide]%
\>[0]\AgdaFunction{M2T[<]}\AgdaSpace{}%
\AgdaInductiveConstructor{<₁}\AgdaSpace{}%
\AgdaSymbol{=}\AgdaSpace{}%
\AgdaInductiveConstructor{<₁}\<%
\\
\>[0]\AgdaFunction{M2T[<]}\AgdaSpace{}%
\AgdaSymbol{(}\AgdaInductiveConstructor{<₂}\AgdaSpace{}%
\AgdaBound{r}\AgdaSymbol{)}\AgdaSpace{}%
\AgdaSymbol{=}\AgdaSpace{}%
\AgdaInductiveConstructor{<₂}\AgdaSpace{}%
\AgdaSymbol{(}\AgdaFunction{M2T[<]}\AgdaSpace{}%
\AgdaBound{r}\AgdaSymbol{)}\<%
\\
\>[0]\AgdaFunction{M2T[<]}\AgdaSpace{}%
\AgdaSymbol{(}\AgdaInductiveConstructor{<₃}\AgdaSpace{}%
\AgdaBound{e}\AgdaSpace{}%
\AgdaBound{r}\AgdaSymbol{)}\AgdaSpace{}%
\AgdaSymbol{=}\AgdaSpace{}%
\AgdaInductiveConstructor{<₃}\AgdaSpace{}%
\AgdaSymbol{(}\AgdaFunction{cong}\AgdaSpace{}%
\AgdaFunction{M2T}\AgdaSpace{}%
\AgdaBound{e}\AgdaSymbol{)}\AgdaSpace{}%
\AgdaSymbol{(}\AgdaFunction{M2T[<]}\AgdaSpace{}%
\AgdaBound{r}\AgdaSymbol{)}\<%
\\
\\[\AgdaEmptyExtraSkip]%
\>[0]\AgdaFunction{M2T[≥fst]}\AgdaSpace{}%
\AgdaInductiveConstructor{𝟎}\AgdaSpace{}%
\AgdaBound{r}\AgdaSpace{}%
\AgdaSymbol{=}\AgdaSpace{}%
\AgdaFunction{S.≥𝟎}\<%
\\
\>[0]\AgdaFunction{M2T[≥fst]}\AgdaSpace{}%
\AgdaSymbol{(}\AgdaOperator{\AgdaInductiveConstructor{ω\textasciicircum{}}}\AgdaSpace{}%
\AgdaBound{b}\AgdaSpace{}%
\AgdaOperator{\AgdaInductiveConstructor{+}}\AgdaSpace{}%
\AgdaBound{c}\AgdaSpace{}%
\AgdaOperator{\AgdaInductiveConstructor{[}}\AgdaSpace{}%
\AgdaBound{r}\AgdaSpace{}%
\AgdaOperator{\AgdaInductiveConstructor{]}}\AgdaSymbol{)}\AgdaSpace{}%
\AgdaSymbol{(}\AgdaInductiveConstructor{inj₁}\AgdaSpace{}%
\AgdaBound{a>b}\AgdaSymbol{)}\AgdaSpace{}%
\AgdaSymbol{=}\AgdaSpace{}%
\AgdaInductiveConstructor{inj₁}\AgdaSpace{}%
\AgdaSymbol{(}\AgdaFunction{M2T[<]}\AgdaSpace{}%
\AgdaBound{a>b}\AgdaSymbol{)}\<%
\\
\>[0]\AgdaFunction{M2T[≥fst]}\AgdaSpace{}%
\AgdaSymbol{(}\AgdaOperator{\AgdaInductiveConstructor{ω\textasciicircum{}}}\AgdaSpace{}%
\AgdaBound{b}\AgdaSpace{}%
\AgdaOperator{\AgdaInductiveConstructor{+}}\AgdaSpace{}%
\AgdaBound{c}\AgdaSpace{}%
\AgdaOperator{\AgdaInductiveConstructor{[}}\AgdaSpace{}%
\AgdaBound{r}\AgdaSpace{}%
\AgdaOperator{\AgdaInductiveConstructor{]}}\AgdaSymbol{)}\AgdaSpace{}%
\AgdaSymbol{(}\AgdaInductiveConstructor{inj₂}\AgdaSpace{}%
\AgdaBound{a=b}\AgdaSymbol{)}\AgdaSpace{}%
\AgdaSymbol{=}\AgdaSpace{}%
\AgdaInductiveConstructor{inj₂}\AgdaSpace{}%
\AgdaSymbol{(}\AgdaFunction{cong}\AgdaSpace{}%
\AgdaFunction{M2T}\AgdaSpace{}%
\AgdaBound{a=b}\AgdaSymbol{)}\<%
\end{code}

\newcommand{\isCNFMtoT}{
\begin{code}%
\>[0]\AgdaFunction{isCNF[M2T]}\AgdaSpace{}%
\AgdaSymbol{:}\AgdaSpace{}%
\AgdaSymbol{(}\AgdaBound{a}\AgdaSpace{}%
\AgdaSymbol{:}\AgdaSpace{}%
\AgdaDatatype{MutualOrd}\AgdaSymbol{)}\AgdaSpace{}%
\AgdaSymbol{→}\AgdaSpace{}%
\AgdaDatatype{isCNF}\AgdaSpace{}%
\AgdaSymbol{(}\AgdaFunction{M2T}\AgdaSpace{}%
\AgdaBound{a}\AgdaSymbol{)}\<%
\\
\>[0]\AgdaFunction{isCNF[M2T]}\AgdaSpace{}%
\AgdaInductiveConstructor{𝟎}\AgdaSpace{}%
\AgdaSymbol{=}\AgdaSpace{}%
\AgdaInductiveConstructor{𝟎IsCNF}\<%
\\
\>[0]\AgdaFunction{isCNF[M2T]}\AgdaSpace{}%
\AgdaSymbol{(}\AgdaOperator{\AgdaInductiveConstructor{ω\textasciicircum{}}}\AgdaSpace{}%
\AgdaBound{a}\AgdaSpace{}%
\AgdaOperator{\AgdaInductiveConstructor{+}}\AgdaSpace{}%
\AgdaBound{b}\AgdaSpace{}%
\AgdaOperator{\AgdaInductiveConstructor{[}}\AgdaSpace{}%
\AgdaBound{r}\AgdaSpace{}%
\AgdaOperator{\AgdaInductiveConstructor{]}}\AgdaSymbol{)}\AgdaSpace{}%
\AgdaSymbol{=}\<%
\\
\>[0][@{}l@{\AgdaIndent{0}}]%
\>[1]\AgdaInductiveConstructor{ω\textasciicircum{}+IsCNF}%
\>[357I]\AgdaSymbol{(}\AgdaFunction{isCNF[M2T]}\AgdaSpace{}%
\AgdaBound{a}\AgdaSymbol{)}\AgdaSpace{}%
\AgdaSymbol{(}\AgdaFunction{isCNF[M2T]}\AgdaSpace{}%
\AgdaBound{b}\AgdaSymbol{)}\<%
\\
\>[.][@{}l@{}]\<[357I]%
\>[10]\AgdaSymbol{(}\AgdaFunction{M2T[≥fst]}\AgdaSpace{}%
\AgdaBound{b}\AgdaSpace{}%
\AgdaBound{r}\AgdaSymbol{)}\<%
\end{code}}

\newcommand{\MtoS}{
\begin{code}%
\>[0]\AgdaFunction{M2S}\AgdaSpace{}%
\AgdaSymbol{:}\AgdaSpace{}%
\AgdaDatatype{MutualOrd}\AgdaSpace{}%
\AgdaSymbol{→}\AgdaSpace{}%
\AgdaFunction{SigmaOrd}\<%
\\
\>[0]\AgdaFunction{M2S}\AgdaSpace{}%
\AgdaBound{a}\AgdaSpace{}%
\AgdaSymbol{=}\AgdaSpace{}%
\AgdaSymbol{(}\AgdaFunction{M2T}\AgdaSpace{}%
\AgdaBound{a}\AgdaSpace{}%
\AgdaOperator{\AgdaInductiveConstructor{,}}\AgdaSpace{}%
\AgdaFunction{isCNF[M2T]}\AgdaSpace{}%
\AgdaBound{a}\AgdaSymbol{)}\<%
\end{code}}

\newcommand{\SMids}{
\begin{code}%
\>[0]\AgdaFunction{S2M2T=pr₁}\AgdaSpace{}%
\AgdaSymbol{:}\AgdaSpace{}%
\AgdaSymbol{(}\AgdaBound{a}\AgdaSpace{}%
\AgdaSymbol{:}\AgdaSpace{}%
\AgdaFunction{SigmaOrd}\AgdaSymbol{)}\AgdaSpace{}%
\AgdaSymbol{→}\AgdaSpace{}%
\AgdaFunction{M2T}\AgdaSpace{}%
\AgdaSymbol{(}\AgdaFunction{S2M}\AgdaSpace{}%
\AgdaBound{a}\AgdaSymbol{)}\AgdaSpace{}%
\AgdaOperator{\AgdaFunction{≡}}\AgdaSpace{}%
\AgdaField{pr₁}\AgdaSpace{}%
\AgdaBound{a}\<%
\\
\>[0]\AgdaFunction{S2M2T=pr₁}\AgdaSpace{}%
\AgdaSymbol{(}\AgdaInductiveConstructor{𝟎}\AgdaSpace{}%
\AgdaOperator{\AgdaInductiveConstructor{,}}\AgdaSpace{}%
\AgdaInductiveConstructor{𝟎IsCNF}\AgdaSymbol{)}\AgdaSpace{}%
\AgdaSymbol{=}\AgdaSpace{}%
\AgdaFunction{refl}\<%
\\
\>[0]\AgdaFunction{S2M2T=pr₁}\AgdaSpace{}%
\AgdaSymbol{(}\AgdaOperator{\AgdaInductiveConstructor{ω\textasciicircum{}}}\AgdaSpace{}%
\AgdaBound{a}\AgdaSpace{}%
\AgdaOperator{\AgdaInductiveConstructor{+}}\AgdaSpace{}%
\AgdaBound{b}\AgdaSpace{}%
\AgdaOperator{\AgdaInductiveConstructor{,}}\AgdaSpace{}%
\AgdaInductiveConstructor{ω\textasciicircum{}+IsCNF}\AgdaSpace{}%
\AgdaBound{p}\AgdaSpace{}%
\AgdaBound{q}\AgdaSpace{}%
\AgdaBound{r}\AgdaSymbol{)}\AgdaSpace{}%
\AgdaSymbol{=}\<%
\\
\>[0][@{}l@{\AgdaIndent{0}}]%
\>[5]\AgdaFunction{cong₂}\AgdaSpace{}%
\AgdaOperator{\AgdaInductiveConstructor{ω\textasciicircum{}\AgdaUnderscore{}+\AgdaUnderscore{}}}\AgdaSpace{}%
\AgdaSymbol{(}\AgdaFunction{S2M2T=pr₁}\AgdaSpace{}%
\AgdaSymbol{(}\AgdaBound{a}\AgdaSpace{}%
\AgdaOperator{\AgdaInductiveConstructor{,}}\AgdaSpace{}%
\AgdaBound{p}\AgdaSymbol{))}\AgdaSpace{}%
\AgdaSymbol{(}\AgdaFunction{S2M2T=pr₁}\AgdaSpace{}%
\AgdaSymbol{(}\AgdaBound{b}\AgdaSpace{}%
\AgdaOperator{\AgdaInductiveConstructor{,}}\AgdaSpace{}%
\AgdaBound{q}\AgdaSymbol{))}\<%
\\
\\[\AgdaEmptyExtraSkip]%
\>[0]\AgdaFunction{S2M2S=id}\AgdaSpace{}%
\AgdaSymbol{:}\AgdaSpace{}%
\AgdaSymbol{(}\AgdaBound{a}\AgdaSpace{}%
\AgdaSymbol{:}\AgdaSpace{}%
\AgdaFunction{SigmaOrd}\AgdaSymbol{)}\AgdaSpace{}%
\AgdaSymbol{→}\AgdaSpace{}%
\AgdaFunction{M2S}\AgdaSpace{}%
\AgdaSymbol{(}\AgdaFunction{S2M}\AgdaSpace{}%
\AgdaBound{a}\AgdaSymbol{)}\AgdaSpace{}%
\AgdaOperator{\AgdaFunction{≡}}\AgdaSpace{}%
\AgdaBound{a}\<%
\\
\>[0]\AgdaFunction{S2M2S=id}\AgdaSpace{}%
\AgdaBound{a}\AgdaSpace{}%
\AgdaSymbol{=}\AgdaSpace{}%
\AgdaFunction{SigmaOrd⁼}\AgdaSpace{}%
\AgdaSymbol{(}\AgdaFunction{S2M2T=pr₁}\AgdaSpace{}%
\AgdaBound{a}\AgdaSymbol{)}\<%
\\
\\[\AgdaEmptyExtraSkip]%
\>[0]\AgdaFunction{M2S2M=id}\AgdaSpace{}%
\AgdaSymbol{:}\AgdaSpace{}%
\AgdaSymbol{(}\AgdaBound{a}\AgdaSpace{}%
\AgdaSymbol{:}\AgdaSpace{}%
\AgdaDatatype{MutualOrd}\AgdaSymbol{)}\AgdaSpace{}%
\AgdaSymbol{→}\AgdaSpace{}%
\AgdaFunction{S2M}\AgdaSpace{}%
\AgdaSymbol{(}\AgdaFunction{M2S}\AgdaSpace{}%
\AgdaBound{a}\AgdaSymbol{)}\AgdaSpace{}%
\AgdaOperator{\AgdaFunction{≡}}\AgdaSpace{}%
\AgdaBound{a}\<%
\\
\>[0]\AgdaFunction{M2S2M=id}\AgdaSpace{}%
\AgdaInductiveConstructor{𝟎}\AgdaSpace{}%
\AgdaSymbol{=}\AgdaSpace{}%
\AgdaFunction{refl}\<%
\\
\>[0]\AgdaFunction{M2S2M=id}\AgdaSpace{}%
\AgdaSymbol{(}\AgdaOperator{\AgdaInductiveConstructor{ω\textasciicircum{}}}\AgdaSpace{}%
\AgdaBound{a}\AgdaSpace{}%
\AgdaOperator{\AgdaInductiveConstructor{+}}\AgdaSpace{}%
\AgdaBound{b}\AgdaSpace{}%
\AgdaOperator{\AgdaInductiveConstructor{[}}\AgdaSpace{}%
\AgdaSymbol{\AgdaUnderscore{}}\AgdaSpace{}%
\AgdaOperator{\AgdaInductiveConstructor{]}}\AgdaSymbol{)}\AgdaSpace{}%
\AgdaSymbol{=}\<%
\\
\>[0][@{}l@{\AgdaIndent{0}}]%
\>[5]\AgdaFunction{MutualOrd⁼}\AgdaSpace{}%
\AgdaSymbol{(}\AgdaFunction{M2S2M=id}\AgdaSpace{}%
\AgdaBound{a}\AgdaSymbol{)}\AgdaSpace{}%
\AgdaSymbol{(}\AgdaFunction{M2S2M=id}\AgdaSpace{}%
\AgdaBound{b}\AgdaSymbol{)}\<%
\end{code}}

\newcommand{\eqSM}{
\begin{code}[inline]%
\>[0]\AgdaFunction{S≃M}\AgdaSpace{}%
\AgdaSymbol{:}\AgdaSpace{}%
\AgdaFunction{SigmaOrd}\AgdaSpace{}%
\AgdaOperator{\AgdaFunction{≃}}\AgdaSpace{}%
\AgdaDatatype{MutualOrd}\<%
\end{code}}

\newcommand{\pathSM}{
\begin{code}[inline]%
\>[0]\AgdaFunction{S≡M}\AgdaSpace{}%
\AgdaSymbol{:}\AgdaSpace{}%
\AgdaFunction{SigmaOrd}\AgdaSpace{}%
\AgdaOperator{\AgdaFunction{≡}}\AgdaSpace{}%
\AgdaDatatype{MutualOrd}\<%
\end{code}}

\newcommand{\MtoH}{
\begin{code}%
\>[0]\AgdaFunction{M2H}\AgdaSpace{}%
\AgdaSymbol{:}\AgdaSpace{}%
\AgdaDatatype{MutualOrd}\AgdaSpace{}%
\AgdaSymbol{→}\AgdaSpace{}%
\AgdaDatatype{HITOrd}\<%
\\
\>[0]\AgdaFunction{M2H}\AgdaSpace{}%
\AgdaInductiveConstructor{𝟎}\AgdaSpace{}%
\AgdaSymbol{=}\AgdaSpace{}%
\AgdaInductiveConstructor{𝟎}\<%
\\
\>[0]\AgdaFunction{M2H}\AgdaSpace{}%
\AgdaSymbol{(}\AgdaOperator{\AgdaInductiveConstructor{ω\textasciicircum{}}}\AgdaSpace{}%
\AgdaBound{a}\AgdaSpace{}%
\AgdaOperator{\AgdaInductiveConstructor{+}}\AgdaSpace{}%
\AgdaBound{b}\AgdaSpace{}%
\AgdaOperator{\AgdaInductiveConstructor{[}}\AgdaSpace{}%
\AgdaSymbol{\AgdaUnderscore{}}\AgdaSpace{}%
\AgdaOperator{\AgdaInductiveConstructor{]}}\AgdaSymbol{)}\AgdaSpace{}%
\AgdaSymbol{=}\AgdaSpace{}%
\AgdaOperator{\AgdaInductiveConstructor{ω\textasciicircum{}}}\AgdaSpace{}%
\AgdaSymbol{(}\AgdaFunction{M2H}\AgdaSpace{}%
\AgdaBound{a}\AgdaSymbol{)}\AgdaSpace{}%
\AgdaOperator{\AgdaInductiveConstructor{⊕}}\AgdaSpace{}%
\AgdaSymbol{(}\AgdaFunction{M2H}\AgdaSpace{}%
\AgdaBound{b}\AgdaSymbol{)}\<%
\end{code}}

\newcommand{\insertType}{
\begin{code}%
\>[0]\AgdaFunction{insert}\AgdaSpace{}%
\AgdaSymbol{:}\AgdaSpace{}%
\AgdaDatatype{MutualOrd}\AgdaSpace{}%
\AgdaSymbol{→}\AgdaSpace{}%
\AgdaDatatype{MutualOrd}\AgdaSpace{}%
\AgdaSymbol{→}\AgdaSpace{}%
\AgdaDatatype{MutualOrd}\<%
\\
\>[0]\AgdaFunction{≥fst-insert}%
\>[494I]\AgdaSymbol{:}\AgdaSpace{}%
\AgdaSymbol{\{}\AgdaBound{a}\AgdaSpace{}%
\AgdaBound{b}\AgdaSpace{}%
\AgdaSymbol{:}\AgdaSpace{}%
\AgdaDatatype{MutualOrd}\AgdaSymbol{\}}\AgdaSpace{}%
\AgdaSymbol{(}\AgdaBound{c}\AgdaSpace{}%
\AgdaSymbol{:}\AgdaSpace{}%
\AgdaDatatype{MutualOrd}\AgdaSymbol{)}\<%
\\
\>[.][@{}l@{}]\<[494I]%
\>[12]\AgdaSymbol{→}\AgdaSpace{}%
\AgdaBound{b}\AgdaSpace{}%
\AgdaOperator{\AgdaFunction{M.≥}}\AgdaSpace{}%
\AgdaFunction{M.fst}\AgdaSpace{}%
\AgdaBound{c}\AgdaSpace{}%
\AgdaSymbol{→}\AgdaSpace{}%
\AgdaBound{a}\AgdaSpace{}%
\AgdaOperator{\AgdaDatatype{M.<}}\AgdaSpace{}%
\AgdaBound{b}\<%
\\
\>[12]\AgdaSymbol{→}\AgdaSpace{}%
\AgdaBound{b}\AgdaSpace{}%
\AgdaOperator{\AgdaFunction{M.≥}}\AgdaSpace{}%
\AgdaFunction{M.fst}\AgdaSpace{}%
\AgdaSymbol{(}\AgdaFunction{insert}\AgdaSpace{}%
\AgdaBound{a}\AgdaSpace{}%
\AgdaBound{c}\AgdaSymbol{)}\<%
\end{code}}

\newcommand{\insertDef}{
\begin{code}%
\>[0]\AgdaFunction{insert}\AgdaSpace{}%
\AgdaBound{a}\AgdaSpace{}%
\AgdaInductiveConstructor{𝟎}\AgdaSpace{}%
\AgdaSymbol{=}\AgdaSpace{}%
\AgdaOperator{\AgdaInductiveConstructor{ω\textasciicircum{}}}\AgdaSpace{}%
\AgdaBound{a}\AgdaSpace{}%
\AgdaOperator{\AgdaInductiveConstructor{+}}\AgdaSpace{}%
\AgdaInductiveConstructor{𝟎}\AgdaSpace{}%
\AgdaOperator{\AgdaInductiveConstructor{[}}\AgdaSpace{}%
\AgdaFunction{M.≥𝟎}\AgdaSpace{}%
\AgdaOperator{\AgdaInductiveConstructor{]}}\<%
\\
\>[0]\AgdaFunction{insert}\AgdaSpace{}%
\AgdaBound{a}\AgdaSpace{}%
\AgdaSymbol{(}\AgdaOperator{\AgdaInductiveConstructor{ω\textasciicircum{}}}\AgdaSpace{}%
\AgdaBound{b}\AgdaSpace{}%
\AgdaOperator{\AgdaInductiveConstructor{+}}\AgdaSpace{}%
\AgdaBound{c}\AgdaSpace{}%
\AgdaOperator{\AgdaInductiveConstructor{[}}\AgdaSpace{}%
\AgdaBound{r}\AgdaSpace{}%
\AgdaOperator{\AgdaInductiveConstructor{]}}\AgdaSymbol{)}\AgdaSpace{}%
\AgdaKeyword{with}\AgdaSpace{}%
\AgdaFunction{<-tri}\AgdaSpace{}%
\AgdaBound{a}\AgdaSpace{}%
\AgdaBound{b}\<%
\\
\>[0]\AgdaSymbol{...}\AgdaSpace{}%
\AgdaSymbol{|}\AgdaSpace{}%
\AgdaInductiveConstructor{inj₁}\AgdaSpace{}%
\AgdaBound{a<b}%
\>[16]\AgdaSymbol{=}\AgdaSpace{}%
\AgdaOperator{\AgdaInductiveConstructor{ω\textasciicircum{}}}\AgdaSpace{}%
\AgdaBound{b}\AgdaSpace{}%
\AgdaOperator{\AgdaInductiveConstructor{+}}\AgdaSpace{}%
\AgdaFunction{insert}\AgdaSpace{}%
\AgdaBound{a}\AgdaSpace{}%
\AgdaBound{c}\AgdaSpace{}%
\AgdaOperator{\AgdaInductiveConstructor{[}}\AgdaSpace{}%
\AgdaFunction{≥fst-insert}\AgdaSpace{}%
\AgdaBound{c}\AgdaSpace{}%
\AgdaBound{r}\AgdaSpace{}%
\AgdaBound{a<b}\AgdaSpace{}%
\AgdaOperator{\AgdaInductiveConstructor{]}}\<%
\\
\>[0]\AgdaSymbol{...}\AgdaSpace{}%
\AgdaSymbol{|}\AgdaSpace{}%
\AgdaInductiveConstructor{inj₂}\AgdaSpace{}%
\AgdaBound{a≥b}%
\>[16]\AgdaSymbol{=}\AgdaSpace{}%
\AgdaOperator{\AgdaInductiveConstructor{ω\textasciicircum{}}}\AgdaSpace{}%
\AgdaBound{a}\AgdaSpace{}%
\AgdaOperator{\AgdaInductiveConstructor{+}}\AgdaSpace{}%
\AgdaOperator{\AgdaInductiveConstructor{ω\textasciicircum{}}}\AgdaSpace{}%
\AgdaBound{b}\AgdaSpace{}%
\AgdaOperator{\AgdaInductiveConstructor{+}}\AgdaSpace{}%
\AgdaBound{c}\AgdaSpace{}%
\AgdaOperator{\AgdaInductiveConstructor{[}}\AgdaSpace{}%
\AgdaBound{r}\AgdaSpace{}%
\AgdaOperator{\AgdaInductiveConstructor{]}}\AgdaSpace{}%
\AgdaOperator{\AgdaInductiveConstructor{[}}\AgdaSpace{}%
\AgdaBound{a≥b}\AgdaSpace{}%
\AgdaOperator{\AgdaInductiveConstructor{]}}\<%
\end{code}}

\newcommand{\fstInsert}{
\begin{code}%
\>[0]\AgdaFunction{≥fst-insert}\AgdaSpace{}%
\AgdaSymbol{\{}\AgdaBound{a}\AgdaSymbol{\}}\AgdaSpace{}%
\AgdaInductiveConstructor{𝟎}\AgdaSpace{}%
\AgdaSymbol{\AgdaUnderscore{}}\AgdaSpace{}%
\AgdaBound{a<b}\AgdaSpace{}%
\AgdaSymbol{=}\AgdaSpace{}%
\AgdaInductiveConstructor{inj₁}\AgdaSpace{}%
\AgdaBound{a<b}\<%
\\
\>[0]\AgdaFunction{≥fst-insert}\AgdaSpace{}%
\AgdaSymbol{\{}\AgdaBound{a}\AgdaSymbol{\}}\AgdaSpace{}%
\AgdaSymbol{(}\AgdaOperator{\AgdaInductiveConstructor{ω\textasciicircum{}}}\AgdaSpace{}%
\AgdaBound{c}\AgdaSpace{}%
\AgdaOperator{\AgdaInductiveConstructor{+}}\AgdaSpace{}%
\AgdaBound{d}\AgdaSpace{}%
\AgdaOperator{\AgdaInductiveConstructor{[}}\AgdaSpace{}%
\AgdaSymbol{\AgdaUnderscore{}}\AgdaSpace{}%
\AgdaOperator{\AgdaInductiveConstructor{]}}\AgdaSymbol{)}\AgdaSpace{}%
\AgdaBound{b≥c}\AgdaSpace{}%
\AgdaBound{a<b}\AgdaSpace{}%
\AgdaKeyword{with}\AgdaSpace{}%
\AgdaFunction{<-tri}\AgdaSpace{}%
\AgdaBound{a}\AgdaSpace{}%
\AgdaBound{c}\<%
\\
\>[0]\AgdaSymbol{...}\AgdaSpace{}%
\AgdaSymbol{|}\AgdaSpace{}%
\AgdaInductiveConstructor{inj₁}\AgdaSpace{}%
\AgdaBound{a<c}%
\>[16]\AgdaSymbol{=}\AgdaSpace{}%
\AgdaBound{b≥c}\<%
\\
\>[0]\AgdaSymbol{...}\AgdaSpace{}%
\AgdaSymbol{|}\AgdaSpace{}%
\AgdaInductiveConstructor{inj₂}\AgdaSpace{}%
\AgdaBound{a≥c}%
\>[16]\AgdaSymbol{=}\AgdaSpace{}%
\AgdaInductiveConstructor{inj₁}\AgdaSpace{}%
\AgdaBound{a<b}\<%
\end{code}}

\newcommand{\insertSwap}{
\begin{code}%
\>[0]\AgdaFunction{insert-swap}%
\>[599I]\AgdaSymbol{:}\AgdaSpace{}%
\AgdaSymbol{(}\AgdaBound{x}\AgdaSpace{}%
\AgdaBound{y}\AgdaSpace{}%
\AgdaBound{z}\AgdaSpace{}%
\AgdaSymbol{:}\AgdaSpace{}%
\AgdaDatatype{MutualOrd}\AgdaSymbol{)}\<%
\\
\>[.][@{}l@{}]\<[599I]%
\>[12]\AgdaSymbol{→}\AgdaSpace{}%
\AgdaFunction{insert}\AgdaSpace{}%
\AgdaBound{x}\AgdaSpace{}%
\AgdaSymbol{(}\AgdaFunction{insert}\AgdaSpace{}%
\AgdaBound{y}\AgdaSpace{}%
\AgdaBound{z}\AgdaSymbol{)}\AgdaSpace{}%
\AgdaOperator{\AgdaFunction{≡}}\AgdaSpace{}%
\AgdaFunction{insert}\AgdaSpace{}%
\AgdaBound{y}\AgdaSpace{}%
\AgdaSymbol{(}\AgdaFunction{insert}\AgdaSpace{}%
\AgdaBound{x}\AgdaSpace{}%
\AgdaBound{z}\AgdaSymbol{)}\<%
\end{code}}

\begin{code}[hide]%
\>[0]\<%
\\
\>[0]\AgdaFunction{insert-swap}\AgdaSpace{}%
\AgdaBound{x}\AgdaSpace{}%
\AgdaBound{y}\AgdaSpace{}%
\AgdaInductiveConstructor{𝟎}\AgdaSpace{}%
\AgdaKeyword{with}\AgdaSpace{}%
\AgdaFunction{<-tri}\AgdaSpace{}%
\AgdaBound{x}\AgdaSpace{}%
\AgdaBound{y}\<%
\\
\>[0]\AgdaFunction{insert-swap}\AgdaSpace{}%
\AgdaBound{x}\AgdaSpace{}%
\AgdaBound{y}\AgdaSpace{}%
\AgdaInductiveConstructor{𝟎}\AgdaSpace{}%
\AgdaSymbol{|}\AgdaSpace{}%
\AgdaInductiveConstructor{inj₁}\AgdaSpace{}%
\AgdaBound{x<y}\AgdaSpace{}%
\AgdaKeyword{with}\AgdaSpace{}%
\AgdaFunction{<-tri}\AgdaSpace{}%
\AgdaBound{y}\AgdaSpace{}%
\AgdaBound{x}\<%
\\
\>[0]\AgdaFunction{insert-swap}\AgdaSpace{}%
\AgdaBound{x}\AgdaSpace{}%
\AgdaBound{y}\AgdaSpace{}%
\AgdaInductiveConstructor{𝟎}\AgdaSpace{}%
\AgdaSymbol{|}\AgdaSpace{}%
\AgdaInductiveConstructor{inj₁}\AgdaSpace{}%
\AgdaBound{x<y}\AgdaSpace{}%
\AgdaSymbol{|}\AgdaSpace{}%
\AgdaInductiveConstructor{inj₁}\AgdaSpace{}%
\AgdaBound{y<x}\AgdaSpace{}%
\AgdaSymbol{=}\AgdaSpace{}%
\AgdaFunction{⊥-elim}\AgdaSpace{}%
\AgdaSymbol{(}\AgdaFunction{Lm[<→¬≥]}\AgdaSpace{}%
\AgdaBound{x<y}\AgdaSpace{}%
\AgdaSymbol{(}\AgdaInductiveConstructor{inj₁}\AgdaSpace{}%
\AgdaBound{y<x}\AgdaSymbol{))}\<%
\\
\>[0]\AgdaFunction{insert-swap}\AgdaSpace{}%
\AgdaBound{x}\AgdaSpace{}%
\AgdaBound{y}\AgdaSpace{}%
\AgdaInductiveConstructor{𝟎}\AgdaSpace{}%
\AgdaSymbol{|}\AgdaSpace{}%
\AgdaInductiveConstructor{inj₁}\AgdaSpace{}%
\AgdaBound{x<y}\AgdaSpace{}%
\AgdaSymbol{|}\AgdaSpace{}%
\AgdaInductiveConstructor{inj₂}\AgdaSpace{}%
\AgdaBound{y≥x}\AgdaSpace{}%
\AgdaSymbol{=}\AgdaSpace{}%
\AgdaFunction{MutualOrd⁼}\AgdaSpace{}%
\AgdaFunction{refl}\AgdaSpace{}%
\AgdaSymbol{(}\AgdaFunction{MutualOrd⁼}\AgdaSpace{}%
\AgdaFunction{refl}\AgdaSpace{}%
\AgdaFunction{refl}\AgdaSymbol{)}\<%
\\
\>[0]\AgdaFunction{insert-swap}\AgdaSpace{}%
\AgdaBound{x}\AgdaSpace{}%
\AgdaBound{y}\AgdaSpace{}%
\AgdaInductiveConstructor{𝟎}\AgdaSpace{}%
\AgdaSymbol{|}\AgdaSpace{}%
\AgdaInductiveConstructor{inj₂}\AgdaSpace{}%
\AgdaBound{x≥y}\AgdaSpace{}%
\AgdaKeyword{with}\AgdaSpace{}%
\AgdaFunction{<-tri}\AgdaSpace{}%
\AgdaBound{y}\AgdaSpace{}%
\AgdaBound{x}\<%
\\
\>[0]\AgdaFunction{insert-swap}\AgdaSpace{}%
\AgdaBound{x}\AgdaSpace{}%
\AgdaBound{y}\AgdaSpace{}%
\AgdaInductiveConstructor{𝟎}\AgdaSpace{}%
\AgdaSymbol{|}\AgdaSpace{}%
\AgdaInductiveConstructor{inj₂}\AgdaSpace{}%
\AgdaBound{x≥y}\AgdaSpace{}%
\AgdaSymbol{|}\AgdaSpace{}%
\AgdaInductiveConstructor{inj₁}\AgdaSpace{}%
\AgdaBound{y<x}\AgdaSpace{}%
\AgdaSymbol{=}\AgdaSpace{}%
\AgdaFunction{MutualOrd⁼}\AgdaSpace{}%
\AgdaFunction{refl}\AgdaSpace{}%
\AgdaSymbol{(}\AgdaFunction{MutualOrd⁼}\AgdaSpace{}%
\AgdaFunction{refl}\AgdaSpace{}%
\AgdaFunction{refl}\AgdaSymbol{)}\<%
\\
\>[0]\AgdaFunction{insert-swap}\AgdaSpace{}%
\AgdaBound{x}\AgdaSpace{}%
\AgdaBound{y}\AgdaSpace{}%
\AgdaInductiveConstructor{𝟎}\AgdaSpace{}%
\AgdaSymbol{|}\AgdaSpace{}%
\AgdaInductiveConstructor{inj₂}\AgdaSpace{}%
\AgdaBound{x≥y}\AgdaSpace{}%
\AgdaSymbol{|}\AgdaSpace{}%
\AgdaInductiveConstructor{inj₂}\AgdaSpace{}%
\AgdaBound{y≥x}\AgdaSpace{}%
\AgdaSymbol{=}\AgdaSpace{}%
\AgdaFunction{MutualOrd⁼}\AgdaSpace{}%
\AgdaSymbol{(}\AgdaFunction{≤≥→≡}\AgdaSpace{}%
\AgdaBound{y≥x}\AgdaSpace{}%
\AgdaBound{x≥y}\AgdaSymbol{)}\AgdaSpace{}%
\AgdaSymbol{(}\AgdaFunction{MutualOrd⁼}\AgdaSpace{}%
\AgdaSymbol{(}\AgdaFunction{≤≥→≡}\AgdaSpace{}%
\AgdaBound{x≥y}\AgdaSpace{}%
\AgdaBound{y≥x}\AgdaSymbol{)}\AgdaSpace{}%
\AgdaFunction{refl}\AgdaSymbol{)}\<%
\\
\>[0]\AgdaFunction{insert-swap}\AgdaSpace{}%
\AgdaBound{x}\AgdaSpace{}%
\AgdaBound{y}\AgdaSpace{}%
\AgdaSymbol{(}\AgdaOperator{\AgdaInductiveConstructor{ω\textasciicircum{}}}\AgdaSpace{}%
\AgdaBound{a}\AgdaSpace{}%
\AgdaOperator{\AgdaInductiveConstructor{+}}\AgdaSpace{}%
\AgdaBound{b}\AgdaSpace{}%
\AgdaOperator{\AgdaInductiveConstructor{[}}\AgdaSpace{}%
\AgdaSymbol{\AgdaUnderscore{}}\AgdaSpace{}%
\AgdaOperator{\AgdaInductiveConstructor{]}}\AgdaSymbol{)}\AgdaSpace{}%
\AgdaKeyword{with}\AgdaSpace{}%
\AgdaFunction{<-tri}\AgdaSpace{}%
\AgdaBound{y}\AgdaSpace{}%
\AgdaBound{a}\<%
\\
\>[0]\AgdaFunction{insert-swap}\AgdaSpace{}%
\AgdaBound{x}\AgdaSpace{}%
\AgdaBound{y}\AgdaSpace{}%
\AgdaSymbol{(}\AgdaOperator{\AgdaInductiveConstructor{ω\textasciicircum{}}}\AgdaSpace{}%
\AgdaBound{a}\AgdaSpace{}%
\AgdaOperator{\AgdaInductiveConstructor{+}}\AgdaSpace{}%
\AgdaBound{b}\AgdaSpace{}%
\AgdaOperator{\AgdaInductiveConstructor{[}}\AgdaSpace{}%
\AgdaSymbol{\AgdaUnderscore{}}\AgdaSpace{}%
\AgdaOperator{\AgdaInductiveConstructor{]}}\AgdaSymbol{)}\AgdaSpace{}%
\AgdaSymbol{|}\AgdaSpace{}%
\AgdaInductiveConstructor{inj₁}\AgdaSpace{}%
\AgdaBound{y<a}\AgdaSpace{}%
\AgdaKeyword{with}\AgdaSpace{}%
\AgdaFunction{<-tri}\AgdaSpace{}%
\AgdaBound{x}\AgdaSpace{}%
\AgdaBound{a}\<%
\\
\>[0]\AgdaFunction{insert-swap}\AgdaSpace{}%
\AgdaBound{x}\AgdaSpace{}%
\AgdaBound{y}\AgdaSpace{}%
\AgdaSymbol{(}\AgdaOperator{\AgdaInductiveConstructor{ω\textasciicircum{}}}\AgdaSpace{}%
\AgdaBound{a}\AgdaSpace{}%
\AgdaOperator{\AgdaInductiveConstructor{+}}\AgdaSpace{}%
\AgdaBound{b}\AgdaSpace{}%
\AgdaOperator{\AgdaInductiveConstructor{[}}\AgdaSpace{}%
\AgdaSymbol{\AgdaUnderscore{}}\AgdaSpace{}%
\AgdaOperator{\AgdaInductiveConstructor{]}}\AgdaSymbol{)}\AgdaSpace{}%
\AgdaSymbol{|}\AgdaSpace{}%
\AgdaInductiveConstructor{inj₁}\AgdaSpace{}%
\AgdaBound{y<a}\AgdaSpace{}%
\AgdaSymbol{|}\AgdaSpace{}%
\AgdaInductiveConstructor{inj₁}\AgdaSpace{}%
\AgdaBound{x<a}\AgdaSpace{}%
\AgdaKeyword{with}\AgdaSpace{}%
\AgdaFunction{<-tri}\AgdaSpace{}%
\AgdaBound{y}\AgdaSpace{}%
\AgdaBound{a}\<%
\\
\>[0]\AgdaFunction{insert-swap}\AgdaSpace{}%
\AgdaBound{x}\AgdaSpace{}%
\AgdaBound{y}\AgdaSpace{}%
\AgdaSymbol{(}\AgdaOperator{\AgdaInductiveConstructor{ω\textasciicircum{}}}\AgdaSpace{}%
\AgdaBound{a}\AgdaSpace{}%
\AgdaOperator{\AgdaInductiveConstructor{+}}\AgdaSpace{}%
\AgdaBound{b}\AgdaSpace{}%
\AgdaOperator{\AgdaInductiveConstructor{[}}\AgdaSpace{}%
\AgdaSymbol{\AgdaUnderscore{}}\AgdaSpace{}%
\AgdaOperator{\AgdaInductiveConstructor{]}}\AgdaSymbol{)}\AgdaSpace{}%
\AgdaSymbol{|}\AgdaSpace{}%
\AgdaInductiveConstructor{inj₁}\AgdaSpace{}%
\AgdaBound{y<a}\AgdaSpace{}%
\AgdaSymbol{|}\AgdaSpace{}%
\AgdaInductiveConstructor{inj₁}\AgdaSpace{}%
\AgdaBound{x<a}\AgdaSpace{}%
\AgdaSymbol{|}\AgdaSpace{}%
\AgdaInductiveConstructor{inj₁}\AgdaSpace{}%
\AgdaBound{y<a'}\AgdaSpace{}%
\AgdaSymbol{=}\AgdaSpace{}%
\AgdaFunction{MutualOrd⁼}\AgdaSpace{}%
\AgdaFunction{refl}\AgdaSpace{}%
\AgdaSymbol{(}\AgdaFunction{insert-swap}\AgdaSpace{}%
\AgdaBound{x}\AgdaSpace{}%
\AgdaBound{y}\AgdaSpace{}%
\AgdaBound{b}\AgdaSymbol{)}\<%
\\
\>[0]\AgdaFunction{insert-swap}\AgdaSpace{}%
\AgdaBound{x}\AgdaSpace{}%
\AgdaBound{y}\AgdaSpace{}%
\AgdaSymbol{(}\AgdaOperator{\AgdaInductiveConstructor{ω\textasciicircum{}}}\AgdaSpace{}%
\AgdaBound{a}\AgdaSpace{}%
\AgdaOperator{\AgdaInductiveConstructor{+}}\AgdaSpace{}%
\AgdaBound{b}\AgdaSpace{}%
\AgdaOperator{\AgdaInductiveConstructor{[}}\AgdaSpace{}%
\AgdaSymbol{\AgdaUnderscore{}}\AgdaSpace{}%
\AgdaOperator{\AgdaInductiveConstructor{]}}\AgdaSymbol{)}\AgdaSpace{}%
\AgdaSymbol{|}\AgdaSpace{}%
\AgdaInductiveConstructor{inj₁}\AgdaSpace{}%
\AgdaBound{y<a}\AgdaSpace{}%
\AgdaSymbol{|}\AgdaSpace{}%
\AgdaInductiveConstructor{inj₁}\AgdaSpace{}%
\AgdaBound{x<a}\AgdaSpace{}%
\AgdaSymbol{|}\AgdaSpace{}%
\AgdaInductiveConstructor{inj₂}\AgdaSpace{}%
\AgdaBound{y≥a}%
\>[67]\AgdaSymbol{=}\AgdaSpace{}%
\AgdaFunction{⊥-elim}\AgdaSpace{}%
\AgdaSymbol{(}\AgdaFunction{Lm[<→¬≥]}\AgdaSpace{}%
\AgdaBound{y<a}\AgdaSpace{}%
\AgdaBound{y≥a}\AgdaSymbol{)}\<%
\\
\>[0]\AgdaFunction{insert-swap}\AgdaSpace{}%
\AgdaBound{x}\AgdaSpace{}%
\AgdaBound{y}\AgdaSpace{}%
\AgdaSymbol{(}\AgdaOperator{\AgdaInductiveConstructor{ω\textasciicircum{}}}\AgdaSpace{}%
\AgdaBound{a}\AgdaSpace{}%
\AgdaOperator{\AgdaInductiveConstructor{+}}\AgdaSpace{}%
\AgdaBound{b}\AgdaSpace{}%
\AgdaOperator{\AgdaInductiveConstructor{[}}\AgdaSpace{}%
\AgdaSymbol{\AgdaUnderscore{}}\AgdaSpace{}%
\AgdaOperator{\AgdaInductiveConstructor{]}}\AgdaSymbol{)}\AgdaSpace{}%
\AgdaSymbol{|}\AgdaSpace{}%
\AgdaInductiveConstructor{inj₁}\AgdaSpace{}%
\AgdaBound{y<a}\AgdaSpace{}%
\AgdaSymbol{|}\AgdaSpace{}%
\AgdaInductiveConstructor{inj₂}\AgdaSpace{}%
\AgdaBound{x≥a}\AgdaSpace{}%
\AgdaKeyword{with}\AgdaSpace{}%
\AgdaFunction{<-tri}\AgdaSpace{}%
\AgdaBound{y}\AgdaSpace{}%
\AgdaBound{x}\<%
\\
\>[0]\AgdaFunction{insert-swap}\AgdaSpace{}%
\AgdaBound{x}\AgdaSpace{}%
\AgdaBound{y}\AgdaSpace{}%
\AgdaSymbol{(}\AgdaOperator{\AgdaInductiveConstructor{ω\textasciicircum{}}}\AgdaSpace{}%
\AgdaBound{a}\AgdaSpace{}%
\AgdaOperator{\AgdaInductiveConstructor{+}}\AgdaSpace{}%
\AgdaBound{b}\AgdaSpace{}%
\AgdaOperator{\AgdaInductiveConstructor{[}}\AgdaSpace{}%
\AgdaSymbol{\AgdaUnderscore{}}\AgdaSpace{}%
\AgdaOperator{\AgdaInductiveConstructor{]}}\AgdaSymbol{)}\AgdaSpace{}%
\AgdaSymbol{|}\AgdaSpace{}%
\AgdaInductiveConstructor{inj₁}\AgdaSpace{}%
\AgdaBound{y<a}\AgdaSpace{}%
\AgdaSymbol{|}\AgdaSpace{}%
\AgdaInductiveConstructor{inj₂}\AgdaSpace{}%
\AgdaBound{x≥a}\AgdaSpace{}%
\AgdaSymbol{|}\AgdaSpace{}%
\AgdaInductiveConstructor{inj₁}\AgdaSpace{}%
\AgdaBound{y<x}\AgdaSpace{}%
\AgdaKeyword{with}\AgdaSpace{}%
\AgdaFunction{<-tri}\AgdaSpace{}%
\AgdaBound{y}\AgdaSpace{}%
\AgdaBound{a}\<%
\\
\>[0]\AgdaFunction{insert-swap}\AgdaSpace{}%
\AgdaBound{x}\AgdaSpace{}%
\AgdaBound{y}\AgdaSpace{}%
\AgdaSymbol{(}\AgdaOperator{\AgdaInductiveConstructor{ω\textasciicircum{}}}\AgdaSpace{}%
\AgdaBound{a}\AgdaSpace{}%
\AgdaOperator{\AgdaInductiveConstructor{+}}\AgdaSpace{}%
\AgdaBound{b}\AgdaSpace{}%
\AgdaOperator{\AgdaInductiveConstructor{[}}\AgdaSpace{}%
\AgdaSymbol{\AgdaUnderscore{}}\AgdaSpace{}%
\AgdaOperator{\AgdaInductiveConstructor{]}}\AgdaSymbol{)}\AgdaSpace{}%
\AgdaSymbol{|}\AgdaSpace{}%
\AgdaInductiveConstructor{inj₁}\AgdaSpace{}%
\AgdaBound{y<a}\AgdaSpace{}%
\AgdaSymbol{|}\AgdaSpace{}%
\AgdaInductiveConstructor{inj₂}\AgdaSpace{}%
\AgdaBound{x≥a}\AgdaSpace{}%
\AgdaSymbol{|}\AgdaSpace{}%
\AgdaInductiveConstructor{inj₁}\AgdaSpace{}%
\AgdaBound{y<x}\AgdaSpace{}%
\AgdaSymbol{|}\AgdaSpace{}%
\AgdaInductiveConstructor{inj₁}\AgdaSpace{}%
\AgdaBound{y<a'}\AgdaSpace{}%
\AgdaSymbol{=}\AgdaSpace{}%
\AgdaFunction{MutualOrd⁼}\AgdaSpace{}%
\AgdaFunction{refl}\AgdaSpace{}%
\AgdaSymbol{(}\AgdaFunction{MutualOrd⁼}\AgdaSpace{}%
\AgdaFunction{refl}\AgdaSpace{}%
\AgdaFunction{refl}\AgdaSymbol{)}\<%
\\
\>[0]\AgdaFunction{insert-swap}\AgdaSpace{}%
\AgdaBound{x}\AgdaSpace{}%
\AgdaBound{y}\AgdaSpace{}%
\AgdaSymbol{(}\AgdaOperator{\AgdaInductiveConstructor{ω\textasciicircum{}}}\AgdaSpace{}%
\AgdaBound{a}\AgdaSpace{}%
\AgdaOperator{\AgdaInductiveConstructor{+}}\AgdaSpace{}%
\AgdaBound{b}\AgdaSpace{}%
\AgdaOperator{\AgdaInductiveConstructor{[}}\AgdaSpace{}%
\AgdaSymbol{\AgdaUnderscore{}}\AgdaSpace{}%
\AgdaOperator{\AgdaInductiveConstructor{]}}\AgdaSymbol{)}\AgdaSpace{}%
\AgdaSymbol{|}\AgdaSpace{}%
\AgdaInductiveConstructor{inj₁}\AgdaSpace{}%
\AgdaBound{y<a}\AgdaSpace{}%
\AgdaSymbol{|}\AgdaSpace{}%
\AgdaInductiveConstructor{inj₂}\AgdaSpace{}%
\AgdaBound{x≥a}\AgdaSpace{}%
\AgdaSymbol{|}\AgdaSpace{}%
\AgdaInductiveConstructor{inj₁}\AgdaSpace{}%
\AgdaBound{y<x}\AgdaSpace{}%
\AgdaSymbol{|}\AgdaSpace{}%
\AgdaInductiveConstructor{inj₂}\AgdaSpace{}%
\AgdaBound{y≥a}%
\>[78]\AgdaSymbol{=}\AgdaSpace{}%
\AgdaFunction{⊥-elim}\AgdaSpace{}%
\AgdaSymbol{(}\AgdaFunction{Lm[<→¬≥]}\AgdaSpace{}%
\AgdaBound{y<a}\AgdaSpace{}%
\AgdaBound{y≥a}\AgdaSymbol{)}\<%
\\
\>[0]\AgdaFunction{insert-swap}\AgdaSpace{}%
\AgdaBound{x}\AgdaSpace{}%
\AgdaBound{y}\AgdaSpace{}%
\AgdaSymbol{(}\AgdaOperator{\AgdaInductiveConstructor{ω\textasciicircum{}}}\AgdaSpace{}%
\AgdaBound{a}\AgdaSpace{}%
\AgdaOperator{\AgdaInductiveConstructor{+}}\AgdaSpace{}%
\AgdaBound{b}\AgdaSpace{}%
\AgdaOperator{\AgdaInductiveConstructor{[}}\AgdaSpace{}%
\AgdaSymbol{\AgdaUnderscore{}}\AgdaSpace{}%
\AgdaOperator{\AgdaInductiveConstructor{]}}\AgdaSymbol{)}\AgdaSpace{}%
\AgdaSymbol{|}\AgdaSpace{}%
\AgdaInductiveConstructor{inj₁}\AgdaSpace{}%
\AgdaBound{y<a}\AgdaSpace{}%
\AgdaSymbol{|}\AgdaSpace{}%
\AgdaInductiveConstructor{inj₂}\AgdaSpace{}%
\AgdaBound{x≥a}\AgdaSpace{}%
\AgdaSymbol{|}\AgdaSpace{}%
\AgdaInductiveConstructor{inj₂}\AgdaSpace{}%
\AgdaBound{y≥x}\AgdaSpace{}%
\AgdaSymbol{=}\AgdaSpace{}%
\AgdaFunction{⊥-elim}\AgdaSpace{}%
\AgdaSymbol{(}\AgdaFunction{Lm[<→¬≥]}\AgdaSpace{}%
\AgdaBound{y<a}\AgdaSpace{}%
\AgdaSymbol{(}\AgdaFunction{≤-trans}\AgdaSpace{}%
\AgdaBound{x≥a}\AgdaSpace{}%
\AgdaBound{y≥x}\AgdaSymbol{))}\<%
\\
\>[0]\AgdaFunction{insert-swap}\AgdaSpace{}%
\AgdaBound{x}\AgdaSpace{}%
\AgdaBound{y}\AgdaSpace{}%
\AgdaSymbol{(}\AgdaOperator{\AgdaInductiveConstructor{ω\textasciicircum{}}}\AgdaSpace{}%
\AgdaBound{a}\AgdaSpace{}%
\AgdaOperator{\AgdaInductiveConstructor{+}}\AgdaSpace{}%
\AgdaBound{b}\AgdaSpace{}%
\AgdaOperator{\AgdaInductiveConstructor{[}}\AgdaSpace{}%
\AgdaSymbol{\AgdaUnderscore{}}\AgdaSpace{}%
\AgdaOperator{\AgdaInductiveConstructor{]}}\AgdaSymbol{)}\AgdaSpace{}%
\AgdaSymbol{|}\AgdaSpace{}%
\AgdaInductiveConstructor{inj₂}\AgdaSpace{}%
\AgdaBound{y≥a}\AgdaSpace{}%
\AgdaKeyword{with}\AgdaSpace{}%
\AgdaFunction{<-tri}\AgdaSpace{}%
\AgdaBound{x}\AgdaSpace{}%
\AgdaBound{a}\<%
\\
\>[0]\AgdaFunction{insert-swap}\AgdaSpace{}%
\AgdaBound{x}\AgdaSpace{}%
\AgdaBound{y}\AgdaSpace{}%
\AgdaSymbol{(}\AgdaOperator{\AgdaInductiveConstructor{ω\textasciicircum{}}}\AgdaSpace{}%
\AgdaBound{a}\AgdaSpace{}%
\AgdaOperator{\AgdaInductiveConstructor{+}}\AgdaSpace{}%
\AgdaBound{b}\AgdaSpace{}%
\AgdaOperator{\AgdaInductiveConstructor{[}}\AgdaSpace{}%
\AgdaSymbol{\AgdaUnderscore{}}\AgdaSpace{}%
\AgdaOperator{\AgdaInductiveConstructor{]}}\AgdaSymbol{)}\AgdaSpace{}%
\AgdaSymbol{|}\AgdaSpace{}%
\AgdaInductiveConstructor{inj₂}\AgdaSpace{}%
\AgdaBound{y≥a}\AgdaSpace{}%
\AgdaSymbol{|}\AgdaSpace{}%
\AgdaInductiveConstructor{inj₁}\AgdaSpace{}%
\AgdaBound{x<a}\AgdaSpace{}%
\AgdaKeyword{with}\AgdaSpace{}%
\AgdaFunction{<-tri}\AgdaSpace{}%
\AgdaBound{y}\AgdaSpace{}%
\AgdaBound{a}\<%
\\
\>[0]\AgdaFunction{insert-swap}\AgdaSpace{}%
\AgdaBound{x}\AgdaSpace{}%
\AgdaBound{y}\AgdaSpace{}%
\AgdaSymbol{(}\AgdaOperator{\AgdaInductiveConstructor{ω\textasciicircum{}}}\AgdaSpace{}%
\AgdaBound{a}\AgdaSpace{}%
\AgdaOperator{\AgdaInductiveConstructor{+}}\AgdaSpace{}%
\AgdaBound{b}\AgdaSpace{}%
\AgdaOperator{\AgdaInductiveConstructor{[}}\AgdaSpace{}%
\AgdaSymbol{\AgdaUnderscore{}}\AgdaSpace{}%
\AgdaOperator{\AgdaInductiveConstructor{]}}\AgdaSymbol{)}\AgdaSpace{}%
\AgdaSymbol{|}\AgdaSpace{}%
\AgdaInductiveConstructor{inj₂}\AgdaSpace{}%
\AgdaBound{y≥a}\AgdaSpace{}%
\AgdaSymbol{|}\AgdaSpace{}%
\AgdaInductiveConstructor{inj₁}\AgdaSpace{}%
\AgdaBound{x<a}\AgdaSpace{}%
\AgdaSymbol{|}\AgdaSpace{}%
\AgdaInductiveConstructor{inj₁}\AgdaSpace{}%
\AgdaBound{y<a}%
\>[67]\AgdaSymbol{=}\AgdaSpace{}%
\AgdaFunction{⊥-elim}\AgdaSpace{}%
\AgdaSymbol{(}\AgdaFunction{Lm[<→¬≥]}\AgdaSpace{}%
\AgdaBound{y<a}\AgdaSpace{}%
\AgdaBound{y≥a}\AgdaSymbol{)}\<%
\\
\>[0]\AgdaFunction{insert-swap}\AgdaSpace{}%
\AgdaBound{x}\AgdaSpace{}%
\AgdaBound{y}\AgdaSpace{}%
\AgdaSymbol{(}\AgdaOperator{\AgdaInductiveConstructor{ω\textasciicircum{}}}\AgdaSpace{}%
\AgdaBound{a}\AgdaSpace{}%
\AgdaOperator{\AgdaInductiveConstructor{+}}\AgdaSpace{}%
\AgdaBound{b}\AgdaSpace{}%
\AgdaOperator{\AgdaInductiveConstructor{[}}\AgdaSpace{}%
\AgdaSymbol{\AgdaUnderscore{}}\AgdaSpace{}%
\AgdaOperator{\AgdaInductiveConstructor{]}}\AgdaSymbol{)}\AgdaSpace{}%
\AgdaSymbol{|}\AgdaSpace{}%
\AgdaInductiveConstructor{inj₂}\AgdaSpace{}%
\AgdaBound{y≥a}\AgdaSpace{}%
\AgdaSymbol{|}\AgdaSpace{}%
\AgdaInductiveConstructor{inj₁}\AgdaSpace{}%
\AgdaBound{x<a}\AgdaSpace{}%
\AgdaSymbol{|}\AgdaSpace{}%
\AgdaInductiveConstructor{inj₂}\AgdaSpace{}%
\AgdaBound{y≥a'}\AgdaSpace{}%
\AgdaKeyword{with}\AgdaSpace{}%
\AgdaFunction{<-tri}\AgdaSpace{}%
\AgdaBound{x}\AgdaSpace{}%
\AgdaBound{y}\<%
\\
\>[0]\AgdaFunction{insert-swap}\AgdaSpace{}%
\AgdaBound{x}\AgdaSpace{}%
\AgdaBound{y}\AgdaSpace{}%
\AgdaSymbol{(}\AgdaOperator{\AgdaInductiveConstructor{ω\textasciicircum{}}}\AgdaSpace{}%
\AgdaBound{a}\AgdaSpace{}%
\AgdaOperator{\AgdaInductiveConstructor{+}}\AgdaSpace{}%
\AgdaBound{b}\AgdaSpace{}%
\AgdaOperator{\AgdaInductiveConstructor{[}}\AgdaSpace{}%
\AgdaSymbol{\AgdaUnderscore{}}\AgdaSpace{}%
\AgdaOperator{\AgdaInductiveConstructor{]}}\AgdaSymbol{)}\AgdaSpace{}%
\AgdaSymbol{|}\AgdaSpace{}%
\AgdaInductiveConstructor{inj₂}\AgdaSpace{}%
\AgdaBound{y≥a}\AgdaSpace{}%
\AgdaSymbol{|}\AgdaSpace{}%
\AgdaInductiveConstructor{inj₁}\AgdaSpace{}%
\AgdaBound{x<a}\AgdaSpace{}%
\AgdaSymbol{|}\AgdaSpace{}%
\AgdaInductiveConstructor{inj₂}\AgdaSpace{}%
\AgdaBound{y≥a'}\AgdaSpace{}%
\AgdaSymbol{|}\AgdaSpace{}%
\AgdaInductiveConstructor{inj₁}\AgdaSpace{}%
\AgdaBound{x<y}\AgdaSpace{}%
\AgdaKeyword{with}\AgdaSpace{}%
\AgdaFunction{<-tri}\AgdaSpace{}%
\AgdaBound{x}\AgdaSpace{}%
\AgdaBound{a}\<%
\\
\>[0]\AgdaFunction{insert-swap}\AgdaSpace{}%
\AgdaBound{x}\AgdaSpace{}%
\AgdaBound{y}\AgdaSpace{}%
\AgdaSymbol{(}\AgdaOperator{\AgdaInductiveConstructor{ω\textasciicircum{}}}\AgdaSpace{}%
\AgdaBound{a}\AgdaSpace{}%
\AgdaOperator{\AgdaInductiveConstructor{+}}\AgdaSpace{}%
\AgdaBound{b}\AgdaSpace{}%
\AgdaOperator{\AgdaInductiveConstructor{[}}\AgdaSpace{}%
\AgdaSymbol{\AgdaUnderscore{}}\AgdaSpace{}%
\AgdaOperator{\AgdaInductiveConstructor{]}}\AgdaSymbol{)}\AgdaSpace{}%
\AgdaSymbol{|}\AgdaSpace{}%
\AgdaInductiveConstructor{inj₂}\AgdaSpace{}%
\AgdaBound{y≥a}\AgdaSpace{}%
\AgdaSymbol{|}\AgdaSpace{}%
\AgdaInductiveConstructor{inj₁}\AgdaSpace{}%
\AgdaBound{x<a}\AgdaSpace{}%
\AgdaSymbol{|}\AgdaSpace{}%
\AgdaInductiveConstructor{inj₂}\AgdaSpace{}%
\AgdaBound{y≥a'}\AgdaSpace{}%
\AgdaSymbol{|}\AgdaSpace{}%
\AgdaInductiveConstructor{inj₁}\AgdaSpace{}%
\AgdaBound{x<y}\AgdaSpace{}%
\AgdaSymbol{|}\AgdaSpace{}%
\AgdaInductiveConstructor{inj₁}\AgdaSpace{}%
\AgdaBound{x<a'}\AgdaSpace{}%
\AgdaSymbol{=}\AgdaSpace{}%
\AgdaFunction{MutualOrd⁼}\AgdaSpace{}%
\AgdaFunction{refl}\AgdaSpace{}%
\AgdaSymbol{(}\AgdaFunction{MutualOrd⁼}\AgdaSpace{}%
\AgdaFunction{refl}\AgdaSpace{}%
\AgdaFunction{refl}\AgdaSymbol{)}\<%
\\
\>[0]\AgdaFunction{insert-swap}\AgdaSpace{}%
\AgdaBound{x}\AgdaSpace{}%
\AgdaBound{y}\AgdaSpace{}%
\AgdaSymbol{(}\AgdaOperator{\AgdaInductiveConstructor{ω\textasciicircum{}}}\AgdaSpace{}%
\AgdaBound{a}\AgdaSpace{}%
\AgdaOperator{\AgdaInductiveConstructor{+}}\AgdaSpace{}%
\AgdaBound{b}\AgdaSpace{}%
\AgdaOperator{\AgdaInductiveConstructor{[}}\AgdaSpace{}%
\AgdaSymbol{\AgdaUnderscore{}}\AgdaSpace{}%
\AgdaOperator{\AgdaInductiveConstructor{]}}\AgdaSymbol{)}\AgdaSpace{}%
\AgdaSymbol{|}\AgdaSpace{}%
\AgdaInductiveConstructor{inj₂}\AgdaSpace{}%
\AgdaBound{y≥a}\AgdaSpace{}%
\AgdaSymbol{|}\AgdaSpace{}%
\AgdaInductiveConstructor{inj₁}\AgdaSpace{}%
\AgdaBound{x<a}\AgdaSpace{}%
\AgdaSymbol{|}\AgdaSpace{}%
\AgdaInductiveConstructor{inj₂}\AgdaSpace{}%
\AgdaBound{y≥a'}\AgdaSpace{}%
\AgdaSymbol{|}\AgdaSpace{}%
\AgdaInductiveConstructor{inj₁}\AgdaSpace{}%
\AgdaBound{x<y}\AgdaSpace{}%
\AgdaSymbol{|}\AgdaSpace{}%
\AgdaInductiveConstructor{inj₂}\AgdaSpace{}%
\AgdaBound{x≥a}%
\>[90]\AgdaSymbol{=}\AgdaSpace{}%
\AgdaFunction{⊥-elim}\AgdaSpace{}%
\AgdaSymbol{(}\AgdaFunction{Lm[<→¬≥]}\AgdaSpace{}%
\AgdaBound{x<a}\AgdaSpace{}%
\AgdaBound{x≥a}\AgdaSymbol{)}\<%
\\
\>[0]\AgdaFunction{insert-swap}\AgdaSpace{}%
\AgdaBound{x}\AgdaSpace{}%
\AgdaBound{y}\AgdaSpace{}%
\AgdaSymbol{(}\AgdaOperator{\AgdaInductiveConstructor{ω\textasciicircum{}}}\AgdaSpace{}%
\AgdaBound{a}\AgdaSpace{}%
\AgdaOperator{\AgdaInductiveConstructor{+}}\AgdaSpace{}%
\AgdaBound{b}\AgdaSpace{}%
\AgdaOperator{\AgdaInductiveConstructor{[}}\AgdaSpace{}%
\AgdaSymbol{\AgdaUnderscore{}}\AgdaSpace{}%
\AgdaOperator{\AgdaInductiveConstructor{]}}\AgdaSymbol{)}\AgdaSpace{}%
\AgdaSymbol{|}\AgdaSpace{}%
\AgdaInductiveConstructor{inj₂}\AgdaSpace{}%
\AgdaBound{y≥a}\AgdaSpace{}%
\AgdaSymbol{|}\AgdaSpace{}%
\AgdaInductiveConstructor{inj₁}\AgdaSpace{}%
\AgdaBound{x<a}\AgdaSpace{}%
\AgdaSymbol{|}\AgdaSpace{}%
\AgdaInductiveConstructor{inj₂}\AgdaSpace{}%
\AgdaBound{y≥a'}\AgdaSpace{}%
\AgdaSymbol{|}\AgdaSpace{}%
\AgdaInductiveConstructor{inj₂}\AgdaSpace{}%
\AgdaBound{x≥y}\AgdaSpace{}%
\AgdaSymbol{=}\AgdaSpace{}%
\AgdaFunction{⊥-elim}\AgdaSpace{}%
\AgdaSymbol{(}\AgdaFunction{Lm[<→¬≥]}\AgdaSpace{}%
\AgdaBound{x<a}\AgdaSpace{}%
\AgdaSymbol{(}\AgdaFunction{≤-trans}\AgdaSpace{}%
\AgdaBound{y≥a}\AgdaSpace{}%
\AgdaBound{x≥y}\AgdaSymbol{))}\<%
\\
\>[0]\AgdaFunction{insert-swap}\AgdaSpace{}%
\AgdaBound{x}\AgdaSpace{}%
\AgdaBound{y}\AgdaSpace{}%
\AgdaSymbol{(}\AgdaOperator{\AgdaInductiveConstructor{ω\textasciicircum{}}}\AgdaSpace{}%
\AgdaBound{a}\AgdaSpace{}%
\AgdaOperator{\AgdaInductiveConstructor{+}}\AgdaSpace{}%
\AgdaBound{b}\AgdaSpace{}%
\AgdaOperator{\AgdaInductiveConstructor{[}}\AgdaSpace{}%
\AgdaSymbol{\AgdaUnderscore{}}\AgdaSpace{}%
\AgdaOperator{\AgdaInductiveConstructor{]}}\AgdaSymbol{)}\AgdaSpace{}%
\AgdaSymbol{|}\AgdaSpace{}%
\AgdaInductiveConstructor{inj₂}\AgdaSpace{}%
\AgdaBound{y≥a}\AgdaSpace{}%
\AgdaSymbol{|}\AgdaSpace{}%
\AgdaInductiveConstructor{inj₂}\AgdaSpace{}%
\AgdaBound{x≥a}\AgdaSpace{}%
\AgdaKeyword{with}\AgdaSpace{}%
\AgdaFunction{<-tri}\AgdaSpace{}%
\AgdaBound{x}\AgdaSpace{}%
\AgdaBound{y}\<%
\\
\>[0]\AgdaFunction{insert-swap}\AgdaSpace{}%
\AgdaBound{x}\AgdaSpace{}%
\AgdaBound{y}\AgdaSpace{}%
\AgdaSymbol{(}\AgdaOperator{\AgdaInductiveConstructor{ω\textasciicircum{}}}\AgdaSpace{}%
\AgdaBound{a}\AgdaSpace{}%
\AgdaOperator{\AgdaInductiveConstructor{+}}\AgdaSpace{}%
\AgdaBound{b}\AgdaSpace{}%
\AgdaOperator{\AgdaInductiveConstructor{[}}\AgdaSpace{}%
\AgdaSymbol{\AgdaUnderscore{}}\AgdaSpace{}%
\AgdaOperator{\AgdaInductiveConstructor{]}}\AgdaSymbol{)}\AgdaSpace{}%
\AgdaSymbol{|}\AgdaSpace{}%
\AgdaInductiveConstructor{inj₂}\AgdaSpace{}%
\AgdaBound{y≥a}\AgdaSpace{}%
\AgdaSymbol{|}\AgdaSpace{}%
\AgdaInductiveConstructor{inj₂}\AgdaSpace{}%
\AgdaBound{x≥a}\AgdaSpace{}%
\AgdaSymbol{|}\AgdaSpace{}%
\AgdaInductiveConstructor{inj₁}\AgdaSpace{}%
\AgdaBound{x<y}\AgdaSpace{}%
\AgdaKeyword{with}\AgdaSpace{}%
\AgdaFunction{<-tri}\AgdaSpace{}%
\AgdaBound{x}\AgdaSpace{}%
\AgdaBound{a}\<%
\\
\>[0]\AgdaFunction{insert-swap}\AgdaSpace{}%
\AgdaBound{x}\AgdaSpace{}%
\AgdaBound{y}\AgdaSpace{}%
\AgdaSymbol{(}\AgdaOperator{\AgdaInductiveConstructor{ω\textasciicircum{}}}\AgdaSpace{}%
\AgdaBound{a}\AgdaSpace{}%
\AgdaOperator{\AgdaInductiveConstructor{+}}\AgdaSpace{}%
\AgdaBound{b}\AgdaSpace{}%
\AgdaOperator{\AgdaInductiveConstructor{[}}\AgdaSpace{}%
\AgdaSymbol{\AgdaUnderscore{}}\AgdaSpace{}%
\AgdaOperator{\AgdaInductiveConstructor{]}}\AgdaSymbol{)}\AgdaSpace{}%
\AgdaSymbol{|}\AgdaSpace{}%
\AgdaInductiveConstructor{inj₂}\AgdaSpace{}%
\AgdaBound{y≥a}\AgdaSpace{}%
\AgdaSymbol{|}\AgdaSpace{}%
\AgdaInductiveConstructor{inj₂}\AgdaSpace{}%
\AgdaBound{x≥a}\AgdaSpace{}%
\AgdaSymbol{|}\AgdaSpace{}%
\AgdaInductiveConstructor{inj₁}\AgdaSpace{}%
\AgdaBound{x<y}\AgdaSpace{}%
\AgdaSymbol{|}\AgdaSpace{}%
\AgdaInductiveConstructor{inj₁}\AgdaSpace{}%
\AgdaBound{x<a}%
\>[78]\AgdaSymbol{=}\AgdaSpace{}%
\AgdaFunction{⊥-elim}\AgdaSpace{}%
\AgdaSymbol{(}\AgdaFunction{Lm[<→¬≥]}\AgdaSpace{}%
\AgdaBound{x<a}\AgdaSpace{}%
\AgdaBound{x≥a}\AgdaSymbol{)}\<%
\\
\>[0]\AgdaFunction{insert-swap}\AgdaSpace{}%
\AgdaBound{x}\AgdaSpace{}%
\AgdaBound{y}\AgdaSpace{}%
\AgdaSymbol{(}\AgdaOperator{\AgdaInductiveConstructor{ω\textasciicircum{}}}\AgdaSpace{}%
\AgdaBound{a}\AgdaSpace{}%
\AgdaOperator{\AgdaInductiveConstructor{+}}\AgdaSpace{}%
\AgdaBound{b}\AgdaSpace{}%
\AgdaOperator{\AgdaInductiveConstructor{[}}\AgdaSpace{}%
\AgdaSymbol{\AgdaUnderscore{}}\AgdaSpace{}%
\AgdaOperator{\AgdaInductiveConstructor{]}}\AgdaSymbol{)}\AgdaSpace{}%
\AgdaSymbol{|}\AgdaSpace{}%
\AgdaInductiveConstructor{inj₂}\AgdaSpace{}%
\AgdaBound{y≥a}\AgdaSpace{}%
\AgdaSymbol{|}\AgdaSpace{}%
\AgdaInductiveConstructor{inj₂}\AgdaSpace{}%
\AgdaBound{x≥a}\AgdaSpace{}%
\AgdaSymbol{|}\AgdaSpace{}%
\AgdaInductiveConstructor{inj₁}\AgdaSpace{}%
\AgdaBound{x<y}\AgdaSpace{}%
\AgdaSymbol{|}\AgdaSpace{}%
\AgdaInductiveConstructor{inj₂}\AgdaSpace{}%
\AgdaBound{x≥a'}\AgdaSpace{}%
\AgdaKeyword{with}\AgdaSpace{}%
\AgdaFunction{<-tri}\AgdaSpace{}%
\AgdaBound{y}\AgdaSpace{}%
\AgdaBound{x}\<%
\\
\>[0]\AgdaFunction{insert-swap}\AgdaSpace{}%
\AgdaBound{x}\AgdaSpace{}%
\AgdaBound{y}\AgdaSpace{}%
\AgdaSymbol{(}\AgdaOperator{\AgdaInductiveConstructor{ω\textasciicircum{}}}\AgdaSpace{}%
\AgdaBound{a}\AgdaSpace{}%
\AgdaOperator{\AgdaInductiveConstructor{+}}\AgdaSpace{}%
\AgdaBound{b}\AgdaSpace{}%
\AgdaOperator{\AgdaInductiveConstructor{[}}\AgdaSpace{}%
\AgdaSymbol{\AgdaUnderscore{}}\AgdaSpace{}%
\AgdaOperator{\AgdaInductiveConstructor{]}}\AgdaSymbol{)}\AgdaSpace{}%
\AgdaSymbol{|}\AgdaSpace{}%
\AgdaInductiveConstructor{inj₂}\AgdaSpace{}%
\AgdaBound{y≥a}\AgdaSpace{}%
\AgdaSymbol{|}\AgdaSpace{}%
\AgdaInductiveConstructor{inj₂}\AgdaSpace{}%
\AgdaBound{x≥a}\AgdaSpace{}%
\AgdaSymbol{|}\AgdaSpace{}%
\AgdaInductiveConstructor{inj₁}\AgdaSpace{}%
\AgdaBound{x<y}\AgdaSpace{}%
\AgdaSymbol{|}\AgdaSpace{}%
\AgdaInductiveConstructor{inj₂}\AgdaSpace{}%
\AgdaBound{x≥a'}\AgdaSpace{}%
\AgdaSymbol{|}\AgdaSpace{}%
\AgdaInductiveConstructor{inj₁}\AgdaSpace{}%
\AgdaBound{y<x}\AgdaSpace{}%
\AgdaSymbol{=}\AgdaSpace{}%
\AgdaFunction{⊥-elim}\AgdaSpace{}%
\AgdaSymbol{(}\AgdaFunction{Lm[<→¬≥]}\AgdaSpace{}%
\AgdaBound{x<y}\AgdaSpace{}%
\AgdaSymbol{(}\AgdaInductiveConstructor{inj₁}\AgdaSpace{}%
\AgdaBound{y<x}\AgdaSymbol{))}\<%
\\
\>[0]\AgdaFunction{insert-swap}\AgdaSpace{}%
\AgdaBound{x}\AgdaSpace{}%
\AgdaBound{y}\AgdaSpace{}%
\AgdaSymbol{(}\AgdaOperator{\AgdaInductiveConstructor{ω\textasciicircum{}}}\AgdaSpace{}%
\AgdaBound{a}\AgdaSpace{}%
\AgdaOperator{\AgdaInductiveConstructor{+}}\AgdaSpace{}%
\AgdaBound{b}\AgdaSpace{}%
\AgdaOperator{\AgdaInductiveConstructor{[}}\AgdaSpace{}%
\AgdaSymbol{\AgdaUnderscore{}}\AgdaSpace{}%
\AgdaOperator{\AgdaInductiveConstructor{]}}\AgdaSymbol{)}\AgdaSpace{}%
\AgdaSymbol{|}\AgdaSpace{}%
\AgdaInductiveConstructor{inj₂}\AgdaSpace{}%
\AgdaBound{y≥a}\AgdaSpace{}%
\AgdaSymbol{|}\AgdaSpace{}%
\AgdaInductiveConstructor{inj₂}\AgdaSpace{}%
\AgdaBound{x≥a}\AgdaSpace{}%
\AgdaSymbol{|}\AgdaSpace{}%
\AgdaInductiveConstructor{inj₁}\AgdaSpace{}%
\AgdaBound{x<y}\AgdaSpace{}%
\AgdaSymbol{|}\AgdaSpace{}%
\AgdaInductiveConstructor{inj₂}\AgdaSpace{}%
\AgdaBound{x≥a'}\AgdaSpace{}%
\AgdaSymbol{|}\AgdaSpace{}%
\AgdaInductiveConstructor{inj₂}\AgdaSpace{}%
\AgdaBound{y≥x}\AgdaSpace{}%
\AgdaSymbol{=}\AgdaSpace{}%
\AgdaFunction{MutualOrd⁼}\AgdaSpace{}%
\AgdaFunction{refl}\AgdaSpace{}%
\AgdaSymbol{(}\AgdaFunction{MutualOrd⁼}\AgdaSpace{}%
\AgdaFunction{refl}\AgdaSpace{}%
\AgdaFunction{refl}\AgdaSymbol{)}\<%
\\
\>[0]\AgdaFunction{insert-swap}\AgdaSpace{}%
\AgdaBound{x}\AgdaSpace{}%
\AgdaBound{y}\AgdaSpace{}%
\AgdaSymbol{(}\AgdaOperator{\AgdaInductiveConstructor{ω\textasciicircum{}}}\AgdaSpace{}%
\AgdaBound{a}\AgdaSpace{}%
\AgdaOperator{\AgdaInductiveConstructor{+}}\AgdaSpace{}%
\AgdaBound{b}\AgdaSpace{}%
\AgdaOperator{\AgdaInductiveConstructor{[}}\AgdaSpace{}%
\AgdaSymbol{\AgdaUnderscore{}}\AgdaSpace{}%
\AgdaOperator{\AgdaInductiveConstructor{]}}\AgdaSymbol{)}\AgdaSpace{}%
\AgdaSymbol{|}\AgdaSpace{}%
\AgdaInductiveConstructor{inj₂}\AgdaSpace{}%
\AgdaBound{y≥a}\AgdaSpace{}%
\AgdaSymbol{|}\AgdaSpace{}%
\AgdaInductiveConstructor{inj₂}\AgdaSpace{}%
\AgdaBound{x≥a}\AgdaSpace{}%
\AgdaSymbol{|}\AgdaSpace{}%
\AgdaInductiveConstructor{inj₂}\AgdaSpace{}%
\AgdaBound{x≥y}\AgdaSpace{}%
\AgdaKeyword{with}\AgdaSpace{}%
\AgdaFunction{<-tri}\AgdaSpace{}%
\AgdaBound{y}\AgdaSpace{}%
\AgdaBound{x}\<%
\\
\>[0]\AgdaFunction{insert-swap}\AgdaSpace{}%
\AgdaBound{x}\AgdaSpace{}%
\AgdaBound{y}\AgdaSpace{}%
\AgdaSymbol{(}\AgdaOperator{\AgdaInductiveConstructor{ω\textasciicircum{}}}\AgdaSpace{}%
\AgdaBound{a}\AgdaSpace{}%
\AgdaOperator{\AgdaInductiveConstructor{+}}\AgdaSpace{}%
\AgdaBound{b}\AgdaSpace{}%
\AgdaOperator{\AgdaInductiveConstructor{[}}\AgdaSpace{}%
\AgdaSymbol{\AgdaUnderscore{}}\AgdaSpace{}%
\AgdaOperator{\AgdaInductiveConstructor{]}}\AgdaSymbol{)}\AgdaSpace{}%
\AgdaSymbol{|}\AgdaSpace{}%
\AgdaInductiveConstructor{inj₂}\AgdaSpace{}%
\AgdaBound{y≥a}\AgdaSpace{}%
\AgdaSymbol{|}\AgdaSpace{}%
\AgdaInductiveConstructor{inj₂}\AgdaSpace{}%
\AgdaBound{x≥a}\AgdaSpace{}%
\AgdaSymbol{|}\AgdaSpace{}%
\AgdaInductiveConstructor{inj₂}\AgdaSpace{}%
\AgdaBound{x≥y}\AgdaSpace{}%
\AgdaSymbol{|}\AgdaSpace{}%
\AgdaInductiveConstructor{inj₁}\AgdaSpace{}%
\AgdaBound{y<x}\AgdaSpace{}%
\AgdaKeyword{with}\AgdaSpace{}%
\AgdaFunction{<-tri}\AgdaSpace{}%
\AgdaBound{y}\AgdaSpace{}%
\AgdaBound{a}\<%
\\
\>[0]\AgdaFunction{insert-swap}\AgdaSpace{}%
\AgdaBound{x}\AgdaSpace{}%
\AgdaBound{y}\AgdaSpace{}%
\AgdaSymbol{(}\AgdaOperator{\AgdaInductiveConstructor{ω\textasciicircum{}}}\AgdaSpace{}%
\AgdaBound{a}\AgdaSpace{}%
\AgdaOperator{\AgdaInductiveConstructor{+}}\AgdaSpace{}%
\AgdaBound{b}\AgdaSpace{}%
\AgdaOperator{\AgdaInductiveConstructor{[}}\AgdaSpace{}%
\AgdaSymbol{\AgdaUnderscore{}}\AgdaSpace{}%
\AgdaOperator{\AgdaInductiveConstructor{]}}\AgdaSymbol{)}\AgdaSpace{}%
\AgdaSymbol{|}\AgdaSpace{}%
\AgdaInductiveConstructor{inj₂}\AgdaSpace{}%
\AgdaBound{y≥a}\AgdaSpace{}%
\AgdaSymbol{|}\AgdaSpace{}%
\AgdaInductiveConstructor{inj₂}\AgdaSpace{}%
\AgdaBound{x≥a}\AgdaSpace{}%
\AgdaSymbol{|}\AgdaSpace{}%
\AgdaInductiveConstructor{inj₂}\AgdaSpace{}%
\AgdaBound{x≥y}\AgdaSpace{}%
\AgdaSymbol{|}\AgdaSpace{}%
\AgdaInductiveConstructor{inj₁}\AgdaSpace{}%
\AgdaBound{y<x}\AgdaSpace{}%
\AgdaSymbol{|}\AgdaSpace{}%
\AgdaInductiveConstructor{inj₁}\AgdaSpace{}%
\AgdaBound{y<a}%
\>[89]\AgdaSymbol{=}\AgdaSpace{}%
\AgdaFunction{⊥-elim}\AgdaSpace{}%
\AgdaSymbol{(}\AgdaFunction{Lm[<→¬≥]}\AgdaSpace{}%
\AgdaBound{y<a}\AgdaSpace{}%
\AgdaBound{y≥a}\AgdaSymbol{)}\<%
\\
\>[0]\AgdaFunction{insert-swap}\AgdaSpace{}%
\AgdaBound{x}\AgdaSpace{}%
\AgdaBound{y}\AgdaSpace{}%
\AgdaSymbol{(}\AgdaOperator{\AgdaInductiveConstructor{ω\textasciicircum{}}}\AgdaSpace{}%
\AgdaBound{a}\AgdaSpace{}%
\AgdaOperator{\AgdaInductiveConstructor{+}}\AgdaSpace{}%
\AgdaBound{b}\AgdaSpace{}%
\AgdaOperator{\AgdaInductiveConstructor{[}}\AgdaSpace{}%
\AgdaSymbol{\AgdaUnderscore{}}\AgdaSpace{}%
\AgdaOperator{\AgdaInductiveConstructor{]}}\AgdaSymbol{)}\AgdaSpace{}%
\AgdaSymbol{|}\AgdaSpace{}%
\AgdaInductiveConstructor{inj₂}\AgdaSpace{}%
\AgdaBound{y≥a}\AgdaSpace{}%
\AgdaSymbol{|}\AgdaSpace{}%
\AgdaInductiveConstructor{inj₂}\AgdaSpace{}%
\AgdaBound{x≥a}\AgdaSpace{}%
\AgdaSymbol{|}\AgdaSpace{}%
\AgdaInductiveConstructor{inj₂}\AgdaSpace{}%
\AgdaBound{x≥y}\AgdaSpace{}%
\AgdaSymbol{|}\AgdaSpace{}%
\AgdaInductiveConstructor{inj₁}\AgdaSpace{}%
\AgdaBound{y<x}\AgdaSpace{}%
\AgdaSymbol{|}\AgdaSpace{}%
\AgdaInductiveConstructor{inj₂}\AgdaSpace{}%
\AgdaBound{y≥a'}\AgdaSpace{}%
\AgdaSymbol{=}\AgdaSpace{}%
\AgdaFunction{MutualOrd⁼}\AgdaSpace{}%
\AgdaFunction{refl}\AgdaSpace{}%
\AgdaSymbol{(}\AgdaFunction{MutualOrd⁼}\AgdaSpace{}%
\AgdaFunction{refl}\AgdaSpace{}%
\AgdaFunction{refl}\AgdaSymbol{)}\<%
\\
\>[0]\AgdaFunction{insert-swap}\AgdaSpace{}%
\AgdaBound{x}\AgdaSpace{}%
\AgdaBound{y}\AgdaSpace{}%
\AgdaSymbol{(}\AgdaOperator{\AgdaInductiveConstructor{ω\textasciicircum{}}}\AgdaSpace{}%
\AgdaBound{a}\AgdaSpace{}%
\AgdaOperator{\AgdaInductiveConstructor{+}}\AgdaSpace{}%
\AgdaBound{b}\AgdaSpace{}%
\AgdaOperator{\AgdaInductiveConstructor{[}}\AgdaSpace{}%
\AgdaSymbol{\AgdaUnderscore{}}\AgdaSpace{}%
\AgdaOperator{\AgdaInductiveConstructor{]}}\AgdaSymbol{)}\AgdaSpace{}%
\AgdaSymbol{|}\AgdaSpace{}%
\AgdaInductiveConstructor{inj₂}\AgdaSpace{}%
\AgdaBound{y≥a}\AgdaSpace{}%
\AgdaSymbol{|}\AgdaSpace{}%
\AgdaInductiveConstructor{inj₂}\AgdaSpace{}%
\AgdaBound{x≥a}\AgdaSpace{}%
\AgdaSymbol{|}\AgdaSpace{}%
\AgdaInductiveConstructor{inj₂}\AgdaSpace{}%
\AgdaBound{x≥y}\AgdaSpace{}%
\AgdaSymbol{|}\AgdaSpace{}%
\AgdaInductiveConstructor{inj₂}\AgdaSpace{}%
\AgdaBound{y≥x}\AgdaSpace{}%
\AgdaSymbol{=}\AgdaSpace{}%
\AgdaFunction{MutualOrd⁼}\AgdaSpace{}%
\AgdaSymbol{(}\AgdaFunction{≤≥→≡}\AgdaSpace{}%
\AgdaBound{y≥x}\AgdaSpace{}%
\AgdaBound{x≥y}\AgdaSymbol{)}\AgdaSpace{}%
\AgdaSymbol{(}\AgdaFunction{MutualOrd⁼}\AgdaSpace{}%
\AgdaSymbol{(}\AgdaFunction{≤≥→≡}\AgdaSpace{}%
\AgdaBound{x≥y}\AgdaSpace{}%
\AgdaBound{y≥x}\AgdaSymbol{)}\AgdaSpace{}%
\AgdaFunction{refl}\AgdaSymbol{)}\<%
\\
\>[0]\<%
\end{code}

\newcommand{\HtoM}{
\begin{code}%
\>[0]\AgdaFunction{H2M}\AgdaSpace{}%
\AgdaSymbol{:}\AgdaSpace{}%
\AgdaDatatype{HITOrd}\AgdaSpace{}%
\AgdaSymbol{→}\AgdaSpace{}%
\AgdaDatatype{MutualOrd}\<%
\\
\>[0]\AgdaFunction{H2M}\AgdaSpace{}%
\AgdaSymbol{=}\AgdaSpace{}%
\AgdaFunction{rec}\AgdaSpace{}%
\AgdaFunction{MutualOrdIsSet}\AgdaSpace{}%
\AgdaInductiveConstructor{𝟎}\AgdaSpace{}%
\AgdaFunction{insert}\AgdaSpace{}%
\AgdaFunction{insert-swap}\<%
\end{code}}

\newcommand{\insertPlus}{
\begin{code}%
\>[0]\AgdaFunction{insert-+}%
\>[1407I]\AgdaSymbol{:}\AgdaSpace{}%
\AgdaSymbol{(}\AgdaBound{a}\AgdaSpace{}%
\AgdaBound{b}\AgdaSpace{}%
\AgdaSymbol{:}\AgdaSpace{}%
\AgdaDatatype{MutualOrd}\AgdaSymbol{)}\AgdaSpace{}%
\AgdaSymbol{(}\AgdaBound{r}\AgdaSpace{}%
\AgdaSymbol{:}\AgdaSpace{}%
\AgdaBound{a}\AgdaSpace{}%
\AgdaOperator{\AgdaFunction{M.≥}}\AgdaSpace{}%
\AgdaFunction{M.fst}\AgdaSpace{}%
\AgdaBound{b}\AgdaSymbol{)}\<%
\\
\>[.][@{}l@{}]\<[1407I]%
\>[9]\AgdaSymbol{→}\AgdaSpace{}%
\AgdaFunction{insert}\AgdaSpace{}%
\AgdaBound{a}\AgdaSpace{}%
\AgdaBound{b}\AgdaSpace{}%
\AgdaOperator{\AgdaFunction{≡}}\AgdaSpace{}%
\AgdaOperator{\AgdaInductiveConstructor{ω\textasciicircum{}}}\AgdaSpace{}%
\AgdaBound{a}\AgdaSpace{}%
\AgdaOperator{\AgdaInductiveConstructor{+}}\AgdaSpace{}%
\AgdaBound{b}\AgdaSpace{}%
\AgdaOperator{\AgdaInductiveConstructor{[}}\AgdaSpace{}%
\AgdaBound{r}\AgdaSpace{}%
\AgdaOperator{\AgdaInductiveConstructor{]}}\<%
\end{code}}

\begin{code}[hide]%
\>[0]\<%
\\
\>[0]\AgdaFunction{insert-+}\AgdaSpace{}%
\AgdaBound{a}\AgdaSpace{}%
\AgdaInductiveConstructor{𝟎}\AgdaSpace{}%
\AgdaSymbol{\AgdaUnderscore{}}\AgdaSpace{}%
\AgdaSymbol{=}\AgdaSpace{}%
\AgdaFunction{MutualOrd⁼}\AgdaSpace{}%
\AgdaFunction{refl}\AgdaSpace{}%
\AgdaFunction{refl}\<%
\\
\>[0]\AgdaFunction{insert-+}\AgdaSpace{}%
\AgdaBound{a}\AgdaSpace{}%
\AgdaSymbol{(}\AgdaOperator{\AgdaInductiveConstructor{ω\textasciicircum{}}}\AgdaSpace{}%
\AgdaBound{b}\AgdaSpace{}%
\AgdaOperator{\AgdaInductiveConstructor{+}}\AgdaSpace{}%
\AgdaBound{c}\AgdaSpace{}%
\AgdaOperator{\AgdaInductiveConstructor{[}}\AgdaSpace{}%
\AgdaSymbol{\AgdaUnderscore{}}\AgdaSpace{}%
\AgdaOperator{\AgdaInductiveConstructor{]}}\AgdaSymbol{)}\AgdaSpace{}%
\AgdaBound{a≥b}\AgdaSpace{}%
\AgdaKeyword{with}\AgdaSpace{}%
\AgdaFunction{<-tri}\AgdaSpace{}%
\AgdaBound{a}\AgdaSpace{}%
\AgdaBound{b}\<%
\\
\>[0]\AgdaFunction{insert-+}\AgdaSpace{}%
\AgdaBound{a}\AgdaSpace{}%
\AgdaSymbol{(}\AgdaOperator{\AgdaInductiveConstructor{ω\textasciicircum{}}}\AgdaSpace{}%
\AgdaBound{b}\AgdaSpace{}%
\AgdaOperator{\AgdaInductiveConstructor{+}}\AgdaSpace{}%
\AgdaBound{c}\AgdaSpace{}%
\AgdaOperator{\AgdaInductiveConstructor{[}}\AgdaSpace{}%
\AgdaSymbol{\AgdaUnderscore{}}\AgdaSpace{}%
\AgdaOperator{\AgdaInductiveConstructor{]}}\AgdaSymbol{)}\AgdaSpace{}%
\AgdaBound{a≥b}\AgdaSpace{}%
\AgdaSymbol{|}\AgdaSpace{}%
\AgdaInductiveConstructor{inj₁}\AgdaSpace{}%
\AgdaBound{a<b}%
\>[44]\AgdaSymbol{=}\AgdaSpace{}%
\AgdaFunction{⊥-elim}\AgdaSpace{}%
\AgdaSymbol{(}\AgdaFunction{Lm[<→¬≥]}\AgdaSpace{}%
\AgdaBound{a<b}\AgdaSpace{}%
\AgdaBound{a≥b}\AgdaSymbol{)}\<%
\\
\>[0]\AgdaFunction{insert-+}\AgdaSpace{}%
\AgdaBound{a}\AgdaSpace{}%
\AgdaSymbol{(}\AgdaOperator{\AgdaInductiveConstructor{ω\textasciicircum{}}}\AgdaSpace{}%
\AgdaBound{b}\AgdaSpace{}%
\AgdaOperator{\AgdaInductiveConstructor{+}}\AgdaSpace{}%
\AgdaBound{c}\AgdaSpace{}%
\AgdaOperator{\AgdaInductiveConstructor{[}}\AgdaSpace{}%
\AgdaSymbol{\AgdaUnderscore{}}\AgdaSpace{}%
\AgdaOperator{\AgdaInductiveConstructor{]}}\AgdaSymbol{)}\AgdaSpace{}%
\AgdaBound{a≥b}\AgdaSpace{}%
\AgdaSymbol{|}\AgdaSpace{}%
\AgdaInductiveConstructor{inj₂}\AgdaSpace{}%
\AgdaBound{a≥b'}\AgdaSpace{}%
\AgdaSymbol{=}\AgdaSpace{}%
\AgdaFunction{MutualOrd⁼}\AgdaSpace{}%
\AgdaFunction{refl}\AgdaSpace{}%
\AgdaSymbol{(}\AgdaFunction{MutualOrd⁼}\AgdaSpace{}%
\AgdaFunction{refl}\AgdaSpace{}%
\AgdaFunction{refl}\AgdaSymbol{)}\<%
\\
\>[0]\<%
\end{code}

\newcommand{\MtoHtoM}{
\begin{code}%
\>[0]\AgdaFunction{M2H2M=id}\AgdaSpace{}%
\AgdaSymbol{:}\AgdaSpace{}%
\AgdaSymbol{(}\AgdaBound{a}\AgdaSpace{}%
\AgdaSymbol{:}\AgdaSpace{}%
\AgdaDatatype{MutualOrd}\AgdaSymbol{)}\AgdaSpace{}%
\AgdaSymbol{→}\AgdaSpace{}%
\AgdaFunction{H2M}\AgdaSpace{}%
\AgdaSymbol{(}\AgdaFunction{M2H}\AgdaSpace{}%
\AgdaBound{a}\AgdaSymbol{)}\AgdaSpace{}%
\AgdaOperator{\AgdaFunction{≡}}\AgdaSpace{}%
\AgdaBound{a}\<%
\\
\>[0]\AgdaFunction{M2H2M=id}\AgdaSpace{}%
\AgdaInductiveConstructor{𝟎}\AgdaSpace{}%
\AgdaSymbol{=}\AgdaSpace{}%
\AgdaFunction{refl}\<%
\\
\>[0]\AgdaFunction{M2H2M=id}\AgdaSpace{}%
\AgdaSymbol{(}\AgdaOperator{\AgdaInductiveConstructor{ω\textasciicircum{}}}\AgdaSpace{}%
\AgdaBound{a}\AgdaSpace{}%
\AgdaOperator{\AgdaInductiveConstructor{+}}\AgdaSpace{}%
\AgdaBound{b}\AgdaSpace{}%
\AgdaOperator{\AgdaInductiveConstructor{[}}\AgdaSpace{}%
\AgdaBound{r}\AgdaSpace{}%
\AgdaOperator{\AgdaInductiveConstructor{]}}\AgdaSymbol{)}\AgdaSpace{}%
\AgdaSymbol{=}\AgdaSpace{}%
\AgdaOperator{\AgdaFunction{begin}}\<%
\\
\>[0][@{}l@{\AgdaIndent{0}}]%
\>[4]\AgdaFunction{H2M}\AgdaSpace{}%
\AgdaSymbol{(}\AgdaFunction{M2H}\AgdaSpace{}%
\AgdaSymbol{(}\AgdaOperator{\AgdaInductiveConstructor{ω\textasciicircum{}}}\AgdaSpace{}%
\AgdaBound{a}\AgdaSpace{}%
\AgdaOperator{\AgdaInductiveConstructor{+}}\AgdaSpace{}%
\AgdaBound{b}\AgdaSpace{}%
\AgdaOperator{\AgdaInductiveConstructor{[}}\AgdaSpace{}%
\AgdaBound{r}\AgdaSpace{}%
\AgdaOperator{\AgdaInductiveConstructor{]}}\AgdaSymbol{))}\<%
\\
\>[0][@{}l@{\AgdaIndent{0}}]%
\>[2]\AgdaOperator{\AgdaFunction{≡⟨}}\AgdaSpace{}%
\AgdaFunction{refl}\AgdaSpace{}%
\AgdaOperator{\AgdaFunction{⟩}}\<%
\\
\>[2][@{}l@{\AgdaIndent{0}}]%
\>[4]\AgdaFunction{H2M}\AgdaSpace{}%
\AgdaSymbol{(}\AgdaOperator{\AgdaInductiveConstructor{ω\textasciicircum{}}}\AgdaSpace{}%
\AgdaSymbol{(}\AgdaFunction{M2H}\AgdaSpace{}%
\AgdaBound{a}\AgdaSymbol{)}\AgdaSpace{}%
\AgdaOperator{\AgdaInductiveConstructor{⊕}}\AgdaSpace{}%
\AgdaSymbol{(}\AgdaFunction{M2H}\AgdaSpace{}%
\AgdaBound{b}\AgdaSymbol{))}\<%
\\
\>[2]\AgdaOperator{\AgdaFunction{≡⟨}}\AgdaSpace{}%
\AgdaFunction{refl}\AgdaSpace{}%
\AgdaOperator{\AgdaFunction{⟩}}\<%
\\
\>[2][@{}l@{\AgdaIndent{0}}]%
\>[4]\AgdaFunction{insert}\AgdaSpace{}%
\AgdaSymbol{(}\AgdaFunction{H2M}\AgdaSpace{}%
\AgdaSymbol{(}\AgdaFunction{M2H}\AgdaSpace{}%
\AgdaBound{a}\AgdaSymbol{))}\AgdaSpace{}%
\AgdaSymbol{(}\AgdaFunction{H2M}\AgdaSpace{}%
\AgdaSymbol{(}\AgdaFunction{M2H}\AgdaSpace{}%
\AgdaBound{b}\AgdaSymbol{))}\<%
\\
\>[2]\AgdaOperator{\AgdaFunction{≡⟨}}\AgdaSpace{}%
\AgdaFunction{cong₂}\AgdaSpace{}%
\AgdaFunction{insert}\AgdaSpace{}%
\AgdaSymbol{(}\AgdaFunction{M2H2M=id}\AgdaSpace{}%
\AgdaBound{a}\AgdaSymbol{)}\AgdaSpace{}%
\AgdaSymbol{(}\AgdaFunction{M2H2M=id}\AgdaSpace{}%
\AgdaBound{b}\AgdaSymbol{)}\AgdaSpace{}%
\AgdaOperator{\AgdaFunction{⟩}}\<%
\\
\>[2][@{}l@{\AgdaIndent{0}}]%
\>[4]\AgdaFunction{insert}\AgdaSpace{}%
\AgdaBound{a}\AgdaSpace{}%
\AgdaBound{b}\<%
\\
\>[2]\AgdaOperator{\AgdaFunction{≡⟨}}\AgdaSpace{}%
\AgdaFunction{insert-+}\AgdaSpace{}%
\AgdaBound{a}\AgdaSpace{}%
\AgdaBound{b}\AgdaSpace{}%
\AgdaBound{r}\AgdaSpace{}%
\AgdaOperator{\AgdaFunction{⟩}}\<%
\\
\>[2][@{}l@{\AgdaIndent{0}}]%
\>[4]\AgdaOperator{\AgdaInductiveConstructor{ω\textasciicircum{}}}\AgdaSpace{}%
\AgdaBound{a}\AgdaSpace{}%
\AgdaOperator{\AgdaInductiveConstructor{+}}\AgdaSpace{}%
\AgdaBound{b}\AgdaSpace{}%
\AgdaOperator{\AgdaInductiveConstructor{[}}\AgdaSpace{}%
\AgdaBound{r}\AgdaSpace{}%
\AgdaOperator{\AgdaInductiveConstructor{]}}%
\>[20]\AgdaOperator{\AgdaFunction{∎}}\<%
\end{code}}

\newcommand{\insertHash}{
\begin{code}%
\>[0]\AgdaFunction{insert-⊕}%
\>[1549I]\AgdaSymbol{:}\AgdaSpace{}%
\AgdaSymbol{(}\AgdaBound{a}\AgdaSpace{}%
\AgdaBound{b}\AgdaSpace{}%
\AgdaSymbol{:}\AgdaSpace{}%
\AgdaDatatype{MutualOrd}\AgdaSymbol{)}\<%
\\
\>[.][@{}l@{}]\<[1549I]%
\>[9]\AgdaSymbol{→}\AgdaSpace{}%
\AgdaFunction{M2H}\AgdaSpace{}%
\AgdaSymbol{(}\AgdaFunction{insert}\AgdaSpace{}%
\AgdaBound{a}\AgdaSpace{}%
\AgdaBound{b}\AgdaSymbol{)}\AgdaSpace{}%
\AgdaOperator{\AgdaFunction{≡}}\AgdaSpace{}%
\AgdaOperator{\AgdaInductiveConstructor{ω\textasciicircum{}}}\AgdaSpace{}%
\AgdaSymbol{(}\AgdaFunction{M2H}\AgdaSpace{}%
\AgdaBound{a}\AgdaSymbol{)}\AgdaSpace{}%
\AgdaOperator{\AgdaInductiveConstructor{⊕}}\AgdaSpace{}%
\AgdaSymbol{(}\AgdaFunction{M2H}\AgdaSpace{}%
\AgdaBound{b}\AgdaSymbol{)}\<%
\end{code}}

\begin{code}[hide]%
\>[0]\<%
\\
\>[0]\AgdaFunction{insert-⊕}\AgdaSpace{}%
\AgdaBound{a}\AgdaSpace{}%
\AgdaInductiveConstructor{𝟎}\AgdaSpace{}%
\AgdaSymbol{=}\AgdaSpace{}%
\AgdaFunction{refl}\<%
\\
\>[0]\AgdaFunction{insert-⊕}\AgdaSpace{}%
\AgdaBound{a}\AgdaSpace{}%
\AgdaSymbol{(}\AgdaOperator{\AgdaInductiveConstructor{ω\textasciicircum{}}}\AgdaSpace{}%
\AgdaBound{b}\AgdaSpace{}%
\AgdaOperator{\AgdaInductiveConstructor{+}}\AgdaSpace{}%
\AgdaBound{c}\AgdaSpace{}%
\AgdaOperator{\AgdaInductiveConstructor{[}}\AgdaSpace{}%
\AgdaSymbol{\AgdaUnderscore{}}\AgdaSpace{}%
\AgdaOperator{\AgdaInductiveConstructor{]}}\AgdaSymbol{)}\AgdaSpace{}%
\AgdaKeyword{with}\AgdaSpace{}%
\AgdaFunction{<-tri}\AgdaSpace{}%
\AgdaBound{a}\AgdaSpace{}%
\AgdaBound{b}\<%
\\
\>[0]\AgdaFunction{insert-⊕}\AgdaSpace{}%
\AgdaBound{a}\AgdaSpace{}%
\AgdaSymbol{(}\AgdaOperator{\AgdaInductiveConstructor{ω\textasciicircum{}}}\AgdaSpace{}%
\AgdaBound{b}\AgdaSpace{}%
\AgdaOperator{\AgdaInductiveConstructor{+}}\AgdaSpace{}%
\AgdaBound{c}\AgdaSpace{}%
\AgdaOperator{\AgdaInductiveConstructor{[}}\AgdaSpace{}%
\AgdaSymbol{\AgdaUnderscore{}}\AgdaSpace{}%
\AgdaOperator{\AgdaInductiveConstructor{]}}\AgdaSymbol{)}\AgdaSpace{}%
\AgdaSymbol{|}\AgdaSpace{}%
\AgdaInductiveConstructor{inj₁}\AgdaSpace{}%
\AgdaBound{a<b}\AgdaSpace{}%
\AgdaSymbol{=}\AgdaSpace{}%
\AgdaFunction{cong}\AgdaSpace{}%
\AgdaSymbol{(}\AgdaOperator{\AgdaInductiveConstructor{ω\textasciicircum{}}}\AgdaSpace{}%
\AgdaFunction{M2H}\AgdaSpace{}%
\AgdaBound{b}\AgdaSpace{}%
\AgdaOperator{\AgdaInductiveConstructor{⊕\AgdaUnderscore{}}}\AgdaSymbol{)}\AgdaSpace{}%
\AgdaSymbol{(}\AgdaFunction{insert-⊕}\AgdaSpace{}%
\AgdaBound{a}\AgdaSpace{}%
\AgdaBound{c}\AgdaSymbol{)}\AgdaSpace{}%
\AgdaOperator{\AgdaFunction{∙}}\AgdaSpace{}%
\AgdaInductiveConstructor{swap}\AgdaSpace{}%
\AgdaSymbol{(}\AgdaFunction{M2H}\AgdaSpace{}%
\AgdaBound{b}\AgdaSymbol{)}\AgdaSpace{}%
\AgdaSymbol{(}\AgdaFunction{M2H}\AgdaSpace{}%
\AgdaBound{a}\AgdaSymbol{)}\AgdaSpace{}%
\AgdaSymbol{(}\AgdaFunction{M2H}\AgdaSpace{}%
\AgdaBound{c}\AgdaSymbol{)}\<%
\\
\>[0]\AgdaFunction{insert-⊕}\AgdaSpace{}%
\AgdaBound{a}\AgdaSpace{}%
\AgdaSymbol{(}\AgdaOperator{\AgdaInductiveConstructor{ω\textasciicircum{}}}\AgdaSpace{}%
\AgdaBound{b}\AgdaSpace{}%
\AgdaOperator{\AgdaInductiveConstructor{+}}\AgdaSpace{}%
\AgdaBound{c}\AgdaSpace{}%
\AgdaOperator{\AgdaInductiveConstructor{[}}\AgdaSpace{}%
\AgdaSymbol{\AgdaUnderscore{}}\AgdaSpace{}%
\AgdaOperator{\AgdaInductiveConstructor{]}}\AgdaSymbol{)}\AgdaSpace{}%
\AgdaSymbol{|}\AgdaSpace{}%
\AgdaInductiveConstructor{inj₂}\AgdaSpace{}%
\AgdaBound{a≥b}\AgdaSpace{}%
\AgdaSymbol{=}\AgdaSpace{}%
\AgdaFunction{refl}\<%
\\
\>[0]\<%
\end{code}

\newcommand{\HtoMtoH}{
\begin{code}%
\>[0]\AgdaFunction{H2M2H=id}\AgdaSpace{}%
\AgdaSymbol{:}\AgdaSpace{}%
\AgdaSymbol{(}\AgdaBound{a}\AgdaSpace{}%
\AgdaSymbol{:}\AgdaSpace{}%
\AgdaDatatype{HITOrd}\AgdaSymbol{)}\AgdaSpace{}%
\AgdaSymbol{→}\AgdaSpace{}%
\AgdaFunction{M2H}\AgdaSpace{}%
\AgdaSymbol{(}\AgdaFunction{H2M}\AgdaSpace{}%
\AgdaBound{a}\AgdaSymbol{)}\AgdaSpace{}%
\AgdaOperator{\AgdaFunction{≡}}\AgdaSpace{}%
\AgdaBound{a}\<%
\\
\>[0]\AgdaFunction{H2M2H=id}\AgdaSpace{}%
\AgdaSymbol{=}\AgdaSpace{}%
\AgdaFunction{indProp}\AgdaSpace{}%
\AgdaFunction{P}\AgdaSpace{}%
\AgdaInductiveConstructor{trunc}\AgdaSpace{}%
\AgdaFunction{base}\AgdaSpace{}%
\AgdaFunction{step}\<%
\\
\>[0][@{}l@{\AgdaIndent{0}}]%
\>[1]\AgdaKeyword{where}\<%
\\
\>[1][@{}l@{\AgdaIndent{0}}]%
\>[2]\AgdaFunction{P}\AgdaSpace{}%
\AgdaSymbol{:}\AgdaSpace{}%
\AgdaDatatype{HITOrd}\AgdaSpace{}%
\AgdaSymbol{→}\AgdaSpace{}%
\AgdaFunction{Type₀}\<%
\\
\>[2]\AgdaFunction{P}\AgdaSpace{}%
\AgdaBound{x}\AgdaSpace{}%
\AgdaSymbol{=}\AgdaSpace{}%
\AgdaFunction{M2H}\AgdaSpace{}%
\AgdaSymbol{(}\AgdaFunction{H2M}\AgdaSpace{}%
\AgdaBound{x}\AgdaSymbol{)}\AgdaSpace{}%
\AgdaOperator{\AgdaFunction{≡}}\AgdaSpace{}%
\AgdaBound{x}\<%
\\
\>[2]\AgdaFunction{base}\AgdaSpace{}%
\AgdaSymbol{:}\AgdaSpace{}%
\AgdaFunction{P}\AgdaSpace{}%
\AgdaInductiveConstructor{𝟎}\<%
\\
\>[2]\AgdaFunction{base}\AgdaSpace{}%
\AgdaSymbol{=}\AgdaSpace{}%
\AgdaFunction{refl}\<%
\\
\>[2]\AgdaFunction{step}\AgdaSpace{}%
\AgdaSymbol{:}\AgdaSpace{}%
\AgdaSymbol{∀}\AgdaSpace{}%
\AgdaSymbol{\{}\AgdaBound{x}\AgdaSpace{}%
\AgdaBound{y}\AgdaSymbol{\}}\AgdaSpace{}%
\AgdaSymbol{→}\AgdaSpace{}%
\AgdaFunction{P}\AgdaSpace{}%
\AgdaBound{x}\AgdaSpace{}%
\AgdaSymbol{→}\AgdaSpace{}%
\AgdaFunction{P}\AgdaSpace{}%
\AgdaBound{y}\AgdaSpace{}%
\AgdaSymbol{→}\AgdaSpace{}%
\AgdaFunction{P}\AgdaSpace{}%
\AgdaSymbol{(}\AgdaOperator{\AgdaInductiveConstructor{ω\textasciicircum{}}}\AgdaSpace{}%
\AgdaBound{x}\AgdaSpace{}%
\AgdaOperator{\AgdaInductiveConstructor{⊕}}\AgdaSpace{}%
\AgdaBound{y}\AgdaSymbol{)}\<%
\\
\>[2]\AgdaFunction{step}\AgdaSpace{}%
\AgdaSymbol{\{}\AgdaBound{x}\AgdaSymbol{\}}\AgdaSpace{}%
\AgdaSymbol{\{}\AgdaBound{y}\AgdaSymbol{\}}\AgdaSpace{}%
\AgdaBound{p}\AgdaSpace{}%
\AgdaBound{q}\AgdaSpace{}%
\AgdaSymbol{=}\AgdaSpace{}%
\AgdaOperator{\AgdaFunction{begin}}\<%
\\
\>[2][@{}l@{\AgdaIndent{0}}]%
\>[5]\AgdaFunction{M2H}\AgdaSpace{}%
\AgdaSymbol{(}\AgdaFunction{H2M}\AgdaSpace{}%
\AgdaSymbol{(}\AgdaOperator{\AgdaInductiveConstructor{ω\textasciicircum{}}}\AgdaSpace{}%
\AgdaBound{x}\AgdaSpace{}%
\AgdaOperator{\AgdaInductiveConstructor{⊕}}\AgdaSpace{}%
\AgdaBound{y}\AgdaSymbol{))}\<%
\\
\>[2][@{}l@{\AgdaIndent{0}}]%
\>[3]\AgdaOperator{\AgdaFunction{≡⟨}}\AgdaSpace{}%
\AgdaFunction{insert-⊕}\AgdaSpace{}%
\AgdaSymbol{(}\AgdaFunction{H2M}\AgdaSpace{}%
\AgdaBound{x}\AgdaSymbol{)}\AgdaSpace{}%
\AgdaSymbol{(}\AgdaFunction{H2M}\AgdaSpace{}%
\AgdaBound{y}\AgdaSymbol{)}\AgdaSpace{}%
\AgdaOperator{\AgdaFunction{⟩}}\<%
\\
\>[3][@{}l@{\AgdaIndent{0}}]%
\>[5]\AgdaOperator{\AgdaInductiveConstructor{ω\textasciicircum{}}}\AgdaSpace{}%
\AgdaFunction{M2H}\AgdaSpace{}%
\AgdaSymbol{(}\AgdaFunction{H2M}\AgdaSpace{}%
\AgdaBound{x}\AgdaSymbol{)}\AgdaSpace{}%
\AgdaOperator{\AgdaInductiveConstructor{⊕}}\AgdaSpace{}%
\AgdaFunction{M2H}\AgdaSpace{}%
\AgdaSymbol{(}\AgdaFunction{H2M}\AgdaSpace{}%
\AgdaBound{y}\AgdaSymbol{)}\<%
\\
\>[3]\AgdaOperator{\AgdaFunction{≡⟨}}\AgdaSpace{}%
\AgdaFunction{cong₂}\AgdaSpace{}%
\AgdaOperator{\AgdaInductiveConstructor{ω\textasciicircum{}\AgdaUnderscore{}⊕\AgdaUnderscore{}}}\AgdaSpace{}%
\AgdaBound{p}\AgdaSpace{}%
\AgdaBound{q}\AgdaSpace{}%
\AgdaOperator{\AgdaFunction{⟩}}\<%
\\
\>[3][@{}l@{\AgdaIndent{0}}]%
\>[5]\AgdaOperator{\AgdaInductiveConstructor{ω\textasciicircum{}}}\AgdaSpace{}%
\AgdaBound{x}\AgdaSpace{}%
\AgdaOperator{\AgdaInductiveConstructor{⊕}}\AgdaSpace{}%
\AgdaBound{y}%
\>[15]\AgdaOperator{\AgdaFunction{∎}}\<%
\end{code}}

\newcommand{\eqMH}{
\begin{code}[inline]%
\>[0]\AgdaFunction{M≃H}\AgdaSpace{}%
\AgdaSymbol{:}\AgdaSpace{}%
\AgdaDatatype{MutualOrd}\AgdaSpace{}%
\AgdaOperator{\AgdaFunction{≃}}\AgdaSpace{}%
\AgdaDatatype{HITOrd}\<%
\end{code}}

\begin{code}[hide]%
\>[0]\<%
\\
\>[0]\AgdaFunction{M≃H}\AgdaSpace{}%
\AgdaSymbol{=}\AgdaSpace{}%
\AgdaFunction{isoToEquiv}\AgdaSpace{}%
\AgdaSymbol{(}\AgdaInductiveConstructor{iso}\AgdaSpace{}%
\AgdaFunction{M2H}\AgdaSpace{}%
\AgdaFunction{H2M}\AgdaSpace{}%
\AgdaFunction{H2M2H=id}\AgdaSpace{}%
\AgdaFunction{M2H2M=id}\AgdaSymbol{)}\<%
\\
\>[0]\<%
\end{code}

\newcommand{\pathMH}{
\begin{code}[inline]%
\>[0]\AgdaFunction{M≡H}\AgdaSpace{}%
\AgdaSymbol{:}\AgdaSpace{}%
\AgdaDatatype{MutualOrd}\AgdaSpace{}%
\AgdaOperator{\AgdaFunction{≡}}\AgdaSpace{}%
\AgdaDatatype{HITOrd}\<%
\end{code}}

\begin{code}[hide]%
\>[0]\<%
\\
\>[0]\AgdaFunction{M≡H}\AgdaSpace{}%
\AgdaSymbol{=}\AgdaSpace{}%
\AgdaFunction{ua}\AgdaSpace{}%
\AgdaFunction{M≃H}\<%
\\
\>[0]\<%
\end{code}

\begin{code}[hide]%
\>[0]\<%
\\
\>[0]\AgdaSymbol{\{-\#}\AgdaSpace{}%
\AgdaKeyword{OPTIONS}\AgdaSpace{}%
\AgdaPragma{--cubical}\AgdaSpace{}%
\AgdaSymbol{\#-\}}\<%
\\
\\[\AgdaEmptyExtraSkip]%
\>[0]\AgdaKeyword{module}\AgdaSpace{}%
\AgdaModule{TransfiniteInduction}\AgdaSpace{}%
\AgdaKeyword{where}\<%
\\
\\[\AgdaEmptyExtraSkip]%
\>[0]\AgdaKeyword{open}\AgdaSpace{}%
\AgdaKeyword{import}\AgdaSpace{}%
\AgdaModule{Cubical.Foundations.Everything}\<%
\\
\>[0][@{}l@{\AgdaIndent{0}}]%
\>[1]\AgdaKeyword{using}%
\>[7I]\AgdaSymbol{(}\AgdaFunction{Type₀}\AgdaSpace{}%
\AgdaSymbol{;}\AgdaSpace{}%
\AgdaFunction{Type₁}\AgdaSpace{}%
\AgdaSymbol{;}\AgdaSpace{}%
\AgdaFunction{Type}\AgdaSpace{}%
\AgdaSymbol{;}\AgdaSpace{}%
\AgdaOperator{\AgdaFunction{\AgdaUnderscore{}≡\AgdaUnderscore{}}}\AgdaSpace{}%
\AgdaSymbol{;}\AgdaSpace{}%
\AgdaFunction{refl}\AgdaSpace{}%
\AgdaSymbol{;}\AgdaSpace{}%
\AgdaFunction{J}\AgdaSpace{}%
\AgdaSymbol{;}\AgdaSpace{}%
\AgdaFunction{cong}\AgdaSpace{}%
\AgdaSymbol{;}\AgdaSpace{}%
\AgdaFunction{cong₂}\AgdaSpace{}%
\AgdaSymbol{;}\AgdaSpace{}%
\AgdaFunction{subst}\AgdaSpace{}%
\AgdaSymbol{;}\AgdaSpace{}%
\AgdaOperator{\AgdaFunction{\AgdaUnderscore{}⁻¹}}\AgdaSpace{}%
\AgdaSymbol{;}\AgdaSpace{}%
\AgdaOperator{\AgdaFunction{\AgdaUnderscore{}∙\AgdaUnderscore{}}}\AgdaSymbol{;}\AgdaSpace{}%
\AgdaFunction{transport}\AgdaSpace{}%
\AgdaSymbol{;}\<%
\\
\>[7I][@{}l@{\AgdaIndent{0}}]%
\>[8]\AgdaPostulate{PathP}\AgdaSpace{}%
\AgdaSymbol{;}\AgdaSpace{}%
\AgdaPrimitive{transp}\AgdaSpace{}%
\AgdaSymbol{;}\AgdaSpace{}%
\AgdaPrimitive{\textasciitilde{}\AgdaUnderscore{}}\AgdaSymbol{)}\<%
\\
\>[1]\AgdaKeyword{renaming}\AgdaSpace{}%
\AgdaSymbol{(}\AgdaPrimitive{\AgdaUnderscore{}∧\AgdaUnderscore{}}\AgdaSpace{}%
\AgdaSymbol{to}\AgdaSpace{}%
\AgdaPrimitive{\AgdaUnderscore{}⊓\AgdaUnderscore{}}\AgdaSymbol{)}\<%
\\
\\[\AgdaEmptyExtraSkip]%
\>[0]\AgdaKeyword{open}\AgdaSpace{}%
\AgdaKeyword{import}\AgdaSpace{}%
\AgdaModule{Cubical.Data.Empty}\<%
\\
\\[\AgdaEmptyExtraSkip]%
\>[0]\AgdaKeyword{open}\AgdaSpace{}%
\AgdaKeyword{import}\AgdaSpace{}%
\AgdaModule{Cubical.Relation.Nullary}\AgdaSpace{}%
\AgdaKeyword{hiding}\AgdaSpace{}%
\AgdaSymbol{(}\AgdaDatatype{Dec}\AgdaSymbol{)}\<%
\\
\>[0]\AgdaKeyword{open}\AgdaSpace{}%
\AgdaKeyword{import}\AgdaSpace{}%
\AgdaModule{Cubical.Relation.Nullary.DecidableEq}\<%
\\
\\[\AgdaEmptyExtraSkip]%
\>[0]\AgdaKeyword{open}\AgdaSpace{}%
\AgdaKeyword{import}\AgdaSpace{}%
\AgdaModule{Preliminaries}\<%
\\
\>[0]\AgdaKeyword{open}\AgdaSpace{}%
\AgdaKeyword{import}\AgdaSpace{}%
\AgdaModule{MutualOrd}%
\>[23]\AgdaSymbol{as}\AgdaSpace{}%
\AgdaModule{M}\AgdaSpace{}%
\AgdaKeyword{renaming}\AgdaSpace{}%
\AgdaSymbol{(}\AgdaFunction{<-dec}\AgdaSpace{}%
\AgdaSymbol{to}\AgdaSpace{}%
\AgdaFunction{<-dec'}\AgdaSymbol{)}\<%
\\
\>[0]\AgdaKeyword{open}\AgdaSpace{}%
\AgdaKeyword{import}\AgdaSpace{}%
\AgdaModule{HITOrd}%
\>[23]\AgdaSymbol{as}\AgdaSpace{}%
\AgdaModule{H}\<%
\\
\>[0]\AgdaKeyword{open}\AgdaSpace{}%
\AgdaKeyword{import}\AgdaSpace{}%
\AgdaModule{Equivalences}\<%
\\
\>[0]\<%
\end{code}

\newcommand{\TItype}{
\begin{code}%
\>[0]\AgdaFunction{TI}\AgdaSpace{}%
\AgdaSymbol{:}%
\>[60I]\AgdaSymbol{(}\AgdaBound{A}\AgdaSpace{}%
\AgdaSymbol{:}\AgdaSpace{}%
\AgdaFunction{Type}\AgdaSpace{}%
\AgdaGeneralizable{ℓ}\AgdaSymbol{)}\AgdaSpace{}%
\AgdaSymbol{→}\AgdaSpace{}%
\AgdaSymbol{(}\AgdaBound{A}\AgdaSpace{}%
\AgdaSymbol{→}\AgdaSpace{}%
\AgdaBound{A}\AgdaSpace{}%
\AgdaSymbol{→}\AgdaSpace{}%
\AgdaFunction{Type}\AgdaSpace{}%
\AgdaGeneralizable{ℓ'}\AgdaSymbol{)}\AgdaSpace{}%
\AgdaSymbol{→}\<%
\\
\>[.][@{}l@{}]\<[60I]%
\>[5]\AgdaSymbol{∀}\AgdaSpace{}%
\AgdaBound{ℓ''}\AgdaSpace{}%
\AgdaSymbol{→}\AgdaSpace{}%
\AgdaFunction{Type}\AgdaSpace{}%
\AgdaSymbol{(}\AgdaGeneralizable{ℓ}\AgdaSpace{}%
\AgdaOperator{\AgdaPrimitive{⊔}}\AgdaSpace{}%
\AgdaGeneralizable{ℓ'}\AgdaSpace{}%
\AgdaOperator{\AgdaPrimitive{⊔}}\AgdaSpace{}%
\AgdaPrimitive{lsuc}\AgdaSpace{}%
\AgdaBound{ℓ''}\AgdaSymbol{)}\<%
\\
\>[0]\AgdaFunction{TI}\AgdaSpace{}%
\AgdaBound{A}\AgdaSpace{}%
\AgdaOperator{\AgdaBound{\AgdaUnderscore{}<\AgdaUnderscore{}}}\AgdaSpace{}%
\AgdaBound{ℓ''}%
\>[14]\AgdaSymbol{=}\AgdaSpace{}%
\AgdaSymbol{(}\AgdaBound{P}\AgdaSpace{}%
\AgdaSymbol{:}\AgdaSpace{}%
\AgdaBound{A}\AgdaSpace{}%
\AgdaSymbol{→}\AgdaSpace{}%
\AgdaFunction{Type}\AgdaSpace{}%
\AgdaBound{ℓ''}\AgdaSymbol{)}\<%
\\
\>[14]\AgdaSymbol{→}\AgdaSpace{}%
\AgdaSymbol{(∀}\AgdaSpace{}%
\AgdaBound{x}\AgdaSpace{}%
\AgdaSymbol{→}\AgdaSpace{}%
\AgdaSymbol{(∀}\AgdaSpace{}%
\AgdaBound{y}\AgdaSpace{}%
\AgdaSymbol{→}\AgdaSpace{}%
\AgdaBound{y}\AgdaSpace{}%
\AgdaOperator{\AgdaBound{<}}\AgdaSpace{}%
\AgdaBound{x}\AgdaSpace{}%
\AgdaSymbol{→}\AgdaSpace{}%
\AgdaBound{P}\AgdaSpace{}%
\AgdaBound{y}\AgdaSymbol{)}\AgdaSpace{}%
\AgdaSymbol{→}\AgdaSpace{}%
\AgdaBound{P}\AgdaSpace{}%
\AgdaBound{x}\AgdaSymbol{)}\<%
\\
\>[14]\AgdaSymbol{→}\AgdaSpace{}%
\AgdaSymbol{∀}\AgdaSpace{}%
\AgdaBound{x}\AgdaSpace{}%
\AgdaSymbol{→}\AgdaSpace{}%
\AgdaBound{P}\AgdaSpace{}%
\AgdaBound{x}\<%
\end{code}}

\newcommand{\Acc}{
\begin{code}%
\>[0]\AgdaKeyword{module}\AgdaSpace{}%
\AgdaModule{Acc}\AgdaSpace{}%
\AgdaSymbol{(}\AgdaBound{A}\AgdaSpace{}%
\AgdaSymbol{:}\AgdaSpace{}%
\AgdaFunction{Type}\AgdaSpace{}%
\AgdaGeneralizable{ℓ}\AgdaSymbol{)}\AgdaSpace{}%
\AgdaSymbol{(}\AgdaOperator{\AgdaBound{\AgdaUnderscore{}<\AgdaUnderscore{}}}\AgdaSpace{}%
\AgdaSymbol{:}\AgdaSpace{}%
\AgdaBound{A}\AgdaSpace{}%
\AgdaSymbol{→}\AgdaSpace{}%
\AgdaBound{A}\AgdaSpace{}%
\AgdaSymbol{→}\AgdaSpace{}%
\AgdaFunction{Type}\AgdaSpace{}%
\AgdaGeneralizable{ℓ'}\AgdaSymbol{)}\AgdaSpace{}%
\AgdaKeyword{where}\<%
\\
\\[\AgdaEmptyExtraSkip]%
\>[0][@{}l@{\AgdaIndent{0}}]%
\>[1]\AgdaKeyword{data}\AgdaSpace{}%
\AgdaDatatype{isAccessible}\AgdaSpace{}%
\AgdaSymbol{(}\AgdaBound{x}\AgdaSpace{}%
\AgdaSymbol{:}\AgdaSpace{}%
\AgdaBound{A}\AgdaSymbol{)}\AgdaSpace{}%
\AgdaSymbol{:}\AgdaSpace{}%
\AgdaFunction{Type}\AgdaSpace{}%
\AgdaSymbol{(}\AgdaBound{ℓ}\AgdaSpace{}%
\AgdaOperator{\AgdaPrimitive{⊔}}\AgdaSpace{}%
\AgdaBound{ℓ'}\AgdaSymbol{)}\AgdaSpace{}%
\AgdaKeyword{where}\<%
\\
\>[1][@{}l@{\AgdaIndent{0}}]%
\>[2]\AgdaInductiveConstructor{next}\AgdaSpace{}%
\AgdaSymbol{:}\AgdaSpace{}%
\AgdaSymbol{(∀}\AgdaSpace{}%
\AgdaBound{y}\AgdaSpace{}%
\AgdaSymbol{→}\AgdaSpace{}%
\AgdaBound{y}\AgdaSpace{}%
\AgdaOperator{\AgdaBound{<}}\AgdaSpace{}%
\AgdaBound{x}\AgdaSpace{}%
\AgdaSymbol{→}\AgdaSpace{}%
\AgdaDatatype{isAccessible}\AgdaSpace{}%
\AgdaBound{y}\AgdaSymbol{)}\AgdaSpace{}%
\AgdaSymbol{→}\AgdaSpace{}%
\AgdaDatatype{isAccessible}\AgdaSpace{}%
\AgdaBound{x}\<%
\\
\\[\AgdaEmptyExtraSkip]%
\>[1]\AgdaFunction{accInd}%
\>[147I]\AgdaSymbol{:}\AgdaSpace{}%
\AgdaSymbol{(}\AgdaBound{P}\AgdaSpace{}%
\AgdaSymbol{:}\AgdaSpace{}%
\AgdaBound{A}\AgdaSpace{}%
\AgdaSymbol{→}\AgdaSpace{}%
\AgdaFunction{Type}\AgdaSpace{}%
\AgdaGeneralizable{ℓ''}\AgdaSymbol{)}\<%
\\
\>[.][@{}l@{}]\<[147I]%
\>[8]\AgdaSymbol{→}\AgdaSpace{}%
\AgdaSymbol{(∀}\AgdaSpace{}%
\AgdaBound{x}\AgdaSpace{}%
\AgdaSymbol{→}\AgdaSpace{}%
\AgdaSymbol{(∀}\AgdaSpace{}%
\AgdaBound{y}\AgdaSpace{}%
\AgdaSymbol{→}\AgdaSpace{}%
\AgdaBound{y}\AgdaSpace{}%
\AgdaOperator{\AgdaBound{<}}\AgdaSpace{}%
\AgdaBound{x}\AgdaSpace{}%
\AgdaSymbol{→}\AgdaSpace{}%
\AgdaBound{P}\AgdaSpace{}%
\AgdaBound{y}\AgdaSymbol{)}\AgdaSpace{}%
\AgdaSymbol{→}\AgdaSpace{}%
\AgdaBound{P}\AgdaSpace{}%
\AgdaBound{x}\AgdaSymbol{)}\<%
\\
\>[8]\AgdaSymbol{→}\AgdaSpace{}%
\AgdaSymbol{∀}\AgdaSpace{}%
\AgdaBound{x}\AgdaSpace{}%
\AgdaSymbol{→}\AgdaSpace{}%
\AgdaDatatype{isAccessible}\AgdaSpace{}%
\AgdaBound{x}\AgdaSpace{}%
\AgdaSymbol{→}\AgdaSpace{}%
\AgdaBound{P}\AgdaSpace{}%
\AgdaBound{x}\<%
\\
\>[1]\AgdaFunction{accInd}\AgdaSpace{}%
\AgdaBound{P}\AgdaSpace{}%
\AgdaBound{step}\AgdaSpace{}%
\AgdaBound{x}\AgdaSpace{}%
\AgdaSymbol{(}\AgdaInductiveConstructor{next}\AgdaSpace{}%
\AgdaBound{δ}\AgdaSymbol{)}\AgdaSpace{}%
\AgdaSymbol{=}\<%
\\
\>[1][@{}l@{\AgdaIndent{0}}]%
\>[3]\AgdaBound{step}\AgdaSpace{}%
\AgdaBound{x}\AgdaSpace{}%
\AgdaSymbol{(λ}\AgdaSpace{}%
\AgdaBound{y}\AgdaSpace{}%
\AgdaBound{r}\AgdaSpace{}%
\AgdaSymbol{→}\AgdaSpace{}%
\AgdaFunction{accInd}\AgdaSpace{}%
\AgdaBound{P}\AgdaSpace{}%
\AgdaBound{step}\AgdaSpace{}%
\AgdaBound{y}\AgdaSpace{}%
\AgdaSymbol{(}\AgdaBound{δ}\AgdaSpace{}%
\AgdaBound{y}\AgdaSpace{}%
\AgdaBound{r}\AgdaSymbol{))}\<%
\\
\\[\AgdaEmptyExtraSkip]%
\>[0]\AgdaKeyword{open}\AgdaSpace{}%
\AgdaModule{Acc}\AgdaSpace{}%
\AgdaDatatype{MutualOrd}\AgdaSpace{}%
\AgdaOperator{\AgdaDatatype{\AgdaUnderscore{}<\AgdaUnderscore{}}}\<%
\end{code}}

\newcommand{\zeroAcc}{
\begin{code}[inline]%
\>[0]\AgdaFunction{𝟎Acc}\AgdaSpace{}%
\AgdaSymbol{:}\AgdaSpace{}%
\AgdaDatatype{isAccessible}\AgdaSpace{}%
\AgdaInductiveConstructor{𝟎}\<%
\end{code}}

\newcommand{\WFlemmas}{
\begin{code}%
\>[0]\AgdaFunction{fstAcc}\AgdaSpace{}%
\AgdaSymbol{:}\AgdaSpace{}%
\AgdaSymbol{∀}\AgdaSpace{}%
\AgdaSymbol{\{}\AgdaBound{a}\AgdaSpace{}%
\AgdaBound{b}\AgdaSpace{}%
\AgdaBound{x}\AgdaSymbol{\}}\AgdaSpace{}%
\AgdaSymbol{→}\AgdaSpace{}%
\AgdaDatatype{isAccessible}\AgdaSpace{}%
\AgdaBound{a}\AgdaSpace{}%
\AgdaSymbol{→}\AgdaSpace{}%
\AgdaDatatype{isAccessible}\AgdaSpace{}%
\AgdaBound{b}\<%
\\
\>[0][@{}l@{\AgdaIndent{0}}]%
\>[2]\AgdaSymbol{→}\AgdaSpace{}%
\AgdaBound{x}\AgdaSpace{}%
\AgdaOperator{\AgdaDatatype{<}}\AgdaSpace{}%
\AgdaBound{a}\AgdaSpace{}%
\AgdaSymbol{→}\AgdaSpace{}%
\AgdaSymbol{(}\AgdaBound{r}\AgdaSpace{}%
\AgdaSymbol{:}\AgdaSpace{}%
\AgdaBound{x}\AgdaSpace{}%
\AgdaOperator{\AgdaFunction{≥}}\AgdaSpace{}%
\AgdaFunction{fst}\AgdaSpace{}%
\AgdaBound{b}\AgdaSymbol{)}\<%
\\
\>[2]\AgdaSymbol{→}\AgdaSpace{}%
\AgdaDatatype{isAccessible}\AgdaSpace{}%
\AgdaSymbol{(}\AgdaOperator{\AgdaInductiveConstructor{ω\textasciicircum{}}}\AgdaSpace{}%
\AgdaBound{x}\AgdaSpace{}%
\AgdaOperator{\AgdaInductiveConstructor{+}}\AgdaSpace{}%
\AgdaBound{b}\AgdaSpace{}%
\AgdaOperator{\AgdaInductiveConstructor{[}}\AgdaSpace{}%
\AgdaBound{r}\AgdaSpace{}%
\AgdaOperator{\AgdaInductiveConstructor{]}}\AgdaSymbol{)}\<%
\\
\>[0]\AgdaFunction{sndAcc}\AgdaSpace{}%
\AgdaSymbol{:}\AgdaSpace{}%
\AgdaSymbol{∀}\AgdaSpace{}%
\AgdaSymbol{\{}\AgdaBound{a}\AgdaSpace{}%
\AgdaBound{b}\AgdaSpace{}%
\AgdaBound{y}\AgdaSymbol{\}}\AgdaSpace{}%
\AgdaSymbol{→}\AgdaSpace{}%
\AgdaDatatype{isAccessible}\AgdaSpace{}%
\AgdaBound{a}\AgdaSpace{}%
\AgdaSymbol{→}\AgdaSpace{}%
\AgdaDatatype{isAccessible}\AgdaSpace{}%
\AgdaBound{b}\<%
\\
\>[0][@{}l@{\AgdaIndent{0}}]%
\>[2]\AgdaSymbol{→}\AgdaSpace{}%
\AgdaBound{y}\AgdaSpace{}%
\AgdaOperator{\AgdaDatatype{<}}\AgdaSpace{}%
\AgdaBound{b}\AgdaSpace{}%
\AgdaSymbol{→}\AgdaSpace{}%
\AgdaSymbol{(}\AgdaBound{r}\AgdaSpace{}%
\AgdaSymbol{:}\AgdaSpace{}%
\AgdaBound{a}\AgdaSpace{}%
\AgdaOperator{\AgdaFunction{≥}}\AgdaSpace{}%
\AgdaFunction{fst}\AgdaSpace{}%
\AgdaBound{y}\AgdaSymbol{)}\<%
\\
\>[2]\AgdaSymbol{→}\AgdaSpace{}%
\AgdaDatatype{isAccessible}\AgdaSpace{}%
\AgdaSymbol{(}\AgdaOperator{\AgdaInductiveConstructor{ω\textasciicircum{}}}\AgdaSpace{}%
\AgdaBound{a}\AgdaSpace{}%
\AgdaOperator{\AgdaInductiveConstructor{+}}\AgdaSpace{}%
\AgdaBound{y}\AgdaSpace{}%
\AgdaOperator{\AgdaInductiveConstructor{[}}\AgdaSpace{}%
\AgdaBound{r}\AgdaSpace{}%
\AgdaOperator{\AgdaInductiveConstructor{]}}\AgdaSymbol{)}\<%
\end{code}}

\begin{code}[hide]%
\>[0]\AgdaFunction{fstAcc'}\AgdaSpace{}%
\AgdaSymbol{:}\AgdaSpace{}%
\AgdaSymbol{∀}\AgdaSpace{}%
\AgdaSymbol{\{}\AgdaBound{a}\AgdaSpace{}%
\AgdaBound{a'}\AgdaSymbol{\}}\AgdaSpace{}%
\AgdaSymbol{→}\AgdaSpace{}%
\AgdaDatatype{isAccessible}\AgdaSpace{}%
\AgdaBound{a'}\AgdaSpace{}%
\AgdaSymbol{→}\AgdaSpace{}%
\AgdaBound{a}\AgdaSpace{}%
\AgdaOperator{\AgdaFunction{≡}}\AgdaSpace{}%
\AgdaBound{a'}\<%
\\
\>[0][@{}l@{\AgdaIndent{0}}]%
\>[2]\AgdaSymbol{→}\AgdaSpace{}%
\AgdaSymbol{∀}\AgdaSpace{}%
\AgdaSymbol{\{}\AgdaBound{b}\AgdaSpace{}%
\AgdaBound{x}\AgdaSymbol{\}}\AgdaSpace{}%
\AgdaSymbol{→}\AgdaSpace{}%
\AgdaDatatype{isAccessible}\AgdaSpace{}%
\AgdaBound{b}\AgdaSpace{}%
\AgdaSymbol{→}\AgdaSpace{}%
\AgdaBound{x}\AgdaSpace{}%
\AgdaOperator{\AgdaDatatype{<}}\AgdaSpace{}%
\AgdaBound{a'}\AgdaSpace{}%
\AgdaSymbol{→}\AgdaSpace{}%
\AgdaSymbol{(}\AgdaBound{r}\AgdaSpace{}%
\AgdaSymbol{:}\AgdaSpace{}%
\AgdaBound{x}\AgdaSpace{}%
\AgdaOperator{\AgdaFunction{≥}}\AgdaSpace{}%
\AgdaFunction{fst}\AgdaSpace{}%
\AgdaBound{b}\AgdaSymbol{)}\<%
\\
\>[2]\AgdaSymbol{→}\AgdaSpace{}%
\AgdaDatatype{isAccessible}\AgdaSpace{}%
\AgdaSymbol{(}\AgdaOperator{\AgdaInductiveConstructor{ω\textasciicircum{}}}\AgdaSpace{}%
\AgdaBound{x}\AgdaSpace{}%
\AgdaOperator{\AgdaInductiveConstructor{+}}\AgdaSpace{}%
\AgdaBound{b}\AgdaSpace{}%
\AgdaOperator{\AgdaInductiveConstructor{[}}\AgdaSpace{}%
\AgdaBound{r}\AgdaSpace{}%
\AgdaOperator{\AgdaInductiveConstructor{]}}\AgdaSymbol{)}\<%
\\
\>[0]\AgdaFunction{sndAcc'}\AgdaSpace{}%
\AgdaSymbol{:}\AgdaSpace{}%
\AgdaSymbol{∀}\AgdaSpace{}%
\AgdaSymbol{\{}\AgdaBound{a}\AgdaSpace{}%
\AgdaBound{a'}\AgdaSymbol{\}}\AgdaSpace{}%
\AgdaSymbol{→}\AgdaSpace{}%
\AgdaDatatype{isAccessible}\AgdaSpace{}%
\AgdaBound{a'}\AgdaSpace{}%
\AgdaSymbol{→}\AgdaSpace{}%
\AgdaBound{a}\AgdaSpace{}%
\AgdaOperator{\AgdaFunction{≡}}\AgdaSpace{}%
\AgdaBound{a'}\<%
\\
\>[0][@{}l@{\AgdaIndent{0}}]%
\>[2]\AgdaSymbol{→}\AgdaSpace{}%
\AgdaSymbol{∀}\AgdaSpace{}%
\AgdaSymbol{\{}\AgdaBound{c}\AgdaSpace{}%
\AgdaBound{y}\AgdaSymbol{\}}\AgdaSpace{}%
\AgdaSymbol{→}\AgdaSpace{}%
\AgdaDatatype{isAccessible}\AgdaSpace{}%
\AgdaBound{c}\AgdaSpace{}%
\AgdaSymbol{→}\AgdaSpace{}%
\AgdaBound{y}\AgdaSpace{}%
\AgdaOperator{\AgdaDatatype{<}}\AgdaSpace{}%
\AgdaBound{c}\AgdaSpace{}%
\AgdaSymbol{→}\AgdaSpace{}%
\AgdaSymbol{(}\AgdaBound{r}\AgdaSpace{}%
\AgdaSymbol{:}\AgdaSpace{}%
\AgdaBound{a}\AgdaSpace{}%
\AgdaOperator{\AgdaFunction{≥}}\AgdaSpace{}%
\AgdaFunction{fst}\AgdaSpace{}%
\AgdaBound{y}\AgdaSymbol{)}\<%
\\
\>[2]\AgdaSymbol{→}\AgdaSpace{}%
\AgdaDatatype{isAccessible}\AgdaSpace{}%
\AgdaSymbol{(}\AgdaOperator{\AgdaInductiveConstructor{ω\textasciicircum{}}}\AgdaSpace{}%
\AgdaBound{a}\AgdaSpace{}%
\AgdaOperator{\AgdaInductiveConstructor{+}}\AgdaSpace{}%
\AgdaBound{y}\AgdaSpace{}%
\AgdaOperator{\AgdaInductiveConstructor{[}}\AgdaSpace{}%
\AgdaBound{r}\AgdaSpace{}%
\AgdaOperator{\AgdaInductiveConstructor{]}}\AgdaSymbol{)}\<%
\\
\\[\AgdaEmptyExtraSkip]%
\>[0]\AgdaFunction{𝟎Acc}\AgdaSpace{}%
\AgdaSymbol{=}\AgdaSpace{}%
\AgdaInductiveConstructor{next}\AgdaSpace{}%
\AgdaSymbol{(λ}\AgdaSpace{}%
\AgdaBound{x}\AgdaSpace{}%
\AgdaBound{x<𝟎}\AgdaSpace{}%
\AgdaSymbol{→}\AgdaSpace{}%
\AgdaFunction{⊥-elim}\AgdaSpace{}%
\AgdaSymbol{(}\AgdaFunction{≮𝟎}\AgdaSpace{}%
\AgdaBound{x<𝟎}\AgdaSymbol{))}\<%
\\
\\[\AgdaEmptyExtraSkip]%
\>[0]\AgdaFunction{fstAcc'}\AgdaSpace{}%
\AgdaSymbol{\{}\AgdaBound{a}\AgdaSymbol{\}}\AgdaSpace{}%
\AgdaSymbol{\{}\AgdaBound{a'}\AgdaSymbol{\}}\AgdaSpace{}%
\AgdaSymbol{(}\AgdaInductiveConstructor{next}\AgdaSpace{}%
\AgdaBound{ξ}\AgdaSymbol{)}\AgdaSpace{}%
\AgdaBound{a=a'}\AgdaSpace{}%
\AgdaSymbol{\{}\AgdaBound{b}\AgdaSymbol{\}}\AgdaSpace{}%
\AgdaSymbol{\{}\AgdaBound{x}\AgdaSymbol{\}}\AgdaSpace{}%
\AgdaBound{acᵇ}\AgdaSpace{}%
\AgdaBound{x<a}\AgdaSpace{}%
\AgdaBound{r}\AgdaSpace{}%
\AgdaSymbol{=}\AgdaSpace{}%
\AgdaInductiveConstructor{next}\AgdaSpace{}%
\AgdaFunction{goal}\<%
\\
\>[0][@{}l@{\AgdaIndent{0}}]%
\>[2]\AgdaKeyword{where}\<%
\\
\>[2][@{}l@{\AgdaIndent{0}}]%
\>[3]\AgdaFunction{goal}\AgdaSpace{}%
\AgdaSymbol{:}\AgdaSpace{}%
\AgdaSymbol{∀}\AgdaSpace{}%
\AgdaBound{z}\AgdaSpace{}%
\AgdaSymbol{→}\AgdaSpace{}%
\AgdaBound{z}\AgdaSpace{}%
\AgdaOperator{\AgdaDatatype{<}}\AgdaSpace{}%
\AgdaOperator{\AgdaInductiveConstructor{ω\textasciicircum{}}}\AgdaSpace{}%
\AgdaBound{x}\AgdaSpace{}%
\AgdaOperator{\AgdaInductiveConstructor{+}}\AgdaSpace{}%
\AgdaBound{b}\AgdaSpace{}%
\AgdaOperator{\AgdaInductiveConstructor{[}}\AgdaSpace{}%
\AgdaBound{r}\AgdaSpace{}%
\AgdaOperator{\AgdaInductiveConstructor{]}}\AgdaSpace{}%
\AgdaSymbol{→}\AgdaSpace{}%
\AgdaDatatype{isAccessible}\AgdaSpace{}%
\AgdaBound{z}\<%
\\
\>[3]\AgdaFunction{goal}\AgdaSpace{}%
\AgdaInductiveConstructor{𝟎}\AgdaSpace{}%
\AgdaInductiveConstructor{<₁}\AgdaSpace{}%
\AgdaSymbol{=}\AgdaSpace{}%
\AgdaFunction{𝟎Acc}\<%
\\
\>[3]\AgdaFunction{goal}\AgdaSpace{}%
\AgdaSymbol{(}\AgdaOperator{\AgdaInductiveConstructor{ω\textasciicircum{}}}\AgdaSpace{}%
\AgdaBound{c}\AgdaSpace{}%
\AgdaOperator{\AgdaInductiveConstructor{+}}\AgdaSpace{}%
\AgdaBound{d}\AgdaSpace{}%
\AgdaOperator{\AgdaInductiveConstructor{[}}\AgdaSpace{}%
\AgdaBound{s}\AgdaSpace{}%
\AgdaOperator{\AgdaInductiveConstructor{]}}\AgdaSymbol{)}\AgdaSpace{}%
\AgdaSymbol{(}\AgdaInductiveConstructor{<₂}\AgdaSpace{}%
\AgdaBound{c<y}\AgdaSymbol{)}\AgdaSpace{}%
\AgdaSymbol{=}\AgdaSpace{}%
\AgdaFunction{fstAcc'}\AgdaSpace{}%
\AgdaSymbol{(}\AgdaBound{ξ}\AgdaSpace{}%
\AgdaBound{x}\AgdaSpace{}%
\AgdaBound{x<a}\AgdaSymbol{)}\AgdaSpace{}%
\AgdaFunction{refl}\AgdaSpace{}%
\AgdaSymbol{(}\AgdaFunction{goal}\AgdaSpace{}%
\AgdaBound{d}\AgdaSpace{}%
\AgdaSymbol{(}\AgdaFunction{<-trans}\AgdaSpace{}%
\AgdaSymbol{(}\AgdaFunction{rest<}\AgdaSpace{}%
\AgdaBound{c}\AgdaSpace{}%
\AgdaBound{d}\AgdaSpace{}%
\AgdaBound{s}\AgdaSymbol{)}\AgdaSpace{}%
\AgdaSymbol{(}\AgdaInductiveConstructor{<₂}\AgdaSpace{}%
\AgdaBound{c<y}\AgdaSymbol{)))}\AgdaSpace{}%
\AgdaBound{c<y}\AgdaSpace{}%
\AgdaBound{s}\<%
\\
\>[3]\AgdaFunction{goal}\AgdaSpace{}%
\AgdaSymbol{(}\AgdaOperator{\AgdaInductiveConstructor{ω\textasciicircum{}}}\AgdaSpace{}%
\AgdaBound{c}\AgdaSpace{}%
\AgdaOperator{\AgdaInductiveConstructor{+}}\AgdaSpace{}%
\AgdaBound{d}\AgdaSpace{}%
\AgdaOperator{\AgdaInductiveConstructor{[}}\AgdaSpace{}%
\AgdaBound{s}\AgdaSpace{}%
\AgdaOperator{\AgdaInductiveConstructor{]}}\AgdaSymbol{)}\AgdaSpace{}%
\AgdaSymbol{(}\AgdaInductiveConstructor{<₃}\AgdaSpace{}%
\AgdaBound{c=y}\AgdaSpace{}%
\AgdaBound{d<b}\AgdaSymbol{)}\AgdaSpace{}%
\AgdaSymbol{=}\AgdaSpace{}%
\AgdaFunction{sndAcc'}\AgdaSpace{}%
\AgdaSymbol{(}\AgdaBound{ξ}\AgdaSpace{}%
\AgdaBound{x}\AgdaSpace{}%
\AgdaBound{x<a}\AgdaSymbol{)}\AgdaSpace{}%
\AgdaBound{c=y}\AgdaSpace{}%
\AgdaBound{acᵇ}\AgdaSpace{}%
\AgdaBound{d<b}\AgdaSpace{}%
\AgdaBound{s}\<%
\\
\\[\AgdaEmptyExtraSkip]%
\>[0]\AgdaFunction{sndAcc'}\AgdaSpace{}%
\AgdaSymbol{\{}\AgdaBound{a}\AgdaSymbol{\}}\AgdaSpace{}%
\AgdaSymbol{\{}\AgdaBound{a'}\AgdaSymbol{\}}\AgdaSpace{}%
\AgdaBound{acᵃ}\AgdaSpace{}%
\AgdaBound{a=a'}\AgdaSpace{}%
\AgdaSymbol{\{}\AgdaBound{c}\AgdaSymbol{\}}\AgdaSpace{}%
\AgdaSymbol{\{}\AgdaBound{y}\AgdaSymbol{\}}\AgdaSpace{}%
\AgdaSymbol{(}\AgdaInductiveConstructor{next}\AgdaSpace{}%
\AgdaBound{ξᶜ}\AgdaSymbol{)}\AgdaSpace{}%
\AgdaBound{y<c}\AgdaSpace{}%
\AgdaBound{r}\AgdaSpace{}%
\AgdaSymbol{=}\AgdaSpace{}%
\AgdaInductiveConstructor{next}\AgdaSpace{}%
\AgdaFunction{goal}\<%
\\
\>[0][@{}l@{\AgdaIndent{0}}]%
\>[2]\AgdaKeyword{where}\<%
\\
\>[2][@{}l@{\AgdaIndent{0}}]%
\>[3]\AgdaFunction{goal}\AgdaSpace{}%
\AgdaSymbol{:}\AgdaSpace{}%
\AgdaSymbol{∀}\AgdaSpace{}%
\AgdaBound{z}\AgdaSpace{}%
\AgdaSymbol{→}\AgdaSpace{}%
\AgdaBound{z}\AgdaSpace{}%
\AgdaOperator{\AgdaDatatype{<}}\AgdaSpace{}%
\AgdaOperator{\AgdaInductiveConstructor{ω\textasciicircum{}}}\AgdaSpace{}%
\AgdaBound{a}\AgdaSpace{}%
\AgdaOperator{\AgdaInductiveConstructor{+}}\AgdaSpace{}%
\AgdaBound{y}\AgdaSpace{}%
\AgdaOperator{\AgdaInductiveConstructor{[}}\AgdaSpace{}%
\AgdaBound{r}\AgdaSpace{}%
\AgdaOperator{\AgdaInductiveConstructor{]}}\AgdaSpace{}%
\AgdaSymbol{→}\AgdaSpace{}%
\AgdaDatatype{isAccessible}\AgdaSpace{}%
\AgdaBound{z}\<%
\\
\>[3]\AgdaFunction{goal}\AgdaSpace{}%
\AgdaInductiveConstructor{𝟎}\AgdaSpace{}%
\AgdaInductiveConstructor{<₁}\AgdaSpace{}%
\AgdaSymbol{=}\AgdaSpace{}%
\AgdaFunction{𝟎Acc}\<%
\\
\>[3]\AgdaFunction{goal}\AgdaSpace{}%
\AgdaSymbol{(}\AgdaOperator{\AgdaInductiveConstructor{ω\textasciicircum{}}}\AgdaSpace{}%
\AgdaBound{b}\AgdaSpace{}%
\AgdaOperator{\AgdaInductiveConstructor{+}}\AgdaSpace{}%
\AgdaBound{d}\AgdaSpace{}%
\AgdaOperator{\AgdaInductiveConstructor{[}}\AgdaSpace{}%
\AgdaBound{t}\AgdaSpace{}%
\AgdaOperator{\AgdaInductiveConstructor{]}}\AgdaSymbol{)}\AgdaSpace{}%
\AgdaSymbol{(}\AgdaInductiveConstructor{<₂}\AgdaSpace{}%
\AgdaBound{b<a}\AgdaSymbol{)}\AgdaSpace{}%
\AgdaSymbol{=}\AgdaSpace{}%
\AgdaFunction{fstAcc'}\AgdaSpace{}%
\AgdaBound{acᵃ}\AgdaSpace{}%
\AgdaBound{a=a'}\AgdaSpace{}%
\AgdaSymbol{(}\AgdaFunction{goal}\AgdaSpace{}%
\AgdaBound{d}\AgdaSpace{}%
\AgdaSymbol{(}\AgdaFunction{<-trans}\AgdaSpace{}%
\AgdaSymbol{(}\AgdaFunction{rest<}\AgdaSpace{}%
\AgdaBound{b}\AgdaSpace{}%
\AgdaBound{d}\AgdaSpace{}%
\AgdaBound{t}\AgdaSymbol{)}\AgdaSpace{}%
\AgdaSymbol{(}\AgdaInductiveConstructor{<₂}\AgdaSpace{}%
\AgdaBound{b<a}\AgdaSymbol{)))}\AgdaSpace{}%
\AgdaSymbol{(}\AgdaFunction{subst}\AgdaSpace{}%
\AgdaSymbol{(}\AgdaBound{b}\AgdaSpace{}%
\AgdaOperator{\AgdaDatatype{<\AgdaUnderscore{}}}\AgdaSymbol{)}\AgdaSpace{}%
\AgdaBound{a=a'}\AgdaSpace{}%
\AgdaBound{b<a}\AgdaSymbol{)}\AgdaSpace{}%
\AgdaBound{t}\<%
\\
\>[3]\AgdaFunction{goal}\AgdaSpace{}%
\AgdaSymbol{(}\AgdaOperator{\AgdaInductiveConstructor{ω\textasciicircum{}}}\AgdaSpace{}%
\AgdaBound{b}\AgdaSpace{}%
\AgdaOperator{\AgdaInductiveConstructor{+}}\AgdaSpace{}%
\AgdaBound{d}\AgdaSpace{}%
\AgdaOperator{\AgdaInductiveConstructor{[}}\AgdaSpace{}%
\AgdaBound{t}\AgdaSpace{}%
\AgdaOperator{\AgdaInductiveConstructor{]}}\AgdaSymbol{)}\AgdaSpace{}%
\AgdaSymbol{(}\AgdaInductiveConstructor{<₃}\AgdaSpace{}%
\AgdaBound{b=a}\AgdaSpace{}%
\AgdaBound{d<y}\AgdaSymbol{)}\AgdaSpace{}%
\AgdaSymbol{=}\AgdaSpace{}%
\AgdaFunction{sndAcc'}\AgdaSpace{}%
\AgdaBound{acᵃ}\AgdaSpace{}%
\AgdaSymbol{(}\AgdaBound{b=a}\AgdaSpace{}%
\AgdaOperator{\AgdaFunction{∙}}\AgdaSpace{}%
\AgdaBound{a=a'}\AgdaSymbol{)}\AgdaSpace{}%
\AgdaSymbol{(}\AgdaBound{ξᶜ}\AgdaSpace{}%
\AgdaBound{y}\AgdaSpace{}%
\AgdaBound{y<c}\AgdaSymbol{)}\AgdaSpace{}%
\AgdaBound{d<y}\AgdaSpace{}%
\AgdaBound{t}\<%
\\
\>[0]\<%
\end{code}

\newcommand{\plusAcc}{
\begin{code}%
\>[0]\AgdaFunction{ω+Acc}\AgdaSpace{}%
\AgdaSymbol{:}\AgdaSpace{}%
\AgdaSymbol{(}\AgdaBound{a}\AgdaSpace{}%
\AgdaBound{b}\AgdaSpace{}%
\AgdaSymbol{:}\AgdaSpace{}%
\AgdaDatatype{MutualOrd}\AgdaSymbol{)}\AgdaSpace{}%
\AgdaSymbol{(}\AgdaBound{r}\AgdaSpace{}%
\AgdaSymbol{:}\AgdaSpace{}%
\AgdaBound{a}\AgdaSpace{}%
\AgdaOperator{\AgdaFunction{≥}}\AgdaSpace{}%
\AgdaFunction{fst}\AgdaSpace{}%
\AgdaBound{b}\AgdaSymbol{)}\<%
\\
\>[0][@{}l@{\AgdaIndent{0}}]%
\>[5]\AgdaSymbol{→}\AgdaSpace{}%
\AgdaDatatype{isAccessible}\AgdaSpace{}%
\AgdaBound{a}\AgdaSpace{}%
\AgdaSymbol{→}\AgdaSpace{}%
\AgdaDatatype{isAccessible}\AgdaSpace{}%
\AgdaBound{b}\<%
\\
\>[5]\AgdaSymbol{→}\AgdaSpace{}%
\AgdaDatatype{isAccessible}\AgdaSpace{}%
\AgdaSymbol{(}\AgdaOperator{\AgdaInductiveConstructor{ω\textasciicircum{}}}\AgdaSpace{}%
\AgdaBound{a}\AgdaSpace{}%
\AgdaOperator{\AgdaInductiveConstructor{+}}\AgdaSpace{}%
\AgdaBound{b}\AgdaSpace{}%
\AgdaOperator{\AgdaInductiveConstructor{[}}\AgdaSpace{}%
\AgdaBound{r}\AgdaSpace{}%
\AgdaOperator{\AgdaInductiveConstructor{]}}\AgdaSymbol{)}\<%
\end{code}}

\begin{code}[hide]%
\>[0]\AgdaFunction{fstAcc}\AgdaSpace{}%
\AgdaBound{ac}\AgdaSpace{}%
\AgdaSymbol{=}\AgdaSpace{}%
\AgdaFunction{fstAcc'}\AgdaSpace{}%
\AgdaBound{ac}\AgdaSpace{}%
\AgdaFunction{refl}\<%
\\
\>[0]\AgdaFunction{sndAcc}\AgdaSpace{}%
\AgdaBound{ac}\AgdaSpace{}%
\AgdaSymbol{=}\AgdaSpace{}%
\AgdaFunction{sndAcc'}\AgdaSpace{}%
\AgdaBound{ac}\AgdaSpace{}%
\AgdaFunction{refl}\<%
\\
\>[0]\AgdaFunction{ω+Acc}\AgdaSpace{}%
\AgdaBound{a}\AgdaSpace{}%
\AgdaBound{b}\AgdaSpace{}%
\AgdaBound{r}\AgdaSpace{}%
\AgdaBound{θ}\AgdaSpace{}%
\AgdaBound{δ}\AgdaSpace{}%
\AgdaSymbol{=}\AgdaSpace{}%
\AgdaInductiveConstructor{next}\AgdaSpace{}%
\AgdaFunction{goal}\<%
\\
\>[0][@{}l@{\AgdaIndent{0}}]%
\>[1]\AgdaKeyword{where}\<%
\\
\>[1][@{}l@{\AgdaIndent{0}}]%
\>[2]\AgdaFunction{goal}\AgdaSpace{}%
\AgdaSymbol{:}\AgdaSpace{}%
\AgdaSymbol{∀}\AgdaSpace{}%
\AgdaBound{z}\AgdaSpace{}%
\AgdaSymbol{→}\AgdaSpace{}%
\AgdaBound{z}\AgdaSpace{}%
\AgdaOperator{\AgdaDatatype{<}}\AgdaSpace{}%
\AgdaOperator{\AgdaInductiveConstructor{ω\textasciicircum{}}}\AgdaSpace{}%
\AgdaBound{a}\AgdaSpace{}%
\AgdaOperator{\AgdaInductiveConstructor{+}}\AgdaSpace{}%
\AgdaBound{b}\AgdaSpace{}%
\AgdaOperator{\AgdaInductiveConstructor{[}}\AgdaSpace{}%
\AgdaBound{r}\AgdaSpace{}%
\AgdaOperator{\AgdaInductiveConstructor{]}}\AgdaSpace{}%
\AgdaSymbol{→}\AgdaSpace{}%
\AgdaDatatype{isAccessible}\AgdaSpace{}%
\AgdaBound{z}\<%
\\
\>[2]\AgdaFunction{goal}\AgdaSpace{}%
\AgdaInductiveConstructor{𝟎}\AgdaSpace{}%
\AgdaInductiveConstructor{<₁}\AgdaSpace{}%
\AgdaSymbol{=}\AgdaSpace{}%
\AgdaFunction{𝟎Acc}\<%
\\
\>[2]\AgdaFunction{goal}\AgdaSpace{}%
\AgdaOperator{\AgdaInductiveConstructor{ω\textasciicircum{}}}\AgdaSpace{}%
\AgdaBound{c}\AgdaSpace{}%
\AgdaOperator{\AgdaInductiveConstructor{+}}\AgdaSpace{}%
\AgdaBound{d}\AgdaSpace{}%
\AgdaOperator{\AgdaInductiveConstructor{[}}\AgdaSpace{}%
\AgdaBound{s}\AgdaSpace{}%
\AgdaOperator{\AgdaInductiveConstructor{]}}\AgdaSpace{}%
\AgdaSymbol{(}\AgdaInductiveConstructor{<₂}\AgdaSpace{}%
\AgdaBound{c<a}\AgdaSymbol{)}\AgdaSpace{}%
\AgdaSymbol{=}\AgdaSpace{}%
\AgdaFunction{fstAcc}\AgdaSpace{}%
\AgdaBound{θ}\AgdaSpace{}%
\AgdaFunction{IH}\AgdaSpace{}%
\AgdaBound{c<a}\AgdaSpace{}%
\AgdaBound{s}\<%
\\
\>[2][@{}l@{\AgdaIndent{0}}]%
\>[3]\AgdaKeyword{where}\<%
\\
\>[3][@{}l@{\AgdaIndent{0}}]%
\>[4]\AgdaFunction{IH}\AgdaSpace{}%
\AgdaSymbol{:}\AgdaSpace{}%
\AgdaDatatype{isAccessible}\AgdaSpace{}%
\AgdaBound{d}\<%
\\
\>[4]\AgdaFunction{IH}\AgdaSpace{}%
\AgdaSymbol{=}\AgdaSpace{}%
\AgdaFunction{goal}\AgdaSpace{}%
\AgdaBound{d}\AgdaSpace{}%
\AgdaSymbol{(}\AgdaFunction{<-trans}\AgdaSpace{}%
\AgdaSymbol{(}\AgdaFunction{rest<}\AgdaSpace{}%
\AgdaBound{c}\AgdaSpace{}%
\AgdaBound{d}\AgdaSpace{}%
\AgdaBound{s}\AgdaSymbol{)}\AgdaSpace{}%
\AgdaSymbol{(}\AgdaInductiveConstructor{<₂}\AgdaSpace{}%
\AgdaBound{c<a}\AgdaSymbol{))}\<%
\\
\>[2]\AgdaFunction{goal}\AgdaSpace{}%
\AgdaOperator{\AgdaInductiveConstructor{ω\textasciicircum{}}}\AgdaSpace{}%
\AgdaBound{c}\AgdaSpace{}%
\AgdaOperator{\AgdaInductiveConstructor{+}}\AgdaSpace{}%
\AgdaBound{d}\AgdaSpace{}%
\AgdaOperator{\AgdaInductiveConstructor{[}}\AgdaSpace{}%
\AgdaBound{s}\AgdaSpace{}%
\AgdaOperator{\AgdaInductiveConstructor{]}}\AgdaSpace{}%
\AgdaSymbol{(}\AgdaInductiveConstructor{<₃}\AgdaSpace{}%
\AgdaBound{c=a}\AgdaSpace{}%
\AgdaBound{d<b}\AgdaSymbol{)}\AgdaSpace{}%
\AgdaSymbol{=}\AgdaSpace{}%
\AgdaFunction{sndAcc'}\AgdaSpace{}%
\AgdaBound{θ}\AgdaSpace{}%
\AgdaBound{c=a}\AgdaSpace{}%
\AgdaBound{δ}\AgdaSpace{}%
\AgdaBound{d<b}\AgdaSpace{}%
\AgdaBound{s}\<%
\end{code}

\newcommand{\wellfoundedness}{
\begin{code}%
\>[0]\AgdaFunction{WF}\AgdaSpace{}%
\AgdaSymbol{:}\AgdaSpace{}%
\AgdaSymbol{(}\AgdaBound{x}\AgdaSpace{}%
\AgdaSymbol{:}\AgdaSpace{}%
\AgdaDatatype{MutualOrd}\AgdaSymbol{)}\AgdaSpace{}%
\AgdaSymbol{→}\AgdaSpace{}%
\AgdaDatatype{isAccessible}\AgdaSpace{}%
\AgdaBound{x}\<%
\\
\>[0]\AgdaFunction{WF}\AgdaSpace{}%
\AgdaInductiveConstructor{𝟎}\AgdaSpace{}%
\AgdaSymbol{=}\AgdaSpace{}%
\AgdaFunction{𝟎Acc}\<%
\\
\>[0]\AgdaFunction{WF}\AgdaSpace{}%
\AgdaSymbol{(}\AgdaOperator{\AgdaInductiveConstructor{ω\textasciicircum{}}}\AgdaSpace{}%
\AgdaBound{a}\AgdaSpace{}%
\AgdaOperator{\AgdaInductiveConstructor{+}}\AgdaSpace{}%
\AgdaBound{b}\AgdaSpace{}%
\AgdaOperator{\AgdaInductiveConstructor{[}}\AgdaSpace{}%
\AgdaBound{r}\AgdaSpace{}%
\AgdaOperator{\AgdaInductiveConstructor{]}}\AgdaSymbol{)}\AgdaSpace{}%
\AgdaSymbol{=}\AgdaSpace{}%
\AgdaFunction{ω+Acc}\AgdaSpace{}%
\AgdaBound{a}\AgdaSpace{}%
\AgdaBound{b}\AgdaSpace{}%
\AgdaBound{r}\AgdaSpace{}%
\AgdaSymbol{(}\AgdaFunction{WF}\AgdaSpace{}%
\AgdaBound{a}\AgdaSymbol{)}\AgdaSpace{}%
\AgdaSymbol{(}\AgdaFunction{WF}\AgdaSpace{}%
\AgdaBound{b}\AgdaSymbol{)}\<%
\end{code}}

\newcommand{\MTI}{
\begin{code}[inline]%
\>[0]\AgdaFunction{MTI}\AgdaSpace{}%
\AgdaSymbol{:}\AgdaSpace{}%
\AgdaFunction{TI}\AgdaSpace{}%
\AgdaDatatype{MutualOrd}\AgdaSpace{}%
\AgdaOperator{\AgdaDatatype{\AgdaUnderscore{}<\AgdaUnderscore{}}}\AgdaSpace{}%
\AgdaGeneralizable{ℓ}\<%
\end{code}}

\begin{code}[hide]%
\>[0]\<%
\\
\>[0]\AgdaFunction{<-≡}\AgdaSpace{}%
\AgdaSymbol{:}\AgdaSpace{}%
\AgdaSymbol{\{}\AgdaBound{a}\AgdaSpace{}%
\AgdaBound{b}\AgdaSpace{}%
\AgdaBound{c}\AgdaSpace{}%
\AgdaSymbol{:}\AgdaSpace{}%
\AgdaDatatype{MutualOrd}\AgdaSymbol{\}}\AgdaSpace{}%
\AgdaSymbol{→}\AgdaSpace{}%
\AgdaBound{a}\AgdaSpace{}%
\AgdaOperator{\AgdaDatatype{<}}\AgdaSpace{}%
\AgdaBound{b}\AgdaSpace{}%
\AgdaSymbol{→}\AgdaSpace{}%
\AgdaBound{b}\AgdaSpace{}%
\AgdaOperator{\AgdaFunction{≡}}\AgdaSpace{}%
\AgdaBound{c}\AgdaSpace{}%
\AgdaSymbol{→}\AgdaSpace{}%
\AgdaBound{a}\AgdaSpace{}%
\AgdaOperator{\AgdaDatatype{<}}\AgdaSpace{}%
\AgdaBound{c}\<%
\\
\>[0]\AgdaFunction{<-≡}\AgdaSpace{}%
\AgdaSymbol{\{}\AgdaBound{a}\AgdaSymbol{\}}\AgdaSpace{}%
\AgdaBound{e}\AgdaSpace{}%
\AgdaBound{r}\AgdaSpace{}%
\AgdaSymbol{=}\AgdaSpace{}%
\AgdaFunction{subst}\AgdaSpace{}%
\AgdaSymbol{(}\AgdaBound{a}\AgdaSpace{}%
\AgdaOperator{\AgdaDatatype{<\AgdaUnderscore{}}}\AgdaSymbol{)}\AgdaSpace{}%
\AgdaBound{r}\AgdaSpace{}%
\AgdaBound{e}\<%
\\
\>[0]\<%
\end{code}

\newcommand{\psdDes}{
\begin{code}%
\>[0]\AgdaFunction{pseudo-descending}\AgdaSpace{}%
\AgdaSymbol{:}\AgdaSpace{}%
\AgdaSymbol{(}\AgdaDatatype{ℕ}\AgdaSpace{}%
\AgdaSymbol{→}\AgdaSpace{}%
\AgdaDatatype{MutualOrd}\AgdaSymbol{)}\AgdaSpace{}%
\AgdaSymbol{→}\AgdaSpace{}%
\AgdaFunction{Type₀}\<%
\\
\>[0]\AgdaFunction{pseudo-descending}\AgdaSpace{}%
\AgdaBound{f}\AgdaSpace{}%
\AgdaSymbol{=}\<%
\\
\>[0][@{}l@{\AgdaIndent{0}}]%
\>[2]\AgdaSymbol{∀}\AgdaSpace{}%
\AgdaBound{i}\AgdaSpace{}%
\AgdaSymbol{→}\AgdaSpace{}%
\AgdaBound{f}\AgdaSpace{}%
\AgdaBound{i}\AgdaSpace{}%
\AgdaOperator{\AgdaFunction{>}}\AgdaSpace{}%
\AgdaBound{f}\AgdaSpace{}%
\AgdaSymbol{(}\AgdaInductiveConstructor{suc}\AgdaSpace{}%
\AgdaBound{i}\AgdaSymbol{)}\AgdaSpace{}%
\AgdaOperator{\AgdaDatatype{⊎}}\AgdaSpace{}%
\AgdaSymbol{(}\AgdaBound{f}\AgdaSpace{}%
\AgdaBound{i}\AgdaSpace{}%
\AgdaOperator{\AgdaFunction{≡}}\AgdaSpace{}%
\AgdaInductiveConstructor{𝟎}\AgdaSpace{}%
\AgdaOperator{\AgdaFunction{×}}\AgdaSpace{}%
\AgdaBound{f}\AgdaSpace{}%
\AgdaSymbol{(}\AgdaInductiveConstructor{suc}\AgdaSpace{}%
\AgdaBound{i}\AgdaSymbol{)}\AgdaSpace{}%
\AgdaOperator{\AgdaFunction{≡}}\AgdaSpace{}%
\AgdaInductiveConstructor{𝟎}\AgdaSymbol{)}\<%
\end{code}}

\newcommand{\psdDesFacts}{
\begin{code}%
\>[0]\AgdaFunction{zeroPoint}\AgdaSpace{}%
\AgdaSymbol{:}\AgdaSpace{}%
\AgdaSymbol{∀}\AgdaSpace{}%
\AgdaSymbol{\{}\AgdaBound{f}\AgdaSymbol{\}}\AgdaSpace{}%
\AgdaSymbol{→}\AgdaSpace{}%
\AgdaFunction{pseudo-descending}\AgdaSpace{}%
\AgdaBound{f}\<%
\\
\>[0][@{}l@{\AgdaIndent{0}}]%
\>[2]\AgdaSymbol{→}\AgdaSpace{}%
\AgdaSymbol{∀}\AgdaSpace{}%
\AgdaSymbol{\{}\AgdaBound{i}\AgdaSymbol{\}}\AgdaSpace{}%
\AgdaSymbol{→}\AgdaSpace{}%
\AgdaBound{f}\AgdaSpace{}%
\AgdaBound{i}\AgdaSpace{}%
\AgdaOperator{\AgdaFunction{≡}}\AgdaSpace{}%
\AgdaInductiveConstructor{𝟎}\AgdaSpace{}%
\AgdaSymbol{→}\AgdaSpace{}%
\AgdaSymbol{∀}\AgdaSpace{}%
\AgdaBound{j}\AgdaSpace{}%
\AgdaSymbol{→}\AgdaSpace{}%
\AgdaBound{j}\AgdaSpace{}%
\AgdaOperator{\AgdaFunction{≥ᴺ}}\AgdaSpace{}%
\AgdaBound{i}\AgdaSpace{}%
\AgdaSymbol{→}\AgdaSpace{}%
\AgdaBound{f}\AgdaSpace{}%
\AgdaBound{j}\AgdaSpace{}%
\AgdaOperator{\AgdaFunction{≡}}\AgdaSpace{}%
\AgdaInductiveConstructor{𝟎}\<%
\\
\>[0]\AgdaFunction{nonzeroPoint}\AgdaSpace{}%
\AgdaSymbol{:}\AgdaSpace{}%
\AgdaSymbol{∀}\AgdaSpace{}%
\AgdaSymbol{\{}\AgdaBound{f}\AgdaSymbol{\}}\AgdaSpace{}%
\AgdaSymbol{→}\AgdaSpace{}%
\AgdaFunction{pseudo-descending}\AgdaSpace{}%
\AgdaBound{f}\<%
\\
\>[0][@{}l@{\AgdaIndent{0}}]%
\>[2]\AgdaSymbol{→}\AgdaSpace{}%
\AgdaSymbol{∀}\AgdaSpace{}%
\AgdaSymbol{\{}\AgdaBound{i}\AgdaSymbol{\}}\AgdaSpace{}%
\AgdaSymbol{→}\AgdaSpace{}%
\AgdaBound{f}\AgdaSpace{}%
\AgdaBound{i}\AgdaSpace{}%
\AgdaOperator{\AgdaFunction{>}}\AgdaSpace{}%
\AgdaInductiveConstructor{𝟎}\AgdaSpace{}%
\AgdaSymbol{→}\AgdaSpace{}%
\AgdaBound{f}\AgdaSpace{}%
\AgdaBound{i}\AgdaSpace{}%
\AgdaOperator{\AgdaFunction{>}}\AgdaSpace{}%
\AgdaBound{f}\AgdaSpace{}%
\AgdaSymbol{(}\AgdaInductiveConstructor{suc}\AgdaSpace{}%
\AgdaBound{i}\AgdaSymbol{)}\<%
\end{code}}

\begin{code}[hide]%
\>[0]\<%
\\
\>[0]\AgdaFunction{zeroPoint}\AgdaSpace{}%
\AgdaBound{df}\AgdaSpace{}%
\AgdaBound{f0=0}%
\>[20]\AgdaNumber{0}%
\>[27]\AgdaInductiveConstructor{z≤n}%
\>[49]\AgdaSymbol{=}\AgdaSpace{}%
\AgdaBound{f0=0}\<%
\\
\>[0]\AgdaFunction{zeroPoint}\AgdaSpace{}%
\AgdaBound{df}\AgdaSpace{}%
\AgdaBound{f0=0}%
\>[19]\AgdaSymbol{(}\AgdaInductiveConstructor{suc}\AgdaSpace{}%
\AgdaBound{j}\AgdaSymbol{)}\AgdaSpace{}%
\AgdaInductiveConstructor{z≤n}\AgdaSpace{}%
\AgdaKeyword{with}\AgdaSpace{}%
\AgdaBound{df}\AgdaSpace{}%
\AgdaNumber{0}\<%
\\
\>[0]\AgdaFunction{zeroPoint}\AgdaSpace{}%
\AgdaBound{df}\AgdaSpace{}%
\AgdaBound{f0=0}%
\>[19]\AgdaSymbol{(}\AgdaInductiveConstructor{suc}\AgdaSpace{}%
\AgdaBound{j}\AgdaSymbol{)}\AgdaSpace{}%
\AgdaInductiveConstructor{z≤n}\AgdaSpace{}%
\AgdaSymbol{|}\AgdaSpace{}%
\AgdaInductiveConstructor{inj₁}\AgdaSpace{}%
\AgdaBound{f1<f0}%
\>[49]\AgdaSymbol{=}\AgdaSpace{}%
\AgdaFunction{⊥-elim}\AgdaSpace{}%
\AgdaSymbol{(}\AgdaFunction{≮𝟎}\AgdaSpace{}%
\AgdaSymbol{(}\AgdaFunction{<-≡}\AgdaSpace{}%
\AgdaBound{f1<f0}\AgdaSpace{}%
\AgdaBound{f0=0}\AgdaSymbol{))}\<%
\\
\>[0]\AgdaFunction{zeroPoint}\AgdaSpace{}%
\AgdaBound{df}\AgdaSpace{}%
\AgdaBound{f0=0}%
\>[19]\AgdaSymbol{(}\AgdaInductiveConstructor{suc}\AgdaSpace{}%
\AgdaBound{j}\AgdaSymbol{)}\AgdaSpace{}%
\AgdaInductiveConstructor{z≤n}\AgdaSpace{}%
\AgdaSymbol{|}\AgdaSpace{}%
\AgdaInductiveConstructor{inj₂}\AgdaSpace{}%
\AgdaSymbol{(\AgdaUnderscore{}}\AgdaSpace{}%
\AgdaOperator{\AgdaInductiveConstructor{,}}\AgdaSpace{}%
\AgdaBound{f1=0}\AgdaSymbol{)}\AgdaSpace{}%
\AgdaSymbol{=}\AgdaSpace{}%
\AgdaFunction{zeroPoint}\AgdaSpace{}%
\AgdaSymbol{(}\AgdaBound{df}\AgdaSpace{}%
\AgdaOperator{\AgdaFunction{∘}}\AgdaSpace{}%
\AgdaInductiveConstructor{suc}\AgdaSymbol{)}\AgdaSpace{}%
\AgdaBound{f1=0}%
\>[78]\AgdaBound{j}\AgdaSpace{}%
\AgdaInductiveConstructor{z≤n}\<%
\\
\>[0]\AgdaFunction{zeroPoint}\AgdaSpace{}%
\AgdaBound{df}\AgdaSpace{}%
\AgdaBound{fsi=0}\AgdaSpace{}%
\AgdaSymbol{(}\AgdaInductiveConstructor{suc}\AgdaSpace{}%
\AgdaBound{j}\AgdaSymbol{)}\AgdaSpace{}%
\AgdaSymbol{(}\AgdaInductiveConstructor{s≤s}\AgdaSpace{}%
\AgdaBound{i≤j}\AgdaSymbol{)}%
\>[49]\AgdaSymbol{=}\AgdaSpace{}%
\AgdaFunction{zeroPoint}\AgdaSpace{}%
\AgdaSymbol{(}\AgdaBound{df}\AgdaSpace{}%
\AgdaOperator{\AgdaFunction{∘}}\AgdaSpace{}%
\AgdaInductiveConstructor{suc}\AgdaSymbol{)}\AgdaSpace{}%
\AgdaBound{fsi=0}\AgdaSpace{}%
\AgdaBound{j}\AgdaSpace{}%
\AgdaBound{i≤j}\<%
\\
\\[\AgdaEmptyExtraSkip]%
\>[0]\AgdaFunction{nonzeroPoint}\AgdaSpace{}%
\AgdaBound{df}\AgdaSpace{}%
\AgdaBound{fi>0}\AgdaSpace{}%
\AgdaKeyword{with}\AgdaSpace{}%
\AgdaBound{df}\AgdaSpace{}%
\AgdaSymbol{\AgdaUnderscore{}}\<%
\\
\>[0]\AgdaFunction{nonzeroPoint}\AgdaSpace{}%
\AgdaBound{df}\AgdaSpace{}%
\AgdaBound{fi>0}\AgdaSpace{}%
\AgdaSymbol{|}\AgdaSpace{}%
\AgdaInductiveConstructor{inj₁}\AgdaSpace{}%
\AgdaBound{fi+1<fi}%
\>[39]\AgdaSymbol{=}\AgdaSpace{}%
\AgdaBound{fi+1<fi}\<%
\\
\>[0]\AgdaFunction{nonzeroPoint}\AgdaSpace{}%
\AgdaBound{df}\AgdaSpace{}%
\AgdaBound{fi>0}\AgdaSpace{}%
\AgdaSymbol{|}\AgdaSpace{}%
\AgdaInductiveConstructor{inj₂}\AgdaSpace{}%
\AgdaSymbol{(}\AgdaBound{fi=0}\AgdaSpace{}%
\AgdaOperator{\AgdaInductiveConstructor{,}}\AgdaSpace{}%
\AgdaSymbol{\AgdaUnderscore{})}\AgdaSpace{}%
\AgdaSymbol{=}\AgdaSpace{}%
\AgdaFunction{⊥-elim}\AgdaSpace{}%
\AgdaSymbol{(}\AgdaFunction{≮𝟎}\AgdaSpace{}%
\AgdaSymbol{(}\AgdaFunction{<-≡}\AgdaSpace{}%
\AgdaBound{fi>0}\AgdaSpace{}%
\AgdaBound{fi=0}\AgdaSymbol{))}\<%
\\
\>[0]\<%
\end{code}

\newcommand{\evZero}{
\begin{code}%
\>[0]\AgdaFunction{eventually-zero}\AgdaSpace{}%
\AgdaSymbol{:}\AgdaSpace{}%
\AgdaSymbol{(}\AgdaDatatype{ℕ}\AgdaSpace{}%
\AgdaSymbol{→}\AgdaSpace{}%
\AgdaDatatype{MutualOrd}\AgdaSymbol{)}\AgdaSpace{}%
\AgdaSymbol{→}\AgdaSpace{}%
\AgdaFunction{Type₀}\<%
\\
\>[0]\AgdaFunction{eventually-zero}\AgdaSpace{}%
\AgdaBound{f}\AgdaSpace{}%
\AgdaSymbol{=}\AgdaSpace{}%
\AgdaRecord{Σ}\AgdaSpace{}%
\AgdaSymbol{\textbackslash{}(}\AgdaBound{n}\AgdaSpace{}%
\AgdaSymbol{:}\AgdaSpace{}%
\AgdaDatatype{ℕ}\AgdaSymbol{)}\AgdaSpace{}%
\AgdaSymbol{→}\AgdaSpace{}%
\AgdaSymbol{∀}\AgdaSpace{}%
\AgdaBound{i}\AgdaSpace{}%
\AgdaSymbol{→}\AgdaSpace{}%
\AgdaBound{i}\AgdaSpace{}%
\AgdaOperator{\AgdaFunction{≥ᴺ}}\AgdaSpace{}%
\AgdaBound{n}\AgdaSpace{}%
\AgdaSymbol{→}\AgdaSpace{}%
\AgdaBound{f}\AgdaSpace{}%
\AgdaBound{i}\AgdaSpace{}%
\AgdaOperator{\AgdaFunction{≡}}\AgdaSpace{}%
\AgdaInductiveConstructor{𝟎}\<%
\end{code}}

\newcommand{\evZeroFact}{
\begin{code}%
\>[0]\AgdaFunction{eventually-zero-cons}\AgdaSpace{}%
\AgdaSymbol{:}\<%
\\
\>[0][@{}l@{\AgdaIndent{0}}]%
\>[2]\AgdaSymbol{∀}\AgdaSpace{}%
\AgdaBound{f}\AgdaSpace{}%
\AgdaSymbol{→}\AgdaSpace{}%
\AgdaFunction{eventually-zero}\AgdaSpace{}%
\AgdaSymbol{(}\AgdaBound{f}\AgdaSpace{}%
\AgdaOperator{\AgdaFunction{∘}}\AgdaSpace{}%
\AgdaInductiveConstructor{suc}\AgdaSymbol{)}\AgdaSpace{}%
\AgdaSymbol{→}\AgdaSpace{}%
\AgdaFunction{eventually-zero}\AgdaSpace{}%
\AgdaBound{f}\<%
\end{code}}

\begin{code}[hide]%
\>[0]\AgdaFunction{eventually-zero-cons}\AgdaSpace{}%
\AgdaBound{f}\AgdaSpace{}%
\AgdaSymbol{(}\AgdaBound{n}\AgdaSpace{}%
\AgdaOperator{\AgdaInductiveConstructor{,}}\AgdaSpace{}%
\AgdaBound{p}\AgdaSymbol{)}\AgdaSpace{}%
\AgdaSymbol{=}\AgdaSpace{}%
\AgdaInductiveConstructor{suc}\AgdaSpace{}%
\AgdaBound{n}\AgdaSpace{}%
\AgdaOperator{\AgdaInductiveConstructor{,}}\AgdaSpace{}%
\AgdaFunction{p'}\<%
\\
\>[0][@{}l@{\AgdaIndent{0}}]%
\>[1]\AgdaKeyword{where}\<%
\\
\>[1][@{}l@{\AgdaIndent{0}}]%
\>[2]\AgdaFunction{p'}\AgdaSpace{}%
\AgdaSymbol{:}\AgdaSpace{}%
\AgdaSymbol{(}\AgdaBound{i}\AgdaSpace{}%
\AgdaSymbol{:}\AgdaSpace{}%
\AgdaDatatype{ℕ}\AgdaSymbol{)}\AgdaSpace{}%
\AgdaSymbol{→}\AgdaSpace{}%
\AgdaBound{i}\AgdaSpace{}%
\AgdaOperator{\AgdaFunction{≥ᴺ}}\AgdaSpace{}%
\AgdaInductiveConstructor{suc}\AgdaSpace{}%
\AgdaBound{n}\AgdaSpace{}%
\AgdaSymbol{→}\AgdaSpace{}%
\AgdaBound{f}\AgdaSpace{}%
\AgdaBound{i}\AgdaSpace{}%
\AgdaOperator{\AgdaFunction{≡}}\AgdaSpace{}%
\AgdaInductiveConstructor{𝟎}\<%
\\
\>[2]\AgdaFunction{p'}\AgdaSpace{}%
\AgdaSymbol{(}\AgdaInductiveConstructor{suc}\AgdaSpace{}%
\AgdaBound{i}\AgdaSymbol{)}\AgdaSpace{}%
\AgdaSymbol{(}\AgdaInductiveConstructor{s≤s}\AgdaSpace{}%
\AgdaBound{r}\AgdaSymbol{)}\AgdaSpace{}%
\AgdaSymbol{=}\AgdaSpace{}%
\AgdaBound{p}\AgdaSpace{}%
\AgdaBound{i}\AgdaSpace{}%
\AgdaBound{r}\<%
\end{code}

\newcommand{\PDtoEZPred}{
\begin{code}%
\>[0]\AgdaFunction{P}\AgdaSpace{}%
\AgdaSymbol{:}\AgdaSpace{}%
\AgdaDatatype{MutualOrd}\AgdaSpace{}%
\AgdaSymbol{→}\AgdaSpace{}%
\AgdaFunction{Type₀}\<%
\\
\>[0]\AgdaFunction{P}\AgdaSpace{}%
\AgdaBound{a}%
\>[892I]\AgdaSymbol{=}\AgdaSpace{}%
\AgdaSymbol{∀}\AgdaSpace{}%
\AgdaBound{f}\AgdaSpace{}%
\AgdaSymbol{→}\AgdaSpace{}%
\AgdaFunction{pseudo-descending}\AgdaSpace{}%
\AgdaBound{f}\AgdaSpace{}%
\AgdaSymbol{→}\AgdaSpace{}%
\AgdaBound{f}\AgdaSpace{}%
\AgdaNumber{0}\AgdaSpace{}%
\AgdaOperator{\AgdaFunction{≡}}\AgdaSpace{}%
\AgdaBound{a}\<%
\\
\>[.][@{}l@{}]\<[892I]%
\>[4]\AgdaSymbol{→}\AgdaSpace{}%
\AgdaFunction{eventually-zero}\AgdaSpace{}%
\AgdaBound{f}\<%
\end{code}}

\newcommand{\PDtoEZStep}{
\begin{code}%
\>[0]\AgdaFunction{step}\AgdaSpace{}%
\AgdaSymbol{:}\AgdaSpace{}%
\AgdaSymbol{∀}\AgdaSpace{}%
\AgdaBound{x}\AgdaSpace{}%
\AgdaSymbol{→}\AgdaSpace{}%
\AgdaSymbol{(∀}\AgdaSpace{}%
\AgdaBound{y}\AgdaSpace{}%
\AgdaSymbol{→}\AgdaSpace{}%
\AgdaBound{y}\AgdaSpace{}%
\AgdaOperator{\AgdaDatatype{<}}\AgdaSpace{}%
\AgdaBound{x}\AgdaSpace{}%
\AgdaSymbol{→}\AgdaSpace{}%
\AgdaFunction{P}\AgdaSpace{}%
\AgdaBound{y}\AgdaSymbol{)}\AgdaSpace{}%
\AgdaSymbol{→}\AgdaSpace{}%
\AgdaFunction{P}\AgdaSpace{}%
\AgdaBound{x}\<%
\\
\>[0]\AgdaFunction{step}\AgdaSpace{}%
\AgdaBound{x}\AgdaSpace{}%
\AgdaBound{h}\AgdaSpace{}%
\AgdaBound{f}\AgdaSpace{}%
\AgdaBound{df}\AgdaSpace{}%
\AgdaBound{f0=x}\AgdaSpace{}%
\AgdaKeyword{with}\AgdaSpace{}%
\AgdaFunction{≥𝟎}\AgdaSpace{}%
\AgdaSymbol{\{}\AgdaBound{f}\AgdaSpace{}%
\AgdaNumber{0}\AgdaSymbol{\}}\<%
\\
\>[0]\AgdaFunction{step}\AgdaSpace{}%
\AgdaBound{x}\AgdaSpace{}%
\AgdaBound{h}\AgdaSpace{}%
\AgdaBound{f}\AgdaSpace{}%
\AgdaBound{df}\AgdaSpace{}%
\AgdaBound{f0=x}\AgdaSpace{}%
\AgdaSymbol{|}\AgdaSpace{}%
\AgdaInductiveConstructor{inj₁}\AgdaSpace{}%
\AgdaBound{f0>0}\AgdaSpace{}%
\AgdaSymbol{=}\AgdaSpace{}%
\AgdaFunction{goal}\<%
\\
\>[0][@{}l@{\AgdaIndent{0}}]%
\>[1]\AgdaKeyword{where}\<%
\\
\>[1][@{}l@{\AgdaIndent{0}}]%
\>[2]\AgdaFunction{f1<x}\AgdaSpace{}%
\AgdaSymbol{:}\AgdaSpace{}%
\AgdaBound{f}\AgdaSpace{}%
\AgdaNumber{1}\AgdaSpace{}%
\AgdaOperator{\AgdaDatatype{<}}\AgdaSpace{}%
\AgdaBound{x}\<%
\\
\>[2]\AgdaFunction{f1<x}\AgdaSpace{}%
\AgdaSymbol{=}\AgdaSpace{}%
\AgdaFunction{subst}\AgdaSpace{}%
\AgdaSymbol{(}\AgdaBound{f}\AgdaSpace{}%
\AgdaNumber{1}\AgdaSpace{}%
\AgdaOperator{\AgdaDatatype{<\AgdaUnderscore{}}}\AgdaSymbol{)}\AgdaSpace{}%
\AgdaBound{f0=x}\AgdaSpace{}%
\AgdaSymbol{(}\AgdaFunction{nonzeroPoint}\AgdaSpace{}%
\AgdaBound{df}\AgdaSpace{}%
\AgdaBound{f0>0}\AgdaSymbol{)}\<%
\\
\>[2]\AgdaFunction{ezfs}\AgdaSpace{}%
\AgdaSymbol{:}\AgdaSpace{}%
\AgdaFunction{eventually-zero}\AgdaSpace{}%
\AgdaSymbol{(}\AgdaBound{f}\AgdaSpace{}%
\AgdaOperator{\AgdaFunction{∘}}\AgdaSpace{}%
\AgdaInductiveConstructor{suc}\AgdaSymbol{)}\<%
\\
\>[2]\AgdaFunction{ezfs}\AgdaSpace{}%
\AgdaSymbol{=}\AgdaSpace{}%
\AgdaBound{h}\AgdaSpace{}%
\AgdaSymbol{(}\AgdaBound{f}\AgdaSpace{}%
\AgdaNumber{1}\AgdaSymbol{)}\AgdaSpace{}%
\AgdaFunction{f1<x}\AgdaSpace{}%
\AgdaSymbol{(}\AgdaBound{f}\AgdaSpace{}%
\AgdaOperator{\AgdaFunction{∘}}\AgdaSpace{}%
\AgdaInductiveConstructor{suc}\AgdaSymbol{)}\AgdaSpace{}%
\AgdaSymbol{(}\AgdaBound{df}\AgdaSpace{}%
\AgdaOperator{\AgdaFunction{∘}}\AgdaSpace{}%
\AgdaInductiveConstructor{suc}\AgdaSymbol{)}\AgdaSpace{}%
\AgdaFunction{refl}\<%
\\
\>[2]\AgdaFunction{goal}\AgdaSpace{}%
\AgdaSymbol{:}\AgdaSpace{}%
\AgdaFunction{eventually-zero}\AgdaSpace{}%
\AgdaBound{f}\<%
\\
\>[2]\AgdaFunction{goal}\AgdaSpace{}%
\AgdaSymbol{=}\AgdaSpace{}%
\AgdaFunction{eventually-zero-cons}\AgdaSpace{}%
\AgdaBound{f}\AgdaSpace{}%
\AgdaFunction{ezfs}\<%
\\
\>[0]\AgdaFunction{step}\AgdaSpace{}%
\AgdaBound{x}\AgdaSpace{}%
\AgdaBound{h}\AgdaSpace{}%
\AgdaBound{f}\AgdaSpace{}%
\AgdaBound{df}\AgdaSpace{}%
\AgdaBound{f0=x}\AgdaSpace{}%
\AgdaSymbol{|}\AgdaSpace{}%
\AgdaInductiveConstructor{inj₂}\AgdaSpace{}%
\AgdaBound{f0=0}\AgdaSpace{}%
\AgdaSymbol{=}\AgdaSpace{}%
\AgdaFunction{goal}\<%
\\
\>[0][@{}l@{\AgdaIndent{0}}]%
\>[1]\AgdaKeyword{where}\<%
\\
\>[1][@{}l@{\AgdaIndent{0}}]%
\>[2]\AgdaFunction{fi=0}\AgdaSpace{}%
\AgdaSymbol{:}\AgdaSpace{}%
\AgdaSymbol{∀}\AgdaSpace{}%
\AgdaBound{i}\AgdaSpace{}%
\AgdaSymbol{→}\AgdaSpace{}%
\AgdaBound{f}\AgdaSpace{}%
\AgdaBound{i}\AgdaSpace{}%
\AgdaOperator{\AgdaFunction{≡}}\AgdaSpace{}%
\AgdaInductiveConstructor{𝟎}\<%
\\
\>[2]\AgdaFunction{fi=0}\AgdaSpace{}%
\AgdaBound{i}\AgdaSpace{}%
\AgdaSymbol{=}\AgdaSpace{}%
\AgdaFunction{zeroPoint}\AgdaSpace{}%
\AgdaBound{df}\AgdaSpace{}%
\AgdaBound{f0=0}\AgdaSpace{}%
\AgdaBound{i}\AgdaSpace{}%
\AgdaInductiveConstructor{z≤n}\<%
\\
\>[2]\AgdaFunction{goal}\AgdaSpace{}%
\AgdaSymbol{:}\AgdaSpace{}%
\AgdaFunction{eventually-zero}\AgdaSpace{}%
\AgdaBound{f}\<%
\\
\>[2]\AgdaFunction{goal}\AgdaSpace{}%
\AgdaSymbol{=}\AgdaSpace{}%
\AgdaNumber{0}\AgdaSpace{}%
\AgdaOperator{\AgdaInductiveConstructor{,}}\AgdaSpace{}%
\AgdaSymbol{λ}\AgdaSpace{}%
\AgdaBound{i}\AgdaSpace{}%
\AgdaBound{\AgdaUnderscore{}}\AgdaSpace{}%
\AgdaSymbol{→}\AgdaSpace{}%
\AgdaFunction{fi=0}\AgdaSpace{}%
\AgdaBound{i}\<%
\end{code}}

\begin{code}[hide]%
\>[0]\AgdaFunction{PD2EZ}\AgdaSpace{}%
\AgdaSymbol{:}\<%
\end{code}

\newcommand{\PDtoEZ}{
\begin{code}[inline]%
\>[0][@{}l@{\AgdaIndent{1}}]%
\>[1]\AgdaSymbol{∀}\AgdaSpace{}%
\AgdaBound{f}\AgdaSpace{}%
\AgdaSymbol{→}\AgdaSpace{}%
\AgdaFunction{pseudo-descending}\AgdaSpace{}%
\AgdaBound{f}\AgdaSpace{}%
\AgdaSymbol{→}\AgdaSpace{}%
\AgdaFunction{eventually-zero}\AgdaSpace{}%
\AgdaBound{f}\<%
\end{code}}

\newcommand{\PDtoEZProof}{
\begin{code}[inline]%
\>[0]\AgdaFunction{PD2EZ}\AgdaSpace{}%
\AgdaBound{f}\AgdaSpace{}%
\AgdaBound{df}\AgdaSpace{}%
\AgdaSymbol{=}\AgdaSpace{}%
\AgdaFunction{MTI}\AgdaSpace{}%
\AgdaFunction{P}\AgdaSpace{}%
\AgdaFunction{step}\AgdaSpace{}%
\AgdaSymbol{(}\AgdaBound{f}\AgdaSpace{}%
\AgdaNumber{0}\AgdaSymbol{)}\AgdaSpace{}%
\AgdaBound{f}\AgdaSpace{}%
\AgdaBound{df}\AgdaSpace{}%
\AgdaFunction{refl}\<%
\end{code}}

\newcommand{\strDes}{
\begin{code}%
\>[0]\AgdaFunction{strictly-descending}\AgdaSpace{}%
\AgdaSymbol{:}\AgdaSpace{}%
\AgdaSymbol{(}\AgdaDatatype{ℕ}\AgdaSpace{}%
\AgdaSymbol{→}\AgdaSpace{}%
\AgdaDatatype{MutualOrd}\AgdaSymbol{)}\AgdaSpace{}%
\AgdaSymbol{→}\AgdaSpace{}%
\AgdaFunction{Type₀}\<%
\\
\>[0]\AgdaFunction{strictly-descending}\AgdaSpace{}%
\AgdaBound{f}\AgdaSpace{}%
\AgdaSymbol{=}\AgdaSpace{}%
\AgdaSymbol{∀}\AgdaSpace{}%
\AgdaBound{i}\AgdaSpace{}%
\AgdaSymbol{→}\AgdaSpace{}%
\AgdaBound{f}\AgdaSpace{}%
\AgdaBound{i}\AgdaSpace{}%
\AgdaOperator{\AgdaFunction{>}}\AgdaSpace{}%
\AgdaBound{f}\AgdaSpace{}%
\AgdaSymbol{(}\AgdaInductiveConstructor{suc}\AgdaSpace{}%
\AgdaBound{i}\AgdaSymbol{)}\<%
\end{code}}

\newcommand{\NSDS}{
\begin{code}[inline]%
\>[0]\AgdaFunction{NSDS}\AgdaSpace{}%
\AgdaSymbol{:}\AgdaSpace{}%
\AgdaSymbol{∀}\AgdaSpace{}%
\AgdaBound{f}\AgdaSpace{}%
\AgdaSymbol{→}\AgdaSpace{}%
\AgdaFunction{strictly-descending}\AgdaSpace{}%
\AgdaBound{f}\AgdaSpace{}%
\AgdaSymbol{→}\AgdaSpace{}%
\AgdaDatatype{⊥}\<%
\end{code}}

\newcommand{\HITOrdLt}{
\begin{code}%
\>[0]\AgdaOperator{\AgdaFunction{\AgdaUnderscore{}<ᴴ\AgdaUnderscore{}}}\AgdaSpace{}%
\AgdaSymbol{:}\AgdaSpace{}%
\AgdaDatatype{HITOrd}\AgdaSpace{}%
\AgdaSymbol{→}\AgdaSpace{}%
\AgdaDatatype{HITOrd}\AgdaSpace{}%
\AgdaSymbol{→}\AgdaSpace{}%
\AgdaFunction{Type₀}\<%
\\
\>[0]\AgdaOperator{\AgdaFunction{\AgdaUnderscore{}<ᴴ\AgdaUnderscore{}}}\AgdaSpace{}%
\AgdaSymbol{=}\AgdaSpace{}%
\AgdaFunction{transport}\AgdaSpace{}%
\AgdaSymbol{(λ}\AgdaSpace{}%
\AgdaBound{i}\AgdaSpace{}%
\AgdaSymbol{→}\AgdaSpace{}%
\AgdaFunction{M≡H}\AgdaSpace{}%
\AgdaBound{i}\AgdaSpace{}%
\AgdaSymbol{→}\AgdaSpace{}%
\AgdaFunction{M≡H}\AgdaSpace{}%
\AgdaBound{i}\AgdaSpace{}%
\AgdaSymbol{→}\AgdaSpace{}%
\AgdaFunction{Type₀}\AgdaSymbol{)}\AgdaSpace{}%
\AgdaOperator{\AgdaDatatype{\AgdaUnderscore{}<\AgdaUnderscore{}}}\<%
\end{code}}

\newcommand{\HITOrdLtPath}{
\begin{code}%
\>[0]\AgdaFunction{<Path}\AgdaSpace{}%
\AgdaSymbol{:}\AgdaSpace{}%
\AgdaPostulate{PathP}\AgdaSpace{}%
\AgdaSymbol{(λ}\AgdaSpace{}%
\AgdaBound{i}\AgdaSpace{}%
\AgdaSymbol{→}\AgdaSpace{}%
\AgdaFunction{M≡H}\AgdaSpace{}%
\AgdaBound{i}\AgdaSpace{}%
\AgdaSymbol{→}\AgdaSpace{}%
\AgdaFunction{M≡H}\AgdaSpace{}%
\AgdaBound{i}\AgdaSpace{}%
\AgdaSymbol{→}\AgdaSpace{}%
\AgdaFunction{Type₀}\AgdaSymbol{)}\AgdaSpace{}%
\AgdaOperator{\AgdaDatatype{\AgdaUnderscore{}<\AgdaUnderscore{}}}\AgdaSpace{}%
\AgdaOperator{\AgdaFunction{\AgdaUnderscore{}<ᴴ\AgdaUnderscore{}}}\<%
\end{code}}

\begin{code}[hide]%
\>[0]\AgdaFunction{<Path}\AgdaSpace{}%
\AgdaBound{i}\AgdaSpace{}%
\AgdaSymbol{=}\AgdaSpace{}%
\AgdaPrimitive{transp}\AgdaSpace{}%
\AgdaSymbol{(λ}\AgdaSpace{}%
\AgdaBound{j}\AgdaSpace{}%
\AgdaSymbol{→}\AgdaSpace{}%
\AgdaFunction{M≡H}\AgdaSpace{}%
\AgdaSymbol{(}\AgdaBound{i}\AgdaSpace{}%
\AgdaOperator{\AgdaPrimitive{⊓}}\AgdaSpace{}%
\AgdaBound{j}\AgdaSymbol{)}\AgdaSpace{}%
\AgdaSymbol{→}\AgdaSpace{}%
\AgdaFunction{M≡H}\AgdaSpace{}%
\AgdaSymbol{(}\AgdaBound{i}\AgdaSpace{}%
\AgdaOperator{\AgdaPrimitive{⊓}}\AgdaSpace{}%
\AgdaBound{j}\AgdaSymbol{)}\AgdaSpace{}%
\AgdaSymbol{→}\AgdaSpace{}%
\AgdaFunction{Type₀}\AgdaSymbol{)}\AgdaSpace{}%
\AgdaSymbol{(}\AgdaOperator{\AgdaPrimitive{\textasciitilde{}}}\AgdaSpace{}%
\AgdaBound{i}\AgdaSymbol{)}\AgdaSpace{}%
\AgdaOperator{\AgdaDatatype{\AgdaUnderscore{}<\AgdaUnderscore{}}}\<%
\end{code}

\newcommand{\HTI}{
\begin{code}[inline]%
\>[0]\AgdaFunction{HTI}\AgdaSpace{}%
\AgdaSymbol{:}\AgdaSpace{}%
\AgdaFunction{TI}\AgdaSpace{}%
\AgdaDatatype{HITOrd}\AgdaSpace{}%
\AgdaOperator{\AgdaFunction{\AgdaUnderscore{}<ᴴ\AgdaUnderscore{}}}\AgdaSpace{}%
\AgdaGeneralizable{ℓ}\<%
\end{code}}

\newcommand{\DEC}{
\begin{code}%
\>[0]\AgdaFunction{Dec}\AgdaSpace{}%
\AgdaSymbol{:}\AgdaSpace{}%
\AgdaSymbol{(}\AgdaBound{A}\AgdaSpace{}%
\AgdaSymbol{:}\AgdaSpace{}%
\AgdaFunction{Type}\AgdaSpace{}%
\AgdaGeneralizable{ℓ}\AgdaSymbol{)}\AgdaSpace{}%
\AgdaSymbol{→}\AgdaSpace{}%
\AgdaSymbol{(}\AgdaBound{A}\AgdaSpace{}%
\AgdaSymbol{→}\AgdaSpace{}%
\AgdaBound{A}\AgdaSpace{}%
\AgdaSymbol{→}\AgdaSpace{}%
\AgdaFunction{Type}\AgdaSpace{}%
\AgdaGeneralizable{ℓ'}\AgdaSymbol{)}\AgdaSpace{}%
\AgdaSymbol{→}\AgdaSpace{}%
\AgdaFunction{Type}\AgdaSpace{}%
\AgdaSymbol{(}\AgdaGeneralizable{ℓ}\AgdaSpace{}%
\AgdaOperator{\AgdaPrimitive{⊔}}\AgdaSpace{}%
\AgdaGeneralizable{ℓ'}\AgdaSymbol{)}\<%
\\
\>[0]\AgdaFunction{Dec}\AgdaSpace{}%
\AgdaBound{A}\AgdaSpace{}%
\AgdaOperator{\AgdaBound{\AgdaUnderscore{}<\AgdaUnderscore{}}}\AgdaSpace{}%
\AgdaSymbol{=}\AgdaSpace{}%
\AgdaSymbol{(}\AgdaBound{x}\AgdaSpace{}%
\AgdaBound{y}\AgdaSpace{}%
\AgdaSymbol{:}\AgdaSpace{}%
\AgdaBound{A}\AgdaSymbol{)}\AgdaSpace{}%
\AgdaSymbol{→}\AgdaSpace{}%
\AgdaBound{x}\AgdaSpace{}%
\AgdaOperator{\AgdaBound{<}}\AgdaSpace{}%
\AgdaBound{y}\AgdaSpace{}%
\AgdaOperator{\AgdaDatatype{⊎}}\AgdaSpace{}%
\AgdaOperator{\AgdaFunction{¬}}\AgdaSpace{}%
\AgdaBound{x}\AgdaSpace{}%
\AgdaOperator{\AgdaBound{<}}\AgdaSpace{}%
\AgdaBound{y}\<%
\end{code}}

\newcommand{\MOLtDec}{
\begin{code}%
\>[0]\AgdaFunction{<-dec}\AgdaSpace{}%
\AgdaSymbol{:}\AgdaSpace{}%
\AgdaFunction{Dec}\AgdaSpace{}%
\AgdaDatatype{MutualOrd}\AgdaSpace{}%
\AgdaOperator{\AgdaDatatype{\AgdaUnderscore{}<\AgdaUnderscore{}}}\<%
\end{code}}

\newcommand{\HOLtDec}{
\begin{code}%
\>[0]\AgdaFunction{<ᴴ-dec}\AgdaSpace{}%
\AgdaSymbol{:}\AgdaSpace{}%
\AgdaFunction{Dec}\AgdaSpace{}%
\AgdaDatatype{HITOrd}\AgdaSpace{}%
\AgdaOperator{\AgdaFunction{\AgdaUnderscore{}<ᴴ\AgdaUnderscore{}}}\<%
\\
\>[0]\AgdaFunction{<ᴴ-dec}\AgdaSpace{}%
\AgdaSymbol{=}\AgdaSpace{}%
\AgdaFunction{transport}\AgdaSpace{}%
\AgdaSymbol{(λ}\AgdaSpace{}%
\AgdaBound{i}\AgdaSpace{}%
\AgdaSymbol{→}\AgdaSpace{}%
\AgdaFunction{Dec}\AgdaSpace{}%
\AgdaSymbol{(}\AgdaFunction{M≡H}\AgdaSpace{}%
\AgdaBound{i}\AgdaSymbol{)}\AgdaSpace{}%
\AgdaSymbol{(}\AgdaFunction{<Path}\AgdaSpace{}%
\AgdaBound{i}\AgdaSymbol{))}\AgdaSpace{}%
\AgdaFunction{<-dec}\<%
\end{code}}

\begin{code}[hide]%
\>[0]\<%
\\
\>[0]\AgdaFunction{<-dec}\AgdaSpace{}%
\AgdaSymbol{=}\AgdaSpace{}%
\AgdaFunction{<-dec'}\<%
\\
\\[\AgdaEmptyExtraSkip]%
\>[0]\AgdaKeyword{data}\AgdaSpace{}%
\AgdaDatatype{Bool}\AgdaSpace{}%
\AgdaSymbol{:}\AgdaSpace{}%
\AgdaFunction{Type₀}\AgdaSpace{}%
\AgdaKeyword{where}\<%
\\
\>[0][@{}l@{\AgdaIndent{0}}]%
\>[1]\AgdaInductiveConstructor{true}\AgdaSpace{}%
\AgdaInductiveConstructor{false}\AgdaSpace{}%
\AgdaSymbol{:}\AgdaSpace{}%
\AgdaDatatype{Bool}\<%
\\
\\[\AgdaEmptyExtraSkip]%
\>[0]\AgdaFunction{isLeft}%
\>[1252I]\AgdaSymbol{:}\AgdaSpace{}%
\AgdaSymbol{\{}\AgdaBound{A}\AgdaSpace{}%
\AgdaSymbol{:}\AgdaSpace{}%
\AgdaFunction{Type}\AgdaSpace{}%
\AgdaGeneralizable{ℓ}\AgdaSymbol{\}}\AgdaSpace{}%
\AgdaSymbol{\{}\AgdaBound{B}\AgdaSpace{}%
\AgdaSymbol{:}\AgdaSpace{}%
\AgdaFunction{Type}\AgdaSpace{}%
\AgdaGeneralizable{ℓ'}\AgdaSymbol{\}}\<%
\\
\>[.][@{}l@{}]\<[1252I]%
\>[7]\AgdaSymbol{→}\AgdaSpace{}%
\AgdaBound{A}\AgdaSpace{}%
\AgdaOperator{\AgdaDatatype{⊎}}\AgdaSpace{}%
\AgdaBound{B}\AgdaSpace{}%
\AgdaSymbol{→}\AgdaSpace{}%
\AgdaDatatype{Bool}\<%
\\
\>[0]\AgdaFunction{isLeft}\AgdaSpace{}%
\AgdaSymbol{(}\AgdaInductiveConstructor{inj₁}\AgdaSpace{}%
\AgdaSymbol{\AgdaUnderscore{})}\AgdaSpace{}%
\AgdaSymbol{=}\AgdaSpace{}%
\AgdaInductiveConstructor{true}\<%
\\
\>[0]\AgdaFunction{isLeft}\AgdaSpace{}%
\AgdaSymbol{(}\AgdaInductiveConstructor{inj₂}\AgdaSpace{}%
\AgdaSymbol{\AgdaUnderscore{})}\AgdaSpace{}%
\AgdaSymbol{=}\AgdaSpace{}%
\AgdaInductiveConstructor{false}\<%
\\
\>[0]\<%
\end{code}

\newcommand{\HOLtBool}{
\begin{code}%
\>[0]\AgdaFunction{lt}\AgdaSpace{}%
\AgdaSymbol{:}\AgdaSpace{}%
\AgdaDatatype{HITOrd}\AgdaSpace{}%
\AgdaSymbol{→}\AgdaSpace{}%
\AgdaDatatype{HITOrd}\AgdaSpace{}%
\AgdaSymbol{→}\AgdaSpace{}%
\AgdaDatatype{Bool}\<%
\\
\>[0]\AgdaFunction{lt}\AgdaSpace{}%
\AgdaBound{a}\AgdaSpace{}%
\AgdaBound{b}\AgdaSpace{}%
\AgdaSymbol{=}\AgdaSpace{}%
\AgdaFunction{isLeft}\AgdaSpace{}%
\AgdaSymbol{(}\AgdaFunction{<ᴴ-dec}\AgdaSpace{}%
\AgdaBound{a}\AgdaSpace{}%
\AgdaBound{b}\AgdaSymbol{)}\<%
\end{code}}

\begin{code}[hide]%
\>[0]\AgdaKeyword{open}\AgdaSpace{}%
\AgdaKeyword{import}\AgdaSpace{}%
\AgdaModule{Arithmetic}\<%
\end{code}

\newcommand{\ExLtComp}{
\begin{code}%
\>[0]\AgdaFunction{Ex[<ᴴ-decComp]}\AgdaSpace{}%
\AgdaSymbol{:}\<%
\\
\>[0][@{}l@{\AgdaIndent{0}}]%
\>[4]\AgdaFunction{lt}\AgdaSpace{}%
\AgdaInductiveConstructor{𝟎}\AgdaSpace{}%
\AgdaInductiveConstructor{𝟎}\AgdaSpace{}%
\AgdaOperator{\AgdaFunction{≡}}\AgdaSpace{}%
\AgdaInductiveConstructor{false}\<%
\\
\>[0][@{}l@{\AgdaIndent{0}}]%
\>[1]\AgdaOperator{\AgdaFunction{×}}%
\>[4]\AgdaFunction{lt}\AgdaSpace{}%
\AgdaFunction{H.ω}\AgdaSpace{}%
\AgdaSymbol{((}\AgdaFunction{H.𝟏}\AgdaSpace{}%
\AgdaOperator{\AgdaFunction{⊕}}\AgdaSpace{}%
\AgdaFunction{H.𝟏}\AgdaSymbol{)}\AgdaSpace{}%
\AgdaOperator{\AgdaFunction{⊗}}\AgdaSpace{}%
\AgdaFunction{H.ω}\AgdaSymbol{)}\AgdaSpace{}%
\AgdaOperator{\AgdaFunction{≡}}\AgdaSpace{}%
\AgdaInductiveConstructor{true}\<%
\\
\>[1]\AgdaOperator{\AgdaFunction{×}}%
\>[4]\AgdaFunction{lt}\AgdaSpace{}%
\AgdaSymbol{(}\AgdaOperator{\AgdaFunction{H.ω\textasciicircum{}⟨}}\AgdaSpace{}%
\AgdaFunction{H.ω}\AgdaSpace{}%
\AgdaOperator{\AgdaFunction{⟩}}\AgdaSymbol{)}\AgdaSpace{}%
\AgdaSymbol{(}\AgdaOperator{\AgdaFunction{H.ω\textasciicircum{}⟨}}\AgdaSpace{}%
\AgdaFunction{H.𝟏}\AgdaSpace{}%
\AgdaOperator{\AgdaFunction{+ᴴ}}\AgdaSpace{}%
\AgdaFunction{H.ω}\AgdaSpace{}%
\AgdaOperator{\AgdaFunction{⟩}}\AgdaSymbol{)}\AgdaSpace{}%
\AgdaOperator{\AgdaFunction{≡}}\AgdaSpace{}%
\AgdaInductiveConstructor{false}\<%
\\
\>[1]\AgdaOperator{\AgdaFunction{×}}%
\>[4]\AgdaFunction{lt}\AgdaSpace{}%
\AgdaSymbol{(}\AgdaOperator{\AgdaFunction{H.ω\textasciicircum{}⟨}}\AgdaSpace{}%
\AgdaFunction{H.ω}\AgdaSpace{}%
\AgdaOperator{\AgdaFunction{⟩}}\AgdaSymbol{)}\AgdaSpace{}%
\AgdaSymbol{(}\AgdaOperator{\AgdaFunction{H.ω\textasciicircum{}⟨}}\AgdaSpace{}%
\AgdaFunction{H.𝟏}\AgdaSpace{}%
\AgdaOperator{\AgdaFunction{⊕}}\AgdaSpace{}%
\AgdaFunction{H.ω}\AgdaSpace{}%
\AgdaOperator{\AgdaFunction{⟩}}\AgdaSymbol{)}\AgdaSpace{}%
\AgdaOperator{\AgdaFunction{≡}}\AgdaSpace{}%
\AgdaInductiveConstructor{true}\<%
\\
\>[0]\AgdaFunction{Ex[<ᴴ-decComp]}\AgdaSpace{}%
\AgdaSymbol{=}\AgdaSpace{}%
\AgdaSymbol{(}\AgdaFunction{refl}\AgdaSpace{}%
\AgdaOperator{\AgdaInductiveConstructor{,}}\AgdaSpace{}%
\AgdaFunction{refl}\AgdaSpace{}%
\AgdaOperator{\AgdaInductiveConstructor{,}}\AgdaSpace{}%
\AgdaFunction{refl}\AgdaSpace{}%
\AgdaOperator{\AgdaInductiveConstructor{,}}\AgdaSpace{}%
\AgdaFunction{refl}\AgdaSymbol{)}\<%
\end{code}}

\begin{code}[hide]%
\>[0]\<%
\\
\>[0]\AgdaSymbol{\{-\#}\AgdaSpace{}%
\AgdaKeyword{OPTIONS}\AgdaSpace{}%
\AgdaPragma{--cubical}%
\>[23]\AgdaSymbol{\#-\}}\<%
\\
\\[\AgdaEmptyExtraSkip]%
\>[0]\AgdaKeyword{module}\AgdaSpace{}%
\AgdaModule{Arithmetic}\AgdaSpace{}%
\AgdaKeyword{where}\<%
\\
\\[\AgdaEmptyExtraSkip]%
\>[0]\AgdaKeyword{open}\AgdaSpace{}%
\AgdaKeyword{import}\AgdaSpace{}%
\AgdaModule{Cubical.Foundations.Everything}\<%
\\
\>[0][@{}l@{\AgdaIndent{0}}]%
\>[1]\AgdaKeyword{using}%
\>[6I]\AgdaSymbol{(}\AgdaFunction{Type}\AgdaSpace{}%
\AgdaSymbol{;}\AgdaSpace{}%
\AgdaFunction{Type₀}\AgdaSpace{}%
\AgdaSymbol{;}\<%
\\
\>[6I][@{}l@{\AgdaIndent{0}}]%
\>[8]\AgdaOperator{\AgdaFunction{\AgdaUnderscore{}≡\AgdaUnderscore{}}}\AgdaSpace{}%
\AgdaSymbol{;}\AgdaSpace{}%
\AgdaFunction{refl}\AgdaSpace{}%
\AgdaSymbol{;}\AgdaSpace{}%
\AgdaFunction{cong}\AgdaSpace{}%
\AgdaSymbol{;}\AgdaSpace{}%
\AgdaFunction{cong₂}\AgdaSpace{}%
\AgdaSymbol{;}\AgdaSpace{}%
\AgdaOperator{\AgdaFunction{\AgdaUnderscore{}∙\AgdaUnderscore{}}}\AgdaSpace{}%
\AgdaSymbol{;}\AgdaSpace{}%
\AgdaOperator{\AgdaFunction{\AgdaUnderscore{}⁻¹}}\AgdaSpace{}%
\AgdaSymbol{;}\AgdaSpace{}%
\AgdaFunction{subst}\AgdaSpace{}%
\AgdaSymbol{;}\<%
\\
\>[8]\AgdaPostulate{PathP}\AgdaSpace{}%
\AgdaSymbol{;}\AgdaSpace{}%
\AgdaFunction{isOfHLevel→isOfHLevelDep}\AgdaSpace{}%
\AgdaSymbol{;}\AgdaSpace{}%
\AgdaFunction{isProp→isSet}\AgdaSpace{}%
\AgdaSymbol{;}\AgdaSpace{}%
\AgdaFunction{toPathP}\AgdaSpace{}%
\AgdaSymbol{;}\<%
\\
\>[8]\AgdaFunction{transport}\AgdaSpace{}%
\AgdaSymbol{;}\AgdaSpace{}%
\AgdaPrimitive{transp}\AgdaSpace{}%
\AgdaSymbol{;}\AgdaSpace{}%
\AgdaPrimitive{\textasciitilde{}\AgdaUnderscore{}}\AgdaSpace{}%
\AgdaSymbol{;}\AgdaSpace{}%
\AgdaPrimitive{\AgdaUnderscore{}∧\AgdaUnderscore{}}\AgdaSymbol{)}\<%
\\
\\[\AgdaEmptyExtraSkip]%
\>[0]\AgdaKeyword{open}\AgdaSpace{}%
\AgdaKeyword{import}\AgdaSpace{}%
\AgdaModule{Cubical.Data.Empty}\<%
\\
\\[\AgdaEmptyExtraSkip]%
\>[0]\AgdaKeyword{open}\AgdaSpace{}%
\AgdaKeyword{import}\AgdaSpace{}%
\AgdaModule{Preliminaries}\<%
\\
\>[0]\AgdaKeyword{open}\AgdaSpace{}%
\AgdaKeyword{import}\AgdaSpace{}%
\AgdaModule{HITOrd}\AgdaSpace{}%
\AgdaSymbol{as}\AgdaSpace{}%
\AgdaModule{H}\<%
\\
\>[0]\AgdaKeyword{open}\AgdaSpace{}%
\AgdaKeyword{import}\AgdaSpace{}%
\AgdaModule{MutualOrd}\AgdaSpace{}%
\AgdaSymbol{as}\AgdaSpace{}%
\AgdaModule{M}\<%
\\
\>[0]\AgdaKeyword{open}\AgdaSpace{}%
\AgdaKeyword{import}\AgdaSpace{}%
\AgdaModule{Equivalences}\<%
\\
\\[\AgdaEmptyExtraSkip]%
\>[0]\AgdaKeyword{infixr}\AgdaSpace{}%
\AgdaNumber{35}\AgdaSpace{}%
\AgdaOperator{\AgdaFunction{\AgdaUnderscore{}⊕\AgdaUnderscore{}}}\AgdaSpace{}%
\AgdaOperator{\AgdaFunction{\AgdaUnderscore{}⊕ᴹ\AgdaUnderscore{}}}\AgdaSpace{}%
\AgdaOperator{\AgdaFunction{\AgdaUnderscore{}+\AgdaUnderscore{}}}\AgdaSpace{}%
\AgdaOperator{\AgdaFunction{\AgdaUnderscore{}+ᴴ\AgdaUnderscore{}}}\<%
\\
\>[0]\<%
\end{code}

\newcommand{\AssocType}{
\begin{code}%
\>[0]\AgdaFunction{Assoc}\AgdaSpace{}%
\AgdaSymbol{:}\AgdaSpace{}%
\AgdaSymbol{(}\AgdaBound{A}\AgdaSpace{}%
\AgdaSymbol{:}\AgdaSpace{}%
\AgdaFunction{Type₀}\AgdaSymbol{)}\AgdaSpace{}%
\AgdaSymbol{→}\AgdaSpace{}%
\AgdaSymbol{(}\AgdaBound{A}\AgdaSpace{}%
\AgdaSymbol{→}\AgdaSpace{}%
\AgdaBound{A}\AgdaSpace{}%
\AgdaSymbol{→}\AgdaSpace{}%
\AgdaBound{A}\AgdaSymbol{)}\AgdaSpace{}%
\AgdaSymbol{→}\AgdaSpace{}%
\AgdaFunction{Type₀}\<%
\\
\>[0]\AgdaFunction{Assoc}\AgdaSpace{}%
\AgdaBound{A}\AgdaSpace{}%
\AgdaOperator{\AgdaBound{\AgdaUnderscore{}⋆\AgdaUnderscore{}}}\AgdaSpace{}%
\AgdaSymbol{=}\AgdaSpace{}%
\AgdaSymbol{∀}\AgdaSpace{}%
\AgdaBound{a}\AgdaSpace{}%
\AgdaBound{b}\AgdaSpace{}%
\AgdaBound{c}\AgdaSpace{}%
\AgdaSymbol{→}\AgdaSpace{}%
\AgdaBound{a}\AgdaSpace{}%
\AgdaOperator{\AgdaBound{⋆}}\AgdaSpace{}%
\AgdaSymbol{(}\AgdaBound{b}\AgdaSpace{}%
\AgdaOperator{\AgdaBound{⋆}}\AgdaSpace{}%
\AgdaBound{c}\AgdaSymbol{)}\AgdaSpace{}%
\AgdaOperator{\AgdaFunction{≡}}\AgdaSpace{}%
\AgdaSymbol{(}\AgdaBound{a}\AgdaSpace{}%
\AgdaOperator{\AgdaBound{⋆}}\AgdaSpace{}%
\AgdaBound{b}\AgdaSymbol{)}\AgdaSpace{}%
\AgdaOperator{\AgdaBound{⋆}}\AgdaSpace{}%
\AgdaBound{c}\<%
\end{code}}

\newcommand{\CommType}{
\begin{code}%
\>[0]\AgdaFunction{Comm}\AgdaSpace{}%
\AgdaSymbol{:}\AgdaSpace{}%
\AgdaSymbol{(}\AgdaBound{A}\AgdaSpace{}%
\AgdaSymbol{:}\AgdaSpace{}%
\AgdaFunction{Type₀}\AgdaSymbol{)}\AgdaSpace{}%
\AgdaSymbol{→}\AgdaSpace{}%
\AgdaSymbol{(}\AgdaBound{A}\AgdaSpace{}%
\AgdaSymbol{→}\AgdaSpace{}%
\AgdaBound{A}\AgdaSpace{}%
\AgdaSymbol{→}\AgdaSpace{}%
\AgdaBound{A}\AgdaSymbol{)}\AgdaSpace{}%
\AgdaSymbol{→}\AgdaSpace{}%
\AgdaFunction{Type₀}\<%
\\
\>[0]\AgdaFunction{Comm}\AgdaSpace{}%
\AgdaBound{A}\AgdaSpace{}%
\AgdaOperator{\AgdaBound{\AgdaUnderscore{}⋆\AgdaUnderscore{}}}\AgdaSpace{}%
\AgdaSymbol{=}\AgdaSpace{}%
\AgdaSymbol{∀}\AgdaSpace{}%
\AgdaBound{a}\AgdaSpace{}%
\AgdaBound{b}\AgdaSpace{}%
\AgdaSymbol{→}\AgdaSpace{}%
\AgdaBound{a}\AgdaSpace{}%
\AgdaOperator{\AgdaBound{⋆}}\AgdaSpace{}%
\AgdaBound{b}\AgdaSpace{}%
\AgdaOperator{\AgdaFunction{≡}}\AgdaSpace{}%
\AgdaBound{b}\AgdaSpace{}%
\AgdaOperator{\AgdaBound{⋆}}\AgdaSpace{}%
\AgdaBound{a}\<%
\end{code}}

\newcommand{\Hsum}{
\begin{code}%
\>[0]\AgdaOperator{\AgdaFunction{\AgdaUnderscore{}⊕\AgdaUnderscore{}}}\AgdaSpace{}%
\AgdaSymbol{:}\AgdaSpace{}%
\AgdaDatatype{HITOrd}\AgdaSpace{}%
\AgdaSymbol{→}\AgdaSpace{}%
\AgdaDatatype{HITOrd}\AgdaSpace{}%
\AgdaSymbol{→}\AgdaSpace{}%
\AgdaDatatype{HITOrd}\<%
\\
\>[0]\AgdaInductiveConstructor{𝟎}%
\>[17]\AgdaOperator{\AgdaFunction{⊕}}\AgdaSpace{}%
\AgdaBound{y}\AgdaSpace{}%
\AgdaSymbol{=}\AgdaSpace{}%
\AgdaBound{y}\<%
\\
\>[0]\AgdaSymbol{(}\AgdaOperator{\AgdaInductiveConstructor{ω\textasciicircum{}}}\AgdaSpace{}%
\AgdaBound{a}\AgdaSpace{}%
\AgdaOperator{\AgdaInductiveConstructor{⊕}}\AgdaSpace{}%
\AgdaBound{b}\AgdaSymbol{)}%
\>[17]\AgdaOperator{\AgdaFunction{⊕}}\AgdaSpace{}%
\AgdaBound{y}\AgdaSpace{}%
\AgdaSymbol{=}\AgdaSpace{}%
\AgdaOperator{\AgdaInductiveConstructor{ω\textasciicircum{}}}\AgdaSpace{}%
\AgdaBound{a}\AgdaSpace{}%
\AgdaOperator{\AgdaInductiveConstructor{⊕}}\AgdaSpace{}%
\AgdaSymbol{(}\AgdaBound{b}\AgdaSpace{}%
\AgdaOperator{\AgdaFunction{⊕}}\AgdaSpace{}%
\AgdaBound{y}\AgdaSymbol{)}\<%
\\
\>[0]\AgdaSymbol{(}\AgdaInductiveConstructor{swap}\AgdaSpace{}%
\AgdaBound{a}\AgdaSpace{}%
\AgdaBound{b}\AgdaSpace{}%
\AgdaBound{c}\AgdaSpace{}%
\AgdaBound{i}\AgdaSymbol{)}%
\>[17]\AgdaOperator{\AgdaFunction{⊕}}\AgdaSpace{}%
\AgdaBound{y}\AgdaSpace{}%
\AgdaSymbol{=}\AgdaSpace{}%
\AgdaInductiveConstructor{swap}\AgdaSpace{}%
\AgdaBound{a}\AgdaSpace{}%
\AgdaBound{b}\AgdaSpace{}%
\AgdaSymbol{(}\AgdaBound{c}\AgdaSpace{}%
\AgdaOperator{\AgdaFunction{⊕}}\AgdaSpace{}%
\AgdaBound{y}\AgdaSymbol{)}\AgdaSpace{}%
\AgdaBound{i}\<%
\\
\>[0]\AgdaSymbol{(}\AgdaInductiveConstructor{trunc}\AgdaSpace{}%
\AgdaBound{p}\AgdaSpace{}%
\AgdaBound{q}\AgdaSpace{}%
\AgdaBound{i}\AgdaSpace{}%
\AgdaBound{j}\AgdaSymbol{)}%
\>[17]\AgdaOperator{\AgdaFunction{⊕}}\AgdaSpace{}%
\AgdaBound{y}\AgdaSpace{}%
\AgdaSymbol{=}\AgdaSpace{}%
\AgdaInductiveConstructor{trunc}\AgdaSpace{}%
\AgdaSymbol{(}\AgdaFunction{cong}\AgdaSpace{}%
\AgdaSymbol{(}\AgdaOperator{\AgdaFunction{\AgdaUnderscore{}⊕}}\AgdaSpace{}%
\AgdaBound{y}\AgdaSymbol{)}\AgdaSpace{}%
\AgdaBound{p}\AgdaSymbol{)}\AgdaSpace{}%
\AgdaSymbol{(}\AgdaFunction{cong}\AgdaSpace{}%
\AgdaSymbol{(}\AgdaOperator{\AgdaFunction{\AgdaUnderscore{}⊕}}\AgdaSpace{}%
\AgdaBound{y}\AgdaSymbol{)}\AgdaSpace{}%
\AgdaBound{q}\AgdaSymbol{)}\AgdaSpace{}%
\AgdaBound{i}\AgdaSpace{}%
\AgdaBound{j}\<%
\end{code}}

\begin{code}[hide]%
\>[0]\AgdaFunction{⊕unitl}\AgdaSpace{}%
\AgdaSymbol{:}\AgdaSpace{}%
\AgdaSymbol{(}\AgdaBound{a}\AgdaSpace{}%
\AgdaSymbol{:}\AgdaSpace{}%
\AgdaDatatype{HITOrd}\AgdaSymbol{)}\AgdaSpace{}%
\AgdaSymbol{→}\AgdaSpace{}%
\AgdaInductiveConstructor{𝟎}\AgdaSpace{}%
\AgdaOperator{\AgdaFunction{⊕}}\AgdaSpace{}%
\AgdaBound{a}\AgdaSpace{}%
\AgdaOperator{\AgdaFunction{≡}}\AgdaSpace{}%
\AgdaBound{a}\<%
\\
\>[0]\AgdaFunction{⊕unitl}\AgdaSpace{}%
\AgdaBound{a}\AgdaSpace{}%
\AgdaSymbol{=}\AgdaSpace{}%
\AgdaFunction{refl}\<%
\end{code}

\newcommand{\HsumLemmaunitR}{
\begin{code}%
\>[0]\AgdaFunction{⊕unitr}\AgdaSpace{}%
\AgdaSymbol{:}\AgdaSpace{}%
\AgdaSymbol{(}\AgdaBound{a}\AgdaSpace{}%
\AgdaSymbol{:}\AgdaSpace{}%
\AgdaDatatype{HITOrd}\AgdaSymbol{)}\AgdaSpace{}%
\AgdaSymbol{→}\AgdaSpace{}%
\AgdaBound{a}\AgdaSpace{}%
\AgdaOperator{\AgdaFunction{⊕}}\AgdaSpace{}%
\AgdaInductiveConstructor{𝟎}\AgdaSpace{}%
\AgdaOperator{\AgdaFunction{≡}}\AgdaSpace{}%
\AgdaBound{a}\<%
\end{code}}

\newcommand{\HsumLemmas}{
\begin{code}%
\>[0]\AgdaFunction{⊕assoc}\AgdaSpace{}%
\AgdaSymbol{:}\AgdaSpace{}%
\AgdaFunction{Assoc}\AgdaSpace{}%
\AgdaDatatype{HITOrd}\AgdaSpace{}%
\AgdaOperator{\AgdaFunction{\AgdaUnderscore{}⊕\AgdaUnderscore{}}}\<%
\\
\>[0]\AgdaFunction{ω\textasciicircum{}⊕=⊕ω\textasciicircum{}}\AgdaSpace{}%
\AgdaSymbol{:}\AgdaSpace{}%
\AgdaSymbol{(}\AgdaBound{a}\AgdaSpace{}%
\AgdaBound{b}\AgdaSpace{}%
\AgdaSymbol{:}\AgdaSpace{}%
\AgdaDatatype{HITOrd}\AgdaSymbol{)}\AgdaSpace{}%
\AgdaSymbol{→}\AgdaSpace{}%
\AgdaSymbol{(}\AgdaOperator{\AgdaInductiveConstructor{ω\textasciicircum{}}}\AgdaSpace{}%
\AgdaBound{a}\AgdaSpace{}%
\AgdaOperator{\AgdaInductiveConstructor{⊕}}\AgdaSpace{}%
\AgdaBound{b}\AgdaSymbol{)}\AgdaSpace{}%
\AgdaOperator{\AgdaFunction{≡}}\AgdaSpace{}%
\AgdaBound{b}\AgdaSpace{}%
\AgdaOperator{\AgdaFunction{⊕}}\AgdaSpace{}%
\AgdaOperator{\AgdaFunction{H.ω\textasciicircum{}⟨}}\AgdaSpace{}%
\AgdaBound{a}\AgdaSpace{}%
\AgdaOperator{\AgdaFunction{⟩}}\<%
\end{code}}

\begin{code}[hide]%
\>[0]\AgdaFunction{⊕unitr}\AgdaSpace{}%
\AgdaSymbol{=}\AgdaSpace{}%
\AgdaFunction{indProp}\AgdaSpace{}%
\AgdaSymbol{\AgdaUnderscore{}}\AgdaSpace{}%
\AgdaInductiveConstructor{trunc}\AgdaSpace{}%
\AgdaFunction{base}\AgdaSpace{}%
\AgdaFunction{step}\<%
\\
\>[0][@{}l@{\AgdaIndent{0}}]%
\>[1]\AgdaKeyword{where}\<%
\\
\>[1][@{}l@{\AgdaIndent{0}}]%
\>[2]\AgdaFunction{base}\AgdaSpace{}%
\AgdaSymbol{:}\AgdaSpace{}%
\AgdaInductiveConstructor{𝟎}\AgdaSpace{}%
\AgdaOperator{\AgdaFunction{⊕}}\AgdaSpace{}%
\AgdaInductiveConstructor{𝟎}\AgdaSpace{}%
\AgdaOperator{\AgdaFunction{≡}}\AgdaSpace{}%
\AgdaInductiveConstructor{𝟎}\<%
\\
\>[2]\AgdaFunction{base}\AgdaSpace{}%
\AgdaSymbol{=}\AgdaSpace{}%
\AgdaFunction{refl}\<%
\\
\>[2]\AgdaFunction{step}%
\>[219I]\AgdaSymbol{:}\AgdaSpace{}%
\AgdaSymbol{∀}\AgdaSpace{}%
\AgdaSymbol{\{}\AgdaBound{x}\AgdaSpace{}%
\AgdaBound{y}\AgdaSymbol{\}}\AgdaSpace{}%
\AgdaSymbol{→}\AgdaSpace{}%
\AgdaSymbol{\AgdaUnderscore{}}\AgdaSpace{}%
\AgdaSymbol{→}\AgdaSpace{}%
\AgdaBound{y}\AgdaSpace{}%
\AgdaOperator{\AgdaFunction{⊕}}\AgdaSpace{}%
\AgdaInductiveConstructor{𝟎}\AgdaSpace{}%
\AgdaOperator{\AgdaFunction{≡}}\AgdaSpace{}%
\AgdaBound{y}\<%
\\
\>[.][@{}l@{}]\<[219I]%
\>[7]\AgdaSymbol{→}\AgdaSpace{}%
\AgdaSymbol{(}\AgdaOperator{\AgdaInductiveConstructor{ω\textasciicircum{}}}\AgdaSpace{}%
\AgdaBound{x}\AgdaSpace{}%
\AgdaOperator{\AgdaInductiveConstructor{⊕}}\AgdaSpace{}%
\AgdaBound{y}\AgdaSymbol{)}\AgdaSpace{}%
\AgdaOperator{\AgdaFunction{⊕}}\AgdaSpace{}%
\AgdaInductiveConstructor{𝟎}\AgdaSpace{}%
\AgdaOperator{\AgdaFunction{≡}}\AgdaSpace{}%
\AgdaOperator{\AgdaInductiveConstructor{ω\textasciicircum{}}}\AgdaSpace{}%
\AgdaBound{x}\AgdaSpace{}%
\AgdaOperator{\AgdaInductiveConstructor{⊕}}\AgdaSpace{}%
\AgdaBound{y}\<%
\\
\>[2]\AgdaFunction{step}\AgdaSpace{}%
\AgdaSymbol{\{}\AgdaBound{x}\AgdaSymbol{\}}\AgdaSpace{}%
\AgdaSymbol{\AgdaUnderscore{}}\AgdaSpace{}%
\AgdaSymbol{=}\AgdaSpace{}%
\AgdaFunction{cong}\AgdaSpace{}%
\AgdaSymbol{(}\AgdaOperator{\AgdaInductiveConstructor{ω\textasciicircum{}}}\AgdaSpace{}%
\AgdaBound{x}\AgdaSpace{}%
\AgdaOperator{\AgdaInductiveConstructor{⊕\AgdaUnderscore{}}}\AgdaSymbol{)}\<%
\\
\\[\AgdaEmptyExtraSkip]%
\>[0]\AgdaFunction{⊕assoc}\AgdaSpace{}%
\AgdaSymbol{=}\AgdaSpace{}%
\AgdaFunction{indProp}\AgdaSpace{}%
\AgdaFunction{P}\AgdaSpace{}%
\AgdaSymbol{(λ}\AgdaSpace{}%
\AgdaSymbol{\{}\AgdaBound{a}\AgdaSymbol{\}}\AgdaSpace{}%
\AgdaSymbol{→}\AgdaSpace{}%
\AgdaFunction{hprop}\AgdaSpace{}%
\AgdaBound{a}\AgdaSymbol{)}\AgdaSpace{}%
\AgdaFunction{base}\AgdaSpace{}%
\AgdaSymbol{(λ}\AgdaSpace{}%
\AgdaSymbol{\{}\AgdaBound{x}\AgdaSymbol{\}}\AgdaSpace{}%
\AgdaSymbol{\{}\AgdaBound{y}\AgdaSymbol{\}}\AgdaSpace{}%
\AgdaSymbol{→}\AgdaSpace{}%
\AgdaFunction{step}\AgdaSpace{}%
\AgdaBound{x}\AgdaSpace{}%
\AgdaBound{y}\AgdaSymbol{)}\<%
\\
\>[0][@{}l@{\AgdaIndent{0}}]%
\>[1]\AgdaKeyword{where}\<%
\\
\>[1][@{}l@{\AgdaIndent{0}}]%
\>[2]\AgdaFunction{P}\AgdaSpace{}%
\AgdaSymbol{:}\AgdaSpace{}%
\AgdaDatatype{HITOrd}\AgdaSpace{}%
\AgdaSymbol{→}\AgdaSpace{}%
\AgdaFunction{Type₀}\<%
\\
\>[2]\AgdaFunction{P}\AgdaSpace{}%
\AgdaBound{a}\AgdaSpace{}%
\AgdaSymbol{=}\AgdaSpace{}%
\AgdaSymbol{∀}\AgdaSpace{}%
\AgdaBound{b}\AgdaSpace{}%
\AgdaBound{c}\AgdaSpace{}%
\AgdaSymbol{→}\AgdaSpace{}%
\AgdaBound{a}\AgdaSpace{}%
\AgdaOperator{\AgdaFunction{⊕}}\AgdaSpace{}%
\AgdaBound{b}\AgdaSpace{}%
\AgdaOperator{\AgdaFunction{⊕}}\AgdaSpace{}%
\AgdaBound{c}\AgdaSpace{}%
\AgdaOperator{\AgdaFunction{≡}}\AgdaSpace{}%
\AgdaSymbol{(}\AgdaBound{a}\AgdaSpace{}%
\AgdaOperator{\AgdaFunction{⊕}}\AgdaSpace{}%
\AgdaBound{b}\AgdaSymbol{)}\AgdaSpace{}%
\AgdaOperator{\AgdaFunction{⊕}}\AgdaSpace{}%
\AgdaBound{c}\<%
\\
\>[2]\AgdaFunction{hprop}\AgdaSpace{}%
\AgdaSymbol{:}\AgdaSpace{}%
\AgdaSymbol{∀}\AgdaSpace{}%
\AgdaBound{a}\AgdaSpace{}%
\AgdaSymbol{→}\AgdaSpace{}%
\AgdaFunction{isProp}\AgdaSpace{}%
\AgdaSymbol{(}\AgdaFunction{P}\AgdaSpace{}%
\AgdaBound{a}\AgdaSymbol{)}\<%
\\
\>[2]\AgdaFunction{hprop}\AgdaSpace{}%
\AgdaSymbol{\AgdaUnderscore{}}\AgdaSpace{}%
\AgdaBound{p}\AgdaSpace{}%
\AgdaBound{q}\AgdaSpace{}%
\AgdaBound{i}\AgdaSpace{}%
\AgdaBound{b}\AgdaSpace{}%
\AgdaBound{c}\AgdaSpace{}%
\AgdaSymbol{=}\AgdaSpace{}%
\AgdaInductiveConstructor{trunc}\AgdaSpace{}%
\AgdaSymbol{(}\AgdaBound{p}\AgdaSpace{}%
\AgdaBound{b}\AgdaSpace{}%
\AgdaBound{c}\AgdaSymbol{)}\AgdaSpace{}%
\AgdaSymbol{(}\AgdaBound{q}\AgdaSpace{}%
\AgdaBound{b}\AgdaSpace{}%
\AgdaBound{c}\AgdaSymbol{)}\AgdaSpace{}%
\AgdaBound{i}\<%
\\
\>[2]\AgdaFunction{base}\AgdaSpace{}%
\AgdaSymbol{:}\AgdaSpace{}%
\AgdaFunction{P}\AgdaSpace{}%
\AgdaInductiveConstructor{𝟎}\<%
\\
\>[2]\AgdaFunction{base}\AgdaSpace{}%
\AgdaSymbol{\AgdaUnderscore{}}\AgdaSpace{}%
\AgdaSymbol{\AgdaUnderscore{}}\AgdaSpace{}%
\AgdaSymbol{=}\AgdaSpace{}%
\AgdaFunction{refl}\<%
\\
\>[2]\AgdaFunction{step}\AgdaSpace{}%
\AgdaSymbol{:}\AgdaSpace{}%
\AgdaSymbol{∀}\AgdaSpace{}%
\AgdaBound{x}\AgdaSpace{}%
\AgdaBound{y}\AgdaSpace{}%
\AgdaSymbol{→}\AgdaSpace{}%
\AgdaFunction{P}\AgdaSpace{}%
\AgdaBound{x}\AgdaSpace{}%
\AgdaSymbol{→}\AgdaSpace{}%
\AgdaFunction{P}\AgdaSpace{}%
\AgdaBound{y}\AgdaSpace{}%
\AgdaSymbol{→}\AgdaSpace{}%
\AgdaFunction{P}\AgdaSpace{}%
\AgdaSymbol{(}\AgdaOperator{\AgdaInductiveConstructor{ω\textasciicircum{}}}\AgdaSpace{}%
\AgdaBound{x}\AgdaSpace{}%
\AgdaOperator{\AgdaInductiveConstructor{⊕}}\AgdaSpace{}%
\AgdaBound{y}\AgdaSymbol{)}\<%
\\
\>[2]\AgdaFunction{step}\AgdaSpace{}%
\AgdaBound{x}\AgdaSpace{}%
\AgdaBound{y}\AgdaSpace{}%
\AgdaSymbol{\AgdaUnderscore{}}\AgdaSpace{}%
\AgdaBound{p}\AgdaSpace{}%
\AgdaBound{b}\AgdaSpace{}%
\AgdaBound{c}\AgdaSpace{}%
\AgdaSymbol{=}\AgdaSpace{}%
\AgdaFunction{cong}\AgdaSpace{}%
\AgdaSymbol{(}\AgdaOperator{\AgdaInductiveConstructor{ω\textasciicircum{}}}\AgdaSpace{}%
\AgdaBound{x}\AgdaSpace{}%
\AgdaOperator{\AgdaInductiveConstructor{⊕\AgdaUnderscore{}}}\AgdaSymbol{)}\AgdaSpace{}%
\AgdaSymbol{(}\AgdaBound{p}\AgdaSpace{}%
\AgdaBound{b}\AgdaSpace{}%
\AgdaBound{c}\AgdaSymbol{)}\<%
\\
\\[\AgdaEmptyExtraSkip]%
\>[0]\AgdaFunction{ω\textasciicircum{}⊕=⊕ω\textasciicircum{}}\AgdaSpace{}%
\AgdaBound{a}\AgdaSpace{}%
\AgdaSymbol{=}\AgdaSpace{}%
\AgdaFunction{indProp}\AgdaSpace{}%
\AgdaFunction{P}\AgdaSpace{}%
\AgdaInductiveConstructor{trunc}\AgdaSpace{}%
\AgdaFunction{base}\AgdaSpace{}%
\AgdaFunction{step}\<%
\\
\>[0][@{}l@{\AgdaIndent{0}}]%
\>[1]\AgdaKeyword{where}\<%
\\
\>[1][@{}l@{\AgdaIndent{0}}]%
\>[2]\AgdaFunction{P}\AgdaSpace{}%
\AgdaSymbol{:}\AgdaSpace{}%
\AgdaDatatype{HITOrd}\AgdaSpace{}%
\AgdaSymbol{→}\AgdaSpace{}%
\AgdaFunction{Type₀}\<%
\\
\>[2]\AgdaFunction{P}\AgdaSpace{}%
\AgdaBound{b}\AgdaSpace{}%
\AgdaSymbol{=}\AgdaSpace{}%
\AgdaSymbol{(}\AgdaOperator{\AgdaInductiveConstructor{ω\textasciicircum{}}}\AgdaSpace{}%
\AgdaBound{a}\AgdaSpace{}%
\AgdaOperator{\AgdaInductiveConstructor{⊕}}\AgdaSpace{}%
\AgdaBound{b}\AgdaSymbol{)}\AgdaSpace{}%
\AgdaOperator{\AgdaFunction{≡}}\AgdaSpace{}%
\AgdaBound{b}\AgdaSpace{}%
\AgdaOperator{\AgdaFunction{⊕}}\AgdaSpace{}%
\AgdaOperator{\AgdaFunction{H.ω\textasciicircum{}⟨}}\AgdaSpace{}%
\AgdaBound{a}\AgdaSpace{}%
\AgdaOperator{\AgdaFunction{⟩}}\<%
\\
\>[2]\AgdaFunction{base}\AgdaSpace{}%
\AgdaSymbol{:}\AgdaSpace{}%
\AgdaFunction{P}\AgdaSpace{}%
\AgdaInductiveConstructor{𝟎}\<%
\\
\>[2]\AgdaFunction{base}\AgdaSpace{}%
\AgdaSymbol{=}\AgdaSpace{}%
\AgdaFunction{refl}\<%
\\
\>[2]\AgdaFunction{step}\AgdaSpace{}%
\AgdaSymbol{:}\AgdaSpace{}%
\AgdaSymbol{∀}\AgdaSpace{}%
\AgdaSymbol{\{}\AgdaBound{x}\AgdaSpace{}%
\AgdaBound{y}\AgdaSymbol{\}}\AgdaSpace{}%
\AgdaSymbol{→}\AgdaSpace{}%
\AgdaFunction{P}\AgdaSpace{}%
\AgdaBound{x}\AgdaSpace{}%
\AgdaSymbol{→}\AgdaSpace{}%
\AgdaFunction{P}\AgdaSpace{}%
\AgdaBound{y}\AgdaSpace{}%
\AgdaSymbol{→}\AgdaSpace{}%
\AgdaFunction{P}\AgdaSpace{}%
\AgdaSymbol{(}\AgdaOperator{\AgdaInductiveConstructor{ω\textasciicircum{}}}\AgdaSpace{}%
\AgdaBound{x}\AgdaSpace{}%
\AgdaOperator{\AgdaInductiveConstructor{⊕}}\AgdaSpace{}%
\AgdaBound{y}\AgdaSymbol{)}\<%
\\
\>[2]\AgdaFunction{step}\AgdaSpace{}%
\AgdaSymbol{\{}\AgdaBound{x}\AgdaSymbol{\}}\AgdaSpace{}%
\AgdaSymbol{\{}\AgdaBound{y}\AgdaSymbol{\}}\AgdaSpace{}%
\AgdaSymbol{\AgdaUnderscore{}}\AgdaSpace{}%
\AgdaBound{p}\AgdaSpace{}%
\AgdaSymbol{=}\AgdaSpace{}%
\AgdaInductiveConstructor{swap}\AgdaSpace{}%
\AgdaBound{a}\AgdaSpace{}%
\AgdaBound{x}\AgdaSpace{}%
\AgdaBound{y}\AgdaSpace{}%
\AgdaOperator{\AgdaFunction{∙}}\AgdaSpace{}%
\AgdaFunction{cong}\AgdaSpace{}%
\AgdaSymbol{(}\AgdaOperator{\AgdaInductiveConstructor{ω\textasciicircum{}}}\AgdaSpace{}%
\AgdaBound{x}\AgdaSpace{}%
\AgdaOperator{\AgdaInductiveConstructor{⊕\AgdaUnderscore{}}}\AgdaSymbol{)}\AgdaSpace{}%
\AgdaBound{p}\<%
\end{code}

\newcommand{\HsumComm}{
\begin{code}%
\>[0]\AgdaFunction{⊕comm}\AgdaSpace{}%
\AgdaSymbol{:}\AgdaSpace{}%
\AgdaFunction{Comm}\AgdaSpace{}%
\AgdaDatatype{HITOrd}\AgdaSpace{}%
\AgdaOperator{\AgdaFunction{\AgdaUnderscore{}⊕\AgdaUnderscore{}}}\<%
\\
\>[0]\AgdaFunction{⊕comm}\AgdaSpace{}%
\AgdaBound{a}\AgdaSpace{}%
\AgdaSymbol{=}\AgdaSpace{}%
\AgdaFunction{indProp}\AgdaSpace{}%
\AgdaFunction{P}\AgdaSpace{}%
\AgdaInductiveConstructor{trunc}\AgdaSpace{}%
\AgdaFunction{base}\AgdaSpace{}%
\AgdaFunction{step}\<%
\\
\>[0][@{}l@{\AgdaIndent{0}}]%
\>[1]\AgdaKeyword{where}\<%
\\
\>[1][@{}l@{\AgdaIndent{0}}]%
\>[2]\AgdaFunction{P}\AgdaSpace{}%
\AgdaSymbol{:}\AgdaSpace{}%
\AgdaDatatype{HITOrd}\AgdaSpace{}%
\AgdaSymbol{→}\AgdaSpace{}%
\AgdaFunction{Type₀}\<%
\\
\>[2]\AgdaFunction{P}\AgdaSpace{}%
\AgdaBound{b}\AgdaSpace{}%
\AgdaSymbol{=}\AgdaSpace{}%
\AgdaBound{a}\AgdaSpace{}%
\AgdaOperator{\AgdaFunction{⊕}}\AgdaSpace{}%
\AgdaBound{b}\AgdaSpace{}%
\AgdaOperator{\AgdaFunction{≡}}\AgdaSpace{}%
\AgdaBound{b}\AgdaSpace{}%
\AgdaOperator{\AgdaFunction{⊕}}\AgdaSpace{}%
\AgdaBound{a}\<%
\\
\>[2]\AgdaFunction{base}\AgdaSpace{}%
\AgdaSymbol{:}\AgdaSpace{}%
\AgdaFunction{P}\AgdaSpace{}%
\AgdaInductiveConstructor{𝟎}\<%
\\
\>[2]\AgdaFunction{base}\AgdaSpace{}%
\AgdaSymbol{=}\AgdaSpace{}%
\AgdaFunction{⊕unitr}\AgdaSpace{}%
\AgdaBound{a}\<%
\\
\>[2]\AgdaFunction{step}\AgdaSpace{}%
\AgdaSymbol{:}\AgdaSpace{}%
\AgdaSymbol{∀}\AgdaSpace{}%
\AgdaSymbol{\{}\AgdaBound{x}\AgdaSpace{}%
\AgdaBound{y}\AgdaSymbol{\}}\AgdaSpace{}%
\AgdaSymbol{→}\AgdaSpace{}%
\AgdaFunction{P}\AgdaSpace{}%
\AgdaBound{x}\AgdaSpace{}%
\AgdaSymbol{→}\AgdaSpace{}%
\AgdaFunction{P}\AgdaSpace{}%
\AgdaBound{y}\AgdaSpace{}%
\AgdaSymbol{→}\AgdaSpace{}%
\AgdaFunction{P}\AgdaSpace{}%
\AgdaSymbol{(}\AgdaOperator{\AgdaInductiveConstructor{ω\textasciicircum{}}}\AgdaSpace{}%
\AgdaBound{x}\AgdaSpace{}%
\AgdaOperator{\AgdaInductiveConstructor{⊕}}\AgdaSpace{}%
\AgdaBound{y}\AgdaSymbol{)}\<%
\\
\>[2]\AgdaFunction{step}\AgdaSpace{}%
\AgdaSymbol{\{}\AgdaBound{x}\AgdaSymbol{\}}\AgdaSpace{}%
\AgdaSymbol{\{}\AgdaBound{y}\AgdaSymbol{\}}\AgdaSpace{}%
\AgdaSymbol{\AgdaUnderscore{}}\AgdaSpace{}%
\AgdaBound{p}\AgdaSpace{}%
\AgdaSymbol{=}\AgdaSpace{}%
\AgdaOperator{\AgdaFunction{begin}}\<%
\\
\>[2][@{}l@{\AgdaIndent{0}}]%
\>[6]\AgdaBound{a}\AgdaSpace{}%
\AgdaOperator{\AgdaFunction{⊕}}\AgdaSpace{}%
\AgdaSymbol{(}\AgdaOperator{\AgdaInductiveConstructor{ω\textasciicircum{}}}\AgdaSpace{}%
\AgdaBound{x}\AgdaSpace{}%
\AgdaOperator{\AgdaInductiveConstructor{⊕}}\AgdaSpace{}%
\AgdaBound{y}\AgdaSymbol{)}\<%
\\
\>[2][@{}l@{\AgdaIndent{0}}]%
\>[4]\AgdaOperator{\AgdaFunction{≡⟨}}\AgdaSpace{}%
\AgdaFunction{cong}\AgdaSpace{}%
\AgdaSymbol{(}\AgdaBound{a}\AgdaSpace{}%
\AgdaOperator{\AgdaFunction{⊕\AgdaUnderscore{}}}\AgdaSymbol{)}\AgdaSpace{}%
\AgdaSymbol{(}\AgdaFunction{ω\textasciicircum{}⊕=⊕ω\textasciicircum{}}\AgdaSpace{}%
\AgdaBound{x}\AgdaSpace{}%
\AgdaBound{y}\AgdaSymbol{)}\AgdaSpace{}%
\AgdaOperator{\AgdaFunction{⟩}}\<%
\\
\>[4][@{}l@{\AgdaIndent{0}}]%
\>[6]\AgdaBound{a}\AgdaSpace{}%
\AgdaOperator{\AgdaFunction{⊕}}\AgdaSpace{}%
\AgdaSymbol{(}\AgdaBound{y}\AgdaSpace{}%
\AgdaOperator{\AgdaFunction{⊕}}\AgdaSpace{}%
\AgdaOperator{\AgdaFunction{H.ω\textasciicircum{}⟨}}\AgdaSpace{}%
\AgdaBound{x}\AgdaSpace{}%
\AgdaOperator{\AgdaFunction{⟩}}\AgdaSymbol{)}\<%
\\
\>[4]\AgdaOperator{\AgdaFunction{≡⟨}}\AgdaSpace{}%
\AgdaFunction{⊕assoc}\AgdaSpace{}%
\AgdaBound{a}\AgdaSpace{}%
\AgdaBound{y}\AgdaSpace{}%
\AgdaOperator{\AgdaFunction{H.ω\textasciicircum{}⟨}}\AgdaSpace{}%
\AgdaBound{x}\AgdaSpace{}%
\AgdaOperator{\AgdaFunction{⟩}}\AgdaSpace{}%
\AgdaOperator{\AgdaFunction{⟩}}\<%
\\
\>[4][@{}l@{\AgdaIndent{0}}]%
\>[6]\AgdaSymbol{(}\AgdaBound{a}\AgdaSpace{}%
\AgdaOperator{\AgdaFunction{⊕}}\AgdaSpace{}%
\AgdaBound{y}\AgdaSymbol{)}\AgdaSpace{}%
\AgdaOperator{\AgdaFunction{⊕}}\AgdaSpace{}%
\AgdaOperator{\AgdaFunction{H.ω\textasciicircum{}⟨}}\AgdaSpace{}%
\AgdaBound{x}\AgdaSpace{}%
\AgdaOperator{\AgdaFunction{⟩}}\<%
\\
\>[4]\AgdaOperator{\AgdaFunction{≡⟨}}\AgdaSpace{}%
\AgdaFunction{cong}\AgdaSpace{}%
\AgdaSymbol{(}\AgdaOperator{\AgdaFunction{\AgdaUnderscore{}⊕}}\AgdaSpace{}%
\AgdaOperator{\AgdaFunction{H.ω\textasciicircum{}⟨}}\AgdaSpace{}%
\AgdaBound{x}\AgdaSpace{}%
\AgdaOperator{\AgdaFunction{⟩}}\AgdaSymbol{)}\AgdaSpace{}%
\AgdaBound{p}\AgdaSpace{}%
\AgdaOperator{\AgdaFunction{⟩}}\<%
\\
\>[4][@{}l@{\AgdaIndent{0}}]%
\>[6]\AgdaSymbol{(}\AgdaBound{y}\AgdaSpace{}%
\AgdaOperator{\AgdaFunction{⊕}}\AgdaSpace{}%
\AgdaBound{a}\AgdaSymbol{)}\AgdaSpace{}%
\AgdaOperator{\AgdaFunction{⊕}}\AgdaSpace{}%
\AgdaOperator{\AgdaFunction{H.ω\textasciicircum{}⟨}}\AgdaSpace{}%
\AgdaBound{x}\AgdaSpace{}%
\AgdaOperator{\AgdaFunction{⟩}}\<%
\\
\>[4]\AgdaOperator{\AgdaFunction{≡⟨}}\AgdaSpace{}%
\AgdaSymbol{(}\AgdaFunction{ω\textasciicircum{}⊕=⊕ω\textasciicircum{}}\AgdaSpace{}%
\AgdaBound{x}\AgdaSpace{}%
\AgdaSymbol{(}\AgdaBound{y}\AgdaSpace{}%
\AgdaOperator{\AgdaFunction{⊕}}\AgdaSpace{}%
\AgdaBound{a}\AgdaSymbol{))}\AgdaSpace{}%
\AgdaOperator{\AgdaFunction{⁻¹}}\AgdaSpace{}%
\AgdaOperator{\AgdaFunction{⟩}}\<%
\\
\>[4][@{}l@{\AgdaIndent{0}}]%
\>[6]\AgdaSymbol{(}\AgdaOperator{\AgdaInductiveConstructor{ω\textasciicircum{}}}\AgdaSpace{}%
\AgdaBound{x}\AgdaSpace{}%
\AgdaOperator{\AgdaInductiveConstructor{⊕}}\AgdaSpace{}%
\AgdaBound{y}\AgdaSymbol{)}\AgdaSpace{}%
\AgdaOperator{\AgdaFunction{⊕}}\AgdaSpace{}%
\AgdaBound{a}%
\>[22]\AgdaOperator{\AgdaFunction{∎}}\<%
\end{code}}

\newcommand{\pathHM}{
\begin{code}%
\>[0]\AgdaFunction{H≡M}\AgdaSpace{}%
\AgdaSymbol{:}\AgdaSpace{}%
\AgdaDatatype{HITOrd}\AgdaSpace{}%
\AgdaOperator{\AgdaFunction{≡}}\AgdaSpace{}%
\AgdaDatatype{MutualOrd}\<%
\\
\>[0]\AgdaFunction{H≡M}\AgdaSpace{}%
\AgdaBound{i}\AgdaSpace{}%
\AgdaSymbol{=}\AgdaSpace{}%
\AgdaFunction{M≡H}\AgdaSpace{}%
\AgdaSymbol{(}\AgdaOperator{\AgdaPrimitive{\textasciitilde{}}}\AgdaSpace{}%
\AgdaBound{i}\AgdaSymbol{)}\<%
\end{code}}

\newcommand{\MHSum}{
\begin{code}%
\>[0]\AgdaOperator{\AgdaFunction{\AgdaUnderscore{}⊕ᴹ\AgdaUnderscore{}}}\AgdaSpace{}%
\AgdaSymbol{:}\AgdaSpace{}%
\AgdaDatatype{MutualOrd}\AgdaSpace{}%
\AgdaSymbol{→}\AgdaSpace{}%
\AgdaDatatype{MutualOrd}\AgdaSpace{}%
\AgdaSymbol{→}\AgdaSpace{}%
\AgdaDatatype{MutualOrd}\<%
\\
\>[0]\AgdaOperator{\AgdaFunction{\AgdaUnderscore{}⊕ᴹ\AgdaUnderscore{}}}\AgdaSpace{}%
\AgdaSymbol{=}\AgdaSpace{}%
\AgdaFunction{transport}\AgdaSpace{}%
\AgdaSymbol{(λ}\AgdaSpace{}%
\AgdaBound{i}\AgdaSpace{}%
\AgdaSymbol{→}\AgdaSpace{}%
\AgdaFunction{H≡M}\AgdaSpace{}%
\AgdaBound{i}\AgdaSpace{}%
\AgdaSymbol{→}\AgdaSpace{}%
\AgdaFunction{H≡M}\AgdaSpace{}%
\AgdaBound{i}\AgdaSpace{}%
\AgdaSymbol{→}\AgdaSpace{}%
\AgdaFunction{H≡M}\AgdaSpace{}%
\AgdaBound{i}\AgdaSymbol{)}\AgdaSpace{}%
\AgdaOperator{\AgdaFunction{\AgdaUnderscore{}⊕\AgdaUnderscore{}}}\<%
\end{code}}

\newcommand{\HSumPath}{
\begin{code}%
\>[0]\AgdaFunction{⊕Path}\AgdaSpace{}%
\AgdaSymbol{:}\AgdaSpace{}%
\AgdaPostulate{PathP}\AgdaSpace{}%
\AgdaSymbol{(λ}\AgdaSpace{}%
\AgdaBound{i}\AgdaSpace{}%
\AgdaSymbol{→}\AgdaSpace{}%
\AgdaFunction{H≡M}\AgdaSpace{}%
\AgdaBound{i}\AgdaSpace{}%
\AgdaSymbol{→}\AgdaSpace{}%
\AgdaFunction{H≡M}\AgdaSpace{}%
\AgdaBound{i}\AgdaSpace{}%
\AgdaSymbol{→}\AgdaSpace{}%
\AgdaFunction{H≡M}\AgdaSpace{}%
\AgdaBound{i}\AgdaSymbol{)}\AgdaSpace{}%
\AgdaOperator{\AgdaFunction{\AgdaUnderscore{}⊕\AgdaUnderscore{}}}\AgdaSpace{}%
\AgdaOperator{\AgdaFunction{\AgdaUnderscore{}⊕ᴹ\AgdaUnderscore{}}}\<%
\end{code}}

\begin{code}[hide]%
\>[0]\AgdaFunction{⊕Path}\AgdaSpace{}%
\AgdaBound{i}\AgdaSpace{}%
\AgdaSymbol{=}\AgdaSpace{}%
\AgdaPrimitive{transp}\AgdaSpace{}%
\AgdaSymbol{(λ}\AgdaSpace{}%
\AgdaBound{j}\AgdaSpace{}%
\AgdaSymbol{→}\AgdaSpace{}%
\AgdaFunction{H≡M}\AgdaSpace{}%
\AgdaSymbol{(}\AgdaBound{i}\AgdaSpace{}%
\AgdaOperator{\AgdaPrimitive{∧}}\AgdaSpace{}%
\AgdaBound{j}\AgdaSymbol{)}\AgdaSpace{}%
\AgdaSymbol{→}\AgdaSpace{}%
\AgdaFunction{H≡M}\AgdaSpace{}%
\AgdaSymbol{(}\AgdaBound{i}\AgdaSpace{}%
\AgdaOperator{\AgdaPrimitive{∧}}\AgdaSpace{}%
\AgdaBound{j}\AgdaSymbol{)}\AgdaSpace{}%
\AgdaSymbol{→}\AgdaSpace{}%
\AgdaFunction{H≡M}\AgdaSpace{}%
\AgdaSymbol{(}\AgdaBound{i}\AgdaSpace{}%
\AgdaOperator{\AgdaPrimitive{∧}}\AgdaSpace{}%
\AgdaBound{j}\AgdaSymbol{))}\AgdaSpace{}%
\AgdaSymbol{(}\AgdaOperator{\AgdaPrimitive{\textasciitilde{}}}\AgdaSpace{}%
\AgdaBound{i}\AgdaSymbol{)}\AgdaSpace{}%
\AgdaOperator{\AgdaFunction{\AgdaUnderscore{}⊕\AgdaUnderscore{}}}\<%
\end{code}

\newcommand{\MHSumComm}{
\begin{code}%
\>[0]\AgdaFunction{⊕ᴹcomm}\AgdaSpace{}%
\AgdaSymbol{:}\AgdaSpace{}%
\AgdaFunction{Comm}\AgdaSpace{}%
\AgdaDatatype{MutualOrd}\AgdaSpace{}%
\AgdaOperator{\AgdaFunction{\AgdaUnderscore{}⊕ᴹ\AgdaUnderscore{}}}\<%
\\
\>[0]\AgdaFunction{⊕ᴹcomm}\AgdaSpace{}%
\AgdaSymbol{=}\AgdaSpace{}%
\AgdaFunction{transport}%
\>[585I]\AgdaSymbol{(λ}\AgdaSpace{}%
\AgdaBound{i}\AgdaSpace{}%
\AgdaSymbol{→}\AgdaSpace{}%
\AgdaFunction{Comm}\AgdaSpace{}%
\AgdaSymbol{(}\AgdaFunction{H≡M}\AgdaSpace{}%
\AgdaBound{i}\AgdaSymbol{)}\AgdaSpace{}%
\AgdaSymbol{(}\AgdaFunction{⊕Path}\AgdaSpace{}%
\AgdaBound{i}\AgdaSymbol{))}\<%
\\
\>[.][@{}l@{}]\<[585I]%
\>[19]\AgdaFunction{⊕comm}\<%
\end{code}}

\newcommand{\HPointAdd}{
\begin{code}%
\>[0]\AgdaOperator{\AgdaFunction{\AgdaUnderscore{}∔\AgdaUnderscore{}}}\AgdaSpace{}%
\AgdaSymbol{:}\AgdaSpace{}%
\AgdaDatatype{HITOrd}\AgdaSpace{}%
\AgdaSymbol{→}\AgdaSpace{}%
\AgdaDatatype{HITOrd}\AgdaSpace{}%
\AgdaSymbol{→}\AgdaSpace{}%
\AgdaDatatype{HITOrd}\<%
\\
\>[0]\AgdaInductiveConstructor{𝟎}%
\>[17]\AgdaOperator{\AgdaFunction{∔}}\AgdaSpace{}%
\AgdaBound{b}\AgdaSpace{}%
\AgdaSymbol{=}\AgdaSpace{}%
\AgdaInductiveConstructor{𝟎}\<%
\\
\>[0]\AgdaSymbol{(}\AgdaOperator{\AgdaInductiveConstructor{ω\textasciicircum{}}}\AgdaSpace{}%
\AgdaBound{a}\AgdaSpace{}%
\AgdaOperator{\AgdaInductiveConstructor{⊕}}\AgdaSpace{}%
\AgdaBound{c}\AgdaSymbol{)}%
\>[17]\AgdaOperator{\AgdaFunction{∔}}\AgdaSpace{}%
\AgdaBound{b}\AgdaSpace{}%
\AgdaSymbol{=}\AgdaSpace{}%
\AgdaOperator{\AgdaInductiveConstructor{ω\textasciicircum{}}}\AgdaSpace{}%
\AgdaSymbol{(}\AgdaBound{a}\AgdaSpace{}%
\AgdaOperator{\AgdaFunction{⊕}}\AgdaSpace{}%
\AgdaBound{b}\AgdaSymbol{)}\AgdaSpace{}%
\AgdaOperator{\AgdaInductiveConstructor{⊕}}\AgdaSpace{}%
\AgdaSymbol{(}\AgdaBound{c}\AgdaSpace{}%
\AgdaOperator{\AgdaFunction{∔}}\AgdaSpace{}%
\AgdaBound{b}\AgdaSymbol{)}\<%
\\
\>[0]\AgdaSymbol{(}\AgdaInductiveConstructor{swap}\AgdaSpace{}%
\AgdaBound{x}\AgdaSpace{}%
\AgdaBound{y}\AgdaSpace{}%
\AgdaBound{z}\AgdaSpace{}%
\AgdaBound{i}\AgdaSymbol{)}%
\>[17]\AgdaOperator{\AgdaFunction{∔}}\AgdaSpace{}%
\AgdaBound{b}\AgdaSpace{}%
\AgdaSymbol{=}\AgdaSpace{}%
\AgdaInductiveConstructor{swap}\AgdaSpace{}%
\AgdaSymbol{(}\AgdaBound{x}\AgdaSpace{}%
\AgdaOperator{\AgdaFunction{⊕}}\AgdaSpace{}%
\AgdaBound{b}\AgdaSymbol{)}\AgdaSpace{}%
\AgdaSymbol{(}\AgdaBound{y}\AgdaSpace{}%
\AgdaOperator{\AgdaFunction{⊕}}\AgdaSpace{}%
\AgdaBound{b}\AgdaSymbol{)}\AgdaSpace{}%
\AgdaSymbol{(}\AgdaBound{z}\AgdaSpace{}%
\AgdaOperator{\AgdaFunction{∔}}\AgdaSpace{}%
\AgdaBound{b}\AgdaSymbol{)}\AgdaSpace{}%
\AgdaBound{i}\<%
\\
\>[0]\AgdaSymbol{(}\AgdaInductiveConstructor{trunc}\AgdaSpace{}%
\AgdaBound{p}\AgdaSpace{}%
\AgdaBound{q}\AgdaSpace{}%
\AgdaBound{i}\AgdaSpace{}%
\AgdaBound{j}\AgdaSymbol{)}%
\>[17]\AgdaOperator{\AgdaFunction{∔}}\AgdaSpace{}%
\AgdaBound{b}\AgdaSpace{}%
\AgdaSymbol{=}\AgdaSpace{}%
\AgdaInductiveConstructor{trunc}\AgdaSpace{}%
\AgdaSymbol{(}\AgdaFunction{cong}\AgdaSpace{}%
\AgdaSymbol{(}\AgdaOperator{\AgdaFunction{\AgdaUnderscore{}∔}}\AgdaSpace{}%
\AgdaBound{b}\AgdaSymbol{)}\AgdaSpace{}%
\AgdaBound{p}\AgdaSymbol{)}\AgdaSpace{}%
\AgdaSymbol{(}\AgdaFunction{cong}\AgdaSpace{}%
\AgdaSymbol{(}\AgdaOperator{\AgdaFunction{\AgdaUnderscore{}∔}}\AgdaSpace{}%
\AgdaBound{b}\AgdaSymbol{)}\AgdaSpace{}%
\AgdaBound{q}\AgdaSymbol{)}\AgdaSpace{}%
\AgdaBound{i}\AgdaSpace{}%
\AgdaBound{j}\<%
\end{code}}

\newcommand{\HSumSwap}{
\begin{code}%
\>[0]\AgdaFunction{⊕swap}\AgdaSpace{}%
\AgdaSymbol{:}\AgdaSpace{}%
\AgdaSymbol{∀}\AgdaSpace{}%
\AgdaBound{a}\AgdaSpace{}%
\AgdaBound{b}\AgdaSpace{}%
\AgdaBound{c}\AgdaSpace{}%
\AgdaSymbol{→}\AgdaSpace{}%
\AgdaBound{a}\AgdaSpace{}%
\AgdaOperator{\AgdaFunction{⊕}}\AgdaSpace{}%
\AgdaBound{b}\AgdaSpace{}%
\AgdaOperator{\AgdaFunction{⊕}}\AgdaSpace{}%
\AgdaBound{c}\AgdaSpace{}%
\AgdaOperator{\AgdaFunction{≡}}\AgdaSpace{}%
\AgdaBound{b}\AgdaSpace{}%
\AgdaOperator{\AgdaFunction{⊕}}\AgdaSpace{}%
\AgdaBound{a}\AgdaSpace{}%
\AgdaOperator{\AgdaFunction{⊕}}\AgdaSpace{}%
\AgdaBound{c}\<%
\end{code}}

\begin{code}[hide]%
\>[0]\<%
\\
\>[0]\AgdaFunction{⊕swap}\AgdaSpace{}%
\AgdaBound{a}\AgdaSpace{}%
\AgdaBound{b}\AgdaSpace{}%
\AgdaBound{c}\AgdaSpace{}%
\AgdaSymbol{=}\AgdaSpace{}%
\AgdaOperator{\AgdaFunction{begin}}\<%
\\
\>[0][@{}l@{\AgdaIndent{0}}]%
\>[2]\AgdaBound{a}\AgdaSpace{}%
\AgdaOperator{\AgdaFunction{⊕}}\AgdaSpace{}%
\AgdaSymbol{(}\AgdaBound{b}\AgdaSpace{}%
\AgdaOperator{\AgdaFunction{⊕}}\AgdaSpace{}%
\AgdaBound{c}\AgdaSymbol{)}\AgdaSpace{}%
\AgdaOperator{\AgdaFunction{≡⟨}}\AgdaSpace{}%
\AgdaFunction{⊕assoc}\AgdaSpace{}%
\AgdaBound{a}\AgdaSpace{}%
\AgdaBound{b}\AgdaSpace{}%
\AgdaBound{c}\AgdaSpace{}%
\AgdaOperator{\AgdaFunction{⟩}}\<%
\\
\>[2]\AgdaSymbol{(}\AgdaBound{a}\AgdaSpace{}%
\AgdaOperator{\AgdaFunction{⊕}}\AgdaSpace{}%
\AgdaBound{b}\AgdaSymbol{)}\AgdaSpace{}%
\AgdaOperator{\AgdaFunction{⊕}}\AgdaSpace{}%
\AgdaBound{c}\AgdaSpace{}%
\AgdaOperator{\AgdaFunction{≡⟨}}\AgdaSpace{}%
\AgdaFunction{cong}\AgdaSpace{}%
\AgdaSymbol{(}\AgdaOperator{\AgdaFunction{\AgdaUnderscore{}⊕}}\AgdaSpace{}%
\AgdaBound{c}\AgdaSymbol{)}\AgdaSpace{}%
\AgdaSymbol{(}\AgdaFunction{⊕comm}\AgdaSpace{}%
\AgdaBound{a}\AgdaSpace{}%
\AgdaBound{b}\AgdaSymbol{)}\AgdaSpace{}%
\AgdaOperator{\AgdaFunction{⟩}}\<%
\\
\>[2]\AgdaSymbol{(}\AgdaBound{b}\AgdaSpace{}%
\AgdaOperator{\AgdaFunction{⊕}}\AgdaSpace{}%
\AgdaBound{a}\AgdaSymbol{)}\AgdaSpace{}%
\AgdaOperator{\AgdaFunction{⊕}}\AgdaSpace{}%
\AgdaBound{c}\AgdaSpace{}%
\AgdaOperator{\AgdaFunction{≡⟨}}\AgdaSpace{}%
\AgdaSymbol{(}\AgdaFunction{⊕assoc}\AgdaSpace{}%
\AgdaBound{b}\AgdaSpace{}%
\AgdaBound{a}\AgdaSpace{}%
\AgdaBound{c}\AgdaSymbol{)}\AgdaOperator{\AgdaFunction{⁻¹}}\AgdaSpace{}%
\AgdaOperator{\AgdaFunction{⟩}}\<%
\\
\>[2]\AgdaBound{b}\AgdaSpace{}%
\AgdaOperator{\AgdaFunction{⊕}}\AgdaSpace{}%
\AgdaSymbol{(}\AgdaBound{a}\AgdaSpace{}%
\AgdaOperator{\AgdaFunction{⊕}}\AgdaSpace{}%
\AgdaBound{c}\AgdaSymbol{)}\AgdaSpace{}%
\AgdaOperator{\AgdaFunction{∎}}\<%
\\
\>[0]\<%
\end{code}

\newcommand{\HProd}{
\begin{code}%
\>[0]\AgdaOperator{\AgdaFunction{\AgdaUnderscore{}⊗\AgdaUnderscore{}}}\AgdaSpace{}%
\AgdaSymbol{:}\AgdaSpace{}%
\AgdaDatatype{HITOrd}\AgdaSpace{}%
\AgdaSymbol{→}\AgdaSpace{}%
\AgdaDatatype{HITOrd}\AgdaSpace{}%
\AgdaSymbol{→}\AgdaSpace{}%
\AgdaDatatype{HITOrd}\<%
\\
\>[0]\AgdaBound{a}\AgdaSpace{}%
\AgdaOperator{\AgdaFunction{⊗}}\AgdaSpace{}%
\AgdaInductiveConstructor{𝟎}%
\>[21]\AgdaSymbol{=}\AgdaSpace{}%
\AgdaInductiveConstructor{𝟎}\<%
\\
\>[0]\AgdaBound{a}\AgdaSpace{}%
\AgdaOperator{\AgdaFunction{⊗}}\AgdaSpace{}%
\AgdaSymbol{(}\AgdaOperator{\AgdaInductiveConstructor{ω\textasciicircum{}}}\AgdaSpace{}%
\AgdaBound{b}\AgdaSpace{}%
\AgdaOperator{\AgdaInductiveConstructor{⊕}}\AgdaSpace{}%
\AgdaBound{c}\AgdaSymbol{)}%
\>[21]\AgdaSymbol{=}\AgdaSpace{}%
\AgdaSymbol{(}\AgdaBound{a}\AgdaSpace{}%
\AgdaOperator{\AgdaFunction{∔}}\AgdaSpace{}%
\AgdaBound{b}\AgdaSymbol{)}\AgdaSpace{}%
\AgdaOperator{\AgdaFunction{⊕}}\AgdaSpace{}%
\AgdaSymbol{(}\AgdaBound{a}\AgdaSpace{}%
\AgdaOperator{\AgdaFunction{⊗}}\AgdaSpace{}%
\AgdaBound{c}\AgdaSymbol{)}\<%
\\
\>[0]\AgdaBound{a}\AgdaSpace{}%
\AgdaOperator{\AgdaFunction{⊗}}\AgdaSpace{}%
\AgdaSymbol{(}\AgdaInductiveConstructor{swap}\AgdaSpace{}%
\AgdaBound{x}\AgdaSpace{}%
\AgdaBound{y}\AgdaSpace{}%
\AgdaBound{z}\AgdaSpace{}%
\AgdaBound{i}\AgdaSymbol{)}%
\>[21]\AgdaSymbol{=}\AgdaSpace{}%
\AgdaFunction{⊕swap}\AgdaSpace{}%
\AgdaSymbol{(}\AgdaBound{a}\AgdaSpace{}%
\AgdaOperator{\AgdaFunction{∔}}\AgdaSpace{}%
\AgdaBound{x}\AgdaSymbol{)}\AgdaSpace{}%
\AgdaSymbol{(}\AgdaBound{a}\AgdaSpace{}%
\AgdaOperator{\AgdaFunction{∔}}\AgdaSpace{}%
\AgdaBound{y}\AgdaSymbol{)}\AgdaSpace{}%
\AgdaSymbol{(}\AgdaBound{a}\AgdaSpace{}%
\AgdaOperator{\AgdaFunction{⊗}}\AgdaSpace{}%
\AgdaBound{z}\AgdaSymbol{)}\AgdaSpace{}%
\AgdaBound{i}\<%
\\
\>[0]\AgdaBound{a}\AgdaSpace{}%
\AgdaOperator{\AgdaFunction{⊗}}\AgdaSpace{}%
\AgdaSymbol{(}\AgdaInductiveConstructor{trunc}\AgdaSpace{}%
\AgdaBound{p}\AgdaSpace{}%
\AgdaBound{q}\AgdaSpace{}%
\AgdaBound{i}\AgdaSpace{}%
\AgdaBound{j}\AgdaSymbol{)}%
\>[21]\AgdaSymbol{=}\AgdaSpace{}%
\AgdaInductiveConstructor{trunc}\AgdaSpace{}%
\AgdaSymbol{(}\AgdaFunction{cong}\AgdaSpace{}%
\AgdaSymbol{(}\AgdaBound{a}\AgdaSpace{}%
\AgdaOperator{\AgdaFunction{⊗\AgdaUnderscore{}}}\AgdaSymbol{)}\AgdaSpace{}%
\AgdaBound{p}\AgdaSymbol{)}\AgdaSpace{}%
\AgdaSymbol{(}\AgdaFunction{cong}\AgdaSpace{}%
\AgdaSymbol{(}\AgdaBound{a}\AgdaSpace{}%
\AgdaOperator{\AgdaFunction{⊗\AgdaUnderscore{}}}\AgdaSymbol{)}\AgdaSpace{}%
\AgdaBound{q}\AgdaSymbol{)}\AgdaSpace{}%
\AgdaBound{i}\AgdaSpace{}%
\AgdaBound{j}\<%
\end{code}}

\newcommand{\MHProd}{
\begin{code}%
\>[0]\AgdaOperator{\AgdaFunction{\AgdaUnderscore{}⊗ᴹ\AgdaUnderscore{}}}\AgdaSpace{}%
\AgdaSymbol{:}\AgdaSpace{}%
\AgdaDatatype{MutualOrd}\AgdaSpace{}%
\AgdaSymbol{→}\AgdaSpace{}%
\AgdaDatatype{MutualOrd}\AgdaSpace{}%
\AgdaSymbol{→}\AgdaSpace{}%
\AgdaDatatype{MutualOrd}\<%
\\
\>[0]\AgdaOperator{\AgdaFunction{\AgdaUnderscore{}⊗ᴹ\AgdaUnderscore{}}}\AgdaSpace{}%
\AgdaSymbol{=}\AgdaSpace{}%
\AgdaFunction{transport}\AgdaSpace{}%
\AgdaSymbol{(λ}\AgdaSpace{}%
\AgdaBound{i}\AgdaSpace{}%
\AgdaSymbol{→}\AgdaSpace{}%
\AgdaFunction{H≡M}\AgdaSpace{}%
\AgdaBound{i}\AgdaSpace{}%
\AgdaSymbol{→}\AgdaSpace{}%
\AgdaFunction{H≡M}\AgdaSpace{}%
\AgdaBound{i}\AgdaSpace{}%
\AgdaSymbol{→}\AgdaSpace{}%
\AgdaFunction{H≡M}\AgdaSpace{}%
\AgdaBound{i}\AgdaSymbol{)}\AgdaSpace{}%
\AgdaOperator{\AgdaFunction{\AgdaUnderscore{}⊗\AgdaUnderscore{}}}\<%
\end{code}}

\newcommand{\MPlusType}{
\begin{code}%
\>[0]\AgdaOperator{\AgdaFunction{\AgdaUnderscore{}+\AgdaUnderscore{}}}%
\>[7]\AgdaSymbol{:}\AgdaSpace{}%
\AgdaDatatype{MutualOrd}\AgdaSpace{}%
\AgdaSymbol{→}\AgdaSpace{}%
\AgdaDatatype{MutualOrd}\AgdaSpace{}%
\AgdaSymbol{→}\AgdaSpace{}%
\AgdaDatatype{MutualOrd}\<%
\\
\>[0]\AgdaFunction{≥fst+}%
\>[7]\AgdaSymbol{:}\AgdaSpace{}%
\AgdaSymbol{\{}\AgdaBound{a}\AgdaSpace{}%
\AgdaSymbol{:}\AgdaSpace{}%
\AgdaDatatype{MutualOrd}\AgdaSymbol{\}}\AgdaSpace{}%
\AgdaSymbol{(}\AgdaBound{b}\AgdaSpace{}%
\AgdaBound{c}\AgdaSpace{}%
\AgdaSymbol{:}\AgdaSpace{}%
\AgdaDatatype{MutualOrd}\AgdaSymbol{)}\<%
\\
\>[7]\AgdaSymbol{→}\AgdaSpace{}%
\AgdaBound{a}\AgdaSpace{}%
\AgdaOperator{\AgdaFunction{≥}}\AgdaSpace{}%
\AgdaFunction{fst}\AgdaSpace{}%
\AgdaBound{b}\AgdaSpace{}%
\AgdaSymbol{→}\AgdaSpace{}%
\AgdaBound{a}\AgdaSpace{}%
\AgdaOperator{\AgdaFunction{≥}}\AgdaSpace{}%
\AgdaFunction{fst}\AgdaSpace{}%
\AgdaBound{c}\AgdaSpace{}%
\AgdaSymbol{→}\AgdaSpace{}%
\AgdaBound{a}\AgdaSpace{}%
\AgdaOperator{\AgdaFunction{≥}}\AgdaSpace{}%
\AgdaFunction{fst}\AgdaSpace{}%
\AgdaSymbol{(}\AgdaBound{b}\AgdaSpace{}%
\AgdaOperator{\AgdaFunction{+}}\AgdaSpace{}%
\AgdaBound{c}\AgdaSymbol{)}\<%
\end{code}}

\newcommand{\MPlusDef}{
\begin{code}%
\>[0]\AgdaInductiveConstructor{𝟎}\AgdaSpace{}%
\AgdaOperator{\AgdaFunction{+}}\AgdaSpace{}%
\AgdaBound{b}\AgdaSpace{}%
\AgdaSymbol{=}\AgdaSpace{}%
\AgdaBound{b}\<%
\\
\>[0]\AgdaCatchallClause{\AgdaBound{a}}\AgdaSpace{}%
\AgdaCatchallClause{\AgdaOperator{\AgdaFunction{+}}}\AgdaSpace{}%
\AgdaCatchallClause{\AgdaInductiveConstructor{𝟎}}\AgdaSpace{}%
\AgdaSymbol{=}\AgdaSpace{}%
\AgdaBound{a}\<%
\\
\>[0]\AgdaSymbol{(}\AgdaOperator{\AgdaInductiveConstructor{ω\textasciicircum{}}}\AgdaSpace{}%
\AgdaBound{a}\AgdaSpace{}%
\AgdaOperator{\AgdaInductiveConstructor{+}}\AgdaSpace{}%
\AgdaBound{c}\AgdaSpace{}%
\AgdaOperator{\AgdaInductiveConstructor{[}}\AgdaSpace{}%
\AgdaBound{r}\AgdaSpace{}%
\AgdaOperator{\AgdaInductiveConstructor{]}}\AgdaSymbol{)}\AgdaSpace{}%
\AgdaOperator{\AgdaFunction{+}}\AgdaSpace{}%
\AgdaSymbol{(}\AgdaOperator{\AgdaInductiveConstructor{ω\textasciicircum{}}}\AgdaSpace{}%
\AgdaBound{b}\AgdaSpace{}%
\AgdaOperator{\AgdaInductiveConstructor{+}}\AgdaSpace{}%
\AgdaBound{d}\AgdaSpace{}%
\AgdaOperator{\AgdaInductiveConstructor{[}}\AgdaSpace{}%
\AgdaBound{s}\AgdaSpace{}%
\AgdaOperator{\AgdaInductiveConstructor{]}}\AgdaSymbol{)}\AgdaSpace{}%
\AgdaKeyword{with}\AgdaSpace{}%
\AgdaFunction{<-tri}\AgdaSpace{}%
\AgdaBound{a}\AgdaSpace{}%
\AgdaBound{b}\<%
\\
\>[0]\AgdaSymbol{...}\AgdaSpace{}%
\AgdaSymbol{|}\AgdaSpace{}%
\AgdaInductiveConstructor{inj₁}\AgdaSpace{}%
\AgdaBound{a<b}%
\>[16]\AgdaSymbol{=}\AgdaSpace{}%
\AgdaOperator{\AgdaInductiveConstructor{ω\textasciicircum{}}}\AgdaSpace{}%
\AgdaBound{b}\AgdaSpace{}%
\AgdaOperator{\AgdaInductiveConstructor{+}}\AgdaSpace{}%
\AgdaBound{d}\AgdaSpace{}%
\AgdaOperator{\AgdaInductiveConstructor{[}}\AgdaSpace{}%
\AgdaBound{s}\AgdaSpace{}%
\AgdaOperator{\AgdaInductiveConstructor{]}}\<%
\\
\>[0]\AgdaSymbol{...}\AgdaSpace{}%
\AgdaSymbol{|}\AgdaSpace{}%
\AgdaInductiveConstructor{inj₂}\AgdaSpace{}%
\AgdaBound{a≥b}%
\>[16]\AgdaSymbol{=}\AgdaSpace{}%
\AgdaOperator{\AgdaInductiveConstructor{ω\textasciicircum{}}}\AgdaSpace{}%
\AgdaBound{a}\AgdaSpace{}%
\AgdaOperator{\AgdaInductiveConstructor{+}}\AgdaSpace{}%
\AgdaSymbol{(}\AgdaBound{c}\AgdaSpace{}%
\AgdaOperator{\AgdaFunction{+}}\AgdaSpace{}%
\AgdaOperator{\AgdaInductiveConstructor{ω\textasciicircum{}}}\AgdaSpace{}%
\AgdaBound{b}\AgdaSpace{}%
\AgdaOperator{\AgdaInductiveConstructor{+}}\AgdaSpace{}%
\AgdaBound{d}\AgdaSpace{}%
\AgdaOperator{\AgdaInductiveConstructor{[}}\AgdaSpace{}%
\AgdaBound{s}\AgdaSpace{}%
\AgdaOperator{\AgdaInductiveConstructor{]}}\AgdaSymbol{)}\AgdaSpace{}%
\AgdaOperator{\AgdaInductiveConstructor{[}}\AgdaSpace{}%
\AgdaFunction{≥fst+}\AgdaSpace{}%
\AgdaBound{c}\AgdaSpace{}%
\AgdaSymbol{\AgdaUnderscore{}}\AgdaSpace{}%
\AgdaBound{r}\AgdaSpace{}%
\AgdaBound{a≥b}\AgdaSpace{}%
\AgdaOperator{\AgdaInductiveConstructor{]}}\<%
\end{code}}

\newcommand{\MPlusProof}{
\begin{code}%
\>[0]\AgdaFunction{≥fst+}\AgdaSpace{}%
\AgdaInductiveConstructor{𝟎}\AgdaSpace{}%
\AgdaSymbol{\AgdaUnderscore{}}\AgdaSpace{}%
\AgdaBound{r}\AgdaSpace{}%
\AgdaBound{s}\AgdaSpace{}%
\AgdaSymbol{=}\AgdaSpace{}%
\AgdaBound{s}\<%
\\
\>[0]\AgdaFunction{≥fst+}\AgdaSpace{}%
\AgdaSymbol{(}\AgdaOperator{\AgdaInductiveConstructor{ω\textasciicircum{}}}\AgdaSpace{}%
\AgdaSymbol{\AgdaUnderscore{}}\AgdaSpace{}%
\AgdaOperator{\AgdaInductiveConstructor{+}}\AgdaSpace{}%
\AgdaSymbol{\AgdaUnderscore{}}\AgdaSpace{}%
\AgdaOperator{\AgdaInductiveConstructor{[}}\AgdaSpace{}%
\AgdaSymbol{\AgdaUnderscore{}}\AgdaSpace{}%
\AgdaOperator{\AgdaInductiveConstructor{]}}\AgdaSymbol{)}\AgdaSpace{}%
\AgdaInductiveConstructor{𝟎}\AgdaSpace{}%
\AgdaBound{r}\AgdaSpace{}%
\AgdaBound{s}\AgdaSpace{}%
\AgdaSymbol{=}\AgdaSpace{}%
\AgdaBound{r}\<%
\\
\>[0]\AgdaFunction{≥fst+}\AgdaSpace{}%
\AgdaSymbol{(}\AgdaOperator{\AgdaInductiveConstructor{ω\textasciicircum{}}}\AgdaSpace{}%
\AgdaBound{b}\AgdaSpace{}%
\AgdaOperator{\AgdaInductiveConstructor{+}}\AgdaSpace{}%
\AgdaSymbol{\AgdaUnderscore{}}\AgdaSpace{}%
\AgdaOperator{\AgdaInductiveConstructor{[}}\AgdaSpace{}%
\AgdaSymbol{\AgdaUnderscore{}}\AgdaSpace{}%
\AgdaOperator{\AgdaInductiveConstructor{]}}\AgdaSymbol{)}\AgdaSpace{}%
\AgdaSymbol{(}\AgdaOperator{\AgdaInductiveConstructor{ω\textasciicircum{}}}\AgdaSpace{}%
\AgdaBound{c}\AgdaSpace{}%
\AgdaOperator{\AgdaInductiveConstructor{+}}\AgdaSpace{}%
\AgdaSymbol{\AgdaUnderscore{}}\AgdaSpace{}%
\AgdaOperator{\AgdaInductiveConstructor{[}}\AgdaSpace{}%
\AgdaSymbol{\AgdaUnderscore{}}\AgdaSpace{}%
\AgdaOperator{\AgdaInductiveConstructor{]}}\AgdaSymbol{)}\AgdaSpace{}%
\AgdaBound{r}\AgdaSpace{}%
\AgdaBound{s}\AgdaSpace{}%
\AgdaKeyword{with}\AgdaSpace{}%
\AgdaFunction{<-tri}\AgdaSpace{}%
\AgdaBound{b}\AgdaSpace{}%
\AgdaBound{c}\<%
\\
\>[0]\AgdaSymbol{...}\AgdaSpace{}%
\AgdaSymbol{|}\AgdaSpace{}%
\AgdaInductiveConstructor{inj₁}\AgdaSpace{}%
\AgdaBound{b<c}%
\>[16]\AgdaSymbol{=}\AgdaSpace{}%
\AgdaBound{s}\<%
\\
\>[0]\AgdaSymbol{...}\AgdaSpace{}%
\AgdaSymbol{|}\AgdaSpace{}%
\AgdaInductiveConstructor{inj₂}\AgdaSpace{}%
\AgdaBound{b≥c}%
\>[16]\AgdaSymbol{=}\AgdaSpace{}%
\AgdaBound{r}\<%
\end{code}}

\begin{code}[hide]%
\>[0]\<%
\\
\>[0]\AgdaFunction{+unitl}\AgdaSpace{}%
\AgdaSymbol{:}\AgdaSpace{}%
\AgdaSymbol{(}\AgdaBound{a}\AgdaSpace{}%
\AgdaSymbol{:}\AgdaSpace{}%
\AgdaDatatype{MutualOrd}\AgdaSymbol{)}\AgdaSpace{}%
\AgdaSymbol{→}\AgdaSpace{}%
\AgdaInductiveConstructor{𝟎}\AgdaSpace{}%
\AgdaOperator{\AgdaFunction{+}}\AgdaSpace{}%
\AgdaBound{a}\AgdaSpace{}%
\AgdaOperator{\AgdaFunction{≡}}\AgdaSpace{}%
\AgdaBound{a}\<%
\\
\>[0]\AgdaFunction{+unitl}\AgdaSpace{}%
\AgdaBound{a}\AgdaSpace{}%
\AgdaSymbol{=}\AgdaSpace{}%
\AgdaFunction{refl}\<%
\\
\\[\AgdaEmptyExtraSkip]%
\>[0]\AgdaFunction{+unitr}\AgdaSpace{}%
\AgdaSymbol{:}\AgdaSpace{}%
\AgdaSymbol{(}\AgdaBound{a}\AgdaSpace{}%
\AgdaSymbol{:}\AgdaSpace{}%
\AgdaDatatype{MutualOrd}\AgdaSymbol{)}\AgdaSpace{}%
\AgdaSymbol{→}\AgdaSpace{}%
\AgdaBound{a}\AgdaSpace{}%
\AgdaOperator{\AgdaFunction{+}}\AgdaSpace{}%
\AgdaInductiveConstructor{𝟎}\AgdaSpace{}%
\AgdaOperator{\AgdaFunction{≡}}\AgdaSpace{}%
\AgdaBound{a}\<%
\\
\>[0]\AgdaFunction{+unitr}\AgdaSpace{}%
\AgdaInductiveConstructor{𝟎}\AgdaSpace{}%
\AgdaSymbol{=}\AgdaSpace{}%
\AgdaFunction{refl}\<%
\\
\>[0]\AgdaFunction{+unitr}\AgdaSpace{}%
\AgdaOperator{\AgdaInductiveConstructor{ω\textasciicircum{}}}\AgdaSpace{}%
\AgdaBound{a}\AgdaSpace{}%
\AgdaOperator{\AgdaInductiveConstructor{+}}\AgdaSpace{}%
\AgdaBound{b}\AgdaSpace{}%
\AgdaOperator{\AgdaInductiveConstructor{[}}\AgdaSpace{}%
\AgdaBound{r}\AgdaSpace{}%
\AgdaOperator{\AgdaInductiveConstructor{]}}\AgdaSpace{}%
\AgdaSymbol{=}\AgdaSpace{}%
\AgdaFunction{refl}\<%
\\
\>[0]\<%
\end{code}

\newcommand{\MPlusAssoc}{
\begin{code}%
\>[0]\AgdaFunction{+assoc}\AgdaSpace{}%
\AgdaSymbol{:}\AgdaSpace{}%
\AgdaFunction{Assoc}\AgdaSpace{}%
\AgdaDatatype{MutualOrd}\AgdaSpace{}%
\AgdaOperator{\AgdaFunction{\AgdaUnderscore{}+\AgdaUnderscore{}}}\<%
\end{code}}

\begin{code}[hide]%
\>[0]\<%
\\
\>[0]\AgdaFunction{+assoc}\AgdaSpace{}%
\AgdaInductiveConstructor{𝟎}\AgdaSpace{}%
\AgdaBound{b}\AgdaSpace{}%
\AgdaBound{c}\AgdaSpace{}%
\AgdaSymbol{=}\AgdaSpace{}%
\AgdaFunction{refl}\<%
\\
\>[0]\AgdaFunction{+assoc}\AgdaSpace{}%
\AgdaOperator{\AgdaInductiveConstructor{ω\textasciicircum{}}}\AgdaSpace{}%
\AgdaBound{a}\AgdaSpace{}%
\AgdaOperator{\AgdaInductiveConstructor{+}}\AgdaSpace{}%
\AgdaBound{a'}\AgdaSpace{}%
\AgdaOperator{\AgdaInductiveConstructor{[}}\AgdaSpace{}%
\AgdaSymbol{\AgdaUnderscore{}}\AgdaSpace{}%
\AgdaOperator{\AgdaInductiveConstructor{]}}\AgdaSpace{}%
\AgdaInductiveConstructor{𝟎}\AgdaSpace{}%
\AgdaBound{c}\AgdaSpace{}%
\AgdaSymbol{=}\AgdaSpace{}%
\AgdaFunction{refl}\<%
\\
\>[0]\AgdaFunction{+assoc}\AgdaSpace{}%
\AgdaOperator{\AgdaInductiveConstructor{ω\textasciicircum{}}}\AgdaSpace{}%
\AgdaBound{a}\AgdaSpace{}%
\AgdaOperator{\AgdaInductiveConstructor{+}}\AgdaSpace{}%
\AgdaBound{a'}\AgdaSpace{}%
\AgdaOperator{\AgdaInductiveConstructor{[}}\AgdaSpace{}%
\AgdaSymbol{\AgdaUnderscore{}}\AgdaSpace{}%
\AgdaOperator{\AgdaInductiveConstructor{]}}\AgdaSpace{}%
\AgdaOperator{\AgdaInductiveConstructor{ω\textasciicircum{}}}\AgdaSpace{}%
\AgdaBound{b}\AgdaSpace{}%
\AgdaOperator{\AgdaInductiveConstructor{+}}\AgdaSpace{}%
\AgdaBound{b'}\AgdaSpace{}%
\AgdaOperator{\AgdaInductiveConstructor{[}}\AgdaSpace{}%
\AgdaSymbol{\AgdaUnderscore{}}\AgdaSpace{}%
\AgdaOperator{\AgdaInductiveConstructor{]}}\AgdaSpace{}%
\AgdaInductiveConstructor{𝟎}\AgdaSpace{}%
\AgdaSymbol{=}\AgdaSpace{}%
\AgdaSymbol{(}\AgdaFunction{+unitr}\AgdaSpace{}%
\AgdaSymbol{\AgdaUnderscore{})}\AgdaSpace{}%
\AgdaOperator{\AgdaFunction{⁻¹}}\<%
\\
\>[0]\AgdaFunction{+assoc}\AgdaSpace{}%
\AgdaOperator{\AgdaInductiveConstructor{ω\textasciicircum{}}}\AgdaSpace{}%
\AgdaBound{a}\AgdaSpace{}%
\AgdaOperator{\AgdaInductiveConstructor{+}}\AgdaSpace{}%
\AgdaBound{a'}\AgdaSpace{}%
\AgdaOperator{\AgdaInductiveConstructor{[}}\AgdaSpace{}%
\AgdaSymbol{\AgdaUnderscore{}}\AgdaSpace{}%
\AgdaOperator{\AgdaInductiveConstructor{]}}\AgdaSpace{}%
\AgdaOperator{\AgdaInductiveConstructor{ω\textasciicircum{}}}\AgdaSpace{}%
\AgdaBound{b}\AgdaSpace{}%
\AgdaOperator{\AgdaInductiveConstructor{+}}\AgdaSpace{}%
\AgdaBound{b'}\AgdaSpace{}%
\AgdaOperator{\AgdaInductiveConstructor{[}}\AgdaSpace{}%
\AgdaSymbol{\AgdaUnderscore{}}\AgdaSpace{}%
\AgdaOperator{\AgdaInductiveConstructor{]}}\AgdaSpace{}%
\AgdaOperator{\AgdaInductiveConstructor{ω\textasciicircum{}}}\AgdaSpace{}%
\AgdaBound{c}\AgdaSpace{}%
\AgdaOperator{\AgdaInductiveConstructor{+}}\AgdaSpace{}%
\AgdaBound{c'}\AgdaSpace{}%
\AgdaOperator{\AgdaInductiveConstructor{[}}\AgdaSpace{}%
\AgdaSymbol{\AgdaUnderscore{}}\AgdaSpace{}%
\AgdaOperator{\AgdaInductiveConstructor{]}}\AgdaSpace{}%
\AgdaKeyword{with}\AgdaSpace{}%
\AgdaFunction{<-tri}\AgdaSpace{}%
\AgdaBound{a}\AgdaSpace{}%
\AgdaBound{b}\AgdaSpace{}%
\AgdaSymbol{|}\AgdaSpace{}%
\AgdaFunction{<-tri}\AgdaSpace{}%
\AgdaBound{b}\AgdaSpace{}%
\AgdaBound{c}\<%
\\
\>[0]\AgdaFunction{+assoc}\AgdaSpace{}%
\AgdaOperator{\AgdaInductiveConstructor{ω\textasciicircum{}}}\AgdaSpace{}%
\AgdaBound{a}\AgdaSpace{}%
\AgdaOperator{\AgdaInductiveConstructor{+}}\AgdaSpace{}%
\AgdaBound{a'}\AgdaSpace{}%
\AgdaOperator{\AgdaInductiveConstructor{[}}\AgdaSpace{}%
\AgdaSymbol{\AgdaUnderscore{}}\AgdaSpace{}%
\AgdaOperator{\AgdaInductiveConstructor{]}}\AgdaSpace{}%
\AgdaOperator{\AgdaInductiveConstructor{ω\textasciicircum{}}}\AgdaSpace{}%
\AgdaBound{b}\AgdaSpace{}%
\AgdaOperator{\AgdaInductiveConstructor{+}}\AgdaSpace{}%
\AgdaBound{b'}\AgdaSpace{}%
\AgdaOperator{\AgdaInductiveConstructor{[}}\AgdaSpace{}%
\AgdaSymbol{\AgdaUnderscore{}}\AgdaSpace{}%
\AgdaOperator{\AgdaInductiveConstructor{]}}\AgdaSpace{}%
\AgdaOperator{\AgdaInductiveConstructor{ω\textasciicircum{}}}\AgdaSpace{}%
\AgdaBound{c}\AgdaSpace{}%
\AgdaOperator{\AgdaInductiveConstructor{+}}\AgdaSpace{}%
\AgdaBound{c'}\AgdaSpace{}%
\AgdaOperator{\AgdaInductiveConstructor{[}}\AgdaSpace{}%
\AgdaSymbol{\AgdaUnderscore{}}\AgdaSpace{}%
\AgdaOperator{\AgdaInductiveConstructor{]}}\AgdaSpace{}%
\AgdaSymbol{|}\AgdaSpace{}%
\AgdaInductiveConstructor{inj₁}\AgdaSpace{}%
\AgdaBound{a<b}\AgdaSpace{}%
\AgdaSymbol{|}\AgdaSpace{}%
\AgdaInductiveConstructor{inj₁}\AgdaSpace{}%
\AgdaBound{b<c}\AgdaSpace{}%
\AgdaKeyword{with}\AgdaSpace{}%
\AgdaFunction{<-tri}\AgdaSpace{}%
\AgdaBound{b}\AgdaSpace{}%
\AgdaBound{c}\<%
\\
\>[0]\AgdaFunction{+assoc}\AgdaSpace{}%
\AgdaOperator{\AgdaInductiveConstructor{ω\textasciicircum{}}}\AgdaSpace{}%
\AgdaBound{a}\AgdaSpace{}%
\AgdaOperator{\AgdaInductiveConstructor{+}}\AgdaSpace{}%
\AgdaBound{a'}\AgdaSpace{}%
\AgdaOperator{\AgdaInductiveConstructor{[}}\AgdaSpace{}%
\AgdaSymbol{\AgdaUnderscore{}}\AgdaSpace{}%
\AgdaOperator{\AgdaInductiveConstructor{]}}\AgdaSpace{}%
\AgdaOperator{\AgdaInductiveConstructor{ω\textasciicircum{}}}\AgdaSpace{}%
\AgdaBound{b}\AgdaSpace{}%
\AgdaOperator{\AgdaInductiveConstructor{+}}\AgdaSpace{}%
\AgdaBound{b'}\AgdaSpace{}%
\AgdaOperator{\AgdaInductiveConstructor{[}}\AgdaSpace{}%
\AgdaSymbol{\AgdaUnderscore{}}\AgdaSpace{}%
\AgdaOperator{\AgdaInductiveConstructor{]}}\AgdaSpace{}%
\AgdaOperator{\AgdaInductiveConstructor{ω\textasciicircum{}}}\AgdaSpace{}%
\AgdaBound{c}\AgdaSpace{}%
\AgdaOperator{\AgdaInductiveConstructor{+}}\AgdaSpace{}%
\AgdaBound{c'}\AgdaSpace{}%
\AgdaOperator{\AgdaInductiveConstructor{[}}\AgdaSpace{}%
\AgdaSymbol{\AgdaUnderscore{}}\AgdaSpace{}%
\AgdaOperator{\AgdaInductiveConstructor{]}}\AgdaSpace{}%
\AgdaSymbol{|}\AgdaSpace{}%
\AgdaInductiveConstructor{inj₁}\AgdaSpace{}%
\AgdaBound{a<b}\AgdaSpace{}%
\AgdaSymbol{|}\AgdaSpace{}%
\AgdaInductiveConstructor{inj₁}\AgdaSpace{}%
\AgdaBound{b<c}\AgdaSpace{}%
\AgdaSymbol{|}\AgdaSpace{}%
\AgdaInductiveConstructor{inj₁}\AgdaSpace{}%
\AgdaSymbol{\AgdaUnderscore{}}%
\>[88]\AgdaKeyword{with}\AgdaSpace{}%
\AgdaFunction{<-tri}\AgdaSpace{}%
\AgdaBound{a}\AgdaSpace{}%
\AgdaBound{c}\<%
\\
\>[0]\AgdaFunction{+assoc}\AgdaSpace{}%
\AgdaOperator{\AgdaInductiveConstructor{ω\textasciicircum{}}}\AgdaSpace{}%
\AgdaBound{a}\AgdaSpace{}%
\AgdaOperator{\AgdaInductiveConstructor{+}}\AgdaSpace{}%
\AgdaBound{a'}\AgdaSpace{}%
\AgdaOperator{\AgdaInductiveConstructor{[}}\AgdaSpace{}%
\AgdaSymbol{\AgdaUnderscore{}}\AgdaSpace{}%
\AgdaOperator{\AgdaInductiveConstructor{]}}\AgdaSpace{}%
\AgdaOperator{\AgdaInductiveConstructor{ω\textasciicircum{}}}\AgdaSpace{}%
\AgdaBound{b}\AgdaSpace{}%
\AgdaOperator{\AgdaInductiveConstructor{+}}\AgdaSpace{}%
\AgdaBound{b'}\AgdaSpace{}%
\AgdaOperator{\AgdaInductiveConstructor{[}}\AgdaSpace{}%
\AgdaSymbol{\AgdaUnderscore{}}\AgdaSpace{}%
\AgdaOperator{\AgdaInductiveConstructor{]}}\AgdaSpace{}%
\AgdaOperator{\AgdaInductiveConstructor{ω\textasciicircum{}}}\AgdaSpace{}%
\AgdaBound{c}\AgdaSpace{}%
\AgdaOperator{\AgdaInductiveConstructor{+}}\AgdaSpace{}%
\AgdaBound{c'}\AgdaSpace{}%
\AgdaOperator{\AgdaInductiveConstructor{[}}\AgdaSpace{}%
\AgdaSymbol{\AgdaUnderscore{}}\AgdaSpace{}%
\AgdaOperator{\AgdaInductiveConstructor{]}}\AgdaSpace{}%
\AgdaSymbol{|}\AgdaSpace{}%
\AgdaInductiveConstructor{inj₁}\AgdaSpace{}%
\AgdaBound{a<b}\AgdaSpace{}%
\AgdaSymbol{|}\AgdaSpace{}%
\AgdaInductiveConstructor{inj₁}\AgdaSpace{}%
\AgdaBound{b<c}\AgdaSpace{}%
\AgdaSymbol{|}\AgdaSpace{}%
\AgdaInductiveConstructor{inj₁}\AgdaSpace{}%
\AgdaSymbol{\AgdaUnderscore{}}%
\>[88]\AgdaSymbol{|}\AgdaSpace{}%
\AgdaInductiveConstructor{inj₁}\AgdaSpace{}%
\AgdaSymbol{\AgdaUnderscore{}}%
\>[99]\AgdaSymbol{=}\AgdaSpace{}%
\AgdaFunction{refl}\<%
\\
\>[0]\AgdaFunction{+assoc}\AgdaSpace{}%
\AgdaOperator{\AgdaInductiveConstructor{ω\textasciicircum{}}}\AgdaSpace{}%
\AgdaBound{a}\AgdaSpace{}%
\AgdaOperator{\AgdaInductiveConstructor{+}}\AgdaSpace{}%
\AgdaBound{a'}\AgdaSpace{}%
\AgdaOperator{\AgdaInductiveConstructor{[}}\AgdaSpace{}%
\AgdaSymbol{\AgdaUnderscore{}}\AgdaSpace{}%
\AgdaOperator{\AgdaInductiveConstructor{]}}\AgdaSpace{}%
\AgdaOperator{\AgdaInductiveConstructor{ω\textasciicircum{}}}\AgdaSpace{}%
\AgdaBound{b}\AgdaSpace{}%
\AgdaOperator{\AgdaInductiveConstructor{+}}\AgdaSpace{}%
\AgdaBound{b'}\AgdaSpace{}%
\AgdaOperator{\AgdaInductiveConstructor{[}}\AgdaSpace{}%
\AgdaSymbol{\AgdaUnderscore{}}\AgdaSpace{}%
\AgdaOperator{\AgdaInductiveConstructor{]}}\AgdaSpace{}%
\AgdaOperator{\AgdaInductiveConstructor{ω\textasciicircum{}}}\AgdaSpace{}%
\AgdaBound{c}\AgdaSpace{}%
\AgdaOperator{\AgdaInductiveConstructor{+}}\AgdaSpace{}%
\AgdaBound{c'}\AgdaSpace{}%
\AgdaOperator{\AgdaInductiveConstructor{[}}\AgdaSpace{}%
\AgdaSymbol{\AgdaUnderscore{}}\AgdaSpace{}%
\AgdaOperator{\AgdaInductiveConstructor{]}}\AgdaSpace{}%
\AgdaSymbol{|}\AgdaSpace{}%
\AgdaInductiveConstructor{inj₁}\AgdaSpace{}%
\AgdaBound{a<b}\AgdaSpace{}%
\AgdaSymbol{|}\AgdaSpace{}%
\AgdaInductiveConstructor{inj₁}\AgdaSpace{}%
\AgdaBound{b<c}\AgdaSpace{}%
\AgdaSymbol{|}\AgdaSpace{}%
\AgdaInductiveConstructor{inj₁}\AgdaSpace{}%
\AgdaSymbol{\AgdaUnderscore{}}%
\>[88]\AgdaSymbol{|}\AgdaSpace{}%
\AgdaInductiveConstructor{inj₂}\AgdaSpace{}%
\AgdaBound{a≥c}\AgdaSpace{}%
\AgdaSymbol{=}\AgdaSpace{}%
\AgdaFunction{⊥-elim}\AgdaSpace{}%
\AgdaSymbol{(}\AgdaFunction{Lm[≥→¬<]}\AgdaSpace{}%
\AgdaBound{a≥c}\AgdaSpace{}%
\AgdaSymbol{(}\AgdaFunction{<-trans}\AgdaSpace{}%
\AgdaBound{a<b}\AgdaSpace{}%
\AgdaBound{b<c}\AgdaSymbol{))}\<%
\\
\>[0]\AgdaFunction{+assoc}\AgdaSpace{}%
\AgdaOperator{\AgdaInductiveConstructor{ω\textasciicircum{}}}\AgdaSpace{}%
\AgdaBound{a}\AgdaSpace{}%
\AgdaOperator{\AgdaInductiveConstructor{+}}\AgdaSpace{}%
\AgdaBound{a'}\AgdaSpace{}%
\AgdaOperator{\AgdaInductiveConstructor{[}}\AgdaSpace{}%
\AgdaSymbol{\AgdaUnderscore{}}\AgdaSpace{}%
\AgdaOperator{\AgdaInductiveConstructor{]}}\AgdaSpace{}%
\AgdaOperator{\AgdaInductiveConstructor{ω\textasciicircum{}}}\AgdaSpace{}%
\AgdaBound{b}\AgdaSpace{}%
\AgdaOperator{\AgdaInductiveConstructor{+}}\AgdaSpace{}%
\AgdaBound{b'}\AgdaSpace{}%
\AgdaOperator{\AgdaInductiveConstructor{[}}\AgdaSpace{}%
\AgdaSymbol{\AgdaUnderscore{}}\AgdaSpace{}%
\AgdaOperator{\AgdaInductiveConstructor{]}}\AgdaSpace{}%
\AgdaOperator{\AgdaInductiveConstructor{ω\textasciicircum{}}}\AgdaSpace{}%
\AgdaBound{c}\AgdaSpace{}%
\AgdaOperator{\AgdaInductiveConstructor{+}}\AgdaSpace{}%
\AgdaBound{c'}\AgdaSpace{}%
\AgdaOperator{\AgdaInductiveConstructor{[}}\AgdaSpace{}%
\AgdaSymbol{\AgdaUnderscore{}}\AgdaSpace{}%
\AgdaOperator{\AgdaInductiveConstructor{]}}\AgdaSpace{}%
\AgdaSymbol{|}\AgdaSpace{}%
\AgdaInductiveConstructor{inj₁}\AgdaSpace{}%
\AgdaBound{a<b}\AgdaSpace{}%
\AgdaSymbol{|}\AgdaSpace{}%
\AgdaInductiveConstructor{inj₁}\AgdaSpace{}%
\AgdaBound{b<c}\AgdaSpace{}%
\AgdaSymbol{|}\AgdaSpace{}%
\AgdaInductiveConstructor{inj₂}\AgdaSpace{}%
\AgdaBound{b≥c}\AgdaSpace{}%
\AgdaSymbol{=}\AgdaSpace{}%
\AgdaFunction{⊥-elim}\AgdaSpace{}%
\AgdaSymbol{(}\AgdaFunction{Lm[≥→¬<]}\AgdaSpace{}%
\AgdaBound{b≥c}\AgdaSpace{}%
\AgdaBound{b<c}\AgdaSymbol{)}\<%
\\
\>[0]\AgdaFunction{+assoc}\AgdaSpace{}%
\AgdaOperator{\AgdaInductiveConstructor{ω\textasciicircum{}}}\AgdaSpace{}%
\AgdaBound{a}\AgdaSpace{}%
\AgdaOperator{\AgdaInductiveConstructor{+}}\AgdaSpace{}%
\AgdaBound{a'}\AgdaSpace{}%
\AgdaOperator{\AgdaInductiveConstructor{[}}\AgdaSpace{}%
\AgdaSymbol{\AgdaUnderscore{}}\AgdaSpace{}%
\AgdaOperator{\AgdaInductiveConstructor{]}}\AgdaSpace{}%
\AgdaOperator{\AgdaInductiveConstructor{ω\textasciicircum{}}}\AgdaSpace{}%
\AgdaBound{b}\AgdaSpace{}%
\AgdaOperator{\AgdaInductiveConstructor{+}}\AgdaSpace{}%
\AgdaBound{b'}\AgdaSpace{}%
\AgdaOperator{\AgdaInductiveConstructor{[}}\AgdaSpace{}%
\AgdaSymbol{\AgdaUnderscore{}}\AgdaSpace{}%
\AgdaOperator{\AgdaInductiveConstructor{]}}\AgdaSpace{}%
\AgdaOperator{\AgdaInductiveConstructor{ω\textasciicircum{}}}\AgdaSpace{}%
\AgdaBound{c}\AgdaSpace{}%
\AgdaOperator{\AgdaInductiveConstructor{+}}\AgdaSpace{}%
\AgdaBound{c'}\AgdaSpace{}%
\AgdaOperator{\AgdaInductiveConstructor{[}}\AgdaSpace{}%
\AgdaSymbol{\AgdaUnderscore{}}\AgdaSpace{}%
\AgdaOperator{\AgdaInductiveConstructor{]}}\AgdaSpace{}%
\AgdaSymbol{|}\AgdaSpace{}%
\AgdaInductiveConstructor{inj₁}\AgdaSpace{}%
\AgdaBound{a<b}\AgdaSpace{}%
\AgdaSymbol{|}\AgdaSpace{}%
\AgdaInductiveConstructor{inj₂}\AgdaSpace{}%
\AgdaBound{b≥c}\AgdaSpace{}%
\AgdaKeyword{with}\AgdaSpace{}%
\AgdaFunction{<-tri}\AgdaSpace{}%
\AgdaBound{a}\AgdaSpace{}%
\AgdaBound{b}\<%
\\
\>[0]\AgdaFunction{+assoc}\AgdaSpace{}%
\AgdaOperator{\AgdaInductiveConstructor{ω\textasciicircum{}}}\AgdaSpace{}%
\AgdaBound{a}\AgdaSpace{}%
\AgdaOperator{\AgdaInductiveConstructor{+}}\AgdaSpace{}%
\AgdaBound{a'}\AgdaSpace{}%
\AgdaOperator{\AgdaInductiveConstructor{[}}\AgdaSpace{}%
\AgdaSymbol{\AgdaUnderscore{}}\AgdaSpace{}%
\AgdaOperator{\AgdaInductiveConstructor{]}}\AgdaSpace{}%
\AgdaOperator{\AgdaInductiveConstructor{ω\textasciicircum{}}}\AgdaSpace{}%
\AgdaBound{b}\AgdaSpace{}%
\AgdaOperator{\AgdaInductiveConstructor{+}}\AgdaSpace{}%
\AgdaBound{b'}\AgdaSpace{}%
\AgdaOperator{\AgdaInductiveConstructor{[}}\AgdaSpace{}%
\AgdaSymbol{\AgdaUnderscore{}}\AgdaSpace{}%
\AgdaOperator{\AgdaInductiveConstructor{]}}\AgdaSpace{}%
\AgdaOperator{\AgdaInductiveConstructor{ω\textasciicircum{}}}\AgdaSpace{}%
\AgdaBound{c}\AgdaSpace{}%
\AgdaOperator{\AgdaInductiveConstructor{+}}\AgdaSpace{}%
\AgdaBound{c'}\AgdaSpace{}%
\AgdaOperator{\AgdaInductiveConstructor{[}}\AgdaSpace{}%
\AgdaSymbol{\AgdaUnderscore{}}\AgdaSpace{}%
\AgdaOperator{\AgdaInductiveConstructor{]}}\AgdaSpace{}%
\AgdaSymbol{|}\AgdaSpace{}%
\AgdaInductiveConstructor{inj₁}\AgdaSpace{}%
\AgdaBound{a<b}\AgdaSpace{}%
\AgdaSymbol{|}\AgdaSpace{}%
\AgdaInductiveConstructor{inj₂}\AgdaSpace{}%
\AgdaBound{b≥c}\AgdaSpace{}%
\AgdaSymbol{|}\AgdaSpace{}%
\AgdaInductiveConstructor{inj₁}\AgdaSpace{}%
\AgdaSymbol{\AgdaUnderscore{}}%
\>[88]\AgdaKeyword{with}\AgdaSpace{}%
\AgdaFunction{<-tri}\AgdaSpace{}%
\AgdaBound{b}\AgdaSpace{}%
\AgdaBound{c}\<%
\\
\>[0]\AgdaFunction{+assoc}\AgdaSpace{}%
\AgdaOperator{\AgdaInductiveConstructor{ω\textasciicircum{}}}\AgdaSpace{}%
\AgdaBound{a}\AgdaSpace{}%
\AgdaOperator{\AgdaInductiveConstructor{+}}\AgdaSpace{}%
\AgdaBound{a'}\AgdaSpace{}%
\AgdaOperator{\AgdaInductiveConstructor{[}}\AgdaSpace{}%
\AgdaSymbol{\AgdaUnderscore{}}\AgdaSpace{}%
\AgdaOperator{\AgdaInductiveConstructor{]}}\AgdaSpace{}%
\AgdaOperator{\AgdaInductiveConstructor{ω\textasciicircum{}}}\AgdaSpace{}%
\AgdaBound{b}\AgdaSpace{}%
\AgdaOperator{\AgdaInductiveConstructor{+}}\AgdaSpace{}%
\AgdaBound{b'}\AgdaSpace{}%
\AgdaOperator{\AgdaInductiveConstructor{[}}\AgdaSpace{}%
\AgdaSymbol{\AgdaUnderscore{}}\AgdaSpace{}%
\AgdaOperator{\AgdaInductiveConstructor{]}}\AgdaSpace{}%
\AgdaOperator{\AgdaInductiveConstructor{ω\textasciicircum{}}}\AgdaSpace{}%
\AgdaBound{c}\AgdaSpace{}%
\AgdaOperator{\AgdaInductiveConstructor{+}}\AgdaSpace{}%
\AgdaBound{c'}\AgdaSpace{}%
\AgdaOperator{\AgdaInductiveConstructor{[}}\AgdaSpace{}%
\AgdaSymbol{\AgdaUnderscore{}}\AgdaSpace{}%
\AgdaOperator{\AgdaInductiveConstructor{]}}\AgdaSpace{}%
\AgdaSymbol{|}\AgdaSpace{}%
\AgdaInductiveConstructor{inj₁}\AgdaSpace{}%
\AgdaBound{a<b}\AgdaSpace{}%
\AgdaSymbol{|}\AgdaSpace{}%
\AgdaInductiveConstructor{inj₂}\AgdaSpace{}%
\AgdaBound{b≥c}\AgdaSpace{}%
\AgdaSymbol{|}\AgdaSpace{}%
\AgdaInductiveConstructor{inj₁}\AgdaSpace{}%
\AgdaSymbol{\AgdaUnderscore{}}%
\>[88]\AgdaSymbol{|}\AgdaSpace{}%
\AgdaInductiveConstructor{inj₁}\AgdaSpace{}%
\AgdaBound{b<c}\AgdaSpace{}%
\AgdaSymbol{=}\AgdaSpace{}%
\AgdaFunction{⊥-elim}\AgdaSpace{}%
\AgdaSymbol{(}\AgdaFunction{Lm[≥→¬<]}\AgdaSpace{}%
\AgdaBound{b≥c}\AgdaSpace{}%
\AgdaBound{b<c}\AgdaSymbol{)}\<%
\\
\>[0]\AgdaFunction{+assoc}\AgdaSpace{}%
\AgdaOperator{\AgdaInductiveConstructor{ω\textasciicircum{}}}\AgdaSpace{}%
\AgdaBound{a}\AgdaSpace{}%
\AgdaOperator{\AgdaInductiveConstructor{+}}\AgdaSpace{}%
\AgdaBound{a'}\AgdaSpace{}%
\AgdaOperator{\AgdaInductiveConstructor{[}}\AgdaSpace{}%
\AgdaSymbol{\AgdaUnderscore{}}\AgdaSpace{}%
\AgdaOperator{\AgdaInductiveConstructor{]}}\AgdaSpace{}%
\AgdaOperator{\AgdaInductiveConstructor{ω\textasciicircum{}}}\AgdaSpace{}%
\AgdaBound{b}\AgdaSpace{}%
\AgdaOperator{\AgdaInductiveConstructor{+}}\AgdaSpace{}%
\AgdaBound{b'}\AgdaSpace{}%
\AgdaOperator{\AgdaInductiveConstructor{[}}\AgdaSpace{}%
\AgdaSymbol{\AgdaUnderscore{}}\AgdaSpace{}%
\AgdaOperator{\AgdaInductiveConstructor{]}}\AgdaSpace{}%
\AgdaOperator{\AgdaInductiveConstructor{ω\textasciicircum{}}}\AgdaSpace{}%
\AgdaBound{c}\AgdaSpace{}%
\AgdaOperator{\AgdaInductiveConstructor{+}}\AgdaSpace{}%
\AgdaBound{c'}\AgdaSpace{}%
\AgdaOperator{\AgdaInductiveConstructor{[}}\AgdaSpace{}%
\AgdaSymbol{\AgdaUnderscore{}}\AgdaSpace{}%
\AgdaOperator{\AgdaInductiveConstructor{]}}\AgdaSpace{}%
\AgdaSymbol{|}\AgdaSpace{}%
\AgdaInductiveConstructor{inj₁}\AgdaSpace{}%
\AgdaBound{a<b}\AgdaSpace{}%
\AgdaSymbol{|}\AgdaSpace{}%
\AgdaInductiveConstructor{inj₂}\AgdaSpace{}%
\AgdaBound{b≥c}\AgdaSpace{}%
\AgdaSymbol{|}\AgdaSpace{}%
\AgdaInductiveConstructor{inj₁}\AgdaSpace{}%
\AgdaSymbol{\AgdaUnderscore{}}%
\>[88]\AgdaSymbol{|}\AgdaSpace{}%
\AgdaInductiveConstructor{inj₂}\AgdaSpace{}%
\AgdaSymbol{\AgdaUnderscore{}}%
\>[99]\AgdaSymbol{=}\AgdaSpace{}%
\AgdaFunction{MutualOrd⁼}\AgdaSpace{}%
\AgdaFunction{refl}\AgdaSpace{}%
\AgdaFunction{refl}\<%
\\
\>[0]\AgdaFunction{+assoc}\AgdaSpace{}%
\AgdaOperator{\AgdaInductiveConstructor{ω\textasciicircum{}}}\AgdaSpace{}%
\AgdaBound{a}\AgdaSpace{}%
\AgdaOperator{\AgdaInductiveConstructor{+}}\AgdaSpace{}%
\AgdaBound{a'}\AgdaSpace{}%
\AgdaOperator{\AgdaInductiveConstructor{[}}\AgdaSpace{}%
\AgdaSymbol{\AgdaUnderscore{}}\AgdaSpace{}%
\AgdaOperator{\AgdaInductiveConstructor{]}}\AgdaSpace{}%
\AgdaOperator{\AgdaInductiveConstructor{ω\textasciicircum{}}}\AgdaSpace{}%
\AgdaBound{b}\AgdaSpace{}%
\AgdaOperator{\AgdaInductiveConstructor{+}}\AgdaSpace{}%
\AgdaBound{b'}\AgdaSpace{}%
\AgdaOperator{\AgdaInductiveConstructor{[}}\AgdaSpace{}%
\AgdaSymbol{\AgdaUnderscore{}}\AgdaSpace{}%
\AgdaOperator{\AgdaInductiveConstructor{]}}\AgdaSpace{}%
\AgdaOperator{\AgdaInductiveConstructor{ω\textasciicircum{}}}\AgdaSpace{}%
\AgdaBound{c}\AgdaSpace{}%
\AgdaOperator{\AgdaInductiveConstructor{+}}\AgdaSpace{}%
\AgdaBound{c'}\AgdaSpace{}%
\AgdaOperator{\AgdaInductiveConstructor{[}}\AgdaSpace{}%
\AgdaSymbol{\AgdaUnderscore{}}\AgdaSpace{}%
\AgdaOperator{\AgdaInductiveConstructor{]}}\AgdaSpace{}%
\AgdaSymbol{|}\AgdaSpace{}%
\AgdaInductiveConstructor{inj₁}\AgdaSpace{}%
\AgdaBound{a<b}\AgdaSpace{}%
\AgdaSymbol{|}\AgdaSpace{}%
\AgdaInductiveConstructor{inj₂}\AgdaSpace{}%
\AgdaBound{b≥c}\AgdaSpace{}%
\AgdaSymbol{|}\AgdaSpace{}%
\AgdaInductiveConstructor{inj₂}\AgdaSpace{}%
\AgdaBound{a≥b}\AgdaSpace{}%
\AgdaSymbol{=}\AgdaSpace{}%
\AgdaFunction{⊥-elim}\AgdaSpace{}%
\AgdaSymbol{(}\AgdaFunction{Lm[≥→¬<]}\AgdaSpace{}%
\AgdaBound{a≥b}\AgdaSpace{}%
\AgdaBound{a<b}\AgdaSymbol{)}\<%
\\
\>[0]\AgdaFunction{+assoc}\AgdaSpace{}%
\AgdaOperator{\AgdaInductiveConstructor{ω\textasciicircum{}}}\AgdaSpace{}%
\AgdaBound{a}\AgdaSpace{}%
\AgdaOperator{\AgdaInductiveConstructor{+}}\AgdaSpace{}%
\AgdaBound{a'}\AgdaSpace{}%
\AgdaOperator{\AgdaInductiveConstructor{[}}\AgdaSpace{}%
\AgdaSymbol{\AgdaUnderscore{}}\AgdaSpace{}%
\AgdaOperator{\AgdaInductiveConstructor{]}}\AgdaSpace{}%
\AgdaOperator{\AgdaInductiveConstructor{ω\textasciicircum{}}}\AgdaSpace{}%
\AgdaBound{b}\AgdaSpace{}%
\AgdaOperator{\AgdaInductiveConstructor{+}}\AgdaSpace{}%
\AgdaBound{b'}\AgdaSpace{}%
\AgdaOperator{\AgdaInductiveConstructor{[}}\AgdaSpace{}%
\AgdaSymbol{\AgdaUnderscore{}}\AgdaSpace{}%
\AgdaOperator{\AgdaInductiveConstructor{]}}\AgdaSpace{}%
\AgdaOperator{\AgdaInductiveConstructor{ω\textasciicircum{}}}\AgdaSpace{}%
\AgdaBound{c}\AgdaSpace{}%
\AgdaOperator{\AgdaInductiveConstructor{+}}\AgdaSpace{}%
\AgdaBound{c'}\AgdaSpace{}%
\AgdaOperator{\AgdaInductiveConstructor{[}}\AgdaSpace{}%
\AgdaSymbol{\AgdaUnderscore{}}\AgdaSpace{}%
\AgdaOperator{\AgdaInductiveConstructor{]}}\AgdaSpace{}%
\AgdaSymbol{|}\AgdaSpace{}%
\AgdaInductiveConstructor{inj₂}\AgdaSpace{}%
\AgdaBound{a≥b}\AgdaSpace{}%
\AgdaSymbol{|}\AgdaSpace{}%
\AgdaInductiveConstructor{inj₁}\AgdaSpace{}%
\AgdaBound{b<c}\AgdaSpace{}%
\AgdaKeyword{with}\AgdaSpace{}%
\AgdaFunction{<-tri}\AgdaSpace{}%
\AgdaBound{a}\AgdaSpace{}%
\AgdaBound{c}\<%
\\
\>[0]\AgdaFunction{+assoc}\AgdaSpace{}%
\AgdaOperator{\AgdaInductiveConstructor{ω\textasciicircum{}}}\AgdaSpace{}%
\AgdaBound{a}\AgdaSpace{}%
\AgdaOperator{\AgdaInductiveConstructor{+}}\AgdaSpace{}%
\AgdaBound{a'}\AgdaSpace{}%
\AgdaOperator{\AgdaInductiveConstructor{[}}\AgdaSpace{}%
\AgdaSymbol{\AgdaUnderscore{}}\AgdaSpace{}%
\AgdaOperator{\AgdaInductiveConstructor{]}}\AgdaSpace{}%
\AgdaOperator{\AgdaInductiveConstructor{ω\textasciicircum{}}}\AgdaSpace{}%
\AgdaBound{b}\AgdaSpace{}%
\AgdaOperator{\AgdaInductiveConstructor{+}}\AgdaSpace{}%
\AgdaBound{b'}\AgdaSpace{}%
\AgdaOperator{\AgdaInductiveConstructor{[}}\AgdaSpace{}%
\AgdaSymbol{\AgdaUnderscore{}}\AgdaSpace{}%
\AgdaOperator{\AgdaInductiveConstructor{]}}\AgdaSpace{}%
\AgdaOperator{\AgdaInductiveConstructor{ω\textasciicircum{}}}\AgdaSpace{}%
\AgdaBound{c}\AgdaSpace{}%
\AgdaOperator{\AgdaInductiveConstructor{+}}\AgdaSpace{}%
\AgdaBound{c'}\AgdaSpace{}%
\AgdaOperator{\AgdaInductiveConstructor{[}}\AgdaSpace{}%
\AgdaSymbol{\AgdaUnderscore{}}\AgdaSpace{}%
\AgdaOperator{\AgdaInductiveConstructor{]}}\AgdaSpace{}%
\AgdaSymbol{|}\AgdaSpace{}%
\AgdaInductiveConstructor{inj₂}\AgdaSpace{}%
\AgdaBound{a≥b}\AgdaSpace{}%
\AgdaSymbol{|}\AgdaSpace{}%
\AgdaInductiveConstructor{inj₁}\AgdaSpace{}%
\AgdaBound{b<c}\AgdaSpace{}%
\AgdaSymbol{|}\AgdaSpace{}%
\AgdaInductiveConstructor{inj₁}\AgdaSpace{}%
\AgdaBound{a<c}\AgdaSpace{}%
\AgdaSymbol{=}\AgdaSpace{}%
\AgdaFunction{refl}\<%
\\
\>[0]\AgdaFunction{+assoc}\AgdaSpace{}%
\AgdaOperator{\AgdaInductiveConstructor{ω\textasciicircum{}}}\AgdaSpace{}%
\AgdaBound{a}\AgdaSpace{}%
\AgdaOperator{\AgdaInductiveConstructor{+}}\AgdaSpace{}%
\AgdaBound{a'}\AgdaSpace{}%
\AgdaOperator{\AgdaInductiveConstructor{[}}\AgdaSpace{}%
\AgdaBound{r}\AgdaSpace{}%
\AgdaOperator{\AgdaInductiveConstructor{]}}\AgdaSpace{}%
\AgdaOperator{\AgdaInductiveConstructor{ω\textasciicircum{}}}\AgdaSpace{}%
\AgdaBound{b}\AgdaSpace{}%
\AgdaOperator{\AgdaInductiveConstructor{+}}\AgdaSpace{}%
\AgdaBound{b'}\AgdaSpace{}%
\AgdaOperator{\AgdaInductiveConstructor{[}}\AgdaSpace{}%
\AgdaBound{s}\AgdaSpace{}%
\AgdaOperator{\AgdaInductiveConstructor{]}}\AgdaSpace{}%
\AgdaOperator{\AgdaInductiveConstructor{ω\textasciicircum{}}}\AgdaSpace{}%
\AgdaBound{c}\AgdaSpace{}%
\AgdaOperator{\AgdaInductiveConstructor{+}}\AgdaSpace{}%
\AgdaBound{c'}\AgdaSpace{}%
\AgdaOperator{\AgdaInductiveConstructor{[}}\AgdaSpace{}%
\AgdaBound{t}\AgdaSpace{}%
\AgdaOperator{\AgdaInductiveConstructor{]}}\AgdaSpace{}%
\AgdaSymbol{|}\AgdaSpace{}%
\AgdaInductiveConstructor{inj₂}\AgdaSpace{}%
\AgdaBound{a≥b}\AgdaSpace{}%
\AgdaSymbol{|}\AgdaSpace{}%
\AgdaInductiveConstructor{inj₁}\AgdaSpace{}%
\AgdaBound{b<c}\AgdaSpace{}%
\AgdaSymbol{|}\AgdaSpace{}%
\AgdaInductiveConstructor{inj₂}\AgdaSpace{}%
\AgdaBound{a≥c}\AgdaSpace{}%
\AgdaSymbol{=}\AgdaSpace{}%
\AgdaFunction{MutualOrd⁼}\AgdaSpace{}%
\AgdaFunction{refl}\AgdaSpace{}%
\AgdaSymbol{(}\AgdaFunction{claim}\AgdaSpace{}%
\AgdaOperator{\AgdaFunction{∙}}\AgdaSpace{}%
\AgdaFunction{IH}\AgdaSymbol{)}\<%
\\
\>[0][@{}l@{\AgdaIndent{0}}]%
\>[1]\AgdaKeyword{where}\<%
\\
\>[1][@{}l@{\AgdaIndent{0}}]%
\>[2]\AgdaFunction{IH}\AgdaSpace{}%
\AgdaSymbol{:}\AgdaSpace{}%
\AgdaBound{a'}\AgdaSpace{}%
\AgdaOperator{\AgdaFunction{+}}\AgdaSpace{}%
\AgdaSymbol{(}\AgdaOperator{\AgdaInductiveConstructor{ω\textasciicircum{}}}\AgdaSpace{}%
\AgdaBound{b}\AgdaSpace{}%
\AgdaOperator{\AgdaInductiveConstructor{+}}\AgdaSpace{}%
\AgdaBound{b'}\AgdaSpace{}%
\AgdaOperator{\AgdaInductiveConstructor{[}}\AgdaSpace{}%
\AgdaBound{s}\AgdaSpace{}%
\AgdaOperator{\AgdaInductiveConstructor{]}}\AgdaSpace{}%
\AgdaOperator{\AgdaFunction{+}}\AgdaSpace{}%
\AgdaOperator{\AgdaInductiveConstructor{ω\textasciicircum{}}}\AgdaSpace{}%
\AgdaBound{c}\AgdaSpace{}%
\AgdaOperator{\AgdaInductiveConstructor{+}}\AgdaSpace{}%
\AgdaBound{c'}\AgdaSpace{}%
\AgdaOperator{\AgdaInductiveConstructor{[}}\AgdaSpace{}%
\AgdaBound{t}\AgdaSpace{}%
\AgdaOperator{\AgdaInductiveConstructor{]}}\AgdaSymbol{)}\AgdaSpace{}%
\AgdaOperator{\AgdaFunction{≡}}\AgdaSpace{}%
\AgdaSymbol{(}\AgdaBound{a'}\AgdaSpace{}%
\AgdaOperator{\AgdaFunction{+}}\AgdaSpace{}%
\AgdaOperator{\AgdaInductiveConstructor{ω\textasciicircum{}}}\AgdaSpace{}%
\AgdaBound{b}\AgdaSpace{}%
\AgdaOperator{\AgdaInductiveConstructor{+}}\AgdaSpace{}%
\AgdaBound{b'}\AgdaSpace{}%
\AgdaOperator{\AgdaInductiveConstructor{[}}\AgdaSpace{}%
\AgdaBound{s}\AgdaSpace{}%
\AgdaOperator{\AgdaInductiveConstructor{]}}\AgdaSymbol{)}\AgdaSpace{}%
\AgdaOperator{\AgdaFunction{+}}\AgdaSpace{}%
\AgdaOperator{\AgdaInductiveConstructor{ω\textasciicircum{}}}\AgdaSpace{}%
\AgdaBound{c}\AgdaSpace{}%
\AgdaOperator{\AgdaInductiveConstructor{+}}\AgdaSpace{}%
\AgdaBound{c'}\AgdaSpace{}%
\AgdaOperator{\AgdaInductiveConstructor{[}}\AgdaSpace{}%
\AgdaBound{t}\AgdaSpace{}%
\AgdaOperator{\AgdaInductiveConstructor{]}}\<%
\\
\>[2]\AgdaFunction{IH}\AgdaSpace{}%
\AgdaSymbol{=}\AgdaSpace{}%
\AgdaFunction{+assoc}\AgdaSpace{}%
\AgdaBound{a'}\AgdaSpace{}%
\AgdaSymbol{(}\AgdaOperator{\AgdaInductiveConstructor{ω\textasciicircum{}}}\AgdaSpace{}%
\AgdaBound{b}\AgdaSpace{}%
\AgdaOperator{\AgdaInductiveConstructor{+}}\AgdaSpace{}%
\AgdaBound{b'}\AgdaSpace{}%
\AgdaOperator{\AgdaInductiveConstructor{[}}\AgdaSpace{}%
\AgdaBound{s}\AgdaSpace{}%
\AgdaOperator{\AgdaInductiveConstructor{]}}\AgdaSymbol{)}\AgdaSpace{}%
\AgdaSymbol{(}\AgdaOperator{\AgdaInductiveConstructor{ω\textasciicircum{}}}\AgdaSpace{}%
\AgdaBound{c}\AgdaSpace{}%
\AgdaOperator{\AgdaInductiveConstructor{+}}\AgdaSpace{}%
\AgdaBound{c'}\AgdaSpace{}%
\AgdaOperator{\AgdaInductiveConstructor{[}}\AgdaSpace{}%
\AgdaBound{t}\AgdaSpace{}%
\AgdaOperator{\AgdaInductiveConstructor{]}}\AgdaSymbol{)}\<%
\\
\>[2]\AgdaFunction{fact}\AgdaSpace{}%
\AgdaSymbol{:}\AgdaSpace{}%
\AgdaOperator{\AgdaInductiveConstructor{ω\textasciicircum{}}}\AgdaSpace{}%
\AgdaBound{c}\AgdaSpace{}%
\AgdaOperator{\AgdaInductiveConstructor{+}}\AgdaSpace{}%
\AgdaBound{c'}\AgdaSpace{}%
\AgdaOperator{\AgdaInductiveConstructor{[}}\AgdaSpace{}%
\AgdaBound{t}\AgdaSpace{}%
\AgdaOperator{\AgdaInductiveConstructor{]}}\AgdaSpace{}%
\AgdaOperator{\AgdaFunction{≡}}\AgdaSpace{}%
\AgdaOperator{\AgdaInductiveConstructor{ω\textasciicircum{}}}\AgdaSpace{}%
\AgdaBound{b}\AgdaSpace{}%
\AgdaOperator{\AgdaInductiveConstructor{+}}\AgdaSpace{}%
\AgdaBound{b'}\AgdaSpace{}%
\AgdaOperator{\AgdaInductiveConstructor{[}}\AgdaSpace{}%
\AgdaBound{s}\AgdaSpace{}%
\AgdaOperator{\AgdaInductiveConstructor{]}}\AgdaSpace{}%
\AgdaOperator{\AgdaFunction{+}}\AgdaSpace{}%
\AgdaOperator{\AgdaInductiveConstructor{ω\textasciicircum{}}}\AgdaSpace{}%
\AgdaBound{c}\AgdaSpace{}%
\AgdaOperator{\AgdaInductiveConstructor{+}}\AgdaSpace{}%
\AgdaBound{c'}\AgdaSpace{}%
\AgdaOperator{\AgdaInductiveConstructor{[}}\AgdaSpace{}%
\AgdaBound{t}\AgdaSpace{}%
\AgdaOperator{\AgdaInductiveConstructor{]}}\<%
\\
\>[2]\AgdaFunction{fact}\AgdaSpace{}%
\AgdaKeyword{with}\AgdaSpace{}%
\AgdaFunction{<-tri}\AgdaSpace{}%
\AgdaBound{b}\AgdaSpace{}%
\AgdaBound{c}\<%
\\
\>[2]\AgdaFunction{fact}\AgdaSpace{}%
\AgdaSymbol{|}\AgdaSpace{}%
\AgdaInductiveConstructor{inj₁}\AgdaSpace{}%
\AgdaSymbol{\AgdaUnderscore{}}%
\>[18]\AgdaSymbol{=}\AgdaSpace{}%
\AgdaFunction{refl}\<%
\\
\>[2]\AgdaFunction{fact}\AgdaSpace{}%
\AgdaSymbol{|}\AgdaSpace{}%
\AgdaInductiveConstructor{inj₂}\AgdaSpace{}%
\AgdaBound{b≥c}\AgdaSpace{}%
\AgdaSymbol{=}\AgdaSpace{}%
\AgdaFunction{⊥-elim}\AgdaSpace{}%
\AgdaSymbol{(}\AgdaFunction{Lm[≥→¬<]}\AgdaSpace{}%
\AgdaBound{b≥c}\AgdaSpace{}%
\AgdaBound{b<c}\AgdaSymbol{)}\<%
\\
\>[2]\AgdaFunction{claim}\AgdaSpace{}%
\AgdaSymbol{:}\AgdaSpace{}%
\AgdaBound{a'}\AgdaSpace{}%
\AgdaOperator{\AgdaFunction{+}}\AgdaSpace{}%
\AgdaOperator{\AgdaInductiveConstructor{ω\textasciicircum{}}}\AgdaSpace{}%
\AgdaBound{c}\AgdaSpace{}%
\AgdaOperator{\AgdaInductiveConstructor{+}}\AgdaSpace{}%
\AgdaBound{c'}\AgdaSpace{}%
\AgdaOperator{\AgdaInductiveConstructor{[}}\AgdaSpace{}%
\AgdaBound{t}\AgdaSpace{}%
\AgdaOperator{\AgdaInductiveConstructor{]}}\AgdaSpace{}%
\AgdaOperator{\AgdaFunction{≡}}\AgdaSpace{}%
\AgdaBound{a'}\AgdaSpace{}%
\AgdaOperator{\AgdaFunction{+}}\AgdaSpace{}%
\AgdaOperator{\AgdaInductiveConstructor{ω\textasciicircum{}}}\AgdaSpace{}%
\AgdaBound{b}\AgdaSpace{}%
\AgdaOperator{\AgdaInductiveConstructor{+}}\AgdaSpace{}%
\AgdaBound{b'}\AgdaSpace{}%
\AgdaOperator{\AgdaInductiveConstructor{[}}\AgdaSpace{}%
\AgdaBound{s}\AgdaSpace{}%
\AgdaOperator{\AgdaInductiveConstructor{]}}\AgdaSpace{}%
\AgdaOperator{\AgdaFunction{+}}\AgdaSpace{}%
\AgdaOperator{\AgdaInductiveConstructor{ω\textasciicircum{}}}\AgdaSpace{}%
\AgdaBound{c}\AgdaSpace{}%
\AgdaOperator{\AgdaInductiveConstructor{+}}\AgdaSpace{}%
\AgdaBound{c'}\AgdaSpace{}%
\AgdaOperator{\AgdaInductiveConstructor{[}}\AgdaSpace{}%
\AgdaBound{t}\AgdaSpace{}%
\AgdaOperator{\AgdaInductiveConstructor{]}}\<%
\\
\>[2]\AgdaFunction{claim}\AgdaSpace{}%
\AgdaSymbol{=}\AgdaSpace{}%
\AgdaFunction{cong}\AgdaSpace{}%
\AgdaSymbol{(}\AgdaBound{a'}\AgdaSpace{}%
\AgdaOperator{\AgdaFunction{+\AgdaUnderscore{}}}\AgdaSymbol{)}\AgdaSpace{}%
\AgdaFunction{fact}\<%
\\
\>[0]\AgdaFunction{+assoc}\AgdaSpace{}%
\AgdaOperator{\AgdaInductiveConstructor{ω\textasciicircum{}}}\AgdaSpace{}%
\AgdaBound{a}\AgdaSpace{}%
\AgdaOperator{\AgdaInductiveConstructor{+}}\AgdaSpace{}%
\AgdaBound{a'}\AgdaSpace{}%
\AgdaOperator{\AgdaInductiveConstructor{[}}\AgdaSpace{}%
\AgdaSymbol{\AgdaUnderscore{}}\AgdaSpace{}%
\AgdaOperator{\AgdaInductiveConstructor{]}}\AgdaSpace{}%
\AgdaOperator{\AgdaInductiveConstructor{ω\textasciicircum{}}}\AgdaSpace{}%
\AgdaBound{b}\AgdaSpace{}%
\AgdaOperator{\AgdaInductiveConstructor{+}}\AgdaSpace{}%
\AgdaBound{b'}\AgdaSpace{}%
\AgdaOperator{\AgdaInductiveConstructor{[}}\AgdaSpace{}%
\AgdaSymbol{\AgdaUnderscore{}}\AgdaSpace{}%
\AgdaOperator{\AgdaInductiveConstructor{]}}\AgdaSpace{}%
\AgdaOperator{\AgdaInductiveConstructor{ω\textasciicircum{}}}\AgdaSpace{}%
\AgdaBound{c}\AgdaSpace{}%
\AgdaOperator{\AgdaInductiveConstructor{+}}\AgdaSpace{}%
\AgdaBound{c'}\AgdaSpace{}%
\AgdaOperator{\AgdaInductiveConstructor{[}}\AgdaSpace{}%
\AgdaSymbol{\AgdaUnderscore{}}\AgdaSpace{}%
\AgdaOperator{\AgdaInductiveConstructor{]}}\AgdaSpace{}%
\AgdaSymbol{|}\AgdaSpace{}%
\AgdaInductiveConstructor{inj₂}\AgdaSpace{}%
\AgdaBound{a≥b}\AgdaSpace{}%
\AgdaSymbol{|}\AgdaSpace{}%
\AgdaInductiveConstructor{inj₂}\AgdaSpace{}%
\AgdaBound{b≥c}\AgdaSpace{}%
\AgdaKeyword{with}\AgdaSpace{}%
\AgdaFunction{<-tri}\AgdaSpace{}%
\AgdaBound{a}\AgdaSpace{}%
\AgdaBound{b}\<%
\\
\>[0]\AgdaFunction{+assoc}\AgdaSpace{}%
\AgdaOperator{\AgdaInductiveConstructor{ω\textasciicircum{}}}\AgdaSpace{}%
\AgdaBound{a}\AgdaSpace{}%
\AgdaOperator{\AgdaInductiveConstructor{+}}\AgdaSpace{}%
\AgdaBound{a'}\AgdaSpace{}%
\AgdaOperator{\AgdaInductiveConstructor{[}}\AgdaSpace{}%
\AgdaSymbol{\AgdaUnderscore{}}\AgdaSpace{}%
\AgdaOperator{\AgdaInductiveConstructor{]}}\AgdaSpace{}%
\AgdaOperator{\AgdaInductiveConstructor{ω\textasciicircum{}}}\AgdaSpace{}%
\AgdaBound{b}\AgdaSpace{}%
\AgdaOperator{\AgdaInductiveConstructor{+}}\AgdaSpace{}%
\AgdaBound{b'}\AgdaSpace{}%
\AgdaOperator{\AgdaInductiveConstructor{[}}\AgdaSpace{}%
\AgdaSymbol{\AgdaUnderscore{}}\AgdaSpace{}%
\AgdaOperator{\AgdaInductiveConstructor{]}}\AgdaSpace{}%
\AgdaOperator{\AgdaInductiveConstructor{ω\textasciicircum{}}}\AgdaSpace{}%
\AgdaBound{c}\AgdaSpace{}%
\AgdaOperator{\AgdaInductiveConstructor{+}}\AgdaSpace{}%
\AgdaBound{c'}\AgdaSpace{}%
\AgdaOperator{\AgdaInductiveConstructor{[}}\AgdaSpace{}%
\AgdaSymbol{\AgdaUnderscore{}}\AgdaSpace{}%
\AgdaOperator{\AgdaInductiveConstructor{]}}\AgdaSpace{}%
\AgdaSymbol{|}\AgdaSpace{}%
\AgdaInductiveConstructor{inj₂}\AgdaSpace{}%
\AgdaBound{a≥b}\AgdaSpace{}%
\AgdaSymbol{|}\AgdaSpace{}%
\AgdaInductiveConstructor{inj₂}\AgdaSpace{}%
\AgdaBound{b≥c}\AgdaSpace{}%
\AgdaSymbol{|}\AgdaSpace{}%
\AgdaInductiveConstructor{inj₁}\AgdaSpace{}%
\AgdaBound{a<b}\AgdaSpace{}%
\AgdaSymbol{=}\AgdaSpace{}%
\AgdaFunction{⊥-elim}\AgdaSpace{}%
\AgdaSymbol{(}\AgdaFunction{Lm[≥→¬<]}\AgdaSpace{}%
\AgdaBound{a≥b}\AgdaSpace{}%
\AgdaBound{a<b}\AgdaSymbol{)}\<%
\\
\>[0]\AgdaFunction{+assoc}\AgdaSpace{}%
\AgdaOperator{\AgdaInductiveConstructor{ω\textasciicircum{}}}\AgdaSpace{}%
\AgdaBound{a}\AgdaSpace{}%
\AgdaOperator{\AgdaInductiveConstructor{+}}\AgdaSpace{}%
\AgdaBound{a'}\AgdaSpace{}%
\AgdaOperator{\AgdaInductiveConstructor{[}}\AgdaSpace{}%
\AgdaSymbol{\AgdaUnderscore{}}\AgdaSpace{}%
\AgdaOperator{\AgdaInductiveConstructor{]}}\AgdaSpace{}%
\AgdaOperator{\AgdaInductiveConstructor{ω\textasciicircum{}}}\AgdaSpace{}%
\AgdaBound{b}\AgdaSpace{}%
\AgdaOperator{\AgdaInductiveConstructor{+}}\AgdaSpace{}%
\AgdaBound{b'}\AgdaSpace{}%
\AgdaOperator{\AgdaInductiveConstructor{[}}\AgdaSpace{}%
\AgdaSymbol{\AgdaUnderscore{}}\AgdaSpace{}%
\AgdaOperator{\AgdaInductiveConstructor{]}}\AgdaSpace{}%
\AgdaOperator{\AgdaInductiveConstructor{ω\textasciicircum{}}}\AgdaSpace{}%
\AgdaBound{c}\AgdaSpace{}%
\AgdaOperator{\AgdaInductiveConstructor{+}}\AgdaSpace{}%
\AgdaBound{c'}\AgdaSpace{}%
\AgdaOperator{\AgdaInductiveConstructor{[}}\AgdaSpace{}%
\AgdaSymbol{\AgdaUnderscore{}}\AgdaSpace{}%
\AgdaOperator{\AgdaInductiveConstructor{]}}\AgdaSpace{}%
\AgdaSymbol{|}\AgdaSpace{}%
\AgdaInductiveConstructor{inj₂}\AgdaSpace{}%
\AgdaBound{a≥b}\AgdaSpace{}%
\AgdaSymbol{|}\AgdaSpace{}%
\AgdaInductiveConstructor{inj₂}\AgdaSpace{}%
\AgdaBound{b≥c}\AgdaSpace{}%
\AgdaSymbol{|}\AgdaSpace{}%
\AgdaInductiveConstructor{inj₂}\AgdaSpace{}%
\AgdaSymbol{\AgdaUnderscore{}}%
\>[88]\AgdaKeyword{with}\AgdaSpace{}%
\AgdaFunction{<-tri}\AgdaSpace{}%
\AgdaBound{a}\AgdaSpace{}%
\AgdaBound{c}\<%
\\
\>[0]\AgdaFunction{+assoc}\AgdaSpace{}%
\AgdaOperator{\AgdaInductiveConstructor{ω\textasciicircum{}}}\AgdaSpace{}%
\AgdaBound{a}\AgdaSpace{}%
\AgdaOperator{\AgdaInductiveConstructor{+}}\AgdaSpace{}%
\AgdaBound{a'}\AgdaSpace{}%
\AgdaOperator{\AgdaInductiveConstructor{[}}\AgdaSpace{}%
\AgdaSymbol{\AgdaUnderscore{}}\AgdaSpace{}%
\AgdaOperator{\AgdaInductiveConstructor{]}}\AgdaSpace{}%
\AgdaOperator{\AgdaInductiveConstructor{ω\textasciicircum{}}}\AgdaSpace{}%
\AgdaBound{b}\AgdaSpace{}%
\AgdaOperator{\AgdaInductiveConstructor{+}}\AgdaSpace{}%
\AgdaBound{b'}\AgdaSpace{}%
\AgdaOperator{\AgdaInductiveConstructor{[}}\AgdaSpace{}%
\AgdaSymbol{\AgdaUnderscore{}}\AgdaSpace{}%
\AgdaOperator{\AgdaInductiveConstructor{]}}\AgdaSpace{}%
\AgdaOperator{\AgdaInductiveConstructor{ω\textasciicircum{}}}\AgdaSpace{}%
\AgdaBound{c}\AgdaSpace{}%
\AgdaOperator{\AgdaInductiveConstructor{+}}\AgdaSpace{}%
\AgdaBound{c'}\AgdaSpace{}%
\AgdaOperator{\AgdaInductiveConstructor{[}}\AgdaSpace{}%
\AgdaSymbol{\AgdaUnderscore{}}\AgdaSpace{}%
\AgdaOperator{\AgdaInductiveConstructor{]}}\AgdaSpace{}%
\AgdaSymbol{|}\AgdaSpace{}%
\AgdaInductiveConstructor{inj₂}\AgdaSpace{}%
\AgdaBound{a≥b}\AgdaSpace{}%
\AgdaSymbol{|}\AgdaSpace{}%
\AgdaInductiveConstructor{inj₂}\AgdaSpace{}%
\AgdaBound{b≥c}\AgdaSpace{}%
\AgdaSymbol{|}\AgdaSpace{}%
\AgdaInductiveConstructor{inj₂}\AgdaSpace{}%
\AgdaSymbol{\AgdaUnderscore{}}\AgdaSpace{}%
\AgdaSymbol{|}\AgdaSpace{}%
\AgdaInductiveConstructor{inj₁}\AgdaSpace{}%
\AgdaBound{a<c}\AgdaSpace{}%
\AgdaSymbol{=}\AgdaSpace{}%
\AgdaFunction{⊥-elim}\AgdaSpace{}%
\AgdaSymbol{(}\AgdaFunction{Lm[≥→¬<]}\AgdaSpace{}%
\AgdaSymbol{(}\AgdaFunction{≤-trans}\AgdaSpace{}%
\AgdaBound{b≥c}\AgdaSpace{}%
\AgdaBound{a≥b}\AgdaSymbol{)}\AgdaSpace{}%
\AgdaBound{a<c}\AgdaSymbol{)}\<%
\\
\>[0]\AgdaFunction{+assoc}\AgdaSpace{}%
\AgdaOperator{\AgdaInductiveConstructor{ω\textasciicircum{}}}\AgdaSpace{}%
\AgdaBound{a}\AgdaSpace{}%
\AgdaOperator{\AgdaInductiveConstructor{+}}\AgdaSpace{}%
\AgdaBound{a'}\AgdaSpace{}%
\AgdaOperator{\AgdaInductiveConstructor{[}}\AgdaSpace{}%
\AgdaBound{r}\AgdaSpace{}%
\AgdaOperator{\AgdaInductiveConstructor{]}}\AgdaSpace{}%
\AgdaOperator{\AgdaInductiveConstructor{ω\textasciicircum{}}}\AgdaSpace{}%
\AgdaBound{b}\AgdaSpace{}%
\AgdaOperator{\AgdaInductiveConstructor{+}}\AgdaSpace{}%
\AgdaBound{b'}\AgdaSpace{}%
\AgdaOperator{\AgdaInductiveConstructor{[}}\AgdaSpace{}%
\AgdaBound{s}\AgdaSpace{}%
\AgdaOperator{\AgdaInductiveConstructor{]}}\AgdaSpace{}%
\AgdaOperator{\AgdaInductiveConstructor{ω\textasciicircum{}}}\AgdaSpace{}%
\AgdaBound{c}\AgdaSpace{}%
\AgdaOperator{\AgdaInductiveConstructor{+}}\AgdaSpace{}%
\AgdaBound{c'}\AgdaSpace{}%
\AgdaOperator{\AgdaInductiveConstructor{[}}\AgdaSpace{}%
\AgdaBound{t}\AgdaSpace{}%
\AgdaOperator{\AgdaInductiveConstructor{]}}\AgdaSpace{}%
\AgdaSymbol{|}\AgdaSpace{}%
\AgdaInductiveConstructor{inj₂}\AgdaSpace{}%
\AgdaBound{a≥b}\AgdaSpace{}%
\AgdaSymbol{|}\AgdaSpace{}%
\AgdaInductiveConstructor{inj₂}\AgdaSpace{}%
\AgdaBound{b≥c}\AgdaSpace{}%
\AgdaSymbol{|}\AgdaSpace{}%
\AgdaInductiveConstructor{inj₂}\AgdaSpace{}%
\AgdaSymbol{\AgdaUnderscore{}}\AgdaSpace{}%
\AgdaSymbol{|}\AgdaSpace{}%
\AgdaInductiveConstructor{inj₂}\AgdaSpace{}%
\AgdaBound{a≥c}\AgdaSpace{}%
\AgdaSymbol{=}\AgdaSpace{}%
\AgdaFunction{MutualOrd⁼}\AgdaSpace{}%
\AgdaFunction{refl}\AgdaSpace{}%
\AgdaSymbol{(}\AgdaFunction{claim}\AgdaSpace{}%
\AgdaOperator{\AgdaFunction{∙}}\AgdaSpace{}%
\AgdaFunction{IH}\AgdaSymbol{)}\<%
\\
\>[0][@{}l@{\AgdaIndent{0}}]%
\>[1]\AgdaKeyword{where}\<%
\\
\>[1][@{}l@{\AgdaIndent{0}}]%
\>[2]\AgdaFunction{fact}\AgdaSpace{}%
\AgdaSymbol{:}\AgdaSpace{}%
\AgdaOperator{\AgdaInductiveConstructor{ω\textasciicircum{}}}\AgdaSpace{}%
\AgdaBound{b}\AgdaSpace{}%
\AgdaOperator{\AgdaInductiveConstructor{+}}\AgdaSpace{}%
\AgdaSymbol{(}\AgdaBound{b'}\AgdaSpace{}%
\AgdaOperator{\AgdaFunction{+}}\AgdaSpace{}%
\AgdaOperator{\AgdaInductiveConstructor{ω\textasciicircum{}}}\AgdaSpace{}%
\AgdaBound{c}\AgdaSpace{}%
\AgdaOperator{\AgdaInductiveConstructor{+}}\AgdaSpace{}%
\AgdaBound{c'}\AgdaSpace{}%
\AgdaOperator{\AgdaInductiveConstructor{[}}\AgdaSpace{}%
\AgdaBound{t}\AgdaSpace{}%
\AgdaOperator{\AgdaInductiveConstructor{]}}\AgdaSymbol{)}\AgdaSpace{}%
\AgdaOperator{\AgdaInductiveConstructor{[}}\AgdaSpace{}%
\AgdaFunction{≥fst+}\AgdaSpace{}%
\AgdaBound{b'}\AgdaSpace{}%
\AgdaSymbol{\AgdaUnderscore{}}\AgdaSpace{}%
\AgdaBound{s}\AgdaSpace{}%
\AgdaBound{b≥c}\AgdaSpace{}%
\AgdaOperator{\AgdaInductiveConstructor{]}}\AgdaSpace{}%
\AgdaOperator{\AgdaFunction{≡}}\AgdaSpace{}%
\AgdaOperator{\AgdaInductiveConstructor{ω\textasciicircum{}}}\AgdaSpace{}%
\AgdaBound{b}\AgdaSpace{}%
\AgdaOperator{\AgdaInductiveConstructor{+}}\AgdaSpace{}%
\AgdaBound{b'}\AgdaSpace{}%
\AgdaOperator{\AgdaInductiveConstructor{[}}\AgdaSpace{}%
\AgdaBound{s}\AgdaSpace{}%
\AgdaOperator{\AgdaInductiveConstructor{]}}\AgdaSpace{}%
\AgdaOperator{\AgdaFunction{+}}\AgdaSpace{}%
\AgdaOperator{\AgdaInductiveConstructor{ω\textasciicircum{}}}\AgdaSpace{}%
\AgdaBound{c}\AgdaSpace{}%
\AgdaOperator{\AgdaInductiveConstructor{+}}\AgdaSpace{}%
\AgdaBound{c'}\AgdaSpace{}%
\AgdaOperator{\AgdaInductiveConstructor{[}}\AgdaSpace{}%
\AgdaBound{t}\AgdaSpace{}%
\AgdaOperator{\AgdaInductiveConstructor{]}}\<%
\\
\>[2]\AgdaFunction{fact}\AgdaSpace{}%
\AgdaKeyword{with}\AgdaSpace{}%
\AgdaFunction{<-tri}\AgdaSpace{}%
\AgdaBound{b}\AgdaSpace{}%
\AgdaBound{c}\<%
\\
\>[2]\AgdaFunction{fact}\AgdaSpace{}%
\AgdaSymbol{|}\AgdaSpace{}%
\AgdaInductiveConstructor{inj₁}\AgdaSpace{}%
\AgdaBound{b<c}\AgdaSpace{}%
\AgdaSymbol{=}\AgdaSpace{}%
\AgdaFunction{⊥-elim}\AgdaSpace{}%
\AgdaSymbol{(}\AgdaFunction{Lm[≥→¬<]}\AgdaSpace{}%
\AgdaBound{b≥c}\AgdaSpace{}%
\AgdaBound{b<c}\AgdaSymbol{)}\<%
\\
\>[2]\AgdaFunction{fact}\AgdaSpace{}%
\AgdaSymbol{|}\AgdaSpace{}%
\AgdaInductiveConstructor{inj₂}\AgdaSpace{}%
\AgdaSymbol{\AgdaUnderscore{}}%
\>[18]\AgdaSymbol{=}\AgdaSpace{}%
\AgdaFunction{MutualOrd⁼}\AgdaSpace{}%
\AgdaFunction{refl}\AgdaSpace{}%
\AgdaFunction{refl}\<%
\\
\>[2]\AgdaFunction{claim}\AgdaSpace{}%
\AgdaSymbol{:}\AgdaSpace{}%
\AgdaBound{a'}\AgdaSpace{}%
\AgdaOperator{\AgdaFunction{+}}\AgdaSpace{}%
\AgdaOperator{\AgdaInductiveConstructor{ω\textasciicircum{}}}\AgdaSpace{}%
\AgdaBound{b}\AgdaSpace{}%
\AgdaOperator{\AgdaInductiveConstructor{+}}\AgdaSpace{}%
\AgdaSymbol{(}\AgdaBound{b'}\AgdaSpace{}%
\AgdaOperator{\AgdaFunction{+}}\AgdaSpace{}%
\AgdaOperator{\AgdaInductiveConstructor{ω\textasciicircum{}}}\AgdaSpace{}%
\AgdaBound{c}\AgdaSpace{}%
\AgdaOperator{\AgdaInductiveConstructor{+}}\AgdaSpace{}%
\AgdaBound{c'}\AgdaSpace{}%
\AgdaOperator{\AgdaInductiveConstructor{[}}\AgdaSpace{}%
\AgdaBound{t}\AgdaSpace{}%
\AgdaOperator{\AgdaInductiveConstructor{]}}\AgdaSymbol{)}\AgdaSpace{}%
\AgdaOperator{\AgdaInductiveConstructor{[}}\AgdaSpace{}%
\AgdaFunction{≥fst+}\AgdaSpace{}%
\AgdaBound{b'}\AgdaSpace{}%
\AgdaSymbol{\AgdaUnderscore{}}\AgdaSpace{}%
\AgdaBound{s}\AgdaSpace{}%
\AgdaBound{b≥c}\AgdaSpace{}%
\AgdaOperator{\AgdaInductiveConstructor{]}}\AgdaSpace{}%
\AgdaOperator{\AgdaFunction{≡}}\AgdaSpace{}%
\AgdaBound{a'}\AgdaSpace{}%
\AgdaOperator{\AgdaFunction{+}}\AgdaSpace{}%
\AgdaOperator{\AgdaInductiveConstructor{ω\textasciicircum{}}}\AgdaSpace{}%
\AgdaBound{b}\AgdaSpace{}%
\AgdaOperator{\AgdaInductiveConstructor{+}}\AgdaSpace{}%
\AgdaBound{b'}\AgdaSpace{}%
\AgdaOperator{\AgdaInductiveConstructor{[}}\AgdaSpace{}%
\AgdaBound{s}\AgdaSpace{}%
\AgdaOperator{\AgdaInductiveConstructor{]}}\AgdaSpace{}%
\AgdaOperator{\AgdaFunction{+}}\AgdaSpace{}%
\AgdaOperator{\AgdaInductiveConstructor{ω\textasciicircum{}}}\AgdaSpace{}%
\AgdaBound{c}\AgdaSpace{}%
\AgdaOperator{\AgdaInductiveConstructor{+}}\AgdaSpace{}%
\AgdaBound{c'}\AgdaSpace{}%
\AgdaOperator{\AgdaInductiveConstructor{[}}\AgdaSpace{}%
\AgdaBound{t}\AgdaSpace{}%
\AgdaOperator{\AgdaInductiveConstructor{]}}\<%
\\
\>[2]\AgdaFunction{claim}\AgdaSpace{}%
\AgdaSymbol{=}\AgdaSpace{}%
\AgdaFunction{cong}\AgdaSpace{}%
\AgdaSymbol{(}\AgdaBound{a'}\AgdaSpace{}%
\AgdaOperator{\AgdaFunction{+\AgdaUnderscore{}}}\AgdaSymbol{)}\AgdaSpace{}%
\AgdaFunction{fact}\<%
\\
\>[2]\AgdaFunction{IH}\AgdaSpace{}%
\AgdaSymbol{:}\AgdaSpace{}%
\AgdaBound{a'}\AgdaSpace{}%
\AgdaOperator{\AgdaFunction{+}}\AgdaSpace{}%
\AgdaOperator{\AgdaInductiveConstructor{ω\textasciicircum{}}}\AgdaSpace{}%
\AgdaBound{b}\AgdaSpace{}%
\AgdaOperator{\AgdaInductiveConstructor{+}}\AgdaSpace{}%
\AgdaBound{b'}\AgdaSpace{}%
\AgdaOperator{\AgdaInductiveConstructor{[}}\AgdaSpace{}%
\AgdaBound{s}\AgdaSpace{}%
\AgdaOperator{\AgdaInductiveConstructor{]}}\AgdaSpace{}%
\AgdaOperator{\AgdaFunction{+}}\AgdaSpace{}%
\AgdaOperator{\AgdaInductiveConstructor{ω\textasciicircum{}}}\AgdaSpace{}%
\AgdaBound{c}\AgdaSpace{}%
\AgdaOperator{\AgdaInductiveConstructor{+}}\AgdaSpace{}%
\AgdaBound{c'}\AgdaSpace{}%
\AgdaOperator{\AgdaInductiveConstructor{[}}\AgdaSpace{}%
\AgdaBound{t}\AgdaSpace{}%
\AgdaOperator{\AgdaInductiveConstructor{]}}\AgdaSpace{}%
\AgdaOperator{\AgdaFunction{≡}}\AgdaSpace{}%
\AgdaSymbol{(}\AgdaBound{a'}\AgdaSpace{}%
\AgdaOperator{\AgdaFunction{+}}\AgdaSpace{}%
\AgdaOperator{\AgdaInductiveConstructor{ω\textasciicircum{}}}\AgdaSpace{}%
\AgdaBound{b}\AgdaSpace{}%
\AgdaOperator{\AgdaInductiveConstructor{+}}\AgdaSpace{}%
\AgdaBound{b'}\AgdaSpace{}%
\AgdaOperator{\AgdaInductiveConstructor{[}}\AgdaSpace{}%
\AgdaBound{s}\AgdaSpace{}%
\AgdaOperator{\AgdaInductiveConstructor{]}}\AgdaSymbol{)}\AgdaSpace{}%
\AgdaOperator{\AgdaFunction{+}}\AgdaSpace{}%
\AgdaOperator{\AgdaInductiveConstructor{ω\textasciicircum{}}}\AgdaSpace{}%
\AgdaBound{c}\AgdaSpace{}%
\AgdaOperator{\AgdaInductiveConstructor{+}}\AgdaSpace{}%
\AgdaBound{c'}\AgdaSpace{}%
\AgdaOperator{\AgdaInductiveConstructor{[}}\AgdaSpace{}%
\AgdaBound{t}\AgdaSpace{}%
\AgdaOperator{\AgdaInductiveConstructor{]}}\<%
\\
\>[2]\AgdaFunction{IH}\AgdaSpace{}%
\AgdaSymbol{=}\AgdaSpace{}%
\AgdaFunction{+assoc}\AgdaSpace{}%
\AgdaBound{a'}\AgdaSpace{}%
\AgdaSymbol{(}\AgdaOperator{\AgdaInductiveConstructor{ω\textasciicircum{}}}\AgdaSpace{}%
\AgdaBound{b}\AgdaSpace{}%
\AgdaOperator{\AgdaInductiveConstructor{+}}\AgdaSpace{}%
\AgdaBound{b'}\AgdaSpace{}%
\AgdaOperator{\AgdaInductiveConstructor{[}}\AgdaSpace{}%
\AgdaBound{s}\AgdaSpace{}%
\AgdaOperator{\AgdaInductiveConstructor{]}}\AgdaSymbol{)}\AgdaSpace{}%
\AgdaSymbol{(}\AgdaOperator{\AgdaInductiveConstructor{ω\textasciicircum{}}}\AgdaSpace{}%
\AgdaBound{c}\AgdaSpace{}%
\AgdaOperator{\AgdaInductiveConstructor{+}}\AgdaSpace{}%
\AgdaBound{c'}\AgdaSpace{}%
\AgdaOperator{\AgdaInductiveConstructor{[}}\AgdaSpace{}%
\AgdaBound{t}\AgdaSpace{}%
\AgdaOperator{\AgdaInductiveConstructor{]}}\AgdaSymbol{)}\<%
\\
\\[\AgdaEmptyExtraSkip]%
\>[0]\AgdaFunction{b>0→a<a+b}\AgdaSpace{}%
\AgdaSymbol{:}\AgdaSpace{}%
\AgdaSymbol{∀}\AgdaSpace{}%
\AgdaBound{a}\AgdaSpace{}%
\AgdaBound{b}\AgdaSpace{}%
\AgdaSymbol{→}\AgdaSpace{}%
\AgdaBound{b}\AgdaSpace{}%
\AgdaOperator{\AgdaFunction{>}}\AgdaSpace{}%
\AgdaInductiveConstructor{𝟎}\AgdaSpace{}%
\AgdaSymbol{→}\AgdaSpace{}%
\AgdaBound{a}\AgdaSpace{}%
\AgdaOperator{\AgdaDatatype{<}}\AgdaSpace{}%
\AgdaBound{a}\AgdaSpace{}%
\AgdaOperator{\AgdaFunction{+}}\AgdaSpace{}%
\AgdaBound{b}\<%
\\
\>[0]\AgdaFunction{b>0→a<a+b}\AgdaSpace{}%
\AgdaInductiveConstructor{𝟎}\AgdaSpace{}%
\AgdaBound{b}\AgdaSpace{}%
\AgdaBound{r}\AgdaSpace{}%
\AgdaSymbol{=}\AgdaSpace{}%
\AgdaBound{r}\<%
\\
\>[0]\AgdaFunction{b>0→a<a+b}\AgdaSpace{}%
\AgdaOperator{\AgdaInductiveConstructor{ω\textasciicircum{}}}\AgdaSpace{}%
\AgdaBound{a}\AgdaSpace{}%
\AgdaOperator{\AgdaInductiveConstructor{+}}\AgdaSpace{}%
\AgdaBound{c}\AgdaSpace{}%
\AgdaOperator{\AgdaInductiveConstructor{[}}\AgdaSpace{}%
\AgdaBound{r}\AgdaSpace{}%
\AgdaOperator{\AgdaInductiveConstructor{]}}\AgdaSpace{}%
\AgdaOperator{\AgdaInductiveConstructor{ω\textasciicircum{}}}\AgdaSpace{}%
\AgdaBound{b}\AgdaSpace{}%
\AgdaOperator{\AgdaInductiveConstructor{+}}\AgdaSpace{}%
\AgdaBound{d}\AgdaSpace{}%
\AgdaOperator{\AgdaInductiveConstructor{[}}\AgdaSpace{}%
\AgdaBound{s}\AgdaSpace{}%
\AgdaOperator{\AgdaInductiveConstructor{]}}\AgdaSpace{}%
\AgdaInductiveConstructor{<₁}\AgdaSpace{}%
\AgdaKeyword{with}\AgdaSpace{}%
\AgdaFunction{<-tri}\AgdaSpace{}%
\AgdaBound{a}\AgdaSpace{}%
\AgdaBound{b}\<%
\\
\>[0]\AgdaFunction{b>0→a<a+b}\AgdaSpace{}%
\AgdaOperator{\AgdaInductiveConstructor{ω\textasciicircum{}}}\AgdaSpace{}%
\AgdaBound{a}\AgdaSpace{}%
\AgdaOperator{\AgdaInductiveConstructor{+}}\AgdaSpace{}%
\AgdaBound{c}\AgdaSpace{}%
\AgdaOperator{\AgdaInductiveConstructor{[}}\AgdaSpace{}%
\AgdaBound{r}\AgdaSpace{}%
\AgdaOperator{\AgdaInductiveConstructor{]}}\AgdaSpace{}%
\AgdaOperator{\AgdaInductiveConstructor{ω\textasciicircum{}}}\AgdaSpace{}%
\AgdaBound{b}\AgdaSpace{}%
\AgdaOperator{\AgdaInductiveConstructor{+}}\AgdaSpace{}%
\AgdaBound{d}\AgdaSpace{}%
\AgdaOperator{\AgdaInductiveConstructor{[}}\AgdaSpace{}%
\AgdaBound{s}\AgdaSpace{}%
\AgdaOperator{\AgdaInductiveConstructor{]}}\AgdaSpace{}%
\AgdaInductiveConstructor{<₁}\AgdaSpace{}%
\AgdaSymbol{|}\AgdaSpace{}%
\AgdaInductiveConstructor{inj₁}\AgdaSpace{}%
\AgdaBound{a<b}\AgdaSpace{}%
\AgdaSymbol{=}\AgdaSpace{}%
\AgdaInductiveConstructor{<₂}\AgdaSpace{}%
\AgdaBound{a<b}\<%
\\
\>[0]\AgdaFunction{b>0→a<a+b}\AgdaSpace{}%
\AgdaOperator{\AgdaInductiveConstructor{ω\textasciicircum{}}}\AgdaSpace{}%
\AgdaBound{a}\AgdaSpace{}%
\AgdaOperator{\AgdaInductiveConstructor{+}}\AgdaSpace{}%
\AgdaBound{c}\AgdaSpace{}%
\AgdaOperator{\AgdaInductiveConstructor{[}}\AgdaSpace{}%
\AgdaBound{r}\AgdaSpace{}%
\AgdaOperator{\AgdaInductiveConstructor{]}}\AgdaSpace{}%
\AgdaOperator{\AgdaInductiveConstructor{ω\textasciicircum{}}}\AgdaSpace{}%
\AgdaBound{b}\AgdaSpace{}%
\AgdaOperator{\AgdaInductiveConstructor{+}}\AgdaSpace{}%
\AgdaBound{d}\AgdaSpace{}%
\AgdaOperator{\AgdaInductiveConstructor{[}}\AgdaSpace{}%
\AgdaBound{s}\AgdaSpace{}%
\AgdaOperator{\AgdaInductiveConstructor{]}}\AgdaSpace{}%
\AgdaInductiveConstructor{<₁}\AgdaSpace{}%
\AgdaSymbol{|}\AgdaSpace{}%
\AgdaInductiveConstructor{inj₂}\AgdaSpace{}%
\AgdaBound{a≥b}\AgdaSpace{}%
\AgdaSymbol{=}\AgdaSpace{}%
\AgdaInductiveConstructor{<₃}\AgdaSpace{}%
\AgdaFunction{refl}\AgdaSpace{}%
\AgdaSymbol{(}\AgdaFunction{b>0→a<a+b}\AgdaSpace{}%
\AgdaBound{c}\AgdaSpace{}%
\AgdaSymbol{\AgdaUnderscore{}}\AgdaSpace{}%
\AgdaInductiveConstructor{<₁}\AgdaSymbol{)}\<%
\\
\\[\AgdaEmptyExtraSkip]%
\>[0]\AgdaFunction{a≤a+b}\AgdaSpace{}%
\AgdaSymbol{:}\AgdaSpace{}%
\AgdaSymbol{∀}\AgdaSpace{}%
\AgdaBound{a}\AgdaSpace{}%
\AgdaBound{b}\AgdaSpace{}%
\AgdaSymbol{→}\AgdaSpace{}%
\AgdaBound{a}\AgdaSpace{}%
\AgdaOperator{\AgdaFunction{≤}}\AgdaSpace{}%
\AgdaBound{a}\AgdaSpace{}%
\AgdaOperator{\AgdaFunction{+}}\AgdaSpace{}%
\AgdaBound{b}\<%
\\
\>[0]\AgdaFunction{a≤a+b}\AgdaSpace{}%
\AgdaBound{a}\AgdaSpace{}%
\AgdaInductiveConstructor{𝟎}\AgdaSpace{}%
\AgdaSymbol{=}\AgdaSpace{}%
\AgdaInductiveConstructor{inj₂}\AgdaSpace{}%
\AgdaSymbol{(}\AgdaFunction{+unitr}\AgdaSpace{}%
\AgdaBound{a}\AgdaSymbol{)}\<%
\\
\>[0]\AgdaFunction{a≤a+b}\AgdaSpace{}%
\AgdaBound{a}\AgdaSpace{}%
\AgdaOperator{\AgdaInductiveConstructor{ω\textasciicircum{}}}\AgdaSpace{}%
\AgdaBound{b}\AgdaSpace{}%
\AgdaOperator{\AgdaInductiveConstructor{+}}\AgdaSpace{}%
\AgdaBound{c}\AgdaSpace{}%
\AgdaOperator{\AgdaInductiveConstructor{[}}\AgdaSpace{}%
\AgdaBound{r}\AgdaSpace{}%
\AgdaOperator{\AgdaInductiveConstructor{]}}\AgdaSpace{}%
\AgdaSymbol{=}\AgdaSpace{}%
\AgdaInductiveConstructor{inj₁}\AgdaSpace{}%
\AgdaSymbol{(}\AgdaFunction{b>0→a<a+b}\AgdaSpace{}%
\AgdaBound{a}\AgdaSpace{}%
\AgdaSymbol{\AgdaUnderscore{}}\AgdaSpace{}%
\AgdaInductiveConstructor{<₁}\AgdaSymbol{)}\<%
\\
\\[\AgdaEmptyExtraSkip]%
\>[0]\AgdaFunction{b≤a+b}\AgdaSpace{}%
\AgdaSymbol{:}\AgdaSpace{}%
\AgdaSymbol{∀}\AgdaSpace{}%
\AgdaBound{a}\AgdaSpace{}%
\AgdaBound{b}\AgdaSpace{}%
\AgdaSymbol{→}\AgdaSpace{}%
\AgdaBound{b}\AgdaSpace{}%
\AgdaOperator{\AgdaFunction{≤}}\AgdaSpace{}%
\AgdaBound{a}\AgdaSpace{}%
\AgdaOperator{\AgdaFunction{+}}\AgdaSpace{}%
\AgdaBound{b}\<%
\\
\>[0]\AgdaFunction{b≤a+b}\AgdaSpace{}%
\AgdaBound{a}\AgdaSpace{}%
\AgdaInductiveConstructor{𝟎}\AgdaSpace{}%
\AgdaSymbol{=}\AgdaSpace{}%
\AgdaFunction{≥𝟎}\<%
\\
\>[0]\AgdaFunction{b≤a+b}\AgdaSpace{}%
\AgdaInductiveConstructor{𝟎}\AgdaSpace{}%
\AgdaOperator{\AgdaInductiveConstructor{ω\textasciicircum{}}}\AgdaSpace{}%
\AgdaBound{b}\AgdaSpace{}%
\AgdaOperator{\AgdaInductiveConstructor{+}}\AgdaSpace{}%
\AgdaBound{d}\AgdaSpace{}%
\AgdaOperator{\AgdaInductiveConstructor{[}}\AgdaSpace{}%
\AgdaSymbol{\AgdaUnderscore{}}\AgdaSpace{}%
\AgdaOperator{\AgdaInductiveConstructor{]}}\AgdaSpace{}%
\AgdaSymbol{=}\AgdaSpace{}%
\AgdaInductiveConstructor{inj₂}\AgdaSpace{}%
\AgdaFunction{refl}\<%
\\
\>[0]\AgdaFunction{b≤a+b}\AgdaSpace{}%
\AgdaOperator{\AgdaInductiveConstructor{ω\textasciicircum{}}}\AgdaSpace{}%
\AgdaBound{a}\AgdaSpace{}%
\AgdaOperator{\AgdaInductiveConstructor{+}}\AgdaSpace{}%
\AgdaBound{c}\AgdaSpace{}%
\AgdaOperator{\AgdaInductiveConstructor{[}}\AgdaSpace{}%
\AgdaSymbol{\AgdaUnderscore{}}\AgdaSpace{}%
\AgdaOperator{\AgdaInductiveConstructor{]}}\AgdaSpace{}%
\AgdaOperator{\AgdaInductiveConstructor{ω\textasciicircum{}}}\AgdaSpace{}%
\AgdaBound{b}\AgdaSpace{}%
\AgdaOperator{\AgdaInductiveConstructor{+}}\AgdaSpace{}%
\AgdaBound{d}\AgdaSpace{}%
\AgdaOperator{\AgdaInductiveConstructor{[}}\AgdaSpace{}%
\AgdaSymbol{\AgdaUnderscore{}}\AgdaSpace{}%
\AgdaOperator{\AgdaInductiveConstructor{]}}\AgdaSpace{}%
\AgdaKeyword{with}\AgdaSpace{}%
\AgdaFunction{<-tri}\AgdaSpace{}%
\AgdaBound{a}\AgdaSpace{}%
\AgdaBound{b}\<%
\\
\>[0]\AgdaFunction{b≤a+b}\AgdaSpace{}%
\AgdaOperator{\AgdaInductiveConstructor{ω\textasciicircum{}}}\AgdaSpace{}%
\AgdaBound{a}\AgdaSpace{}%
\AgdaOperator{\AgdaInductiveConstructor{+}}\AgdaSpace{}%
\AgdaBound{c}\AgdaSpace{}%
\AgdaOperator{\AgdaInductiveConstructor{[}}\AgdaSpace{}%
\AgdaSymbol{\AgdaUnderscore{}}\AgdaSpace{}%
\AgdaOperator{\AgdaInductiveConstructor{]}}\AgdaSpace{}%
\AgdaOperator{\AgdaInductiveConstructor{ω\textasciicircum{}}}\AgdaSpace{}%
\AgdaBound{b}\AgdaSpace{}%
\AgdaOperator{\AgdaInductiveConstructor{+}}\AgdaSpace{}%
\AgdaBound{d}\AgdaSpace{}%
\AgdaOperator{\AgdaInductiveConstructor{[}}\AgdaSpace{}%
\AgdaSymbol{\AgdaUnderscore{}}\AgdaSpace{}%
\AgdaOperator{\AgdaInductiveConstructor{]}}\AgdaSpace{}%
\AgdaSymbol{|}\AgdaSpace{}%
\AgdaInductiveConstructor{inj₁}\AgdaSpace{}%
\AgdaBound{a<b}\AgdaSpace{}%
\AgdaSymbol{=}\AgdaSpace{}%
\AgdaInductiveConstructor{inj₂}\AgdaSpace{}%
\AgdaFunction{refl}\<%
\\
\>[0]\AgdaFunction{b≤a+b}\AgdaSpace{}%
\AgdaOperator{\AgdaInductiveConstructor{ω\textasciicircum{}}}\AgdaSpace{}%
\AgdaBound{a}\AgdaSpace{}%
\AgdaOperator{\AgdaInductiveConstructor{+}}\AgdaSpace{}%
\AgdaBound{c}\AgdaSpace{}%
\AgdaOperator{\AgdaInductiveConstructor{[}}\AgdaSpace{}%
\AgdaSymbol{\AgdaUnderscore{}}\AgdaSpace{}%
\AgdaOperator{\AgdaInductiveConstructor{]}}\AgdaSpace{}%
\AgdaOperator{\AgdaInductiveConstructor{ω\textasciicircum{}}}\AgdaSpace{}%
\AgdaBound{b}\AgdaSpace{}%
\AgdaOperator{\AgdaInductiveConstructor{+}}\AgdaSpace{}%
\AgdaBound{d}\AgdaSpace{}%
\AgdaOperator{\AgdaInductiveConstructor{[}}\AgdaSpace{}%
\AgdaSymbol{\AgdaUnderscore{}}\AgdaSpace{}%
\AgdaOperator{\AgdaInductiveConstructor{]}}\AgdaSpace{}%
\AgdaSymbol{|}\AgdaSpace{}%
\AgdaInductiveConstructor{inj₂}\AgdaSpace{}%
\AgdaSymbol{(}\AgdaInductiveConstructor{inj₁}\AgdaSpace{}%
\AgdaBound{a>b}\AgdaSymbol{)}\AgdaSpace{}%
\AgdaSymbol{=}\AgdaSpace{}%
\AgdaInductiveConstructor{inj₁}\AgdaSpace{}%
\AgdaSymbol{(}\AgdaInductiveConstructor{<₂}\AgdaSpace{}%
\AgdaBound{a>b}\AgdaSymbol{)}\<%
\\
\>[0]\AgdaFunction{b≤a+b}\AgdaSpace{}%
\AgdaOperator{\AgdaInductiveConstructor{ω\textasciicircum{}}}\AgdaSpace{}%
\AgdaBound{a}\AgdaSpace{}%
\AgdaOperator{\AgdaInductiveConstructor{+}}\AgdaSpace{}%
\AgdaBound{c}\AgdaSpace{}%
\AgdaOperator{\AgdaInductiveConstructor{[}}\AgdaSpace{}%
\AgdaBound{r}\AgdaSpace{}%
\AgdaOperator{\AgdaInductiveConstructor{]}}\AgdaSpace{}%
\AgdaOperator{\AgdaInductiveConstructor{ω\textasciicircum{}}}\AgdaSpace{}%
\AgdaBound{b}\AgdaSpace{}%
\AgdaOperator{\AgdaInductiveConstructor{+}}\AgdaSpace{}%
\AgdaBound{d}\AgdaSpace{}%
\AgdaOperator{\AgdaInductiveConstructor{[}}\AgdaSpace{}%
\AgdaBound{s}\AgdaSpace{}%
\AgdaOperator{\AgdaInductiveConstructor{]}}\AgdaSpace{}%
\AgdaSymbol{|}\AgdaSpace{}%
\AgdaInductiveConstructor{inj₂}\AgdaSpace{}%
\AgdaSymbol{(}\AgdaInductiveConstructor{inj₂}\AgdaSpace{}%
\AgdaBound{a≡b}\AgdaSymbol{)}\AgdaSpace{}%
\AgdaSymbol{=}\AgdaSpace{}%
\AgdaInductiveConstructor{inj₁}\AgdaSpace{}%
\AgdaSymbol{(}\AgdaInductiveConstructor{<₃}\AgdaSpace{}%
\AgdaSymbol{(}\AgdaBound{a≡b}\AgdaSpace{}%
\AgdaOperator{\AgdaFunction{⁻¹}}\AgdaSymbol{)}\AgdaSpace{}%
\AgdaSymbol{(}\AgdaFunction{<≤-trans}\AgdaSpace{}%
\AgdaFunction{fact}\AgdaSpace{}%
\AgdaFunction{IH}\AgdaSymbol{))}\<%
\\
\>[0][@{}l@{\AgdaIndent{0}}]%
\>[1]\AgdaKeyword{where}\<%
\\
\>[1][@{}l@{\AgdaIndent{0}}]%
\>[2]\AgdaFunction{IH}\AgdaSpace{}%
\AgdaSymbol{:}\AgdaSpace{}%
\AgdaOperator{\AgdaInductiveConstructor{ω\textasciicircum{}}}\AgdaSpace{}%
\AgdaBound{b}\AgdaSpace{}%
\AgdaOperator{\AgdaInductiveConstructor{+}}\AgdaSpace{}%
\AgdaBound{d}\AgdaSpace{}%
\AgdaOperator{\AgdaInductiveConstructor{[}}\AgdaSpace{}%
\AgdaBound{s}\AgdaSpace{}%
\AgdaOperator{\AgdaInductiveConstructor{]}}\AgdaSpace{}%
\AgdaOperator{\AgdaFunction{≤}}\AgdaSpace{}%
\AgdaBound{c}\AgdaSpace{}%
\AgdaOperator{\AgdaFunction{+}}\AgdaSpace{}%
\AgdaOperator{\AgdaInductiveConstructor{ω\textasciicircum{}}}\AgdaSpace{}%
\AgdaBound{b}\AgdaSpace{}%
\AgdaOperator{\AgdaInductiveConstructor{+}}\AgdaSpace{}%
\AgdaBound{d}\AgdaSpace{}%
\AgdaOperator{\AgdaInductiveConstructor{[}}\AgdaSpace{}%
\AgdaBound{s}\AgdaSpace{}%
\AgdaOperator{\AgdaInductiveConstructor{]}}\<%
\\
\>[2]\AgdaFunction{IH}\AgdaSpace{}%
\AgdaSymbol{=}\AgdaSpace{}%
\AgdaFunction{b≤a+b}\AgdaSpace{}%
\AgdaBound{c}\AgdaSpace{}%
\AgdaSymbol{(}\AgdaOperator{\AgdaInductiveConstructor{ω\textasciicircum{}}}\AgdaSpace{}%
\AgdaBound{b}\AgdaSpace{}%
\AgdaOperator{\AgdaInductiveConstructor{+}}\AgdaSpace{}%
\AgdaBound{d}\AgdaSpace{}%
\AgdaOperator{\AgdaInductiveConstructor{[}}\AgdaSpace{}%
\AgdaBound{s}\AgdaSpace{}%
\AgdaOperator{\AgdaInductiveConstructor{]}}\AgdaSymbol{)}\<%
\\
\>[2]\AgdaFunction{fact}\AgdaSpace{}%
\AgdaSymbol{:}\AgdaSpace{}%
\AgdaBound{d}\AgdaSpace{}%
\AgdaOperator{\AgdaDatatype{<}}\AgdaSpace{}%
\AgdaOperator{\AgdaInductiveConstructor{ω\textasciicircum{}}}\AgdaSpace{}%
\AgdaBound{b}\AgdaSpace{}%
\AgdaOperator{\AgdaInductiveConstructor{+}}\AgdaSpace{}%
\AgdaBound{d}\AgdaSpace{}%
\AgdaOperator{\AgdaInductiveConstructor{[}}\AgdaSpace{}%
\AgdaBound{s}\AgdaSpace{}%
\AgdaOperator{\AgdaInductiveConstructor{]}}\<%
\\
\>[2]\AgdaFunction{fact}\AgdaSpace{}%
\AgdaSymbol{=}\AgdaSpace{}%
\AgdaFunction{rest<}\AgdaSpace{}%
\AgdaBound{b}\AgdaSpace{}%
\AgdaBound{d}\AgdaSpace{}%
\AgdaBound{s}\<%
\\
\>[0]\<%
\end{code}

\newcommand{\HPlus}{
\begin{code}%
\>[0]\AgdaOperator{\AgdaFunction{\AgdaUnderscore{}+ᴴ\AgdaUnderscore{}}}\AgdaSpace{}%
\AgdaSymbol{:}\AgdaSpace{}%
\AgdaDatatype{HITOrd}\AgdaSpace{}%
\AgdaSymbol{→}\AgdaSpace{}%
\AgdaDatatype{HITOrd}\AgdaSpace{}%
\AgdaSymbol{→}\AgdaSpace{}%
\AgdaDatatype{HITOrd}\<%
\\
\>[0]\AgdaOperator{\AgdaFunction{\AgdaUnderscore{}+ᴴ\AgdaUnderscore{}}}\AgdaSpace{}%
\AgdaSymbol{=}\AgdaSpace{}%
\AgdaFunction{transport}\AgdaSpace{}%
\AgdaSymbol{(λ}\AgdaSpace{}%
\AgdaBound{i}\AgdaSpace{}%
\AgdaSymbol{→}\AgdaSpace{}%
\AgdaFunction{M≡H}\AgdaSpace{}%
\AgdaBound{i}\AgdaSpace{}%
\AgdaSymbol{→}\AgdaSpace{}%
\AgdaFunction{M≡H}\AgdaSpace{}%
\AgdaBound{i}\AgdaSpace{}%
\AgdaSymbol{→}\AgdaSpace{}%
\AgdaFunction{M≡H}\AgdaSpace{}%
\AgdaBound{i}\AgdaSymbol{)}\AgdaSpace{}%
\AgdaOperator{\AgdaFunction{\AgdaUnderscore{}+\AgdaUnderscore{}}}\<%
\end{code}}

\newcommand{\PlusPath}{
\begin{code}%
\>[0]\AgdaFunction{+Path}\AgdaSpace{}%
\AgdaSymbol{:}\AgdaSpace{}%
\AgdaPostulate{PathP}\AgdaSpace{}%
\AgdaSymbol{(λ}\AgdaSpace{}%
\AgdaBound{i}\AgdaSpace{}%
\AgdaSymbol{→}\AgdaSpace{}%
\AgdaFunction{M≡H}\AgdaSpace{}%
\AgdaBound{i}\AgdaSpace{}%
\AgdaSymbol{→}\AgdaSpace{}%
\AgdaFunction{M≡H}\AgdaSpace{}%
\AgdaBound{i}\AgdaSpace{}%
\AgdaSymbol{→}\AgdaSpace{}%
\AgdaFunction{M≡H}\AgdaSpace{}%
\AgdaBound{i}\AgdaSymbol{)}\AgdaSpace{}%
\AgdaOperator{\AgdaFunction{\AgdaUnderscore{}+\AgdaUnderscore{}}}\AgdaSpace{}%
\AgdaOperator{\AgdaFunction{\AgdaUnderscore{}+ᴴ\AgdaUnderscore{}}}\<%
\end{code}}

\begin{code}[hide]%
\>[0]\AgdaFunction{+Path}\AgdaSpace{}%
\AgdaBound{i}\AgdaSpace{}%
\AgdaSymbol{=}\<%
\\
\>[0][@{}l@{\AgdaIndent{0}}]%
\>[1]\AgdaPrimitive{transp}%
\>[2220I]\AgdaSymbol{(λ}\AgdaSpace{}%
\AgdaBound{j}\AgdaSpace{}%
\AgdaSymbol{→}\AgdaSpace{}%
\AgdaFunction{M≡H}\AgdaSpace{}%
\AgdaSymbol{(}\AgdaBound{i}\AgdaSpace{}%
\AgdaOperator{\AgdaPrimitive{∧}}\AgdaSpace{}%
\AgdaBound{j}\AgdaSymbol{)}\AgdaSpace{}%
\AgdaSymbol{→}\AgdaSpace{}%
\AgdaFunction{M≡H}\AgdaSpace{}%
\AgdaSymbol{(}\AgdaBound{i}\AgdaSpace{}%
\AgdaOperator{\AgdaPrimitive{∧}}\AgdaSpace{}%
\AgdaBound{j}\AgdaSymbol{)}\AgdaSpace{}%
\AgdaSymbol{→}\AgdaSpace{}%
\AgdaFunction{M≡H}\AgdaSpace{}%
\AgdaSymbol{(}\AgdaBound{i}\AgdaSpace{}%
\AgdaOperator{\AgdaPrimitive{∧}}\AgdaSpace{}%
\AgdaBound{j}\AgdaSymbol{))}\<%
\\
\>[.][@{}l@{}]\<[2220I]%
\>[8]\AgdaSymbol{(}\AgdaOperator{\AgdaPrimitive{\textasciitilde{}}}\AgdaSpace{}%
\AgdaBound{i}\AgdaSymbol{)}\AgdaSpace{}%
\AgdaOperator{\AgdaFunction{\AgdaUnderscore{}+\AgdaUnderscore{}}}\<%
\end{code}

\newcommand{\HPlusAssoc}{
\begin{code}%
\>[0]\AgdaFunction{+ᴴassoc}\AgdaSpace{}%
\AgdaSymbol{:}\AgdaSpace{}%
\AgdaFunction{Assoc}\AgdaSpace{}%
\AgdaDatatype{HITOrd}\AgdaSpace{}%
\AgdaOperator{\AgdaFunction{\AgdaUnderscore{}+ᴴ\AgdaUnderscore{}}}\<%
\\
\>[0]\AgdaFunction{+ᴴassoc}\AgdaSpace{}%
\AgdaSymbol{=}\AgdaSpace{}%
\AgdaFunction{transport}\AgdaSpace{}%
\AgdaSymbol{(λ}\AgdaSpace{}%
\AgdaBound{i}\AgdaSpace{}%
\AgdaSymbol{→}\AgdaSpace{}%
\AgdaFunction{Assoc}\AgdaSpace{}%
\AgdaSymbol{(}\AgdaFunction{M≡H}\AgdaSpace{}%
\AgdaBound{i}\AgdaSymbol{)}\AgdaSpace{}%
\AgdaSymbol{(}\AgdaFunction{+Path}\AgdaSpace{}%
\AgdaBound{i}\AgdaSymbol{))}\AgdaSpace{}%
\AgdaFunction{+assoc}\<%
\end{code}}

\begin{code}[hide]%
\>[0]\<%
\\
\>[0]\AgdaKeyword{infixr}\AgdaSpace{}%
\AgdaNumber{36}\AgdaSpace{}%
\AgdaOperator{\AgdaFunction{\AgdaUnderscore{}·\AgdaUnderscore{}}}\AgdaSpace{}%
\AgdaOperator{\AgdaFunction{\AgdaUnderscore{}·ᴴ\AgdaUnderscore{}}}\<%
\\
\>[0]\<%
\end{code}

\newcommand{\MDot}{
\begin{code}%
\>[0]\AgdaOperator{\AgdaFunction{\AgdaUnderscore{}·\AgdaUnderscore{}}}\AgdaSpace{}%
\AgdaSymbol{:}\AgdaSpace{}%
\AgdaDatatype{MutualOrd}\AgdaSpace{}%
\AgdaSymbol{→}\AgdaSpace{}%
\AgdaDatatype{MutualOrd}\AgdaSpace{}%
\AgdaSymbol{→}\AgdaSpace{}%
\AgdaDatatype{MutualOrd}\<%
\\
\>[0]\AgdaInductiveConstructor{𝟎}\AgdaSpace{}%
\AgdaOperator{\AgdaFunction{·}}\AgdaSpace{}%
\AgdaBound{b}\AgdaSpace{}%
\AgdaSymbol{=}\AgdaSpace{}%
\AgdaInductiveConstructor{𝟎}\<%
\\
\>[0]\AgdaCatchallClause{\AgdaBound{a}}\AgdaSpace{}%
\AgdaCatchallClause{\AgdaOperator{\AgdaFunction{·}}}\AgdaSpace{}%
\AgdaCatchallClause{\AgdaInductiveConstructor{𝟎}}\AgdaSpace{}%
\AgdaSymbol{=}\AgdaSpace{}%
\AgdaInductiveConstructor{𝟎}\<%
\\
\>[0]\AgdaCatchallClause{\AgdaBound{a}}\AgdaSpace{}%
\AgdaCatchallClause{\AgdaOperator{\AgdaFunction{·}}}\AgdaSpace{}%
\AgdaCatchallClause{\AgdaSymbol{(}}\AgdaCatchallClause{\AgdaOperator{\AgdaInductiveConstructor{ω\textasciicircum{}}}}\AgdaSpace{}%
\AgdaCatchallClause{\AgdaInductiveConstructor{𝟎}}\AgdaSpace{}%
\AgdaCatchallClause{\AgdaOperator{\AgdaInductiveConstructor{+}}}\AgdaSpace{}%
\AgdaCatchallClause{\AgdaBound{d}}\AgdaSpace{}%
\AgdaCatchallClause{\AgdaOperator{\AgdaInductiveConstructor{[}}}\AgdaSpace{}%
\AgdaCatchallClause{\AgdaBound{r}}\AgdaSpace{}%
\AgdaCatchallClause{\AgdaOperator{\AgdaInductiveConstructor{]}}}\AgdaCatchallClause{\AgdaSymbol{)}}\AgdaSpace{}%
\AgdaSymbol{=}\AgdaSpace{}%
\AgdaBound{a}\AgdaSpace{}%
\AgdaOperator{\AgdaFunction{+}}\AgdaSpace{}%
\AgdaBound{a}\AgdaSpace{}%
\AgdaOperator{\AgdaFunction{·}}\AgdaSpace{}%
\AgdaBound{d}\<%
\\
\>[0]\AgdaCatchallClause{\AgdaSymbol{(}}\AgdaCatchallClause{\AgdaOperator{\AgdaInductiveConstructor{ω\textasciicircum{}}}}\AgdaSpace{}%
\AgdaCatchallClause{\AgdaBound{a}}\AgdaSpace{}%
\AgdaCatchallClause{\AgdaOperator{\AgdaInductiveConstructor{+}}}\AgdaSpace{}%
\AgdaCatchallClause{\AgdaBound{c}}\AgdaSpace{}%
\AgdaCatchallClause{\AgdaOperator{\AgdaInductiveConstructor{[}}}\AgdaSpace{}%
\AgdaCatchallClause{\AgdaBound{r}}\AgdaSpace{}%
\AgdaCatchallClause{\AgdaOperator{\AgdaInductiveConstructor{]}}}\AgdaCatchallClause{\AgdaSymbol{)}}\AgdaSpace{}%
\AgdaCatchallClause{\AgdaOperator{\AgdaFunction{·}}}\AgdaSpace{}%
\AgdaCatchallClause{\AgdaSymbol{(}}\AgdaCatchallClause{\AgdaOperator{\AgdaInductiveConstructor{ω\textasciicircum{}}}}\AgdaSpace{}%
\AgdaCatchallClause{\AgdaBound{b}}\AgdaSpace{}%
\AgdaCatchallClause{\AgdaOperator{\AgdaInductiveConstructor{+}}}\AgdaSpace{}%
\AgdaCatchallClause{\AgdaBound{d}}\AgdaSpace{}%
\AgdaCatchallClause{\AgdaOperator{\AgdaInductiveConstructor{[}}}\AgdaSpace{}%
\AgdaCatchallClause{\AgdaBound{s}}\AgdaSpace{}%
\AgdaCatchallClause{\AgdaOperator{\AgdaInductiveConstructor{]}}}\AgdaCatchallClause{\AgdaSymbol{)}}\AgdaSpace{}%
\AgdaSymbol{=}\<%
\\
\>[0][@{}l@{\AgdaIndent{0}}]%
\>[2]\AgdaOperator{\AgdaFunction{M.ω\textasciicircum{}⟨}}\AgdaSpace{}%
\AgdaBound{a}\AgdaSpace{}%
\AgdaOperator{\AgdaFunction{+}}\AgdaSpace{}%
\AgdaBound{b}\AgdaSpace{}%
\AgdaOperator{\AgdaFunction{⟩}}\AgdaSpace{}%
\AgdaOperator{\AgdaFunction{+}}\AgdaSpace{}%
\AgdaSymbol{(}\AgdaOperator{\AgdaInductiveConstructor{ω\textasciicircum{}}}\AgdaSpace{}%
\AgdaBound{a}\AgdaSpace{}%
\AgdaOperator{\AgdaInductiveConstructor{+}}\AgdaSpace{}%
\AgdaBound{c}\AgdaSpace{}%
\AgdaOperator{\AgdaInductiveConstructor{[}}\AgdaSpace{}%
\AgdaBound{r}\AgdaSpace{}%
\AgdaOperator{\AgdaInductiveConstructor{]}}\AgdaSpace{}%
\AgdaOperator{\AgdaFunction{·}}\AgdaSpace{}%
\AgdaBound{d}\AgdaSymbol{)}\<%
\end{code}}

\newcommand{\HDot}{
\begin{code}%
\>[0]\AgdaOperator{\AgdaFunction{\AgdaUnderscore{}·ᴴ\AgdaUnderscore{}}}\AgdaSpace{}%
\AgdaSymbol{:}\AgdaSpace{}%
\AgdaDatatype{HITOrd}\AgdaSpace{}%
\AgdaSymbol{→}\AgdaSpace{}%
\AgdaDatatype{HITOrd}\AgdaSpace{}%
\AgdaSymbol{→}\AgdaSpace{}%
\AgdaDatatype{HITOrd}\<%
\\
\>[0]\AgdaOperator{\AgdaFunction{\AgdaUnderscore{}·ᴴ\AgdaUnderscore{}}}\AgdaSpace{}%
\AgdaSymbol{=}\AgdaSpace{}%
\AgdaFunction{transport}\AgdaSpace{}%
\AgdaSymbol{(λ}\AgdaSpace{}%
\AgdaBound{i}\AgdaSpace{}%
\AgdaSymbol{→}\AgdaSpace{}%
\AgdaFunction{M≡H}\AgdaSpace{}%
\AgdaBound{i}\AgdaSpace{}%
\AgdaSymbol{→}\AgdaSpace{}%
\AgdaFunction{M≡H}\AgdaSpace{}%
\AgdaBound{i}\AgdaSpace{}%
\AgdaSymbol{→}\AgdaSpace{}%
\AgdaFunction{M≡H}\AgdaSpace{}%
\AgdaBound{i}\AgdaSymbol{)}\AgdaSpace{}%
\AgdaOperator{\AgdaFunction{\AgdaUnderscore{}·\AgdaUnderscore{}}}\<%
\end{code}}

\begin{code}[hide]%
\>[0]\<%
\\
\>[0]\AgdaComment{-- Ordinals of the form ω\textasciicircum{}⟨ c ⟩ are closed under addition}\<%
\\
\\[\AgdaEmptyExtraSkip]%
\>[0]\AgdaFunction{Lm[≮0]}\AgdaSpace{}%
\AgdaSymbol{:}\AgdaSpace{}%
\AgdaSymbol{\{}\AgdaBound{a}\AgdaSpace{}%
\AgdaSymbol{:}\AgdaSpace{}%
\AgdaDatatype{MutualOrd}\AgdaSymbol{\}}\AgdaSpace{}%
\AgdaSymbol{→}\AgdaSpace{}%
\AgdaSymbol{(}\AgdaBound{a}\AgdaSpace{}%
\AgdaOperator{\AgdaDatatype{<}}\AgdaSpace{}%
\AgdaInductiveConstructor{𝟎}\AgdaSymbol{)}\AgdaSpace{}%
\AgdaSymbol{→}\AgdaSpace{}%
\AgdaDatatype{⊥}\<%
\\
\>[0]\AgdaFunction{Lm[≮0]}\AgdaSpace{}%
\AgdaSymbol{\{}\AgdaBound{a}\AgdaSymbol{\}}\AgdaSpace{}%
\AgdaSymbol{()}\<%
\\
\\[\AgdaEmptyExtraSkip]%
\>[0]\AgdaFunction{additive-principal}%
\>[2346I]\AgdaSymbol{:}\AgdaSpace{}%
\AgdaSymbol{\{}\AgdaBound{a}\AgdaSpace{}%
\AgdaBound{b}\AgdaSpace{}%
\AgdaSymbol{:}\AgdaSpace{}%
\AgdaDatatype{MutualOrd}\AgdaSymbol{\}}\<%
\\
\>[.][@{}l@{}]\<[2346I]%
\>[19]\AgdaSymbol{→}\AgdaSpace{}%
\AgdaBound{a}\AgdaSpace{}%
\AgdaOperator{\AgdaDatatype{<}}\AgdaSpace{}%
\AgdaOperator{\AgdaFunction{M.ω\textasciicircum{}⟨}}\AgdaSpace{}%
\AgdaBound{b}\AgdaSpace{}%
\AgdaOperator{\AgdaFunction{⟩}}\AgdaSpace{}%
\AgdaSymbol{→}\AgdaSpace{}%
\AgdaSymbol{(}\AgdaBound{a}\AgdaSpace{}%
\AgdaOperator{\AgdaFunction{+}}\AgdaSpace{}%
\AgdaOperator{\AgdaFunction{M.ω\textasciicircum{}⟨}}\AgdaSpace{}%
\AgdaBound{b}\AgdaSpace{}%
\AgdaOperator{\AgdaFunction{⟩}}\AgdaSymbol{)}\AgdaSpace{}%
\AgdaOperator{\AgdaFunction{≡}}\AgdaSpace{}%
\AgdaOperator{\AgdaFunction{M.ω\textasciicircum{}⟨}}\AgdaSpace{}%
\AgdaBound{b}\AgdaSpace{}%
\AgdaOperator{\AgdaFunction{⟩}}\<%
\\
\>[0]\AgdaFunction{additive-principal}\AgdaSpace{}%
\AgdaInductiveConstructor{<₁}\AgdaSpace{}%
\AgdaSymbol{=}\AgdaSpace{}%
\AgdaFunction{refl}\<%
\\
\>[0]\AgdaFunction{additive-principal}\AgdaSpace{}%
\AgdaSymbol{\{}\AgdaOperator{\AgdaInductiveConstructor{ω\textasciicircum{}}}\AgdaSpace{}%
\AgdaBound{a}\AgdaSpace{}%
\AgdaOperator{\AgdaInductiveConstructor{+}}\AgdaSpace{}%
\AgdaBound{c}\AgdaSpace{}%
\AgdaOperator{\AgdaInductiveConstructor{[}}\AgdaSpace{}%
\AgdaBound{r}\AgdaSpace{}%
\AgdaOperator{\AgdaInductiveConstructor{]}}\AgdaSymbol{\}}\AgdaSpace{}%
\AgdaSymbol{\{}\AgdaBound{b}\AgdaSymbol{\}}\AgdaSpace{}%
\AgdaSymbol{(}\AgdaInductiveConstructor{<₂}\AgdaSpace{}%
\AgdaBound{a<b}\AgdaSymbol{)}%
\>[50]\AgdaKeyword{with}\AgdaSpace{}%
\AgdaFunction{<-tri}\AgdaSpace{}%
\AgdaBound{a}\AgdaSpace{}%
\AgdaBound{b}\<%
\\
\>[0]\AgdaSymbol{...}\AgdaSpace{}%
\AgdaSymbol{|}\AgdaSpace{}%
\AgdaInductiveConstructor{inj₁}\AgdaSpace{}%
\AgdaSymbol{\AgdaUnderscore{}}\AgdaSpace{}%
\AgdaSymbol{=}\AgdaSpace{}%
\AgdaFunction{refl}\<%
\\
\>[0]\AgdaSymbol{...}\AgdaSpace{}%
\AgdaSymbol{|}\AgdaSpace{}%
\AgdaInductiveConstructor{inj₂}\AgdaSpace{}%
\AgdaSymbol{(}\AgdaInductiveConstructor{inj₁}\AgdaSpace{}%
\AgdaBound{b<a}\AgdaSymbol{)}\AgdaSpace{}%
\AgdaSymbol{=}\AgdaSpace{}%
\AgdaFunction{⊥-elim}\AgdaSpace{}%
\AgdaSymbol{(}\AgdaFunction{<-irrefl}\AgdaSpace{}%
\AgdaSymbol{(}\AgdaFunction{<-trans}\AgdaSpace{}%
\AgdaBound{a<b}\AgdaSpace{}%
\AgdaBound{b<a}\AgdaSymbol{))}\<%
\\
\>[0]\AgdaSymbol{...}\AgdaSpace{}%
\AgdaSymbol{|}\AgdaSpace{}%
\AgdaInductiveConstructor{inj₂}\AgdaSpace{}%
\AgdaSymbol{(}\AgdaInductiveConstructor{inj₂}\AgdaSpace{}%
\AgdaBound{a=b}\AgdaSymbol{)}\AgdaSpace{}%
\AgdaSymbol{=}\AgdaSpace{}%
\AgdaFunction{⊥-elim}\AgdaSpace{}%
\AgdaSymbol{(}\AgdaFunction{<-irrefl}\AgdaSpace{}%
\AgdaSymbol{(}\AgdaFunction{<≤-trans}\AgdaSpace{}%
\AgdaBound{a<b}\AgdaSpace{}%
\AgdaSymbol{(}\AgdaInductiveConstructor{inj₂}\AgdaSpace{}%
\AgdaBound{a=b}\AgdaSymbol{)))}\<%
\\
\>[0]\AgdaFunction{additive-principal}\AgdaSpace{}%
\AgdaSymbol{(}\AgdaInductiveConstructor{<₃}\AgdaSpace{}%
\AgdaBound{q}\AgdaSpace{}%
\AgdaSymbol{())}\<%
\\
\\[\AgdaEmptyExtraSkip]%
\>[0]\AgdaFunction{additive-principal-closure}%
\>[2411I]\AgdaSymbol{:}\AgdaSpace{}%
\AgdaSymbol{\{}\AgdaBound{a}\AgdaSpace{}%
\AgdaBound{b}\AgdaSpace{}%
\AgdaBound{c}\AgdaSpace{}%
\AgdaSymbol{:}\AgdaSpace{}%
\AgdaDatatype{MutualOrd}\AgdaSymbol{\}}\<%
\\
\>[2411I][@{}l@{\AgdaIndent{0}}]%
\>[28]\AgdaSymbol{→}\AgdaSpace{}%
\AgdaBound{a}\AgdaSpace{}%
\AgdaOperator{\AgdaDatatype{<}}\AgdaSpace{}%
\AgdaOperator{\AgdaFunction{M.ω\textasciicircum{}⟨}}\AgdaSpace{}%
\AgdaBound{c}\AgdaSpace{}%
\AgdaOperator{\AgdaFunction{⟩}}\AgdaSpace{}%
\AgdaSymbol{→}\AgdaSpace{}%
\AgdaBound{b}\AgdaSpace{}%
\AgdaOperator{\AgdaDatatype{<}}\AgdaSpace{}%
\AgdaOperator{\AgdaFunction{M.ω\textasciicircum{}⟨}}\AgdaSpace{}%
\AgdaBound{c}\AgdaSpace{}%
\AgdaOperator{\AgdaFunction{⟩}}\<%
\\
\>[28]\AgdaSymbol{→}\AgdaSpace{}%
\AgdaSymbol{(}\AgdaBound{a}\AgdaSpace{}%
\AgdaOperator{\AgdaFunction{+}}\AgdaSpace{}%
\AgdaBound{b}\AgdaSymbol{)}\AgdaSpace{}%
\AgdaOperator{\AgdaDatatype{<}}\AgdaSpace{}%
\AgdaOperator{\AgdaFunction{M.ω\textasciicircum{}⟨}}\AgdaSpace{}%
\AgdaBound{c}\AgdaSpace{}%
\AgdaOperator{\AgdaFunction{⟩}}\<%
\\
\>[0]\AgdaFunction{additive-principal-closure}\AgdaSpace{}%
\AgdaSymbol{\{}\AgdaArgument{b}\AgdaSpace{}%
\AgdaSymbol{=}\AgdaSpace{}%
\AgdaInductiveConstructor{𝟎}\AgdaSymbol{\}}\AgdaSpace{}%
\AgdaSymbol{\{}\AgdaBound{c}\AgdaSymbol{\}}\AgdaSpace{}%
\AgdaInductiveConstructor{<₁}\AgdaSpace{}%
\AgdaBound{q}\AgdaSpace{}%
\AgdaSymbol{=}\AgdaSpace{}%
\AgdaBound{q}\<%
\\
\>[0]\AgdaFunction{additive-principal-closure}\AgdaSpace{}%
\AgdaSymbol{\{}\AgdaArgument{b}\AgdaSpace{}%
\AgdaSymbol{=}\AgdaSpace{}%
\AgdaOperator{\AgdaInductiveConstructor{ω\textasciicircum{}}}\AgdaSpace{}%
\AgdaSymbol{\AgdaUnderscore{}}\AgdaSpace{}%
\AgdaOperator{\AgdaInductiveConstructor{+}}\AgdaSpace{}%
\AgdaSymbol{\AgdaUnderscore{}}\AgdaSpace{}%
\AgdaOperator{\AgdaInductiveConstructor{[}}\AgdaSpace{}%
\AgdaSymbol{\AgdaUnderscore{}}\AgdaSpace{}%
\AgdaOperator{\AgdaInductiveConstructor{]}}\AgdaSymbol{\}}\AgdaSpace{}%
\AgdaSymbol{\{}\AgdaBound{c}\AgdaSymbol{\}}\AgdaSpace{}%
\AgdaInductiveConstructor{<₁}\AgdaSpace{}%
\AgdaBound{q}\AgdaSpace{}%
\AgdaSymbol{=}\AgdaSpace{}%
\AgdaBound{q}\<%
\\
\>[0]\AgdaFunction{additive-principal-closure}\AgdaSpace{}%
\AgdaSymbol{\{}\AgdaOperator{\AgdaInductiveConstructor{ω\textasciicircum{}}}\AgdaSpace{}%
\AgdaBound{a}\AgdaSpace{}%
\AgdaOperator{\AgdaInductiveConstructor{+}}\AgdaSpace{}%
\AgdaBound{d}\AgdaSpace{}%
\AgdaOperator{\AgdaInductiveConstructor{[}}\AgdaSpace{}%
\AgdaBound{r}\AgdaSpace{}%
\AgdaOperator{\AgdaInductiveConstructor{]}}\AgdaSymbol{\}}\AgdaSpace{}%
\AgdaSymbol{\{}\AgdaArgument{c}\AgdaSpace{}%
\AgdaSymbol{=}\AgdaSpace{}%
\AgdaBound{c}\AgdaSymbol{\}}\AgdaSpace{}%
\AgdaSymbol{(}\AgdaInductiveConstructor{<₂}\AgdaSpace{}%
\AgdaBound{p}\AgdaSymbol{)}\AgdaSpace{}%
\AgdaInductiveConstructor{<₁}\AgdaSpace{}%
\AgdaSymbol{=}\AgdaSpace{}%
\AgdaInductiveConstructor{<₂}\AgdaSpace{}%
\AgdaBound{p}\<%
\\
\>[0]\AgdaFunction{additive-principal-closure}\AgdaSpace{}%
\AgdaSymbol{\{}\AgdaBound{a}\AgdaSymbol{\}}\AgdaSpace{}%
\AgdaSymbol{\{}\AgdaBound{b}\AgdaSymbol{\}}\AgdaSpace{}%
\AgdaSymbol{(}\AgdaInductiveConstructor{<₂}\AgdaSpace{}%
\AgdaBound{p}\AgdaSymbol{)}\AgdaSpace{}%
\AgdaSymbol{(}\AgdaInductiveConstructor{<₂}\AgdaSpace{}%
\AgdaBound{q}\AgdaSymbol{)}\AgdaSpace{}%
\AgdaKeyword{with}\AgdaSpace{}%
\AgdaFunction{<-tri}\AgdaSpace{}%
\AgdaSymbol{(}\AgdaFunction{fst}\AgdaSpace{}%
\AgdaBound{a}\AgdaSymbol{)}\AgdaSpace{}%
\AgdaSymbol{(}\AgdaFunction{fst}\AgdaSpace{}%
\AgdaBound{b}\AgdaSymbol{)}\<%
\\
\>[0]\AgdaSymbol{...}\AgdaSpace{}%
\AgdaSymbol{|}\AgdaSpace{}%
\AgdaInductiveConstructor{inj₁}\AgdaSpace{}%
\AgdaSymbol{\AgdaUnderscore{}}\AgdaSpace{}%
\AgdaSymbol{=}\AgdaSpace{}%
\AgdaInductiveConstructor{<₂}\AgdaSpace{}%
\AgdaBound{q}\<%
\\
\>[0]\AgdaSymbol{...}\AgdaSpace{}%
\AgdaSymbol{|}\AgdaSpace{}%
\AgdaInductiveConstructor{inj₂}\AgdaSpace{}%
\AgdaBound{b≥a}\AgdaSpace{}%
\AgdaSymbol{=}\AgdaSpace{}%
\AgdaInductiveConstructor{<₂}\AgdaSpace{}%
\AgdaBound{p}\<%
\\
\>[0]\AgdaFunction{additive-principal-closure}\AgdaSpace{}%
\AgdaSymbol{(}\AgdaInductiveConstructor{<₂}\AgdaSpace{}%
\AgdaBound{p}\AgdaSymbol{)}\AgdaSpace{}%
\AgdaSymbol{(}\AgdaInductiveConstructor{<₃}\AgdaSpace{}%
\AgdaBound{q}\AgdaSpace{}%
\AgdaSymbol{())}\<%
\\
\>[0]\AgdaFunction{additive-principal-closure}\AgdaSpace{}%
\AgdaSymbol{(}\AgdaInductiveConstructor{<₃}\AgdaSpace{}%
\AgdaBound{p}\AgdaSpace{}%
\AgdaSymbol{())}\AgdaSpace{}%
\AgdaBound{q}\<%
\\
\>[0]\<%
\end{code}

\newcommand{\ExNonCommAdd}{
\begin{code}%
\>[0]\AgdaFunction{Ex[+NonComm]}%
\>[2506I]\AgdaSymbol{:}%
\>[16]\AgdaFunction{M.𝟏}\AgdaSpace{}%
\AgdaOperator{\AgdaFunction{+}}\AgdaSpace{}%
\AgdaFunction{M.ω}%
\>[27]\AgdaOperator{\AgdaFunction{≡}}\AgdaSpace{}%
\AgdaFunction{M.ω}\<%
\\
\>[.][@{}l@{}]\<[2506I]%
\>[13]\AgdaOperator{\AgdaFunction{×}}%
\>[16]\AgdaFunction{M.ω}\AgdaSpace{}%
\AgdaOperator{\AgdaFunction{+}}\AgdaSpace{}%
\AgdaFunction{M.𝟏}%
\>[27]\AgdaOperator{\AgdaFunction{>}}\AgdaSpace{}%
\AgdaFunction{M.ω}\<%
\\
\>[0]\AgdaFunction{Ex[+NonComm]}\AgdaSpace{}%
\AgdaSymbol{=}\AgdaSpace{}%
\AgdaSymbol{(}\AgdaFunction{refl}\AgdaSpace{}%
\AgdaOperator{\AgdaInductiveConstructor{,}}\AgdaSpace{}%
\AgdaInductiveConstructor{<₃}\AgdaSpace{}%
\AgdaFunction{refl}\AgdaSpace{}%
\AgdaInductiveConstructor{<₁}\AgdaSymbol{)}\<%
\end{code}}

\newcommand{\ExNonCommMul}{
\begin{code}%
\>[0]\AgdaFunction{Ex[·NonComm]}%
\>[2519I]\AgdaSymbol{:}%
\>[16]\AgdaSymbol{(}\AgdaFunction{M.𝟏}\AgdaSpace{}%
\AgdaOperator{\AgdaFunction{+}}\AgdaSpace{}%
\AgdaFunction{M.𝟏}\AgdaSymbol{)}\AgdaSpace{}%
\AgdaOperator{\AgdaFunction{·}}\AgdaSpace{}%
\AgdaFunction{M.ω}%
\>[35]\AgdaOperator{\AgdaFunction{≡}}\AgdaSpace{}%
\AgdaFunction{M.ω}\<%
\\
\>[.][@{}l@{}]\<[2519I]%
\>[13]\AgdaOperator{\AgdaFunction{×}}%
\>[16]\AgdaFunction{M.ω}%
\>[35]\AgdaOperator{\AgdaDatatype{<}}\AgdaSpace{}%
\AgdaFunction{M.ω}\AgdaSpace{}%
\AgdaOperator{\AgdaFunction{+}}\AgdaSpace{}%
\AgdaFunction{M.ω}\<%
\\
\>[13]\AgdaOperator{\AgdaFunction{×}}%
\>[16]\AgdaFunction{M.ω}\AgdaSpace{}%
\AgdaOperator{\AgdaFunction{+}}\AgdaSpace{}%
\AgdaFunction{M.ω}%
\>[35]\AgdaOperator{\AgdaFunction{≡}}\AgdaSpace{}%
\AgdaFunction{M.ω}\AgdaSpace{}%
\AgdaOperator{\AgdaFunction{·}}\AgdaSpace{}%
\AgdaSymbol{(}\AgdaFunction{M.𝟏}\AgdaSpace{}%
\AgdaOperator{\AgdaFunction{+}}\AgdaSpace{}%
\AgdaFunction{M.𝟏}\AgdaSymbol{)}\<%
\\
\>[0]\AgdaFunction{Ex[·NonComm]}\AgdaSpace{}%
\AgdaSymbol{=}\AgdaSpace{}%
\AgdaSymbol{(}\AgdaFunction{refl}\AgdaSpace{}%
\AgdaOperator{\AgdaInductiveConstructor{,}}\AgdaSpace{}%
\AgdaInductiveConstructor{<₃}\AgdaSpace{}%
\AgdaFunction{refl}\AgdaSpace{}%
\AgdaInductiveConstructor{<₁}\AgdaSpace{}%
\AgdaOperator{\AgdaInductiveConstructor{,}}\AgdaSpace{}%
\AgdaFunction{refl}\AgdaSymbol{)}\<%
\end{code}}

\newcommand{\ExCompHAdd}{
\begin{code}%
\>[0]\AgdaFunction{Ex[+ᴴComp]}\AgdaSpace{}%
\AgdaSymbol{:}\AgdaSpace{}%
\AgdaFunction{H.𝟏}\AgdaSpace{}%
\AgdaOperator{\AgdaFunction{+ᴴ}}\AgdaSpace{}%
\AgdaFunction{H.ω}\AgdaSpace{}%
\AgdaOperator{\AgdaFunction{≡}}\AgdaSpace{}%
\AgdaOperator{\AgdaInductiveConstructor{ω\textasciicircum{}}}\AgdaSpace{}%
\AgdaSymbol{(}\AgdaOperator{\AgdaInductiveConstructor{ω\textasciicircum{}}}\AgdaSpace{}%
\AgdaInductiveConstructor{𝟎}\AgdaSpace{}%
\AgdaOperator{\AgdaInductiveConstructor{⊕}}\AgdaSpace{}%
\AgdaInductiveConstructor{𝟎}\AgdaSymbol{)}\AgdaSpace{}%
\AgdaOperator{\AgdaInductiveConstructor{⊕}}\AgdaSpace{}%
\AgdaInductiveConstructor{𝟎}\<%
\\
\>[0]\AgdaFunction{Ex[+ᴴComp]}\AgdaSpace{}%
\AgdaSymbol{=}\AgdaSpace{}%
\AgdaFunction{refl}\<%
\end{code}}

\newcommand{\ExCompHMul}{
\begin{code}%
\>[0]\AgdaFunction{Ex[·ᴴComp]}%
\>[2557I]\AgdaSymbol{:}\AgdaSpace{}%
\AgdaFunction{H.ω}\AgdaSpace{}%
\AgdaOperator{\AgdaFunction{·ᴴ}}\AgdaSpace{}%
\AgdaSymbol{(}\AgdaFunction{H.𝟏}\AgdaSpace{}%
\AgdaOperator{\AgdaFunction{+ᴴ}}\AgdaSpace{}%
\AgdaFunction{H.𝟏}\AgdaSymbol{)}\<%
\\
\>[.][@{}l@{}]\<[2557I]%
\>[11]\AgdaOperator{\AgdaFunction{≡}}\AgdaSpace{}%
\AgdaOperator{\AgdaInductiveConstructor{ω\textasciicircum{}}}\AgdaSpace{}%
\AgdaSymbol{(}\AgdaOperator{\AgdaInductiveConstructor{ω\textasciicircum{}}}\AgdaSpace{}%
\AgdaInductiveConstructor{𝟎}\AgdaSpace{}%
\AgdaOperator{\AgdaInductiveConstructor{⊕}}\AgdaSpace{}%
\AgdaInductiveConstructor{𝟎}\AgdaSymbol{)}\AgdaSpace{}%
\AgdaOperator{\AgdaInductiveConstructor{⊕}}\AgdaSpace{}%
\AgdaOperator{\AgdaInductiveConstructor{ω\textasciicircum{}}}\AgdaSpace{}%
\AgdaSymbol{(}\AgdaOperator{\AgdaInductiveConstructor{ω\textasciicircum{}}}\AgdaSpace{}%
\AgdaInductiveConstructor{𝟎}\AgdaSpace{}%
\AgdaOperator{\AgdaInductiveConstructor{⊕}}\AgdaSpace{}%
\AgdaInductiveConstructor{𝟎}\AgdaSymbol{)}\AgdaSpace{}%
\AgdaOperator{\AgdaInductiveConstructor{⊕}}\AgdaSpace{}%
\AgdaInductiveConstructor{𝟎}\<%
\\
\>[0]\AgdaFunction{Ex[·ᴴComp]}\AgdaSpace{}%
\AgdaSymbol{=}\AgdaSpace{}%
\AgdaFunction{refl}\<%
\end{code}}

\newcommand{\ExHsbAdd}{
\begin{code}%
\>[0]\AgdaFunction{Ex[⊕concat]}\AgdaSpace{}%
\AgdaSymbol{:}\<%
\\
\>[0][@{}l@{\AgdaIndent{0}}]%
\>[4]\AgdaFunction{H.𝟏}%
\>[10]\AgdaOperator{\AgdaFunction{⊕}}\AgdaSpace{}%
\AgdaOperator{\AgdaFunction{H.ω\textasciicircum{}⟨}}\AgdaSpace{}%
\AgdaFunction{H.ω}\AgdaSpace{}%
\AgdaOperator{\AgdaFunction{⟩}}%
\>[36]\AgdaOperator{\AgdaFunction{⊕}}\AgdaSpace{}%
\AgdaFunction{H.ω}\<%
\\
\>[0][@{}l@{\AgdaIndent{0}}]%
\>[1]\AgdaOperator{\AgdaFunction{≡}}%
\>[4]\AgdaOperator{\AgdaInductiveConstructor{ω\textasciicircum{}}}\AgdaSpace{}%
\AgdaInductiveConstructor{𝟎}%
\>[10]\AgdaOperator{\AgdaInductiveConstructor{⊕}}\AgdaSpace{}%
\AgdaOperator{\AgdaInductiveConstructor{ω\textasciicircum{}}}\AgdaSpace{}%
\AgdaSymbol{(}\AgdaOperator{\AgdaInductiveConstructor{ω\textasciicircum{}}}\AgdaSpace{}%
\AgdaSymbol{(}\AgdaOperator{\AgdaInductiveConstructor{ω\textasciicircum{}}}\AgdaSpace{}%
\AgdaInductiveConstructor{𝟎}\AgdaSpace{}%
\AgdaOperator{\AgdaInductiveConstructor{⊕}}\AgdaSpace{}%
\AgdaInductiveConstructor{𝟎}\AgdaSymbol{)}\AgdaSpace{}%
\AgdaOperator{\AgdaInductiveConstructor{⊕}}\AgdaSpace{}%
\AgdaInductiveConstructor{𝟎}\AgdaSymbol{)}%
\>[36]\AgdaOperator{\AgdaInductiveConstructor{⊕}}\AgdaSpace{}%
\AgdaOperator{\AgdaInductiveConstructor{ω\textasciicircum{}}}\AgdaSpace{}%
\AgdaSymbol{(}\AgdaOperator{\AgdaInductiveConstructor{ω\textasciicircum{}}}\AgdaSpace{}%
\AgdaInductiveConstructor{𝟎}\AgdaSpace{}%
\AgdaOperator{\AgdaInductiveConstructor{⊕}}\AgdaSpace{}%
\AgdaInductiveConstructor{𝟎}\AgdaSymbol{)}\AgdaSpace{}%
\AgdaOperator{\AgdaInductiveConstructor{⊕}}\AgdaSpace{}%
\AgdaInductiveConstructor{𝟎}\<%
\\
\>[0]\AgdaFunction{Ex[⊕concat]}\AgdaSpace{}%
\AgdaSymbol{=}\AgdaSpace{}%
\AgdaFunction{refl}\<%
\end{code}}

\newcommand{\ExCompMHAdd}{
\begin{code}%
\>[0]\AgdaFunction{Ex[⊕ᴹComp]}\AgdaSpace{}%
\AgdaSymbol{:}\AgdaSpace{}%
\AgdaFunction{M.𝟏}\AgdaSpace{}%
\AgdaOperator{\AgdaFunction{⊕ᴹ}}\AgdaSpace{}%
\AgdaFunction{M.ω}\AgdaSpace{}%
\AgdaOperator{\AgdaFunction{≡}}\AgdaSpace{}%
\AgdaFunction{M.ω}\AgdaSpace{}%
\AgdaOperator{\AgdaFunction{+}}\AgdaSpace{}%
\AgdaFunction{M.𝟏}\<%
\\
\>[0]\AgdaFunction{Ex[⊕ᴹComp]}\AgdaSpace{}%
\AgdaSymbol{=}\AgdaSpace{}%
\AgdaFunction{refl}\<%
\end{code}}

\newcommand{\ExCompMHMul}{
\begin{code}%
\>[0]\AgdaFunction{Ex[⊗ᴹComp]}\AgdaSpace{}%
\AgdaSymbol{:}\AgdaSpace{}%
\AgdaSymbol{(}\AgdaFunction{M.𝟏}\AgdaSpace{}%
\AgdaOperator{\AgdaFunction{+}}\AgdaSpace{}%
\AgdaFunction{M.𝟏}\AgdaSymbol{)}\AgdaSpace{}%
\AgdaOperator{\AgdaFunction{⊗ᴹ}}\AgdaSpace{}%
\AgdaFunction{M.ω}\AgdaSpace{}%
\AgdaOperator{\AgdaFunction{≡}}\AgdaSpace{}%
\AgdaFunction{M.ω}\AgdaSpace{}%
\AgdaOperator{\AgdaFunction{+}}\AgdaSpace{}%
\AgdaFunction{M.ω}\<%
\\
\>[0]\AgdaFunction{Ex[⊗ᴹComp]}\AgdaSpace{}%
\AgdaSymbol{=}\AgdaSpace{}%
\AgdaFunction{refl}\<%
\end{code}}

\begin{document}

\sloppy

%% Title
\title[Three Equivalent Ordinal Notation Systems in Cubical Agda]
      {Three Equivalent Ordinal Notation Systems \\ in Cubical Agda}

%% Author 1
\author{Fredrik Nordvall Forsberg}
\orcid{0000-0001-6157-9288}
\affiliation{
  \department{Computer and Information Sciences}
  \institution{University of Strathclyde}
  \streetaddress{Livingstone Tower, 26 Richmond Street}
  \city{Glasgow}
  \postcode{G1 1XH}
  \country{United Kingdom}
}
\email{fredrik.nordvall-forsberg@strath.ac.uk}

%% Author 2
\author{Chuangjie Xu}
\orcid{0000-0001-6838-4221}
\affiliation{
  \department{Mathematisches Institut}
  \institution{Ludwig-Maximilians-Universit\"at M\"unchen}
  \streetaddress{Theresienstr.~39}
  \city{Munich}
  \postcode{80333}
  \country{Germany}
}
\email{cj-xu@outlook.com}

%% Author 3
\author{Neil Ghani}
\orcid{0000-0002-3988-2560}
\affiliation{
  \department{Computer and Information Sciences}
  \institution{University of Strathclyde}
  \streetaddress{Livingstone Tower, 26 Richmond Street}
  \city{Glasgow}
  \postcode{G1 1XH}
  \country{United Kingdom}
}
\email{neil.ghani@strath.ac.uk}

\begin{abstract}
  We present three ordinal notation systems representing ordinals
  below $\epsn$ in type theory, using recent type-theoretical
  innovations such as mutual inductive-inductive definitions and
  higher inductive types. We show how ordinal arithmetic can be
  developed for these systems, and how they admit a transfinite
  induction principle. We prove that all three notation systems are
  equivalent, so that we can transport results between them using the
  univalence principle. All our constructions have been implemented in
  cubical Agda.
\end{abstract}

%% 2012 ACM Computing Classification System (CSS) concepts
%% Generate at 'http://dl.acm.org/ccs/ccs.cfm'.
\begin{CCSXML}
<ccs2012>
<concept>
<concept_id>10003752.10003790.10003792</concept_id>
<concept_desc>Theory of computation~Proof theory</concept_desc>
<concept_significance>500</concept_significance>
</concept>
<concept>
<concept_id>10003752.10003790.10011740</concept_id>
<concept_desc>Theory of computation~Type theory</concept_desc>
<concept_significance>500</concept_significance>
</concept>
</ccs2012>
\end{CCSXML}

\ccsdesc[500]{Theory of computation~Proof theory}
\ccsdesc[500]{Theory of computation~Type theory}
%% End of generated code

%% Keywords
\keywords{Ordinal notation, Cantor normal form, inductive-inductive definitions, higher inductive types, cubical Agda.}

%% \maketitle
\maketitle

%% ---------------------------------------- %%
\section{Introduction}
\label{sec:intro}

% TODO: fine-tune references

Ordinals and ordinal notation systems play an important role in
program verification, since they can be used to prove termination of
programs --- using ordinals to verify that programs terminate was
suggested already by Turing~\cite{turing:ordinals}. The idea is to
assign an ordinal to each input, and then prove that the assigned
ordinal decreases for each recursive call. Hence the program must
terminate by the well-foundedness of the order on ordinals.  At first,
such proofs were carried out using pen and
paper~\cite{Floyd:1967,dershowitz:termination}, but with advances in
proof assistants, also machine-checked proofs can be
produced~\cite{MV:ord:acl2, schmitt:ord:key}. As a first step, one
must then represent ordinals inside a theorem prover. This is usually
done via some kind of ordinal notation system (however see Blanchette
\emph{et~al}.~\cite{BPT:card:isa} for well-orders encoded directly in
Isabelle/HOL, and Schmitt~\cite{schmitt:ord:key} for an axiomatic
method, which is implemented in the KeY program verification system).
Typically, ordinals are represented by
trees~\cite{dershowitz:ord:tree,DR:ord:list}; for instance, binary
trees can represent the ordinals below $\epsn$ as follows: the leaf
represents 0, and a tree with subtrees representing ordinals $\alpha$
and $\beta$ represents the sum $\omega^\alpha + \beta$. However, an
ordinal may have multiple such representations. As a result,
traditional approaches to ordinal notation
systems~\cite{buchholz:notation,schuette:book,takeuti:book} usually
have to single out a subset of ordinal terms in order to provide
unique representations. In this paper, we show how modern
type-theoretic features can be used to directly give faithful
representations of ordinals below $\epsn$.

The first feature we use is mutual inductive-inductive
definitions~\cite{nordvallforsberg2013thesis}, which are well
supported in the proof assistant Agda. This allows us to define an
ordinal notation system for ordinals below $\epsn$, simultaneously
with an order relation on it (Section~\ref{sec:mutual}). This means
that we can recover uniqueness of representation, by insisting that
subtrees representing ordinals are given in a decreasing order. This
is similar to the traditional approach which first freely generate
ordinal terms, and then later restrict attention to a subset of
well-behaved terms (Section~\ref{sec:subset}). The advantage of the
mutual approach is that there are no intermediate ``junk'' terms, and
that the more precise types often suggests necessary lemmas to
prove. However this is mostly an ergonomic advantage, since the two
approaches are equivalent (Section~\ref{sec:equiv}).

We also use the feature of higher inductive
types~\cite{lumsdaine:shulman:hit} that has recently been added to
Agda under the \texttt{--cubical} flag~\cite{VMA:cubical:agda}. We
define a different ordinal notation system for ordinals below $\epsn$
as a quotient inductive type~\cite{QIITs}, where we represent ordinals
by finite hereditary multisets (Section~\ref{sec:hit}). Path constructors
are used to identify multiple representations of the same ordinal, so
that we again recover uniqueness. Also this approach is equivalent to
the other two approaches (Section~\ref{sec:equiv}).

Different representations are convenient for different purposes. For
instance, the higher inductive type approach to define the ordinal
notation system is convenient for defining \eg\ the commutative
Hessenberg sum of ordinals (Section~\ref{sec:arith}), while the mutual
representation is convenient for proving transfinite induction
(Section~\ref{sec:ti}). Using the univalence principle~\cite{hottbook}, we can
transport such constructions and properties between the different ordinal notation
systems as needed.

\paragraph{Contributions} We make the following contributions:
\begin{itemize}
\item We give two to our knowledge new ordinal notation systems in
  type theory, representing ordinals below $\epsn$. These can be used
  to verify \eg\ termination of programs inside type-theory-based proof
  assistants such as Agda.
\item We prove that our ordinal notation systems are equivalent, and
  also equivalent to a third, well-known ordinal notation system based
  on a predicate of being in Cantor normal form. This allows us to
  transport constructions and properties between them using the
  univalence principle.
\item We prove that our ordinal notation systems allow the principle
  of transfinite induction. This, and the rest of the development, is
  completely computational and axiom-free, in particular we do not
  need to assume \eg\ excluded middle or countable choice.
\item In general, we show how recent features of Agda such as
  simultaneous definitions and higher inductive types can be used to
  obtain user-friendly constructions, and how to work around common
  pitfalls.
\end{itemize}

\paragraph{Agda Formalization}

Our full Agda development can be found at
\url{https://doi.org/10.5281/zenodo.3588624}.

%% -------------------- %%
\section{Cubical Agda}
\label{sec:cubical:agda}

We start by giving a brief introduction to cubical Agda, an
implementation of Cubical Type Theory~\cite{CCHM} in the Agda proof
assistant~\cite{norell:thesis}. We refer to the Agda Wiki\footnote{Agda Wiki:
  \url{https://wiki.portal.chalmers.se/agda/pmwiki.php}} and the Agda
User Manual\footnote{Agda User Manual:
  \url{https://agda.readthedocs.io/}} for more resources on Agda, and
to Vezzosi, M{\"ortberg} and Abel~\cite{VMA:cubical:agda} for the technical
details of the cubical extension of type theory.

Agda has a hierarchy of \emph{universes} called \AF{Set}s. The Cubical
Agda library\footnote{Cubical Agda library:
  \url{https://github.com/agda/cubical}} renames them to \AF{Type}s to
avoid the confusion with the notion of set in Homotopy Type
Theory~\cite{hottbook}. The lowest universe is now called
\AF{Type$_0$}, and it lives in \AF{Type$_1$}. More generally, there is
a universe \AF{Type}~\AB{$\ell$} \AS{:}
\AF{Type}~\AS{(}\AF{lsuc}~\AB{$\ell$}\AS{)} for each \AB{$\ell$}
\AS{:} \AF{Level}, where \AF{lsuc} \AS{:} \AF{Level} \ASto \AF{Level}
is the successor function of universe levels.

We make use of Agda features such as \emph{mixfix operators},
\emph{implicit arguments} and \emph{generalizable variables} to
improve the readability of our Agda code. In turn, they work as
follows: A mixfix operator may contain one or more name parts and one
or more underscores \AgdaUnderscore{}. When applied, its arguments go
in place of the underscore. For instance, when using the \AF{Level}
maximum function
\AF{\AgdaUnderscore{}$\sqcup$\AgdaUnderscore{}}~\AS{:}~\AF{Level}~\ASto~\AF{Level}~\ASto~\AF{Level},
we can write \AB{$\ell$}~\AF{$\sqcup$}~\AB{$\ell'$} which is the same
as
\AF{\AgdaUnderscore{}$\sqcup$\AgdaUnderscore{}}~\AB{$\ell$}~\AB{$\ell'$}. The
\AgdaUnderscore{} symbol also has other usages: when an argument is
not (explicitly) needed in a definition, or a term can be inferred by
Agda's unifier, we can replace it by~\AgdaUnderscore{}. We can even
omit~\AgdaUnderscore{} using implicit arguments, which are declared
using curly braces \AgdaSymbol{\{\}}. For instance, if we define
\IdFun
then \AF{id}~\AC{zero} type-checks, because the type checker knows
\AC{zero}~\AS{:}~\AD{$\mathbb{N}$} and
\AD{$\mathbb{N}$}~\AS{:}~\AF{Type$_0$} and hence can infer that the
implicit argument \AB{$\ell$} is \AF{lzero} (the lowest \AF{Level}),
and that \AB{A} is \AD{$\mathbb{N}$}. To explicitly give an implicit
argument, we just enclose it in curly braces. For example, we can also
write
\AF{id}~\AS{\{}\AgdaUnderscore{}\AS{\}}~\AS{\{}\AD{$\mathbb{N}$}\AS{\}}~\AC{zero}. We
often want our types and functions to be universe polymorphic by
adding \AF{Level} arguments in the declaration as in the above
example. We can further omit
\AS{\{}\AB{$\ell$}~\AS{:}~\AF{Level}\AS{\}} by using generalizable
variables: throughout our Agda development, we declare \LevelVar and
then bindings for them are inserted automatically in declarations
where they are not bound explicitly. For instance, now the identity
function can be declared as \IdFunNoLevel where \AB{$\ell$} is
implicitly universally quantified. We also use generalizable arguments
for the different notions of ordinal terms considered in this paper.

Agda supports simultaneous definition of several mutually dependent data types such as in the schemes of \emph{inductive-recursive}~\cite{dybjer2000IR} and \emph{inductive-inductive}~\cite{nordvallforsberg2013thesis} definitions. Both schemes permit the simultaneous definition of an inductive type \AB{A}, together with a type family \AB{B} over \AB{A}; the difference between them lies in whether \AB{B} is defined recursively over the inductive structure of \AB{A}, or if \AB{B} is itself inductively defined. The type \AB{A} is allowed to refer to \AB{B} and vice versa, so that one may for instance define \AB{A} simultaneously with a predicate or relation \AB{B} on \AB{A}. In this paper, we will use this to define a type of ordinal notations simultaneously with their order relation (Section~\ref{sec:mutual}). The Agda syntax for mutual definitions is to place the type signature of all the mutually defined data types and/or functions before their definitions.

The cubical mode extends Agda with various features from Cubical Type Theory~\cite{CCHM}. To use Agda's cubical mode, we have to place
\OpCubical
at the top of the file. First of all, cubical Agda has a primitive \emph{interval} type \AF{I} with two distinguished endpoints \AC{i0} and \AC{i1}. Paths in a type \AB{A}, representing equality between elements of \AB{A}, are functions \AF{I}~\ASto~\AB{A}; hence they can be introduced using $\lambda$-abstraction and eliminated using function application. There is a special primitive
\PathP
which can be considered as the type of \emph{dependent paths} whose endpoints are in different types. The type of non-dependent paths is defined by
\Eq
where \AS{\{}\AB{A}~\AS{=}~\AB{A}\AS{\}} tells Agda to bind the implicit argument \AB{A} declared in the type of \AF{\AgdaUnderscore{}\ASequiv\AgdaUnderscore{}} to a variable also named \AB{A}, which is used in the definition of \AF{\AgdaUnderscore{}\ASequiv\AgdaUnderscore{}}.
In this paper, we will need the following path-related proofs from the cubical Agda library:
\ListOfPathFunctions
A type is called a \emph{proposition} if all its elements are identical, and is called a \emph{set} if all its path spaces are propositions. In Agda, this is formulated as follows:
\isPS
These univalent concepts play an important role in the development of mathematics in Homotopy Type Theory. Cubical Agda also supports a general schema of \emph{higher inductive types}~\cite{cubicalHITs}, a generalization of inductive types allowing constructors to produce paths. In this paper, we will construct an ordinal notation system as a higher inductive type (Section~\ref{sec:hit}).

Another important concept from Homotopy Type Theory is the notion of
type equivalence. We say that two types \AB{A} and \AB{B} are
\emph{equivalent}, and write \AB{A}~\AF{$\simeq$}~\AB{B}, if there is a
function $f : \AB{A} \to \AB{B}$ with an two-sided inverse
$g : \AB{B} \to \AB{A}$, and if the proofs that $f$ and $g$ are
inverses are coherent in a suitable sense. Importantly, every
isomorphism (\ie\ a function with a two-sided inverse, but without
coherence conditions on the inverse proofs) gives rise to an
equivalence, \ie\ we have
\IsoEquiv
where we have written \AF{Iso}\AgdaSpace{}\AB{A}\AgdaSpace{}\AB{B} for
the type of isomorphisms between \AB{A} and \AB{B}.
The \emph{univalence principle} \AS{(}\AB{A}~\ASequiv~\AB{B}\AS{)} \AF{$\simeq$} \AS{(}\AB{A}~\AF{$\simeq$}~\AB{B}\AS{)} is provable in cubical Agda. In particular, there is a function \AF{ua}~\AS{:} \AB{A}~\AF{$\simeq$}~\AB{B} \ASto \AB{A}~\ASequiv~\AB{B} generating a path between two types from a proof that they are equivalent. We will use univalence to construct paths between equivalent systems of ordinal notations (Section~\ref{sec:equiv}) and then transport various constructions and proofs between them along these paths (Sections~\ref{sec:arith} and~\ref{sec:ti}).

We will also use the following standard Agda data types:
\begin{itemize}
\item The empty type (with no constructors)
\EmptyType
\item Coproduct types (disjoint unions)
  \SumType
\item $\Sigma$-types (dependent pairs)
  \SigmaType
\item Cartesian products (non-dependent pairs)
  \BinProd
\item The natural numbers, and the standard order relation on them
  \NatDef
  \NatLeq
\end{itemize}
When the type of a variable $x$ can be inferred, we will adopt the notational convention \AS{$\forall$} \AB{x} \ASto \AB{P} for \AS{(}\AB{x}~\AF{:}~\AgdaUnderscore\AS{)} \ASto \AB{P}, and similarly \AS{$\forall$} \AS{\{}\AB{x}\AS{\}} \ASto \AB{P} for \AS{\{}\AB{x}~\AF{:}~\AgdaUnderscore\AS{\}} \ASto \AB{P}.

When reasoning using chains of equations, we may write
\EqSyntax
for readability, where \AB{p} \AS{:} \AB{x} \ASequiv \AB{y} and \AB{q}
\AS{:} \AB{y} \ASequiv \AB{z}. This desugars to uses of transitivity
\AB{p} \AgdaOperator{\AF{∙}} \AB{q}, but has the advantage of keeping \AB{x}, \AB{y} and \AB{z} explicit.

%% ---------------------------------------- %%
\section{Notation Systems for Ordinals Below \texorpdfstring{$\epsn$}{epsilon-zero}}

The classical set-theoretic theory of ordinals defines an ordinal to
be a set $\alpha$ which is transitive
(\ie~$x \in \alpha \to x \subseteq \alpha$) and connected
(\ie~${x \not= y} \to {x \in y} \vee {y \in x}$ for any
$x,y\in\alpha$). For program verification, the perhaps most important
consequence of this definition is that $\in$ is a well ordering on
ordinals --- we hence often write $\alpha < \beta$ for
$\alpha \in \beta$ --- since this implies that properties of ordinals
can be proven by transfinite induction, which in turn implies that
there can be no infinitely descending chains of ordinals
\[
  \alpha_0 > \alpha_1 > \alpha_2 > \ldots
\]
--- in other words, any process that can be assigned a decreasing
sequence of ordinals must terminate.

Obviously the empty set $\emptyset$ is an ordinal (commonly denoted
0), and if $\alpha$ is an ordinal, it is not hard to see that its
\emph{successor} $\alpha + 1 \eqdef \alpha \cup \{\alpha\}$ is also an
ordinal.  This way, we can construct all finite ordinals 1 = 0 + 1, 2
= 1 + 1, 3 = 2 + 1, \ldots, and then take their limit
$\omega = \{0, 1, 2, 3, \ldots\}$.  We can then continue constructing
$\omega + 1$, $\omega + 2$, \ldots, eventually reaching
$\omega + \omega = \omega \cdot 2$, then $\omega \cdot 3$, \ldots and
thus eventually $\omega\cdot \omega = \omega^2$. Iterating this
process, we can construct $\omega^\omega$, and then take the limit of
the sequence
\[
  \omega^\omega, \omega^{\omega^\omega}, \omega^{\omega^{\omega^\omega}}, \ldots
\]
The resulting ordinal is denoted $\epsn$, and is the minimal ordinal
$\alpha$ such that $\omega^\alpha = \alpha$. It is well known that
every ordinal $\alpha$ can be written uniquely in so-called Cantor normal form
\[
\alpha = \omega^{\beta_1} + \omega^{\beta_2} + \cdots + \omega^{\beta_{n}}
\qquad \text{with } \beta_1 \geq \beta_2 \geq \cdots \geq \beta_{n}
\]
for some natural number $n$ and ordinals $\beta_i$ (the special case
$\alpha = 0$ is written as the empty sum with $n = 0$). Our primary
interest in $\epsn$ is that if $\alpha < \epsn$, then every exponent
$\beta_i$ in the Cantor normal form of $\alpha$ satisfies
$\beta_i < \alpha$. Hence if we in turn write
$\beta_i = \omega^{\gamma_1} + \cdots + \omega^{\gamma_{m}}$ in Cantor
normal form, we discover a decreasing sequence
\[
  \alpha > \beta_i > \gamma_j > \ldots
\]
of ordinals, which hence must terminate in finitely many steps.  As a
result, we have a finitary \emph{notation system} which we can hope to
implement inside a theorem prover in order to represent the ideal
concept of ordinals below $\epsn$ in it.  In the rest of this section,
we explore three different approaches for achieving this in Agda.

%% -------------------- %%
\subsection{The Subset Approach \texorpdfstring{\AgdaModule{SigmaOrd}}{SigmaOrd}}
\label{sec:subset}

Traditional approaches to ordinal notation systems such as
Buchholz~\cite{buchholz:notation}, Sch\"utte~\cite{schuette:book} and
Takeuti~\cite{takeuti:book} usually start by generating ordinal terms
inductively, and then single out a subset in order to provide a unique
representation for ordinals. Along this direction, we construct a
notation system of ordinals below $\epsn$ as a sigma type in an Agda
module \AgdaModule{SigmaOrd}.

The first step is to define ordinal terms, which are simply
\emph{binary trees}, albeit with highly suggestive constructor names:
\Tree
The idea is that \AgdaInductiveConstructor{𝟎} represents the ordinal
0, and \AgdaInductiveConstructor{ω\textasciicircum{}} \AgdaBound{a}
\AgdaInductiveConstructor{+} \AgdaBound{b} represents
$\omega^\alpha + \beta$ if \AgdaBound{a} and \AgdaBound{b} represent
$\alpha$ and $\beta$ respectively. However $\omega^\alpha + \beta$
might not be in Cantor normal form, and might have multiple such
representations, because no order constraint has been imposed in the
exponents occurring in \AgdaInductiveConstructor{ω\textasciicircum{}}
\AgdaBound{a} \AgdaInductiveConstructor{+} \AgdaBound{b}. To remedy
this flaw, we define an \emph{ordering} on trees as follows (where
\AgdaGeneralizable{a b c d} \AgdaSymbol{:} \AgdaDatatype{Tree}):
\TreeLt
The first constructor \AgdaInductiveConstructor{<₁} states that
\AgdaInductiveConstructor{𝟎} is smaller than any other tree, and the
constructors \AgdaInductiveConstructor{<₂} and
\AgdaInductiveConstructor{<₃} say that
non-\AgdaInductiveConstructor{𝟎} trees are compared lexicographically.
However, this is not a well-founded order on \AgdaDatatype{Tree}! To recover well-foundedness, we must restrict to trees that are in Cantor normal form. Towards this, we define the non-strict order \AgdaFunction{≥} in terms of the strict order \AgdaDatatype{<}:
\TreeGeq
Then we can define the predicate of being in \emph{Cantor normal
  form}: \AgdaInductiveConstructor{𝟎} is in normal form, and
\AgdaInductiveConstructor{ω\textasciicircum{}} \AgdaBound{a}
\AgdaInductiveConstructor{+} \AgdaBound{b} is in normal form if also
\AgdaBound{a} and \AgdaBound{b} are, and in addition \AgdaBound{a} is
greater than or equal to the first exponent in \AgdaBound{b}, formally expressed
using the following function:
\TreeFst
We construct the predicate \AgdaFunction{isCNF} formally using the following indexed inductive definition:
\isCNF
For instance, if \AgdaGeneralizable{a b c d} \AgdaSymbol{:} \AgdaDatatype{Tree} are in Cantor normal form and \AgdaGeneralizable{a}~\AgdaOperator{\AgdaFunction{≥}}~\AgdaGeneralizable{b}~\AgdaOperator{\AgdaFunction{≥}}~\AgdaGeneralizable{c}~\AgdaOperator{\AgdaFunction{≥}}~\AgdaGeneralizable{d}, then \AgdaFunction{isCNF} \AgdaSymbol{(}\AgdaInductiveConstructor{ω\textasciicircum{}} \AgdaBound{a} \AgdaInductiveConstructor{+} \AgdaInductiveConstructor{ω\textasciicircum{}} \AgdaBound{b} \AgdaInductiveConstructor{+} \AgdaInductiveConstructor{ω\textasciicircum{}} \AgdaBound{c} \AgdaInductiveConstructor{+} \AgdaInductiveConstructor{ω\textasciicircum{}} \AgdaBound{d} \AgdaInductiveConstructor{+} \AgdaInductiveConstructor{𝟎}\AgdaSymbol{)} is inhabited.

Finally, we can form the subset of trees in Cantor normal form by the following dependent pair type:
\SigmaOrd
We are justified in using the ``subset'' terminology, because we can prove that \AgdaFunction{isCNF} is proof-irrelevant, \ie
\isCNFIsPropValued
the proof of which in turn relies on the following facts:
\Sfacts
Therefore, equality on \SO is determined only by the \AgdaDatatype{Tree} component, \ie~we can prove
\SigmaOrdEq
For the formal proofs, we refer to our Agda development.  This
approach gives a faithful representation of ordinals below $\epsn$,
but it is sometimes inconvenient to work with, \eg\ one has to
explicitly prove that all operations preserve being in Cantor normal
form. Agda's termination checker is often happier with curried
functions, which further discourages use of \AgdaFunction{SigmaOrd} as
a programming abstraction.

%% -------------------- %%
\subsection{The Mutual Approach \texorpdfstring{\AgdaModule{MutualOrd}}{MutualOrd}}
\label{sec:mutual}

Instead of considering an imprecise type of trees, including trees not
in Cantor normal form that do not represent ordinals, we can use
Agda's support for mutual definitions to directly generate trees in
Cantor normal form only, by simultaneously defining ordinal terms and
an ordering on them. The idea is to additionally require the term
representing an ordinal $a$ to be greater than or equal to the first
exponent of the term representing an ordinal $b$ when forming the term
representing $\omega^a+b$. Hence we also need to define the operation
which computes the first exponent of an ordinal term
simultaneously. All in all, in a module \AgdaModule{MutualOrd} we
define
\Mdefs
simultaneously by
\MutualOrd
\MOrdLt
\MOrdFst
where we write \AB{a}~\AgdaOperator{\AF{≥}}~\AB{b} \AgdaSymbol{=} \AB{a}~\AgdaOperator{\AF{>}}~\AB{b}~\AgdaOperator{\AD{⊎}}~\AB{a}~\ASequiv~\AB{b}. Obviously this is very similar to the definitions in
Section~\ref{sec:subset}, but this time, \emph{every} term of type
\MO{} satisfies the order constraint because of the
third argument of the constructor
\AgdaInductiveConstructor{ω\textasciicircum{}\AgdaUnderscore{}+\AgdaUnderscore{}[\AgdaUnderscore{}]}. This
means that every term that we can form is already in Cantor normal
form, and there is no need for a separate predicate.

\begin{remark}
  Because of the coproduct hidden in the constructor argument
  \AB{a}~\AgdaOperator{\AF{≥}}~\AgdaFunction{fst}~\AB{b}, and the
  function \AgdaFunction{fst} occurring in it, \MO{} is a
  nested~\cite{bird1998nested}
  inductive-inductive-recursive~\cite{nordvallforsberg2013thesis}
  definition. However, by replacing the constructor with a coproduct
  argument by two constructors (one for each summand), and by defining
  the \emph{graph} of \AgdaFunction{fst} inductively instead of the
  function itself recursively, it is possible to define an equivalent
  non-nested, non-inductive-recursive type. This justifies the
  soundness of our current definition.
\end{remark}

Just like in the subset approach, we can prove that the order relation
\AgdaOperator{\AgdaDatatype{\AgdaUnderscore{}<\AgdaUnderscore{}}} is
proof-irrelevant, \ie{} there is a proof of \AgdaFunction{isProp}~\AgdaSymbol{(}\AgdaGeneralizable{a}~\AgdaOperator{\AgdaDatatype{<}}~\AgdaGeneralizable{b}\AgdaSymbol{)} for every \AgdaGeneralizable{a} and \AgdaGeneralizable{b}.
However, because of the mutual nature of the definitions,
the following facts has to be proved simultaneously:
\Mfacts
One advantage when working with a tighter type such as
\MO compared to the looser \AgdaDatatype{Tree} is
that the right lemma is often naturally suggested in the course of a
construction: for example, when proving \AgdaFunction{MutualOrdIsSet},
the lemma \AgdaFunction{MutualOrd⁼} falls more or less immediately out
as required by one of the subgoals.

For later use in Section~\ref{sec:arith}, we note that we can prove
(constructively) that the ordering
\AgdaOperator{\AgdaDatatype{\AgdaUnderscore{}<\AgdaUnderscore{}}}
is trichotomous, \ie
\Mtri
The proof is the same as a simpler proof for \AgdaDatatype{Tree} from
Section~\ref{sec:subset}, except that we have to make essential use of
\AgdaFunction{MutualOrd⁼}.

%% -------------------- %%
\subsection{The Higher Inductive Approach \texorpdfstring{\AgdaModule{HITOrd}}{HITOrd}}
\label{sec:hit}

In our third approach, an ordinal may have multiple representations,
but we ensure that all of them are identical in the sense of Cubical
Type Theory. We do this by defining a \emph{higher inductive type},
which is given by freely generated terms and paths between them.
Instead of representing an ordinal by a list of ordinal
representations (the exponents in its Cantor normal form), where the
order matters, we instead consider finite multisets of ordinal
representations, where the order of elements does not matter.  Such
finite multisets can be defined in a first-order way as a higher
inductive type, as in Licata~\cite{licata:hott:talk}.  Because the
elements of the multiset again are ordinal representations, what we
need is a higher inductive type of so-called finite hereditary
multisets. We make the following definition in the module
\AgdaModule{HITOrd}:
\HITOrd
This is a higher inductive type, since it is given by listing its
generating term constructors \AgdaInductiveConstructor{𝟎} and \AgdaInductiveConstructor{ω\textasciicircum{}\AgdaUnderscore{}⊕\AgdaUnderscore{}}, as well as its generating path constructors \AgdaInductiveConstructor{swap} and \AgdaInductiveConstructor{trunc}.
Cubical Agda supports higher inductive types natively, and their soundness is guaranteed by the cubical sets model~\cite{cubicalHITs}.
As hinted at by the name of the constructor
\AgdaInductiveConstructor{ω\textasciicircum{}\AgdaUnderscore{}⊕\AgdaUnderscore{}},
our intention for a term
\AgdaInductiveConstructor{ω\textasciicircum{}} \AgdaBound{a}
\AgdaInductiveConstructor{⊕} \AgdaBound{b} is no longer to represent
the non-commutative sum of ordinals $\omega^\alpha + \beta$ where
$\alpha$ and $\beta$ are represented by \AgdaBound{a} and
\AgdaBound{b} respectively, but rather the commutative
\emph{Hessenberg sum} $\omega^{\alpha} \oplus \beta$ (see Section~\ref{sec:hessenberg-arith}). This is
justified by the inclusion of the path constructor
\AgdaInductiveConstructor{swap}, which states that terms with permuted
exponents are identical, as illustrated by the following example (using equational reasoning combinators from the end of Section~\ref{sec:cubical:agda}):
\Hexample
Adding just the \AgdaInductiveConstructor{swap} constructor would
result in a lack of higher-dimensional coherence (\eg~we would expect
\AgdaInductiveConstructor{swap}\AgdaSpace{}\AgdaBound{a}\AgdaSpace{}\AgdaBound{b}\AgdaSpace{}\AgdaBound{c}\AgdaSpace{}
\AgdaFunction{∙}\AgdaSpace{}
\AgdaInductiveConstructor{swap}\AgdaSpace{}\AgdaBound{b}\AgdaSpace{}\AgdaBound{a}\AgdaSpace{}\AgdaBound{c}\AgdaSpace{}
to be the reflexivity path), and so we also include the
\AgdaInductiveConstructor{trunc} constructor which forces
\HO to be a set. This means that we can prove the
following recursion principle for \HO:
\Hrec
This recursion principle states that there is a function from \HO to
any other type \AgdaBound{A} which is closed under the same
``constructors'' as \HO. In other words, to define a function \HO
\ASto \AgdaBound{A} using the recursion principle, \AgdaBound{A} needs
to be a set, and one needs not only a point of \AgdaBound{A} and an
operator \AgdaOperator{\AgdaUnderscore{}⋆\AgdaUnderscore{}}
\AgdaSymbol{:} \AgdaBound{A} \ASto \AgdaBound{A} \ASto \AgdaBound{A},
but also a proof of a ``swap'' rule for
\AgdaOperator{\AgdaUnderscore{}⋆\AgdaUnderscore{}}. This stops us from
defining \eg~a function \AgdaFunction{fst} \AgdaSymbol{:} \HO \ASto
\HO with \AgdaFunction{fst}
\AgdaSymbol{(}\AgdaInductiveConstructor{ω\textasciicircum{}}
\AgdaBound{a} \AgdaInductiveConstructor{⊕} \AgdaBound{b}\AgdaSymbol{)}
\AgdaSymbol{=} \AgdaBound{a} by \AgdaBound{a} \AgdaOperator{⋆}
\AgdaBound{b} \AgdaSymbol{=} \AgdaBound{a}, since this would require
\[
a\; \AgdaSymbol{=}\; a
\AgdaOperator{⋆} \AgdaSymbol{(}b \AgdaOperator{⋆} c \AgdaSymbol{)}\AgdaSpace{}\; \AgdaFunction{≡}\; \AgdaSpace{} \AgdaSymbol{(}b
\AgdaOperator{⋆} a\AgdaSymbol{)} \AgdaOperator{⋆} c \;\AgdaSymbol{=}\; b
\]
for any $a$, $b$ \AgdaSymbol{:} \HO, which is clearly not true. In
general, the recursion principle can be used to define
\emph{non-dependent} functions out of \HO{} that respect the
additional path constructors (we will make use of this in
Section~\ref{sec:equiv}). Similarly, to prove properties of \HO{}, we
will make use of the following induction principle for propositions:
\HindProp
Since the motive $\AgdaBound{P}\;\AgdaBound{x}$ is a proposition for
every $\AgdaBound{x}$ by assumption, we do not need to ask for any
methods involving path constructors --- there are no non-trivial
paths in $\AgdaBound{P}\;\AgdaBound{x}$.
Both the recursion principle and the induction principle for
propositions are instances of the full induction principle, which can
be proven by pattern matching in cubical Agda.

%% -------------------- %%
\subsection{Equivalences Between the Three Approaches}
\label{sec:equiv}

We now wish to show that all three approaches are in fact equivalent,
in the strong sense of Homotopy Type Theory. To show \AgdaBound{A}
\AgdaOperator{\AgdaFunction{≃}} \AgdaBound{B}, it suffices to
construct an isomorphism between \AgdaBound{A} and
\AgdaBound{B}. Hence we construct isomorphisms between \SO and
\MO, and between \MO and \HO. In a new
module \AgdaModule{Equivalences}, we import the previous modules:
\importOrds
Since many names are shared between the imported modules (\eg~both
\AgdaModule{SigmaOrd} and \AgdaModule{MutualOrd} define
\AgdaOperator{\AgdaUnderscore{}\AgdaDatatype{<}\AgdaUnderscore{}} and
\AgdaFunction{fst}), we use the short module names \AgdaModule{S},
\AgdaModule{M} and \AgdaModule{H} to qualify ambiguous names, \eg~we
write
\AgdaOperator{\AgdaUnderscore{}\AgdaDatatype{S.<}\AgdaUnderscore{}}
and \AgdaFunction{S.fst} to refer to the concepts from
\AgdaModule{SigmaOrd}, and
\AgdaOperator{\AgdaUnderscore{}\AgdaDatatype{M.<}\AgdaUnderscore{}}
and \AgdaFunction{M.fst} for the ones from \AgdaModule{MutualOrd}.

%% ---------- %%
\subsubsection{\SO is Equivalent to \MO}

We first construct a function \AgdaFunction{T2M} from \SO to \MO.  To
help Agda's termination checker, we define \AgdaFunction{T2M} in
curried form --- in fact the first component \AgdaBound{a}
\AgdaSymbol{:} \AgdaDatatype{Tree} of the sigma type can even be kept
implicit.  Because \MO is defined simultaneously with its ordering,
when defining \AgdaFunction{T2M} we have to simultaneously prove that
it is monotone:
\TtoMsig
We omit the easy proofs of \AF{T2M[<]} and \AF{T2M[≥fst]} here, but
give the definition of \AF{T2M} since it is computationally relevant:
\TtoMimp

\begin{remark}
When implementing \AF{T2M[≥fst]}, we also need the curried equivalent
\AF{T2M[≡]} of \AF{SigmaOrd⁼} specialised to the image of \AF{T2M},
which can be defined using the path induction
principle. Unfortunately, this detour trips up Agda's termination
checker. We work around this by converting a given path
to an inductively defined propositional equality using the following
construction:
\PropEqfromPath
Here \AgdaModule{P} is the builtin module defining propositional
equality \AgdaDatatype{P.≡} as inductively generated by the
constructor \AgdaInductiveConstructor{P.refl}. With this in hand, we
can pattern match directly on the produced propositional equality
instead of using path induction when implementing \AF{T2M[≡]}, which
placates the termination checker:
\TtoMEqsig
Hopefully the termination checker of cubical Agda will be fixed to
accept a direct proof in future versions.
\end{remark}

For the reverse direction, we convert \MO to \AgdaDatatype{Tree}, and then show that the resulting trees are in Cantor normal form:
\MtoT
\isCNFMtoT
We have omitted the easy proofs that \AF{M2T} is monotone:
\MtoSLemmas
Putting all the pieces together, we can now define maps from \MO to
\SO and vice versa:
\StoM
\MtoS
The proofs that the two compositions of \AgdaFunction{S2M} and \AgdaFunction{M2S} are identities rely on the fact that the orderings are proof-irrelevant; more precisely, they use the lemmas \AgdaFunction{SigmaOrd⁼} and \AgdaFunction{MutualOrd⁼}:
\SMids
Since every isomorphism can be extended to an equivalence (using
\AF{isoToEquiv}), and we have just
constructed an isomorphism between \SO and \MO, we have proven:
\begin{theorem}
\SO and \MO are equivalent, \ie\ there is a proof \eqSM.
\end{theorem}
Using \AgdaFunction{ua} \AgdaSymbol{:} \AgdaBound{A} \AgdaOperator{\AgdaFunction{≃}} \AgdaBound{B} \ASto \AgdaBound{A} \ASequiv \AgdaBound{B}, one direction of the univalence principle, we get a path from \SO to \MO.
\begin{corollary}
\SO and \MO are identical, \ie\ there is a proof \pathSM.
\end{corollary}

%% ---------- %%
\subsubsection{\MO is Equivalent to \HO}

A translation from \MO to \HO is easy: we simply forget about the
order witnesses.
\MtoH
The other direction is more interesting. We need a binary operation
\AgdaOperator{\AgdaUnderscore{}⋆\AgdaUnderscore{}} \AgdaSymbol{:} \MO
\ASto \MO \ASto \MO satisfying the ``swap'' rule in order to use the
recursion principle of \HO. For this purpose, we notice that both \MO
and \HO admit a list structure: \AgdaInductiveConstructor{𝟎} is the
empty list; and in
\AgdaOperator{\AgdaInductiveConstructor{ω\textasciicircum{}}}~\AgdaBound{a}~\AgdaOperator{\AgdaInductiveConstructor{+}}~\AgdaBound{b}~\AgdaOperator{\AgdaInductiveConstructor{[}}~\AgdaBound{r}~\AgdaOperator{\AgdaInductiveConstructor{]}}
and
\AgdaInductiveConstructor{ω\textasciicircum{}}~\AgdaBound{a}~\AgdaInductiveConstructor{⊕}~\AgdaBound{b}
respectively, \AgdaBound{a} is the head and \AgdaBound{b} is the
tail. All lists in \MO are in descending order, while those in \HO are
quotiented by permutations so that it is impossible to access the
order of elements in \HO lists. Coming back to the binary operation
\AgdaOperator{\AgdaUnderscore{}⋆\AgdaUnderscore} with this
list-structure intuition in mind, we see that
\AgdaOperator{\AgdaUnderscore{}⋆\AgdaUnderscore} needs to add its
first argument (regarding it as an element) into its second (regarding
it as a list) such that different orders of doing this result in the
same list. One operation satisfying these requirements is list insertion.
Again, we simultaneously need to prove that \AF{insert} preserves the ordering, since \MO is simultaneously defined with it.
\insertType
The \AF{insert} function implements the standard algorithm for list
insertion (slightly obfuscated by our choice of constructor
names). Similarly the proof \AF{≥fst-insert} follows the same call
structure to show that \AF{insert} is order-preserving.
\insertDef
\fstInsert
Here \AgdaFunction{M.≥𝟎} is a proof that \AgdaBound{a} \AgdaFunction{≥} \AgdaInductiveConstructor{𝟎} for every \AgdaBound{a}.
Using that \AF{<} is trichotomous, \ie\ using \AF{<-tri} to compare any two elements, we can prove that \AF{insert} satisfies the swap rule:
\insertSwap
Hence we can use \AF{insert} and the recursion principle for \HO to define
\HtoM
and then show that \AgdaFunction{M2H} and \AgdaFunction{H2M} form an
isomorphism (the step case is using equational reasoning combinators,
as explained in Section~\ref{sec:cubical:agda}):
\MtoHtoM
We omit the easy proof of the lemma
\insertPlus
used in the final step.
For the other direction, we use the induction principle for propositions:
\HtoMtoH
This is using the following lemma:
\insertHash
Putting everything together, we have proven:

\begin{theorem}
\MO and \HO are equivalent, \ie\ there is a proof \eqMH.
\end{theorem}

\begin{corollary}
\MO and \HO are identical, \ie\ there is a proof \pathMH.
\end{corollary}

%% ---------------------------------------- %%
\section{Ordinal Arithmetic}
\label{sec:arith}

In this section, we demonstrate the usability of our definitions by
showing how well-known arithmetic operations can be defined on them.
We have two quite different data structures representing ordinals
below $\epsn$: hereditary descending lists \MO and finite hereditary
multisets \HO. It is more convenient and efficient to construct the
ordinary arithmetic operations on \MO, because comparing the ``heads''
suffices for the constructions rather than iterating through the whole
ordinal terms. On the other hand, constructing the commutative
arithmetic operations such as Hessenberg sums and products is easier
and more natural on \HO, because orders do not play a role in the
constructions. Hence we implement ordinary ordinal addition and
multiplication on \MO, and Hessenberg addition and multiplication on
\HO. We prove some properties of the operations, and then transport
the constructions and proofs between them using the path
\AgdaFunction{M≡H} \AgdaSymbol{:} \MO \ASequiv \HO.

\subsection{Ordinary Addition and Multiplication}
\label{sec:ordinary-arith}

Ordinal arithmetic extends addition and multiplication from the
natural numbers to all ordinals, including transfinite ones. It is
famously non-commutative: $1 + \omega = \omega$, but
$\omega + 1 > \omega$. On \MO, we have to define addition
whilst simultaneously proving the property that it preserves the
ordering.
\MPlusType
The interesting case of this well-known algorithm, when both
summands are non-zero, is guided by the fact that ordinals of the form
$\omega^\beta$ are so-called additive principal ordinals, \ie~if
$\gamma<\omega^\beta$ then $\gamma + \omega^\beta = \omega^\beta$
(after defining addition, this is not hard to prove for \MO). In
particular if $\alpha < \beta$, then $\omega^\alpha < \omega^\beta$
and hence $\omega^\alpha + \omega^\beta = \omega^\beta$.
The proof that addition preserves the ordering again follows the same
structure as addition itself.
\MPlusDef
\MPlusProof
The construction of an element of \MO contains also a proof that it is
in Cantor normal form. When implementing
\AgdaOperator{\AgdaFunction{\AgdaUnderscore{}+\AgdaUnderscore{}}}
above, the construction (more precisely, the last case when
\AgdaBound{a} \AF{≥} \AgdaBound{b}) explicitly tells us what property
of \AgdaOperator{\AgdaFunction{\AgdaUnderscore{}+\AgdaUnderscore{}}}
is required to show that the sum is in Cantor normal form, and we are
led to prove this property simultaneously.  In the traditional subset
approach, one usually constructs addition on all ordinal terms, and
then proves that it preserves Cantor normal form. However one has to
figure out what property of addition is needed for the proof oneself.
The above example of a ``construction-guided'' proof demonstrates one
advantage of the mutual approach.

Moving from programs to proofs, consider the following type stating that a given binary operation is
associative:
\AssocType
We can construct an easy but lengthy proof
\MPlusAssoc
that \AgdaOperator{\AgdaFunction{\AgdaUnderscore{}+\AgdaUnderscore{}}}
on \MO is associative --- the lengthiness is due to the use of a case
distinction on \AgdaFunction{<-tri}~\AgdaBound{a}~\AgdaBound{b} in the
definition of
\AgdaOperator{\AgdaFunction{\AgdaUnderscore{}+\AgdaUnderscore{}}}.
Now, using the path \AgdaFunction{M≡H} \AgdaSymbol{:} \MO
\ASequiv \HO, we can transport both the operation of addition and the proof that it is associative to an associative operation on \HO:
\HPlus
\HPlusAssoc
where
\PlusPath
is a dependent path from \AgdaOperator{\AgdaFunction{\AgdaUnderscore{}+\AgdaUnderscore{}}} to \AgdaOperator{\AgdaFunction{\AgdaUnderscore{}+ᴴ\AgdaUnderscore{}}}.

Similarly, we can implement the standard multiplication algorithm for
ordinals in Cantor normal form
\MDot
where
\AgdaOperator{\AgdaFunction{M.ω\textasciicircum{}⟨}}~\AgdaBound{a}~\AgdaOperator{\AgdaFunction{⟩}}~\AgdaSymbol{=}~\AgdaOperator{\AgdaInductiveConstructor{ω\textasciicircum{}}}~\AgdaBound{a}~\AgdaOperator{\AgdaInductiveConstructor{+}}~\AgdaInductiveConstructor{𝟎}~\AgdaOperator{\AgdaInductiveConstructor{[}}~\AgdaFunction{≥𝟎}~\AgdaOperator{\AgdaInductiveConstructor{]}}. Since
every case is implemented in terms of previously defined functions,
there is no need to prove any simultaneous lemma about preservation of
the order this time. Again, we can transport this definition to get multiplication on \HO for free:
\HDot

Let us look at some examples. We define the \MO representation of the ordinal 1 by \AgdaFunction{M.𝟏} = \AgdaOperator{\AgdaFunction{M.ω\textasciicircum{}⟨}}~\AgdaInductiveConstructor{𝟎}~\AgdaOperator{\AgdaFunction{⟩}} and the one of $\omega$ by \AgdaFunction{M.ω} = \AgdaOperator{\AgdaFunction{M.ω\textasciicircum{}⟨}}~\AgdaFunction{M.𝟏}~\AgdaOperator{\AgdaFunction{⟩}}. The following examples illustrate that ordinal addition and multiplication are not commutative: for addition, we have $1+\omega = \omega \not= \omega +1$, where the equality is definitional, \ie, it computes:
\ExNonCommAdd
Similarly, for multiplication, we have $2 \cdot \omega = \omega \not= \omega + \omega = \omega \cdot 2$:
\ExNonCommMul
For the examples of \AgdaDatatype{HITOrd}, we define \AgdaOperator{\AgdaFunction{H.ω\textasciicircum{}⟨}}~\AgdaBound{a}~\AgdaOperator{\AgdaFunction{⟩}}~\AgdaSymbol{=}~\AgdaOperator{\AgdaInductiveConstructor{ω\textasciicircum{}}}~\AgdaBound{a}~\AgdaOperator{\AgdaInductiveConstructor{⊕}}~\AgdaInductiveConstructor{𝟎}, \AgdaFunction{H.𝟏} = \AgdaOperator{\AgdaFunction{H.ω\textasciicircum{}⟨}}~\AgdaInductiveConstructor{𝟎}~\AgdaOperator{\AgdaFunction{⟩}} and \AgdaFunction{H.ω} = \AgdaOperator{\AgdaFunction{H.ω\textasciicircum{}⟨}}~\AgdaFunction{H.𝟏}~\AgdaOperator{\AgdaFunction{⟩}}. The operations of addition and multiplication on \HO are obtained by transporting those on \MO along \AgdaFunction{M≡H}. We get this path using (one direction of) the univalence axiom which is constructively provable in cubical Agda. Therefore, closed terms of \HO constructed using these operations can be evaluated into normal form, for instance
\ExCompHAdd
\ExCompHMul
Again, note that both equalities are definitional.

\subsection{Hessenberg Addition and Multiplication}
\label{sec:hessenberg-arith}

Hessenberg arithmetic~\cite{hessenberg} is a variant of ordinal
arithmetic which is commutative and associative, but not continuous in
its second argument. On \HO, Hessenberg addition is simply implemented
as the concatenation operation on finite multisets. Here we define it
by pattern matching on the first argument, which is equivalent to
using the recursion principle. Note that we also have to produce
clauses for \AgdaInductiveConstructor{swap} and
\AgdaInductiveConstructor{trunc}, corresponding to proving that the
defined function preserves the generating paths. For instance, for \AgdaInductiveConstructor{swap}, we have to prove that our definition gives identical results for swapped exponents, \ie, a path
$\AgdaOperator{\AgdaInductiveConstructor{ω\textasciicircum{}}}~\AgdaBound{a}~\AgdaOperator{\AgdaInductiveConstructor{⊕}}~\AgdaOperator{\AgdaInductiveConstructor{ω\textasciicircum{}}}~\AgdaBound{b}~\AgdaOperator{\AgdaInductiveConstructor{⊕}}~\AgdaSymbol{(}\AgdaBound{c}~\AgdaOperator{\AgdaFunction{⊕}}~\AgdaBound{y}\AgdaSymbol{)}~\AgdaOperator{\AgdaFunction{≡}}~\AgdaOperator{\AgdaInductiveConstructor{ω\textasciicircum{}}}~\AgdaBound{b}~\AgdaOperator{\AgdaInductiveConstructor{⊕}}~\AgdaOperator{\AgdaInductiveConstructor{ω\textasciicircum{}}}~\AgdaBound{a}~\AgdaOperator{\AgdaInductiveConstructor{⊕}}~\AgdaSymbol{(}\AgdaBound{c}~\AgdaOperator{\AgdaFunction{⊕}}~\AgdaBound{y}\AgdaSymbol{)}$, which is again an instance of \AgdaInductiveConstructor{swap}:
\Hsum
Our goal is now to justify the notation \AgdaInductiveConstructor{⊕} in the constructor name for \HO by showing that \AgdaOperator{\AgdaFunction{\AgdaUnderscore{}⊕\AgdaUnderscore{}}} is commutative. First we define the property of being commutative:
\CommType
Next we can use the induction principle for propositions to prove that indeed \AgdaOperator{\AgdaFunction{\AgdaUnderscore{}⊕\AgdaUnderscore{}}} is commutative. The base case is given by a simple lemma
\HsumLemmaunitR
and the heavy work of the step case is done by the lemmas
\HsumLemmas
which are also proved using the induction principle \AgdaFunction{indProp}.
Using these lemmas, the proof is as follows:
\HsumComm
By transporting along the reversed path
\pathHM
we get a commutative operation on \MO:
\MHSum
\MHSumComm
where
\HSumPath
is a dependent path from
\AgdaOperator{\AgdaFunction{\AgdaUnderscore{}⊕\AgdaUnderscore{}}} to
\AgdaOperator{\AgdaFunction{\AgdaUnderscore{}⊕ᴹ\AgdaUnderscore{}}}.

We also implement Hessenberg multiplication on \HO, which is
essentially pairwise concatenation of elements in finite multisets. We
first define \AgdaBound{a} \AgdaFunction{∔} \AgdaBound{b} which
concatenates every element of \AgdaBound{a} with \AgdaBound{b}. Again,
we are asked to prove that this respects swapping exponents and
set-truncation.
\HPointAdd
Then we define Hessenberg multiplication \AgdaBound{a}~\AgdaOperator{\AgdaFunction{⊗}}~\AgdaBound{b} by using this operation to concatenate \AgdaBound{a} to every exponent of \AgdaBound{b}:
\HProd
where
\HSumSwap
is easily proved using \AgdaFunction{⊕assoc} and \AgdaFunction{⊕comm}.
Finally we can again transport to get Hessenberg multiplication
on \MO:
\MHProd

Let us look at some examples. Hessenberg addition on \HO can be viewed as a concatenation operation, as illustrated below:
\ExHsbAdd
Again, because univalence is computational in cubical Agda, the transported Hessenberg operations on \MO compute. For instance, we have the following definitional equalities --- note that these equations are not true for ordinary addition and multiplication.
\ExCompMHAdd
\ExCompMHMul

%% ---------------------------------------- %%
\section{Transfinite Induction}
\label{sec:ti}

In this section, we prove transfinite induction for \MO, and then
transport it to transfinite induction for \HO. Already defining an
ordering on \HO by hand is non-trivial, and usually requires several
auxiliary concepts such as a subset relation for multisets and
multiset operations such as union and
subtraction~\cite{dershowitz:termination,BFT:nested:mset}.
Now we can simply transport the ordering on \MO to \HO.
Similarly, it seems easier to prove transfinite induction for \MO and
then transport the proof to \HO if needed, rather than proving it
directly.

\subsection{The Transported Ordering on \texorpdfstring{\HO}{HITOrd}}
We firstly transport the ordering on \MO to \HO as follows:
\HITOrdLt
We can further transport the properties of \AgdaOperator{\AgdaFunction{\AgdaUnderscore{}<\AgdaUnderscore{}}} to \AgdaOperator{\AgdaFunction{\AgdaUnderscore{}<ᴴ\AgdaUnderscore{}}}. For instance, let us define the property of decidability
\DEC
We can easily prove
\MOLtDec
and then transport it to get
\HOLtDec
where
\HITOrdLtPath
is a dependent path from \AgdaOperator{\AgdaFunction{\AgdaUnderscore{}<\AgdaUnderscore{}}} to \AgdaOperator{\AgdaFunction{\AgdaUnderscore{}<ᴴ\AgdaUnderscore{}}}.

Now we demonstrate that the transported property \AgdaOperator{\AgdaFunction{\AgdaUnderscore{}<ᴴ\AgdaUnderscore{}}} computes, like the transported constructions in Section~\ref{sec:arith}. To simplify the examples, we turn \AgdaFunction{<ᴴ-dec} into a boolean-valued function by
\HOLtBool
where \AgdaFunction{isLeft} assigns \AgdaInductiveConstructor{true} : \AgdaDatatype{Bool} to the left summand and \AgdaInductiveConstructor{false} : \AgdaDatatype{Bool} to the right. Here are some examples:
\ExLtComp
Again, note that all equalities displayed are definitional.

\subsection{Transfinite Induction}

Transfinite induction for a type $A$ with respect to a relation
\AgdaOperator{\AgdaBound{\AgdaUnderscore{}<\AgdaUnderscore{}}} on $A$
says that if for every $x$ in $A$ a property $P(x)$ is provable
assuming that $P(y)$ holds for all $y < x$, then $P(x)$ holds for
every $x$.
\TItype
It is well-known that transfinite induction is logically equivalent
to every element of $A$ being accessible, in the following sense:
\Acc
The proof of transfinite induction uses \AF{accInd}. We now show that
every element of \MO is accessible:
\wellfoundedness
The base case~\zeroAcc~is trivial. We show the non-zero case
\plusAcc
using the following two lemmas
\WFlemmas
which are simultaneously proved. The idea is that, to prove the accessibility of \oabr, we have to show that \AgdaBound{z} is accessible for any \AgdaBound{z}~\AgdaOperator{\AgdaFunction{<}}~\oabr. There are three cases: (1)~If \AgdaBound{z} is \AgdaInductiveConstructor{𝟎}, then we are done. (2)~If \AgdaBound{z} is \ocds{} with \AgdaBound{c}~\AgdaOperator{\AgdaFunction{<}}~\AgdaBound{a}, then we use \AgdaFunction{fstAcc}. (3)~If \AgdaBound{z} is \ocds{} with \AgdaBound{c}~\AgdaOperator{\AgdaFunction{≡}}~\AgdaBound{a} and \AgdaBound{b}~\AgdaOperator{\AgdaFunction{<}}~\AgdaBound{d}, then we use \AgdaFunction{sndAcc}.

Combining \AF{accInd} and \AF{WF}, we can now prove:
\begin{theorem}
  Transfinite induction holds for \MO, \ie\ there is
  a proof \MTI.
\end{theorem}

Transporting along our path \AF{M≡H}, we also have:

\begin{corollary}
  Transfinite induction holds for \HO, \ie\ there is
  a proof \HTI.
\end{corollary}

\subsection{All Strictly Descending Sequences are Finite}

Now we consider a simple application of transfinite induction: to
prove that all strictly descending sequences of ordinals below $\epsn$
are finite. Formulating this faithfully in Agda is not easy when
representing sequences as functions from the natural numbers, and one
often ends up with the negative formulation ``there is no strictly
descending sequence'' instead. One may replace finiteness by eventual
zeroness, but this would contradict the strictly descending condition.
As a stronger and \emph{computational} formulation, we introduce the
following notion:
\psdDes
Note that it is not enough to require only
\AgdaBound{f}\AgdaSpace{}\AgdaBound{i}\AgdaSpace{}\AgdaOperator{\AgdaFunction{≡}}\AgdaSpace{}\AgdaInductiveConstructor{𝟎}
in the second summand, as that would allow \AgdaBound{f} to
``restart'' at stage
\AgdaInductiveConstructor{suc}\AgdaSpace{}\AgdaBound{i}.  This notion
is obviously weaker than the notion of being strictly descending:
\strDes
The following facts of pseudo-descendingness are trivial but play an important role in the proof.
\psdDesFacts
where inequality
\AgdaOperator{\AgdaDatatype{\AgdaUnderscore{}≤ᴺ\AgdaUnderscore{}}} of
natural numbers is inductively defined in the standard way.
Moreover, we say that a sequence $f$ is \emph{eventually zero} if we can find
an $n$ such that $f(i)$ takes the value zero for every $i$ after $n$:
\evZero
One can easily prove the following fact of eventual-zeroness:
\evZeroFact
Now we can formulate our result positively as follows:
\begin{theorem}
  Every pseudo-descending sequence is eventually zero, \ie~there is a
  proof\\[2pt]
  \AF{PD2EZ}\AgdaSpace{}\AgdaSymbol{:}\AgdaSpace{}\PDtoEZ.
  \end{theorem}
\begin{proof}
  We prove the statement using transfinite induction on
  \AgdaBound{f}\AgdaSpace{}\AgdaNumber{0}, \ie~we use the
  following motive:
  \PDtoEZPred
  We have to prove the following induction step:
  \PDtoEZStep
  It consists of two cases: (1)~If \AgdaBound{f}~\AgdaNumber{0}~\AgdaOperator{\AgdaFunction{>}}~\AgdaInductiveConstructor{𝟎}, then \AgdaBound{f}~\AgdaNumber{1}~\AgdaOperator{\AgdaFunction{<}}~\AgdaBound{x} by the fact \AgdaFunction{nonzeroPoint}. Hence \AgdaBound{f}~\AgdaFunction{$\circ$}~\AgdaInductiveConstructor{suc} is eventually zero by the hypothesis~\AgdaBound{h}, and so is \AgdaBound{f}\,~by the fact \AgdaFunction{eventually-zero-cons}. (2)~If \AgdaBound{f}~\AgdaNumber{0}~\AgdaOperator{\AgdaFunction{≡}}~\AgdaInductiveConstructor{𝟎}, then \AgdaBound{f}\,~is constantly zero by the fact~\AgdaFunction{zeroPoint}. Hence we can take \PDtoEZProof.
\end{proof}
The algorithm encoded in the above proof checks the values of
\AgdaBound{f}\AgdaSpace{}\AgdaNumber{0},
\AgdaBound{f}\AgdaSpace{}\AgdaNumber{1}, \ldots in turn, until it
finds a zero point. By construction, it will thus find the least
\AgdaBound{n} such that
\AgdaBound{f}~\AgdaBound{i}~\ASequiv~\AgdaInductiveConstructor{𝟎} for
all \AgdaBound{i}~\AgdaOperator{\AgdaFunction{≥ᴺ}}~\AgdaBound{n}. The
transfinite induction principle proves that this procedure is
terminating, using the assumption of pseudo-descendingness.

Because strict descendingness implies the pseudo notion, the negative formulation is a simple corollary.
\begin{corollary}
There is no strictly descending sequence, \ie\ there is a proof \NSDS.
\end{corollary}

%% ---------------------------------------- %%
\section{Comparison with Related Work}

In this section, we compare existing work with our development.

\paragraph{Trees as Ordinals}

The relationship between ordinals --- especially ordinals below
$\epsn$ --- and various tree structures is of course well known, and
more or less folklore. Dershowitz~\cite{dershowitz:ord:tree} gives an
overview of different ordinal representations using finite trees, and
Dershowitz and Reingold~\cite{DR:ord:list} construct binary trees
using Lisp-like list structures. This is similar to our definition
\MO, but our systems provide \emph{unique} representations of
ordinals. Jervell~\cite{jervell:finiteTrees} gives a clever total
ordering on finite trees with $\epsn$ the supremum of all binary
trees. It is not straightforward to encode and work with this ordering
in a proof assistant.

\paragraph{Ordinals in Type Theory}
Surprisingly large ordinals can be constructed in basic Martin-L\"of
Type Theory with primitive type of (countable) ordinals, but no
recursion principle for it. Coquand, Hancock and
Setzer~\cite{coquand:ord} show that already in this setting, one can
reach $\phi_{\epsn}(0)$, where $\phi_{\alpha}$ is the Veblen
hierarchy.  Hancock~\cite{hancock:thesis} uses a class of predicate
transformers called lenses to give a clean proof of (half of)
Hancock's conjecture: Martin-L\"of Type Theory with $n$ universes can
reach $\phi_{\phi_{\phi_{\ldots}(0)}(0)}(0)$ with $n$ nestings of
$\phi_{\epsn}(0)$. In contrast, in our work we are not restricting
ourselves to a spartan type theory, but try to take full advantage of
all of Agda, with the goal of producing an easy-to-use representation.
It is clear that we can draw much inspiration from this line of work
when going beyond $\epsn$. See also Setzer~\cite{Setzer:prooftheory}
for a general overview of the ordinals that can be constructed in
different variations of type theory.

\paragraph{Formalisations}

Several formalisations of ordinals and ordinal notation systems exist
in the literature. Manolios and Vroon\cite{MV:ord:acl2} represents
ordinals below $\epsn$ in the ACL2 theorem prover, based on a
variation of Cantor normal form with
\[
\omega^{\beta_1}c_1 + \ldots + \omega^{\beta_n}c_n \quad \text{with $\beta_1 > \ldots > \beta_n$ and all $c_i$ finite}
\]
This is similar to our \SO representation, except that there are no
mechanical guarantees that given inputs actually are in Cantor normal
form. They also provide algorithms for ordinal arithmetic and
comparisons of ordinals, but their correctness proofs have to assume
that the given inputs are in Cantor normal form. In contrast, it is
not possible to construct ordinal terms not in Cantor normal form in
our systems. Similarly, Cast\'eran and Contejean~\cite{CC:ord:coq} and
Grimm~\cite{grimm:ord:coq} develop significant theories of ordinals
below $\epsn$ in Coq, including arithmetic operations and transfinite
induction. This is again similar to our \SO approach (a choice perhaps
made because Coq to date does not support simultaneous definitions or
higher inductive types, which are needed for the \MO and \HO
approaches respectively).

\paragraph{Finite Multisets}

In Isabelle/HOL, Blanchette, Fleury and Traytel~\cite{BFT:nested:mset}
define an inductive datatype of hereditary multisets to represent
ordinals below $\epsn$, similar to our $\HO$ \mbox{approach}. The
representation relies on the notion of multisets in Isabelle's
standard library, which are defined as natural number-valued functions
with a finite support. This can be constructively problematic, for
instance when defining ordinal exponentiation. In contrast, our use of
higher inductive types to define multisets means that our datatypes
are reassuringly first-order. Because hereditary multisets are viewed
as a subtype of nested multisets, the nested multiset ordering and its
well-foundedness proof are ``lifted'' to the hereditary multisets
using the sophisticated machinery in Isabelle. However, defining the
nested multiset ordering~\cite{dershowitz:termination} is non-trivial
and proving its well-foundedness is challenging as admitted
in~\cite{BFT:nested:mset}. In comparison, our ordinal notation system
\MO is convenient to work with for instance to prove its
well-foundedness. By showing that it is equivalent to hereditary
multisets \mbox{\HO}, we obtain also a well-foundedness proof for the
latter.

%% ---------------------------------------- %%
\section{Concluding Remarks}

We have used modern features of cubical Agda such as simultaneous
definitions and higher inductive types to faithfully represent
ordinals below $\epsn$, and shown that our definitions are easy to
work with by defining common operations on, and proofs about, our
ordinal notation systems. Our development is fully constructive.

Of course, in the world of ordinals, $\epsn$ is tiny; already
Martin-L\"of Type Theory with W-types and only one universe has
proof-theoretic strength well beyond
$\epsn$~\cite{Setzer:prooftheory}, and simultaneous
inductive-recursive definitions are known to increase the
proof-theoretic strength even further (a consequence of Hancock's
Conjecture~\cite{hancock:thesis}). Similarly Lumsdaine and
Shulman~\cite{lumsdaine:shulman:hit} show that adding recursive higher
inductive types increases the power of type theory by considering in
particular a higher inductive type encoding of a variation of Brouwer
tree ordinals. To verify \eg~termination of programs exhausting the
strength of such systems, one would have to define even stronger
ordinal notation systems. We conjecture that powerful definitional
principles such as simultaneous inductive-recursive definitions and
higher inductive types --- perhaps combined --- can be used to
faithfully represent also larger ordinals, and hence be useful for
such program verification problems.

\begin{acks}
We thank Nicolai Kraus, Helmut Schwichtenberg, Ryota Akiyoshi, Nils K\"opp, Masahiko Sato and Anders M\"ortberg for many interesting and illuminating discussions, and the anonymous reviewers for their helpful suggestions and comments. This work was supported by funding from the \grantsponsor{EPSRC}{Engineering and Physical Sciences Research Council}{https://epsrc.ukri.org/} [grant number \grantnum{EPSRC}{EP/M016951/1}], the \grantsponsor{Humboldt}{Alexander von Humboldt Foundation}{https://www.humboldt-foundation.de/}, and the \grantsponsor{LMU}{LMUexcellent initiative}{http://www.en.uni-muenchen.de/about_lmu/research/excellence_initiative/}.
\end{acks}

%% Bibliography
\balance
\bibliography{ordbib}


\begin{thebibliography}{30}

%%% the Web address.

\ifx \showCODEN    \undefined \def \showCODEN     #1{\unskip}     \fi
\ifx \showDOI      \undefined \def \showDOI       #1{#1}\fi
\ifx \showISBNx    \undefined \def \showISBNx     #1{\unskip}     \fi
\ifx \showISBNxiii \undefined \def \showISBNxiii  #1{\unskip}     \fi
\ifx \showISSN     \undefined \def \showISSN      #1{\unskip}     \fi
\ifx \showLCCN     \undefined \def \showLCCN      #1{\unskip}     \fi
\ifx \shownote     \undefined \def \shownote      #1{#1}          \fi
\ifx \showarticletitle \undefined \def \showarticletitle #1{#1}   \fi
\ifx \showURL      \undefined \def \showURL       {\relax}        \fi
% invisible to TeX
\providecommand\bibfield[2]{#2}
\providecommand\bibinfo[2]{#2}
\providecommand\natexlab[1]{#1}
\providecommand\showeprint[2][]{arXiv:#2}

\bibitem[\protect\citeauthoryear{Altenkirch, Capriotti, Dijkstra, Kraus, and
  Nordvall~Forsberg}{Altenkirch et~al\mbox{.}}{2018}]%
        {QIITs}
\bibfield{author}{\bibinfo{person}{Thorsten Altenkirch}, \bibinfo{person}{Paolo
  Capriotti}, \bibinfo{person}{Gabe Dijkstra}, \bibinfo{person}{Nicolai Kraus},
  {and} \bibinfo{person}{Fredrik Nordvall~Forsberg}.}
  \bibinfo{year}{2018}\natexlab{}.
\newblock \showarticletitle{Quotient inductive-inductive types}. In
  \bibinfo{booktitle}{\emph{Foundations of Software Science and Computation
  Structures}} \emph{(\bibinfo{series}{Lecture Notes in Computer Science})},
  \bibfield{editor}{\bibinfo{person}{Christel Baier} {and} \bibinfo{person}{Ugo
  Dal~Lago}} (Eds.), Vol.~\bibinfo{volume}{10803}.
  \bibinfo{publisher}{Springer}, \bibinfo{address}{Heidelberg, Germany},
  \bibinfo{pages}{293--310}.
\newblock


\bibitem[\protect\citeauthoryear{Bird and Meertens}{Bird and Meertens}{1998}]%
        {bird1998nested}
\bibfield{author}{\bibinfo{person}{Richard Bird} {and} \bibinfo{person}{Lambert
  Meertens}.} \bibinfo{year}{1998}\natexlab{}.
\newblock \showarticletitle{Nested datatypes}. In
  \bibinfo{booktitle}{\emph{Mathematics of Program Construction}}
  \emph{(\bibinfo{series}{Lecture Notes in Computer Science})},
  \bibfield{editor}{\bibinfo{person}{Johan Jeuring}} (Ed.),
  Vol.~\bibinfo{volume}{1422}. \bibinfo{publisher}{Springer},
  \bibinfo{address}{Heidelberg, Germany}, \bibinfo{pages}{52--67}.
\newblock


\bibitem[\protect\citeauthoryear{Blanchette, Fleury, and Traytel}{Blanchette
  et~al\mbox{.}}{2017}]%
        {BFT:nested:mset}
\bibfield{author}{\bibinfo{person}{Jasmin~Christian Blanchette},
  \bibinfo{person}{Mathias Fleury}, {and} \bibinfo{person}{Dmitriy Traytel}.}
  \bibinfo{year}{2017}\natexlab{}.
\newblock \showarticletitle{Nested multisets, hereditary multisets, and
  syntactic ordinals in {I}sabelle/{HOL}}. In \bibinfo{booktitle}{\emph{Formal
  Structures for Computation and Deduction}} \emph{(\bibinfo{series}{Leibniz
  International Proceedings in Informatics (LIPIcs)})},
  \bibfield{editor}{\bibinfo{person}{Dale Miller}} (Ed.),
  Vol.~\bibinfo{volume}{84}. \bibinfo{publisher}{Schloss
  Dagstuhl--Leibniz-Zentrum f{\"u}r Informatik}, \bibinfo{address}{Dagstuhl,
  Germany}, \bibinfo{pages}{11:1--11:18}.
\newblock


\bibitem[\protect\citeauthoryear{Blanchette, Popescu, and Traytel}{Blanchette
  et~al\mbox{.}}{2014}]%
        {BPT:card:isa}
\bibfield{author}{\bibinfo{person}{Jasmin~Christian Blanchette},
  \bibinfo{person}{Andrei Popescu}, {and} \bibinfo{person}{Dmitriy Traytel}.}
  \bibinfo{year}{2014}\natexlab{}.
\newblock \showarticletitle{Cardinals in {I}sabelle/{HOL}}. In
  \bibinfo{booktitle}{\emph{Interactive Theorem Proving}}
  \emph{(\bibinfo{series}{Lecture Notes in Computer Science})},
  \bibfield{editor}{\bibinfo{person}{Gerwin Klein} {and} \bibinfo{person}{Ruben
  Gamboa}} (Eds.), Vol.~\bibinfo{volume}{8558}. \bibinfo{publisher}{Springer},
  \bibinfo{address}{Heidelberg, Germany}, \bibinfo{pages}{111--127}.
\newblock


\bibitem[\protect\citeauthoryear{Buchholz}{Buchholz}{1991}]%
        {buchholz:notation}
\bibfield{author}{\bibinfo{person}{Wilfried Buchholz}.}
  \bibinfo{year}{1991}\natexlab{}.
\newblock \showarticletitle{Notation systems for infinitary derivations}.
\newblock \bibinfo{journal}{\emph{Archive for {M}athematical {L}ogic}}
  \bibinfo{volume}{30} (\bibinfo{year}{1991}), \bibinfo{pages}{227--296}.
\newblock


\bibitem[\protect\citeauthoryear{Cast\'{e}ran and Contejean}{Cast\'{e}ran and
  Contejean}{2006}]%
        {CC:ord:coq}
\bibfield{author}{\bibinfo{person}{Pierre Cast\'{e}ran} {and}
  \bibinfo{person}{Evelyne Contejean}.} \bibinfo{year}{2006}\natexlab{}.
\newblock \bibinfo{title}{On ordinal notations}.  (\bibinfo{year}{2006}).
\newblock
\newblock
\shownote{Available at \url{http://coq.inria.fr/V8.2pl1/contribs/Cantor.html}.}


\bibitem[\protect\citeauthoryear{Cohen, Coquand, Huber, and M\"ortberg}{Cohen
  et~al\mbox{.}}{2018}]%
        {CCHM}
\bibfield{author}{\bibinfo{person}{Cyril Cohen}, \bibinfo{person}{Thierry
  Coquand}, \bibinfo{person}{Simon Huber}, {and} \bibinfo{person}{Anders
  M\"ortberg}.} \bibinfo{year}{2018}\natexlab{}.
\newblock \showarticletitle{Cubical {T}ype {T}heory: a constructive
  interpretation of the {U}nivalence {A}xiom}. In
  \bibinfo{booktitle}{\emph{21st International Conference on Types for Proofs
  and Programs 2015}} \emph{(\bibinfo{series}{Leibniz International Proceedings
  in Informatics (LIPIcs)})}, \bibfield{editor}{\bibinfo{person}{Tarmo
  Uustalu}} (Ed.), Vol.~\bibinfo{volume}{69}. \bibinfo{publisher}{Schloss
  Dagstuhl--Leibniz-Zentrum f{\"u}r Informatik}, \bibinfo{address}{Dagstuhl,
  Germany}, \bibinfo{pages}{5:1--5:34}.
\newblock


\bibitem[\protect\citeauthoryear{Coquand, Hancock, and Setzer}{Coquand
  et~al\mbox{.}}{1997}]%
        {coquand:ord}
\bibfield{author}{\bibinfo{person}{Thierry Coquand}, \bibinfo{person}{Peter
  Hancock}, {and} \bibinfo{person}{Anton Setzer}.}
  \bibinfo{year}{1997}\natexlab{}.
\newblock \bibinfo{title}{Ordinals in type theory}.
\newblock \bibinfo{howpublished}{Invited talk for CSL '97}.
\newblock


\bibitem[\protect\citeauthoryear{Coquand, Huber, and M\"{o}rtberg}{Coquand
  et~al\mbox{.}}{2018}]%
        {cubicalHITs}
\bibfield{author}{\bibinfo{person}{Thierry Coquand}, \bibinfo{person}{Simon
  Huber}, {and} \bibinfo{person}{Anders M\"{o}rtberg}.}
  \bibinfo{year}{2018}\natexlab{}.
\newblock \showarticletitle{On Higher Inductive Types in Cubical Type Theory}.
  In \bibinfo{booktitle}{\emph{Logic in Computer Science}}.
  \bibinfo{publisher}{ACM}, \bibinfo{address}{New York, USA},
  \bibinfo{pages}{255--264}.
\newblock
\showISBNx{978-1-4503-5583-4}


\bibitem[\protect\citeauthoryear{Dershowitz}{Dershowitz}{1993}]%
        {dershowitz:ord:tree}
\bibfield{author}{\bibinfo{person}{Nachum Dershowitz}.}
  \bibinfo{year}{1993}\natexlab{}.
\newblock \showarticletitle{Trees, ordinals and termination}. In
  \bibinfo{booktitle}{\emph{Theory and Practice of Software Development}}
  \emph{(\bibinfo{series}{Lecture Notes in Computer Science})},
  \bibfield{editor}{\bibinfo{person}{Marie-Claude Gaudel} {and}
  \bibinfo{person}{Jean-Pierre Jouannaud}} (Eds.), Vol.~\bibinfo{volume}{668}.
  \bibinfo{publisher}{Springer}, \bibinfo{address}{Heidelberg, Germany},
  \bibinfo{pages}{243--250}.
\newblock


\bibitem[\protect\citeauthoryear{Dershowitz and Manna}{Dershowitz and
  Manna}{1979}]%
        {dershowitz:termination}
\bibfield{author}{\bibinfo{person}{Nachum Dershowitz} {and}
  \bibinfo{person}{Zohar Manna}.} \bibinfo{year}{1979}\natexlab{}.
\newblock \showarticletitle{Proving termination with multiset orderings}.
\newblock \bibinfo{journal}{\emph{Commun. ACM}} \bibinfo{volume}{22},
  \bibinfo{number}{8} (\bibinfo{year}{1979}), \bibinfo{pages}{465--476}.
\newblock


\bibitem[\protect\citeauthoryear{Dershowitz and Reingold}{Dershowitz and
  Reingold}{1992}]%
        {DR:ord:list}
\bibfield{author}{\bibinfo{person}{Nachum Dershowitz} {and}
  \bibinfo{person}{Edward~M. Reingold}.} \bibinfo{year}{1992}\natexlab{}.
\newblock \showarticletitle{Ordinal arithmetic with list structures}. In
  \bibinfo{booktitle}{\emph{Logical Foundations of Computer Science}}
  \emph{(\bibinfo{series}{Lecture Notes in Computer Science})},
  \bibfield{editor}{\bibinfo{person}{Anil Nerode} {and}
  \bibinfo{person}{Michael Taitslin}} (Eds.), Vol.~\bibinfo{volume}{620}.
  \bibinfo{publisher}{Springer}, \bibinfo{address}{Heidelberg, Germany},
  \bibinfo{pages}{117--138}.
\newblock


\bibitem[\protect\citeauthoryear{Dybjer}{Dybjer}{2000}]%
        {dybjer2000IR}
\bibfield{author}{\bibinfo{person}{Peter Dybjer}.}
  \bibinfo{year}{2000}\natexlab{}.
\newblock \showarticletitle{A general formulation of simultaneous
  inductive-recursive definitions in type theory}.
\newblock \bibinfo{journal}{\emph{Journal of Symbolic Logic}}
  \bibinfo{volume}{65}, \bibinfo{number}{2} (\bibinfo{year}{2000}),
  \bibinfo{pages}{525--549}.
\newblock


\bibitem[\protect\citeauthoryear{Floyd}{Floyd}{1967}]%
        {Floyd:1967}
\bibfield{author}{\bibinfo{person}{Robert~W. Floyd}.}
  \bibinfo{year}{1967}\natexlab{}.
\newblock \showarticletitle{Assigning Meanings to Programs}. In
  \bibinfo{booktitle}{\emph{Symposium on Applied Mathematics}},
  \bibfield{editor}{\bibinfo{person}{J.T. Schwartz}} (Ed.),
  Vol.~\bibinfo{volume}{19}. \bibinfo{publisher}{American Mathematical
  Society}, \bibinfo{address}{Providence, USA}, \bibinfo{pages}{19--32}.
\newblock


\bibitem[\protect\citeauthoryear{Grimm}{Grimm}{2013}]%
        {grimm:ord:coq}
\bibfield{author}{\bibinfo{person}{Jos\'{e} Grimm}.}
  \bibinfo{year}{2013}\natexlab{}.
\newblock \bibinfo{booktitle}{\emph{Implementation of three types of ordinals
  in {C}oq}}.
\newblock \bibinfo{type}{{T}echnical {R}eport} RR-8407.
  \bibinfo{institution}{INRIA}.
\newblock
\newblock
\shownote{Available at \url{https://hal.inria.fr/hal-00911710}.}


\bibitem[\protect\citeauthoryear{Hancock}{Hancock}{2000}]%
        {hancock:thesis}
\bibfield{author}{\bibinfo{person}{Peter Hancock}.}
  \bibinfo{year}{2000}\natexlab{}.
\newblock \emph{\bibinfo{title}{Ordinals and Interactive Programs}}.
\newblock \bibinfo{thesistype}{Ph.D. Dissertation}. \bibinfo{school}{University
  of Edinburgh}.
\newblock


\bibitem[\protect\citeauthoryear{Hessenberg}{Hessenberg}{1906}]%
        {hessenberg}
\bibfield{author}{\bibinfo{person}{Gerhard Hessenberg}.}
  \bibinfo{year}{1906}\natexlab{}.
\newblock \bibinfo{booktitle}{\emph{Grundbegriffe der Mengenlehre}}.
  Vol.~\bibinfo{volume}{1}.
\newblock \bibinfo{publisher}{Vandenhoeck \& Ruprecht},
  \bibinfo{address}{G{\"o}ttingen, Germany}.
\newblock


\bibitem[\protect\citeauthoryear{Jervell}{Jervell}{2005}]%
        {jervell:finiteTrees}
\bibfield{author}{\bibinfo{person}{Herman~Ruge Jervell}.}
  \bibinfo{year}{2005}\natexlab{}.
\newblock \showarticletitle{Finite Trees as Ordinals}. In
  \bibinfo{booktitle}{\emph{New Computational Paradigms}},
  \bibfield{editor}{\bibinfo{person}{S.~Barry Cooper},
  \bibinfo{person}{Benedikt L{\"o}we}, {and} \bibinfo{person}{Leen Torenvliet}}
  (Eds.). \bibinfo{publisher}{Springer}, \bibinfo{address}{Heidelberg,
  Germany}, \bibinfo{pages}{211--220}.
\newblock


\bibitem[\protect\citeauthoryear{Licata}{Licata}{2014}]%
        {licata:hott:talk}
\bibfield{author}{\bibinfo{person}{Dan Licata}.}
  \bibinfo{year}{2014}\natexlab{}.
\newblock \bibinfo{title}{What is Homotopy Type Theory?}
\newblock
\newblock
\newblock
\shownote{Invited talk at Coq Workshop 2014. Slides available at
  \url{http://dlicata.web.wesleyan.edu/pubs/l14coq/l14coq.pdf}.}


\bibitem[\protect\citeauthoryear{Lumsdaine and Shulman}{Lumsdaine and
  Shulman}{2019}]%
        {lumsdaine:shulman:hit}
\bibfield{author}{\bibinfo{person}{Peter~Lefanu Lumsdaine} {and}
  \bibinfo{person}{Michael Shulman}.} \bibinfo{year}{2019}\natexlab{}.
\newblock \showarticletitle{Semantics of higher inductive types}.
\newblock \bibinfo{journal}{\emph{Mathematical Proceedings of the Cambridge
  Philosophical Society}} (\bibinfo{year}{2019}), \bibinfo{pages}{1--50}.
\newblock


\bibitem[\protect\citeauthoryear{Manolios and Vroon}{Manolios and
  Vroon}{2005}]%
        {MV:ord:acl2}
\bibfield{author}{\bibinfo{person}{Panagiotis Manolios} {and}
  \bibinfo{person}{Daron Vroon}.} \bibinfo{year}{2005}\natexlab{}.
\newblock \showarticletitle{Ordinal arithmetic: algorithms and mechanization}.
\newblock \bibinfo{journal}{\emph{Journal of {A}utomated {R}easoning}}
  \bibinfo{volume}{34}, \bibinfo{number}{4} (\bibinfo{year}{2005}),
  \bibinfo{pages}{387--423}.
\newblock


\bibitem[\protect\citeauthoryear{Nordvall~Forsberg}{Nordvall~Forsberg}{2013}]%
        {nordvallforsberg2013thesis}
\bibfield{author}{\bibinfo{person}{Fredrik Nordvall~Forsberg}.}
  \bibinfo{year}{2013}\natexlab{}.
\newblock \emph{\bibinfo{title}{Inductive-inductive definitions}}.
\newblock \bibinfo{thesistype}{Ph.D. Dissertation}. \bibinfo{school}{Swansea
  University}.
\newblock


\bibitem[\protect\citeauthoryear{Norell}{Norell}{2007}]%
        {norell:thesis}
\bibfield{author}{\bibinfo{person}{Ulf Norell}.}
  \bibinfo{year}{2007}\natexlab{}.
\newblock \emph{\bibinfo{title}{Towards a practical programming language based
  on dependent type theory}}.
\newblock \bibinfo{thesistype}{Ph.D. Dissertation}. \bibinfo{school}{Chalmers
  University of Technology}.
\newblock


\bibitem[\protect\citeauthoryear{Schmitt}{Schmitt}{2017}]%
        {schmitt:ord:key}
\bibfield{author}{\bibinfo{person}{Peter~H. Schmitt}.}
  \bibinfo{year}{2017}\natexlab{}.
\newblock \showarticletitle{A mechanizable first-order theory of ordinals}. In
  \bibinfo{booktitle}{\emph{Automated Reasoning with Analytic Tableaux and
  Related Methods}} \emph{(\bibinfo{series}{Lecture Notes in Computer
  Science})}, \bibfield{editor}{\bibinfo{person}{Renate Schmidt} {and}
  \bibinfo{person}{Cl{\'a}udia Nalon}} (Eds.), Vol.~\bibinfo{volume}{10501}.
  \bibinfo{publisher}{Springer}, \bibinfo{address}{Heidelberg, Germany},
  \bibinfo{pages}{331--346}.
\newblock


\bibitem[\protect\citeauthoryear{Sch\"{u}tte}{Sch\"{u}tte}{1977}]%
        {schuette:book}
\bibfield{author}{\bibinfo{person}{Kurt Sch\"{u}tte}.}
  \bibinfo{year}{1977}\natexlab{}.
\newblock \bibinfo{booktitle}{\emph{{P}roof {T}heory}}.
\newblock \bibinfo{publisher}{Springer}, \bibinfo{address}{Heidelberg,
  Germany}.
\newblock


\bibitem[\protect\citeauthoryear{Setzer}{Setzer}{2004}]%
        {Setzer:prooftheory}
\bibfield{author}{\bibinfo{person}{Anton Setzer}.}
  \bibinfo{year}{2004}\natexlab{}.
\newblock \showarticletitle{Proof theory of {M}artin-{L}{\"o}f {T}ype {T}heory
  -- An overview}.
\newblock \bibinfo{journal}{\emph{Mathematiques et Sciences Humaines}}
  \bibinfo{volume}{42 ann{\'e}e, n$^o$165} (\bibinfo{year}{2004}),
  \bibinfo{pages}{59--99}.
\newblock


\bibitem[\protect\citeauthoryear{Takeuti}{Takeuti}{1987}]%
        {takeuti:book}
\bibfield{author}{\bibinfo{person}{Gaisi Takeuti}.}
  \bibinfo{year}{1987}\natexlab{}.
\newblock \bibinfo{booktitle}{\emph{{P}roof {T}heory} (\bibinfo{edition}{2}
  ed.)}.
\newblock \bibinfo{publisher}{North-Holland Publishing Company},
  \bibinfo{address}{Amsterdam}.
\newblock


\bibitem[\protect\citeauthoryear{{The Univalent Foundations Program}}{{The
  Univalent Foundations Program}}{2013}]%
        {hottbook}
\bibfield{author}{\bibinfo{person}{{The Univalent Foundations Program}}.}
  \bibinfo{year}{2013}\natexlab{}.
\newblock \bibinfo{booktitle}{\emph{Homotopy {T}ype {T}heory: {U}nivalent
  {F}oundations of {M}athematics}}.
\newblock \bibinfo{publisher}{\url{https://homotopytypetheory.org/book}},
  \bibinfo{address}{Institute for Advanced Study}.
\newblock


\bibitem[\protect\citeauthoryear{Turing}{Turing}{1949}]%
        {turing:ordinals}
\bibfield{author}{\bibinfo{person}{Alan Turing}.}
  \bibinfo{year}{1949}\natexlab{}.
\newblock \showarticletitle{Checking a Large Routine}. In
  \bibinfo{booktitle}{\emph{Report of a Conference on High Speed Automatic
  Calculating Machines}}. \bibinfo{publisher}{University Mathematical
  Laboratory}, \bibinfo{address}{Cambridge, UK}, \bibinfo{pages}{67--69}.
\newblock


\bibitem[\protect\citeauthoryear{Vezzosi, M\"{o}rtberg, and Abel}{Vezzosi
  et~al\mbox{.}}{2019}]%
        {VMA:cubical:agda}
\bibfield{author}{\bibinfo{person}{Andrea Vezzosi}, \bibinfo{person}{Anders
  M\"{o}rtberg}, {and} \bibinfo{person}{Andreas Abel}.}
  \bibinfo{year}{2019}\natexlab{}.
\newblock \showarticletitle{Cubical {A}gda: a dependently typed programming
  language with univalence and higher inductive types}.
\newblock \bibinfo{journal}{\emph{Proceedings of the ACM on Programming
  Languages}} \bibinfo{volume}{3}, \bibinfo{number}{ICFP}
  (\bibinfo{year}{2019}), \bibinfo{pages}{87:1--87:29}.
\newblock


\end{thebibliography}

\end{document}